%% file: main.tex
\documentclass[12pt]{article}
\usepackage[margin=1in]{geometry}
\usepackage{amssymb}
\usepackage{amsmath}
\usepackage{graphicx}
\graphicspath{ {./images/} }
\usepackage{manfnt}
\usepackage{amsthm}
\usepackage{stmaryrd}
\usepackage{url}
\usepackage{mathtools}
\usepackage{color}
\usepackage{enumitem}
\usepackage{array}
\usepackage{tikz}
\usepackage{tikz-cd}
\usetikzlibrary{decorations.pathmorphing}
\usepackage{mathrsfs}
\numberwithin{equation}{section}
\usepackage{hyperref}
\usepackage{stackengine}
\usepackage{scalerel}
\usepackage{bm}
\usepackage[utf8]{inputenc}
\usepackage{titletoc}

\makeatletter
\def\@seccntformat#1{%
  \expandafter\ifx\csname c@#1\endcsname\c@section\else
  \csname the#1\endcsname\quad
  \fi}
\makeatother

\tikzset{
  symbol/.style={
    draw=none,
    every to/.append style={
      edge node={node [sloped, allow upside down, auto=false]{$#1$}}}
  }
}

\newcommand{\bbm}{\begin{bmatrix}}
\newcommand{\ebm}{\end{bmatrix}}
\newcommand{\bv}{\begin{vmatrix}}
\newcommand{\ev}{\end{vmatrix}}

\newcommand{\NN}{\mathbb{N}}

\newcommand{\C}{\mathbb{C}}
\newcommand{\G}{\mathbb{G}}
\newcommand{\acts}{\circlearrowright}
\newcommand{\mf}{\mathfrak}
\newcommand{\mc}{\mathcal}
\newcommand{\Q}{\mathbb{Q}}
\newcommand{\Z}{\mathbb{Z}}
\newcommand{\F}{\mathbb{F}}
\newcommand{\R}{\mathbb{R}}
\newcommand{\PP}{\mathbb{P}}
\newcommand{\OO}{\mathbb{O}}
\newcommand{\A}{\mathbb{A}}
\newcommand{\bp}{\begin{pmatrix}}
\newcommand{\ep}{\end{pmatrix}}
\newcommand{\ind}{\mathbbold{1}}
\newcommand{\pt}{\mathrm{pt}}
\newcommand{\tr}{tr}

\makeatletter
\newcommand*\leftdash{\rotatebox[origin=c]{-45}{$\dabar@\dabar@\dabar@$}}
\newcommand*\rightdash{\rotatebox[origin=c]{45}{$\dabar@\dabar@\dabar@$}}
\makeatother

\DeclareMathOperator{\Spec}{Spec}
\DeclareMathOperator{\Frob}{Frob_p}
\DeclareMathOperator{\lFrob}{\overline{Frob}_p}
\DeclareMathOperator{\Pic}{Pic}

\DeclareMathOperator{\Max}{Max}
\DeclareMathOperator{\Gal}{Gal}
\DeclareMathOperator{\Tr}{Tr}
\DeclareMathOperator{\Nm}{Nm}
\DeclareMathOperator{\Disc}{Disc}
\DeclareMathOperator{\Jac}{Jac}
\DeclareMathOperator{\Aut}{Aut}
\DeclareMathOperator{\End}{End}
\DeclareMathOperator{\val}{val}
\DeclareMathOperator{\Rep}{Rep}
\DeclareMathOperator{\Lie}{Lie}
\DeclareMathOperator{\Norm}{Norm}
\DeclareMathOperator{\Ind}{Ind}
\DeclareMathOperator{\stab}{stab}
\DeclareMathOperator{\Res}{Res}
\DeclareMathOperator{\Hom}{Hom}
\DeclareMathOperator{\Perv}{Perv}
\DeclareMathOperator{\ad}{ad}
\DeclareMathOperator{\codim}{codim}
\DeclareMathOperator{\Irr}{Irr}
\DeclareMathOperator{\graph}{graph}
\DeclareMathOperator{\Fun}{Fun}
\DeclareMathOperator{\Mat}{Mat}
\DeclareMathOperator{\Coh}{Coh}
\DeclareMathOperator{\supp}{supp}
\DeclareMathOperator{\Bun}{Bun}
\DeclareMathOperator{\Sh}{Sh}
\DeclareMathOperator{\Loc}{Loc}
\DeclareMathOperator{\sTr}{sTr}
\DeclareMathOperator{\GL}{GL}
\DeclareMathOperator{\SL}{SL}
\DeclareMathOperator{\PGL}{PGL}
\DeclareMathOperator{\res}{res}
\DeclareMathOperator{\Vect}{Vect}
\DeclareMathOperator{\im}{Im}
\DeclareMathOperator{\coker}{coker}
\DeclareMathOperator{\Diads}{Diads}
\DeclareMathOperator{\Ext}{Ext}
\DeclareMathOperator{\Sym}{Sym}
\DeclareMathOperator{\Mod}{Mod}

\DeclareMathOperator{\IC}{IC}

\DeclareSymbolFont{bbold}{U}{bbold}{m}{n}
\DeclareSymbolFontAlphabet{\mathbbold}{bbold}

\theoremstyle{definition}
\newtheorem{claim}{Claim}[section]
\newtheorem{notation}{Notation}[section]
\newtheorem{definition}[claim]{Definition}
\newtheorem{theorem}[claim]{Theorem}
\newtheorem{remark}[claim]{Remark}
\newtheorem{example}[claim]{Example}
\newtheorem{exercise}[claim]{Exercise}
\newtheorem{proposition}[claim]{Proposition}
\newtheorem{lemma}[claim]{Lemma}

\newtheorem{question}[claim]{Question}

\newtheorem{corollary}[claim]{Corollary}

\begin{document}

\titlecontents{section}[0em]
{\vskip 0.5ex}%
{\bf}
{\itshape}
{}%

\title{Langlands correspondence and Bezrukavnikov's equivalence}
\author{Anna Romanov and Geordie Williamson}
\date{\today}
\maketitle

\input{sec-abstract.tex}
\pagebreak
\tableofcontents
\pagebreak
\input{lecture-1.tex}
\pagebreak
\input{lecture-2.tex}
\pagebreak
\input{lecture-3.tex}

\pagebreak
\input{lecture-4.tex}
\pagebreak
\input{lecture-5.tex}
\pagebreak
\input{lecture-6.tex}
\pagebreak
\input{lecture-7.tex}
\pagebreak
\input{lecture-8.tex}
\pagebreak
\input{lecture-9.tex}
\pagebreak
\input{lecture-10.tex}
\pagebreak
\input{lecture-11.tex}
\pagebreak
\input{lecture-12.tex}
\pagebreak
\input{lecture-13.tex}
\pagebreak
\input{lecture-14.tex}
\pagebreak
\input{lecture-15.tex}
\pagebreak
\input{lecture-16.tex}
\pagebreak
\input{lecture-17.tex}
\pagebreak
\input{lecture-18.tex}
\pagebreak
\input{lecture-19.tex}
\pagebreak
\input{lecture-20.tex}
\pagebreak
\input{lecture-21.tex}
\pagebreak
\input{lecture-22.tex}
\pagebreak
\input{lecture-23.tex}
\pagebreak
\input{lecture-24.tex}
\pagebreak
\input{lecture-25.tex}
\pagebreak
\input{lecture-26.tex}

\pagebreak
\input{lecture-27.tex}

\pagebreak
\input{lecture-28.tex}
\pagebreak
\input{lecture-29.tex}
\pagebreak
\input{lecture-30.tex}
\pagebreak
\input{lecture-31.tex}
\pagebreak
\input{lecture-32.tex}
\pagebreak
\input{lecture-33.tex}
\pagebreak
\input{lecture-34.tex}
\pagebreak

\bibliographystyle{alpha}
\nocite{*}
\bibliography{LanglandsAndBezrukavnikov}

\end{document}

%% file: sec-abstract.tex
This document contains notes (taken by the first author) from a course (taught by the second author) at the University of Sydney. The course took place over two extended semesters, and consisted of two hours of lectures per week between 2019-2020. The course is divided in two parts, reflecting the two terms.

The first part (concluding with Lecture 11) is an attempt to give an informal introduction to what the Langlands program is about, from an arithmetical point of view. We assume the audience (like the lecturer) is a beginner in this subject, but had a first course in complex analysis, Galois theory, topology and representation theory. At times we also assume background in algebraic geometry. Not much is proved, but we try to give enough detail to convince the reader that there is a lot of marvellous mathematics here.

The second longer part (Lectures 12 through 34) tries to give enough background in geometric representation theory to understand Bezrukavnikov’s equivalence. Taking the tamely ramified local Langlands correspondence as motivation, we pass through work of Deligne-Lusztig and Ginzburg giving a coherent construction of the affine Hecke algebra. Here we are roughly following the book of Chriss-Ginzburg, but our route is different at times. Bezrukavnikov’s equivalence is a categorification of this isomorphism, and is our aim for the rest of the notes. In order to even understand the statement of Bezrukavnikov’s equivalence, we need much of the toolbox of modern geometric representation theory: perverse sheaves; highest weight categories; Koszul duality; the geometric Satake equivalence, etc. A major role is also played by monoidal categories and their actions (“higher representation theory”). We try to spend enough time on each of these topics to make students feel somewhat comfortable with the ideas. It goes without saying that all of this took much longer to cover than anyone expected, and at the very end we arrive at a statement of Bezrukavnikov’s equivalence. Its proof will have to wait for the next lecture series!

The spirit of these notes is very informal, and this is intentional. We hope that they are nonetheless useful for a casual reader trying to orient themselves in a fascinating but potentially intimidating landscape. We assume that the reader is willing to take some things on faith, and have tried to be honest. Audience members were encouraged to do exercises throughout, and this wouldn’t be bad advice for any potential reader either.

%% file: lecture-1.tex
\section{Lecture 1: Reciprocity Laws} 
\label{lecture 1}

If you do nothing else with this course this semester beyond attending the first lecture, you should at least try to read \cite{Lan90}.

\subsection{Reciprocity Laws}

We start at the natural starting place: an equation. Consider the equation
\[
x^2+1=0.
\]
If $p$ is a prime, one might wonder: how many solutions does this equation have, modulo $p$? Some calculations will reveal the following table. 

\begin{center}
\begin{tabular}{ | c|c|c| c| c| c| c| c| c| c |} 
\hline
$p$ & 2 & 3 & 5 & 7 & 11 & 13 & 17 & 19 & 23 \\ 
\hline
\# of sol's mod $p$ & 1 & 0 & 2 & 0 & 0 & 2 & 2 & 0 & 0 \\ 
\hline
$p$ mod 4 & 2 & 3 & 1 & 3 & 3 & 1 & 1 & 3 & 3 \\ 
\hline
\end{tabular} ...
\end{center}

\noindent
We see a pattern. The prime 2 is weird, so we ignore it. But for the rest, it seems that 
\[
\# \text{ of solutions mod } p \neq 2 =
\begin{cases}
2 & \text{ if } p = 1 \mod 4, \\
0 & \text{ if } p =3 \mod 4.
\end{cases}
\]
This pattern is surprising. It appears to be saying that there is a global rule governing the number of solutions mod $p$; that is, that the different primes somehow ``talk to one another''.

Here we can give a simple proof of why our claim above must be true. Assume $p \neq 2$. We have a short exact sequence 
\[
1 \rightarrow (\mathbb{F}_p^\times)^2 \rightarrow \mathbb{F}_p^\times \rightarrow \{ \pm 1\}\rightarrow 1, 
\]
where the third arrow is given by $x \mapsto x^{\frac{p-1}{2}}$. Therefore, 
\[
-1 \text{ is a square mod $p$} \iff (-1)^\frac{p-1}{2} =1 \iff \frac{p-1}{2} \text{ is even} \iff p=1 \mod 4. 
\]
Let's do another example. Consider the equation 
\[
x^2-3=0. 
\]
We ask the same question and compute: 

\begin{center}
\begin{tabular}{ | c|c|c| c| c| c| c| c| c| c |c|c|} 
\hline
$p$ & 2 & 3 & 5 & 7 & 11 & 13 & 17 & 19 & 23 &29 & 31\\ 
\hline
\# of sol's mod $p$ & 1 & 1 & 0 & 0 & 2 & 2 & 0 & 0 & 2 & 0 & 0\\ 
\hline
$p$ mod 12 & 2 & 3 & 5 & 7 & 11 & 1 & 5 & 7 & 11 &5 & 7\\ 
\hline
\end{tabular} ...
\end{center}
Again, we have some small weird primes (2 and 3), so we throw them out. For the rest, we make a guess:
\[
\# \text{ of solutions mod } p \neq 2,3 =
\begin{cases}
2 & \text{ if } p = 1,11 \mod 12, \\
0 & \text{ if } p =5,7 \mod 12.
\end{cases}
\]
To prove that this is indeed the case, we introduce a little more technology. Let $p\neq 2$ be a prime. Define 
\[
\epsilon(p) = \begin{cases} 0 &\text{ if } p = 1 \mod 4, \\
1 &\text{ if } p = 3 \mod 4 \end{cases}
\]
and the Legendre symbol
\[
\left( \frac{x}{p} \right) = x^\frac{p-1}{2} \mod p 
= \begin{cases} 1 &\text{ if } x \text{ is a square mod $p$}, \\
-1 &\text{ if } x \text{ is not a square mod $p$}. 
\end{cases}
\]
(For example, we saw above that $\left( \frac{-1}{p}\right)=(-1)^{\epsilon(p)}$.)

\begin{theorem} ({\bf Gauss's Law}) Let $p,q$ be distinct primes $\neq 2$. Then 
\[
\left( \frac{p}{q} \right) \left( \frac{q}{p} \right) = (-1)^{\epsilon(p)\epsilon(q)}.
\]
\end{theorem}
With this we can prove that our guess was correct. Assume $p \neq 2,3$. Then 
\begin{align*}
x^2-3 \text{ has 2 solutions mod $p$} &\iff \left( \frac{3}{p} \right) = 1\\
&\iff \left( \frac{p}{3} \right) (-1)^{\epsilon(p) \epsilon(3)}=1 \\
&\iff \left( \frac{p}{3} \right) (-1)^{\epsilon(p)}=1 \\
& \iff \begin{cases} p = 1 \mod 3 \text{ and } p= 1 \mod 4 \\ 
p=2 \mod 3 \text{ and } p=3 \mod 4 
\end{cases} \\
& \iff p = 1 \text{ or } -1 \mod 12. 
\end{align*}

These are examples of {\bf reciprocity laws}. All polynomials of degree 2 can be worked out analogously to the ones above using quadratic reciprocity (Gauss's law). There was much activity on this problem starting with Gauss's work, which finally led to Artin's reciprocity law. This implied all known reciprocity laws at the time, and in particular treats polynomials of degree $3$ and $4$.  However, we get stuck at 5. For example, consider 
\[
x^5 + 20x + 16=0.
\]
We can construct a table
\begin{center}
\begin{tabular}{ | c|c|c| c| c| c| c| c| c| c |c|c|c|c|c|c|c|} 
\hline
$p$ & 2 & 3 & 5 & 7 & 11 & 13 & 17 & 19 & 23 &29 & 31&37&41&43&47&53\\ 
\hline
\# of sol's mod $p$ & 1 & 0 & 1 & 2 & 2 & 0 & 2 & 0 & 1 & 2 & 0 & 1 & 2 & 2 & 0 & 0\\ 
\hline
\end{tabular} ...
\end{center}
but no obvious pattern emerges. (For a table that goes much further than this one, see the sheet on the course website.) It turns out that there is a pattern, but it is very well-hidden, and to find it, we need analysis. 

\subsection{Higher dimensional varieties}
We could ask similar questions for polynomials in two variables. Consider the equation 
\[
y^2 = x^3 + 1. 
\]
How many solutions does this equation have modulo $p$? Let's try to answer this for one specific prime. Let $p=5$, and we can compile our results in the following table:   
\begin{center}
\begin{tabular}{ | c|c|c| c| c| c|} 
\hline
$y\backslash x$ & 0 & 1 & 2 & 3 & 4\\
\hline
0 & $\times$ & $\times$ & $\times$ & $\times$ & $\checkmark$ \\
\hline 
1 & $\checkmark$ & $\times$ & $\times$ & $\times$  & $\times$ \\
\hline 
2 & $\times$ & $\times$ & $\checkmark$ & $\times$ & $\times$ \\
\hline 
3 & $\times$ & $\times$ & $\checkmark$ & $\times$ & $\times$\\
\hline 
4 & $\checkmark$ & $\times$ & $\times$ & $\times$ & $\times$ \\
\hline
\end{tabular} 
\end{center}
So here we found that there are five solutions modulo 5. In general, how many solutions do we expect? Well, the map $x \mapsto x^3+1$ in $\mathbb{F}_p$ is ``roughly random,'' about half the elements of $\mathbb{F}_p$ are squares, and for every square we get two solutions, so  we expect {\it approximately p solutions.} But how often is this actually the case? We can measure the accuracy of this estimation by studying the {\bf Sato-Tate error term}: 
\[
ST(p)=p-\#(\text{solutions modulo } p).
\]
Here is another table. 
\begin{center}
\begin{tabular}{ | c|c|c| c| c| c| c| c| c| c |c|c|c|c|c|c|c|} 
\hline
$p$ & 2 & 3 & 5 & 7 & 11 & 13 & 17 & 19 & 23 &29 & 31&37&41&43&47&53\\ 
\hline
\# $ST(p)$&  0 & 0 & 0 & -4 & 0 & 2 & 0 & 8 & 0 & 0 & -4 & -10 & 0 & 8 & 0 & 0\\
\hline
\end{tabular} ...
\end{center}

Notice how frequently the Sato-Tate error term is zero! We can now study this table and see if any patterns emerge. This is the content of the {\bf Sato-Tate conjecture}, which is basically known thanks to recent work of Harris, Taylor, Clozel, and many others. 

\subsection{What is going on here? What does this have to do with representation theory?}

Let $f(x) \in \Z[x]$ be an irreducible polynomial with integral coefficients. We can consider the {\bf splitting field} of $f$
\[
K=\Q(\alpha_1, \ldots, \alpha_n),
\]
and the associated {\bf Galois group}
\[
\Gamma = \Gal(K/\Q), 
\]
which acts on the set of roots $\{\alpha_1, \ldots \alpha_n\}$. As representation theorists, our natural instinct when we see a group action is to linearize. Doing this here results in the {\bf permutation representation}
\[
\Gamma \acts H=\bigoplus_{i=1}^n \C \alpha_i.
\]
We can also consider the reduction of $f$ modulo $p$, $\bar{f}(x) \in \mathbb{F}_p[x]$, as we did in the previous section. In general, $\bar{f}$ will be reducible. If $p \nmid \Delta(f)$ (that is, $p$ is not one of the ``weird'' primes we encountered earlier), then $\bar{f}(x)$ has $n$ roots, $\overline{\alpha}_1 , \ldots , \overline{\alpha}_n \in \mathbb{F}_{p^n}$. Recall that the Galois group $\Gal(\mathbb{F}_{p^n}/\mathbb{F}_p)=\Z/n\Z$ is generated by the Frobenius map
\[
\lFrob: x \mapsto x^p.
\]
Then $\F_p = (\F_{p^n})^{\lFrob}$, and the number of solutions of $\bar{f}$ is the number of fixed points of $\lFrob$ on $\{\overline{\alpha}_1, \ldots, \overline{\alpha}_n\}$. For such a $p$ (``unramified'') and after a choice (``prime in $\mathbb{O}$ above $p$''), we get a bijection 
\[
\{\alpha_1, \ldots, \alpha_n\} \longrightarrow \{\overline{\alpha}_1, \ldots \overline{\alpha}_n\},
\]
and an element $\Frob\in \Gamma$ such that the action of $\Frob$ on $\{\alpha_1, \ldots, \alpha_n\}$ aligns with the action of $\lFrob$ on $\{ \overline{\alpha}_1, \ldots, \overline{\alpha}_n\}$ under the bijection above. 
\begin{remark}
Different choices of ``prime in $\mathbb{O}$ over $p$'' result in conjugate $\Frob$'s. Hence, it is best to think of $\Frob$ as a conjugacy class instead of an individual element. 
\end{remark}

\noindent
The upshot of the discussion above is that
\begin{align*}
    \# \text{ solutions modulo $p$ }& = \# \text{ fixed points of $\lFrob$ on } \{ \overline{\alpha}_1, \ldots, \overline{\alpha}_n\} \\ &= \# \text{ of fixed points of $\Frob$ on } \{ \alpha_1 \ldots, \alpha_n\} \\ 
    &= \Tr(\Frob, H),
\end{align*}
where $H$ is the permutation representation introduced at the beginning of this section. The number $\Tr (\Frob, H)$ is completely canonical - it doesn't depend on any of our choices! So we've reduced our question of finding solutions of polynomials modulo $p$ to computing something that looks very much like the character of a representation. 

\vspace{5mm}
\noindent
{\bf The Punchline:} If $p \nmid \Delta(f)$,
\[
\# \text{ solutions of $f$ mod $p$ } = \Tr(\Frob, H). 
\]

\subsection{Schematic picture of the Langlands correspondence}

({\em Do not worry if this makes no sense!}) A caricature of the Langlands correspondence is captured in the diagram below. 
\[
\begin{tikzcd}
\text{ ``geometric'' reps  $V$ of }\Gal(\overline{\Q}/\Q) \arrow[r]
&  \text{``character'' }\Tr(\Frob, V)\arrow[d]
& \text{ automorphic forms}\arrow[ld]\\
& \text{$L$-functions (analytic)}
\end{tikzcd}
\]
From any ``geometric'' representation of $\Gal(\overline{\Q}/\Q)$, we can take the trace of Frobenius, as we did in the previous section for the permutation representation $H$. We should think of this procedure as taking the character of the representation. To $\Tr(\Frob, V)$, we can attach the associated ``$L$-function,'' which is an analytic object. (For example, when we start with the trivial representation of $\Gal(\overline{\Q}/\Q)$, the resulting $L$-function is the Riemann $\zeta$-function.) On the other hand, there is also a procedure for constructing $L$-functions from automorphic forms. The Langlands correspondence is an attempt to align these two sources of $L$-functions. 

{\em This is very deep.} For example, two-dimensional representations of $\Gal(\overline{\Q}/\Q)$ result in Hecke $L$-functions, and the corresponding automorphic forms are modular forms. It turns out that working out the correspondence for 2-dimensional representations is enough to prove Fermat's last theorem. 

\subsection{Chebotarev density theorem}

If we talk of $\Tr(\Frob,H)$ as a ``character,'' we would like to know at least that the set $\{\Frob\}$ for all $p$ unramified cover the set of all conjugacy classes of $\Gamma$. This is a deep theorem. 

\begin{theorem}
\label{Chebotarev density}
(Chebotarev density theorem) Fix a conjugacy class $C \subset \Gamma$. Then 
\[
\{ p \text{ unramified } | \Frob=C\}
\]
has density $|C|/|\Gamma|$.
\end{theorem}
Here density refers to either the natural density or the analytic density of the set of primes.

\begin{example}
\label{example1}
Let $f(x)=x^2+1 \in \Z[x]$. The set of roots of $f(x)$ is $\{i, -i\}$. The splitting field is $K=\Q(i)$ and $\Gal(\Q(i)/\Q) = \Z / 2\Z = \{id, s\}$. In this example, all $p \neq 2$ are unramified. Then for such an unramified $p$, 
\[
\Frob: i \mapsto i^p.
\]
Hence, 
\[
\Frob=\begin{cases}
id &\text{ if } p=1 \mod 4, \\
s &\text{ if } p=3 \mod 4. 
\end{cases}
\]
\end{example}

\begin{exercise} (Mandatory)
Check that $\Frob$ is indeed given as above!
\end{exercise}

\begin{exercise}
(Harder) By considering cyclotomic extensions (i.e. $\Q(e^{2\pi i / m}))$, show that Chebotarev's density theorem implies Dirichlet's theorem on primes in arithmetic progression. 
\end{exercise}

At the beginning of today's lecture, we discussed patterns in the number of solutions of a given polynomial modulo $p$. There is a sheet on the course webpage which shows tables of these patterns for the polynomials $x^2 + 1$, $x^2-3$, $x^2+x+1$, $x^2+2x+3$, and $x^2-x-1$. A somewhat mysterious feature of these tables was the modulus appearing in the patterns. (For example, we showed that $x^2-3$ has two solutions modulo $p\neq 2,3$ if and only if $p = 1$ or $11 \mod 12$. Where did 12 come from?) We'll complete today's lecture with an example to demonstrate where this modulus comes from.
\begin{example}
Consider the polynomial $f(x)=x^2-x-1$. We can see from the patterns on the handout that $f(x)$ has 2 solutions mod $p$ if $p=1$ or $9 \mod 10$ and $f(x)$ has 0 solutions mod $p$ if $p = 3$ or $7 \mod 10$ (for $p$ unramified). In this example, the splitting field is $K=\Q(\phi)$, where $\phi = \frac{1 + \sqrt{5}}{2}=2\cos(\pi/5)$ is the golden ratio, and $\Gal(\Q(\phi)/\Q)=\Z/2\Z = \langle s \rangle$. Note that $\phi = \zeta + \overline{\zeta}$, where $\zeta = e^{\pi i /5}$ is a fifth root of unity. Hence we can embed 
\[
K \hookrightarrow \Q(e^{2 \pi i/10})=\Q(\zeta)
\]
via $\phi \mapsto \zeta + \overline{\zeta}$. As in the last example, for unramified $p$, 
\[
\Frob: \zeta \mapsto \zeta^p.
\]
Hence, 
\[
\Frob: \begin{cases}
\phi \mapsto \phi & \text{ if } p=1 \text{ or } 9 \mod 10, \\
\phi \mapsto s(\phi) & \text{ if } p = 3 \text{ or }7 \mod 10. 
\end{cases}
\]
In general, the modulus for quadratic fields are determined by embeddings of the splitting field into cyclotomic fields:
\[
K \hookrightarrow \Q (e^{2\pi i / \text{modulus}}).
\]
\end{example}
\begin{remark}
Consider the degree 5 polynomial $f(x)=x^5+20x+16$ we discussed in the first section. In all of the primes that occurred in our table, $f(x)$ had either 0 or 2 solutions. However, we would expect that for certain primes, $f(x)$ should have 5 solutions by Chebotarev's theorem. Should we be worried that we haven't seen any 5's in our chart? Geordie reassures us that we shouldn't be worried. In this example, $\Gamma = A_5$ has order 60. Then, by our discussion earlier,
\begin{align*}
\# \text{ solutions modulo }p =5 &\iff \text{all solutions in $\F_{p^m}$ are defined over $\F_p$} \\
&\iff \text{all solutions are fixed by $\lFrob$ }
\\
&\iff \Frob = id.
\end{align*}
Since $id$ is in its own conjugacy class, we expect $f(x)$ to have 5 solutions modulo p about $1/60$th of the time by Chebotarev's density theorem. We included fewer than 60 primes in our table, so we shouldn't be surprised that we haven't seen this happen yet. 
\end{remark}

It may seem like considering the number of solutions of a polynomial over a finite field is a cute, but not particularly important problem. However, it is actually of fundamental importance in number theory. A {\bf number field} is a finite extension of $\Q$. All number fields (which are Galois extensions) are splitting fields of polynomials $f(x) \in \Z[x]$. One of the the most basic open questions in number theory is the following:

\begin{question}
How many number fields are there?
\end{question}

\noindent 
We can determine the field extension $K$ corresponding to the polynomial $f(x)$ by reducing mod $p$:
\begin{theorem}
The set $\left\{ p \mid  p \text{ unramified and }  f(x) \text{ splits completely mod $p$} \right\}$ completely determines $K$. 
\end{theorem}
So our motivational problem may have been cute, but it certainly isn't unimportant. 

%% file: lecture-2.tex
\section{Lecture 2: Review of some algebraic number theory}
\label{lecture 2}

Last time we discussed how by Chebotarev's density theorem, the equation $f(x)=x^5+20x+16$ should have five solutions modulo p about $1/60^{th}$ of the time. Joel (+ a computer) computed that in the set of all primes below $500,000,$ there are $16,613$ where $f(x)$ has no solutions, $10,367$ where $f(x)$ has one solution, $13,885$ with two solutions, and $673$ with five solutions. In this case, we know that the Galois group is $A_5$, so it is order 60, but if we didn't know the Galois group, we could use this data to predict its order. 

\begin{exercise}
Check the consistency of the numbers above with Chebotarev's density theorem. 
\end{exercise}

The goal of today's lecture is to give the necessary background in algebraic number theory to continue. It is roughly based on a lecture by Dick Gross \cite{gross}. 

\subsection{Number fields}
A {\bf number field} is a finite extension of $\Q$. Given a number field $K/\Q$ of degree $n$ (in this lecture, our field extensions will always be degree $n$), there is an associated {\bf ring of integers} $\OO \subset K$ consisting of all elements of $K$ which satisfy a monic polynomial with coefficients in $\Z$. The ring of integers $\OO$ is a free $\Z$-module of rank $n$, as well as a Dedekind domain (i.e. Noetherian, normal, Krull dimension 1). 

\begin{exercise} Show that the following field extensions have the following rings of integers:
\begin{enumerate}
    \item $K=\Q(i)$, $\OO=\Z[i]$.
    \item $K=\Z(\sqrt{2})$, $\OO=\Z[\sqrt{2}]$.
    \item $K=\Q(\sqrt{5}),$ $\OO=\Z[\phi]$, where $\phi = \frac{1 + \sqrt{5}}{2}$.
\end{enumerate}
\end{exercise}
\noindent 
Generally, for a complicated extension, it is not easy to find $\OO$.

We can measure how complicated a number field is using something called the {\bf discriminant}. It is defined as follows. Let $K/\Q$ be a number field. Given $x \in K$, we get a $\Q$-linear map $x \cdot: K \rightarrow K$. Using this we define 
\begin{align*}
    \Tr:&K \rightarrow \Q\\
    \Nm:&K^\times \rightarrow \Q^\times 
\end{align*}
by $\Tr(x):=\Tr(x \cdot)$, $\Nm(x):= \det(x \cdot )$. This gives us a bilinear form called the {\bf trace form}: 
\begin{align*}
    K \times K &\rightarrow \Q \\
    (x,y)&:=\Tr(xy).
\end{align*}
The trace form is nondegenerate over $\Q$ because $\Tr(1)=n$, hence $\Tr(xx^{-1})=n \neq 0$. Since any element of $\OO$ satisfies a monic polynomial with integer coefficients, the trace form restricts to a map $\OO \times \OO \rightarrow \Z$. Choose a $\Z$-basis $\{
\alpha_1, \ldots, \alpha_n\}$ for $\OO$. Then the {\bf discriminant} of the field $K$ is
\[
\Disc(K):=\det((\alpha_i,\alpha_j)).
\]
This is a close relative of the discriminant of a polynomial. 
\begin{remark}
We have no idea how many number fields there are, so it is useful to have a measurement of how complicated a number field is. This is one of the reasons the discriminant is so useful. 
\end{remark}
\begin{example}
\label{Q(i)}
Let $K=\Q(i)$. Then $\OO=\Z[i]$ has basis $\{\alpha_1, \alpha_2 \} =\{1, i\}$. We can compute 
\[
((\alpha_i, \alpha_j))=\begin{pmatrix} 2 & 0 \\ 0 & -2 \end{pmatrix},
\]
so $\Disc(K)=\det((\alpha_i, \alpha_j))=-4$. 
\end{example}
\begin{exercise} \begin{enumerate}
    \item Let $\alpha \in \Z$ be square-free. Let $K=\Q(\sqrt{\alpha})$. Then 
    \[
    \OO=\begin{cases} \Z[\sqrt{\alpha}] & \text{ if } \alpha \neq 1 \mod 4, \\
    \Z\left[\frac{1+\sqrt{\alpha}}{2}\right] & \text{ if } \alpha = 1 \mod 4. 
    \end{cases} 
    \]
    Hence, calculate 
    \[
    \Disc(K)=\begin{cases} 4 \alpha & \text{ if } \alpha \neq 1 \mod 4, \\
    \alpha & \text{ if } \alpha = 1 \mod 4. 
    \end{cases} 
    \]
    \item Calculate the discriminant of $\Q(e^{2 \pi i/3})$. 
\end{enumerate}
\end{exercise}

Let $K/\Q$ be an \'{e}tale $\Q$-algebra (i.e. a finite separable extension of $\Q$). Then $K \otimes_\Q \R$ is an \'{e}tale $\R$-algebra, so $K \otimes_\Q \R \simeq \R^{r_1} \times \C^{r_2}$, where $n=r_1 + 2r_2$. The field $K$ is {\bf totally real} if $r_2=0$ (or, equivalently, if every embedding $K \hookrightarrow \C$ lands in $\R$). For example, this happens if $K$ is the splitting field of a polynomial with real roots.

\begin{exercise}
Show that the signature of the trace form on $K$ totally real is $(r_1+r_2, r_2)$. In particular, 
\[
K \text{ is totally real } \iff r_2 = 0 \iff ( \cdot, \cdot ) \text{ is positive definite}.
\]
\end{exercise}

\subsection{An analogy}
Next we will explore a useful analogy which will be a theme of this course. 
\begin{align*}
\left\{ \begin{array}{c}\text{finite extensions}
\\
\text{$K$ of $\C(x)=\text{Frac}\C[x]$} \\ 
\text{(called {\bf function fields})}
\end{array} \right\}  &\leftarrow \left\{ \begin{array}{c} \text{smooth complex} 
\\
\text{projective curves} \\
\text{$C$ over $\PP^1\C$}
\end{array}   \right\} \leftrightarrow \left\{  \begin{array}{c} \text{compact Riemann}
\\
\text{surfaces with a }\\
\text{map to $\PP^1\C$}
\end{array}   \right\}
\end{align*}
The first arrow going left is given by taking the function field of the curve. We can also go in the other direction. Given 
\[
\begin{tikzcd}
K  \arrow[d]
& \OO\arrow[d]  \arrow[l, hook] \\
\C(x)
& \C[x] \arrow[l, hook]
\end{tikzcd}
\]
we obtain a map $\Spec\OO \rightarrow \Spec\C[x]=\A^1$, so we have a unique compactification and an equivalence between the first two sets. 

Similarly, there is a bijection 
\begin{align*}
\left\{ \begin{array}{c}\text{finite extensions}
\\
\text{$K$ of $\F_p(x)$} 
\end{array} \right\}  &\leftrightarrow \left\{ \begin{array}{c} \text{smooth projective} 
\\
\text{ curves $C$ over $\PP^1_{\F_p}$}
\end{array}   \right\}.
\end{align*}
Classically, people worked on problems in number theory in the algebraic world. Artin moved to the geometric world and proved deep results there. Many people in modern number theory work on the geometric side and hope to prove something about number fields. Here is a (very rough) schematic of difficulty: 
\[
\text{function fields over }\C << \text{function fields over }\F_p << \text{ number fields}
\]

For a very inspiring reference for all of this, see Andr\'{e} Weil's letter to his sister Simone Weil on the role of analogy in mathematics \cite{Weil}. 

\begin{exercise} (Use of Google allowed.) Show that Fermat's last theorem is true in function fields; i.e. if $f,g,h \in k[x]$ are relatively prime and $f^n+g^n=h^n$, then $n=2$.  
\end{exercise}

\subsection{The fundamental exact sequence}

Let $K/\Q$ be a number field and $\OO$ the ring of integers of $K$. A {\bf fractional ideal} is a finitely generated $\OO$-submodule of $K$. Given two fractional ideals $I,J$, we can construct their product: 
\[
IJ:=\left\{\sum\alpha_i\beta_j | \alpha _i \in I, \beta_j \in J \right\}.
\]
(This is the ``union'' in the sense of algebraic geometry.) Since $\OO$ is a Dedekind domain, 
\begin{itemize}
    \item every prime ideal $\mathfrak{p} \neq 0$ is maximal, and 
    \item every fractional ideal has a unique factorization $I=\prod \mf{p}_i^{e_i}$, where $\mathfrak{p}_i$ are prime ideals. 
\end{itemize}
Denote by $\mathcal{J}=\bigoplus_{\mathfrak{p}\neq 0 \text{ prime}} \Z\mathfrak{p}$ the group of nonzero fractional ideals under this product. We have the following fundamental exact sequence:
\[
\{1\}\rightarrow \OO^\times \hookrightarrow K^\times \rightarrow \mathcal{J} \rightarrow \mathcal{C}\ell(K) \rightarrow 0.
\]
Here $\mathcal{C}\ell(K)$ is the {\bf ideal class group} of $K$, which measures the failure of $\OO$ to be a PID. The ideal class group is difficult to calculate, and we know very little about it in general. The image of the second map in this exact sequence is $\mathcal{P}$, the set of all principal ideals (that is, ideals of the form $x\OO$ for some $x \in K^\times$) of $K$.
\begin{theorem} (Fundamental finiteness theorems)
\begin{enumerate}
    \item The ideal class group $\mathcal{C}\ell(K)$ is finite. 
    \item The group $\OO^\times$ is finitely generated of rank $r_1 + r_2 - 1$. 
\end{enumerate}
\end{theorem}
\begin{exercise} Compute $\OO^\times$ for $\Q(\sqrt{2})$ and $\Q(\sqrt{3})$. ({\em Hint}: Pell's equation)
\end{exercise}

We can study an analogue of this exact sequence for a smooth projective curve. Let $C$ be a compact Riemann surface. Then under the analogy, 
\begin{align*}
    0 \neq \mf{p} \text{ prime ideals }= \text{ maximal ideals} &\leftrightarrow \text{ points of }C,
\end{align*}
and we have the following exact sequence:
\[
\{1\} \rightarrow \C^\times \rightarrow K^\times \rightarrow \mathcal{P} \hookrightarrow \bigoplus_{x \in C} \Z x \twoheadrightarrow \Pic(C)\rightarrow 0.
\]
Here $K$ is the function field of $C$, $\mathcal{P}$ is the set of divisors of meromorphic functions (``principal divisors''), and $\Pic(C)$ is the Picard group of $C$ (isomorphism classes of line bundles on $C$). 

\begin{remark}
The group $\Pic(C) = \Jac(C) \times \Z$ is very far from finite.  Also, $\C^\times$ is not finitely generated. So in this setting, neither of the fundamental finiteness theorems hold. 
\end{remark}

\begin{exercise}
(For those who know some algebraic geometry) Show that the analogues of $\mathcal{C}\ell(K)$ and $\OO^\times$ are finite if $C$ is an affine curve defined over a finite field.
\end{exercise}

\begin{exercise}
(If you know what $K_0$ is) Show that $K_0(\OO)=\Z \oplus \mathcal{C}\ell(K)$. 
\end{exercise}

\subsection{Ramification}
\label{ramification}
 First consider a smooth projective curve $C\xrightarrow{f} \mathbb{P}^1\C$ (e.g. $y^2=f(x)$ for some polynomial $f(x)$ without repeated roots.) After deleting a finite set of points from $\mathbb{P}^1\C$ (the ``discriminant'' of $f$), then $f$ is \'{e}tale, and hence gives us a finite covering of open sets $U\subset \mathbb{P}^1\C$. So all fibres are ``the same'' away from finitely many points where $f$ is ``ramified''. A picture: 
 
 \begin{center}
 \includegraphics{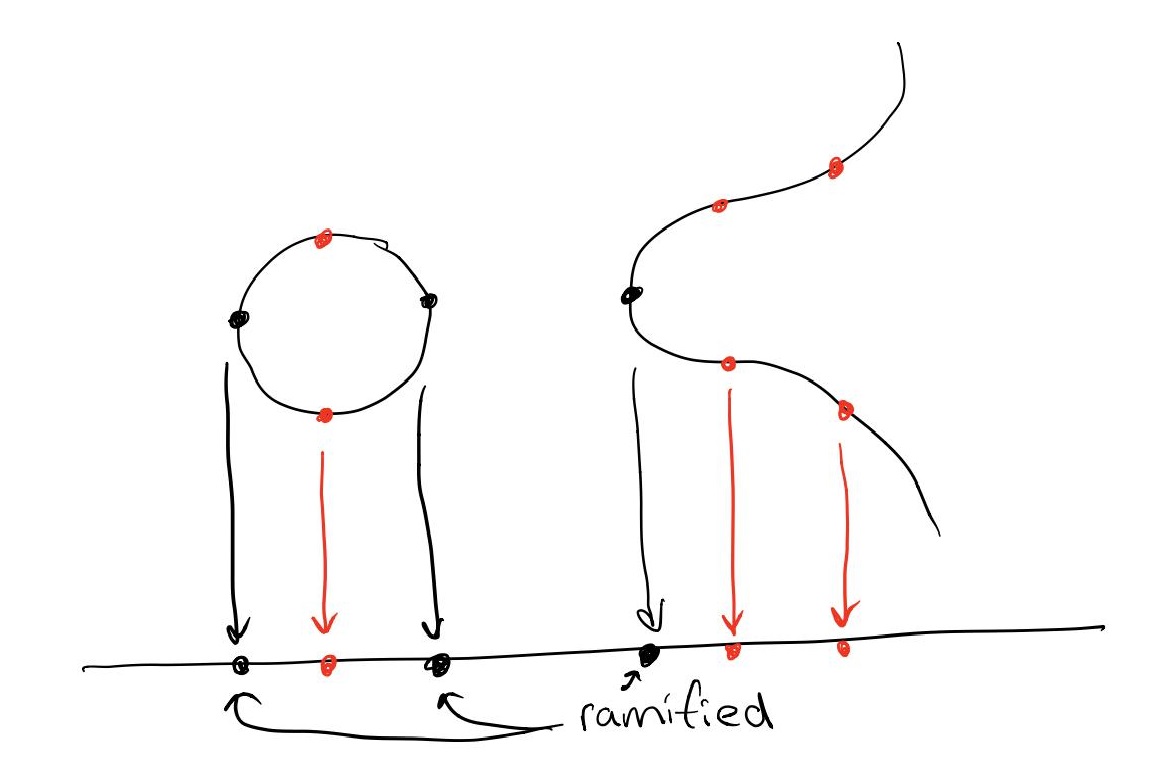}
\end{center}

 Over finite fields, almost the same thing happens. Again, if we have a smooth projective curve $C \xrightarrow{f} \mathbb{P}^1_{\F_p}$, $f$ is \'{e}tale after deleting finitely many points. For example, consider $C=\Spec\F_p[x,y]/(y^2=x^3+1)\rightarrow \Spec\F_p[x]$. Fibres of this map over a point $x=\lambda$ are singular if $\lambda$ is a third root of unity, or else  
 \[
\F_p[x,y]/(y^2=x^3+1)\otimes_{\F_p[x]}\F_p = \F_p[y]/(y^2=\lambda^3+1) = \begin{cases} 
\F_p \times \F_p  &\text{ if } \lambda^3+1 \text{ is a square} \\
\F_{p^2} &\text{ if } \lambda^3 + 1 \text{ is not a square}
\end{cases}.
 \]
 A picture: 
 
  \begin{center}
 \includegraphics{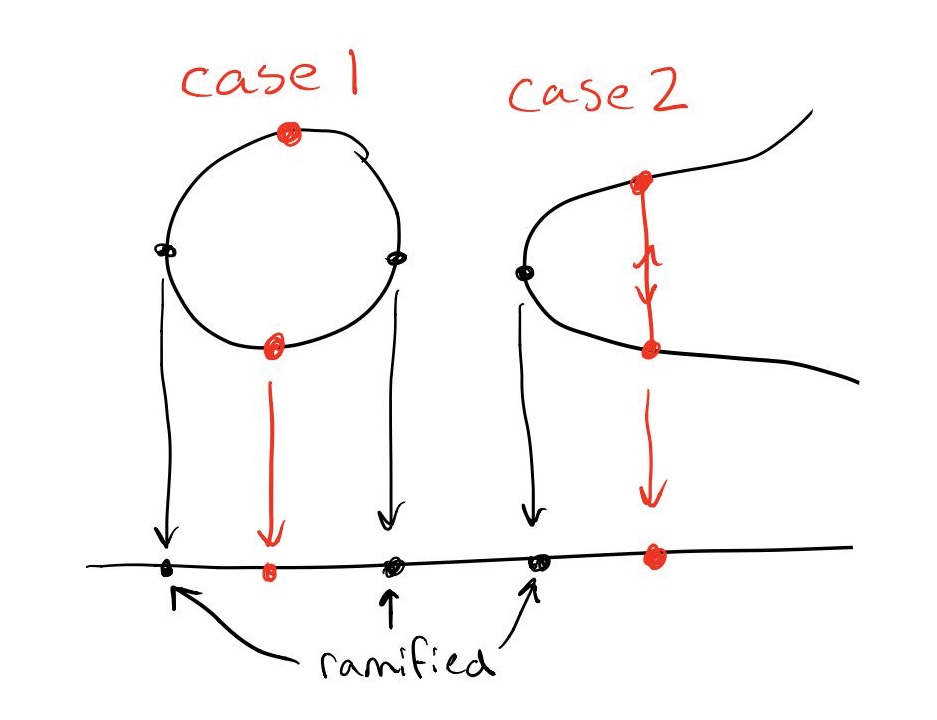}
\end{center}
 
 A similar picture holds for number fields:
 
  \begin{center}
 \includegraphics[scale=0.75]{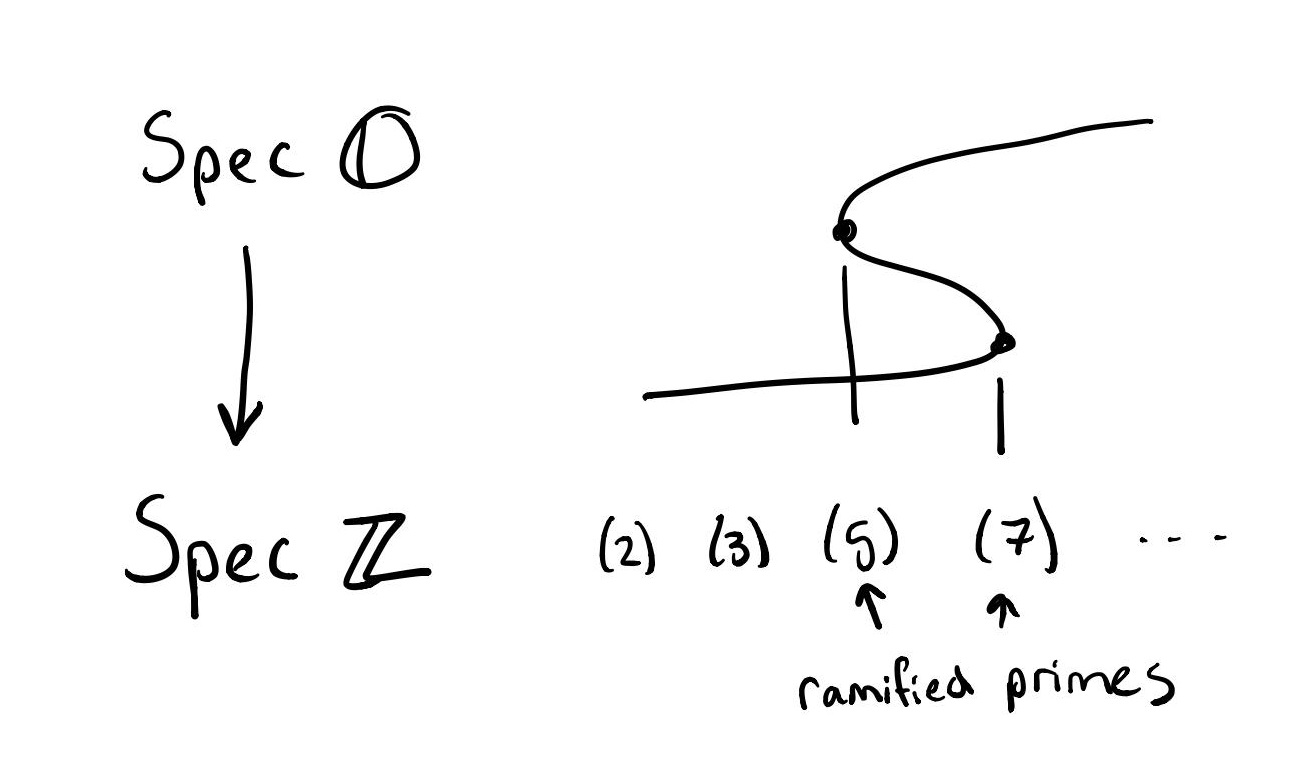}
\end{center}
 
 We can make this picture rigorous. Let $K/\Q$ be a number field and $\OO$ its ring of integers. For each prime $p \in \Z$, the corresponding ideal in $\OO$ decomposes into the product of prime ideals in $\OO$:
 \[
 (p)=\prod_{i=1}^{g_p} \mathfrak{p}_i^{e_i}. 
 \]
 We say the primes $\mathfrak{p}_i$ appearing in this decomposition are the ``primes above $(p)$''. The number $e_i$ is the {\bf ramification index} of $\mf{p}_i$. For each $\mf{p}_i$, the field $\OO/\mf{p}_i$ is a finite extension of $\F_p$ (so $\OO/\mf{p}_i=\F_{p^{f_i}}$ for some $f_i$), and the degree $f_i$ of this extension is the {\bf inertia degree}. 
 
 \begin{exercise}
 \label{important}
 (Important!) Show that $\displaystyle n=\sum_{i=1}^{g_p} e_if_i$.
 \end{exercise}
 
 \begin{definition}
 \begin{enumerate}[label=(\alph*)]
     \item The ideal $(p)$ is {\bf unramified} if all $e_i=1$. Otherwise it is {\bf ramified}.
     \item The ideal $(p)$ {\bf splits completely} if $f_i=e_i=1$ for all $i$. 
     \item The ideal $(p)$ is {\bf inert} if $g_p=1$ and $e_1=1$. 
 \end{enumerate}
\end{definition}

\begin{theorem}
\label{unramified}
The ideal $(p)$ is unramified in $\OO$ if and only if $p \nmid \Disc(K)$. 
\end{theorem}

\begin{exercise}
Prove Theorem \ref{unramified}. ({\em Hint}: Show that a finite-dimensional commutative $\F_p$-algebra is \'{e}tale if and only if its trace form is nondegenerate.) 
\end{exercise}

\begin{example}
\label{Q(i)splitting}
Let $K=\Q(i)$, so $\OO=\Z[i]$ and $\Disc(K)=-4$. (See Exercise \ref{Q(i)}.) Since $(1+i)^2=1+2i-1=2i$, we see that the ideal $(2)\subset \OO$ decomposes as $(2)=(1+i)^2$, and is therefore ramified. By Theorem \ref{unramified}, all other primes are unramified. Let $p \neq 2$ be prime. Then to determine if $(p)$ splits, we notice that 
\[
\Z[i]/(p)=\Z[x]/(x^2+1,p)=\F_p[x]/(x^2+1)=\begin{cases} \F_p \times \F_p &\text{ if } \left( \frac{-1}{p} \right) =1, \\
\F_{p^2} &\text{ if } \left( \frac{-1}{p} \right) =-1.
\end{cases}
\]
Therefore $(p)$ splits completely if and only if $p = 1 \mod 4$. (For example, $(5)=(2+i)(2-i)$), and $(p)$ is inert if and only if $p = 3 \mod 4$. 
\end{example}

We can adapt the strategy of Example \ref{Q(i)splitting} to determine splitting behavior in general. Let $K/\Q$ be a number field, and choose a primitive element (i.e. a generating element) $\theta$ of $K$. We can assume without loss of generality that $\theta \in \OO$. Then $\theta$ satisfies a monic polynomial $f(x) \in \Z[x]$ of degree $n$, so $\Z[\theta]\subset \OO$ is of finite index $m$. Then for $p$ such the $p \nmid m$ and $p \nmid \Disc(K)$, 
\[
\OO/(p)=\Z[\theta]/(p)=\Z[x]/(p,f(x))=\F_p[x]/(f(x))=\F_{p^{f_1}}\times \cdots \times \F_{p^{f_k}},
\]
where $f_i$'s are the degrees of irreducible factors of $f(x)$ modulo $p$. So in particular, 
\begin{align*}
    (p) \text{ splits completely } & \iff f \text{ is reducible modulo }p, \\
    (p) \text{ is inert } &\iff f \text{ is irreducible modulo }p. 
\end{align*}
Hence we could have (and probably should have) phrased last week's lecture in terms of splitting of primes in a ring of integers. 

\subsection{The case of Galois extensions}

Assume $K/\Q$ is a Galois extension, with Galois group $\Gal(K/\Q)=G$. Then $G$ acts on $\OO$, and Frobenius proved the following result. 
\begin{theorem}
If $(p)=\prod\mf{p}_i^{e_i}$, then $G$ acts transitively on the primes $\mf{p}_i$. 
\end{theorem}
Therefore, all $e_i$ (resp. $f_i$) are equal. Denote their common value $e$ (resp. $f$). Exercise \ref{important} ($n=\sum e_i f_i$) implies that $n=efg_p$. Fix a prime $\mf{p}_i=:\mf{p}$ lying over $(p)$. Denote by $G_\mf{p}$ the stabilizer of $\mf{p}$ in $G$ (the ``decomposition group''). Then 
\[
|G/G_\mf{p}|=\# \text{ primes over p }=g_p,
\]
so $|G_\mf{p}|=ef$. Let $I_\mf{p}$ be the subgroup of $G_\mf{p}$ which acts trivially on $\OO_\mf{p}$ (the ``inertia group''). This group has order $e$. We have the following exact sequence 
\[
1 \rightarrow I_\mf{p} \rightarrow G_\mf{p} \twoheadrightarrow \Aut(\OO/\mf{p}) \simeq \Z/f\Z.
\]
Since $\OO/\mf{p} \simeq \F_{p^f}$, the group $\Aut(\OO/\mf{p})$ is generated by $\lFrob:x \mapsto x^p$. If $p$ is unramified (i.e. $p \nmid \Disc(K)$), then $e=1$, so $I_\mf{p}$ is trivial and $G_\mf{p} \simeq \Z/f\Z$. In this case, $\Frob \in G_\mf{p}$ is defined to be the element that maps to $\lFrob \in \Aut(\OO/\mf{p})$. A different choice of $\mf{p}$ lying over $p$ results in conjugate $G_\mf{p}$ and conjugate $\Frob$. This explains rigorously the $\Frob$ from the last lecture. 

%% file: lecture-3.tex
\section{Lecture 3: $L$-functions}
\label{lecture 3}

Last class we reviewed some algebraic number theory. This class we will review some analytic number theory. The motivation for this lecture is the following. We start with an arithmetic problem (for example, counting the number of $x \in \F_p$ such that $x^5+20x+16=0$), and assign to it an $L$-function (something like a ``character'' for the representation theorists in the audience), which can be studied analytically.  
\subsection{The Riemann $\zeta$-function}

Define 
\[
\zeta(s)=\sum_{n \geq 1} n^{-s},
\]
where $s\in \C$ is a complex variable. We can compare this sum to the integral
\[
\int_1^\infty x^{-s}dx, 
\]
which converges to $\frac{1}{s-1}$ for real $s>1$. When viewed as a holomorphic function, the integral converges absolutely for $s \in \C$ such that $\Re(s)>1$. Hence the sum $\zeta(s)$ converges for all $s \in \C$ such that $\Re(s)>1$. 

This function has a rich history. Euler computed special values (e.g. $\zeta(2)=\frac{\pi^2}{2}$), and noticed that the $\zeta$-function may also be given as the Euler product:
\[
\sum_{n \geq 1} n^{-s}=(1 + 2^{-s} + (2^2)^{-s} + \cdots)(1 + 3^{-s} + (3^2)^{-s}+ \cdots) \cdots = \prod_{p \text{ prime}} \left( \frac{1}{1-p^s} \right). 
\]
This product relates an analytic object, $\zeta(s)$, to the prime numbers. {\em This relationship lets us study properties of primes using analysis!} For example, the Euler product immediately gives us two proofs of the infinitude of primes: (1) the divergence of $\sum_{n \geq 1} \frac{1}{n}$ implies the product on the right hand side must be infinite, and (2) since $\zeta(2)=\frac{\pi^2}{2}$, the irrationality of $\pi^2$ also implies that the right hand product must be infinite.  

Riemann showed that $\zeta$ admits a meromorphic continuation to all of $\C$. This is the {\bf Riemann zeta function}. He also showed that $\zeta(s)$ has a simple pole at $s=1$ (Exercise: Show that $\zeta(s)-\frac{1}{1-s}$ converges for $\Re(s)>0$), and ``trivial zeros'' at $-2, -4, \ldots$. Furthermore, he established the {\bf functional equation}:  $\Lambda(s):=\pi^{-s/2}\Gamma(\frac{s}{2})\zeta(s)$ satisfies  
\[
\Lambda(s)=\Lambda(1-s). 
\]
Here $\Gamma(s)=\int_o^\infty x^{s-1}e^{-x}dx$ is the gamma function. And finally, he proposed the following conjecture, which eventually became a millenium question. 

\vspace{5mm}
\noindent
{\bf Riemann hypothesis:} If $z$ is a non-trivial zero of $\zeta(z)$, then $\Re(z)=\frac{1}{2}$. 

\vspace{5mm}
\noindent
The Riemann hypothesis is known to be true for the first $10^{12}$ zeros of $\zeta(s)$.

\subsection{Why do we care?}

Here's the slogan of this story: ``The zeros of the Riemann zeta function are the Fourier modes of the primes''. We will spend the rest of the lecture trying to make this precise. 

One of Riemann's motivations was the following theorem, which was a conjecture during his lifetime. 
\begin{theorem}
({\bf Prime Number Theorem}) Let $\pi(x)$ be the number of primes less than or equal to $x\in \R$. Then 
\[
\pi(x) \sim \frac{x}{\log{x}}.
\]
\end{theorem}
\noindent
Here $\sim$ means that $\lim_{x \rightarrow \infty}\frac{\pi(x)}{x/\log{x}}=1$. Riemann was interested in two questions about the Prime Number Theorem:
\begin{itemize}
    \item Why does it hold?
    \item What is the error term?
\end{itemize}
Instead of considering $\pi(x)$ directly, we can examine the {\bf von Mangoldt function}, which ``makes noise at prime powers'': 
\[
\Lambda_{vm}(n)=\begin{cases} \log{p} &\text{ if } n=p^m, \\
0 &\text{ otherwise}. 
\end{cases}
\]
Define 
\[
\psi(n)=\sum_{m\leq n} \Lambda_{vm}(m).
\]
\begin{exercise}
Show that the Prime Number Theorem is equivalent to $\psi(n) \sim n$. 
\end{exercise}

Riemann discovered an explicit formula for $\psi(x)$ at non-integers. 

\begin{theorem}
(Riemann) For $x \in \R-\Z$, there is equality
\[
\psi(x)=x-\sum_\rho \frac{x^\rho}{\rho} - \log(2 \pi),
\]
where the sum is taken over all zeros $\rho$ of the Riemann $\zeta$-function. 
\end{theorem}
In other words,
\[
x-\psi(x)=\log(2\pi) + \sum_\rho \frac{x^\rho}{\rho},
\]
so {\em the zeros of the Riemann $\zeta$-function measure the error term in the prime number theorem.} We can examine what effect the different types of zeros have on the right hand side of the equality above. 
\begin{itemize}
    \item {\bf Trivial zeros}: The function $\frac{x^{-2n}}{-2n}=-\frac{1}{2nx^n}$ decays very quickly, so for large $x$, trivial zeros have almost no effect on the formula. 
    \item {\bf A pair $\rho$, $\bar{\rho}$ of non-trivial conjugate zeros}: Each such pair contribues 
    \[
    \lambda \cdot x^{\Re(\rho)}\cdot \cos(\gamma + \log x),
    \]
    where $\lambda$ and $\gamma$ depend in a simple way on $\rho$ and $\bar{\rho}$. So $\Re(\rho)$ is crucial in the contribution of $\rho$ and $\bar{\rho}$ to the error term, and if the Riemann hypothesis is true, the growth of this contribution looks roughly like the product of $\cos(x)$ and $x^{1/2}$. Also, as $\Im(\rho)$ gets bigger, $\lambda$ gets smaller. Thus, if the Riemann hypothesis is true, small zeros will contribute larger variations. A counterexample to the Riemann hypothesis would cause unexpectedly large fluctuations in $x-\psi(x)$.
\end{itemize}
\begin{exercise}
 \begin{enumerate}[label=(\alph*)]
 \item Show that the Prime Number Theorem is equivalent to $\zeta(s)$ having no zeros $z$ with $\Re(z)=1$. 
 \item Show that the Riemann hypothesis is equivalent to $x-\psi(x) \in O(x^{1/2})$.
 \item Find the error term in $\pi(x)-\frac{x}{\log{x}}$ assuming the Riemann hypothesis. 
 \end{enumerate}
\end{exercise}
 
\begin{remark}
 It is unknown whether there exist non-trivial zeros $\zeta$ of $\zeta(s)$ with $\Re(z)=1-\epsilon$ for {\em any} $\epsilon>0$. 
\end{remark} 

\subsection{Dirichlet $L$-functions}
It is natural to ask questions about primes satisfying certain properties. For example, fix $m \in \Z_{\geq 0}$, and $a \in \left( \Z/m\Z \right)^\times$. Consider the set 
\[
\{p \text{ prime } | p = a \mod m\}. 
\]
Is this set infinite? Is there an analogue to the Prime Number Theorem in this setting? A naive attempt to show that this set is infinite would be to consider the product
\[
\prod_{p} \frac{1}{1-p^{-s}}
\]
 taken over all primes $p$ such that $p = a \mod m$ and recreate one of the arguments for the infinitude of primes given in the previous section. However, there is no Euler product in this setting, so this approach fails. Another approach is representation theory. 
 
 \vspace{5mm}
 \noindent
 {\bf What do we learn from representation theory?}
 
 \vspace{2mm}
 \noindent 
 Consider the set of all functions from
 \[
 (\Z/m\Z)^\times \rightarrow \C.
 \]
 This set has two natural bases:
 \begin{enumerate}
     \item Indicator functions: $x \mapsto \delta_{x,a}$ for $a \in \left( \Z/m\Z \right) ^\times$.
     \item Irreducible characters: $x \mapsto e^{2 \pi ij/\phi(m)}$ for $j=0,\ldots, \phi(m)-1$, where $\phi(m)=\left| \left( \Z/m\Z \right) \right|.$
 \end{enumerate}
 In many ways, the basis of characters is more natural. Dirichlet borrowed this idea of characters to adapt our naive attempt above into something that works. 
 
 \begin{definition}
 A {\bf Dirichlet character modulo m} is a character $\chi: \left( \Z/m\Z\right)^\times \rightarrow \C^\times$ extended by zero to all of $\Z$. In other words, it is a function $\chi: \Z \rightarrow \C^\times$ such that $\chi(n)=0$ if $(n,m)>1$, $\chi(n)$ depends only on $n \mod m$, and $\chi$ induces a character $\chi: \left( \Z/m\Z\right)^\times \rightarrow \C^\times$.
 \end{definition}
 
 Such a function $\chi$ has the property that $\chi(nn')=\chi(n)\chi(n')$. Using these characters, Dirichlet defined a {\bf Dirichlet $L$-function}:
 \[
 L(\chi, s)=\sum_{n \geq 1} \chi(n)n^{-s} = \prod_{p \text{ prime}} \frac{1}{1-\chi(p)p^{-s}}.
 \]
 With this adjustment by a character, the sum {\em does} admit an Euler product, as well as a meromorphic extension to all of $\C$, and a (rather complicated) functional equation. If $\chi$ is trivial, we recover the Riemann $\zeta$-function. If $\chi$ is not trivial, then $L(\chi, s)$ is entire. 
 
 Dirichlet's theorem (that $\{p | p = a \mod m\}$ is infinite, with distribution $\frac{1}{\phi(m)}\frac{\pi}{\log{x}}$) is an easy consequence of the fact that $L(\chi, 1) \neq 0$ if $\chi$ is not trivial. Furthermore, there is an analogue of the Riemann hypothesis in this setting, usually called the ``generalized Riemann hypothesis,'' and all non-trivial zeros of $L(\chi,s)$ lie on the critical line (i.e. satisfy $\Re(z)=\frac{1}{2}$). 
 
 \subsection{Dedekind $\zeta$-functions}
 Another natural question of this flavor is how primes behave in $\Z[i]$, or other rings of integers. To answer this question, Dedekind introduced his version of a $\zeta$-function. 
 
 Return to the setting of last week: Let $K/\Q$ be a number field with ring of integers $\OO$. For a nontrivial fractional ideal $I \subset \OO$, let $\NN I=\#\OO/I$. (For example, if $K=\Q$, $\NN (p)=p$.) The {\bf Dedekind $\zeta$-function} is \[
 \zeta_K(s) = \sum_{0 \neq I \subset \OO} (\NN I) ^{-s} = \prod_{\mf{p} \subset \OO \text{ prime}} \frac{1}{1-\NN \mf{p}^{-s}}.
 \]
 This sum admits an Euler product because of the uniqueness of factorization of ideals in $\OO$. Again, we have a meromorphic continuation and functional equation in this setting. (Historical note: The functional equation first appeared in Hecke's thesis in the 1920's, but the proof was very complicated in Hecke's work. A much simpler proof was given by Tate in his thesis in 1950.) The key example is the following. 
 
 \begin{example}
 \label{factorization}
 Let $K=\Q(i) \supset \OO=\Z[i]$. Then 
 \begin{align*}
     \zeta_K(s)&=\prod_{\mf{p}\subset\OO \text{ prime}} \frac{1}{1-\NN \mf{p}^{-s}} \\
     &=\left(\frac{1}{1-2^{-s}}\right) \prod_{\mf{p} \text{ s.t. } (p) \text{ splits}}\left( \frac{1}{1-\NN \mf{p}^{-s}}\right) \left( \frac{1}{1-\NN \mf{p}^{-s}}\right) \prod_{\mf{p} \text{ s.t. } (p) \text{ is inert}} \left( \frac{1}{1-\NN \mf{p}^{-2s}} \right)   \\
     &= \left( \frac{1}{1-2^{-s}} \right) \prod_{p = 1 \mod 4} \left( \frac{1}{1-p^{-s}}\right)^2 \prod_{p=3 \mod 4} \left( \frac{1}{1-p^{-s}}\right) \left( \frac{1}{1+p^{-s}}\right) \\
     &= \prod_p \left( \frac{1}{1-p^{-s}} \right) \prod_{p \neq 2} \left( \frac{1}{1-\chi(p)p^{-s}}\right)\\
     &=\zeta(s)L(\chi,s),
 \end{align*}
 where 
 \[
 \chi(p)=\begin{cases} 1 & \text{ if } p = 1 \mod 4, \\
 -1 & \text{ if } p = 3 \mod 4 
 \end{cases}
 \]
 is a Dirichlet character on $\Z/4\Z$. The moral of this example is that {\em we can understand $\Z(i)$ in terms of $\Z$!} We are witnessing the beginnings of {\bf class field theory}.  
 \end{example}
 
 \subsection{Walking across the bridge}
 Next we will cross our bridge of analogy and see what happens in the geometric world. Recall that the objects analogous to a ring of integers $\OO$ contained in a number field $K$ are smooth projective curves over $\F_p$ and their function fields over $\F_p$. 
 
 \begin{example}
 The simplest example of such a smooth projective curve is $\PP^1_{\F_p}$. Instead we will work with $\mathbb{A}^1\F_p$, where the analogue of a ring of integers is $\mathcal{O}(\mathbb{A}^1\F_p) = \F_p[x]$. Here, 
 \begin{align*}
     \zeta_{\mathbb{A}^1\F_p}(s)&=\sum_{o \neq I \subset \F_p[x]} (\NN I)^{-s} \\
     &= \sum_{f \text{ monic}} (\NN(f))^{-s} \\ 
     &= \sum_{d \geq 1} \left( \sum_{f \text{ monic degree }d} (p^d)^{-s} \right) \\
     &=\sum_{d \geq 0} p^d (p^d)^{-s} \\ 
     &= \sum_{d \geq 0} (p^{-s + 1})^d \\
     &= \frac{1}{1-p^{-s+1}}.
 \end{align*}
 So $\zeta_{\mathbb{A}^1_{\F_p}}(s)$ is a rational function with a unique pole at $s=1$ and {\em no zeros}. Since $\zeta_{\mathbb{A}^1_{\F_p}}(s)$ measures the error term of the prime number theorem, this means that we can count primes exactly in this setting! In fact, we have (Gauss)
 \[
 \# \text{ irred. polys of degree $d$} = \frac{1}{d} \sum_{m|d}\mu\left(\frac{d}{m}\right) p^m. 
 \]
 Here $\mu$ is the Mobius function (i.e. $\mu(n)$ is the sum of the primitive $n^{th}$ roots of unity). So in this setting, we know exactly how many primes there are in a given interval.  
 \end{example}
 
 Our example was a little too simple, so we should bump it up a notch. In Artin's thesis (1923), he instead considered $X=\F_p[x,y]/(y^2-f(x))$ for $f(x)$ square free. (Under our analogy, this is the analogue of a quadratic field $\Q(\sqrt{a})$ for $a$ square free.) Artin showed that 
 \[
 \zeta_X(s)=\frac{(1-\alpha p^s)(1-\bar{\alpha}p^s)}{1-p^{-s+1}}
 \]
 is again a rational function, with $\alpha \in \C$ of norm $p^{1/2}$. {\em Hence all zeros of $\zeta_X(s)$ have $\Re{z}=\frac{1}{2}$, and the Riemann hypothesis is true in this setting.} However, unlike the zeros of $\zeta(s)$, the zeros of $\zeta_X(s)$ are distributed evenly along the critical line. 
 
 The analogue to this statement for all curves was proven by Weil, and the case of arbitrary varieties was completed by Deligne ($\sim$1970). These accomplishments were some of the crowning glories of 20$^{th}$ century mathematics. 
 
 \subsection{Artin $L$-functions}
 Now we enter the non-abelian world. The year is 1927, about 31 years after Frobenius started developing the theory of group characters. Let $K/\Q$ be a Galois extension, with $G=\Gal(K/\Q)$, and $d_K=\Disc(K)$. Recall from previous lectures that:
 \begin{itemize}
     \item $p$ is unramified if and only if $p \nmid d_K$, for unramified $p$,
     \item a choice of prime $\mf{p}$ over $p$ results in an element $\Frob \in G$, and
     \item different choices of $\mf{p}$ lead to conjugate $\Frob$. 
 \end{itemize} 
 Fix a finite dimensional complex representation 
 \[
 \rho:G \rightarrow \GL(V)
 \]
 of the Galois group. Notice that the conjugacy class of $\rho(\Frob)$ is completely determined by its characteristic polynomial $\det(1-t\Frob|_V)$. The (first approximation) of an {\bf Artin $L$-function} is 
 \[
 L_{ur}(V,s)=\prod_{p \nmid d_K} \frac{1}{\det(1-p^{-s}\Frob|_V)}.
 \]
 \begin{example}
 \begin{enumerate}[label=(\alph*)]
     \item $V$ is the trivial representation. Then 
     \[
     L(V,s)=\prod_{p \nmid d_K} \frac{1}{1-p^{-s}} = \zeta(s) \text{ up to finitely many factors.}
     \]
     So the Artin $L$-function recovers the Riemann $\zeta$-function as a special case. 
     \item Let $K=\Z(\sqrt{\alpha})/\Q$ be a quadratic field, $G=\{\pm 1\}$, and $\rho:G \rightarrow \GL_1(C)$ be the identity map $\{\pm1\} \mapsto \{\pm1\}$. Then for $p \nmid d_K$, 
     \[
     \rho(\Frob) = \begin{cases} 1 &\text{ if } p \text{ splits } \left(\iff \left( \frac{p}{\alpha}\right)=1\right),\\
     -1 & \text{ if } p \text{ is inert } \left(\iff \left( \frac{p}{\alpha}\right)=-1\right). \end{cases}
     \]
     Therefore, by quadratic reciprocity, 
     \begin{align*}
     L_{ur}(\rho,s)&=\prod_p \left( 1-\left( \frac{p}{\alpha}\right)p^{-s}\right)^{-1} \\
     &= \prod_p (1 - \chi(p)p^{-s})^{-1} \\
     &=L(\chi,s) \text{ (up to finitely many factors)}
     \end{align*}
     for some Dirichlet character $\chi:\Z/4\alpha\Z \rightarrow \C^\times$. So Artin L-functions also recover Dirichlet $L$-functions. 
 \end{enumerate}
  \end{example}

 \begin{exercise}
 \begin{enumerate}[label=(\alph*)]
     \item Show that $L_{ur}(V_1 \oplus V_2,s) =L_{ur}(V_1, s) L_{ur}(V_2,s)$. 
     \item Show that $L(V,s)$ converges absolutely for $\Re s>1$. ({\em Hint}: Compare it to a product of $\dim V$ copies of the $\zeta$-function.)  
     \item (Beautiful! Do it!) For a Galois extension $K/\Q$, let $V_{reg}$ be the regular representation of $G$. Show that $L_{ur}(V_{reg},s)=\zeta_K(s)$ up to finitely many factors. 
 \end{enumerate} 
 \end{exercise}

\noindent 
A consequence of the exercise is the {\bf Artin decomposition}:
\[
\zeta_K(s)=\prod_{V \text{ irrep of }G} L_{ur}(V,s)^{\dim V} \text{ (up to finitely many factors).}
\]
Hence $\zeta_\Q(s)$ always divides $\zeta_K(s)$! This is a generalization of the factorization $\zeta_{\Q(i)}(s)=\zeta(s)L(\chi,s)$ that we saw in Example \ref{factorization}. 

\vspace{5mm}
\noindent
{\bf The Moral:} All of the difficulty in $K$ is contained in the difficulty of $\Q$ if we allow for ``non-abelian'' difficulty (i.e. general representations of the Galois group).

\begin{remark}
 We can get rid of the ``up to finitely many factors'' caveat in all of these statements by modifying $L_{ur}(V,s)$ with ``general local factors'' at ramified primes. See Geordie's hand-written lecture notes or Gus's lecture for a description of this procedure. 
\end{remark}

%% file: lecture-4.tex
\section{Lecture 4: The Sato-Tate conjecture}
\label{sato-tate}

This will be our final lecture of global motivation before we zoom in on the local story. The goal of today's lecture is to describe the {\bf Sato-Tate conjecture} and its relationship to the global Langlands picture. We will see that seemingly innocuous statements in the Langlands correspondence can have very powerful repercussions. 

\subsection{Equidistribution in representation theory}

We say that the real numbers $\alpha_1, \alpha_2, \ldots \in [0,1]$ are {\bf equidistributed} if 
\[
\lim_{n \rightarrow \infty} \frac{1}{n} \# \{ \alpha_i \mid \alpha_i \in (a,b) \text{ for } i=1, \ldots, n \} = b-a = \int_0^1 1_{(a,b)} dx
\]
for any interval $(a,b) \subset [0,1]$. Here $1_{(a,b)}$ is the indicator function on $(a,b)$. Because indicator functions are dense in complex-valued Riemann integrable functions on $[0,1]$, this condition is equivalent to saying that the discrete average of any function on this set and continuous average of the same function agree; that is,  
\[
\lim_{n \rightarrow \infty} \frac{1}{n} \sum_{i=1}^n f(\alpha_i) = \int_0^1 f(x) dx
\]
for all Riemann integrable $f:[0,1]\rightarrow \C$. Now we can approximate any Riemann integrable $f:[0,1]\rightarrow \C$ with a Fourier series, so this is in turn equivalent to 
\[
\lim_{n \rightarrow \infty} \frac{1}{n} \sum_{i=1}^n e(m\alpha_i) = \int_0^1 e(mx)dx = \begin{cases} 1 & \text{ if } m=0, \\ 0 & \text{ otherwise} 
\end{cases}
\]
for all $m \in \Z$. Here $e(x)=\exp(2 \pi i x)$. The condition above always holds for $m=0$, so to check if $\alpha_1, \alpha_2, \ldots$ are equidistributed, it suffices to check that for $m \neq 0$, 
\[
\lim_{n \rightarrow \infty} \sum_{i=1}^n e(m \alpha_i) = 0. 
\]
Here is an application of this observation. Let $\xi \in \R$ be irrational. Consider $ (\xi), (2 \xi), (3\xi), \ldots $ where $(m \xi):= m\xi \mod 1$. Are these numbers equidistributed? Weyl used the observation above to show that they are. Choose $m \neq 0$, and let $\eta = m \xi$. Then
\[
\left| \frac{1}{n} \sum_{j=1}^n e(mj\xi) \right| = \left| \frac{1}{n} (e(\eta) + e(2\eta) + \cdots e(n \eta) ) \right| = \left| \frac{1}{n} \frac{e((n+1)\eta) - e(\eta)}{e(\eta) -1} \right| \leq \frac{1}{n} \left| \frac{2}{e(\eta) - 1} \right| \rightarrow 0 
\]
as $n \rightarrow \infty$. 

\begin{remark} This leads to many questions of a similar flavor (e.g. what about $(\xi), (4\xi), (9\xi), \ldots $?) Weyl's paper \cite{Weyl} gives an affirmative answer for polynomials $f(x) \in n\R[n]$ whose coefficients are not all rational!
\end{remark}

We can reinterpret equidistribution via representation theory. The above reasoning shows (after identifying integers mod $1$ with $S^1$) that a sequence $z_1, z_2, \ldots  \in S^1\subset \C$ {\bf equidistributes} if for every nontrivial rational character $\chi$ of $S^1$ (e.g. $\chi: z \mapsto z^m$ for $m \neq 0$), 
\[
\lim_{n \rightarrow \infty} \frac{1}{n} \sum_{i=1}^n \chi(z_i) = 0. 
\]
If this condition holds, we say that the sum is ``little o of $n$,'' and write 
\[
\sum_{i=1}^n \chi(z_i) = o(n).
\]

Now suppose $G$ is a compact group with Haar measure $\mu$, and let $X$ denote the space of conjugacy classes of $G$. (Recall that the {\bf Haar measure} is the unique left (and then necessarily also right) $G$-invariant measure on $G$ with $\mu(G)=1$.) Let $C(X)$ be the Banach space of continuous complex-valued functions on $X$. Two properties of irreducible characters on compact groups are the following:
\begin{theorem} ({\bf The Peter-Weyl Theorem}) The irreducible characters span a dense subspace of $C(X)$. 
\end{theorem}
\begin{theorem} ({\bf Orthogonality of characters}) If $\chi, \chi'$ are irreducible characters, then 
\[
\int_G \chi(g) \overline{\chi'(g)} d\mu = \begin{cases} 1 &\text{ if } \chi = \chi', \\
0 &\text{ otherwise}.
\end{cases}
\]
\end{theorem}

Hence, we have the following theorem about equidistribution of sequences in $G$. 
\begin{theorem}
A sequence $\alpha_1, \alpha_2, \ldots \in X$ is equidistributed with respect to the (push forward of the) Haar measure if and only if 
\[
\sum_{i=1}^n \chi(\alpha_i) = o(n) 
\]
for all irreducible nontrivial characters $\chi$. 
\end{theorem}

\begin{example}
Let $G=S_3$. The conjugacy classes are determined by cycle type: $C_1=\{id\}$, $C_2=\{(12),(23),(13)\}$, $C_3=\{(123),(132)\}$. The character table of $S_3$ is the following. 
\[
\begin{tabular}{c|c|c|c|c|c}
    \text{Haar measure} & $\#C_i$ & \text{conj. class} & \text{triv} & \text{nat} & \text{sgn}  \\
    \hline 
     $1/6$ & 1 & $C_1$ & 1 & 2 & 1\\
     $1/2$ & 3 & $C_2$ & 1 & 0 & -1 \\
     $1/3$ & 2 & $C_3$ & 1 & -1 & 1 
\end{tabular}
\]
We can see from this table that for a sequence to be equidistributed, it should spend twice the time in $C_3$ as it does in $C_1$ (as dictated by the column corresponding to the natural representation), and should spend as much time in $C_2$ as in $C_1 \cup C_3$ (as dictated by the column corresponding to the sign representation).
\end{example}

\begin{example}
Let $G=SU(2)$. Since all matrices in $SU(2)$ are diagonalizable, conjugacy classes are determined by eigenvalues, which come in conjugate pairs. So $X=\{(\gamma, \bar{\gamma})| \gamma \in S_1\}$ can be identified with $S^1_+=\{z \in S^1 | \Re(z)\geq 0 \} \simeq [0,\pi]$. By the Weyl character formula for $SU(2)$, irreducible characters are of the form
\[
\chi_m(z)=z^m+z^{m-2}+\cdots + z^{-m},
\]
where $z=e^{i\theta } \in S^1$. (Here we are using the identification $X\simeq S^1_+$ when defining these characters.) The Haar measure on the space of conjugacy classes (after identifying $S^1$ with $[0,2\pi]$ via the exponential function) is then 
\[
\frac{2}{\pi} \sin^2\theta d \theta.
\]
\begin{exercise}
Check that this is correct by showing that \[
\frac{2}{\pi}\int_0^\pi \sin^2\theta d \theta = 1
\]
and 
\[
\int_0^\pi \chi_m(e^{i \theta}) \sin^2 \theta d \theta =0
\]
for $m \neq 0$.
\end{exercise}
The figure below is a plot of $\frac{2}{\pi} \sin^2 \theta$. From this we see the distribution of eigenvalues of matrices in $SU(2)$. A first observation is that there are many matrices in $SU(2)$ with eigenvalues $\{i, -i\}$ (corresponding to $\theta = \frac{\pi}{2}$), and very few matrices with eigenvalues $\{1,1\}$ and $\{-1, -1\}$ (corresponding to $\theta = 0$ and $\theta = \pi$, respectively). In fact, there is exactly one matrix in each case: $I$ and $-I$. 

\begin{remark}
This demonstrates that it is more likely for matrices in $SU(2)$ to have eigenvalues which are ``far away'' (meaning that the angle between them in the complex plane is large), an important fact in random matrix theory.
\end{remark}

\begin{center}
 \includegraphics[scale=0.5]{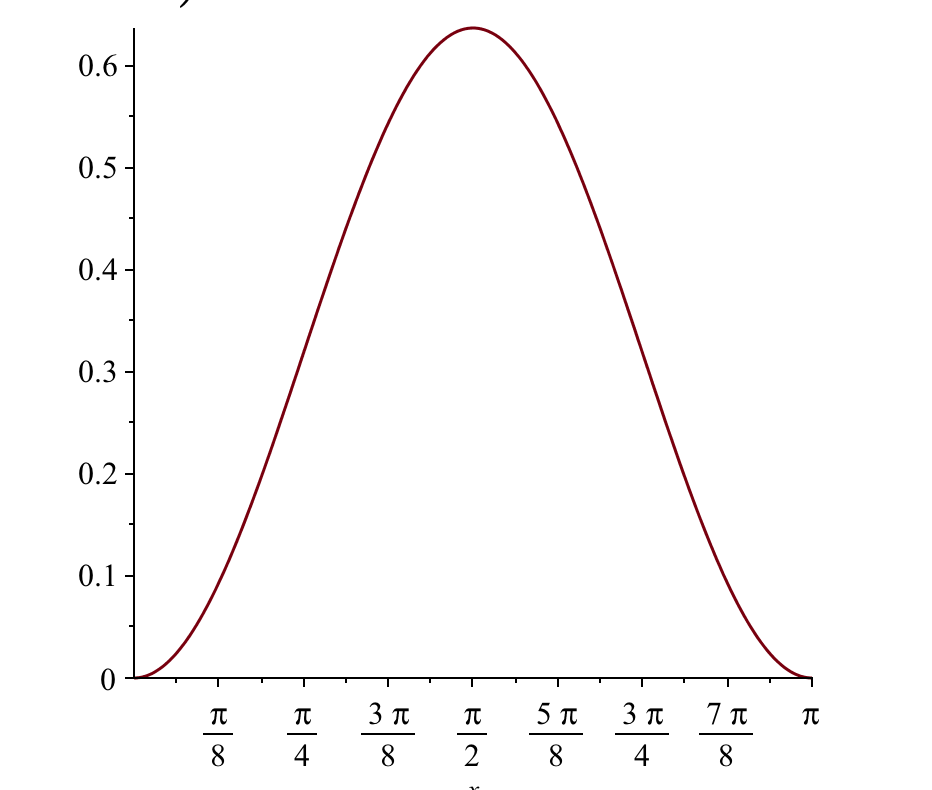}
\end{center}
Hence if you had a sequence of pairs of complex numbers which you suspect are eigenvalues of random matrices in $SU(2)$, you could tell pretty quickly whether or not it was plausible. 
\end{example}

\subsection{Elliptic curves and the Sato-Tate conjecture}

Next we discuss elliptic curves. (This seems to be completely unrelated, but we should have faith that it will come full circle.) Let $k$ be a field, and $E$ an elliptic curve. (That is, a smooth projective curve over $k$ of genus $1$ with a fixed rational point $0 \in E(k)$.) Assume the characteristic of $k$ is not $2$ or $3$. Then $E$ can be made to be of the form (the projective closure of)
\[
y^2=x^3+ax+b, \text{ with } 0=(0:0:1) \text{ its point at infinity}.
\]
Here smoothness translates into the fact that $x^3+ax+b$ has no repeated roots, i.e. that $4a^3-27b^2\neq 0$. 

First assume we are working over $\C$. Then $E$ is a compact Riemann surface of genus $g =1$, so $E=\C/\Lambda$ for some lattice $\Lambda \simeq \Z^2 \hookrightarrow \C$. Hence
\begin{align*}
\left\{ \begin{array}{c}\text{elliptic curves}
\\
\text{over $\C$} 
\end{array} \right\}_{/\text{ iso}}  &\simeq \left\{ \begin{array}{c} \text{lattices} 
\\
\text{ $\Lambda \subset \C$}
\end{array}   \right\}_{/\text{ iso}}\xrightarrow{\text{``period ratio''}} \prescript{}{\SL_2(\Z)}{\backslash \mathbb{H}},
\end{align*}
where $\mathbb{H}$ is the upper half-plane. The quotient $\prescript{}{\SL_2(\Z)}{\backslash \mathbb{H}}$ can be identified (up to some ambiguity on the boundary) with its fundamental domain
\begin{center}
 \includegraphics[scale=0.45]{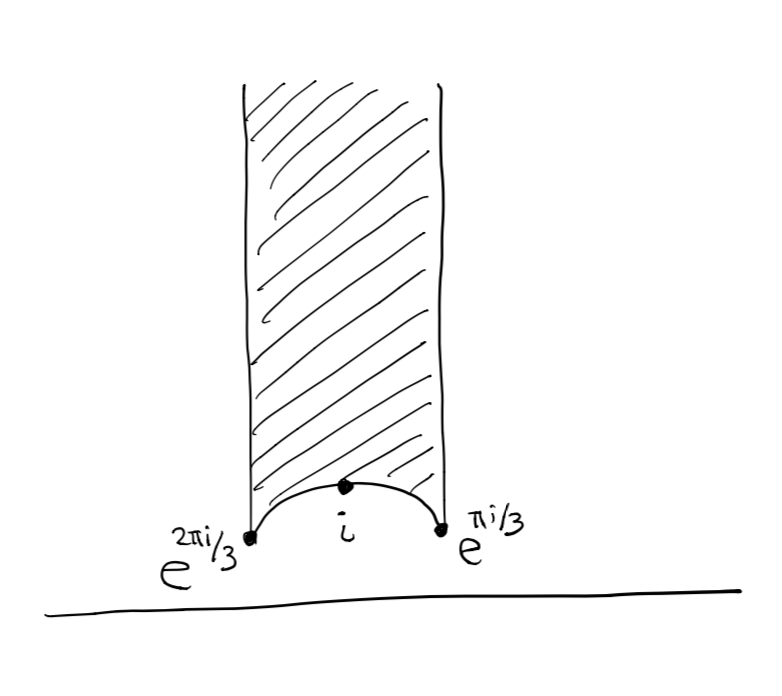}
\end{center}
so we can consider all complex elliptic curves as points in the region above. 

Certain elliptic curves have extra structure called complex multiplication. Let $E$ be a complex elliptic curve. As a real Lie group, $E$ is isomorphic to $S^1 \times S^1$, hence 
\begin{align*}
\End_{\text{Lie gp}}(E)&=\Z^2\\
(z,w)\mapsto (z^m, w^n) & \leftrightarrow (m,n)
\end{align*}
But $E$ has additional structure - it is an elliptic curve ($E \simeq \C/\Lambda$), and a complex algebraic group. By a miracle of abelian varieties, the elliptic curve endomorphisms of $E$ fixing $0$ are the same as the complex algebraic group endomorphisms of $E$. Hence,
\[
\End_{0 \mapsto 0}(E) = \End_{\text{alg gp}}(E) = \{z \mid z \Lambda \subset \Lambda \} = \begin{cases} \Z &\text{ if } \Lambda \cdot \Lambda \not\subset \Lambda, \\ 
\Lambda & \text{ if } \Lambda \cdot \Lambda \subset \Lambda. \end{cases}
\]
In the second case, we say $E$ has {\bf complex multiplication}. In the fundamental domain, the elliptic curves with complex multiplication are of the form $i \sqrt{d}$ (the ``corner point'' in the picture above):
\begin{center}
 \includegraphics[scale=0.5]{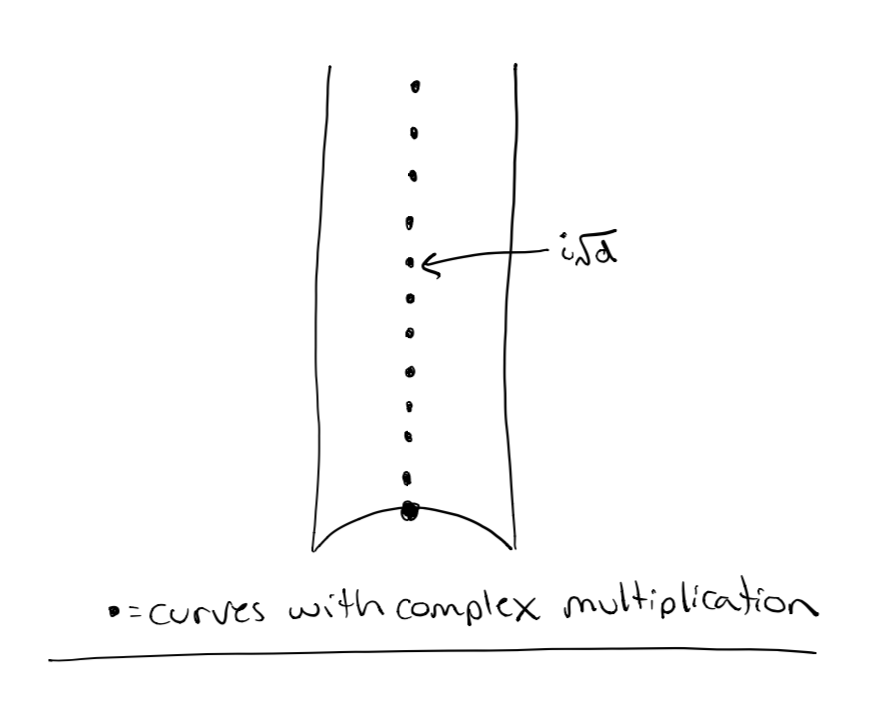}
\end{center}
Thus one can think of curves with complex multiplication as being very special. 

\begin{exercise} Show that if $\Lambda \cdot \Lambda \subset \Lambda$ (i.e. $E$ has complex multiplication), $\Lambda \otimes \Q$ is an imaginary quadratic field. In particular, if $E$ has complex multiplication then $\Lambda$ is what is called an {\em order} in an imaginary quadratic field.
\end{exercise}

Next we work over $\Q$, and consider the elliptic curve $E$ given by $y^2=x^3+ax+b$ for $a,b \in \Z$. To understand $E$, we reduce modulo $p$ and study the number of points of $E(\mathbb{F}_p)$. If $E_{\F_p}=E \times \Spec \F_p$ is nonsingular, then $p$ is a prime of {\bf good reduction}. (This is the analogue for algebraic varieties of a prime being unramified.) We can bound $\#E(\F_p)$ by the  {\bf Hasse-Weil bound}:
\begin{theorem} (Hasse-Weil) 
\[
\#E(\F_p) = 1 + p - \alpha_p,
\]
where $|\alpha_p|\leq 2 \sqrt{p}$. 
\end{theorem}
The Hasse-Weil bound tells us that $1+p$ is a ``square root accurate'' approximation for $\#E(\F_p)$. By our discussions last week, hopefully you are convinced that this is an analogue of the Riemann hypothesis for elliptic curves. 

\vspace{5mm}
\noindent
{\bf Big question:} How do the $\alpha_p$'s behave? 
\vspace{5mm}

We can start to answer this question using the {\bf Grothendieck--Lefshetz trace formula.} Let $H^*(E)$ be the \'{e}tale cohomology of $E$. For all $p$ outside a finite set, we have an action of $\Frob$ on (\'{e}tale, but don't worry if you don't know what this is) cohomology: 
\[
\begin{array}{cccc}
    \dim: & 1 & 2 & 1 \\
    & H^0(E) & H^1(E) & H^2(E) \\
    &\acts & \acts & \acts \\
    &\Frob = 1 & \Frob \sim \bp \gamma_p & 0 \\ 0 & \overline{\gamma_p} \ep & \Frob=p  
\end{array}
\]
The Grothendieck-Lefshetz trace formula is 
\[
\#E(\F_p) = \sum(-1)^i \Tr(\Frob:H^i) = 1+p - (\gamma_p + \overline{\gamma_p}),
\]
where $\gamma_p, \overline{\gamma_p}$ are the eigenvalues of Frobenius on $H^1(E)$. Hence to determine the number of solutions of $E(\F_p)$, it is enough to examine $\gamma_p$, and it is true (but not so easy to see) that the Riemann hypothesis for $E$ is equivalent to $|\gamma_p|=\frac{1}{2}$. 

We can renormalize so that
\[
\theta_p := \frac{1}{\sqrt{p}}\gamma_p \in S_+^1.
\]
This leaves us with a sequence $\theta_2, \theta_3, \theta_5, \ldots$ of points on a semicircle $S^1_+$ controlling the number of points of $E$ modulo $p$. 
\begin{center}
 \includegraphics[scale=0.5]{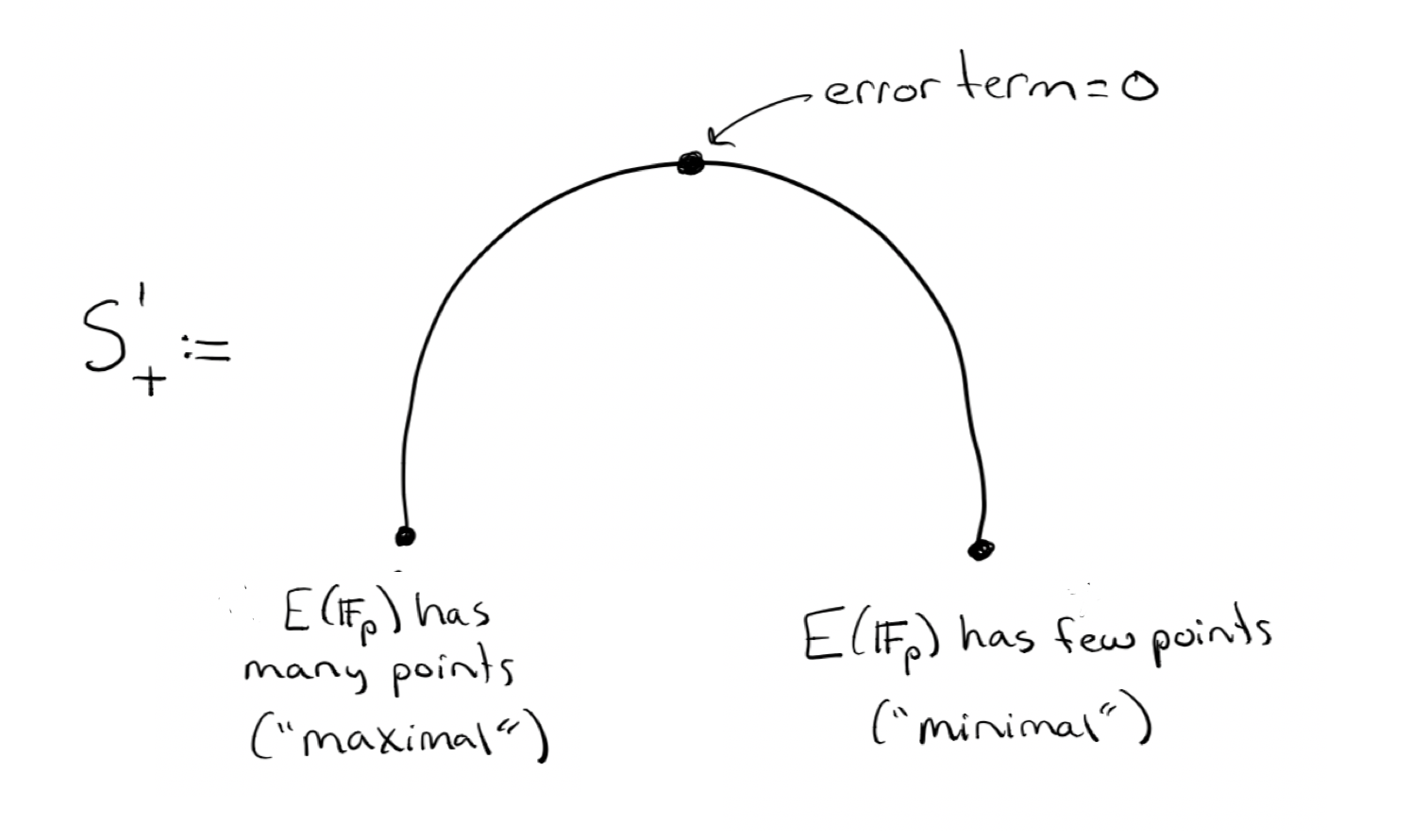}
\end{center}

\vspace{5mm}
\noindent
{\bf Sato--Tate conjecture (1960):} Suppose $E$ does not have complex multiplication. If we identify $S^1_+ \xrightarrow{\simeq}[0,\pi]$, then 
\[
\lim_{n \rightarrow \infty}\frac{1}{\pi(n)}\sum_{p \leq n} \mu_{\theta_p} = \frac{2}{\pi} \sin^2 \theta d \theta.
\]
Here $\mu_{\theta_p}$ is the Dirac distribution. 

\vspace{5mm}
{\em In other words, the Sato--Tate conjecture is that the $\theta_p$'s look like eigenvalues of random matrices in $SU(2)$!} Below is a very beautiful illustration of this phenomenon. For the elliptic curve $y^2 + y = x^3+x+3x+5$ (which has no complex multiplication), the following is the plot of $\theta_p$ for the first $5,000$ primes. (The further out the dot, the bigger the prime.) 
\begin{center}
     \includegraphics[scale=0.4]{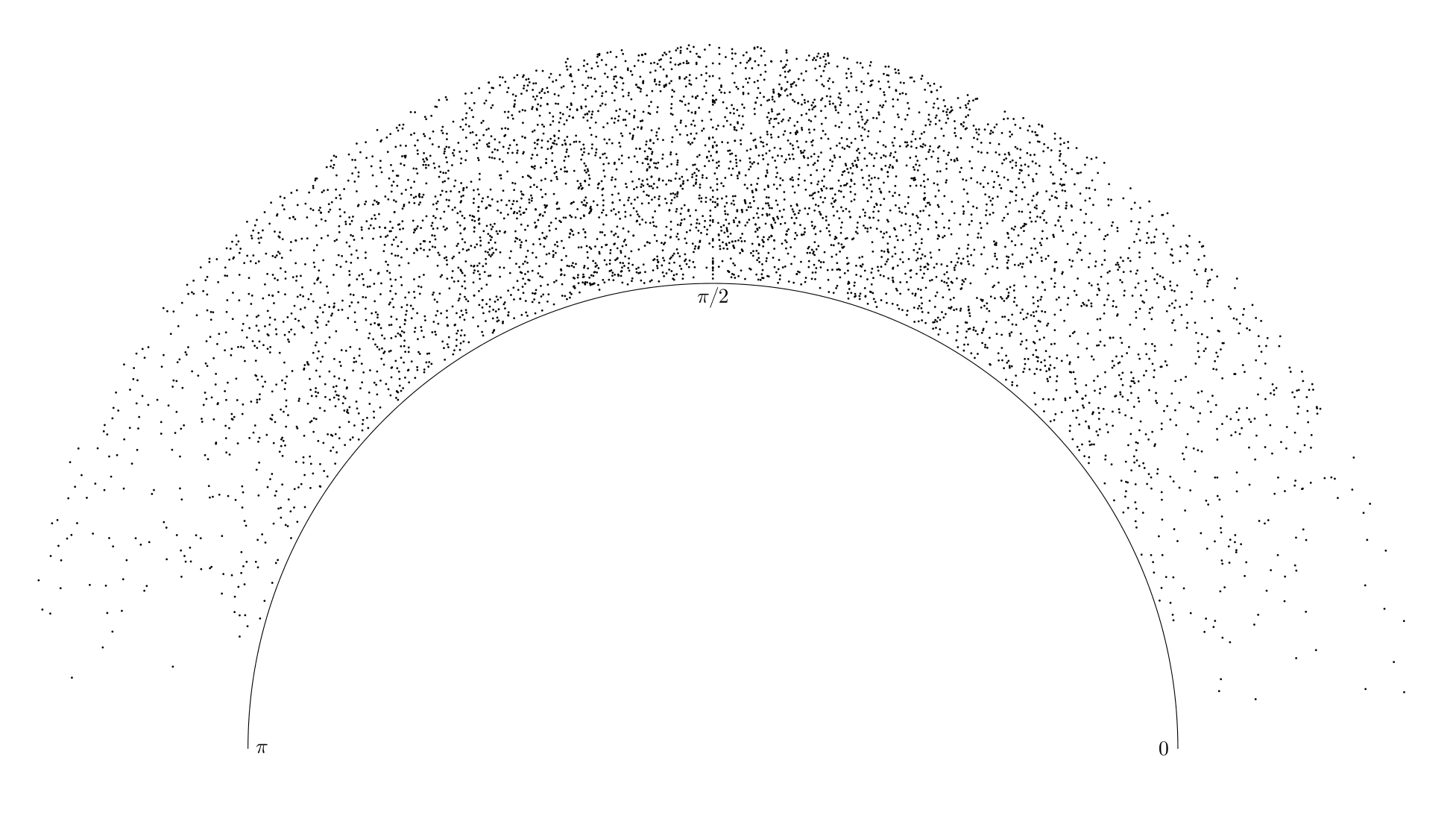}
\end{center}
In contrast to this, consider the elliptic curve $y^2=x^3+1$. Here complex multiplication is given by the Eisenstein integers. The plot below shows $\theta_p$ for the first $5,000$ primes. Norm is linear in the prime (so, again, the further the dot, the bigger the prime). 
\begin{center}
     \includegraphics[scale=0.4]{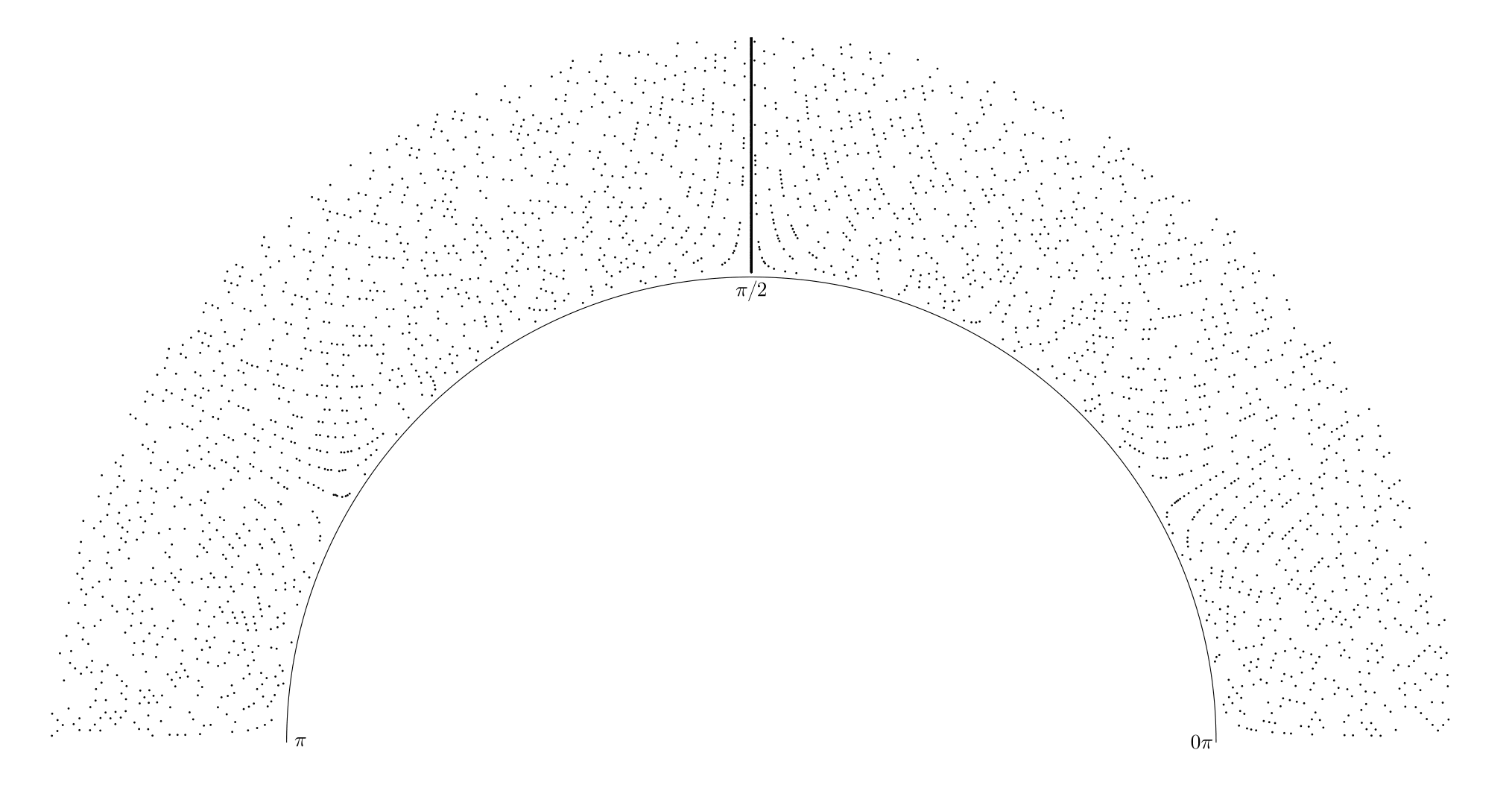}
\end{center}
The difference between these two curves is striking!

\begin{remark}
The case of complex multiplication is well understood. The idea (Geordie thinks) is that the extra endomorphisms force the $\Frob$ to lie in a subgroup of $SU(2)$. (Geordie points out that he is ``the exact opposite of an expert on this topic...'')
\end{remark}

\begin{remark}
One can think of Sato-Tate (roughly) as a higher-dimensional analogue of Chebotarev density: $\Gal(K/\Q) \leftrightarrow SU(2)$. 
\end{remark}

\subsection{Equidistribution and $L$-functions}

In the final part of this lecture, we will describe how the Sato--Tate conjecture follows from a simple part of the Langlands correspondence (which is still conjectural). The Sato--Tate conjecture has been proven using other methods, but the proof is very involved and heavily influenced by ideas from the Langlands program. This example is meant to showcase the power of the Langlands correspondence. 

Let $G$ be a compact group, $X$ its space of conjugacy classes, and $x_p \in X$ a family of elements parameterized by primes $p$ (perhaps outside some finite set of ``bad'' primes). For an irreducible character $\chi$, we can define an $L$-function analogously to how we defined Artin $L$-functions for Galois groups:
\[
L(\chi, s)=\prod_p \frac{1}{\det(1-\rho(x_p)p^{-s})},
\]
where $\rho$ is the representation afforded by $\chi$. This converges for $\Re(s)>1$. Assume additionally that $L(\chi, s)$ extends to a meromorphic function on $\Re(s) \geq 1$ having neither zeros nor poles along $\Re(s)=1$ except possibly at $s=1$. Let $-c_\chi$ be the order of $L(\chi, s)$ at $s=1$ (so $c_\chi>0$ pole, $c_\chi<0$ zero). With these assumptions, we have the following theorem. 
\begin{theorem}
\[
\sum_{p \leq n} \chi(x_p) = c_\chi \cdot \frac{n}{\log(n)} + o\left(\frac{n}{\log(n)}\right).
\]
\end{theorem}
The proof of this theorem involves some complex analysis and tricks with sums, but is not difficult. This theorem has an important corollary. 
\begin{corollary}
\label{equidistribution}
If for all nontrivial $\chi$, $L(\chi, s)$ is holomorphic and nonzero at $s=1$, then the $x_p$ are equidistributed in $X$. 
\end{corollary}

So we can use $L$-functions to test whether a sequence in a compact group is equidistributed! 

\begin{exercise}
\begin{enumerate}[label=(\alph*)]
\item Show that $L(\chi, 1)\neq 0$ implies Dirichlet's theorem. 
\item Let $K$ be a number field. It's known that $\zeta_K(s)/\zeta(s)$ is holomorphic and nonvanishing at $s=1$. Using this, deduce Chebotarev's density theorem. 
\end{enumerate}
\end{exercise}

An important consequence of Corollary \ref{equidistribution} is that it lets us reframe the Sato-Tate conjecture in terms of $L$-functions.
\begin{example} (Serre) Assume that for all $m \geq 1$, the symmetric $L$-power function 
\[
L(S_\chi^m, s) := \prod_p \frac{1}{ (1 - \theta_p^m p^{-s}) ( 1- \theta_p^{m-2}p^{-s}) \cdots (1-\theta_p^{-m}p^{-s})}
\]
satisfies the extra assumption above. (Here $S^m_\chi$ is the $m^{th}$ symmetric power of the representation with character $\chi$ sending $\Frob\mapsto \bp \theta_p & 0 \\ 0 & \theta_p^{-1} \ep$.)  Then the Sato-Tate conjecture holds!
\end{example}

Recall our cartoon of the Langlands correspondence: 
\[
\begin{tikzcd}
\left\{ \begin{array}{c} \text{``geometric'' $n$-dimensional} 
\\
\text{ rep'ns of $\Gal(K)=\Gal(\overline{K}/K)$}
\end{array}   \right\}\arrow[rd] & \xleftrightarrow{ \text{ finite-to-one }} &
\left\{ \begin{array}{c} \text{``automorphic'' rep'ns} 
\\
\text{ of $\GL_n(\mathbb{A})$}
\end{array}   \right\}\arrow[ld]\\
& \text{$L$-functions}
\end{tikzcd}
\]

An extremely important part of this picture is {\bf functoriality}. That is, the diagram should be compatible with
\begin{enumerate}
    \item composition of representations of the Galois group with algebraic representations of $\GL_n$ (so $\rho: \Gal(K) \rightarrow \GL_n \xrightarrow{\text{alg rep'n}} \GL_m$ should correspond to some operation on automorphic representations), and
    \item changing the field $K$. 
\end{enumerate}
A key piece of the Langlands correspondence is that $L$-functions coming from automorphic representations have many desirable properties, which are extremely difficult to establish for $L$-functions coming from geometric representations of $\Gal(K)$. For example, once we know that an $L$-function comes from an automorphic representation, we immediately know that it admits a meromorphic continuation and has a functional equation.   

\vspace{5mm}
\noindent
{\bf The Punchline:} {\em In the simple example of the algebraic representation $\GL_2 \rightarrow \GL_m$ via symmetric powers, the prediction of functoriality implies the Sato-Tate conjecture! }

%% file: lecture-5.tex
\section{Lecture 5: Infinite Galois theory and global class field theory}
\label{lecture 5}

The topic of today's lecture is {\bf class field theory}. But before diving in, we will start with a motivating question and a review of infinite Galois theory.

\vspace{5mm}
\noindent
{\bf Basic Question:} What are all finite extensions of a number field $K$ (e.g. $\Q$)? 

\vspace{5mm}
\noindent 
This is certainly a question of fundamental importance in number theory. One could also ask more specific questions, such as ``How many number fields have a given Galois group and discriminant?'' For almost every question of this sort, we have no idea what the answer is. Class field theory develops techniques for answering these questions in the abelian setting. We will say more precisely what this means later in the lecture, but for now, let's take a look at an example. 

\begin{example}
\label{profinite Z}
Let $K=\F_p$, and $\overline{K}$ its algebraic closure. For all $n\geq 1$, there is a unique subfield $K_n \subset \overline{K}$ with $p^n$ elements, and $\overline{K}=\bigcup_{n \geq 1} K_n$. We have the following picture of field extensions and corresponding Galois groups:
\begin{center}
     \includegraphics[scale=0.4]{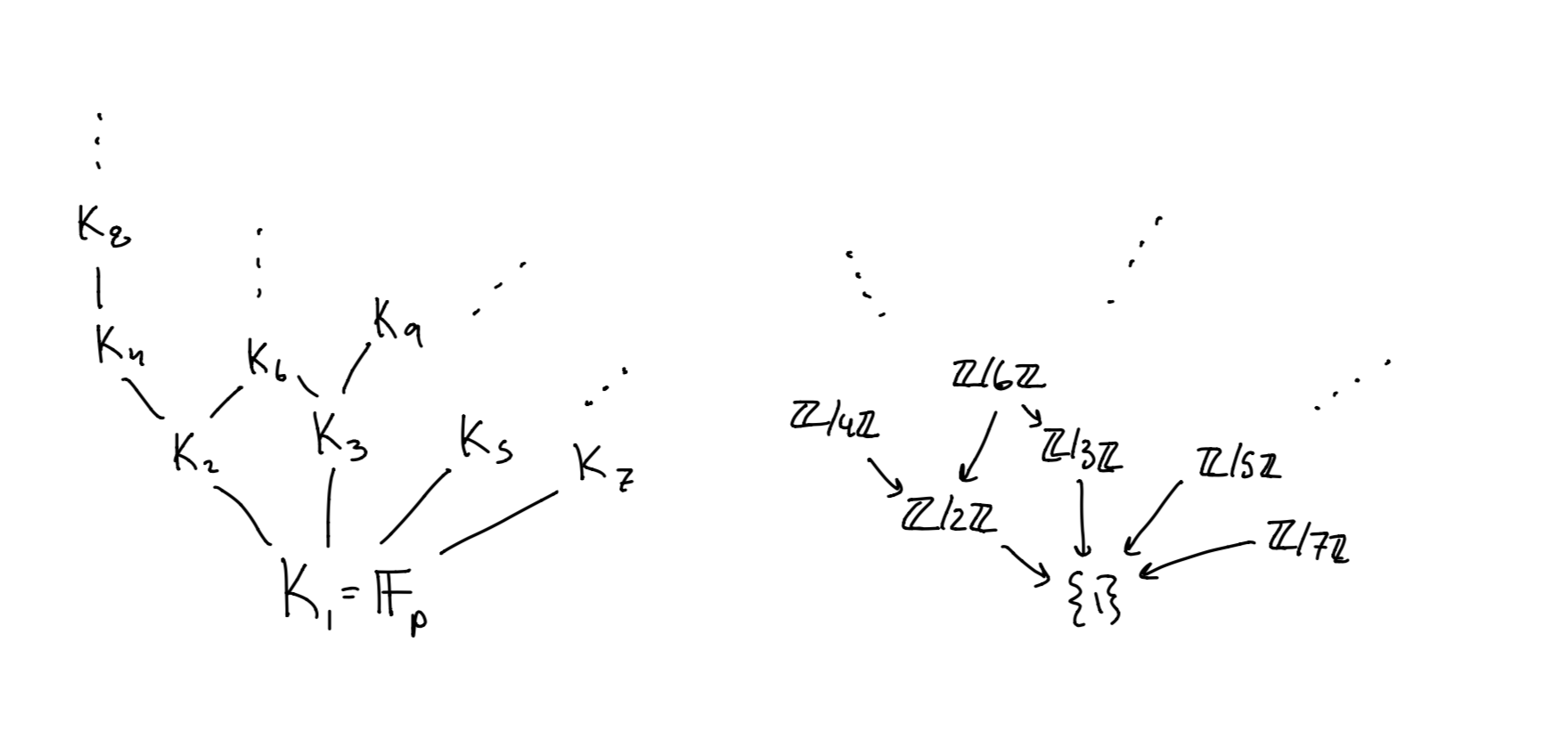}
\end{center}
Let's calculate $\Gal(\overline{K}/K)$. Because $\overline{K}=\bigcup_{n \geq 1} K_n$, we have an injection 
\[
\Gal(\overline{K}/K) \hookrightarrow \prod_{n \geq 1} \Gal(K_n/K) = \prod_{n \geq 1} \Z/n\Z.
\]
For $\varphi \in \Gal(\overline{K}/K)$, let the sequence $(\varphi_n)$ be its image in $\prod_{n \geq 1}\Gal(K_n/K)$. A sequence $(\gamma_n) \in \prod_{n \geq 1}\Gal(K_n/K)$ is in the image if and only if $\gamma_n=\gamma_m \mod m$ whenever $m | n$. Hence, 
\[
\Gal(\overline{K}/K)=\lim_{\longleftarrow} \Z/n\Z = \widehat{\Z}
\]
is the ``profinite completion of $\Z$''. 
\end{example}

\begin{exercise}
\begin{enumerate}[label=(\alph*)]
\item Show that $\widehat{Z} =\prod_{p \text{ prime}} \Z_p$. 
\item Show that $\widehat{\Z}$ has uncountably many subgroups. Hence a naive Galois correspondence cannot hold.
\end{enumerate}
\end{exercise}

\subsection{Infinite Galois theory}
Let $L/K$ be a Galois extension (algebraic, normal, separable, not necessarily finite). Then we have an injection
\[
\Gal(L/K) \hookrightarrow \prod_{K \subset L' \subset L} \Gal(L'/K),
\]
where the product is taken over all towers of field extensions $K \subset L'\subset L$, where the extension $L'/K$ is finite Galois. For all towers of extensions $K \subset L' \subset L'' \subset L$, where the extensions $L'/K$ and $L''/K$ are finite Galois, there is a corresponding map $\Gal(L''/K) \rightarrow \Gal(L'/K)$. This determines the image of this injection; that is, 
\[
\Gal(L/K)=
  \lim_{\substack{\longleftarrow \\ K \subset L' \subset L}} \Gal(L'/K).
\]
This is a topological group. Indeed, if we give $\prod_{K \subset L' \subset L} \Gal(L'/K)$ the product topology (which is compact, by Tychonov), then $\Gal(L/K)$ inherits the subspace topology. 
\begin{exercise}
\label{compact}
The group $\Gal(L/K)$ is closed (therefore compact) in $\prod_{K \subset L' \subset L} \Gal(L'/K)$.
\end{exercise}
\begin{example}
  We see from Exercise \ref{compact} and Example \ref{profinite Z} that the group $\widehat{\Z}$ is compact, which might look strange.
\end{example}

\begin{exercise}
(Important, can be used as a definition) A basis of open neighborhoods of $1 \in \Gal(L/K)$ is given by kernels of the maps 
\[
\Gal(L/K)\rightarrow \Gal(L'/L)
\]
for $L'/K$ finite Galois. 
\end{exercise}

For a general group $G$, define 
\[
\widehat{G}=\lim_{\substack{ H \subseteq G \\ \text{ normal}\\\text{finite index}}} G/H. 
\]
The group $G$ is {\bf profinite} if $G \xrightarrow{\simeq}\widehat{G}$, otherwise, say $\widehat{G}$ is the {\bf profinite completion} of $G$. {\em So Galois groups are profinite groups!} The key example to keep in mind is the following. 

\begin{example}
\label{3-adics}
Consider the group $\displaystyle{\Z_p=\lim_{\longleftarrow}\Z/p^n\Z}$. What does this group look like? Here's a picture for $p=3$:
\begin{center}
     \includegraphics[scale=0.4]{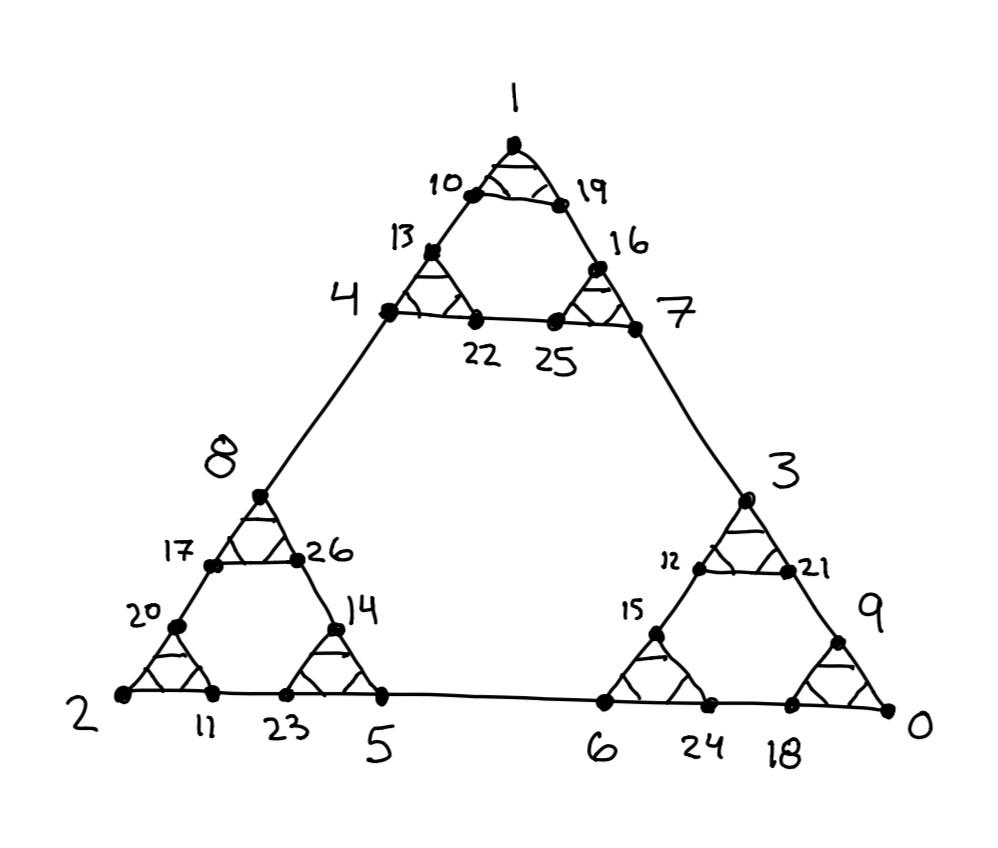}
\end{center}
\begin{exercise}
Here are some fun thought experiments. 
\begin{enumerate}[label=(\alph*)]
    \item Where is $-1$ in this picture?
    \item Think about $\Q_3$. 
\end{enumerate} 
\end{exercise}
\end{example}

\vspace{5mm}
\noindent
{\bf The motto:} Galois groups are fractal-like objects! 
\vspace{3mm}

\begin{theorem}
(Fundamental theorem of infinite Galois theory) Let $L/K$ be a Galois extension. Then there exists a canonical bijection
\begin{align*} 
\left\{ K \subset L' \subset L  \right\} &\leftrightarrow \left\{  \begin{array}{c} \text{closed subgroups}
\\
\text{of }\Gal(L/K)
\end{array}   \right\}\\
L^H& \mapsfrom H \\
L' &\mapsto \Gal(L/L')
\end{align*}
Moreover, under this bijection, 
\begin{align*}
    \text{ finite extensions } & \leftrightarrow \text{ closed and open subgroups } \\
    \text{ Galois extensions } & \leftrightarrow \text{ normal subgroups}
\end{align*}
\end{theorem}

\begin{exercise}
Show that the only closed subgroups of $\widehat{\Z}$ are $n\widehat{\Z}$ for $n \in \Z$. If $n \geq 1$, then the subgroup $n \widehat{\Z}$ corresponds to the extension $K_n$ under the bijection above, and $n=0$ corresponds to $\overline{K}$. (So $0\widehat{\Z}$ is the only closed subgroup which isn't open.) 
\end{exercise}

Now we return to the problem posed at the start of this lecture:
\[
\text{\bf describe all number fields over $\Q$}
\]
or, equivalently, 
\[
\text{ \bf describe all closed subgroups of $\Gal(\overline{\Q}/\Q)$.}
\]
However, after some thought, one sees that this isn't really a well-defined question (from a philosophical point of view), because $\overline{\Q}$ involves a {\em choice} (or many choices), so there is no concrete canonical realisation of $\overline{\Q}$. Hence $\Gal(\overline{\Q}/\Q)$ is only really a ``group up to conjugacy,'' in the sense that any meaningful statements one can make about this group must be invariant under conjugation. (One cannot talk about individual elements.) 

\vspace{5mm}
\noindent
{\bf The Punchline:} \begin{enumerate}
    \item Isomorphism classes of representations of ``a group up to conjugacy'' are canonical, so it makes sense to talk about representations of $\Gal(\overline{\Q}/\Q)$. {\em This is one reason why representations are so important in the Langlands program!}
    \item The maximal abelian extension $\Q^{ab}$ of $\Q$ (which is the extension corresponding to $\overline{[\Gal(\overline{\Q}/\Q), \Gal(\overline{\Q}/\Q)]}$) {\em is} canonical, and we can hope to describe it by studying the maximal abelian quotient $G^{ab}:=G/\overline{[G,G]}$ of $G=\Gal(\overline{\Q}/{\Q})$. {\em This is class field theory!}
\end{enumerate}

\subsection{Global class field theory, first version}
\label{global class field theory}

The key example to keep in mind is the maximal abelian extension of $\Q$.  
\begin{example}
Let $K_m:=\Q(\zeta_m)$ for $\zeta_m=e^{2 \pi i/m}$. Define $\Q(\mu_\infty):=\bigcup_{m \geq 1} K_m$. The Galois group of $K_m/\Q$ is $\left( \Z/m\Z\right)^\times$, hence 
\[
\Gal(\Q(\mu_\infty)/\Q) = \lim_{\longleftarrow} \left(\Z/m\Z\right)^\times = \prod_p \Z_p ^ \times. 
\]
{\bf Fact:} (``Kronecker Jugendtraum'') $\Q(\mu_\infty)$ is the maximal abelian extension of $\Q$. 

\vspace{5mm}
\noindent 
This fact is not easy! (It will follow from global class field theory.) The hope of Kronecker was to predict this starting just from $\Q$, without calculating extensions.
\end{example}

\begin{exercise}
Use Kronecker Jugentraum to show that any continuous character $\chi: \Gal(\overline{\Q}/\Q) \rightarrow \C^\times$ ``is'' a Dirichlet character. (Part of the exercise is to work out what ``is'' means in this context.)
\end{exercise}

Given a number field $K$, it is useful to consider all norms at once. Let $\OO\subset K$ be the ring of integers. A {\bf place} $v$ is an equivalence class of nontrivial multiplicative norms 
\[
|\cdot|_v: K \rightarrow \R_{\geq0}
\]
on $K$.

\begin{theorem}
All places of a number field $K$ are of the following form. 
\begin{itemize}
    \item {\bf Finite places:} $|x|_v:= \left( \# \OO/ \mf{p} \right)^{-\text{val}_\mf{p}(x)}$ for $\mf{p} \subset \OO$ prime.
    \item {\bf Real places:} $|x|_v:=|i(x)|$ for some real embedding $i:K \hookrightarrow \R$, 
    \item {\bf Complex places:} $|x|_v:=|i(x)|^2$ for some pair of conjugate $i:K \hookrightarrow \C$ not landing in $\R$. 
\end{itemize}
\end{theorem}
These are all possible notions of distance on a number field. 

\begin{exercise}
Show that there are no nontrivial norms on a finite field. 
\end{exercise}

Note that we could have chosen any scalar $\lambda >1$ in place of $\#\OO/\mf{p}$ above. The reason for the the above normalization is the beautiful {\bf product formula}: For $x \in K^\times$, 
\[
\prod_{\text{places }v} |x|_v = 1.
\]

Note that this product makes sense because all but finitely many places are $1$. The function field case of this formula is the statement that the number of poles and number of zeros (with multiplicity) agree. 

\subsection{Global class field theory \'{a} la Artin}

Fix a finite abelian Galois extension $L/K$ with abelian Galois group $\Gal(L/K)$. Let $\OO_L\subset L$ and $\OO_K \subset K$ be the rings of integers, and let $S_f\subset \OO_K$ be the set of ramified primes. We have seen that for a prime $\mf{p} \subset \OO_K$ such that $\mf{p} \not \in S_f$, there is a corresponding conjugacy class $\Frob \in \Gal(L/K)$. In general $\Frob$ is only defined up to conjugacy, but since we are assuming that $\Gal(L/K)$ is abelian, $\Frob$ is a single element. Hence we obtain a map (the {\bf Artin map}):
\begin{align*}
\mc{J}^{S_f} = \bigoplus_{\mf{p} \not \in S_f} \Z\mf{p} &\rightarrow \Gal(L/K) \\
\sum m_\mf{p}\mf{p} &\mapsto \prod \Frob^{m_{\mf{p}}} 
\end{align*}

Here $\mc{J}^{S_f} \subset \mc{J}$ is a subgroup of the group of nonzero fractional ideals $\mc{J}=\bigoplus_{\text{primes }\mf{p}}\Z\mf{p}$ discussed in Lecture 2. By Chebotarev's density theorem, the Artin map is surjective. 

\vspace{5mm}
\noindent
{\bf Question:} What is the kernel of the Artin map?
\vspace{5mm}

The answer to this question is related to an observation we made in the very first lecture. Recall our motivating problem for the course of determining the number of solutions of a polynomial modulo $p$ for various primes $p$. For quadratic polynomials, we used quadratic reciprocity to find some modulus $m \in \Z$ such that the number of solutions of the polynomial modulo $p$ was given by the residue of $p$ modulo $m$. At first, the modulus $m$ seemed to be somewhat mysterious, but eventually we observed that it was obtained from the {\em ramified primes} (that is, the ``weird primes'' which we ignored). For example, for the polynomial $x^2+1$, which has 2 solutions modulo $p$ if $p=1 \mod 4$ and $0$ solutions if $p=3 \mod 4$, the modulus $4$ is the square of $2$, our only ramified prime. 

Returning to the setting of the Artin map, we define a {\bf modulus} $m$ supported in a set of places $S$ to be a formal $\Z$-linear combination of places $m=\sum m_i v_i$ such that $m_i \in \{0, 1\}$ for real places and $m_i=0$ for all complex places and places $v_i \not \in S$. Given a modulus $m$ supported in a set of places $S$, we can define an associated group (the Ray class group) as follows. Consider the following two subsets of $K^\times$:
\[
K^S =\{\lambda \in K^\times \mid \text{val}_{\mf{p}}(\lambda)=0 \text{ for all } \mf{p} \in S\} , \text{ and }
\]
\[
K^{m,1}=\{\lambda \in K^s \mid  \text{val}_\mf{p}(\lambda - 1) \geq m_i \text{ for finite places, and }
i(\lambda) \in \R_{>0}^\times \text{ for real places }m_i=1\}.
\]
The set $K^{m,1}$ is the set of $\lambda$ which are ``$m$ close to $1$''. We have the following maps: 
\[
\begin{tikzcd}
K^\times\arrow[rd, twoheadrightarrow] \arrow[rr, "\text{val}"] & &
\mc{J} \arrow[r, twoheadrightarrow] & \mc{C}\ell_K\\
& \{\text{principal ideals}\} \arrow[ur, hookrightarrow] \\
K^s \arrow[uu, hookrightarrow] \arrow[rr, "\text{val}"]& & \mc{J}^S \arrow[uu, hookrightarrow]\\
K^{m,1} \arrow[urr, "\text{val}"] \arrow[u, hookrightarrow, ]
\end{tikzcd}
\]
The {\bf Ray class group} associated to the modulus $m$ is the quotient 
\[
\mc{C}\ell_K^m:=\mc{J}^S/\text{val}(K^{m,1}).
\]
\begin{example}
    If $m=0$, then $\mc{C}\ell_K^0=\mc{C}\ell_K$ is the class group. 
\end{example}
\begin{example} Let $K=\Q$, and $m=n(p)$ for a prime $p \in \Z$. Then $m$ is a modulus supported on $S=\{(p)\}$. We have 
\[
K^S=\left\{ \frac{a}{b} \mid a,b \text{ are coprime to }p \right\} = \Z_{(p)}^\times, \text{ and}
\]
\[
K^{m,1} = \left\{ \frac{a}{b} \mid a,b \text{ coprime to }p \text{ and } \text{val}_p(\frac{a}{b}-1)\geq n\right\} = \left\{ \frac{a}{b} \mid \frac{a}{b} = 1 \mod p^n \right\}. 
\]
Hence 
\[
K^s/K^{m,1} = \Z^\times_{(p)} / K^{m,1} = \left( \Z/p^n\Z \right)^\times = \mc{C}\ell_\Q^m. 
\]
\end{example}
\begin{theorem}
For any modulus $m$, the Ray class group is finite and surjects onto the class group of $K$. 
\end{theorem}

\begin{theorem}
\label{Artin kernel}
  (Artin) For any $L/K$ as above, there exists a modulus $m$ with $\text{supp}(m) \cap \{\text{finite places}\} = S_f$ such that $\text{val}(K^{m,1})$ is contained in the kernel of the Artin map. Moreover, for any modulus $m$ and any quotient $q:\mc{C}\ell_K^m \twoheadrightarrow Q$, there exists an abelian extension $L/K$ ramified only at primes in $\text{supp}(m)$ such that \[
  \begin{tikzcd}
  \mc{C}\ell^m_K \arrow[r, "\text{Artin}"] \arrow[rd, twoheadrightarrow]& \Gal(L/K) \\
  & Q\arrow[u, equal]
  \end{tikzcd}
  \]
  commutes. 
\end{theorem}
A weaker form of this theorem provides a more direct answer to our question from the beginning of this section.  
\begin{theorem}
  For any abelian Galois extension $L/K$, there exists an $\epsilon >0$ such that if $\lambda \in K^{S_f}$ is $\epsilon$ close to $1$ for all places $v \in S$, then $\text{val}(\lambda)$ is in the kernel of the Artin map.
\end{theorem}

\begin{example}
  Consider the extension $\Q(i)/Q$, which is the splitting field of the polynomial $x^2+1$. The ramified primes are $2$ and $\infty$ (but we haven't discussed what it means for $\infty$ to be ramified, so we are sweeping this under the rug), so $S^f=\{(2)\}$. For any $p \neq 2$, $\Frob(i)=i^p$, hence in the Galois group, $\Frob\leftrightarrow p \mod 4 \in \left(\Z / 4\Z\right)^\times = \Gal (\Q(i)/\Q)$. We have 
  \[
  \mc{J}^{S_f}=\left\{\frac{a}{b} \mid a,b>0 \text{ coprime to }2 \right\}=\Z_{(2),>0}^\times. 
\]
Hence the Artin map is the natural map 
\[
\Z_{(2),>0}^\times \rightarrow \left(\Z/4\Z\right)^\times. 
\]
Its kernel is the set of all elements satisfying a ``congruence condition at 2'':
\[
\left\{ \lambda \in \Z_{(2),>0}^\times \mid \text{val}_2(\lambda-1) \geq 2\right\}. 
\]
We see from this example that the finitely many ``weird'' (ramified) primes from the first lecture are the primes determining our congruences.  
\end{example}

\begin{example}
  If $m=0$, then $\mc{C}\ell_K^0=\mc{C}\ell_K$. Hence the existence statement in Theorem \ref{Artin kernel} implies that there exists an unramified everywhere extension $L/K$ with $\Gal(L/K)=\mc{C}\ell_K$. This extension is the {\bf Hilbert class field}. 
\end{example}

%% file: lecture-6.tex
\section{Lecture 6: Local Galois groups and local class field theory}
\label{lecture 6}

\subsection{A potted history of class field theory}
Sometimes if we find something difficult, it can be comforting to know that other people also found it difficult. Accordingly, we'll start today's lecture with a potted history of class field theory, following \cite{conrad, miyake}.  

\begin{itemize}
    \item {\bf Kronecker, Weber (1850-1880)}: Kronecker's Jugentraum (that $\Q(\mu_\infty)$ is the maximal abelian extension of $\Q$), explicit class field theory  (describing $K^{ab}$ explicitly, not just its Galois group) for $\Q$ and $\Q(i)$, relevance of complex multiplication.   
    \item {\bf Dedekind, Frobenius (1880)}: defined $\Frob$ (then everyone promptly forgot for $40$ years).
    \item {\bf Hilbert (189-1900)}: first correct proof of Jugentraum for $\Q$, emphasis on ``places at $\infty$,'' introduction of Hilbert class field, $12^{th}$ problem on Hilbert's list was explicit class field theory for any $K$. (Still open! Even for $\Q(\sqrt{d}), d \geq 0$!)
    \item {\bf Hensel (1897)}: introduction of $p$-adic numbers, took a while to catch on.
    \item {\bf Takagi (1900)}: PhD student of Hilbert in G\"{o}ttingen, thesis on $\Q(i)^{ab}$, proof of existence theorem during WWI (when there was no contact between Germany and Japan), result announced at the ICM in 1920.
    \item {\bf Hasse (1922)}: local global principle, first time $p$-adics were taken seriously by the broader mathematics community. 
    \item {\bf Chebotarev (1927)}: density theorem. 
    \item {\bf E. Artin (1927)}: introduction of Artin map (the return of $\Frob$!), reciprocity theorem. 
    \item {\bf Schmidt (1930)}: deduced local class field theory from global class field theory (proofs still analytic). 
    \item {\bf Noether (1930s)}: local theory should be simpler and come first.
    \item {\bf Chevalley (1940)}: algebraic proof of local class field theory.
\end{itemize}

\subsection{A trip across the bridge}

Recall that the ``bridge'' in this course is the motivating analogy between number fields and function fields.

\begin{remark}
This bridge was what motivated Hensel's advocation for the introduction of the $p$-adic numbers.  
\end{remark}

Let's look at local class field theory through this analogy. Let $L/K$ be a finite Galois extension. Across the bridge, this should correspond to a surjective map (i.e. ramified cover) $f:X\rightarrow Y$ of algebraic curves/Riemann surfaces over $\C$. Recall that our $\R$-picture and $\C$-picture of such a map are the following:
\begin{center}
     \includegraphics[scale=0.4]{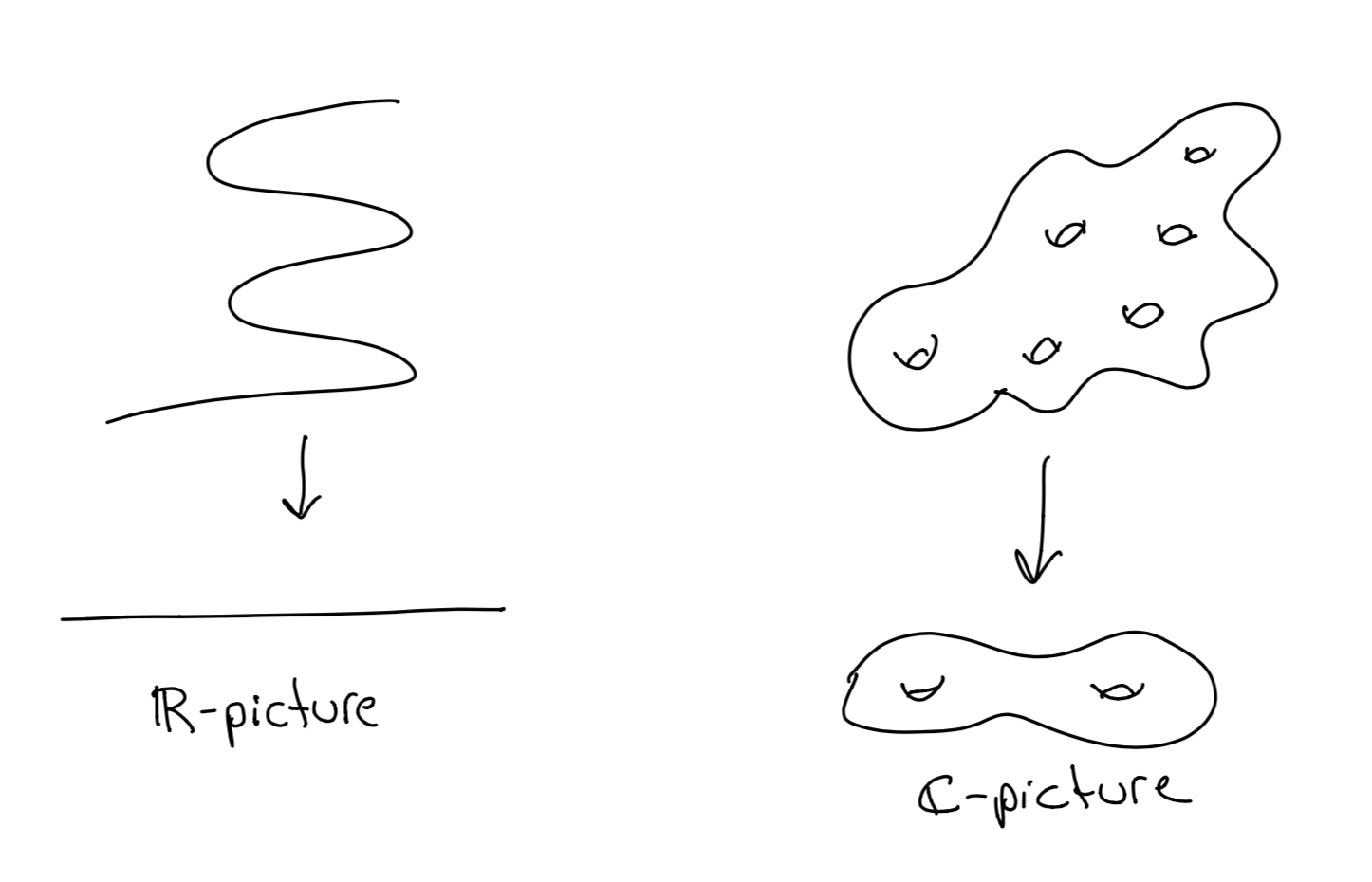}
\end{center}
For all $y \in Y$ outside of a finite set, $f$ is \'{e}tale; that is, a smooth $n:1$ cover in a neighborhood of $y$:
\begin{center}
     \includegraphics[scale=0.4]{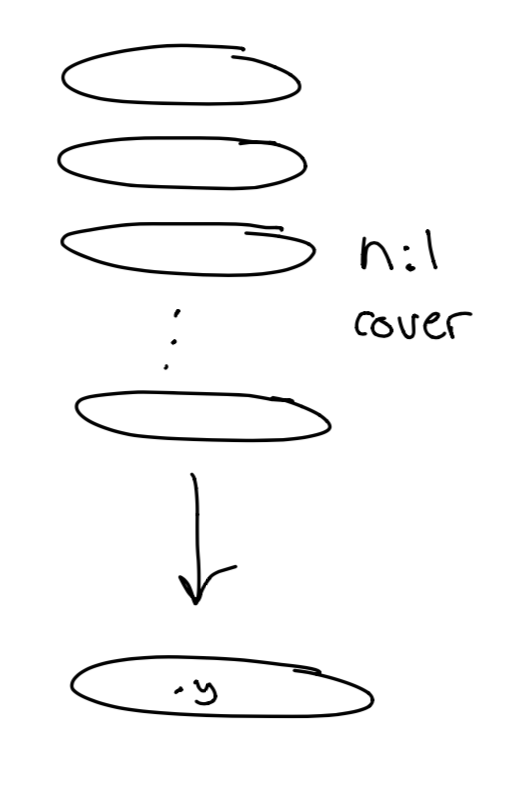}
\end{center}
For a finite set of points $y \in Y$, $f$ is not smooth, and locally sends $z\mapsto z^{n_i}$, such that $\sum n_i=n$: 
\begin{center}
     \includegraphics[scale=0.4]{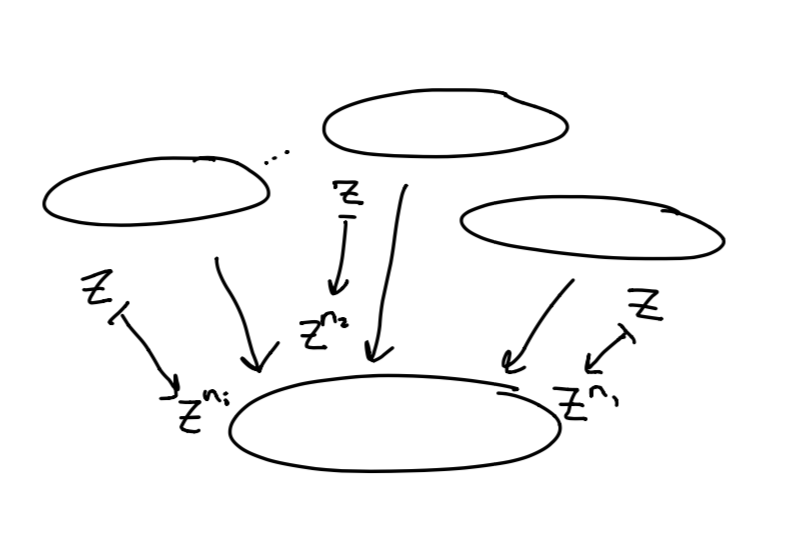}
\end{center}
If $f$ is Galois, then all $n_i$ are equal. 

\vspace{5mm}
\noindent
{\bf The Moral:} The ramified cover $f$ is determined by simple data (``local monodromies'') at finitely many points (``ramified primes of $Y$'').
\vspace{5mm}

\begin{remark}
We have
\begin{center}
     \includegraphics[scale=0.4]{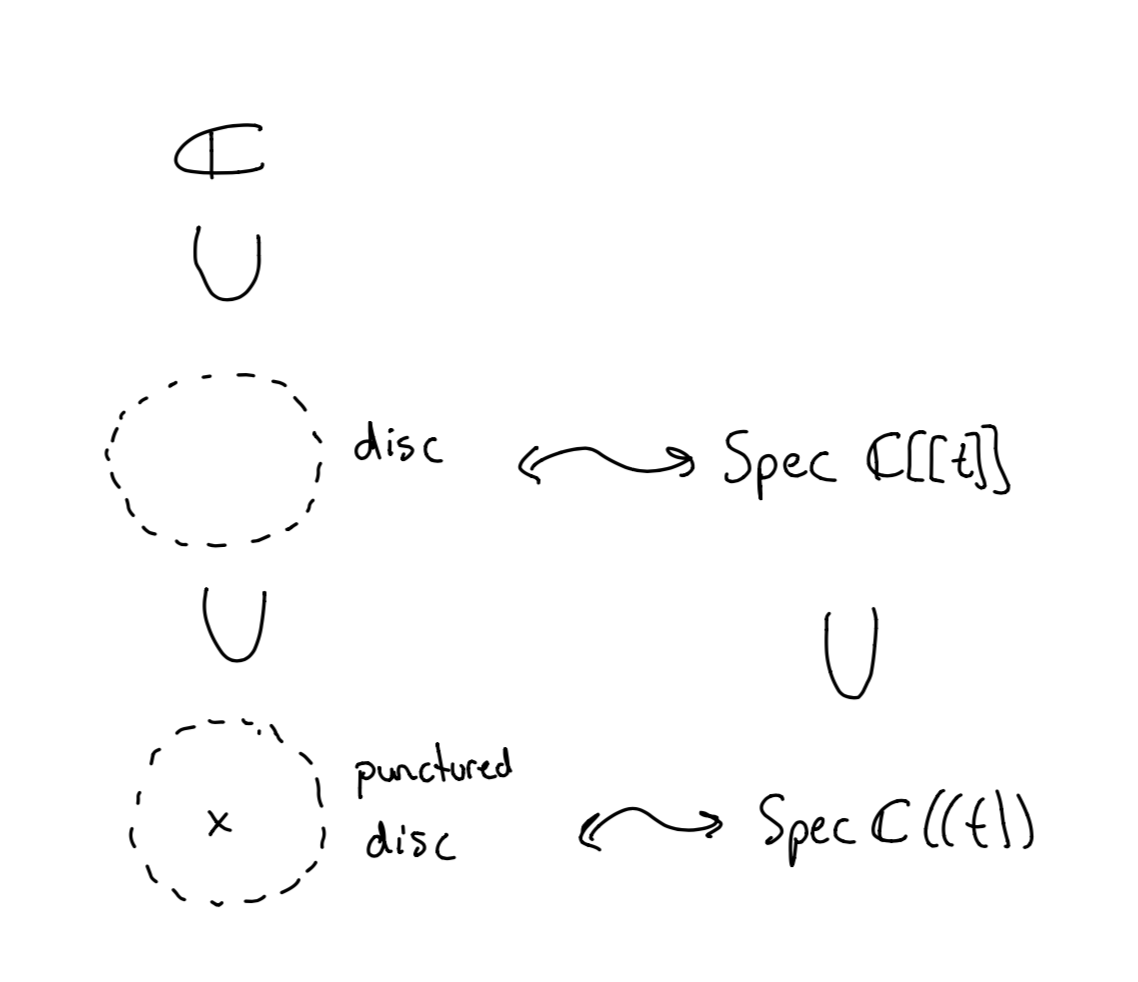}
\end{center}
The algebraic closure of $\C((t))$ is $\displaystyle{\C((t^\Q)):=\bigcup_{n \geq 1} \C((t^{1/n}))}$. Hence, \[
\Gal\left(\overline{\C((t))}/\C((t))\right) = \lim_{\leftarrow} \Z/n\Z = \widehat{\Z}.
\]
This is ``why'' the local field information is so simple in function field/$\C$ case\footnote{Recall that we saw last lecture that $\Gal(\overline{\F_p}/\F_p)=\widehat{\Z}$ as well. This turns out to be a useful coincidence, but we won't comment on it further here.}. In language to come, every extension of $\C((t))$ is ``tamely ramified''.

The upshot is that on the function field side of the bridge, local information is easy, and can be patched together to form the global picture. On the number field side of the bridge, local information is much harder, but the philosophy we learn from our analogy is that it should still be easier than global information, so we should focus on it first. Because of this, essentially the rest of this course will be local.  
\end{remark}

Now we return to number fields. Let $L/K$ be a finite Galois extension. Fix a place $v$ of $K$. If $v$ is finite (corresponding to some $\mf{p} \subset \OO_k$), then we know what it means for a place $v'$ (corresponding to $\mf{q} \subset \OO_L$) of $L$ to ``lie over $v$'': $v'$ lies over $v$ precisely when $\mf{q}$ is a prime above $\mf{p}$ in the sense of lecture $2$ (i.e. $\mf{p}\OO_L \subset \mf{q}$). 

If $v$ is a real or complex place corresponding to $i:K \hookrightarrow \mathbb{K}\in \{\R, \C\}$, then a place $v'$ of $L$ {\bf lies over }$v$ if the corresponding injection $i'$ fits into a commutative diagram. 
\[
\begin{tikzcd}
L \arrow[r, hookrightarrow, "i'"] \arrow[d, dash]& \mathbb{K}' \arrow[d, dash]\\
K \arrow[r, hookrightarrow, "i"] & \mathbb{K}
\end{tikzcd}
\]
Hence if $v$ is real, then $v'$ is either real or complex. A real place $v$ is {\bf ramified} if there exists a $v'$ lying above $v$ that is complex in $L$, and {\bf unramified} if all $v'$ above $v$ are real. If $v$ is complex, then all $v'$ above $v$ are complex, and we say the place $v$ is {\bf unramified}. 

Fix a place $v$ of $K$ and a place $v'$ of $L$ over $v$. Let $L_{v'}$ (resp. $K_v$) be the completion of $L$ (resp. $K$) with respect to the place $v'$ (resp. $v$). Then we have the diagram
\[
\begin{tikzcd}
L \arrow[r, hookrightarrow] \arrow[d, dash]& L_{v'} \arrow[d, dash]\\
K \arrow[r, hookrightarrow] & K_v.
\end{tikzcd}
\]
Set 
\begin{align*}
G_v&=\{ \sigma \in \Gal(L/K) \mid \sigma \text{ acts continuously on }L_v\} \\
&=\{ \sigma \mid \sigma \text{ preserves }v'\} = \begin{cases} G_\mf{q}  &\text{ if }v\text{ is finite}, \\
\{1\} & \text{ if $v$ is unramified infinite},\\
\Z/2\Z & \text{ if $v$ is ramified infinite}. \end{cases}
\end{align*}
Here $G_\mf{q}\subset \Gal(L/K)$ is the decomposition group corresponding to the prime $\mf{q} \subset \OO_L$ determining $v'$. 
\begin{remark}
In the last case (when $v$ is ramified infinite), we get a canonical element $c \in \Gal(L/K)$ corresponding to complex conjugation.
\end{remark}
\vspace{3mm}
\noindent
{\bf The point:} $L_{v'}/K_v$ is a finite Galois extension of local fields with Galois group $G_v$. We will first try to understand such extensions for all places $v$, then piece together this information to understand $L/K$. 
\vspace{3mm}

\subsection{Local class field theory}
Between the ``easy'' world of finite fields and the complicated world of global fields lies the world of local fields. Let $K$ be a field equipped with discrete valuation $\val:K \rightarrow \Z\cup \{\infty\}$. In $K$ lies its ring of integers $\OO_K$, with maximal idea $\mf{m}$ generated by a ``uniformizer'' $\pi \in \mf{m}$:
\[
K \supset \OO_K=\val^{-1}(\Z_{\geq 0} \cup \{ \infty\}) \supset \mf{m} =\val^{-1}(\Z_{>0} \cup \{\infty\})=(\pi). 
\]
The field $k_K = \OO_K/\mf{m}$ is the residue field. Note that in this setup, $K$, $\OO_K$ and $\mf{m}$ are all canonical, but the uniformizer $\pi \in \mf{m}$ is not. For us, a {\bf local field} will be a field $K$ equipped with a discrete valuation as above such that  
\begin{enumerate}
    \item $K$ is complete with respect to $\val$ (i.e. $K$ has the topology coming from $\displaystyle{\OO_K=\lim_{\leftarrow}\OO_K/\mf{m}_K^n}$), and 
    \item $k_K$ is finite. 
\end{enumerate}

\begin{exercise}
Show that $1.$ and $2.$ are equivalent to $K$ being locally compact. 
\end{exercise}

\begin{example}
The field $\Q_p$ is locally compact because it is covered by dilates of $\Z_p$:  
\[
\Q_p=\bigcup_{n \geq 1} p^{-n}\Z_p. 
\]
Recall that $\Z_p$ are compact open sets in $\Q_p$.
\end{example}

\begin{remark}
In some terminology, $\R$ and $\C$ are also referred to as local fields. 
\end{remark}

Let $L/K$ be a finite Galois extension of local fields. Then any element $\sigma \in \Gal(L/K)$ preserves $\OO_L/\OO_K$ and $\mf{m}_L/\mf{m}_K$, and thus acts on $k_L/k_K$. Hence we get maps
\[
1 \rightarrow I_{L/k} \hookrightarrow \Gal(L/K) \twoheadrightarrow \Gal(k_L/k_K)\rightarrow 1,
\]
where $I_{L/K}$ is the inertia subgroup of lecture 2. The Galois group $\Gal(k_L/k_K) \simeq \Z/d\Z$ is generated by $\text{Frob}_q$. For $L=\overline{K}$, this short exact sequence becomes 
\[
1 \rightarrow I_{\overline{K}/K} \hookrightarrow \Gal(\overline{K}/K) \twoheadrightarrow \Gal(\overline{\F_q}/\F_q)\simeq \widehat{\Z}\rightarrow 1.
\]

\noindent
{\bf Local class field theory:} There exists a canonical map\footnote{The canonical map $r_K$ is sometimes called the ``reciprocity map''.} $r_K:K^\times \rightarrow \Gal(\overline{K}/K)^{ab}$ with dense image such that $r_K$ induces an isomorphism $\widehat{K^\times} \xrightarrow{\sim}\Gal(\overline{K}/K)^{ab}$. Here $\widehat{K^\times}$ is the profinite completion of $K^\times$ (that is, the completion with respect to subgroups of finite index). Moreover, 
\begin{enumerate}
    \item the diagram 
    \[
    \begin{tikzcd}
    1 \arrow[r] & I_{\overline{K}/K}^{ab} \arrow[r, hookrightarrow] \arrow[d, equal] & \Gal(\overline{K}/K)^{ab} \arrow[r, twoheadrightarrow] & \Gal(\overline{k_K}/k_K)\simeq\widehat{\Z} \\
    1 \arrow[r] & \OO_K^\times \arrow[r, hookrightarrow] & K^\times \arrow[r, "\text{val}"] \arrow[u, "r_K"]  & \Z \arrow[u]
    \end{tikzcd}
    \]
    commutes, and 
    \item if $L/K$ is finite Galois, then 
    \[
    \begin{tikzcd}
    L^\times \arrow[r, "r_L"] \arrow[d, "\text{Norm}_{L/K}"] &\Gal(\overline{L}/L)^{ab} \arrow[d, "\text{res}"] \\
    K^\times \arrow[r, "r_K"] & \Gal(\overline{L}/K)^{ab} 
    \end{tikzcd}
    \]
    commutes. 
\end{enumerate}
\vspace{5mm}

\begin{remark} The analogous statements for the fields $\R$ and $\C((t))$ are the following:
\begin{enumerate}
    \item {\bf $K=\R$:} In the diagram
    \[
    \begin{tikzcd}
    \C^\times \arrow[r, "r_\C"] \arrow[d, "\text{Norm}_{\C/\R}"]& \Gal(\C/\C)=\{1\} \arrow[d] \\
    \R^\times \arrow[r, "r_\R"] & \Gal(\C/\R) 
    \end{tikzcd}
    \]
    the map $r_\R: -1 \mapsto $ complex conjugation is continuous and surjective. The kernel of $r_\R$ is $\R_{>0}^\times$, the set of norms coming from $\C^\times$ (i.e. the image of $\text{Norm}_{\C/\R}$).
    \item {\bf $K=\C((t))$:}\footnote{Note that $K$ does not quite fit our assumptions so LCFT does not apply, but morally it fits into the same picture.}  The map 
    \[
    r_{\C((t))}:\C((t))^\times \rightarrow \Gal \left(\overline{\C((t))}/\C((t)) \right) = \widehat{\Z}
    \]
    has dense image, so a reasonable choice is valuation $\text{val}:\C((t))^\times \rightarrow \Z$. 
\end{enumerate}
\end{remark}

It is useful to modify the Galois group $\Gal(\overline{K}/K)$ slightly. Define the {\bf Weil group} of $K$ to be the subgroup $W_K\subset \Gal(\overline{K}/K)$ of elements whose projection onto $\widehat{\Z}$ is an integral power of $\text{Frob}_q$; that is, $W_K$ fits into the short exact sequence 
\[
\begin{tikzcd}
I_{\overline{K}/K} \arrow[r, hookrightarrow] \arrow[d, equal] & W_K \arrow[r, twoheadrightarrow] \arrow[d, hookrightarrow] & \Z \arrow[d, hookrightarrow]\\
I_{\overline{K}/K} \arrow[r, hookrightarrow]  & \Gal(\overline{K}/K) \arrow[r, twoheadrightarrow] & \widehat{\Z}. 
\end{tikzcd}
\]
The purpose for this modification of the following fact: the reciprocity map $r_K$ provides an isomorphism between $K^\times$ and the abelianization of the Weil group:
\[
r_K:K^\times \xrightarrow{\simeq} W_K^{ab}. 
\]
With this, we can state the local Langlands correspondence for $\GL_n(K)$. 
\begin{theorem} ({\bf Local Langlands correspondence for $\GL_n(K)$}) (Harris-Taylor) There is a bijection 
\[
\text{Hom}_{cts}(W_K, \GL_n(\C))/_{\text{conj}} \xleftrightarrow{1:1} \left\{  \begin{array}{c} \text{irreps of $\GL_n(K)$}
\\
\text{in $\C$-vector spaces}
\end{array}   \right\}.
\]
\end{theorem}
The continuous group homomorphisms on the left hand side of this bijection are referred to as the {\bf Langlands parameters} of the corresponding $\GL_n(K)$-representations on the right. 

\begin{remark}
Actually, this is not quite correct. Instead we should consider Weil-Deligne reps on the left, and smooth admissible reps on the right. These issues will be addressed in coming lectures. 
\end{remark}

\begin{example} The $n=1$ case of this theorem is true by local class field theory:
\begin{align*}
    \text{Hom}_{cts}(W_K,\GL_1(\C))/_{\text{conj}} &= \text{Hom}_{cts}(W_K, \C^\times) \\
    &= \text{Hom}_{cts}(W_K^{ab}, \C^\times) \\
    &= \text{Hom}_{cts}(K^\times, \C^\times) \\
    &= \{ \text{irreps of $\GL_1(K)$} \}.
\end{align*}
\end{example}
 
\begin{example}
We can see explicitly that local class field theory is true for $\Q_p$. By a local version of the Jugentraum, 
\[
\Q_p^{ab} = \Q_p(\mu_\infty) = \bigcup\Q_p(\zeta_n),
\]
where $\zeta_n$ is an $n^{th}$ root of unity. Hence 
\[
\Q_p^{ab}=\Q_p(\mu_{p'}) \cdot \Q_p(\mu_{p^\infty}),
\]
where $\displaystyle{\Q_p(\mu_{p'}):=\bigcup_{p \nmid n} \Q_p(\zeta_n)}$ and $\displaystyle{\Q_p(\mu_{p^\infty}):= \bigcup_{n \geq 1} \Q_p(\zeta_{p^n})}$. As before, $\Gal(\Q_p(\zeta_{p^n})/\Q_p) = \left( \Z/p^n\Z \right) ^\times,$ so 
\[
\Gal(\Q_p(\mu_{p^\infty})/\Q_p) = \Z_p^\times.
\]
If $p \nmid n$, note that $\Q_p(\zeta_{n})/\Q_p$ is unramified, so $\Q(\mu_{p'})=:\Q_p^{ur}$ is the maximal unramified extension of $\Q_p$. Because $\displaystyle{\overline{\F_p}=\bigcup_{p \nmid n} \F_p(\zeta_n)}$, we have that 
\[
\Gal(\Q_p(\mu_{p'})/\Q_p) = \Gal(\overline{\F_p}/\F_p) = \widehat{\Z}.
\]
We conclude that 
\[
\Gal(\Q_p^{ab}/\Q_p) = \widehat{\Z} \times \Z_p^\times \simeq \widehat{\Q}_p^\times,
\]
which is exactly what is predicted by local class field theory. 
\end{example}

\subsection{Structure of Galois groups of local fields}
\label{ramification filtration}
We will finish this lecture with some remarks on the structure of local Galois groups. This section is hands-on and explicit. Morally, it should come before local class field theory. 

Let $L/K$ be a finite Galois extension of local fields, and $L \supset \OO_L \supset \mf{m}_L =(\pi_L)$, $K \supset \OO_K \supset \mf{m}_K =(\pi_K)$ the respective rings of integers, maximal ideals, and uniformizers. As before, there is a corresponding extension of residue fields $k_L/k_K$. We have a short exact sequence
\[
1 \rightarrow I_{L/K} \rightarrow \Gal(L/K) \twoheadrightarrow \Gal(k_L/k_K) \simeq \Z/n\Z \rightarrow 1,
\]
where $I_{L/K}=\{\sigma \mid \sigma \text{ acts trivially on }k_L\}$, and $\Gal(k_L/k_K)$ is generated by the canonical generator Frob. Note that $\Gal(L/K)$ preserves $\OO_K$, hence $\mf{m}_K$, hence the valuation $v_K$, hence acts on $\OO_K/\mf{m}_K^j$, hence acts continuously on $L$. 

\begin{lemma}
\label{key lemma} (Key lemma) Any $\sigma \in I_{K/L}$ is determined by its action on $\pi_L$. 
\end{lemma}

\begin{exercise}
Prove Lemma \ref{key lemma}. ({\em Hint}: use the fact that any $\sigma \in \Gal(L/K)$ is automatically continuous.) 
\end{exercise}

Set $I:=I_{L/K}$, $I_0:=I$, and $I_j:=\{ \sigma \in I \mid \sigma(\pi) \pi^{-1} \in 1 + \mf{m}_L^j \}$ for $j \geq 1$. 
\begin{proposition}
This defines a filtration
\[
I=I_0 \supset I_1 \supset I_2 \supset \cdots 
\]
of $I$ by normal subgroups. Moreover, 
\begin{enumerate}
    \item This is a finite filtration; i.e. $I_m=\{1\}$ for large enough $m$. 
    \item We have natural injections 
    \begin{align*}
        I_0/I_1 &\xhookrightarrow{\sigma(\pi)\pi^{-1}} k_L^\times, \\
        I_j/I_{j+1} &\xhookrightarrow{\hspace{11mm}} (1+\mf{m}_L^j)/(1+\mf{m}_L^{j+1}) \simeq k_L 
    \end{align*}
    for $j\geq 1$. In particular, $I$ is solvable, $I_0/I_1$ is of order prime to $p$, and $I_1$ is the Sylow $p$ subgroup of $I$. 
\end{enumerate}
\end{proposition}
This filtration is called the {\bf ramification filtration of $I$}. The proof is an easy exercise. 
\begin{definition}
We say that $L/K$ is {\bf tamely ramified} if $I_1=\{1\}$, and $L/K$ is {\bf unramified}\footnote{Note that a quirk of this terminology is that unramified is tamely ramified. It is strange, but we will just have to get used to it.} if $I_0=\{1\}$.
\end{definition}

\begin{remark}
This agrees with our earlier notion of unramified from Lecture 2.
\end{remark}

%% file: lecture-7.tex
\section{Lecture 7: Representation theory of $p$-adic groups}
\label{lecture 7}

\subsection{Tying up some loose ends}

We start today's lecture by tying up some loose ends from previous weeks. Recall that two lectures ago, we discussed how looking for a description of $\Gal(\overline{\Q}/\Q)$ is misguided because $\Gal(\overline{\Q}/\Q)$ is only a ``group up to conjugacy,'' since its definition requires a choice of $\overline{\Q}$. There is an analogy for this idea that Geordie learned from Kevin Buzzard which might be more familiar to us. Let $X$ be a path-connected space. A choice of base point $x \in X$ yields the fundamental group $\pi_1(X,x)$. Another choice of base point $y \in X$ yields the isomorphic group $\pi_1(X,y)$. 
\begin{center}
     \includegraphics[scale=0.4]{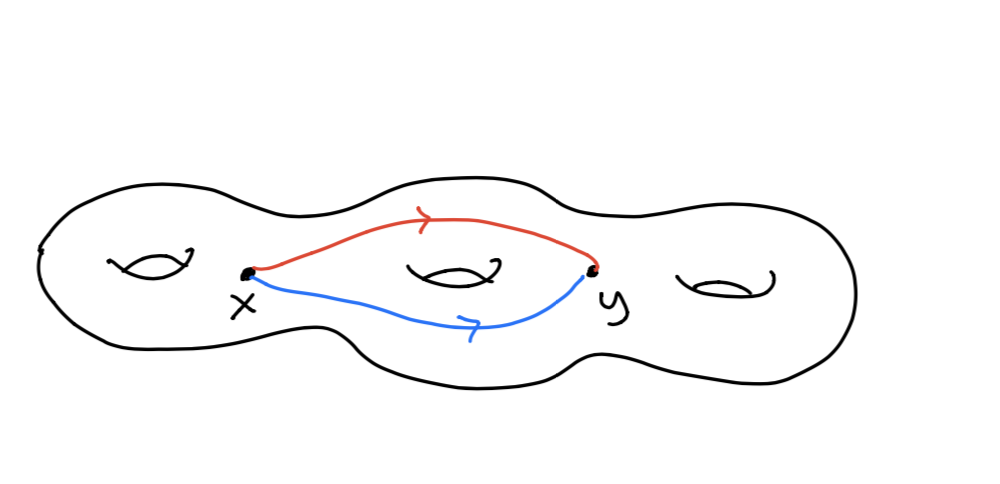}
\end{center}
An isomorphism $\pi_1(X,x) \simeq \pi(X,y)$ requires a choice of path from $x$ to $y$ in $X$. Such a choice of path is {\em not canonical.} Grothendieck taught us an analogue of this for extensions of fields. 
\begin{align*}
    \Q &\leftrightarrow \text{ ``\'{e}tale site''} \Spec{\Q} \text{ (something like a space)} \\
    \text{choice of }\overline{\Q} &\leftrightarrow \text{choice of ``base point'' of }\Spec{\Q} 
\end{align*}
Then the \'{e}tale fundamental group $\pi_1^{\text{\'{e}t}}(\Spec{\Q},\Q)=\Gal(\overline{\Q}/\Q)$. 

\vspace{5mm}
\noindent
{\bf The punchline:} The fundamental group $\pi_1(X,x)$ depends on base point $x$, but $\Rep{\pi_1(X,x)}\simeq\{\text{local systems on $X$}\}$ is canonical! Transporting this statement via Grothendieck's analogy we see that, although the absolute Galois group is not defined canonically, its category of (continuous) representations is. It is this category that the Langlands correspondence tries to understand.
\vspace{5mm}

Last lecture we stated the {\bf local Langlands correspondence for $\GL_n$}: Fix a local field $K$ (i.e. $K$ is a finite extension of $\Q_p$ or $K \simeq \F_q((t))$). There is a canonical bijection 
\begin{align*}
    \left\{ \begin{array}{c} \text{ cts reps of $W_K$} \\ \text{ in }\GL_n(\C) \end{array} \right\} _{/\text{iso}} \xleftrightarrow{1:1} \left\{ \begin{array}{c} \text{ irred blah }\\ \text{ reps of }\GL_n(K) \end{array} \right\} _{/\text{iso}}.
\end{align*}
Here $W_K$ is the {\bf Weil group} of the field $K$. Last week we showed why this follows from local class field theory for $n=1$. However, this statement is not quite precise. On the left hand side we need to consider {\bf Weil-Deligne representations} of $W_K$, and on the right hand side we need to establish exactly what conditions are captured by ``blah''. We'll keep stating versions of this theorem every lecture until we converge on something correct.

\subsection{The no small subgroups argument}

Our final piece of housekeeping is the {\bf no small subgroups argument}. This is a very useful fact that has not fit in naturally to our story so far, so we will slot it in here. 

\begin{definition}
A topological group $G$ has {\bf no small subgroups} if there exists a neighborhood $U$ of the identity in $G$ such that any subgroup contained in $U$ is trivial.
\end{definition}

\begin{example} Here are some examples of groups with no small subgroups. 
\begin{enumerate}
    \item The circle group $S^1$. 
    \item Discrete groups (e.g. finite groups). We can take $U=\{id\}$. 
    \item The real numbers $\R$ and the complex numbers $\C$. (Powers of any non-identity element move far away from the identity.) 
    \item Any Lie group $G$. (Use that $\exp: \Lie{G} \rightarrow G$ is a local diffeomorphism and $\exp(mg)=\exp(g)^m.$)
    \item Any topological subgroup of a group with no small subgroups has no small subgroups (e.g. $\Q/\Z \hookrightarrow S^1$ has no small subgroups). 
\end{enumerate}
\end{example}
 
\begin{remark} In contrast, profinite groups have ``many small subgroups,'' because a basis of neighbourhoods of the identity consists of subgroups of finite index.
\end{remark}

\begin{lemma}
\label{no small subgroups}
Let $\Gamma$ be a profinite group and $G$ a topological group with no small subgroups. Then any continuous group homomorphism $\varphi:\Gamma \rightarrow G$ has finite image. 
\end{lemma}
\begin{proof}
Let $U \subset G$ be an open neighborhood of the identity containing no nontrivial subgroups, and let $\varphi:\Gamma \rightarrow G$ be a continuous group homomorphism. Then $\varphi^{-1}(U) \subset \Gamma$ is open. Since $\Gamma$ is profinite, there exists a normal subgroup $N \subset \varphi^{-1}(U)$ such that $G/N$ is finite. The image $\varphi(N) \subset U$ is a subgroup, so $\varphi(N)=\{1\}$ since $G$ has no small subgroups. Hence $\Gamma$ factors through a finite group, $\varphi:\Gamma \rightarrow \Gamma/N \rightarrow G$.  
\end{proof}

\noindent
{\bf The Moral:} 
\begin{align*}
    \left\{ \begin{array}{c} \text{Fractal-like objects} \\ \text{($p$-adic groups, Galois groups)} \end{array} \right\} \cap \left\{ \begin{array}{c} \text{Euclidean-type objects} \\ \text{(Lie groups)} \end{array} \right\} = \{ \text{finite groups} \} 
\end{align*}

\vspace{5mm}
\noindent 
A consequence of this is that we can not draw any good pictures of $\Z_p$ in $\C$ (or for that matter in any Lie group) which respect the addition or multiplication structure. 

This moral gives us a new perspective of the local Langlands correspondence. The ``no small subgroups'' lemma implies that the left hand side of the LLC consists (roughly) of a collection of finite subgroups of $\GL_n(\C)$, along with surjections from a Galois-group-type object to the subgroups. So very roughly, the LLC provides a classification of irreducible admissible representations of a Lie group over a local field by certain finite subgroups of $\GL_n(\C)$.

\subsection{You could have guessed the LLC for $\GL_2$!}

The goal is this section is to give a heuristic explanation for the LLC. {\em Warning}: this is not precise! Everything we say here will have to be tweaked later. Geordie learned this perspecive from a series of lectures by Dipendra Prasad in Russia. 

\vspace{5mm}
\noindent
{\bf Starting place:} Say we wanted to guess the representation theory of $\GL_n(\Q_p)$. What would we do? 

\vspace{5mm}
\noindent
{\bf Step 1}: We might start by figuring out the representation theory of finite reductive groups. For example, let $G=\SL_2(\F_q)$. There are two maximal tori in $G$, up to conjugacy: 
\[
T_s=\left\{ \bp a & 0 \\ 0 & a^{-1} \ep \mid a \in \F_q^\times \right\}\simeq \Z/(q-1)\Z, \text{ the ``split torus,'' and}
\]
\[
T_a = \{\lambda \in \F_{q^2}^\times \subset \GL_2(\F_q) \mid \Norm(\lambda)=1 \}\simeq \Z/(q+1)\Z, \text{ the ``anisotropic torus''.} 
\]
In the definition of $T_a$ above, we are using the fact that $\GL_2(\F_q)$ is the group of invertible linear transformations of the $\F_q$-vector space $\F_q^2$, so 
\[\F_{q^2}^\times \subset \GL_{\F_q}(\F_q^2)=\GL_2(\F_q).
\]
Roughly, 
\begin{align*}
    \left\{ \begin{array}{c} \text{irred reps of }\\ \text{$\SL_2(\F_q)$ over $\C$} \end{array} \right\} \xleftrightarrow{1:1} \begin{array}{c} \text{ } \\ \left\{ \chi: T_s \rightarrow \C^\times \right\}_{/ \chi \sim \chi^{-1}} \\  \text{ ``principal series''} \end{array} \bigsqcup \begin{array}{c} \text{ } \\ \left\{\theta:T_a \rightarrow \C^\times \right\}_{/\theta \sim \theta^{-1}} \\ \text{ ``discrete series''} \end{array}
\end{align*}

Let us check that we are not too far off by doing a count. Irreducible representations of a finite group are in bijection with conjugacy classes, and conjucagy classes are roughly in bijection with characteristic polynomials, so the sizes of the sets above are roughly 
\begin{align*}
\begin{array}{c} \# \text{ characteristic polynomials }\\
\text{of elements in }\SL_2(\F_q)\end{array}
= |\{ x^2 + ax + 1 \; | \; a \in \F_q \}|=
q= \frac{q-1}{2} + \frac{q+1}{2}.
\end{align*}
For more details and a careful construction of the irreducible representation of $\SL_2(\F_q)$, see the notes from Joe Baine's talks on the Informal Friday Seminar webpage. The upshot is that we obtain almost all irreducible representations of $\SL_2(\F_q)$ through some ``induction'' from characters of the two conjugacy classes of tori.
(Note that the details are much more complicated as there is no actual induction functor. In the setting of finite reductive groups we use Deligne-Lusztig induction.)  

\vspace{5mm}
\noindent
{\bf Step 2}: Once we have a good idea of the representation theory of finite reductive groups, a next natural step might be to understand the representation theory of real reductive groups. For example, let $G=\SL_2(\R)$. Again, there are two conjugacy classes of maximal tori: the ``split torus''
\[
T_s =\left\{ \bp a & 0 \\ 0 & a^{-1} \ep \mid a \in \R^\times \right\} \simeq \R^\times,
\]
and the ``anisotropic torus'' 
\[
T_a = \left\{ \bp \cos \theta & \sin \theta \\ - \sin \theta & \cos \theta \ep \right\} \simeq SO_2. 
\]
Something similar happens in this setting to what we saw with the finite reductive groups. Roughly, 
\begin{align*}
    \left\{ \begin{array}{c} \text{irred admissible } \\ \text{reps of }\SL_2(\R) \end{array} \right\} \xleftrightarrow{1:1} \begin{array}{c} \text{ } \\ \left\{ \begin{array}{c} \text{cts characters }\\ \text{ of $T_s \simeq \R \times \Z/2\Z$} \end{array} \right\} \\ \text{``principal series''} \end{array} \bigsqcup
    \begin{array}{c} \text{ } \\ \left\{ \begin{array}{c} \text{cts characters }\\ \text{ of $T_a \simeq SO_2 \simeq S^1$} \end{array} \right\} \\ \text{``discrete series''} \end{array} 
\end{align*}
So again we see that, roughly, irreducible representations are all obtained by ``inducing'' characters of conjugacy classes of tori.

\vspace{5mm}
\noindent
{\bf Step 3:} We dream that something similar might be true for representations of $p$-adic groups. Let $G=\GL_2(K)$ for a local field $K$. By descent, we have the following relationship:
\begin{align*}
    \left\{ \begin{array}{c} \text{conjugacy classes of }\\ \text{max'l tori in $\GL_2(K)$} \end{array} \right\} &\leftrightarrow \left\{ \begin{array}{c} \text{semisimple $K$-algebras} \\ \text{$L$ s.t. $\dim_KL=2$} \end{array} \right\} \\
    L^\times &\leftrightarrow L
\end{align*}
There are two cases:
\begin{enumerate}
    \item Split torus: $L \simeq K \times K$, so maximal torus is of the form $L^\times \simeq K^\times \times K^ \times$.
    \item Anisotropic torus: $L/K$ degree $2$ extension, so maximal torus is of the form $L^\times$. 
\end{enumerate}
Applying our analogy from earlier, we might expect 
\begin{align*}
    \left\{ \begin{array}{c} \text{irred blah }\\ \text{reps of $\GL_2(K)$} \end{array} \right\} &\xleftrightarrow{\text{roughly }1:1} \left\{ \begin{array}{c} \text{pairs of characters} \\ \chi_1, \chi_2: K^\times \rightarrow \C^\times \end{array} \right\} \bigsqcup \left\{ \begin{array}{c} \text{characters } \theta:L^\times \rightarrow \C^\times \\ \text{where $L/K$ is degree $2$} \end{array} \right\}.  
\end{align*}
Now, local class field theory tells us that $W_K^{ab} \simeq K^\times$ and $W_L^{ab} \simeq L^\times$. Moreover, $W_L \subset W_K$ is an index $2$ subgroup, so
\begin{align*}
    \left\{ \begin{array}{c} \text{irred blah }\\ \text{reps of $\GL_2(K)$} \end{array} \right\}\xleftrightarrow{\text{roughly }1:1}\left\{ \chi_1 \otimes \chi_2:W_K \rightarrow \GL_2(\C) \right\} \sqcup \left\{ \Ind_{W_L}^{W_K}(\theta): W_L \rightarrow \GL_2(\C) \right\}. 
    \end{align*}
It turns out that our dream is a reality:

\vspace{5mm}
\noindent 
{\bf Fact:} If $p \neq 2$, all continuous representations of $W_K$ are either of the form $\chi_1 \otimes \chi_2$ or $\Ind_{W_L}^{W_K}(\theta)$ as above. 
\vspace{5mm}

So we guessed LLC for $\GL_2(K)$! Though again, let us emphasize that this is not actually the correct version of the correspondence (it is for example not compatible with taking duals). However it will not take too much effort to make this into a correct statement next lecture.

\begin{remark}
For $p=2$, the matching still works, but there are more objects on both sides. 
\end{remark}

\subsection{Basic representation theory of $p$-adic groups}

Let $K$ be a local field. Then $\GL_n(K)$ is a topological group, with a basis of open neighborhoods of $id$ given by 
\[
K_j = \left\{ g \in \GL_n(\OO_K) \mid g = id \mod \mf{m}_K^j \right\}.
\]
Note that $\GL_n(\OO_K)/K_j \simeq \GL_n(\OO_K/\mf{m}_K^j)$ is a finite group. 

For example, if $K=\Q_p$, we have a natural surjective map 
\[
\GL_n(\Q_p) \supset \GL_n(\Z_p) \xrightarrow{\varphi_j} \GL_n(\Z/p^j\Z) 
\]
for all $j\in \Z_:\geq 0$, and $K_j=\varphi_j^{-1}(id)$. 

\begin{remark}
\begin{enumerate}
    \item $K_j \subset K_0$ is normal. 
    \item $K_0$ is a maximal compact subgroup. 
\end{enumerate}
\end{remark}

\begin{exercise}
 Let $\pi \in \OO_K \subset K$ be a uniformizer. 
\begin{enumerate}
    \item Establish the Bruhat decomposition:
    \begin{equation}
    \GL_n(K)=\bigsqcup_{\begin{array}{c} \lambda_1 \geq \lambda_2 \geq \ldots \geq \lambda_n \\ \lambda_i \in \Z \end{array}} \GL_n(\OO_K)\bp \pi^{\lambda_1} & & & &  \\ & \pi^{\lambda_2} & & &  \\ & & \pi^{\lambda_3} & & & \\ & & & \cdots & & \\ & & & & &  \pi^{\lambda_n} \ep \GL_n(\OO_K)   \tag{$\ast$}
    \end{equation}
    ({\em Hint}: Gaussian elimination.)
    \item Use ($\ast$) to classify the subgroups $\GL_n(\OO_K)\subset H \subset \GL_n(K)$. 
    \item Hence or otherwise, show that $\GL_n(\OO_K)$ is a maximal compact subgroup of $\GL_n(K)$. 
\end{enumerate}
\end{exercise}

\begin{example}
Consider $\GL_1(\Q_3) = \Q_3^\times = \Z \times \Z_3^\times.$  Recall our picture of the $3$-adics from Example \ref{3-adics}. A picture of the maximal compact subgroup $K_0$ of this group is: 
\begin{center}
     \includegraphics[scale=0.5]{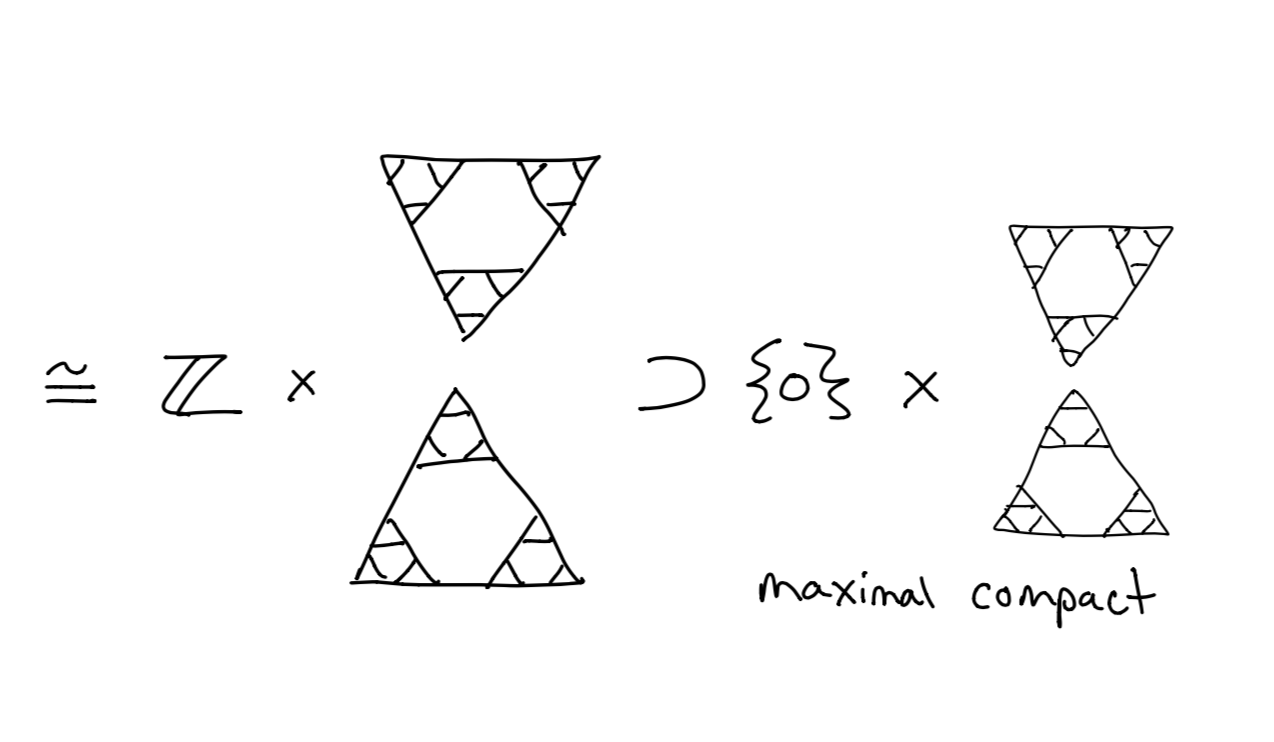}
\end{center}
\end{example}
A space is {\bf totally disconnected} if every point admits a family of compact open neighborhoods; e.g. $\GL_n(K)$ is totally disconnected because each $K_j$ is compact open. Let $G$ be a totally disconnected topological group and $V$ a vector space over a field $\mathbb{K}$ of characteristic $0$. We give $V$ the discrete topology. 
\begin{definition}
A representation $\rho:G \rightarrow \GL(V)$ is 
\begin{enumerate}
    \item {\bf smooth} if for all $v \in V$, $\stab_Gv$ is open, and 
    \item {\bf admissible} if for all open $K \subset G$, $V^K$ is finite-dimensional.  
\end{enumerate}
\end{definition}

\begin{example}
\label{smooth admissible}
\begin{enumerate}
    \item The trivial representation $\mathbb{K}$ is smooth and admissible, $\mathbb{K}^\infty$ is smooth but not admissible. 
    \item The standard representation of $\GL_n(K)$ on $K^n$ is not smooth, as $\stab_{\GL_n(K)}{v}$ is not open for $v \neq 0$.
    \item The group $G=(\Z_p, +)$ acts on the vector space $\mc{F}=\{ \varphi:\Z_p \rightarrow \C \mid \varphi \text{ is locally constant} \}$ in the natural way, forming the ``smooth regular representation''. (This is the $p$-adic analogue of $L^2(G)$.) We claim that this representation is smooth and admissible.
    \begin{itemize}
        \item {\bf Smooth:} Let $\varphi \in \mc{F}$. Then for all $x \in \Z_p$, there is a neighborhood $U_x$ such that $\varphi|_{U_x}$ is constant. This forms a covering of $\Z_p$ by open neighborhoods of the form $U_x=x+p^{n_x}\Z_p$. Since $\Z_p$ is compact, there exists a finite subcovering $U_{x_1}, \ldots , U_{x_m}$. Then $\varphi$ is fixed by $p^n\Z_p$, where $n=\Max\{n_i\}$, so the the stablizer of $\varphi$ is open, hence the representation is smooth. \item {\bf Admissible:} A basis of open neighborhoods of $0$ is given by $p^m\Z_p$, $m \geq 0$. Then 
        \begin{align*}
            \mc{F}^{p^m\Z_p} &= \{\varphi \mid \varphi \text{ is constant on $p^m\Z_p$-orbits} \} \\
            &= \{ \varphi: \Z/p^m\Z \rightarrow \C \}
        \end{align*}
        is finite dimensional, so the representation is admissible. 
    \end{itemize}
    \item (The most important example!) Recall that $\mathbb{P}^1\C$ is covered by the compact sets $D_{\leq 1} = \{ z \mid |z|\leq 1 \}$ and $D_{\leq 1}^{-1}$, so $\mathbb{P}^1\C$ is compact:
    \begin{center}
     \includegraphics[scale=0.5]{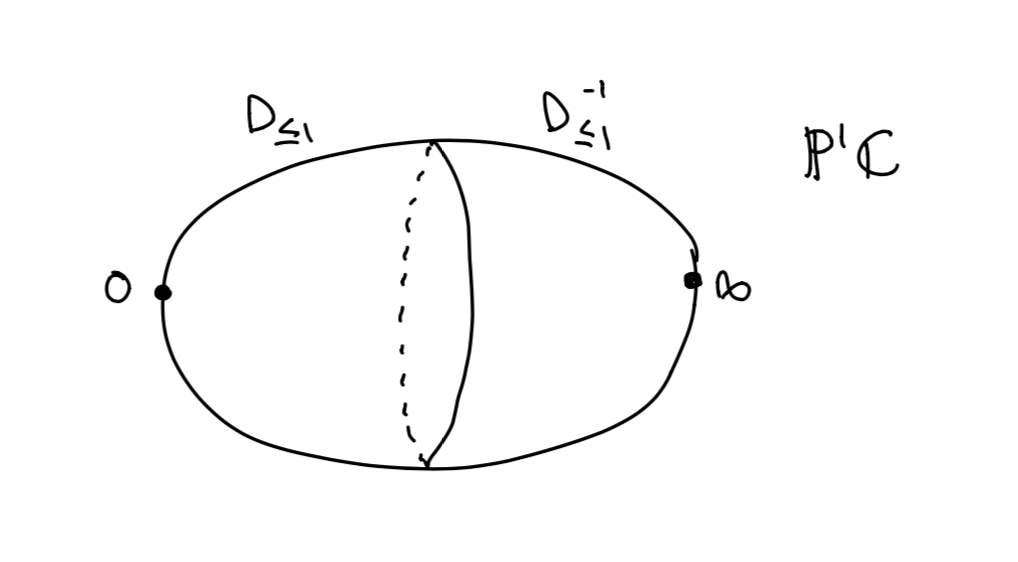}
\end{center}
    Similarly, $\mathbb{P}^1K=K\cup \{\infty\}$ is covered by the compact sets $D_{\leq 1} = \{ z \in K \mid |z|_p \leq 1 \} = \OO_K$ and $D_{\leq 1}^{-1}$, so $\mathbb{P}^1K$ is compact:
    \begin{center}
     \includegraphics[scale=0.5]{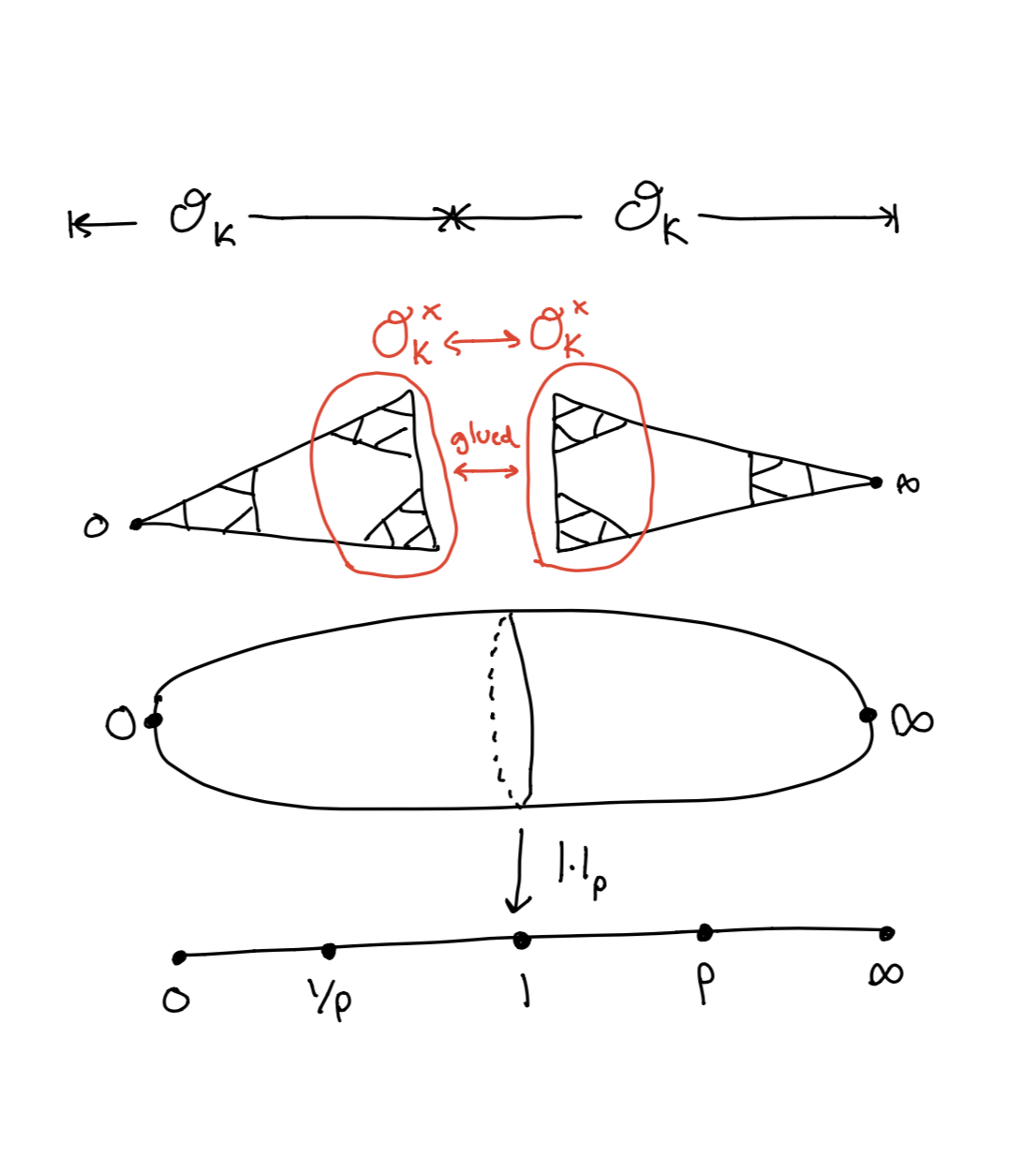}
\end{center}
The vector space 
\[
I=\{f: \mathbb{P}^1K \rightarrow \C \mid f \text{ is locally constant}\} 
\]
admits a natural $\GL_2(K)$-action, and the same argument as in the previous example (using compactness) shows that $I$ is a smooth, admissible representation of $\GL_2(K)$. In fact, 
\[
\text{constant functions} \hookrightarrow I \twoheadrightarrow St
\]
Where $St:=I/\{\text{constant functions}\}$ is the {\bf Steinberg module}. The module $St$ is irreducible (exercise, might be hard with current technology!). 
\end{enumerate}
\end{example}

\begin{exercise}
Show that any smooth finite dimensional representation of $\GL_n(K)$ factors over $\det:\GL_n(K)\rightarrow K^\times$. ({\em Hint}: the kernel of a smooth finite dimensional representation is a finite intersection of stabilizers of a basis, so it must be open and normal, hence contain $\SL_n(K)$).  
\end{exercise}
\begin{exercise}
The representation $\mc{F}'=\{\varphi:\Q_p \rightarrow \C \mid \varphi \text{ is locally constant} \}$ of $\Z_p$ is {\em not} admissible or smooth. 
\end{exercise}

%% file: lecture-8.tex
\section{Lecture 8: Precise statement of local Langlands for $\GL_2$, $p \neq 2$ }
\label{lecture 8}

\subsection{Basic representation theory of $p$-adic groups, continued}
We pick up where we left off in the previous lecture. Let
\[
\mf{m}_K \subset \OO_K \subset K 
\]
be the maximal ideal in the ring of integers of a local field. We are interested in the representation theory of the group $\GL_n(K)$. (Or, more generally, the representation theory of any totally disconnected group $G$, but for concreteness we will work with $\GL_n$.) 

Recall that the sets 
\[
K_j=\{ g \in \GL_n(\OO_K) \mid g = id \mod \mf{m}_K^j \}
\]
form a basis of open neighborhoods of $id \in \GL_n(K)$. In addition to being open neighborhoods of the identity, the $K_i$ are {\em subgroups} of $\GL_n(K)$. (Note the existence of such subgroups which form a basis for open neighborhoods of the identity is only possible because $\GL_n(K)$ is a totally disconnected group; a Lie group could not have such a family of subgroups because Lie groups have no small subgroups.)  

Last week we saw that $K_0$ is a maximal compact subgroup of $\GL_n(K)$. Let $V$ be a representation of $\GL_n(K)$. Because $K_0 \supset K_1 \supset K_2 \supset \cdots$,  we have a chain 
\[
V^{K_0} \subset V^{K_1} \subset V^{K_2}\cdots 
\]
If $V$ is smooth, each vector lies in $V^U$ for some open $U \subset G$, hence lies in some $V^{K_i}$. So the filtration is exhaustive. If $V$ is admissible, each $V^{K_i}$ is finite dimensional. 

Because $K_i \subset K_0$ is normal, the subspace $V^{K_i}$ is stable under action by $K_0$. The subgroup $K_i \subset K_0$ acts trivially on $V^{K_i}$, so the $K_0$-action factors through the finite group $K_0/K_i \simeq \GL_n(\OO_K/\mf{m}_K^i)$; e.g. for $K=\Q_p$, the $K_0$ action on $V^{K_i}$ factors through $\GL_n(\Z/p^i\Z)$. (The key point here is that $K_i \subset K_0$ is {\em normal}, so the quotient $K_0/K_i$ is a {\em group}.) Since representations of finite groups are completely reducible, we have a decomposition
\[
V^{K_i}=\bigoplus_{\rho \in \widehat{K_0/K_i}} V^{K_i}(\rho),
\]
where $V^{K_i}(\rho)$ is the $\rho$-isotypic component of $V^{K_i}$; that is, $V^{K_i}(\rho)$ is the direct sum of all irreducible subrepresentations of $V^{K_i}$ which are isomorphic to $\rho$. Passing to the limit, we obtain a decomposition 
\[
V = \bigoplus_{\rho \in \widehat{K_0}} V(\rho).
\]
Here $\widehat{K_0}$ denotes all representations of $K_0$ which factor over some quotient $K_0/K_i$.
\begin{lemma}
The representation $V$ is admissible if and only if each isotypic component $V(\rho)$ in the decomposition above is finite-dimensional.  
\end{lemma}
\begin{proof}
Assume that $V(\rho)$ is infinite dimensional for some $\rho \in \widehat{K_0}$. By definition, $\rho$ factors through $K_0/K_i$ for some $i$. Hence, $V(\rho) \subset V^{K_i}$ is an infinite dimensional subspace and $V$ is not admissible.

To prove the opposite implication, assume that each $V(\rho)$ is finite-dimensional. For each $i$, we have a decomposition 
\[
V^{K_i} = \bigoplus_{\rho \in \widehat{K_0}} V(\rho)^{K_i}.
\]
But since $V(\rho)$ is the direct sum of irreducible representations which are isomorphic to $\rho$, we have
\[
V(\rho)^{K_i} = \begin{cases} 0 &\text{ if } \rho|_{K_i} \neq \text{triv}, \\ V(\rho) &\text{ otherwise}.
\end{cases}
\]
Hence 
\[
V^{K_i} = \bigoplus_{\begin{array}{c} \rho \in \widehat{K_0} \\ \rho|_{K_i} = \text{triv} \end{array}} V(\rho).
\]
Since $K_0/K_i$ is a finite group, there are only finitely many representations $\rho \in \widehat{K_0}$ which factor through $K_0/K_i$ for any fixed $i$, so decomposition above is a finite direct sum of finite-dimensional representations, hence $V$ is admissible. 
\end{proof}

\begin{remark}
This is like the theory of $K$-finite vectors in representation theory of real Lie groups. A big difference is that the representation theory of, for example $\GL_m(\Z/p^n\Z)$ for large $m$ and $n$ is extremely complicated, whereas we know the representation theory of compact Lie groups rather well (highest weights, etc.). 
\end{remark}

\begin{example}
\begin{enumerate}
    \item Consider the $\Z_p$-representation $\mc{F}=\{\varphi: \Z_p \rightarrow \C \mid \varphi \text{ is locally constant}\}$ from Example \ref{smooth admissible}.3. For an open neighborhood $p^m\Z_p$ of the identity, the invariants are 
    \begin{align*}
    \mc{F}^{p^m\Z_p} &=\{ \varphi \mid \varphi \text{ constant on }p^m\Z_p \text{ orbits}\} \\
    &= \text{ regular representation of } \Z_p/p^m\Z_p. 
    \end{align*}
    Hence, 
    \[
    \mc{F} = \bigoplus_{\text{continuous} \atop \chi:\Z_p \rightarrow \C^\times} \C_\chi .
    \]
    \item Consider the $\GL_2(K)$-representation
    \[
    I=\{f: \PP^1K \rightarrow \C \mid f \text{ is locally constant}\} 
    \]
    from example \ref{smooth admissible}.4. Here 
    \[
    I^{K_n} = \{ \varphi: \PP^1 (\mc{O}_K/\mf{m}_K^n)\rightarrow \C \},
    \]
    so 
    \[
    I = \lim_{\rightarrow}\C[\PP^1(\mc{O}_K/\mf{m}_K^n)].
    \]
\end{enumerate}
\end{example}

Let $V$ be a representation of $\GL_n(K)$. A map $\xi:V\rightarrow \C$ is {\bf smooth} if $\stab_{\GL_n(K)}\xi$ is open. Define the {\bf smooth dual}
\[
\widehat{V}=\{ \text{smooth vectors }\xi:V \rightarrow \C \}.
\]
\begin{lemma}
Assume $V$ is a smooth representation of $\GL_n(K)$. If $V=\bigoplus_{\rho \in \widehat{K_0}} V(\rho)$, then $\widehat{V}=\bigoplus_{\rho \in \widehat{K_0}} V(\rho)^\ast$.
\end{lemma}
In particular, if $V$ is smooth and admissible, then so is $\widehat{V}$, and $V \xrightarrow{\sim} \widehat{\widehat{V}}.$
\begin{proof}
The map $\xi:V\rightarrow \C$ is smooth if and only if $\xi$ vanishes on all but finitely many $V(\rho)$. The lemma follows.  
\end{proof}

The goal for the remainder of this lecture will be to give a bird's eye view on the smooth admissible representations of $\GL_1(K)$ and $\GL_2(K)$. But first, we need a digression on norms. 

\subsection{Canonical norms}
\label{canonical norms}

Recall that to make the product formula of Section \ref{global class field theory} hold, we define three types of equivalence classes of multiplicative norms (``places,'' denoted by $v$) on a local field $K$: 
\begin{itemize}
    \item {\bf finite places:} $|x|_v:=\left( \# \OO_K/\mf{p}\right)^{-\val_p(x)}$ for some prime $\mf{p} \subset \OO_K$, 
    \item {\bf real places:} $|x|_v:=|i(x)|$ for some real embedding $i:K \hookrightarrow \R$, and
    \item {\bf complex places:} $|x|_v:=|i(x)|^2$ for some pair of conjugate embeddings $i:K \hookrightarrow \C$ not landing in $\R$.
\end{itemize}
Different normalizations would also yield multiplicative norms, but we chose the ones above to make the product formula 
\[
\prod_{\text{places }v} |x|_v=1
\]
for $x \in K^\times$ holds. For example, if $K=\Q_p$, $|p|=\epsilon$ gives a norm for any $0 < \epsilon < 1$, so why do we choose $|p|=1/p$? In some sense, this choice is justified by the product formula, but it is still a little mysterious.

Tate made the following observation which further justifies this choice. For a place $v$, the completion $K_v$ is is locally compact. Let $\mu$ be the additive Haar measure on $K_v$. The measure $\mu$ is unique up to a scalar. Define 
\[
|x|_v=\text{ factor by which $x \cdot$ scales the Haar measure};
\]
i.e., $|x|_v=\frac{\mu(x \cdot A)}{\mu(A)}$ for $A \subset K_v$ measurable and $0 < \mu(A) < \infty$. 
\begin{example}
\begin{enumerate}
    \item $K_v=\R$: For $x \in \R$, $|x|_v=\frac{\mu(x[0,1])}{\mu([0,1])} = \mu([0,x])=|x|$. 
    \item $K_v=\C$: For $z \in \C$, and 
    \begin{center}
     \includegraphics[scale=0.5]{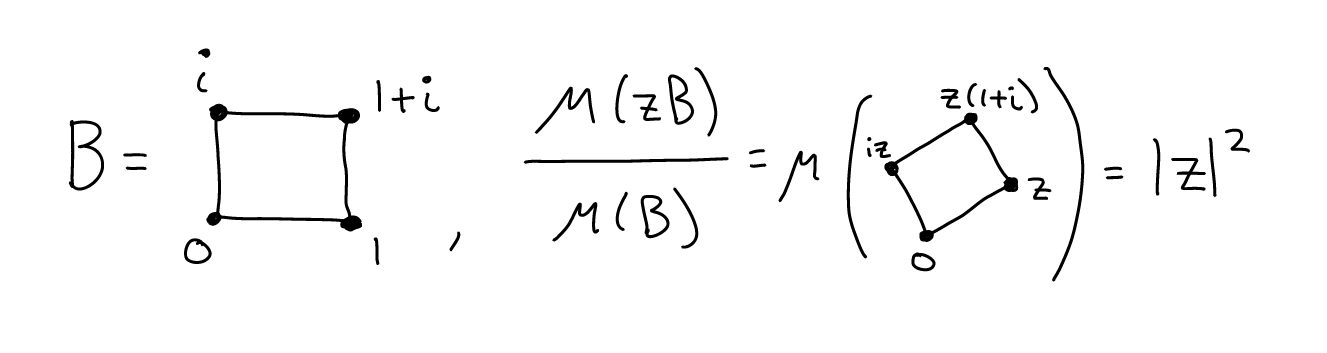}
    \end{center}
    \item $K=\Q_p$: Recall that $\displaystyle{\Z_p=\bigsqcup_{0 \leq m < p} m + p \Z_p}$, so $p\mu(p\Z_p)=\mu(\Z_p)$. Hence,
    \[
    |x|_v=\frac{\mu(p \Z_p)}{\mu(\Z_p)} = \frac{1}{p}. 
    \]
\end{enumerate}
\end{example}
From now on, whenever we consider a norm on a locally compact field, we will always consider this canonical norm, denoted $| \cdot |$. 

\subsection{Smooth admissible representations of $\GL_1(K)$}

Let $V$ be a smooth admissible representation of $\GL_1(K)=K^\times$. Since $V$ is smooth admissible, 
\[
V=\bigcup V^{K_i}
\]
and each $V^{K_i}$ is finite-dimensional. Furthermore, since $K^\times$ is abelian, each of the subgroups $K_j:=1+\mf{m}_K^j \subset \OO_K^\times$ is normal in $K^\times$, and the group 
\[
K^\times/K^j \simeq \Z \times (\OO_K/\mf{m}_K^j)^\times 
\]
acts on $V^{K_j}$. Hence if $V$ is irreducible, $V$ is one-dimensional and determined by a character of the form $|\cdot|^c \chi:K^\times \rightarrow \C$, where $c \in \C$ and $\chi:\OO_K^\times \rightarrow \C$ is a continuous character. 

\begin{remark} The category of smooth admissible representations of $\GL_1(K)$ is not semisimple. For example, the representation 
\[
x \mapsto \bp 1 & \log|x| \\ 0 & 1 \ep 
\]
is a smooth, two-dimensional admissible representation which is not semisimple. 
\end{remark}

\subsection{Smooth admissible representations of $\GL_2(K)$}

Recall from our heuristic description of last lecture that we expect roughly two types of representations of $\GL_2(K)$: ``principal series'' representations coming from a split torus, and ``cuspidal'' representations coming from an anisotropic torus. 

 Let $B \subset \GL_2(K)$ be the subgroup of upper triangular matrices. Given continuous characters $\chi_1, \chi_2: K^\times \rightarrow \C$, define 
\begin{align*}
I(\chi_1, \chi_2):=\{ \varphi:\GL_2(K) \rightarrow \C \mid &\varphi \text{ loc. const., and } \varphi \left( \bp a & b \\ 0 & d \ep \cdot g \right) = \chi_1(a) \chi_2(d) |\frac{a}{d}|^{1/2}\varphi(g)\\
&\text{ for all } \bp a & b \\ 0 & d \ep \in B \}. 
\end{align*}

\begin{example}
\begin{align*}
I(|\cdot|^{-1/2}, |\cdot|^{1/2} )&= \{\varphi: \GL_2(K) \rightarrow \C \mid \varphi \text{ loc. const. and } \varphi\left( \bp a & b \\ 0 & d \ep \cdot g\right) = \varphi(g) \} \\
&= \{ \varphi: \PP^1K\simeq G/B \rightarrow \C \mid \varphi \text{ locally constant} \} 
\end{align*}
We saw last time that $I(|\cdot|^{-1/2}, | \cdot | ^{1/2})$ is smooth and admissible.
\end{example}
The representations $I(\chi_1, \chi_2)$ formed in this way are called {\bf principal series representations}.

\begin{theorem}
\begin{enumerate}
    \item For all $\chi_1, \chi_2$, $I(\chi_1, \chi_2)$ is smooth and admissible. 
    \item $\widehat{I(\chi_1, \chi_2)} \simeq I(\chi_1^{-1}, \chi_2^{-1})$. 
    \item If $\chi_1/\chi_2 = | \cdot |^{-1}$, then we have an exact sequence of representations
    \[
    0 \rightarrow C(\chi_1, \chi_2) \rightarrow I(\chi_1, \chi_2) \rightarrow S(\chi_1, \chi_2) \rightarrow 0
    \]
    with $\dim{C(\chi_1, \chi_2)}=1$ and $S(\chi_1, \chi_2)$ irreducible. 
    \item If $\chi_1/\chi_2 = | \cdot |$, then we have an exact sequence of representations 
    \[
    0 \rightarrow S(\chi_1, \chi_2) \rightarrow I(\chi_2, \chi_2) \rightarrow C(\chi_1, \chi_2) \rightarrow 0 
    \]
    with $\dim{C(\chi_1, \chi_2)}=1$ and $S(\chi_1, \chi_2)$ irreducible. 
    \item Otherwise, $I(\chi_1, \chi_2)$ is irreducible. 
    \item If $\chi_1/\chi_2 \simeq | \cdot |^{-1}$, then $S(\chi_1, \chi_2) \simeq S(\chi_2, \chi_1)$ and $C(\chi_1, \chi_2) \simeq C(\chi_2, \chi_1)$, and if $\chi_1/\chi_2 \not \simeq |\cdot|^{\pm 1}$, then $I(\chi_1, \chi_2) \simeq I(\chi_2, \chi_1).$
\end{enumerate}
\end{theorem}

\begin{remark}
Some remarks on the theorem:
\begin{enumerate}[label=(\alph*)]
    \item The representation $I(\chi_1, \chi_2)$ is an example of an induced representation for a totally disconnected group. 
    \item Why the strange $|\frac{a}{b}|^{1/2}$ factor? It is necessary to make 2. hold! So why does 2. hold? Consider 
    \[
    I(|\cdot|^{1/2}, |\cdot|^{-1/2}) = \{ \varphi \mid \varphi\left( \bp a & b \\ 0 & d \ep \cdot g \right) = | \frac{a}{d} | \varphi(g) \}. 
    \]
    We can define a function 
    \begin{align*}
        \Phi: I(|\cdot|^{1/2}, | \cdot | ^{-1/2})&\rightarrow \C \\
        \varphi &\mapsto \int_{K_0}\varphi d \mu . 
    \end{align*}
    (One can think of $I(|\cdot|^{1/2}, | \cdot | ^{-1/2})$ as being some ``functions'' on $\PP^1K$ and we are integrating over $\PP^1K$ to get a number. More precisely, these are ``densities", but to explain why would take us too far afield.) Here is a minor miracle: the function $\Phi$ is $\GL_2(K)$-invariant, hence
    \[
    C(|\cdot |^{1/2}, | \cdot |^{-1/2}) =\C
    \]
    is the trivial representation. Now, given $\varphi \in I(\chi_1, \chi_2)$ and $\varphi' \in I(\chi_1^{-1}, \chi_2^{-1})$, $\varphi \varphi' \in I(|\cdot |^{1/2}, | \cdot |^{-1/2}).$ By composing with $\Phi$, we get a $\GL_2(K)$-invariant pairing:
    \[
    I(\chi_1,\chi_2) \times I(\chi_1^{-1}, \chi_2^{-1}) \rightarrow \C,
    \]
    which turns out to be non-degenerate, establishing 2. 
    \item Part 6. is the most complicated to prove. It uses an intertwiner $I(\chi_1, \chi_2) \rightarrow I(\chi_2, \chi_1)$ via analytic continuation (there are connections to the Jantzen filtration).
    \item Finally, note that $\Rep{\GL_2(K)}$ is not semisimple, so we cannot just compute homs as in the finite group case. 
\end{enumerate}
\end{remark}

The other class of representations of $\GL_2(K)$ are cuspidal representations. 
\begin{theorem}
For every degree $2$ extension $L/K$ and continuous character $\theta: L^\times \rightarrow \C$ which does not factor through the norm, there exists an irreducible representation $BC_{L/K}(\theta)$. We have that $BC_{L/K}(\theta) \simeq BC_{L/K}(\theta')$ if and only if $\theta^\sigma \simeq \theta'$ for $\sigma \in \Gal(L/K)$.
\end{theorem}

The construction of $BC_{L/K}(\theta)$ is complicated, via the Weil representation. What is going on metaphorically? 
\begin{itemize}
    \item Consider $\SL_2(\F_q)$. We have seen in Joe's Informal Friday Seminar talks that a  character $\theta: T_a \rightarrow \C^\times$ gives rise to a local system $L_\theta$ on $\PP^1_{\F_q}\setminus \PP^1\F_q$. Taking the first cohomology yields a cuspidal representation $R_{T_a}^G(\theta)$.
    \item Consider $\SL_2(\R)$. A character $\theta: SO(2) \rightarrow \C^\times$ (such characters are classified by $\Z$) gives rise to a local system $\mc{O}(n)$ on the upper half plane, taking global sections yields a discrete series representation $\Gamma(\mathbb{H}, \mc{O}(n))$. 
    \item Now take $\GL_2(K)$. A character $\theta: L^\times \rightarrow \C$ gives rise to a local system $\mc{L}_\theta$ on ``Drinfeld space'' $\PP^1\overline{K}/\PP^1K$. Taking first cohomology yields the representation $BC_{L/K}(\theta)$. Note that this is very technical, and is an active area of research. 
\end{itemize}

This can also be viewed through the lens of Langlands functoriality. Let $L/K$ be a degree $n$ extension, so $W_L \subset W_K$ is an index $n$ subgroup. We have the following diagram: 
\[
\begin{tikzcd}
\left\{ \begin{array}{c} \text{1-dim'l reps }\\ \text{ of }L^\times \end{array} \right\} \arrow[r, leftrightarrow] \arrow[d, "\Ind_{W_L}^{W_K}"]&  \left\{ \begin{array}{c} \text{irred. smooth ad.}\\ \text{reps of $\GL_1(L)$} \end{array} \right\} \arrow[d, "BC=\text{``base change''}"]\\
\left\{ \begin{array}{c} \text{$n$-dim'l reps}\\ \text{of $K^\times$} \end{array} \right\} \arrow[r, leftrightarrow] & \left\{ \begin{array}{c} \text{irred smooth}\\\text{reps of $\GL_n(K)$} \end{array} \right\} 
\end{tikzcd}
\]

Thus an innocuous induction functor on the left hand side predicts a highly non-trivial correspondence between irreducible representation on the right hand side!

\subsection{Weil-Deligne representations}
\label{Weil-Deligne}
We are almost ready to make a precise statement of the local Langlands correspondence for $\GL_2(K)$! Recall local class field theory: there is a map 
\[
r_K: W_K \rightarrow K^\times.
\]
Composing with $| \cdot |$ gives us the {\bf norm character}
\[
| \cdot | : W_K \rightarrow \Q^\times.
\]
An $n$-dimensional {\bf Weil-Deligne representation} is a triple $(\rho, V, N)$, where 
\begin{itemize}
    \item $V$ is an $n$-dimensional complex vector space, 
    \item $\rho: W_K \rightarrow \GL(V)$ is a continuous representation, and 
    \item $N \in \End(V)$ is nilpotent such that 
    \begin{equation*}
    \rho(x)N\rho(x)^{-1} = |x|N
        \tag{$\ast$}
    \end{equation*}
    for all $x \in W_K$. (In fact, $(\ast)$ forces $N$ to be nilpotent.) 
\end{itemize}

\begin{example}
\begin{enumerate}
    \item Any $n$-dimensional continuous representation of $W_K$ with $N=0$ is a Weil-Deligne representation. 
    \item The representation $\rho = \bp |\cdot| \chi & 0 \\ 0 & \chi \ep$ for any character $\chi:W_K \rightarrow \C^\times$ and $N=\bp 0 & 1 \\ 0 & 0 \ep$ form a Weil-Deligne representation. Indeed, for $x \in W_K$,
    \begin{align*}
        \rho(x)N\rho(x)^{-1}&= \bp |x| \chi(x) & 0 \\ 0 & \chi(x) \ep \bp 0 & 1 \\ 0 & 0 \ep \bp |x|^{-1} \chi(x)^{-1} & 0 \\ 0 & \chi(x)^{-1} \ep \\
        &= \bp 0 & |x| \\ 0 & 0 \ep \\ 
        &= |x| \bp 0 & 1 \\ 0 & 0 \ep. 
    \end{align*}
    We denote this Weil-Deligne representation $St(\chi, |\cdot| \chi)$. 
\end{enumerate}
\end{example}
A Weil-Deligne representation is {\bf $F$-semisimple} if $V$ is semisimple as a representation of $W_K$. 

\begin{exercise}
Let $\widetilde{\text{Frob}} \in W_K$ be any lift of Frobenius. Show that a Weil-Deligne representation is $F$-semisimple if and only if $\widetilde{\text{Frob}}$ is semisimple. 
\end{exercise}

\subsection{The local Langlands correspondence for $\GL_2$}

\begin{theorem}
({\em local Langlands correspondence for $\GL_2$, $p \neq 2$}) Fix a local field $K$ of residue characteristic $p \neq 2$. There is a canonical bijection. 
\[
\left\{ \begin{array}{c} \text{ $F$-semisimple}\\ \text{$2$-dimensional} \\ 
\text{Weil-Deligne reps} \end{array} \right\}_ {/\simeq} \leftrightarrow \left\{ \begin{array}{c} \text{irred. smooth admiss.} \\ \text{ reps of $\GL_2(K)$} \end{array} \right\}_{/\simeq}
\]
Moreover, this bijection is given as follows. 
\begin{align*}
    \chi_1/\chi_2 = |\cdot|^{\pm1}: \left( \bp \chi_1 & 0 \\ 0 & \chi_2 \ep, N=0 \right) &\leftrightarrow C(\chi_1, \chi_2) \\
    \chi_1/\chi_2 = | \cdot |^{\pm1}: \left( \bp \chi |\cdot| & 0 \\ 0 & \chi \ep, N=\bp 0 & 1 \\ 0 & 0 \ep \right) &\leftrightarrow S(\chi, |\cdot| \chi) \\
    \chi_1/\chi_2 \neq | \cdot |^{\pm 1}: \left( \bp \chi_1 & 0 \\ 0 & \chi_2 \ep , N=0 \right) &\leftrightarrow I(\chi_1, \chi_2) \\ 
    L/K, \theta:L^\times \rightarrow \C^\times: \left( \Ind_{W_K}^{W_L}(\theta), N=0 \right) &\leftrightarrow BC_{L/K}(\theta)  
\end{align*}
Where the character $\theta: L^\times \rightarrow \C^\times$ does not factor through the norm. 
\end{theorem}

%% file: lecture-9.tex
\section{Lecture 9: The case of $p=2$ and the Satake isomorphism}
\label{lecture 9}

Today's lecture has two objectives: to explore what the LLC looks like when $p=2$, and to examine how a simple special case of the LLC for $\GL_n$ leads to the Satake isomorphism. 

\subsection{Ramification filtration revisited} 
To begin, we revisit the ramification filtration of Section \ref{ramification filtration}. Let $L/K$ be a finite Galois extension where $K$ and $L$ are both local fields. We have the following inclusions:
\[
\begin{tikzcd}
L \arrow[r, hookleftarrow] \arrow[d, dash] & \OO_L \arrow[r, hookleftarrow] \arrow[d, dash] &\mf{m}_L \arrow[r, equal] \arrow[d, dash] &(\pi_L) \arrow[d, dash] \\
K \arrow[r, hookleftarrow] &\OO_K \arrow[r, hookleftarrow] & \mf{m}_K \arrow[r, equal] & (\pi_L)  
\end{tikzcd}
\]
Here the uniformizers $\pi_K, \pi_L$ are the only non-canonical objects in the diagram above. We denote by $k_K=\OO_L/\mf{m}_L$ (resp. $k_L=\OO_L/\mf{m}_L$) the residue fields. There is a short exact sequence 
\[
I_{L/K} \hookrightarrow \Gal(L/K) \twoheadrightarrow \Gal(k_L/k_K)=\langle\text{Frob}\rangle\simeq \Z/f\Z. 
\]
\begin{lemma}
An element $\sigma \in I_{L/K}$ in the inertia subgroup is determined by $\sigma(\pi_L)$. 
\end{lemma}
This leads to the {\bf ramification filtration} of the Galois group. Define
\[
I_0 := I_{L/K}, \hspace{5mm}  I_j:=\left\{\sigma \in I_{L/K} \mid \frac{\sigma(\pi_L)}{\pi_L}=1 \mod \mf{m}_L^j \right\}.
\]
Then 
\[
\Gal(L/K)\supset I_0 \supset I_1 \supset I_2 \supset \cdots \supset \{1\}
\]
is a filtration of $\Gal(L/K)$. 

\vspace{5mm}
\noindent
{\bf Key Facts:}
\begin{enumerate}
    \item The ramification filtration is a finite exhaustive filtration. 
    \item There is an injection $I_0/I_1 \hookrightarrow \OO^\times_L/(1+\OO_L^\times)\simeq k_L^\times: \sigma \mapsto \frac{\sigma(\pi_L)}{\pi_L}$. Hence $I_0/I_1$ is cyclic of order prime to $p$.  
    \item For $j \geq 1$, there is an injection $I_j/I_{j+1} \hookrightarrow (1+\mf{m}_L^j)/(1+\mf{m}_L^{j+1}) \simeq (k_L, +): \sigma \mapsto \frac{\sigma(\pi_L)}{\pi_L}$. Hence $I_j/I_{j+1}$ is an abelian $p$-group, and $I_1$ is a Sylow $p$-subgroup of $I_{L/K}$.  
\end{enumerate}
This filtration leads to some nomenclature: The first subquotient $\Gal(L/K)/I_{L/K}$ is canonically isomorphic to $\Z/f\Z$, and is referred to as the {\bf unramified} part of the Galois group. The second subquotient $I_{L/K}/I_1$ is cyclic of order prime to $p$, and is referred to as the {\bf tamely ramified} part of the Galois group. The remaining subquotients of the ramification filtration are abelian $p$-groups (and hence $I_1$ is a solvable $p$-group), and are referred to as the {\bf wildly ramified} part of the Galois group. 

The upshot is that there are significant constraints on which groups can appear as Galois groups of extensions of local fields. (For example, they must be solvable.) This is in sharp contrast to the number field setting, where many types of groups can appear as Galois groups of extensions of $\Q$. (Though exactly which groups appear as Galois groups of number fields is still very much an open problem, the {\em inverse Galois problem.})   

\begin{example}
Consider the extension  $L=\Q_p(\sqrt[p-1]{p})$ of $K=\Q_p$. Here $\pi_L=\sqrt[p-1]{p}$ and $\pi_K=p$ are uniformizers.  (Exercise: Show that the set $\mu_{p-1}$ of all  $(p-1)^{st}$ roots of unity is contained in $\Q_p$.) 
    The Galois group of this extension is 
    \[
    \Gal(L/K) = \{ \sigma_\zeta: \sqrt[p-1]{p} \mapsto \zeta \sqrt[p-1]{p} \text{ for } \zeta \in \mu_{p-1} \} \simeq  k_K^\times \hookrightarrow k_L^\times.
    \]
    One can check that $\Gal(L/K)$ is totally ramified; that is, $I_{L/K}=\Gal(L/K)$. (This follows from the observation that we are adjoining roots of $p$, whose image is zero in the residue field.) Furthermore, if $\sigma_\zeta \in I_{L/K}=\Gal(L/K)$, then $\frac{\sigma_\zeta(\pi_L)}{\pi_L} = \zeta \in k_L^\times$, hence $I_1 = \{1\}$ and $\Gal(L/K)$ is tamely ramified. 
\end{example}
 
 \begin{remark} Examples of wild ramification are almost always hard! We really should spend a lecture on such examples, but we are quickly running out of time, so sadly we will not. 
 \end{remark}

Next we will examine the structure of the absolute Galois group of a local field. For a local field $K$, we have an exact sequence 
\[
I_{\overline{K}/K} \hookrightarrow \Gal(\overline{K}/K) \twoheadrightarrow \widehat{\Z}. 
\]
We would like to pass to the limit to obtain a ramification filtration of the inertia subgroup $I_{\overline{K}/K}$ from the ramification filtrations of the inertia subgroups of finite extensions. However, there is a problem: if $L'/L/K$ is a tower of finite extensions, then the ramification filtration of $I_{L'/K}$ is related to multiples of the ramification filtration of $I_{L/K}$. 

This can be fixed through an ``upper numbering'' procedure which replaces $I_j$ with $I_{L/K}^\lambda$ for $\lambda \in \Q_{\geq 0}$ in a way that is compatible with extensions. (Exactly how one does this appears pretty crazy at first sight. It is explained in Serre's Local Fields \cite{Serre}.) This leads to a ramification filtration $I_{\overline{K}/K}^\lambda$ of the inertia subgroup of the absolute Galois group indexed by rational numbers:

\[
\Gal(\overline{K}/K)\supset I_{\overline{K}/K} \supset I^{>0}_{\overline{K}/K} \supset \cdots \supset I_{\overline{K}/K}^{463/5} \supset \cdots. 
\]
This filtration has the property that $\Gal(\overline{K}/K)/I_{\overline{K}/K} = \widehat{\Z} $  canonically, $I_{\overline{K}/K}/I^{>0}_{\overline{K}/K} \simeq \displaystyle{\prod_{l \neq p \atop \text{prime}} \Z_l}$  non-canonically, and other subquotients are pro-$p$ groups. 

\vspace{5mm}
\noindent
{\bf Important points:}
\begin{enumerate}
    \item The first two steps depend only on the residue characteristic of the field.
    \item Via class field theory, the image of this filtration in $W^{ab}_K\simeq K^\times$ corresponds to the filtration by $1+\mf{m}_K^j \subset \OO_K$. The fact that the only jumps in this filtration are at integers (as opposed to other elements of $\Q$) is the {\bf Hasse--Arf Theorem}. 
\end{enumerate}

\subsection{More details on Weil--Deligne representations}

Next we would like to show two things: (1) why any Weil--Deligne representation is ``close'' to a continuous representation of $\Gal({\overline{K}}/K)$ and (2) why $p=2$ is special in the local Langlands correspondence. 

\begin{proposition}
\label{close to Galois}
Any indecomposible $F$-semisimple Weil--Deligne representation is isomorphic to $St_n\otimes \rho$, where $\rho$ is an irreducible representation of $W_K$. 
\end{proposition}
Here $St_n$ is the Steinberg representation from the previous lecture; e.g.
\[
St_4=\left( \bp |\cdot|^3 & 0 & 0 & 0 \\ 0 & |\cdot|^2 & 0 & 0 \\ 0 & 0 & |\cdot| & 0 \\ 0 & 0 & 0 & 1 \ep , N=\bp 0 & 1 & 0 & 0 \\ 0 & 0 & 1 & 0 \\ 0 & 0 & 0 & 1 \\ 0 & 0 & 0 & 0 \ep \right) 
\]

The proof of this proposition is left as a somewhat tricky exercise. It becomes easier if you know what the weight filtration associated to a nilpotent operator is. 

\begin{proposition}
\label{galois reps}
\begin{enumerate}
    \item Let $\rho:W_K \rightarrow \GL(v)$ be an irreducible representation. Then there exists a continuous character $\chi:W_K \rightarrow \C^\times$ such that $\rho \otimes \chi$ has finite image and hence defines a representation $\rho \otimes \chi:\Gal(\overline{K}/K) \rightarrow \GL(V)$.
    \item Suppose that $\rho:W_K \rightarrow \GL(V)$ is irreducible and not induced from any proper subgroup of $W_K$. Then the restriction to wild intertia is irreducible. In particular, $\dim{V}$ is a power of $p$ (since any irreducible module over a $p$-group has dimension divisible by $p$). 
\end{enumerate}
\end{proposition}

\begin{remark}
Proofs of these two statements can be found in \cite{Tate}. 
\end{remark}

Propositions \ref{close to Galois} and \ref{galois reps}.1 show that every Weil--Deligne representation is ``close'' to a representation of the absolute Galois group, in the sense that every Weil-Deligne representation can be obtained from the Steinberg representation and an irreducible representation of $W_K$, and every irreducible representation of $W_K$ can be upgraded to a representation of $\Gal(\overline{K}/K)$ by tensoring with a character.

The two statements of Proposition \ref{galois reps} are reasonably easy consequences of the following lemma. 

\begin{lemma}
\label{finite subgroup}
Suppose a group $G$ has the form 
\[
\Gamma \hookrightarrow G \twoheadrightarrow \Z
\]
for some finite group $\Gamma$. Then any irreducible $G$-module is either irreducible over $\Gamma$ or induced from a subgroup of the form $\Gamma \rtimes m\Z$. 
\end{lemma}

The proof of this lemma is a worthwhile exercise! 

\subsection{Why is LLC for $p=2$ special?}

Proposition \ref{galois reps} shows that for $p \neq 2$, all irreducible $2$-dimensional representations of $W_K$ are induced from a finite index subgroup. However, for $p=2$, it's possible that there are irreducible representations of $W_K$ which are not induced. So do such representations exist? Yes! 

Consider a continuous two-dimensional representation $\rho:\Gal(\overline{K}/K)\rightarrow \GL_2(\C)$. By the no small subgroups lemma, the image of $\rho$ must lie in a finite subgroup of $\GL_2(\C)$, so in the composition of $\rho$ with the projection 
\[
\GL_2(\C)\rightarrow \PGL_2(\C),
\]
the image must be conjugate to a subgroup of the maximal compact subgroup $SO_3\subset \PGL_2(\C)$. The finite subgroups of $SO_3$ were classified\footnote{by Klein in 1884, \cite{Klein}}! They are of the following types: 
\begin{itemize}
    \item cyclic (orientation preserving symmetries of the product of an $m$-gon and an interval, fixing one end)
    \item dihedral (orientation preserving symmetries of the product of an $m$-gon and an interval)
    \item $A_4$ (orientation preserving symmetries of the tetrahedron)
    \item symmetries of the cube
    \item $A_5$ (orientation preserving symmetries of the icosahedron)
\end{itemize}
Reducible representations have images in cyclic subgroups of $SO_3$, and induced representations have images which are dihedral groups. What about the other three? Are there any representations of $\Gal(\overline{K}/K)$ whose image lies in any of the final three finite subgroups? Since $\Gal(\overline{K}/K)$ is solvable, we can eliminate the non-solvable group $A_5$ from our list. Let's consider the composition series of $A_4$:
\[
K_4=\Z/2\Z \times \Z/2\Z \hookrightarrow A_4 \twoheadrightarrow \Z/3\Z 
\]
By the structure of the ramification filtration, this subgroup structure is only possible for a local Galois group if $p=2$. It turns out that it does indeed occur for some local fields! 

The upshot is that for $p=2$, there are more representations on each side of the LLC, and the extra representations on the Weil group side are this special class of irreducible non-induced representations whose image lies in $A_4$. (Geordie isn't sure if representations corresponding to the symmetries of the cube exist. He is told that they do...) 

\vspace{5mm}
\noindent
{\bf A mystery to ponder:} Let $G$ be a compact Lie group (e.g. a finite group), and let $R(G)_\C$ be its representation ring. What is a character 
\[
\theta: R(G)_\C \rightarrow \C ?
\]

\subsection{Unramified representations} 

One way to convince yourself that the LLC is amazing is to see that simple special cases already have deep consequences. The first example of this that we have seen is local class field theory. The second example we will see now!

The local Langlands correspondence for $\GL_n(K)$ says that there is a canonical bijection:
\[
\begin{tikzcd}
\left\{ \begin{array}{c} \text{$F$-semisimple}\\ 
\text{$n$-dimensional}\\
\text{Weil-Deligne reps} 
\end{array} \right\}_{/\simeq}
 \arrow[r, leftrightarrow, "1:1"] & \left\{ \begin{array}{c}
 \text{irred smooth}\\
 \text{admissible}\\
 \text{reps of $\GL_n(K)$} 
 \end{array} \right\}_{/\simeq}
 \end{tikzcd}
\]
On the left hand side of this bijection, we can consider a special class of {\bf unramified Weil-Deligne representations} consisting of those representations of $W_K$ which are trivial on the inertia subgroup. The corresponding representations on the right hand side are the {\bf spherical representations} of $\GL_n(K)$:
\[
\begin{tikzcd}
\left\{ \begin{array}{c} \text{Weil-Deligne reps}\\ 
\text{s.t. $N=0$ and $\rho$}\\
\text{factors through $\Z$: }\\
W_k \twoheadrightarrow \Z \hookrightarrow \GL_n(\C)
\end{array} \right\}_{/\simeq}
 \arrow[r, leftrightarrow, "1:1"] & \left\{ \begin{array}{c}
 \text{reps of $\GL_n(K)$}\\
 \text{admitting a}\\
 \text{$\GL_n(\OO_K)$-fixed vector} 
 \end{array} \right\}_{/\simeq}
 \end{tikzcd}
\]

Semisimple representations of $W_K$ which factor through $\Z$ are in bijection with semisimple elements of $\GL_n(\C)$, and irreducible representations of $\GL_n(K)$ admitting a $\GL_n(\OO_K)$-fixed vector are in bijection with irreducible representations of the ``spherical Hecke algebra'' (which you are not expected to be familiar with and we will soon define). Thus, the restriction of the local Langlands correspondence to this special case results in a bijection 
\[
\left\{ \begin{array}{c} 
\text{semisimple elements}\\
\text{in $\GL_n(\C)$} 
\end{array} \right\}_{/\text{conj}} \xleftrightarrow{1:1} \left\{ \begin{array}{c} 
\text{irreducible reps of }\\
\mc{H}_{sph}:=\mc{H}(\GL_n(\OO_K),\GL_n(K))
\end{array} \right\}_{/\simeq}
\]
This is the {\bf Satake isomorphism}! We will spend the rest of the lecture explaining this bijection (particularly the right hand side) in more detail. 

\begin{remark}
The left-hand-side of the bijection above is independent of $K$, and even of the residue characteristic $p$! 
\end{remark}

\subsection{Hecke algebras}

Suppose $G$ is a finite group. 

\vspace{5mm}
\noindent
{\bf Case 1:} Consider $N\subset G$ a normal subgroup. If $V$ is a $G$-representation, then $G$ acts on $V^N$ (because for $n\in N, g \in G$, and $v \in V^N$, $n \cdot gv=g \cdot g^{-1}ng \cdot v = gv$), and the action factors over $G/N$. Moreover, one can check that $\End(\Ind_N^G\C)\simeq \C[G/N]$. Hence we have a bijection 
\[
\left\{ \begin{array}{c} 
\text{irred $G$-modules}\\
\text{with an $N$-fixed vector} \end{array} \right\}_{/\simeq} \leftrightarrow \{\text{ irred $G/N$-modules} \}_{/\simeq} .
\]

\noindent
{\bf Case 2:} Consider $H \subset G$ not necessarily normal. Given a $G$-representation $V$, what acts on $V^H$? The Hecke algebra! The operator 
\begin{align*}
    \pi_H:V&\rightarrow V^H \\
    v &\mapsto \frac{1}{|H|} \sum_{h \in H} h \cdot v
\end{align*}
projects onto $H$-invariants. This can be used to define a ``Hecke operator'' $[HgH]$ for every $g \in G$ which makes the following diagram commute:
\[
\begin{tikzcd}
V^H  \arrow[r, "\text{[}HgH\text{]}"] \arrow[d, hookrightarrow] &V^H\arrow[d, leftarrow, "\pi_H"]\\
V \arrow[r, "\cdot g"] &V
\end{tikzcd}
\]
Note that all $g$ in the same double coset yield the same Hecke operator. Alternatively, this operator is the sum  
\[
[HgH]=\frac{1}{|H|} \sum_{g' \in HgH} g'.
\]
The {\bf Hecke algebra} $\mc{H}(H,G)$ of the pair $(H,G)$ is the vector space $^H\C[G]^H$ with multiplication 
\[
(f \ast f')(g):=\frac{1}{|H|}\sum_{g=hh'}f(h)f'(h).
\]
This is an associative unital algebra with unit
\[
1_H=\frac{1}{|H|} \sum_{h \in H} h.
\]

\begin{example}
\begin{enumerate}
    \item If $N$ is normal, $\mc{H}(N,G)=\C[G/N]$. 
    \item If $G=\GL_n(\F_q)$ and $B=\left\{ \bp * & \cdots & * \\ 0 & \ddots & \vdots \\ 0 & 0 & * \ep \right\}$, then $\mc{H}(B,G)$ is the ``Hecke algebra of $S_n$ at $q=|\F_q|$''. This algebra is almost independent of $q$.
\end{enumerate}
\end{example}

\begin{exercise}
(Do it!) Show that 
\[
\End(\Ind_H^G \C) \simeq \mc{H}(H,G).
\]
Hence 
\[
\langle \Ind_H^G \C \rangle \xrightarrow{\sim} \mc{H}(H,G)\text{-mod}.
\]
(Here the angle brackets mean the smallest abelian category generated by kernels, cokernels, extensions, and direct sums.) 
Deduce that 
\[
\left\{ {\text{irred. $G$-modules} \atop \text{with $H$-fixed vector}} \right\} \xleftrightarrow{1:1} \left\{ \text{irred $\mc{H}(H,G)$-modules} \right\}. 
\]
\end{exercise}

\begin{remark}
There is a tendency in the literature to consider one subgroup $H$ at a time, but one can also consider all subgroups (or a particularly nice family of subgroups) at the same time, resulting in a ``Hecke algebroid''. 
\end{remark}

We can also define Hecke algebras of $p$-adic groups. Let $G=\GL_n(K)$ for a local field $K$, and $K_0 = \GL_n(\OO_K)$ the maximal compact subgroup. Then the ``big'' Hecke algebra of $G$ is 
\[
\mc{H}^{big}=\left\{\varphi:G \rightarrow \C \middle | {\varphi \text{ locally constant} \atop \text{ compact support} }\right\}. 
\]
An alternate description is
\[
\mc{H}^{big} = \bigcup_i \left\{\varphi:G \rightarrow \C \middle |  { \varphi \text{ locally constant on $K_i$-double} \atop \text{cosets, non-zero on finitely many}} \right\}.
\]
\begin{exercise}
Prove that the two formulations of $\mc{H}^{big}$ are equivalent.
\end{exercise}
The algebra structure on $\mc{H}^{big}$ is given by 
\[
(f \ast f')(g) = \int_{h \in G} f(h) f'(h^{-1}g) d \mu,
\]
where $\mu$ is the Haar measure. 

\begin{example}
Let $1_{K_i}$ be the indicator function on $K_i$. Then 
\[
1_{K_i} \ast 1_{K_i} (g) = \int_{h \in G} 1_{K_i}(h) 1_{K_i}(h^{-1}g) d \mu = \begin{cases} 0 &\text{ if } g \not \in K_i, \\ \int_{K_i}1 d\mu &\text{ if }g \in K_i.
\end{cases}
\]
In other words, 
\[
1_{K_i} \ast 1_{K_i} = \mu(K_i)1_{K_i},
\]
so $1_{K_i}$ is a quasi-idempotent. 
\end{example}

\begin{remark}
Because any irreducible $G$-module has $V^{K_i} \neq 0$ for some $i$, $\mc{H}^{big}$ can be used to understand all smooth admissible representations of $G$. However, it is very complicated. 
\end{remark}

Assume $\mu(K_0)=1$ so $1_{K_0}$ is idempotent. The {\bf spherical Hecke algebra} is
\[
\mc{H}^{sph}=\mc{H}(K_0,G):=1_{K_0}\mc{H}^{big}1_{K_0}.
\]
\begin{exercise}
(Do it!) Prove the Cartan decomposition of $G$:
\[
    G=\bigsqcup_{\begin{array}{c} \underline{\lambda} \\ \lambda_1 \geq \lambda_2 \geq \ldots \geq \lambda_n \\ \lambda_i \in \Z \end{array}} K_0 \bp \pi^{\lambda_1} & & & &  \\ & \pi^{\lambda_2} & & &  \\ & & \pi^{\lambda_3} & & & \\ & & & \ddots & & \\ & & & & &  \pi^{\lambda_n} \ep K_0 
\]
Hence 
\[
\mc{H}^{sph} = \bigoplus_{\underline{\lambda}} \C 1_{\underline{\lambda}}.
\]
\end{exercise}
There are two miracles. 
\begin{theorem}
\label{satake}
\begin{enumerate}
    \item The spherical Hecke algebra $\mc{H}^{sph}$ is commutative. 
    \item ({\bf The Satake isomorphism}) There exists a canonical bijection 
    \[
    \mc{H}^{sph} \xleftrightarrow{\sim} R(^L\GL_n(\C)).
    \]
\end{enumerate}
\end{theorem}

\begin{remark}
The Langlands dual group $^L\GL_n(\C) \simeq \GL_n(\C)$ so we could have replaced the right hand side of the Satake isomorphism with the representation ring of $\GL_n(\C)$; however, the theorem also holds for general reductive groups and there the dual group is important. 
\end{remark}

Recall our mystery from earlier in the lecture: For a compact Lie group $G$, what is a character $\theta:R(G)_\C\rightarrow \C$ of its representation ring? By the Chevalley restriction theorem, $R(G)_\C\simeq R(T)_\C^W$, where $T\subset G$ is a maximal torus and $W$ is the Weyl group of $G$. So a  character of $R(G)_\C$ is just a choice of a semisimple conjugacy class in $G$! 

Theorem \ref{satake} can be used to establish unramified LLC:
\begin{align*}
    \left\{ \begin{array}{c}\text{``spherical representations'';}\\ \text{i.e. smooth admissible} \\ \text{irred reps of $G$ with} \\ \text{a $K_0$-fixed vector} \end{array} \right\}_{/\simeq}  & \xleftrightarrow[\text{Hecke algebra magic}]{1:1} \left\{ \begin{array}{c} \text{ irreducible} \\ \mc{H}(K_0,G)=\mc{H}^{sph}- \\ \text{modules} \end{array} \right\}_{/\simeq} \\
    &\xleftrightarrow[\text{commutativity of $\mc{H}^{sph}$}]{1:1} \left\{ \begin{array}{c} \text{characters}\\
    \chi:\mc{H}^{sph} \rightarrow \C \end{array} \right\} \\
    & \xleftrightarrow[\text{Satake isomorphism}]{1:1} \left\{ \begin{array}{c} \text{characters} \\ \theta: R(^L\GL_n(\C)) \rightarrow \C \end{array} \right\}\\
    &\xleftrightarrow[\text{the mystery from earlier}]{1:1} \left\{ \begin{array}{c} \text{conjugacy classes} \\
    \text{of semisimple}\\
    \text{ elts in $^L\GL_n(\C)$} \end{array} \right\} \\
    &\xleftrightarrow[\text{ definition}]{1:1} \left\{ \begin{array}{c} \text{unramified}\\
    \text{$n$-dimensional}\\
    \text{Weil--Deligne} \\
    \text{representations} \end{array} \right\}_{/\simeq}
\end{align*}

%% file: lecture-10.tex
\section{Lecture 10: The big picture}
\label{lecture 10}

Today is about the big picture. We start with the very big picture, and finish with the moderately big picture. This is also the final lecture of the first term of this course!

\subsection{The very big picture}

\subsubsection{Dimension 0}

Let us go back to the beginning. Let $f(x)\in \Z[x]$ be a polynomial; e.g. $f(x)=x^2+1$. Back in March, we wondered: How many solutions does $f(x)$ have modulo a prime $p$? We constructed tables:
\begin{center}
\begin{tabular}{ | c|c|c| c| c| c| c| c| c| c |} 
\hline
$p$ & 2 & 3 & 5 & 7 & 11 & 13 & 17 & 19 & 23 \\ 
\hline
\# of sol's mod $p$ & 1 & 0 & 2 & 0 & 0 & 2 & 2 & 0 & 0 \\ 
\hline
$p$ mod 4 & 2 & 3 & 1 & 3 & 3 & 1 & 1 & 3 & 3 \\ 
\hline
\end{tabular} ...
\end{center}
Then we studied this via representation theory. The Galois group $\Gal(\overline{\Q}/\Q)$ acts on the roots $\{\sigma_1, \ldots, \sigma_n\} \subset \overline{\Q}$ of $f(x)$, so we have a permutation representation 
\[
\Gal(\overline{\Q}/\Q) \rightarrow \GL(H),
\]
where $H:=\bigoplus_{i=1}^n = \C \sigma_i$. Then for unramified primes,
\[
\# \text{ solutions of $f(x)$ mod $p$ }=\Tr(\Frob, H).
\]

Even in this innocent (``dimension 0'') case, $H$ is enormously complicated. To simplify things, we instead considered a collection of {\em local} representations $H_{\Q_p}$, defined as follows. For each $p$, consider roots $\sigma_1', \ldots, \sigma_n'$ of $f(x)$ in $\overline{\Q}_p$. Then for each $p$ we have ``local Galois representations''
\[
\Gal(\overline{\Q}_p/\Q_p)\rightarrow \GL(H_{\Q_p}),
\]
where $H_{\Q_p}=\bigoplus \C \sigma_i'$. This gives us a ``categorification'' of the table above:
\begin{center}
\begin{tabular}{ | c|c|c| c| c| c| c| c| c| c |} 
\hline
$p$ & $2$ & $3$ & $5$ & $7$ & $11$ & $13$ & $17$ & $19$ & $23$ \\ 
\hline
& $H_{\Q_2}$ & $H_{\Q_3}$ & $H_{\Q_5}$ & $H_{\Q_7}$ & $H_{\Q_{11}}$ & $H_{\Q_{13}}$ & $H_{\Q_{17}}$ & $H_{\Q_{19}}$ & $H_{\Q_{23}}$ \\ 
\hline
\end{tabular} ...
\end{center}
If $p$ is unramified, then the inertia subgroup $I\subset\Gal(\overline{\Q_p}/\Q_p)$ acts trivially on $H_{\Q_p}$, so the local Galois representation is unramified. By the Satake isomorphism (Theorem \ref{satake}), this implies that $H_{\Q_p}$ is determined by a semisimple conjugacy class $[x] \in \GL_n(\C)$, and 
\[
\# \text{ solutions of $f(x)$ mod $p$ }=\Tr([x]).
\]
For unramified primes, the representation $H_{\Q_p}$ is rather simple. However, for ramified primes, the representation $H_{\Q_p}$ can be quite complicated:
\begin{enumerate}
    \item The study of $H_{\Q_p}$ lets us define local factors in the Artin $L$-function. 
    \item We can hope to understand $H_{\Q_p}$ through the local Langlands correspondence.
\end{enumerate}

\noindent
{\bf Remember our slogan:} There is a lot of substance at ramified primes/points! 

\subsubsection{Dimension $\geq 1$}

The classic example is that of an elliptic curve $E$; e.g. the projective completion of the curve $y^2+y=x^3+x^2+3x+5$ that we studied in the lecture on the Sato-Tate conjecture, Lecture \ref{sato-tate}. What is the analogue of the Galois representation $H$ in this setting?

Recall that $E$ is a group, and the complex points of $E$ are 
\[
E(\C)=\text{ solutions over $\C$ } = \C/\Lambda,
\]
where $\Lambda \subset \C$ is a lattice That is, we obtain $E(\C)$ by identifying opposite edges in this picture, where the dots represent elements of $\Lambda$:
\begin{center}
         \includegraphics[scale=0.5]{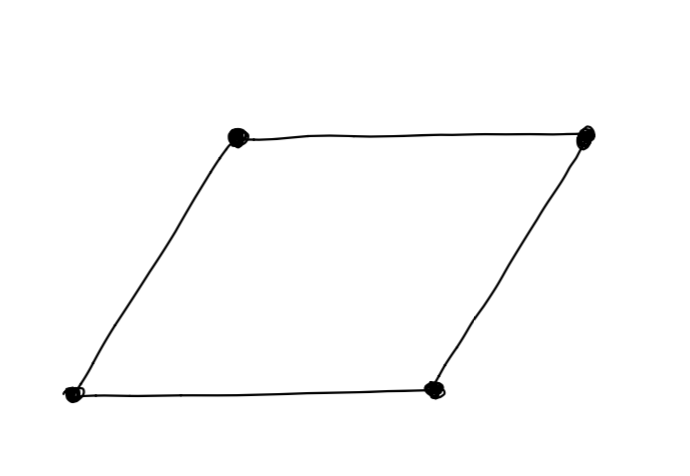}
\end{center}
The Galois group $\Gal(\overline{\Q}/\Q)$ does not act in any meaningful way on $E(\C)$. It does act on $E(\overline{\Q})$, but this is an enormously complicated set, a little too complicated for us! However, for any prime $\ell$, we can consider the ``$\ell^m$-torsion points'': \[
E[\ell^m]:=\{x \in E(\overline{\Q}) \mid \ell^m \cdot x=0 \} \simeq \left( \Z/\ell^m \Z \right) ^2; 
\]
e.g., for $\ell=3, m=1$:
\begin{center}
         \includegraphics[scale=0.5]{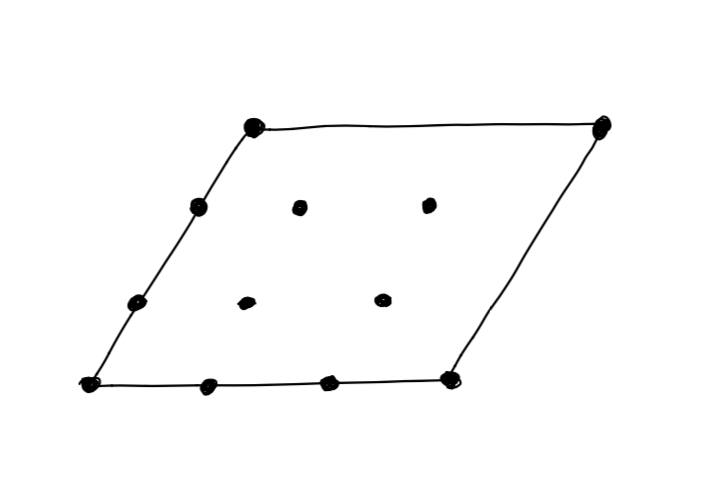}
\end{center}
There is a natural action of $\Gal(\overline{\Q}/\Q)$ on $E[\ell^m]$. The {\bf Tate module} is formed by taking the direct limit of the $E[\ell^m]$:
\[
T_\ell(E) := \lim_{\leftarrow} E[\ell^m] \simeq \Z_\ell^2. 
\]
The Tate module has a continuous action of $\Gal(\overline{\Q}/\Q)$. Moreover, if $E_{\F_p}$ is smooth, then 
\[
\# E(\F_p) = 1+p- \Tr(\Frob, T_\ell(E)). 
\]
So in this classic example of an elliptic curve, the Tate modules play the role of the representation $H$ which appeared in the dimension 0 setting. Notice that in the previous section we constructed a single representation $H$, but there is one Tate module for each prime $\ell$. This is an embarrassment of riches! 

The Tate module $T_\ell(E)$ is an example of ``$\ell$-adic cohomology'':
\[
T_\ell(E)=H_{\acute{e}t}^1(E, \Z_\ell)^\ast.
\]
In general, given a variety $X$ over a field $k$ and a prime $\ell$ such that multiplication by $\ell$ is non-zero in $k$, there is a continuous action of $\Gal(\overline{k}/k)$ on the $\ell$-adic cohomology groups $H_{\acute{e}t}^\ast(X_{\overline{k}}, \Q_\ell).$ Again, it is useful to study these representations via their restriction to local Galois groups $\Gal(\overline{\Q}_p/\Q_p)$. We can always calculate these after base change to $\Spec{\Q_p}$ and at primes of good reduction after base change to $\Spec{\F_p}$. ({\bf Exercise:} Think about what this statement means for $f(x) \in \Z[x]$.) 

In addition to \'{e}tale cohomology, one has several other methods for associating cohomology groups to the variety $X$:
\begin{enumerate}
    \item {\bf singular cohomology:} $H^i(X(\C), \Z)$, $H^i(X(\C),\F_p)$ (related via universal coefficient theorem)
    \item {\bf deRham cohomology:} $H^i_{dR}(X)$, $H^i_{dR}(X_{\F_p})$ (cohomology of differential forms)
    \item {\bf crystalline cohomology:} $H^i_{crys}(X_{\F_p}/\Z_p)$ (a fancy theory that produces $\Z_p$-vector spaces for $\F_p$-schemes)
\end{enumerate}

\noindent
{\bf Grothendieck's philosophy:} All of these cohomology groups should be shadows of a unique object, the ``motive'' of $X$. 

\vspace{3mm}
\noindent
{\bf Scholze:} Perhaps the ``motive'' is more like a sheaf/local system on $\Spec{\Z} \times \Spec{\Z}$. 

\vspace{3mm}
\noindent
{\bf Scholze's ICM picture:}
\begin{center}
         \includegraphics[scale=0.4]{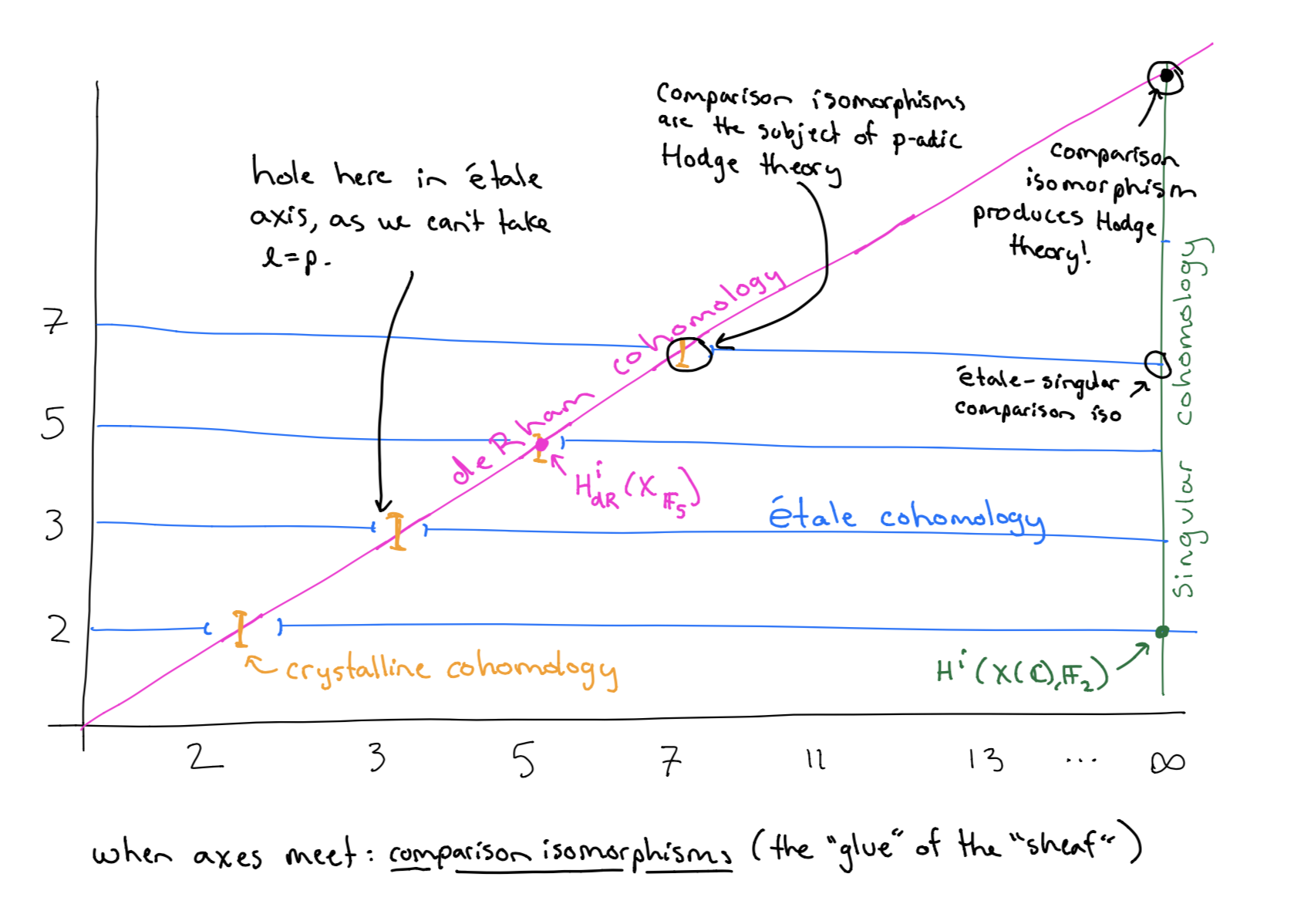}
\end{center}
Scholze also predicts an archimidean theory for varieties in characteristic $p$ which has been missing since the beginning of this subject! 

\vspace{5mm}
\noindent
{\bf Recommendation/Exercise:} Read section $10$ of Scholze's ICM paper.

\vspace{5mm}
How does one compare columns in Scholze's picture? In other words, if  $G_{\Q_p}=\Gal(\Q_p/\Q_p)$, how can we compare 
\[
\rho:G_{\Q_p} \rightarrow \GL_n(\Q_\ell) \text{ and } \rho':G_{\Q_p} \rightarrow \GL_n(\Q_{\ell'})? 
\]
The problem is the topology. The solution is given by Weil--Deligne representations.

A topological group $\Gamma$ is {\bf pro-$p$} if it is profinite and for all open normal subgroups $N \subset \Gamma$, $\Gamma/N$ is a $p$-group. 
\begin{example} Two examples of pro-$p$-groups are:
\begin{enumerate}
    \item wild inertia $\subset G_{\Q_p}$, and 
    \item $1+p\text{Mat}_n\Z_p=K_1 \subset \GL_n(\Q_p)$.
\end{enumerate}
\end{example}
\begin{lemma}
Any continuous group homomorphism 
\[
\rho:\Gamma \rightarrow G
\]
from a pro-$p$ group $\Gamma$ to a pro-$\ell$ group $G$ is trivial. 
\end{lemma}
\begin{corollary}
For a pro-$p$ group $\Gamma$, any continuous group homomorphism 
\[
\rho: \Gamma \rightarrow \GL_n(\Q_p)
\]
has finite image. 
\end{corollary}
We have seen (via the ramification filtration) the $G_{\Q_p}$ has the following structure: 
\begin{center}
     \includegraphics[scale=0.5]{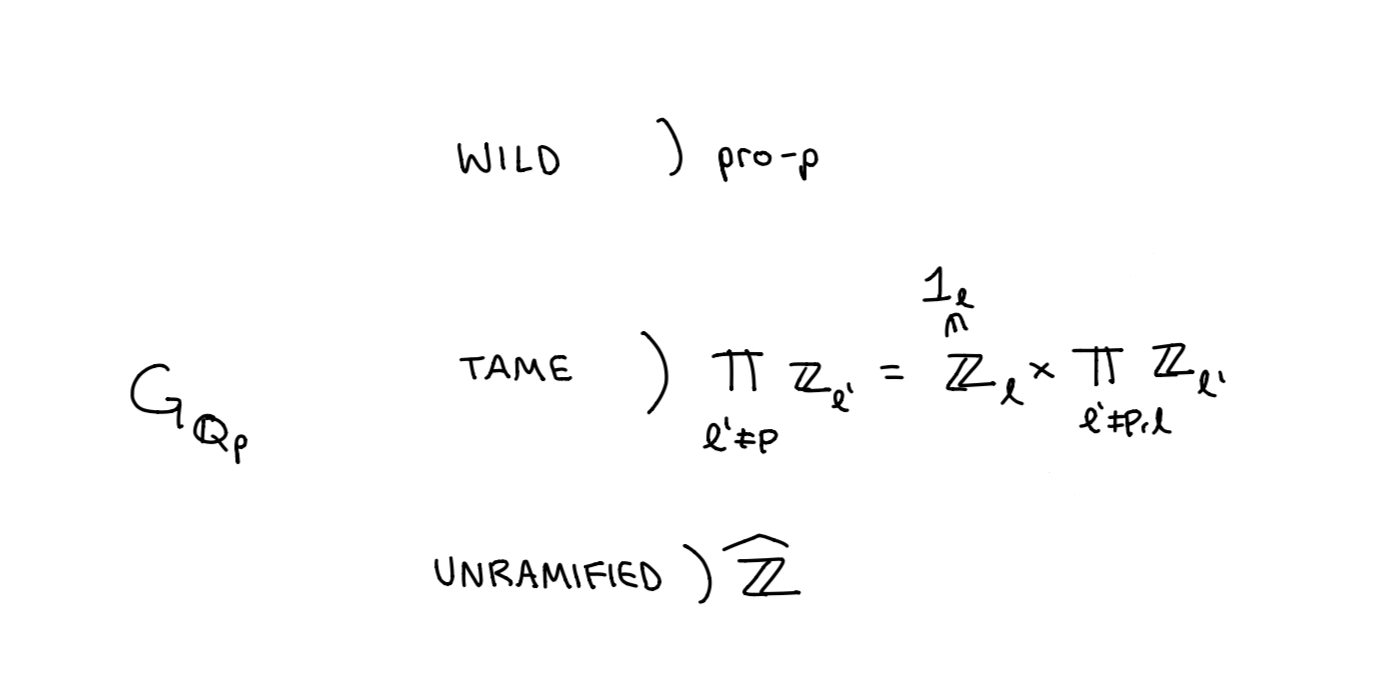}
\end{center}
Grothendieck showed us that the pro-$\ell'$ group $\prod_{\ell' \neq p,\ell} \Z_{\ell'}\subset I$ must have finite image. Moreover, $\rho(1_\ell)\in \GL_n(\Q_\ell)$ is almost unipotent. So what Grothendieck has shown us is that we  can ``take logs to get Weil--Deligne representations''.

\begin{theorem} (Grothendieck)
After identifying $\overline{\Q}_\ell$ with $\C$, one has a canonical injection 
\[
\left\{ \begin{array}{c} \text{cts. reps.} \\ \rho:G_{\Q_p} \rightarrow \GL_n(\Q_\ell) \end{array} \right\} \hookrightarrow \left\{ \begin{array}{c} \text{ Weil--Deligne reps} \\ \text{of $G_{\Q_p}$ over $\C$} \end{array} \right\}. 
\]
\end{theorem}
Notice that the set on the right is independent of $\ell$! 

\begin{remark}
We are being a bit lazy, but one can identify the image of this injection.
\end{remark}

\subsection{Local Langlands correspondence for split groups}

Let $G$ be a split reductive algebraic group over $\Z$ determined by the root datum $(X^\ast \supset R, X_\ast \supset R^\vee)$, and $\widehat{G}$ its dual group, determined by the opposite root datum $(X_\ast \supset R^\vee, X^\ast \supset R)$. Fix a local field $K$ and set $q=|\OO_K/\mf{m}_K|$. 

A {\bf Weil--Deligne representation in $\widehat{G}$} is a pair $(\rho, e)$, where 
\begin{itemize}
    \item $\rho:W_K \rightarrow \widehat{G}(\C)$ is a continuous group homomorphism, and 
    \item $e \in \Lie \widehat{G}(\C)$ is a nilpotent element
\end{itemize}
such that $\rho(g)e \rho(g)^{-1} = |g|e$ for all $g \in W_K$. A Weil--Deligne representation in $\widehat{G}$ is {\bf $F$-semisimple} if $\rho$ is semisimple. 
\begin{example}
A Weil--Deligne representation in $\GL_n$ is just an $n$-dimensional Weil--Deligne representation, as in Section \ref{Weil-Deligne}.
\end{example}
Given a Weil--Deligne representation $(\rho, e)$, consider
\[
Z_{\widehat{G}}(\rho, e) = \{ g \in \widehat{G} \mid g \cdot (\rho, e) = (\rho, e) \}. 
\]

\begin{theorem}
\label{llc}
({\bf Local Langlands correspondence}) There is a canonical correspondence 
\[
\left\{ \begin{array}{c} \text{$F$-semisimple} \\ \text{WD reps in $\widehat{G}$} \end{array} \right\}_{/\text{$\widehat{G}$-conj}} \xleftrightarrow{1:\text{finite}} \left\{ \begin{array}{c} \text{irred smooth} \\ \text{admissible reps} \\ \text{of $G(K)$} \end{array} \right\}_{/\simeq}
\]
Fibres of this map should be indexed by irreducible representations of $Z_{\widehat{G}}(\rho, e)/Z_{\widehat{G}}(\rho, e)^\circ$ and are called ``L-packets''. 
\end{theorem}

\subsection{The Deligne--Langlands conjecture}
Last week we examined (for $G=\GL_n$) a simple special case of the local Langlands correspondence, the case of unramified WD representations, and found that it followed from the Satake isomorphism. Another slightly less simple special case of the LLC is given by {\bf tamely ramified WD representations with unipotent monodromy} (TRUM). By restricting the correspondence in Theorem \ref{llc} to TRUM, we hope to obtain a correspondence:
\[
\left\{ \begin{array}{c} \text{TRUM; i.e. $(\rho, e)$} \\ \text{s.t. $\rho$ factors } \\ W_K \twoheadrightarrow \Z \text{, $e$ arbitrary} \end{array} \right\} \xleftrightarrow{1:\text{finite}} \left\{ \begin{array}{c} \text{reps with an} \\ \text{ Iwahori fixed} \\ \text{vector} \end{array} \right\}   
\]
By analogous arguments to the ones we made last week for $\GL_n$, the set of $\rho$ which factor through $W_K \twoheadrightarrow \Z$ is in bijection with the set of conjugacy classes of semisimple elements in $\widehat{G}$. Hence the left hand side of the correspondence above is in bijection with the set
\[
\{ (s, e) \mid s \in \widehat{G} \text{ semisimple }, e\in \Lie{\widehat{G}} \text{ nilpotent s.t. }ses^{-1}=qe \}_{/\text{$\widehat{G}$-conj}}. 
\]
The right hand side of the corrspondence above is in bijection with the set
\[
\{\text{irred reps of the ``Iwahori-Hecke algebra'' $\mc{H}_{\text{aff}}:=\mc{H}(I,G(K))$}\}.
\]
This motivates the following conjecture of Deligne--Langlands. 
\vspace{5mm}

\noindent
{\bf The Deligne--Langlands conjecture:} As in the set-up above, let $q=|\OO_K/\mf{m}_K|$ be the residue characteristic of the local field $K$. There is a bijection: 
\[
\left\{ (s, e, \chi)\hspace{2mm} \middle|\ \begin{array}{c} s \in \widehat{G}(\C) \text{ semisimple}, \\
e \in \Lie{\widehat{G}} \text{ nilpotent, and } \\
\chi \text{ irred rep of }\pi_0(Z_{\widehat{G}}(\rho, e)) \\ \text{ such that }ses^{-1}=qe \end{array} \right\}_{/ \widehat{G} \text{-conj}} \xleftrightarrow{\hspace{2mm}\simeq\hspace{2mm}} \{\text{irred }\mc{H}_{\text{aff}}\text{-modules}\}_{/\simeq}
\]
\begin{remark}
The affine Hecke algebra $\mc{H}(I,G(K))$ has a presentation in which $q$ becomes a variable. The above conjecture can either be understood with fixed $q=\#|\OO_K/\mf{m}_K|$ or with $q$ as a variable, in which case $q$ is also a variable on the left hand side.
\end{remark}
Recall that the unramified LLC followed from the Satake isomorphism:
\[
{\mc{H}_{\text{sph}}=\mc{H}(G(\OO_K), G(K)) \atop \text{ ``constructible'' }} {\xleftrightarrow{\hspace{2mm}\simeq\hspace{2mm}} R(\widehat{G})=\mc{O}\left({\text{semisimple conj.} \atop \text{classes in $\widehat{G}$}}  \right)  \xleftrightarrow{\hspace{2mm}\simeq\hspace{2mm}} \atop \hspace{3mm} } {K^0(\text{pt}/\widehat{G})=K^{\widehat{G}}(\text{pt}) \atop \text{ ``coherent'' }}
\]
Similarly, the TRUM case of the LLC (which reduces to the Deligne--Langlands conjecture) follows from the Kazhdan--Lusztig isomorphism:
\[
{\mc{H}_{\text{aff}} \atop \text{ ``constructible'' }} { \xleftrightarrow{\hspace{2mm}\simeq \hspace{2mm}} \atop } { K^{\widehat{G} \times \C^\times}(St) \atop \text{ ``coherent'' }}
\]
Indeed, if $\pi:\widetilde{\mc{N}}=T^\ast\mc{B} \rightarrow \mc{N}$ is the Springer resolution, $\mc{B}_e=\pi^{-1}(e)$ is the Springer fibre of a nilpotent element $e \in \mc{N}$, and $St$ is the Steinberg variety, the Kazhdan--Lusztig isomorphism (which is not easy to establish!) implies that there is an action of the affine Hecke algebra $\mc{H}_{\text{aff}}$ on $K^{Z_{\widehat{G}\times \C^\times}(e)}(\mc{B}_e).$
Here we can see that 
\[
Z_{\widehat{G}\times \C^\times}(e) = \{(g, c) \mid c \cdot geg^{-1}=e\} =\{ (g, c) \mid geg^{-1}=c^{-1}\cdot e\} 
\]
looks very close to the parameters in the Deligne--Langlands conjecture. This action shows us that the $K$-theory of Springer fibres provides all simple $\mc{H}_{\text{aff}}$-modules, thus proving the Deligne--Langlands conjecture. 

\subsection{Geometric Satake equivalence}

There is a geometric upgrade of the Satake isomorphism which has proven to be a major tool in geometric representation theory. Set $K=k((t)),$ so $\OO_K=k[[t]]$, where $k=\C$ or $\F_q$. Then 
\[
\mc{H}(G(\OO_K),G(K)) = {G(\OO_K)\text{-invariant functions on the} \atop \text{``affine Grassmanian'' $\mc{G}r_G:=G(K)/G(\OO_K)$} }.
\]
The {\bf geometric Satake equivalence} is the equivalence of categories:
\[
{ (\text{Perv}_{G(\OO_K)}(\mc{G}r, \C), \ast) \atop \text{ ``constructible'' } } { \xrightarrow{\simeq} \atop } { (\Rep{\widehat{G}_\C}, \otimes) \atop \text{ ``coherent'' } }  
\]
This equivalence was key in recent work by V. Laffourg\'{e}s giving an ``automorphic to Galois'' correspondence for global function fields. 

\subsection{Bezrukavnikov's equivalence} 

There is also a geometric upgrade of the Kazhdan--Lusztig isomorphism. With $K=k((t))$ as above, the affine Hecke algebra is
\[
\mc{H}_{\text{aff}}=\text{ Iwahori-invariant functions on $G(K)/I$}.
\]
Here $G(K)/I$ is the set of $k$-points of the ``affine flag variety'' $\mc{F}l_G$. Roughly, Bezrukavnikov's equivalence is an equivalence of categories
\[
{(D^b_{I \times I}(\mc{F}l_G), \ast) \atop \text{ ``constructible'' } } { \xrightarrow{\simeq} \atop }{ (D^b\text{Coh}^{G \times \C^\times}(St), \ast) \atop \text{ ``coherent''} }. 
\]
\begin{remark}
This is a bit of a lie! It would take several more lecture to precisely describe the categories on each side of this equivalence. 
\end{remark}
This equivalence has many applications in geometric representation theory. For example, a mod $p$ version of this equivalence would imply everything that we know about modular representations of algebraic groups!  

%% file: lecture-11.tex
\section{Lecture 11: Review of the first semester}
\label{lecture 11}

\subsection{The local Langlands correspondence}

Recall our setup from last semester. Let $K$ be a local field (i.e. $K$ is a finite extension of $\Q_p$ or $K=\F_q((t))$) with ring of integers $\OO$ and residue field $k$:
\[
K \supset \OO \twoheadrightarrow k
\]
Nothing is lost by just thinking in terms of the example 
\[
\Q_p \supset \Z_p \twoheadrightarrow \F_p. 
\]
Last semester, we worked up to stating the local Langlands correspondence. 

\vspace{5mm}

\noindent
{\bf Local Langlands correspondence for $\GL_n$}: There exists a canonical bijection 
\[
 \left\{ \begin{array}{c} \text{irred. smooth admiss.} \\ \text{ reps of $\GL_n(K)$ on}\\ \C \text{-vector spaces} \end{array} \right\}_{/\simeq} \xleftrightarrow{1:1} \left\{ \begin{array}{c} \text{ $F$-semisimple}\\ \text{$n$-dimensional} \\
\text{Weil-Deligne reps} \end{array} \right\}_ {/\simeq}
\]
It has been a while sine the last lecture, so let us remind ourselves what all of these words mean. The representations on the left hand side of this correspondence are {\em irreducible} representations which are usually infinite-dimensional. The adjective {\bf smooth} means that every vector has an open stabilizer, and the adjective {\bf admissible} means that for any open subgroup $U \subset \GL_n(K)$, $V^U$ is finite-dimensional. We consider the set of such representations up to equivalence. 

The objects on the right hand side are $n$-dimensional representations (where $n$ is the same $n$ appearing in $\GL_n$ on the left) of the Weil group attached to the field $K$, along with some extra data. Recall that the {\bf Weil group} is a subgroup of the absolute Galois group $\Gal(\overline{K}/K)$ defined by the fact that it fits into the following diagram. 
\[
\begin{tikzcd}
I_{\overline{K}/K} \arrow[r, hookrightarrow] \arrow[d, equal] & W_K:=\varphi^{-1}(\Z) \arrow[r, twoheadrightarrow] \arrow[d, hookrightarrow] & \Z=\langle \text{Frob} \rangle  \arrow[d, hookrightarrow]\\
I_{\overline{K}/K} \arrow[r, hookrightarrow]  & \Gal(\overline{K}/K) \arrow[r, twoheadrightarrow, "\varphi"] & \Gal(\overline{k}/k)=\widehat{\Z} 
\end{tikzcd}
\]
Here $I_{\overline{K}/K}$ is the {\bf inertia subgroup}. A {\bf Weil-Deligne representation} is a triple $(V, \rho, N)$, where 
\begin{enumerate}
    \item $V$ is an $n$-dimensional $\C$-vector space, 
    \item $\rho:W_K \rightarrow \GL(V)$ is a continuous representation, and 
    \item $N$ is a nilpotent endomorphism of $V$ such that 
    \[
    \rho(x)N\rho(x)^{-1} = |x|N
    \]
    for all $x \in W_K$.
\end{enumerate}
Here $|\cdot|$ is the canonical norm on $W_k$, which is uniquely determined by the property that $|\text{Frob}|=|k|$. (See Section \ref{canonical norms} for a refresher on why this exists.) A Weil-Deligne representation is {\bf $F$-semisimple} if any lift of Frob acts semisimply. 

\begin{remark}
For the experts: Weil-Deligne representations are distilled out of continuous representations $W_K \rightarrow \GL_n(\overline{\Q}_\ell)$ via ``log of monodromy''. Thus one can think of the right hand side as secretly being genuine representations of a group. It is phrased in this way to remove the auxilary choice of a prime number $\ell$.
\end{remark}

In general, if we replace $\GL_n(K)$ with any split reductive group $G(K)$ (e.g. $\SL_n$, $Sp_{2n}$, etc.), then the LLC changes as follows. 

\vspace{5mm}
\noindent
{\bf Local Langlands correspondence for split reductive groups}: There exists a canonical finite-to-one map 
\[
 \left\{ \begin{array}{c} \text{irred. smooth admiss.} \\ \text{ reps of $G(K)$ on}\\ \C \text{-vector spaces} \end{array} \right\}_{/\simeq} \xrightarrow{\text{finite}:1} \left\{ \begin{array}{c} \text{ $F$-semisimple}\\ 
\text{Weil-Deligne reps}\\ \text{of }W_K \text{ in }G^\vee(\C) \end{array} \right\}_ {/\simeq}
\]
In this setting, a {\bf Weil-Deligne representation of $W_K$ in $G^\vee(\C)$} is a pair $(\rho, N)$, where
\begin{enumerate}
    \item $\rho:W_K \rightarrow G^\vee(\C)$ is a continuous group homomorphism of the Weil group into the complex Langlands dual group of $G$, and 
    \item $N\in \Lie G^\vee(\C)$  such that 
    \[
    \rho(x)N\rho(x)^{-1} = |x|N
    \]
    for all $x \in W_K$.
\end{enumerate}
Note that condition $2$ forces $N \in G^\vee(\C)$ to be a nilpotent element. 

\subsection{The unramified story}
Last semester we ended the course by unpacking the simplest piece of the LLC, the case of {\em unramified representations.} Let us remind ourselves how this story went. By restricting each side of the correspondence above, we obtain a bijection 
\[
\begin{tikzcd}
 \left\{ \begin{array}{c}
 \text{irred. smooth admissible}\\
 \text{unramified reps of $G(K)$}\\
 \text{(i.e. reps of $G(K)$ that admit}\\ \text{a non-zero $G(\OO)$-fixed vector)}
 \end{array} \right\}_{/\simeq}
  \arrow[r, leftrightarrow, "1:1"]
 & \left\{ \begin{array}{c} \text{unramified}\\ 
\text{Weil-Deligne reps;}\\
\text{(i.e. reps which factor }\\ W_K \twoheadrightarrow \Z \rightarrow G^\vee(\C) \\
\text{ with $N=0$)}
\end{array} \right\}_{/\simeq}
 \end{tikzcd}
\]
By Hecke algebra theory, representations of $G(K)$ with a $G(\OO)$-fixed vector are in bijection with irreducible modules for the {\em spherical Hecke algebra}, 
\[
{\mc{H}_{\text{sph}}=\mc{H}(G(K), G(\OO)):=(\text{Fun}_{G(\OO)\times G(\OO)}(G(K), \C), \ast).}
\]
The functions in this definition are continuous with compact support. On the other hand, unramified Weil-Deligne representations are in bijection with the set of conjugacy classes of semisimple elements in $G^\vee(\C)$. Recall the Satake isomorphism. 

\begin{theorem}
\label{satake iso}
(Satake isomorphism) There exists a canonical isomorphism 
\[
\mc{H}_{\text{sph}} \xrightarrow{\sim} [\Rep{G}^\vee]\otimes_{\Z} \C. 
\]
\end{theorem}
This theorem will be a big feature of the course this semester. Let $T^\vee \subset G^\vee$ be a maximal torus, and $W$ the Weyl group of $G^\vee$. By highest weight theory we have 

\[
\begin{tikzcd}
\left[ \Rep{G}^\vee \right] \arrow[r, hookrightarrow] \arrow[rd, "\sim"] & \left[ \Rep{T}^\vee \right] = \Z \left[ X^*(T^\vee)\right] \\
& \Z\left[ X^*(T^\vee)\right]^W \arrow[u, hookrightarrow]
\end{tikzcd}
\]
Hence 
\[
[\Rep{G}^\vee]\otimes_\Z \C = \C[X^*(T^\vee)]^W = \mc{O}(T^\vee / W) = \mc{O}(G^\vee_{ss} / \text{conj}). 
\]
The spherical Hecke algebra is commutative. Therefore, 
\[
\{ \text{irred. } \mc{H}_{\text{sph}}\text{-modules}\} \leftrightarrow \{\chi:[\Rep{G}^\vee]\otimes_\Z \C \rightarrow \C \} = \{\text{semisimple conjugacy classes in }G^\vee(\C)\}
\]
The first bijection follows from the Satake isomorphism and the commutativity of $\mc{H}_{\text{sph}}$, and the second equality is the nullstellensatz. So the Satake isomorphism implies unramified LLC!

\begin{remark}
The goal of the next two Informal Friday Seminars (September 6 and 13) will be to explain the {\bf geometric Satake equivalence}, which is a categorification of Theorem \ref{satake iso}:
\[
(\Perv_{G(\OO)\times G(\OO)}(G(K), \C), \ast) \xleftrightarrow{\sim} (\Rep{G}^\vee_\C, \otimes)
\]
\end{remark}

\subsection{The tamely ramified with unipotent monodromy (TRUM) story}

This semester we will focus on the next simplest piece of the LLC, the case of {\em tamely ramified representations with unipotent monodromy (TRUM)}. In this setting, the LLC gives us a finite-to-one map:
\[
 \left\{ \begin{array}{c} \text{TRUM reps of $G(K)$}\\ \text{(i.e. those which admit} \\ \text{a non-zero $Iw$-fixed vector)} \end{array} \right\}_{/\simeq} \xrightarrow{\text{finite}:1} \left\{ \begin{array}{c} \text{ TRUM Weil-Deligne reps}\\ \text{(i.e. reps which factor }\\ W_K \twoheadrightarrow \Z \rightarrow G^\vee(\C) \\
\text{ with $N$ arbitrary)} \end{array} \right\}_ {/\simeq}
\]
Here $Iw\subset G(K)$ is the {\em Iwahori subgroup}, which sits in the group $G(K)$ in the following way:
\[
\begin{tikzcd}
G(\OO) \arrow[d] \arrow[r, hookleftarrow] & Iw  \arrow[d]\\
G(k) \arrow[r, hookleftarrow] & B \text{ (Borel)}
\end{tikzcd}
\]
By Hecke algebra theory, TRUM representations of $G(K)$ are in bijection with irreducible representations of the {\em Iwahori-Matsumoto Hecke algebra} $\mc{H}_{\text{aff}}=\mc{H}(G(K), I_w)$. On the other hand, TRUM Weil-Deligne representations are parameterized by the set 
\[
\{(s,N) \in G^\vee(\C) \times \mc{N} \mid sNs^{-1}=qN \}/_{\text{conj}}.
\]
Here $s \in G^\vee(\C)$ is the semisimple image of Frobenius, $\mc{N}\subset \Lie{G}^\vee$ is the nilpotent cone, and $q=|k|$. So the LLC predicts a parameterisation of irreducible $\mc{H}_{\text{aff}}$-modules. This prediction is the {\em Deligne--Langland conjecture}, and it served as an important early test case of the Langland's philosophy. 

\vspace{5mm}
\noindent
{\bf Goals for the next few weeks:}
\begin{enumerate}
    \item Discuss the Iwahori-Matsumoto Hecke algebra is some detail. 
    \item Discuss Kazhdan--Lusztig's realisation of the affine Hecke algebra $H$ via equivariant $K$-theory:
    \[
    H \xrightarrow{\sim} K^{G^\vee \times \C^\times}(\mathrm{Steinberg})
    \]
    \item Deduce the Deligne--Langlands conjecture\footnote{Actually, we will not end up deducing the Deligne-Langlands conjecture. We will simply promise the reader that the Deligne-Langlands conjecture is not a long walk from where we get to.}. 
\end{enumerate}
Then we will pass to categorifications! 

\subsection{Affine Weyl groups and affine Hecke algebras}
Let $(X\supset R, X^\vee \supset R^\vee)$ be a root datum. 
\begin{example}
\begin{enumerate}
    \item $\SL_2$: $X=\Z \supset R = \{\pm 2\}$, $X^\vee=\Z \supset R^\vee = \{\pm 1\}$
    \item $\PGL_2$: $X=\Z \supset R=\{\pm 1\}$, $X^\vee =\Z \supset R^\vee = \{\pm 2\}$
\end{enumerate}
We see from this example that $\SL_2$ and $\PGL_2$ are interchanged by swapping roots and coroots, so they are Langlands dual groups. 
\end{example}
Let $W_f$ be the finite Weyl group associated to this root datum. Then $W_f$ acts on both $X$ and $X^\vee$. Assume that our root datum $(X\supset R, X^\vee \supset R^\vee)$ is {\em adjoint}; i.e., $X=\Z R$. (This is the ``most complicated case''.) 
\begin{definition}
\label{extended affine Hecke algebra}
The {\bf extended affine Weyl group}\footnote{This definition is the definition according to Iwahori-Matsumoto and Bourbaki, but we warn the reader that it is {\em not} consistent across all sources! For example, this is not the convention employed in Chriss-Ginzburg.} is
\[
W_{\text{ext}}=\Z X^\vee \rtimes W_f.
\]
\end{definition}
The group $W_{\text{ext}}$ acts on $X_\R^\vee:= X^\vee \otimes_\Z \R$ by ``affine transformations;'' that is, $w \in W_f$ acts as usual, $w (\lambda) = w(\lambda)$, and for $\gamma \in X^\vee,$ $t_\gamma \in \Z X^\vee$ acts by 
\[
t_\gamma(\lambda)= \lambda + \gamma.
\]
To understand $W_{\text{ext}}$, we first consider the {\em affine Weyl group} $W=\Z R^\vee \rtimes W_f \subset W_{\text{ext}}$. For $\alpha \in R$, $m \in \Z$, $\lambda \in X_\R^\vee$, define  
\[
s_{\alpha, m}(\lambda):=\lambda - \langle \lambda, \alpha \rangle \alpha^\vee + m \alpha^\vee. 
\]
Clearly, 
\[
s_{\alpha, m} = t_{m\alpha^\vee} \circ s_\alpha,
\]
so $s_{\alpha, m} \in W$, and $t_{m \alpha^\vee} \in \langle s_{\alpha, m} \mid \alpha \in R, m \in \Z \rangle$. We conclude that 
\[
W=\langle s_{\alpha, m} \mid \alpha \in R, m \in \Z \rangle
\]
is an affine reflection group generated by reflections $s_{\alpha, m}$ through the hyperplanes 
\[
H_{\alpha, m}=\{\lambda \in X_\R^\vee \mid \langle \lambda, \alpha \rangle = m \}.
\]
We call the set $\{ H_{\alpha,m} \}$ the set of \emph{reflecting hyperplanes}. Denote by $\mc{A}$ the corresponding set of {\em alcoves}; that is, the closures of connected components of 
\[
X^\vee_\R \backslash \bigcup_{\alpha \in R \atop m \in \Z} H_{\alpha, m}.
\]
Fix a set of positive roots $R_+ \subset R$, and let 
\[
A_0=\{ \lambda \in X_\R^\vee \mid 0 \leq \langle \lambda, \alpha \rangle \leq 1 \text{ for all }\alpha \in R_+ \}\subset \mc{A}
\]
be the {\em fundamental alcove}. The general (very beautiful) theory of reflection groups gives:
\begin{enumerate}
    \item $W$ is a Coxeter group with Coxeter generators $S:=\{\text{reflections in the walls of $A_0$}\}$.
    \item The length function may be described by 
    \[
    \ell(w)=\# \{ \text{reflecting hyperplanes separating $A_0^{int}$ and $wA_0^{int}$}\}.  
    \]
    \item $A_0$ is a fundamental domain for the $W$-action on $X_\R^\vee$.
\end{enumerate}
So we have an identification 
\begin{align*}
    W &\rightarrow \mc{A} \\
    w &\mapsto wA_0
\end{align*}

\begin{example}
$C_2=B_2$ 
\begin{center}
     \includegraphics[scale=0.6]{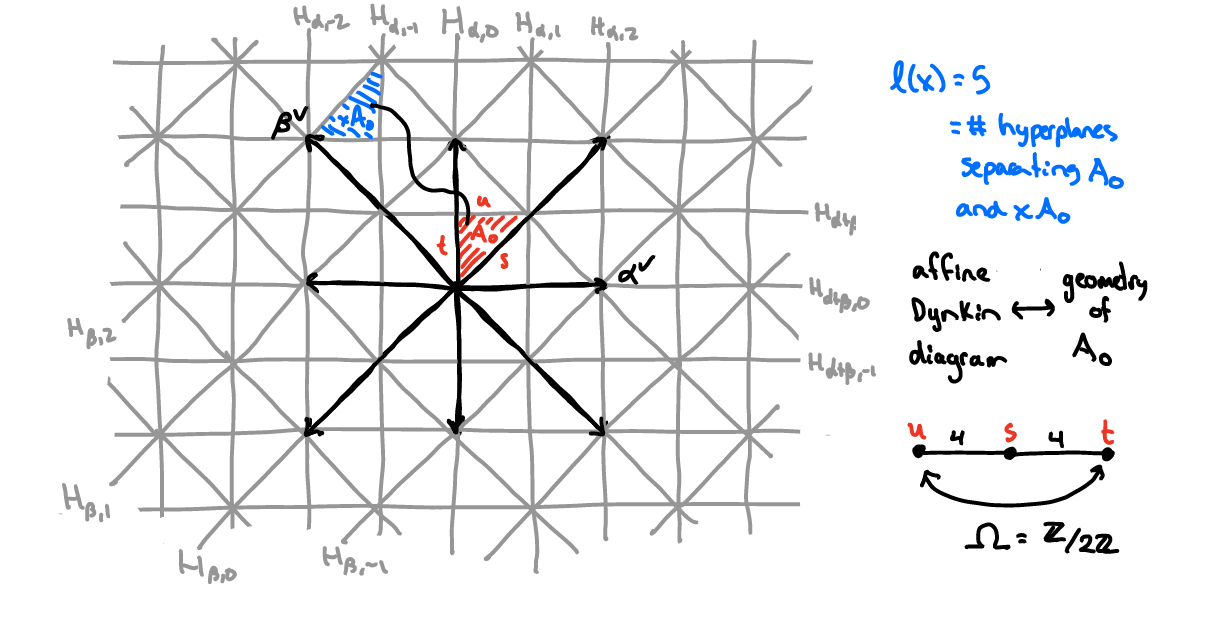}
\end{center}
\end{example}
Now we move on to $W_{\text{ext}}$. There is an action of $W_{\text{ext}}$ on $\mc{A}$ because 
\[
t_\gamma \cdot H_{\alpha, m} = H_{\alpha, m + \langle \gamma, \alpha \rangle}.
\]
Define 
\begin{align*}
    \ell: W_{\text{ext}} &\rightarrow \Z_{\geq 0} \\
    w &\mapsto \ell(w):= \# \{\text{hyperplanes separating $A_0^{int}$ and $wA_0^{int}$}\}
\end{align*}
Define the {\em length zero} elements of $W_{\text{ext}}$ to be 
\[
\Omega:=\ell^{-1}(0).
\]
\begin{lemma}
$W_{\text{ext}}=\Omega \ltimes W$.
\end{lemma}
\begin{proof} 
{\bf Step 1:} $W \subset W_{\text{ext}}$ is normal. \\

Let $\gamma \in X^\vee$, $\lambda \in X_\R^\vee$. Then
\begin{align*}
t_\gamma s_{\alpha, m} t_\gamma^{-1}(\lambda) &=t_\gamma(s_{\alpha, m}(\lambda - \gamma) \\
&= \lambda - \gamma - \langle \lambda - \gamma, \alpha \rangle \alpha^\vee + m \alpha^\vee + \gamma \\
&= \lambda - \langle \lambda, \alpha \rangle \alpha^\vee + (\langle \gamma, \alpha \rangle + m ) \alpha^\vee \\
&=s_{\alpha, \langle \gamma, \alpha \rangle + m}(\lambda). 
\end{align*}

\noindent
{\bf Step 2:} $W_{\text{ext}}=W \cdot \Omega$.\\

Let $w \in W_{\text{ext}}$. Then $wA_0 \in \mc{A}$, so by 3. above, there exists $y \in W$ such that $ywA_0=A_0$. Hence $yw=\omega$ for $\omega \in \Omega$, and $w=y^{-1}\omega$. \\

\noindent
{\bf Step 3:} $W \cap \Omega = \{id\}$.\\

Any $w\in W \cap \Omega$ is length zero, so $w=id$ by 2. above and the fact that W is a Coxeter group.
\end{proof}

\begin{lemma}
\label{IM}
(Iwahori-Matsumoto) $w \in W_f$, $\lambda \in X^\vee$, 
\[
\ell(t_\lambda w) = \sum_{\alpha \in R^+ \atop w^{-1}(\alpha)>0} |\langle \lambda, \alpha \rangle | + \sum_{\alpha \in R^+ \atop w^{-1}(\alpha) < 0} | \langle \lambda, \alpha \rangle -1|.
\]
\end{lemma}
\begin{proof}
Let $x=t_\lambda w$. Then 
\begin{align*}
    \ell(x)&=\# \{ \text{hyperplanes separating $A_0^{int}$ and $xA_0^{int}$}\} \\
    &=\sum_{\alpha \in R_+}\# \{m \mid H_{\alpha, m} \text{ separates $A_0^{int}$ and $xA_0^{int}$}\}.
\end{align*}
Now chose a point $p \in A_0$, very close to zero. Then $\langle p, \alpha \rangle$ is small and positive, and
\begin{align*}
    \ell(x) &= \sum_{\alpha \in R_+} \#\{ \text{integers between $\langle p, \alpha \rangle$ and $\langle xp=\lambda + wp, \alpha \rangle$}\} \\
    &= \sum_{\alpha \in R_+} \begin{cases} |\langle \lambda, \alpha \rangle | &\text{ if $\langle wp, \alpha \rangle >0$} \\ 
    |\langle \lambda, \alpha \rangle -1| &\text{ if $\langle wp, \alpha \rangle <0$} \end{cases}\\ 
    &= \sum_{\alpha \in R_+ \atop w^{-1} \alpha >0} |\langle \lambda, \alpha \rangle| + \sum_{\alpha \in R_+ \atop w^{-1} \alpha < 0 } |\langle \lambda, \alpha \rangle - 1|.
\end{align*}
\end{proof}

\begin{example}
For $\PGL_2$, $X^\vee = \Z \varpi_1$. Then 
\[
\ell(m \varpi_1) = |m|, \text{ and }\ell(m \varpi_1 s)=|m-1|.
\]
Hence $\varpi_1s$ is length zero, and $\Omega = \{ id, t_{\varpi_1}s\}$. 
\end{example}

%% file: lecture-12.tex
\section{Lecture 12: The Deligne-Langlands conjecture}
\label{lecture 12}

\subsection{The Iwahori--Matsumoto Hecke algebra} 

Recall our setup from last week. From an adjoint root datum $(X\supset R, X^\vee \supset R^\vee)$, (i.e.; meaning $X=\Z R$) we construct
\begin{itemize}
    \item the finite Weyl group $W_f$, 
    \item the affine Weyl group $W=\Z R^\vee \rtimes W_f$, and
    \item the extended affine Weyl group $W_{\text{ext}}=\Z X^\vee \rtimes W_f$.
\end{itemize}
We fix a set of positive roots $R_+^\vee \subset R^\vee$, then obtain the fundamental alcove $A_0$, and the corresponding set $S\subset W$ of Coxeter generators. We define a length function $\ell$ by
\begin{align*}
    \ell:W_{\text{ext}}&\rightarrow \Z \\
    x &\mapsto \# 
    \left\{ \begin{array}{c} \text{reflecting hyperplanes }\\ \text{between $A_0^{int}$ and $xA_0^{int}$} \end{array} \right\}.
\end{align*}
We denoted the length zero elements by $\Omega = \ell^{-1}(0)$, and showed that $W_{\text{ext}}$ is a ``quasi Coxeter group,'' meaning that 
\[
W_{\text{ext}}=\Omega \ltimes W,
\]
$W$ is a Coxeter group, and $\Omega$ acts on $W$ via automorphisms of the Coxeter system. \begin{example}
For $\PGL_2$, 
\begin{center}
     \includegraphics[scale=0.5]{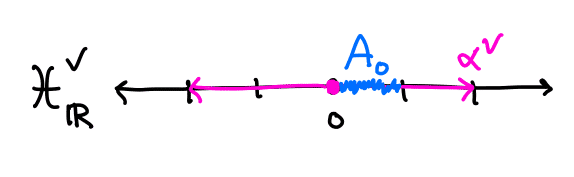}
\end{center}
The affine Weyl group and extended affine Weyl group are 
\[
W=\langle s, t \mid s^2=t^2=id \rangle \subset W_\text{ext}=\langle s, t, \tau=\varpi s \mid s^2=t^2=\tau^2=id, \tau s=t\tau \rangle. 
\]
\end{example}

We can use the extended affine Weyl group to define a Hecke algebra. 
\begin{definition}
The {\em Iwahori--Matsumoto Hecke algebra} $\mc{H}_{\text{ext}}$ is the $\Z[v^{\pm 1}]$-algebra with basis $\{H_x \mid x \in W_\text{ext}\}$ and multiplication 
\begin{align*}
    H_x H_y=H_{xy} \text{ if } \ell(xy)=\ell(x)+\ell(y) \\
    H_s^2 = H_{id} + (v^{-1}-v)H_s \text{ if }s \in S. 
\end{align*}
Define $\mc{H}=\langle H_x \mid x \in W \rangle \subset \mc{H}_\text{ext}$. 
\end{definition}
\begin{remark}
\begin{enumerate}
    \item $\mc{H} \subset \mc{H}_{\text{ext}}$ is a subalgebra. 
    \item For all $x \in W_\text{ext}$, $H_x$ is invertible. 
    \item If $\tau \in \Omega$, then $\ell(x \tau)=\ell(x) = \ell(\tau x)$. Hence 
    \[
    \mc{H}_\text{ext}=\Omega \ltimes \mc{H}.
    \]
\end{enumerate}
\end{remark}

\noindent
{\bf Question:} Where did the loop presentation $W_{\text{ext}}=\Z X^\vee \rtimes W_f$ go?

\vspace{5mm}

For $\lambda \in X^\vee$, write $\lambda = \gamma - \gamma'$ with $\gamma, \gamma' \in X_+^\vee$. Define 
\[
H_\lambda := H_{t_\gamma}H_{t_{\gamma'}}^{-1}.
\]
Note that $H_{t_\lambda} \neq H_\lambda$ in general! For example, if $\lambda \in X^\vee_+$, then $H_\lambda = H_{t_\lambda}$, but if $\lambda \in -X^\vee_+$, then $H_\lambda = H_{t_{-\lambda}}^{-1}$. 

\vspace{5mm}
\noindent
{\bf Why is this well-defined?} Assume that $\lambda = \gamma - \gamma' = \mu - \mu'$ for $\gamma, \gamma', \mu, \mu' \in X_+^\vee$. To show that $H_\lambda$ is well-defined, we need to show that 
\[
H_{t_\gamma}H^{-1}_{t_{\gamma'}} = H_{t_\mu}H^{-1}_{t_{\mu'}}. 
\]
This is equivalent to showing that for $\zeta \in X^\vee$, 
\[
H_{t_\gamma}H^{-1}_{t_{\gamma'}} H_{t_\zeta}= H_{t_\mu}H^{-1}_{t_{\mu'}}H_{t_\zeta}. 
\]
If we choose $\zeta \in X^\vee_+$ very dominant, then $\zeta - \mu'$ is dominant and
\[
H_{t_{\zeta - \mu'}}H_{t_{\mu'}} = H_{t_{\zeta - \mu'}t_{\mu'}}=H_{t_{\zeta}} =H_{t_{\mu'}t_{\zeta - \mu'}}= H_{t_{\mu'}}H_{t_{\zeta - \mu'}}. 
\]
The first and fourth equalities follow from the fact that the lengths of $t_{\zeta - \mu'}$ and $t_{\mu'}$ add by the Iwahori--Matsumoto lemma (Lemma \ref{IM}). Hence, 
\[
H_{t_\gamma}H^{-1}_{t_{\gamma'}} H_{t_\zeta}=H_{t_\gamma}H_{t_{\zeta - \gamma'}}=H_{t_{\gamma+\zeta-\gamma'}}=H_{t_{\mu + \zeta - \mu'}} = H_{t_\mu}H_{t_{\zeta - \mu'}}=H_{t_\mu}H^{-1}_{t_{\mu'}}H_{t_\zeta},
\]
so $H_\lambda$ is well-defined.

\vspace{5mm}
\noindent
{\bf The upshot:} When studying representations of an algebra, it is useful to have a large commutative subalgebra. This is what we have just accomplished for the Iwahori--Matsumoto Hecke algebra: the map $\lambda \mapsto H_\lambda$ determines an embedding 
\[
\Z[v^{\pm 1}][X^\vee] \hookrightarrow \mc{H}_{\text{ext}}.
\]

We can use this commutative subalgebra to describe the center of $\mc{H}_\text{ext}$. 
\begin{theorem}
(Bernstein) For any $\lambda \in X^\vee_+$, define 
\[
z_\lambda:= \sum_{\mu \in W_f \lambda}H_\mu.
\]
Then the center $Z(\mc{H}_\text{ext})$ of $\mc{H}_\text{ext}$ is a free $\Z[v^{\pm 1}]$-module with basis $\{z_\lambda \mid \lambda \in X_+^\vee\}$, and 
\[
Z(\mc{H}_\text{ext})=\left( \Z[v^{\pm 1}][X^\vee] \right) ^{W_f}.
\]
\end{theorem}

Now we can state the Bernstein presentation of $\mc{H}_\text{ext}$. 
\begin{theorem}
\label{Bernstein presentation}
(Bernstein presentation) The Iwahori--Matsumoto Hecke algebra admits the following presentation. 
\begin{enumerate}
    \item $\langle H_s \mid s \in S_f \rangle$ generate a finite Hecke algebra (``finite part''). 
    \item $H_\lambda H_\gamma = H_{\lambda + \gamma}$ for all $\lambda, \gamma \in X^\vee$ (``lattice part'').  
    \item For $\lambda \in X^\vee, s_\alpha \in S_f$, 
    \begin{align*}
        H_{s_\alpha}H_{s_\alpha(\lambda)} - H_\lambda H_{s_\alpha}&=(v-v^{-1}) \left( \frac{H_\lambda - H_{s_\alpha(\lambda)}}{1-H_{-\alpha}} \right)\\
        &=(v-v^{-1})(H_\lambda + H_{\lambda - \alpha} + \cdots + H_{s_\alpha(\lambda)+\alpha}).
    \end{align*}
\end{enumerate}
In other words, we have
\[
\mc{H}_\text{ext}\simeq \Z[v^{\pm 1}][X^\vee]\otimes \mc{H}_f.  
\]
\end{theorem}
We will check the relations for $\PGL_2$. First note that if 3. holds for $\lambda$ and $\gamma$, then it holds for $\lambda+\gamma$:
\begin{align*}
    H_{s_\alpha}H_{s_\alpha(\lambda)+s_\alpha(\gamma)} &=(v-v^{-1})\left( \frac{H_\lambda - H_{s_\alpha(\lambda)}}{1-H_{-\alpha}} \right) H_{s_\alpha(\gamma)}+H_\lambda H_{s_\alpha}H_{s_\alpha(\gamma)} \\
    &=(v-v^{-1})\left( \frac{H_{\lambda + s_\alpha(\gamma)} - H_{s_\alpha(\lambda) + s_\alpha(\gamma)}+H_{\lambda+\gamma}-H_{\lambda - s_\alpha(\gamma)}}{1-H_{-\alpha}}\right) + H_{\lambda + \gamma} H_{s_\alpha}\\
    &= (v-v^{-1})\left( \frac{H_{\lambda + \gamma}-H_{s_\alpha(\lambda + \gamma)}}{1-H_{-\alpha}} \right) +H_{\lambda + \gamma}H_{s_\alpha}.
\end{align*}
In particular, if 3. is true for $\lambda$, then it is true for $-\lambda$. 

\vspace{5mm}
\noindent
{\bf For $\PGL_2$:} We will check 3. for $\lambda=\varpi$. We have that $\langle \varpi, \alpha \rangle=1$, $\tau=t_\varpi s_\alpha$, so $t_\varpi=\tau s_\alpha$, and $H_\varpi=H_\tau H_{s_\alpha}$. We compute
\begin{align*}
    H_{s_\alpha}H_{-\varpi}-H_\varpi H_{s_\alpha}^{-1} &=
    H_{s_\alpha}H_{s_\alpha}^{-1}H_\tau-H_\tau H_{s_\alpha}H_{s_\alpha} \\
    &=H_\tau - H_\tau(1+(v^{-1}-v)H_{s_\alpha}) \\
    &=(v-v^{-1})H_\tau H_{s_\alpha} \\
    &=(v-v^{-1})H_\varpi.
\end{align*}

\subsection{The Deligne--Langlands conjecture}
\label{Deligne-Langlands conjecture}

Now we return to the LLC. Let $K$ be a local field with ring of integers $\OO$ and residue field $k$. Define $q:=|k|$. Let $G/K$ be a split reductive group, and $(X \supset R, X^\vee \supset R^\vee)$ the corresponding root datum. 

Recall that the Deligne--Langlands conjecture tells us that we should expect the following relationships:
\[
\begin{tikzcd}
 \left\{ \begin{array}{c} \text{TRUM reps }\\ \text{of $G(K)$} \end{array} \right\}_{/\simeq} \arrow[r, "finite:1"] \arrow[d, leftrightarrow, "\sim"] & 
 \left\{ \begin{array}{c} \text{ TRUM Weil-} \\ \text{Deligne reps}  \end{array} \right\}_ {/\simeq} \arrow[d, leftrightarrow, "="]\\
 \left\{ \begin{array}{c} \text{irred $\mc{H}(G(K), Iw)$-} \\ \text{modules} \end{array} \right\}_{/\simeq} & 
\left\{ \begin{array}{c} (s,x) \in G_\C^\vee \times \mf{g}_\C^\vee \\ \text{s.t. $s$ is ss, $x$ nilp} \\ \text{and $sxs^{-1}=qx$} \end{array} \right\}_{/conj} 
\end{tikzcd}
\]
The following theorem relates the extended affine Hecke algebra of Iwahori--Matsumoto to this story.
\begin{theorem}
\label{IM theorem}
(Iwahori--Matasumoto) \begin{enumerate}
    \item $G(K)=\bigsqcup_{w \in W_\text{ext}}Iw \cdot w \cdot Iw$ (``Bruhat decomposition')
    \item There is an isomorphism of algebras
    \begin{align*} 
    \mc{H}_\text{ext}\otimes_{\Z[v^{\pm 1}]} \C \xrightarrow{\sim} \mc{H}(G(K), Iw) 
    \end{align*}
    where $\C$ is a $\Z[v^{\pm 1}]$-algebra via $v \mapsto (\sqrt{q})^{-1} \in \R_+ \subset \C$.
\end{enumerate}
Moreover, under 2., $H_x$ is mapped to the indicator function on $Iw \cdot x \cdot Iw$, up to a scalar. 
\end{theorem}

\begin{example}
Let $G=\GL_n$, and fix a uniformizer $\pi \in \OO$. Then 
\[
{W_\text{ext} = \atop \text{ } } {\left\langle \begin{array}{c} \text{permutation}\\ 
\text{matrices} \end{array} \right\rangle \atop \text{``finite part''}}{ \ltimes \atop \text{ } } {\left \langle \bp \pi^{\lambda_1} & \cdots & 0 \\ 0 & \ddots & 0 \\ 0 & \cdots & \pi^{\lambda_n} \ep \right\rangle. \atop \text{ ``lattice part''}}
\]
\end{example}

By Theorem \ref{IM theorem}, we can understand TRUM representations of $G(K)$ by studying irreducible $\mc{H}_\mathrm{ext}$-modules. Denote by $Z:=Z(\mc{H}_\mathrm{ext})=(\Z[v^{\pm 1}][X^\vee])^{W_f}$. By Quillen's Lemma (which is an infinite-dimensional version of Schur's Lemma), $Z$ acts by scalars on any irreducible $\mc{H}_\mathrm{ext}$-module. The Bernstein presentation tells us that $\mc{H}_\mathrm{ext}$ is finite over $R:=\Z[v^{\pm 1}][X^\vee]$. Since $R$ is also finite over $R^{W_f}$, we conclude that any irreducible $\mc{H}_\text{ext}$-module is finite-dimensional, and, in fact, is of dimension $\leq |W_f|^2$.  

Hence, the Deligne--Langlands conjecture predicts the following relationships. 
\[
\begin{tikzcd}
 \left\{ \begin{array}{c} \text{irreps/$\C$ }\\ \text{of $\mc{H}_\text{ext}$} \end{array} \right\} \arrow[r, "finite:1"] \arrow[d, "\text{central character}"] & 
 \left\{ \begin{array}{c} (s,x) \in G_\C^\vee \times \mf{g}_\C^\vee \\ \text{s.t. $s$ is ss, $x$ nilp} \\ \text{and $sxs^{-1}=qx$} \end{array} \right\}_ {/\simeq} \arrow[dd]\\
 \left\{ \begin{array}{c} \text{irreps of $Z$} \end{array} \right\} \arrow[d, leftrightarrow, "="] & \text{ } \\
\left\{\text{pairs } (s', v) \in G_\C^\vee \times \C^\times \right\}_{/G_\C^\vee \text{ conj}} \arrow[r, dashrightarrow] & \{(s, x) \in G_\C^\vee \times \C^\times \}_{/G^\vee_\C \text{ conj}} 
\end{tikzcd}
\]
The dashed arrow should match $v^{-1} \leftrightarrow \sqrt{q}$. 
\begin{remark}
\begin{enumerate}
    \item We are no longer forced to take $q=|k|$. 
    \item In the diagram above, we are repeatedly using the fact from last semester that 
\[
\{ \text{characters } \chi:\Z[X^\vee]^{W_f} \rightarrow \C\} \leftrightarrow \{\text{semisimple elts of $G$ up to conjugacy}\}.
\]
\end{enumerate}
\end{remark}

\noindent
{\bf An easy and interesting case:} Consider the case when $s=id$ and $q=1$. In this case, the right side of the dashed arrow is 
\[
\{x \in \mf{g}^\vee_\C \text{ nilpotent}\}_{/G^\vee_\C \text{ conjugacy}} = \text{ ``nilpotent orbits''.}
\]
This provides a hint that we should not expect to have a good algebraic grip on the problem of understanding the irreducible representations of $\mc{H}_\text{ext}$, because nilpotent orbits are complicated and not combinatorial in general. For the next lecture and a half, we will dive into this geometry. 

\subsection{Geometric setting}
For notational convience, we will temporarily swap $G \leftrightarrow G^\vee$ in this section. Let $G/\C$ be a complex reductive group, and $\mc{N}\subset \Lie{G}=:\mf{g}$ the nilpotent cone. 
\begin{remark}
\label{definition of nilcone}
For $\GL_n$, it is tempting to define $\mc{N}$ as the variety 
\[
\{ x \in \mf{gl}_n(\C) \mid x^n=0\}. 
\]
However, this results in a non-reduced scheme, because the ideal corresponding to the equation $x^n=0$ is not radical. A better definition is to consider 
\[
\{x \in \mf{gl}_n(\C) \mid \text{ coeffients of the characteristic polynomial vanish}\}. 
\]
This still captures what we know as a nilpotent matrix, but results in a better geometric object. (In particular, it is a reduced scheme.)  
\end{remark}
In general, we define the nilpotent cone as follows. Consider the adjoint quotient map
\[
\mf{g} \xrightarrow{q} \mf{g}/G=\mf{h}/W. 
\]
The equality $\mf{g}/G=\mf{h}/W$ follows from Chevalley's theorem. For $\mf{gl}_n$, the map $q$ is ``take coefficients of the characteristic polynomial,'' so this captures what we wanted in Remark \ref{definition of nilcone}.  We define 
\[
\mc{N}=q^{-1}(0). 
\]
\noindent
{\bf Fundamental facts about the nilpotent cone:}
\begin{enumerate}
    \item $\mc{N}$ is irreducible, reduced, and normal. 
    \item $G$ has finitely many orbits on $\mc{N}$, and all are even-dimensional (over $\C$). 
\end{enumerate}
\begin{example}
If $G=\GL_n$, then 
\[
\left\{ \begin{array}{c} \text{nilpotent} \\ \text{matrices} \end{array} \right\}_{/conj} 
\leftrightarrow 
\left\{ \begin{array}{c}  \text{Jordan} \\ \text{normal} \\ \text{form} \end{array} \right\}  \leftrightarrow \left\{ \begin{array}{c} \text{partitions} \\ \lambda \vdash n \end{array}  \right\}
\]
Moreover, if we denote by $\mc{O}_\lambda$ the orbit corresponding to the partition $\lambda$, then
\[
\mc{O}_\mu \subset \overline{\mc{O}_\lambda} \iff \mu \leq \lambda \text{ in dominance order}. 
\]
We can examine these orbit relations explicitly for small $n$. 

\vspace{5mm}
\noindent
{\bf $n=2$:} 
\[
\mc{N} = \left\{ x=\bp a & b \\ c & d \ep \mid \Tr{x}=\det{x}=0 \right\} = \left\{ x=\bp a & b \\ c & -a \ep \mid \det{x}=-a^2-bc=0 \right\} \subset \C^2. 
\]
A $\R$-picture of this is: 
\begin{center}
     \includegraphics[scale=0.5]{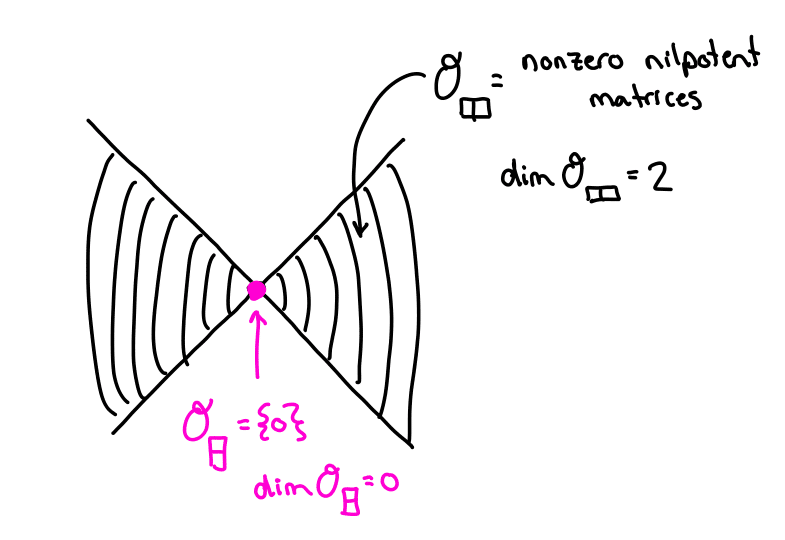}
\end{center}
\noindent
{\bf $n=3$:} For $n=3$, the picture is more of a caricature. 
\begin{center}
     \includegraphics[scale=0.5]{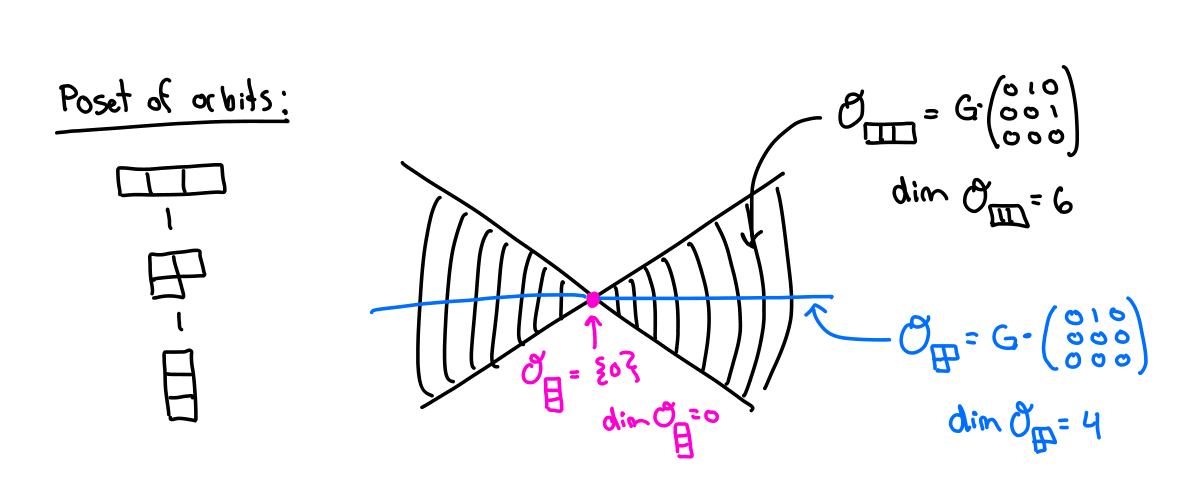}
\end{center}
\begin{exercise}
Try to do this for large $n$ (perhaps 7 or 8) and see that this poset is rather ugly; in particular, it is not graded. 
\end{exercise}
\end{example}

\subsection{The Springer resolution}
We can study these nilpotent orbits by using a resolution of singularities of the nilpotent cone. Let $\mc{B}$ be the variety of Borel subalgebras in $\mf{g}$. For $\mf{gl}_n$, this is the variety of complete flags in $\C^n$. 
\begin{definition}
The {\em Springer resolution} is the map 
\[
\begin{tikzcd}
\widetilde{\mc{N}} \arrow[d]= \{(x, \mf{b}) \in \mc{N} \times \mc{B} \mid x \in \mf{b} \}\simeq T^*\mc{B} \\
\mc{N} 
\end{tikzcd}
\]
which sends $(x, \mf{b}) \mapsto x$. 
\end{definition}
Next week we will study the Springer resolution more carefully, and show that it is proper, smooth, and a resolution of singularities. For $\GL_n$, 
\[
\widetilde{\mc{N}}=\{(x, \{0\}=V_0 \subset V_1 \subset \cdots V_n=\C^n) \mid x \text{ preserves flag } \iff xV^i \subset V^{i-1}\}. 
\]
\begin{example}
We examine $\GL_n$ for small $n$ again. 

\vspace{5mm}
\noindent
{\bf $n=2$:} The variety of Borel subalgebras is $\mc{B} = \PP^1\C = \{\text{lines in }\C^2\}$, and the Springer resolution looks like:
\begin{center}
     \includegraphics[scale=0.5]{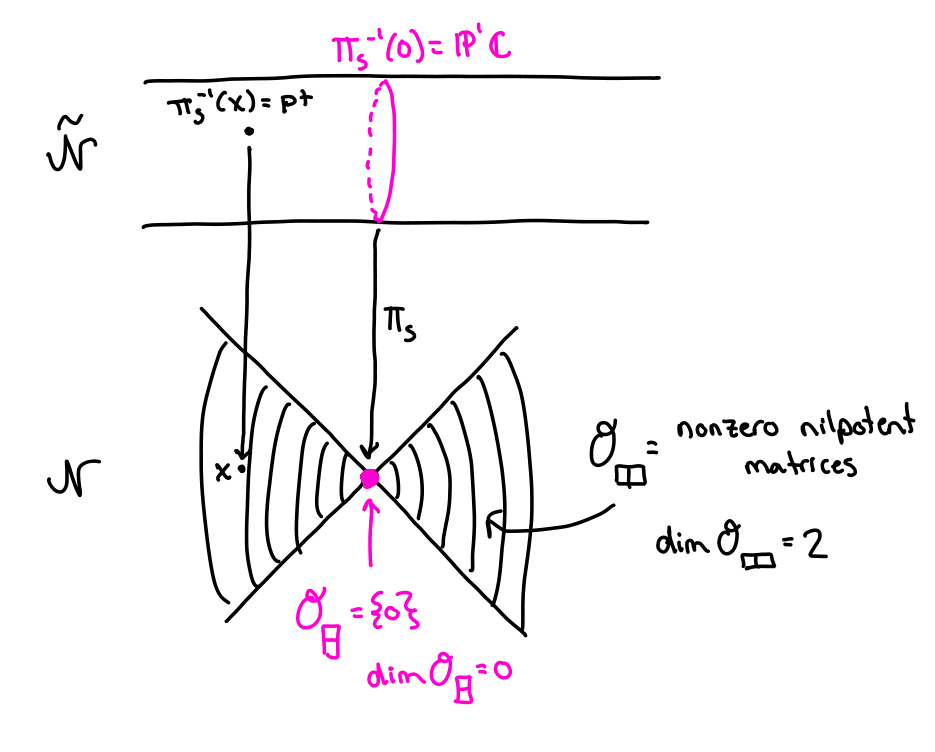}
\end{center}

\noindent
{\bf $n=3$:} The variety of Borel subalgebras is $\mc{B}=\{(\ell, P) \mid \ell \subset P \subset \C^2\}$ and a caricature of the Springer resolution looks like:
\begin{center}
     \includegraphics[scale=0.5]{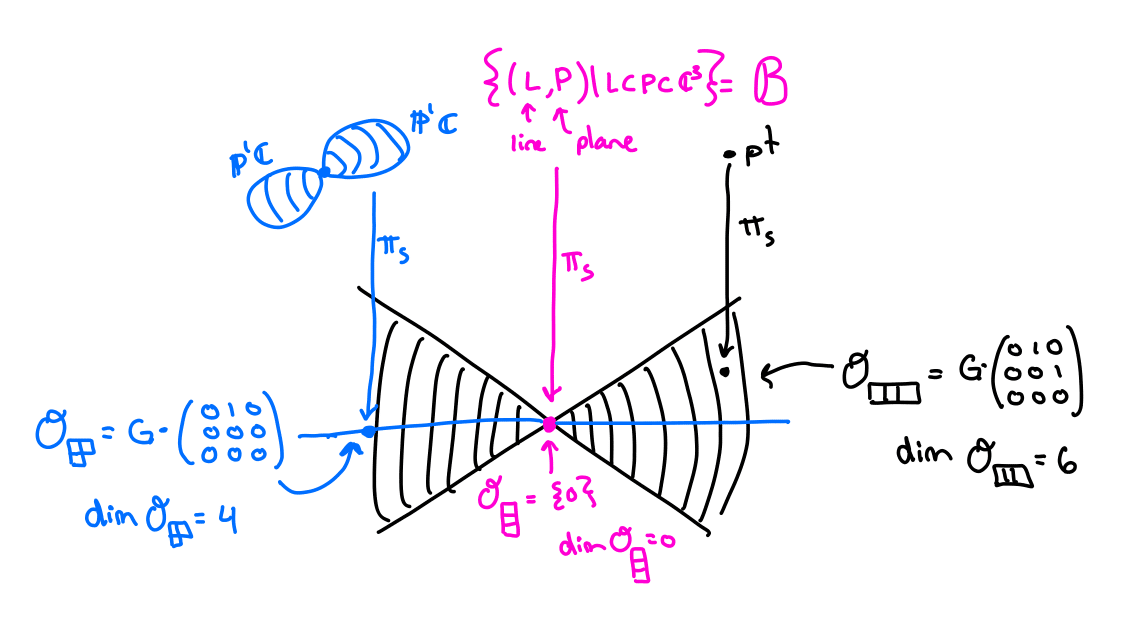}
\end{center}
We will explain in more detail how we arrived at this picture next week.
\end{example}

%% file: lecture-13.tex
\section{Lecture 13: Springer fibres and the Steinberg variety}
\label{lecture 13}

\subsection{The Springer resolution, continued}

Today we are going to pick up where we left off last week and continue to discuss the geometry of nilpotent orbits. To begin, we will discuss in more detail our claim from last week that $\widetilde{N} \simeq T^*\mc{B}$. 

Let $G$ be an algebraic group and $\mf{g}$ its Lie algebra. Canonically, we can express the tangent bundle of $G$ as
\[
TG=G \times \mf{g},  
\]
since the Lie algebra can be identified with $T_e(G)$. Hence if $X$ is a homogeneous space for $G$ (i.e. $G$ acts on $X$ transitively), then we have a surjection 
\[
TG=G \times \mf{g} \rightarrow TX,
\]
and for $x \in X$, $T_xX=\mf{g}/\Lie{(\stab_G{x})}$. Similarly, there is a canonical identification of the cotangent bundle of $G$
\[
T^*G=G\times \mf{g}^*
\]
and for a homogeneous space $X$, 
\[
T^*X \hookrightarrow T^*G.
\]
Moreover, for $x \in X$, $T_xX=(\Lie{(\stab_G{x})})^\perp \subset \mf{g}^*$.

Now assume that $G$ is semisimple and let $\mc{B}$ be the variety of Borel subalgeras of $\mf{g}$. Once we choose a  Borel subalgebra $B \subset G$, we can identify $\mc{B} \simeq G/B$ since $\stab_G\mf{b}=B$. Last lecture we introduced the following space 
\[
\widetilde{\mc{N}} = \{(\mf{b}, x) \in \mc{B} \times \mc{N} \mid x \in \mf{b} \}. 
\]
\begin{claim}
There is a canonical isomorphism $\widetilde{\mc{N}} \simeq T^*\mc{B}$. 
\end{claim}
\begin{proof}
For a point $\mf{b} \in \mc{B}$, the tangent space to $\mc{B}$ at $\mf{b}$ is 
\[
T_{\mf{b}}\mc{B} = \mf{g}/(\Lie{(\stab_G{\mf{b}})}=\mf{b}).
\]
Hence, 
\[
T^*\mc{B}=\{(\mf{b}, v) \in \mc{B} \times \mf{g}^* \mid v \in \mf{b}^\perp\}. 
\]
Recall the Killing form $\kappa:\mf{g} \times \mf{g} \rightarrow \C$, $\kappa(x, y)=\tr{(\ad{x}\ad{y})}$. This is a symmetric, nondegenerate bilinear form. Fix a Cartan subalgebra $\mf{h} \subset \mf{b}$, and we obtain a triangular decomposition $\mf{g}=\mf{n}_- \oplus \mf{h} \oplus \mf{n}_+$, with $\mf{b}=\mf{h} \oplus \mf{n}_+$. The restriction $\kappa|_{\mf{h}\times \mf{h}}$ is nondegenerate, and $\kappa$ gives a nondegenerate pairing 
\[
\kappa: \mf{n}_- \times \mf{n}_+ \rightarrow \C. 
\]
Hence under the identification $\mf{g} \simeq \mf{g}^*$ via the Killing form, $\mf{b}^\perp \subset \mf{g}^*$ corresponds to $\mf{n}_+ \subset \mf{g}$. We conclude that 
\begin{align*}
T^*\mc{B} &= \{ (\mf{b}, x) \in \mc{B} \times \mf{g} \mid x \in \mf{b} \text{ is nilpotent} \}\\
&= \{(\mf{b}, x) \in \mc{B} \times \mc{N} \mid x \in \mf{b}\} \\ 
&=\widetilde{\mc{N}}. 
\end{align*}
\end{proof}

\subsection{Many examples of Springer fibres}

\vspace{5mm}
\begin{center}
``The richness of Springer fibres cannot be underestimated.'' 
\end{center}

\vspace{5mm}

Recall the Springer resolution
\[
\begin{tikzcd}
\widetilde{\mc{N}} \arrow[d, "\pi_s"]= \{(x, \mf{b}) \in \mc{N} \times \mc{B} \mid x \in \mf{b} \}\simeq T^*\mc{B} \\
\mc{N} 
\end{tikzcd}
\]
\begin{definition}
Let $x \in \mc{O}_\lambda \subset \mc{N}$ be a point in a $G$-orbit. The associated {\em Springer fibre} $F_\lambda$ is 
\[
F_\lambda:=\pi_s^{-1}(x).
\]
By equivariance, this is independent of the choice of $x \in \mc{O}_\lambda$, up to isomorphism. 
\end{definition}
Let $G=\SL_n$. This section is devoted to studying Springer fibres for $n=2,3,4$. Recall that for $G=\SL_n$, 
\[
\widetilde{\mc{N}}=\{(F^\cdot, x) \in {\text{complete} \atop \text{flags in } \C^n} \times \mc{N} \mid x \text{ preserves }F^\cdot \} = \{(F^\cdot, x)  \mid xF^i \subset F^{i-1} \}.
\]

 \noindent
 {\bf Example $\mathbf{n=2}$:} The nilpotent cone is the quadric cone
    \[
    \mc{N} = \left\{ \bp a & b \\ c & -a \ep \in \mf{sl}_2 \mid -a^2 -bc = 0 \right\}\subset \C^3. 
    \]
    A $\R$-picture:
    \begin{center}
     \includegraphics[scale=0.5]{images/springer2.png}
\end{center}
A table of orbits:
\[
\begin{tabular}{|c||c|c|}
\hline
    $\lambda$ & \includegraphics[scale=0.15]{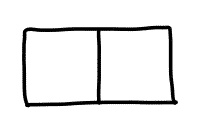} & \includegraphics[scale=0.15]{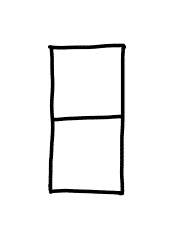}\\
     \hline \hline
    $\dim{\mc{O}_\lambda}$& 2 & 0 \\
     \hline
    $\codim{\mc{O}_\lambda}$ &0 & 2\\
     \hline
     Springer fibre & $pt$ & $\PP^1\C$ \\
     \hline 
     $\dim{\text{fibre}}$ & 0 & 1 \\
     \hline
\end{tabular}
\]

\vspace{5mm}
\noindent
{\bf Example $\mathbf{n=3}$:} 
A caricature:
\begin{center}
\includegraphics[scale=0.5]{images/springer3.png}
\end{center}
We'd like to construct a table as we did in the previous example, but to do so, we need to determine the dimension\footnote{If you'd like to understand some variety, it is a good idea to know its dimension!} of the nilpotent orbits. This is easy for the regular nilpotent orbit and the zero orbit. But what about the other orbit $\mc{O}_\lambda$? Well we know that the matrix 
\[
e=\bp 0 & 1 & 0 \\ 0 & 0 & 0 \\ 0 & 0 & 0 \ep 
\]
is in the orbit, and we observe that
\[
\Lie{(\stab_{\SL_3}{e})}=\stab_{\Lie{\SL_3}}{e}=\ker \ad{e}. 
\]
A computation shows that 
\[
\ker \ad{e} = \left\{ \bp a & b & c \\
0 & a & 0 \\ 
0 & d & -2a \ep \mid a,b,c,d \in \C \right\}
\]
is $4$-dimensional, hence the orbit $\mc{O}_\lambda$ is $(8-4=4)$-dimensional. We also need to determine the Springer fibre corresponding to this orbit. We have that 
\[
F_\lambda = \{(L \subset P \subset \C^3) \mid e\C^3 \subset P, eP \subset L, \text{ and }eL =0\},
\]
but what does this look like? Well we know that
\[
0 \subset \text{im}e \subset \ker{e} \subset \C^3,
\]
and the condition that a flag in $F_\lambda$ must satisfy $eL=0$ implies that $L \subset \ker{e}$. We see that there are two possibilities for flags in $F_\lambda$:
\begin{itemize}
    \item $\{L=\text{im}e$, free choice of $P$, as long as $L \subset P \}\simeq \PP^1\C$
    \item $\{P=\ker{e}$, free choice of $L$, as long as $L \subset P \} \simeq \PP^1\C$ 
\end{itemize}
These two cases intersect when $L=\text{im}e$ and $P=\ker{e}$, which is a single point, so our Springer fibre looks like two $\PP^1\C$'s joined at a point. With this, we can complete our table of orbits: 
\[
\begin{tabular}{|c||c|c|c|}
\hline
    $\lambda$ & \includegraphics[scale=0.15]{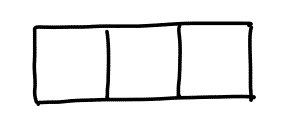} & \includegraphics[scale=0.15]{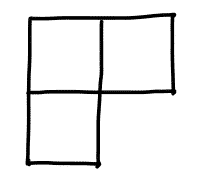} & \includegraphics[scale=0.15]{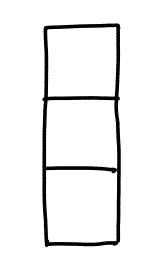}\\
     \hline \hline
    $\dim{\mc{O}_\lambda}$& 6 & 4 & 0\\
     \hline
    $\codim{\mc{O}_\lambda}$ &0 & 2 & 6\\
     \hline
     Springer fibre & $pt$ & \includegraphics[scale=0.19]{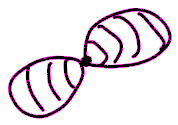}& $\mc{B}$ \\
     \hline 
     $\dim{\text{fibre}}$ & 0 & 1 & 3 \\
     \hline
\end{tabular}
\]

\vspace{5mm}
\noindent
{\bf Example $\mathbf{n=4}$:} Drawing pictures is no longer so reasonable, but we can still count dimensions and make our table. The dimension of $\mc{B}$ is the number of positive roots (which for $\SL_4$ is $6$), so we know that $\dim{T^*{\mc{B}}}=12$. 

To find the dimensions of the Springer fibres, we can play a similar game to what we did in the previous example. Consider the orbit $\mc{O}_\lambda$ corresponding to 
\[
\lambda = \begin{array}{c}\includegraphics[scale=0.2]{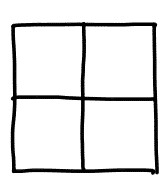}
\end{array}
\]
The matrix
\[
e=\bp 0 & 1 & 0 & 0 \\ 
0 & 0 & 0 & 0 \\
0 & 0 & 0 & 1 \\
0 & 0 & 0 & 0 \ep
\]
is in $\mc{O}_\lambda$, and we have that 
\[
0 \subset \text{im}e =\ker{e} \subset \C^4.
\]
The Springer fibre is
\[
F_\lambda = \{(L \subset P \subset H \subset \C^4) \mid e\C^4 \subset H, eH \subset P, eP \subset L, eL = 0\}. 
\]
Again, we have two possibilities:
\begin{itemize}
    \item (``easy component'') $\{P=\text{im}e=\ker{e}$, free choice of $L, H$ as long as $L\subset P \subset H\} \simeq \PP^1\C \times \PP^1 \C =: \Sigma_1$
    \item (``hard component'') $\{ H=e^{-1}L\}=: \Sigma_2$. There is a natural map $\Sigma_2 \rightarrow \PP^1\C$ sending $(L \subset P \subset H) \mapsto L$, and the fibre over $L'$ is $\{(L' \subset P \subset e^{-1}L'=H)\}\simeq \PP^1\C$, so $\Sigma_2$ is a $\PP^1\C$-bundle over $\PP^1\C$. 
\end{itemize}
The diagonal $\Delta \subset \Sigma_1$ embeds into $\Sigma_2$ as the zero section, so these two components are glued together along a $\PP^1\C$.

We've established that the second component $\Sigma_2$ is a $\PP^1\C$-bundle over $\PP^1\C$, but which one?
\begin{claim}
$\Sigma_2 = \PP(\mc{O}(1) \oplus \mc{O}(-1))$ is the ``second Hirzebruch surface''. 
\end{claim}
For justification of this claim, see Geordie's hand-written notes. 
We finish this example by constructing our table.
\[
\begin{tabular}{|c||c|c|c|c|c|}
\hline
    $\lambda$ & \includegraphics[scale=0.15]{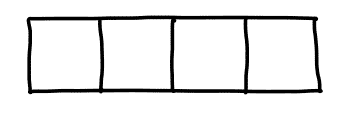} & \includegraphics[scale=0.15]{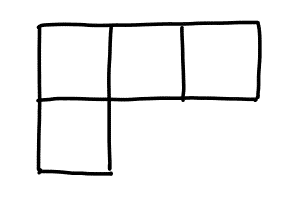} & \includegraphics[scale=0.15]{images/22.png} &
    \includegraphics[scale=0.15]{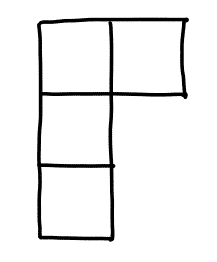} & \includegraphics[scale=0.15]{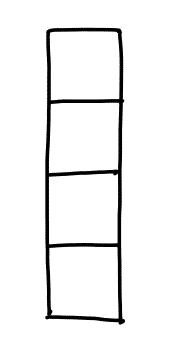}\\
     \hline \hline
    $\dim{\mc{O}_\lambda}$& 12 & 10 & 8 & 6 & 0 \\
     \hline
    $\codim{\mc{O}_\lambda}$ &0 & 2 & 4 & 6 & 12\\
     \hline
     Springer fibre & pt & \includegraphics[scale=0.15]{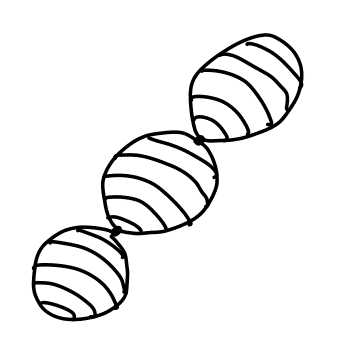}& $\Sigma_1 \cup_{\PP^1\C}\Sigma_2$ & $\mc{B}_{\SL_3} \cup \text{2 other comp.}$ & $\mc{B}_{\SL_4}$ \\
     \hline 
     $\dim{\text{fibre}}$ & 0 & 1 & 2 & 3 & 6 \\
     \hline
\end{tabular}
\]

Now that we've constructed four tables, we can make some observations about patterns we see. 

\vspace{5mm}
\noindent
{\bf Observations:}
\begin{itemize}
    \item $2 \dim{F_\lambda}=\codim{\mc{O}_\lambda}$
    \item For any $\lambda$, one (``easy'') component of $F_\lambda$ is isomorphic to a flag variety of a smaller group. 
    \item Fibres are equidimensional\footnote{This is really remarkable!} and components appear to be smooth\footnote{Sadly this fails in bigger examples.}.
\end{itemize}
\begin{remark}
    An audience member ``Mr. Wiggins'' also observed that in our examples, $\dim{T^*\mc{B}}$ is divisible by all of the $\codim{\mc{O}_\lambda}$, a property we coined ``Wiggins divisibility''. (Though we are not sure if this is a general phenomenon...)
\end{remark}
Many of our observations hold in general. Here are some fundamental properties of Springer fibers. 
\begin{theorem}Let $G$ be an a semisimple algebraic group. 
\begin{enumerate} 
    \item The Springer resolution $\pi_s:\widetilde{\mc{N}}\rightarrow \mc{N}$ is a $G$-equivariant, projective resolution of singularities, and is an isomorphism over $\mc{N}_\text{reg}=\{x \in \mc{N} \mid \dim{Z_GX}=\text{rank}\mf{g}\}$. 
    \item For a nilpotent orbit $\mc{O}_\lambda \subset \mc{N}$, the corresponding Springer fibre $F_\lambda:=\pi_s^{-1}(x)$, $x \in \mc{O}_\lambda$ is equidimensional, and $\dim{F_\lambda}=\frac{1}{2}\codim(\mc{O}_\lambda \subset \mc{N})$. 
    \item $H_\text{odd}(F_\lambda, \Z)=0$.
\end{enumerate}
\end{theorem}
\begin{remark}
\begin{enumerate}
    \item $H_\text{even}(F_\lambda, \Z)$ is well-studied (for example, its Betti numbers are known). 
    \item Part 2. of the theorem implies that $\pi_s$ is semismall and all strata are relevent. 
    \item In type $A$, $F_\lambda$ have cell decompositions, $\C^0 \sqcup (\C^1)^? \sqcup \cdots$. This cell decomposition implies Part $3.$ of the theorem for type $A$. With much more work cell decompositions have been constructed in other classical types, but existence of cell decompositions is unknown in exceptional types. 
    \item In type $A$, 
    \[
    \# \left\{\text{components of $F_\lambda$}\right\} = \dim(\text{irrep of $S_n$ indexed by $\lambda$}).
    \]
    In fact, there exists a canonical bijection 
    \[
    \left\{ \begin{array}{c} \text{standard Young} \\ \text{tableaux of shape $\lambda$} \end{array} \right\} \leftrightarrow \left\{ \begin{array}{c} \text{components} \\ \text{ of $F_\lambda$} \end{array} \right\}. 
    \]
\end{enumerate}
\end{remark}

\subsection{The conormal space}
Let 
\[
X=\bigsqcup_{\lambda \in \Lambda} X_\lambda
\]
be a stratefied variety. For example, $\C=\C^\times \sqcup \{0\}$. Consider the space
\[
\bigcup_{\lambda \in \Lambda} TX_\lambda \subset TX.
\]
This is a horrible space! In our example, $T\C=\C \times \C$, and 
\[
\bigcup TX_\lambda = \{(x, y) \in \C^2 \mid x \neq 0 \text{ or } x=0, y=0\}. 
\]
Here's a picture:
\begin{center}
\includegraphics[scale=0.4]{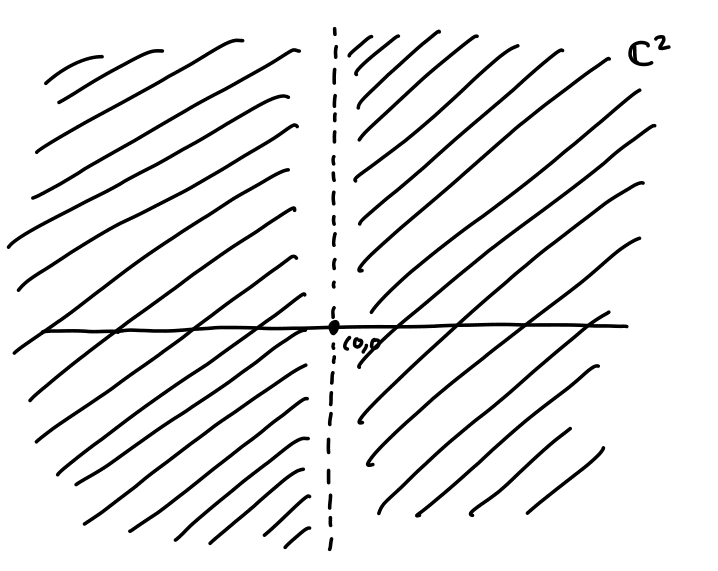}
\end{center}
In contrast to this, the {\em conormal space}
\[
T_\Lambda^*X:= \bigcup_{\lambda \in \Lambda} T_\lambda^*X \subset T^*X, 
\]
where $T_\lambda^*X=\{ \xi \in T_x^*X \text{ for } x \in X_\lambda \mid \xi \text{ vanishes on }TX_\lambda\}$ is very nice. In our example of $\C=\C^\times \sqcup \{0\}$, $
T_\Lambda^*\C = \{(x, y) \in \C^2 \mid xy=0\}$:
\begin{center}
\includegraphics[scale=0.4]{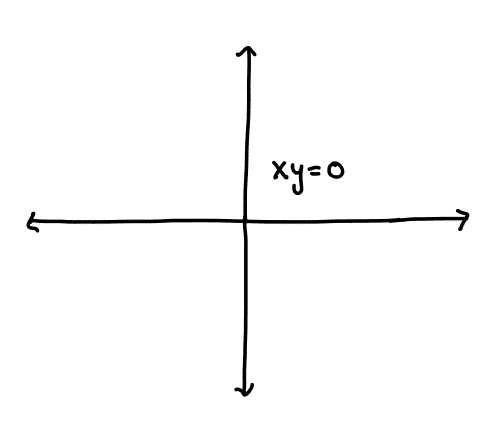}
\end{center}

\vspace{5mm}
\noindent
{\bf Properties of $T_\Lambda^*X$:}
\begin{enumerate}
    \item $T_\Lambda^*X$ is a closed subvariety of $T^*X$.
    \item $\dim{T_{\lambda}^*X}=\dim{X_\lambda} + \codim(X_\lambda \subset X ) = \dim{X}$ is independent of $\lambda$! (So we have a ``democracy of strata''.)
    \item The components of the conormal space are in bijection with the strata.
\end{enumerate}
The conormal space is a fundamental object in microlocal geometry. 

\vspace{5mm} 
\noindent
{\bf Warning:} the intersection pattern of $\overline{T_\lambda^*X}$ may be very complicated. 

\vspace{5mm}

An important object in our story arises as a conormal variety, the Steinberg variety.

\subsection{The Steinberg variety}
Let $H$ be a group that acts on a variety $X$ on the right, and a variety $Y$ on the left. We can form the {\em balanced product}
\[
X\times_HY:= X \times Y / (xh,y) \sim (x,hy). 
\]
This space may not exist as a variety in general, but in all examples we will encounter, it does. 

Choose a Borel subgroup $B \subset G$, so $\mc{B} \simeq G/B$. Then consider the variety
\begin{align*}
G\times_B G/B &\xrightarrow{\sim} G/B \times G/B \\
(g, g'B) &\mapsto (gg'B, g'B)
\end{align*}
The set of $G$-orbits on $G \times_B G/B$ is equal to the set of $B$-orbits on $G/B$, which is parameterized by $W$ by the Bruhat decomposition. So we have a stratification 
\[
\mc{B} \times \mc{B} = \bigsqcup_{x \in W} \mc{O}_x.
\]
In type $A$, this stratification is given by ``pairs of flags in relative position $x$''. 

\begin{example}
There is a stratification of $\PP^1\C \times \PP^1\C$ given by
\[
\PP^1\C \times \PP^1 \C = \Delta \sqcup (\PP^1\C \times \PP^1\C) \backslash \Delta.
\]
Pairs of flags in $\Delta$ are in relative position $id$ (i.e. they are equal), and flags in $(\PP^1\C \times \PP^1\C) \backslash \Delta$ are in relative position $s$.
\end{example}
\begin{remark}
The variety of Borel subalgebras $\mc{B}$ does not depend on any choices, so the product $\mc{B} \times \mc{B}$ does not depend on any choices. Since we can define the Weyl group as the set of $G$-orbits on $\mc{B} \times \mc{B}$, this gives us a canonical definition of the Weyl group that does not depend on a choice!
\end{remark}

\begin{definition}
The {\em Steinberg variety} 
\[
St=\{(\mf{b}, \mf{b}', x) \in \mc{B} \times \mc{B} \times \mc{N} \mid x \in \mf{b}, x \in \mf{b}'\}
\]
is given by the fibre product
\[
\begin{tikzcd}
& St\arrow[dl] \arrow[dr] & \\
\widetilde{\mc{N}} \arrow[dr] & & \widetilde{\mc{N}} \arrow[dl]\\
& \mc{N} & 
\end{tikzcd}
\]
where the maps are
\[
\begin{tikzcd}
& (\mf{b}, \mf{b}', x)\arrow[dl, mapsto] \arrow[dr, mapsto] & \\
(\mf{b},x) \arrow[dr, mapsto] & & (\mf{b}',x) \arrow[dl, mapsto]\\
& x & 
\end{tikzcd}
\]
\end{definition}

\begin{exercise}
(Solution can be found in Geordie's hand-written notes.) Show that the Steinberg variety $St$ is the conormal variety of 
\[
\mc{B} \times \mc{B} = \bigsqcup_{x \in W} \mc{O}_x.
\]
\end{exercise}
\begin{corollary}
There is a canonical bijection 
\[
\left\{ \begin{array}{c} \text{irred. comp.} \\ \text{of }St \end{array} \right\} \leftrightarrow W. 
\]
\end{corollary}
To finish this lecture, let us recall a remarkable geometric construction of the Robinson-Schensted correspondence, due to Steinberg. Geordie is given to understand that this is the origin of the name ``Steinberg variety''. Fix $\lambda$, and two components $C, C'$ of $F_\lambda$. We have the following diagram 
\[
\begin{tikzcd}
\pi_s^{-1}(\mc{O}_\lambda) \arrow[r, hookleftarrow] \arrow[d] & \mc{O}_\lambda \times C \arrow[d] \\
\mc{O}_\lambda \arrow[r, equals] & \mc{O}_\lambda
\end{tikzcd}
\]
The set $\{(x, c, c') \mid x \in \mc{O}_\lambda, c \in C, c' \in C'\}$ is a subvariety in $St$ of dimension 
\[
\dim{\mc{O}_\lambda} + \dim{C} + \dim{C'} = \dim{\mc{O}_\lambda} + \frac{1}{2} \codim(\mc{O}_\lambda \subset \mc{N}) + \frac{1}{2} \codim(\mc{O}_\lambda \subset \mc{N}) = \dim{\mc{N}} = \dim{\mc{B} \times \mc{B}}.
\]
The components of $St$ all have dimension $\dim \mc{B} \times \mc{B}$, so we have a bijection 
\[
\{\text{components of }St \} \leftrightarrow \{(\mc{O}_\lambda, C, C') \text{ as above}\}.
\]
In type $A$, this is a bijection between 
\[
W \leftrightarrow \{(\lambda, T, T') \mid T,T' \text{ are standard Young tableaux of shape }\lambda \}.
\]
This is the Robinson-Schensted correspondence! 

%% file: lecture-14.tex
\section{Lecture 14: Springer correspondence and Borel-Moore homology}
\label{lecture 14}

\subsection{The Springer correspondence}
We pick up in the setting of the last lecture, the {\bf Springer resolution}:  
\[
\begin{tikzcd}
\widetilde{\mc{N}} \arrow[d, "\pi_s"] \simeq T^*\mc{B} & F_x := f^{-1}(x) \arrow[d, mapsto] \\
\mc{N} \subset \mf{g} & x 
\end{tikzcd}
\]
Roughly, the {\bf Springer correspondence} states that $W$ acts on $H^*(F_x; \Q)$, and one obtains all irreducible representations of $W$ in top cohomology. Note that this does {\em not} come from an action of $W$ on $F_x$! This is in contrast to Deligne-Lusztig theory and other settings where we obtain representations of a group in the cohomology of varieties on which the group acts. In the first part of today's lecture, we'll work towards a precise statement of the Springer correspondence.  

\begin{remark}
In type A, the Springer correspondence explains why irreducible representations of the symmetric group and nilpotent orbits are both classified by the same combinatorial data (Young diagrams). 
\end{remark}

\vspace{3mm}
\noindent
{\bf Grothendieck-Springer alteration:}
Let 
\begin{align*}
    \widetilde{\mf{g}} = \{ (x, \mf{b}) \in \mf{g} \times \mc{B} \mid x \in \mf{b} \}.
\end{align*}
In type A, this is equal to 
\[
\{ (x, F) \in \mf{g} \times \mc{F}\ell ags \mid xF^i \subset F^i \}.
\]
\begin{remark}
Note that $\widetilde{\mf{g}} \subset \mc{B} \times \mf{g}$, and its dual is $\widetilde{\mc{N}} = T^*\mc{B}$. (This is because the dual of $x \in \mf{b}$ is $x \in \mf{b}^\perp = \mf{n}$.)
\end{remark}
The key diagram is the following:
\[
\begin{tikzcd}
\widetilde{\mc{N}} \arrow[d, "\pi_s"] \arrow[r, symbol=\subset] & \widetilde{\mf{g}} \arrow[d, "\pi_G"] \arrow[r, symbol= \supset] & \widetilde{\mf{g}}_{r.s.} \arrow[d] \\
\mc{N} \arrow[r, symbol=\subset] & \mf{g} \arrow[r, symbol=\supset] & \mf{g}_{r.s.}
\end{tikzcd}
\]
where $\mf{g}_{r.s.}=\{x \in \mf{g} \mid x \text{ regular semisimple}\}$. Note that $\mf{g} \neq \mc{N} \cup \mf{g}_{r.s.}$. 
\begin{theorem}
\begin{enumerate}
    \item The map $\pi_s$ is semismall. 
    \item The map $\pi_G$ is small (i.e., $\pi_{G*} k_{\widetilde{\mf{g}}}[\dim \mf{g}] = IC(\mf{g}_{r.s.}, \mc{L})$).
    \item Over $\mf{g}_{r.s.}$, $\pi_G$ is a $W$-torsor. (Here $W$ is the Weyl group of $G$.) 
\end{enumerate}
\end{theorem}
For example, in type A, 
\[
\mf{g}_{r.s.} = \{ x \in \mf{g} \mid x \text{ is semisimple with distinct eigenvalues}\}.
\]
To give a flag $F$ preserved by $x$ is equivalent to an ordering of the eigenvalues of $x$. This is a $S_n$-torsor. 

Now we are ready to state the Springer correspondence precisely. Given $x \in \mc{N}$, let $A_G(x)$ be the component group of the centralizer:  $A_G(x):=C_G(x)/C_G(x)^\circ$.
\begin{theorem}(Springer correspondence) 
\begin{enumerate}
    \item   $\displaystyle \pi_{s*} \Q_{\widetilde{\mc{N}}}[\dim \widetilde{\mc{N}}] = \bigoplus_{x \in \mc{N}/G, \atop \rho \in \Irr A_G(x)}H_{top}(F_x)_\rho \otimes IC(\overline{G \cdot x}, \mc{L}_\rho)$.
    \item (**Most important**) $\End(\pi_{s*}\Q_{\widetilde{\mc{N}}}[\dim \widetilde{\mc{N}}]) = \End( \pi_{G*} \Q_{\widetilde{\mf{g}}}[\dim \widetilde{\mf{g}}])=\Q W$.
    \item $\{H_{top}(F_x)_\rho\}$ are all irreducible representations of $W$.
\end{enumerate}
\end{theorem}

\begin{remark}
It is not difficult to see that in fact 1. and 3. are consequences of 2. 
\end{remark}

\begin{remark}
Part 2. of the theorem is true over $\Z$, and indeed over any ring. But part 1. fails over arbitrary rings because the decomposition theorem fails. 
\end{remark}

How might we approach part 2.? Three possible approaches:
\begin{enumerate}[label=(\alph*)]
    \item {\bf Borho - MacPherson}: 
    \begin{itemize}
        \item First note that because $\widetilde{\mf{g}} \rightarrow \mf{g}$ is small, the fact that $\End(\pi_{G*} \Q_{\widetilde{\mf{g}}}[\dim \mf{g}])=\Q W$ is obvious. Indeed, 
        \[
        \End(\pi_{G*} \Q_{\widetilde{\mf{g}}})= \End(\pi_{G*} \Q_{\widetilde{\mf{g}}} |_{reg. ss.}) = \End_{\Q W} (\Q W) = \Q W, 
        \]
        with the first equality following from smallness. 
        \item Then BM point out that we have a homomorphism 
        \[
        \End(\pi_{G*} \Q_{\widetilde{\mf{g}}}) \xrightarrow{r} \End(i^* \pi_{G*} \Q_{\widetilde{\mf{g}}}) = \End(\pi_{s*} \Q_{\widetilde{\mc{N}}})
        \]
        where $i: \mc{N} \hookrightarrow \mf{g}$; and, miraculously, $r$ is an isomorphism. (Proof sketch: The map $r$ is injective because the action of $W$ on $H^*(\mc{B})$ is faithful, then compare dimensions.) 
    \end{itemize}
    \item {\bf Fourier transform}: (Springer's original approach) Because $\widetilde{\mc{N}}$ and $\widetilde{\mf{g}}$ are dual in $\mc{B} \times \mf{g}$, $k_{\widetilde{\mc{N}}}$ and $k_{\widetilde{\mf{g}}}$ are Fourier transforms of one another. This implies that the endomorphism rings are equal:
    \[
    \End(\pi_{s*}k_{\widetilde{\mc{N}}}) = \End(\pi_{G*} k_{\widetilde{\mf{g}}})=kW. 
    \]
    \item {\bf Convolution algebras}: (in Chris-Ginzburg) 
    
    The next part of the lecture will explain this approach. For impatient readers, we'll give away the ending:
    
    \vspace{2mm}
    \noindent
    {\bf The punchline}: There is a notion of homology (Borel-Moore homology, $H_{BM}$) such that $\End(\pi_{s*} k_{\widetilde{\mc{N}}})$ is canonically equal to $H_{BM}(\text{Steinberg variety})$. This gives a concrete realisation of this endomorphism algebra. 
\end{enumerate}

\subsection{Borel-Moore homology} 
Let $X$ be an algebraic variety, $k$ a field, $p:X \rightarrow pt$, and $k_X$ the constant sheaf on $X$. Here are four notions of (co)homology:
\begin{itemize}
    \item $H^*(X; k)=H^*(p_* k_X)$ - cohomology (cochains)
    \item $H_c^*(X; k)=H^{*}(p_! k_X)$ - cohomology with compact support
    \item $H_*(X; k)=H^{-*}(p_!\omega_X)$ - homology (chains)
    \item $H_*^{BM}(X;k)=H^{-*}(p_*\omega_X)$ - Borel-Moore homology (locally finite chains)
\end{itemize}
Here $\omega_X$ is the dualising sheaf, $\mathbb{D} k_X$. The canonical example is the following. 
\begin{example}
Let $X=\{z \in \C \mid 0 < |z| < 1 \}$. Then 
\[
H_0(X; \Z) = \Z \vcenter{\hbox{\includegraphics[scale=0.3]{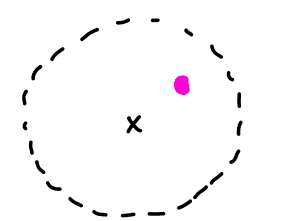}}}, \hspace{2mm} H_1(X; \Z) = \Z \vcenter{\hbox{\includegraphics[scale=0.3]{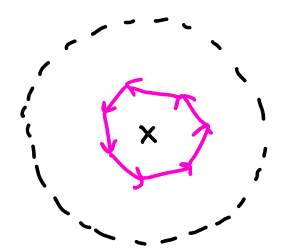}}}, \hspace{2mm} H_2(X, \Z)=0. 
\]
If we compute Borel-Moore homology, we get
\[
H^{BM}_0(X; \Z) = 0, \hspace{2mm}  H^{BM}_1(X, \Z) = \Z \vcenter{\hbox{\includegraphics[scale=0.3]{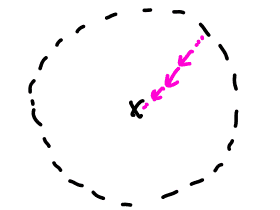}}}, \hspace{2mm} H^{BM}_2(X; \Z) = \Z \text{ ``fundamental class''}.
\]
In the first computation ($H_0^{BM}$), the formal generator of $H_0$ is now a boundary, so we get $0$. In the second computation ($H_1^{BM}$), we can now have cycles from the edge to the center, and these are not boundaries. In the third computation ($H_2^{BM}$), the ``fundamental class'' is a triangulation of $X$. 
\end{example}

\begin{center}
\boxed{\textbf{ From now on in this course, $H_* = H^{BM}_*$.}}
\end{center}

\noindent
{\bf Key properties of Borel-Moore homology}:
\begin{enumerate}
    \item If $X \xhookrightarrow{i} M$ is a closed embedding of $X$ into a smooth, $\C$-dimension $d$ variety $M$, then we have local Poincar\'{e} duality: $\omega_M \simeq k_M[2d]$. Hence
    \begin{align*}
        H_*(X; k) &= H^{-*}(X, \omega_X) \\
        &= H^{-*}(X, i^! \omega_M) \\
        &= H^{-*}(M, i_!i^!k_M[2d])\\
        &=H^{2d-*}(M, i_!i^!k_M) \\
        &=H^{2d-*}(M, M\backslash X; k).
    \end{align*}
    \item $H_*(-)$ is {\em not} functorial for arbitrary maps, but for proper maps $p$, we have $p_*$.
    \item For an open inclusion $U \hookrightarrow X$, we have restriction
    \[
    H_*(X) \rightarrow H_*(U). 
    \]
    ({\bf Thought exercise}: Why? In terms of chains?)
    \item If $X$ is equidimensional of dimension $d$ with components $X_1, X_2, \ldots, X_m$, then 
    \[
    H_{2d}(X) = \bigoplus \Z[X_i],
    \]
    where $[X_i]$ are the fundamental classes of the components.  
\end{enumerate}

\subsection{Geometric convolution algebras}
Let $X_1, X_2, X_3$ be smooth varieties of dimensions $d_1, d_2, d_3$, respectively, and 
\[
Z_{12} \subset X_1 \times X_2, \hspace{3mm} Z_{23} \subset X_2 \times X_3
\]
closed subvarieties (``correspondences''). (For example, we could take $Z_{12}=\graph(f)$ for some $f:X_1 \rightarrow X_2$.) Define 
\[
Z_{12} \circ Z_{23} := \{(x_1, x_3) \in X_1 \times X_3 \mid \text{ there exists }x_2 \in X_2 \text{ s.t. } (x_1, x_2) \in Z_{12}, (x_2, x_3) \in Z_{23} \}.
\]
(So in our example, $\graph(f) \circ \graph(g) = \graph(g \circ f)$.) We have projections 
\[
\begin{tikzcd}
 & X_1 \times X_2 \times X_3 \arrow[ld, "p_{12}"'] \arrow[d, "p_{13}"] \arrow[rd, "p_{23}"] & \\
X_1 \times X_2 & X_1 \times X_3 & X_2 \times X_3 
\end{tikzcd}
\]
We make the following {\bf properness assumption}: From now on, assume that the map
\[
p_{13}: p_{12}^{-1}(Z_{12}) \cap p_{23}^{-1}(Z_{23}) \rightarrow X_1 \times X_3
\]
is proper. 

\begin{definition}
We define a {\bf convolution product} on homology:
\begin{align*}
    H_i(Z_{12}) \times H_j(Z_{23}) &\rightarrow H_{i+j-2d_2}(Z_{12} \circ Z_{23}) \\
    (c_{12}, c_{23}) &\mapsto p_{13*}(p_{12}^*c_{12} \cap p_{23}^*c_{23}) 
\end{align*}
Here $\cap$ is the intersection product in Borel-Moore homology. 
\end{definition}

\noindent
{\bf Most important cases:} Let $\widetilde{X} \xrightarrow{f} X$ be a proper map of a smooth variety $\widetilde{X}$ to $X$, and let $X_i=\widetilde{X}$ for $i=1, 2, 3$.  Here are two cases of the construction above: 
\begin{enumerate}
    \item Let $Z_{12}=Z_{23}=Z_{12} \circ Z_{23}=\widetilde{X} \times_X \widetilde{X} \subset \widetilde{X} \times \widetilde{X}$. Then the convolution product gives an associative algebra structure on $H_*(\widetilde{X} \times_X \widetilde{X})$ (with messy gradings). 
    \item Let $Z_{12}=\widetilde{X} \times_X \widetilde{X}$, $Z_{23}=f^{-1}(X)$, $Z_{12} \circ Z_{23} = Z_{23}$. Then the convolution product gives us a map 
    \[
    H_*(\widetilde{X} \times_X \widetilde{X}) \times H_*(f^{-1}(X)) \rightarrow H_*(f^{-1}(X)).
    \]
    This gives $H_*(f^{-1}(X))$ the structure of a $H_*(\widetilde{X} \times \widetilde{X})$-module. 
\end{enumerate}
So whenever we have a smooth proper map into our variety, we obtain from this formalism an associative algebra and a collection of modules (one for each fibre) over that algebra. Seems promising! 

\vspace{5mm}
\noindent
{\bf Conceptual explanation of convolution algebras:}
Let $f$ be as above, with the fibre product diagram 
\[
\begin{tikzcd}
\widetilde{X} \times_X \widetilde{X} \arrow[d, "g"] \arrow[r, "g"] & \widetilde{X} \arrow[d, "f"]\\ 
\widetilde{X} \arrow[r, "f"] & X
\end{tikzcd}
\]
\noindent
{\bf Claim:} $\End^\cdot(f_*k_{\widetilde{X}}) = H_*(\widetilde{X} \times_X \widetilde{X})$
\begin{proof}
We will only check the statement on the level of vector spaces. Let $d_X:= \dim_\C{X}$. We have 
\begin{align*}
    \Hom^\cdot (f_*k_{\widetilde{X}}, f_*k_{\widetilde{X}}) &= \Hom^\cdot (f^*f_*k_{\widetilde{X}}, k_{\widetilde{X}}) \\
    &=\Hom^\cdot (f^*f_! k_{\widetilde{X}}, k_{\widetilde{X}}) \\
    &= \Hom^\cdot (g_! g^* k_{\widetilde{X}}, k_{\widetilde{X}}) \\
    &= \Hom^\cdot (k_{\widetilde{X} \times_X \widetilde{X}}, g^!k_{\widetilde{X}}) \\
    &= \Hom^\cdot(k_{\widetilde{X} \times_X \widetilde{X}}, g^! \omega_{\widetilde{X}} [-2d_X]) \\
    &=H^{*-2d_X}(\widetilde{X}\times_X \widetilde{X}, \omega_{\widetilde{X} \times \widetilde{X}}) \\ 
    &= H_{2d_X - *}(\widetilde{X} \times_X \widetilde{X}). 
\end{align*}
Here we are using properness of $f$ (second equality), adjunctions (first equality, fourth equality), proper base change (third equality), and local Poincar\'{e} duality (fifth equality). 
\end{proof}

\vspace{5mm}
\noindent
{\bf The Upshot:} Up to gradings, 
\[
H_*(\widetilde{X} \times_X \widetilde{X}) = \End^\cdot (f_* k_{\widetilde{X}} ), 
\]
and with some work we can show that multiplication matches on both sides. Similarly, 
\[
(f_*k_{\widetilde{X}})_x = H_*(f^{-1}(X)),
\]
and the module structure comes from the action $\End^\cdot(f_*k_{\widetilde{X}}) \circlearrowright (f_*k_{\widetilde{X}})_x$. 

\vspace{5mm}
\noindent
{\bf Connection to the Spring correspondence:}

\vspace{2mm}
\noindent
Let $\widetilde{\mc{N}} \xrightarrow{\pi_s} \mc{N}$ be the Springer resolution, and 
\[
St:=\widetilde{\mc{N}} \times_{\mc{N}} \widetilde{\mc{N}} = \begin{array}{c} \text{conormal space to} \\ \text{the $G$-space $\mc{B} \times \mc{B}$} \end{array} = \bigcup_{x \in W} T_x^*,
\]
the Steinberg variety. (See lecture 13.) Here $T_x^*$ is the conormal bundle to the $G$-orbit $\mc{O}_x \subset \mc{B} \times \mc{B}$. Recall that $St$ is equidimensional and all components have dimension equal to $N:=\dim{\mc{B} \times \mc{B}}=\dim T^*\mc{B}$. 

The convolution product in Borel-Moore homology gives us a map 
\[
H_{2N}(St) \times H_{2N}(St) \xrightarrow{*} H_{4N-2\dim{T^*\mc{B}}}(St) = H_{2N}(St).
\]
This gives $H_{2N}(St)$ the structure of an algebra! 
\begin{theorem}
As algebras, 
\[
H_{2N}(St) = kW.
\]
\end{theorem}
By the theory earlier, the action of $kW$ on $H_i(F_x)$ yields the Springer action. In this way we end up with $W$-modules everywhere! 

%% file: lecture-15.tex
\section{Lecture 15: The Kazhdan--Lusztig isomorphism} 
\label{lecture 15}

\subsection{More convolution algebras}
Last week we ended with a discussion on convolution algebras. We'll start today by continuing this discussion. Consider the following two settings:
\begin{enumerate}
    \item Let $G$ be a group and $H \subset G$ a subgroup. With the operation $*$ of convolution of functions, the vector space $\Fun(G, \C)$ of complex-valued functions on $G$ has the structure of an algebra. This is just the group algebra, $\Fun(G, \C) \simeq \C[G]$. The subspace $\Fun_{H \times H}(G, \C)$ of $H$-biinvariant functions forms a subalgebra. This is a {\em Hecke algebra}. As many of us are aware, its representation theory is very complicated in general! 
    \item Now let $X$ be a finite set. The vector space $\Fun(X \times X, \C)$ of complex-valued functions on $X \times X$ can also be given the structure of an algebra. In this setting, the convolution product is given as follows. Let 
    \[
\begin{tikzcd}
 & X \times X \times X \arrow[ld, "p_{12}"'] \arrow[d, "p_{13}"] \arrow[rd, "p_{23}"] & \\
X \times X & X \times X & X \times X 
\end{tikzcd}
\]
be the natural projections. Then for $f, g \in \Fun(X \times X, \C)$, define $f * g \in \Fun(X \times X, \C)$ by
\[
(f * g)(x, z) := \sum_{y \in X} f(x, y) g(y, z)
\]
for $(x, z) \in X \times X$. In other words, 
\[
f*g = p_{13!}(p_{12}^*f \boxtimes p_{23}^*g). 
\]
With $*$, $(\Fun(X \times X, \C), *)$ is an algebra. In fact, it's a familiar\footnote{This is why we don't meet this algebra as often as we meet the Hecke algebra - it's ``too easy,'' in the sense that it is just a matrix ring. However, when we categorify, it becomes more interesting, so we are more likely to encounter it (or regognize it!) in categorified settings.} algebra. Let $e_{x, y}$ be the indicator function on $(x, y) \in X \times X$. Then $e_{x, y} * e_{y', z} = \delta_{y, y'} e_{x, z}$. Hence,
\[
(\Fun(X \times X, \C), *) \simeq \Mat_{X \times X}(\C). 
\]
{\bf Some Variants:} 
\begin{enumerate}
    \item Let $X\rightarrow Y$ be a map of sets, then we can construct the convolution algebra
    \[
    \Fun(X \times_Y X, \C) \simeq \prod_{y \in Y} \Mat_{f^{-1}(y) \times f^{-1}(y)} (\C). 
    \]
    We can formulate Mashke's theorem in this language. For a finite group $G$, let $Y=\{\text{irreps of $G$}\}$, $X = \bigsqcup_{\rho \in Y} \{\text{basis of }\rho\}$, and $X \rightarrow Y$ the map which assigns to a basis element the corresponding representation. Then Mashke's Theorem that the group algebra is semisimple is the following statement:  \begin{theorem}
    (Mashke's Theorem)
    \[
    (\Fun(G, \C), * ) \simeq (\Fun(X \times_Y X, \C), *). 
    \] 
    \end{theorem}
    This realizes a complicated algebra (the group algebra) in terms of much simpler pieces (matrix rings). 
    \item If our sets $X, Y$ come with the additional structure of an action by a group $\Gamma$, and $X \rightarrow Y$ is a $\Gamma$-equivariant map of sets, then we can construct the convolution algebra 
    \[
    \Fun^\Gamma(X \times_Y X, \C).
    \]
\end{enumerate}

\noindent
{\bf The Upshot:} There are two types of convolution algebras, one is hard (Hecke algebras), and the other is easy ($\Fun(X \times X, \C))$. The hard one is of great significance in representation theory. A strategy that we use in representation theory is to try to realise the hard type as the easy type. 
\end{enumerate}

\subsection{Equivariant $K$-theory} 

Let $G$ be an algebraic group acting on a variety $X$. Then we can formulate the notion of an equivariant coherent sheaf, 
\[
\mc{F} \in \Coh_G(X), 
\]
as a coherent sheaf $\mc{F}$ on $X$, coupled with some extra data\footnote{For an excellent description of this construction, see notes from Emily's Sept 20, 2019 talk in the Informal Friday Seminar. Notes from Emily's (and all other) IFS talks can be found at \url{https://sites.google.com/view/ifssydney/home}.} Roughly speaking, one can think of an equivariant coherent sheaf on $X$ as being an equivariant sheaf together with an algebraic action on its sections.

\begin{remark}
If $\mc{F}$ is locally free, then $\mc{F}$ corresponds to a vector bundle $V \rightarrow X$ on $X$. In this setting,  $G$-equivariance of $\mc{F}$ corresponds to an algebraic $G$-action on $V$ which is compatible with projection. In particular,
\[
\Coh_G(pt) \simeq \Rep{G}, 
\]
where $\Rep{G}$ is the category of algebraic representations of $G$.
\end{remark}
Define two types of equivariant $K$-groups:
\begin{align*}
    K^G(X) &:= \text{ Grothendieck group of $G$-equivariant coherent sheaves on $X$}; \\
    K_G(X) &:= \text{ Grothendieck group of $G$-equivariant vector bundles on $X$}.
\end{align*}
\begin{remark}
\begin{enumerate}[label=(\alph*)]
\label{careful about functors}
\item Very loosely, 
\begin{align*}
    K^G(X) & \leftrightarrow \text{ ``Borel--Moore homology,''} \\
    K_G(X) & \leftrightarrow \text{ ``cohomology''.}
\end{align*}
\item There are higher $K$-groups, $K_i^G(X), K^i_G(X)$, which we will ignore here. (Already $K_0^G(X)=K^G(X)$ and $K^0_G(X)=K_G(X)$ are rich enough.) \item There is a natural map
\[
K_G(X) \rightarrow K^G(X).
\]
If $X$ is smooth, then we can use resolutions by coherent sheaves to show that this map is an isomorphism.
\item We have to be careful with functors between $K$-groups. For example, if $X \rightarrow pt$ is projection and $X$ is affine, then $f_*\mc{O}_X = \Gamma(X, \mc{O}_X)$ is usually infinite dimensional. So in general, we need to work to justify that functors we wish to use preserve Coh, or at least $D^b(\text{Coh})$. 
\end{enumerate}
\end{remark}

As we did for $H_*$, we have a convolution product in $K$-theory. We imitate the set-up of the previous lecture: Let $X_1, X_2, X_3$ be smooth varieties of dimensions $d_1, d_2, d_3$, respectively, and 
\[
Z_{12} \subset X_1 \times X_2, \hspace{3mm} Z_{23} \subset X_2 \times X_3
\]
closed subvarieties, with projections $p_{ij}$, $i,j = 1, 2, 3$ as before. Given $\mc{F} \in \Coh(Z_{12}), \mc{G} \in \Coh(Z_{23})$, define 
\[
\mc{F} * \mc{G} := p_{13*}(p_{12}^*\mc{F} \overset{L}{\otimes} p_{23}^* \mc{G}) \in D^b(\Coh(Z_{12} \circ Z_{23})).
\]
As we emphasized in Remark \ref{careful about functors}, we need to justify that this product is well-defined. But sure enough, the push-forward $p_{13*}$ preserves coherence because $p_{13}$ is proper, and the pull-backs $p_{12}^*, p_{23}^*$ are okay because $p_{12}$ and $p_{23}$ are flat. Hence $*$ induces a product 
\[
K_0(Z_{12}) \times K_0(Z_{23}) \rightarrow K_0(Z_{12} \circ Z_{23})
\]
which descends to a product 
\[
K^G(Z_{12}) \times K^G(Z_{23}) \rightarrow K^G(Z_{12} \circ Z_{23}). 
\]
\begin{example}
Let $X$ be a finite set, and $X\rightarrow pt$ projection to a point. In this case, elements of $\Coh(X \times X)$ are ``matrices of vector spaces,'' and 
\[
*: \Coh(X \times X) \times \Coh(X \times X) \rightarrow \Coh(X \times X)
\]
is ``multiplication of matrices''. That is, for vector spaces $V_{ij}$, 
\[
\bp V_{11} & V_{12} \\ V_{21} & V_{22} \ep * \bp V_{11}' & V_{12}' \\ V_{21}' & V_{22}' \ep = \bp V_{11}\otimes V_{11}' \oplus V_{12} \otimes V_{21}' & V_{11} \otimes V_{12}' \oplus V_{12} \otimes V_{22}' \\ 
V_{21} \otimes V_{11}' \oplus V_{22} \otimes V_{21}' & V_{21}\otimes V_{12}' \oplus V_{22} \otimes V_{22}' \ep. 
\]
The $K$-group is $K_0(X \times X) = \Fun(X \times X, \Z)$. 
\end{example}
\begin{remark}
Lusztig noticed that certain categories of the form $\Coh^\Gamma(X \times X)$, where $\Gamma$ is a group acting on a finite set $X$, are central to the classification of unipotent character sheaves. 
\end{remark}

\subsection{The Kazhdan--Lusztig isomorphism}
With this we have enough machinery to state the Kazhdan--Lusztig isomorphism. Let $(X \supset R, X^\vee \supset R^\vee)$ be a root datum, $W_{ext} = W_f \ltimes \Z X^\vee$ the extended affine Weyl group, and $\mc{H}_{ext}$ the affine Hecke algebra. (See lecture \ref{lecture 11} for a refresher on these objects.) Let $G^\vee$ be the corresponding algebraic group, $\mc{N}^\vee \subset \mf{g}^\vee$ the nilpotent cone, and $St^\vee:=T^*\mc{B}^\vee \times_{\mc{N}^\vee} T^* \mc{B}^\vee$ the Steinberg variety. (See Lecture \ref{lecture 13}.)  
\begin{theorem}
\label{KL iso}
(Kazhdan--Lusztig) We have canonical isomorphisms.
\[
\begin{tikzcd}
\mc{H}_{ext} \arrow[r, "\sim"] \arrow[d, "v=1"']& K^{G^\vee \times \C^\times}(St^\vee) \arrow[d, "\text{forget} \atop \C^\times\text{-action}"]\\
\Z W_{ext} \arrow[r, "\sim"] & K^{G^\vee}(St^\vee) 
\end{tikzcd}
\]
Here $\C^\times$ acts on $T^*\mc{B}^\vee$ by dilation along the fibres.
\end{theorem}

Here are two examples of how Theorem \ref{KL iso} gives us insight into the representation theory of affine Hecke algebras. 
\begin{enumerate}
    \item (Bernstein) As we have discussed, there is an inclusion of algebras: 
    \begin{align*}
    \Z[v^{\pm 1}][X^\vee] &\subset \mc{H}_{ext}\\
    \lambda &\mapsto H_\lambda
    \end{align*}
    Bernstein noticed that the center of $\mc{H}_{ext}$ can be realized as $W_f$-invariants of this subalgebra: 
    \[
    \Z[v^{\pm 1}][X^\vee]^{W_f} = Z(\mc{H}_{ext}).
    \]

    \hspace{5mm} We can see this in terms of the Kazhdan--Lusztig isomorphism. Recall that for a map $X \rightarrow Y$ between finite sets, we have 
    \[
    \begin{tikzcd}
    X \arrow[r, hookrightarrow, "\text{diagonal}"]  & X \times X \arrow[r] \arrow[d] & X \arrow[d] \\
     & X \arrow[r] & Y 
    \end{tikzcd}
    \]
    and $\Fun(\text{diag}, \C) \subset \Fun(X \times_Y X, \C)$ are ``diagonal matrices''. Applying this to the diagonal $T^*\mc{B}^\vee \subset St^\vee$, we can think of 
    \[
    K^{G \times \C^\times}(T^*\mc{B}^\vee) \subset K^{G \times \C^\times}(St^\vee)
    \]
    as ``diagonal matrices''. Now by homotopy,
    \[
    K^{G^\vee \times \C^\times}(T^*\mc{B}^\vee) = K^{G^\vee \times \C^\times}(\mc{B}^\vee), 
    \]
    and since $K^{B^\vee}(pt)=[\Rep{B^\vee}]=\Z[X^\vee]$ and $K^{\C^\times}(pt)=[\Rep{\C^\times}]=\Z[v^{\pm 1}]$, 
    \[
    K^{G^\vee \times \C^\times}(\mc{B}^\vee) = K^{G^\vee \times \C^\times}(G^\vee/B^\vee) = K^{B^\vee \times \C^\times}(pt) = \Z[v^{\pm 1}][X^\vee]. 
    \]
    So the subalgebra $\Z[v^{\pm 1}][X^\vee]$ sits inside $\mc{H}_{ext}$ as ``diagonal matrices'' in the $G^\vee \times \C^\times$-equivariant $K$-group of $St^\vee$. Moreover, $K^{G^\vee \times \C^\times}(pt)$ is clearly central in $K^{G^\vee \times \C^\times}(St^\vee)$, and this gives the Bernstein center:
    \[
    K^{G^\vee \times \C^\times}(pt) = (K^{B^\vee \times \C^\times}(pt))^{W_f} = (\Z[v^{\pm 1}][X^\vee])^{W_f}. 
    \]
    \item The Kazhdan--Lusztig basis of $\mc{H}_{ext}$ leads to the notion of (left, right, two-sided) cells\footnote{Pictures of 2-sided cells in rank $2$ can be found on Lusztig's webpage: \url{http://www-math.mit.edu/~gyuri/picture.html}.}. Lusztig noticed that the sets 
    \[
    \left\{ \begin{array}{c} \text{poset of }\\ \text{2-sided} \\ \text{cells} \end{array} \right\} \leftrightarrow \left\{ \begin{array}{c} \text{nilpotent} \\ \text{orbits} \\ \text{in $\mc{N}^\vee$} \end{array} \right\} 
    \]
    appear to match. 
    \[
    \includegraphics[scale=0.4]{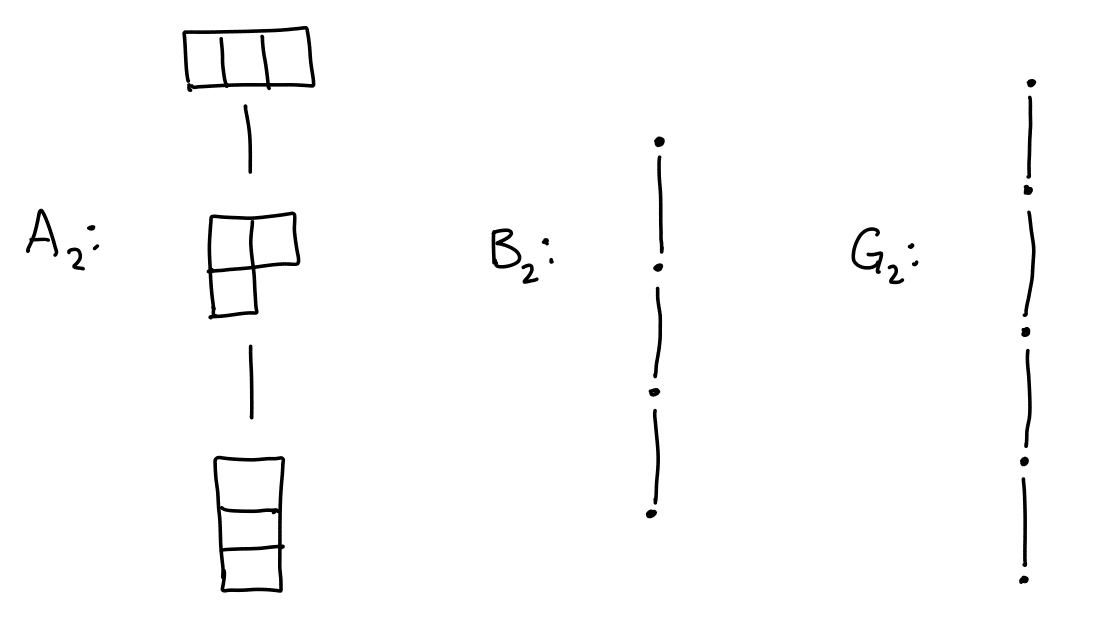}
    \]
    He also noticed other remarkable parallels, such as 
    \[
    \begin{array}{c} \text{2-sided cell} \\ \text{is finite} \end{array} \iff \begin{array}{c} \text{reductive part of} \\
    \text{the centralizer} \\ \text{is finite} \end{array}.
    \]
    These parallels convince one rather quickly that this correspondence is deep. 
    
    \hspace{5mm} We can use Theorem \ref{KL iso} to understand this observation of Lustig. Consider the projections 
    \[
    \begin{tikzcd}
    & St^\vee \arrow[dr] \arrow[dd,"p"] \arrow[dl] & \\
    T^*\mc{B}^\vee \arrow[dr] & & T^*\mc{B}^\vee \arrow[dl]\\
    & \mc{N}^\vee & 
    \end{tikzcd}
    \]
    Convolution makes it clear that any closed $G^\vee$-invariant subvariety $Z \subset \mc{N}^\vee$ gives rise to a two-sided ideal in $K^{G^\vee \times \C^\times}(St^\vee)$:
    \[
    \{[M] \mid \supp{M} \subset p^{-1}(Z)\}. 
    \]
    Hence we get a filtration of $\mc{H}_{ext}$ by nilpotent orbits in $\mc{N}^\vee$! Lusztig and Bezrukavnikov showed that the filtration of $\mc{H}_{ext}$ coming from $2$-sided cells agrees with this filtration coming from geometry, explaining Lusztig's observation that the $2$-sided cell order matches the order on nilpotent orbits. 
    \item We won't explain this in detail due to time constraints, but one can show relatively easily that Theorem \ref{KL iso} implies the Deligne--Langlands conjecture (Section \ref{Deligne-Langlands conjecture}) on representations of $\mc{H}_{ext}$, as long as $q$ is not a root of unity. 
\end{enumerate}

\subsection{Bezrukavnikov's equivalence: rough outline}

Now we are finally in the position to approach the second half of the title of this course. Here is a rough outline of the flow of ideas that lead to Bezrukavnikov's equivalence: 
\begin{itemize}
    \item {\bf Langlands correspondence}: A correspondence between sets satisfying a whole host of properties, and with many extraordinary consequences in number theory and beyond. 
    \item {\bf Weil}: In the function field case, an automorphic form\footnote{We haven't gone into what an automorphic form is yet in this course. For a global field $K$ it is a function of a special form on a certain quotient of the adelic points of a reductive group. In the function field case, Weil realised that this quotient parametrises $G$-bundles on the corresponding curve. In other words, in the function field case an automorphic form can be regarded as a function of the form given above.} can be regarded as a (very special) function 
    \[
    f: \Bun_G(\mathbb{F}_q):=\left\{ \begin{array}{c} \text{iso classes of} \\ \text{$G$-bundles on} \\ \text{smooth curves}/\mathbb{F}_q \end{array} \right\} \rightarrow \C.
    \]
    \item {\bf Grothendieck}: Functions on a variety should be understood as shadows of sheaves (function-sheaf correspondence). Here are some examples. 
    \begin{itemize}
        \item {\bf E.g. 1}: Characters of $\GL_n(\mathbb{F}_q)$ are shadows (trace of Frob) of certain $\ell$-adic sheaves on $G$, ``character sheaves''.
        \item {\bf E.g. 2}: Interesting analytic functions should satisfy many differential equations (i.e. should be solutions of a holonomic $\mc{D}$-module).
    \end{itemize}
    \item {\bf Drinfeld}: The Langlands correspondence should be approached using Grothendieck's dictionary. At its most basic level, this philosophy asserts that automorphic forms in the function field case should arise as traces of Frobenius of certain sheaves on the moduli space of $G$-bundles. At a more sophisticated level, we should expect an equivalence of categories 
    \[
    D^b(\Sh(\Bun_G)) \xleftrightarrow{???} D^b(\Coh(\Loc_G)),
    \]
    i.e. a ``geometric Langlands correspondence''. This allows us to work over $\C$, and provides fertile connections to physics, higher category theory, etc. 
    \begin{remark}
    Geometric Langlands is for {\em function fields}. The classical Langlands correspondence isn't. So they seem to be living in different worlds, but the hope is that they are actually related. For many years this appeared to many as a pretty wild idea. However, recently the geometric Langlands program has been shown to have consequences that number theorists really care about. For example Fargues showed \cite{Fargues} that geometric Langlands for the Fargues-Fontaine curve implies (part of) the LLC for {\em any} $p$-adic group!
    \end{remark}
  \item {\bf Ginzburg}: A small piece of geometric Langlands should be controlled by a categorification of the Kazhdan--Lusztig isomorphism
  \[
  \mc{H}_{ext} \simeq K^{G^\vee \times \C^\times}(St^\vee). 
  \]
  What is sought is a fundamental monoidal category that arises in two Langlands dual ways. Recall that the geometric Satake equivalence categorifies $\mc{H}^{sph} = K^G(pt)$ (Theorem \ref{satake iso}). Thus this equivalence can be seen as one layer of difficulty beyond the geometric Satake equivalence.
  \item {\bf Bezrukavnikov}: realization of (several) such equivalences:
  \begin{align*}
      K^{G^\vee \times \C^\times}(St^\vee) &\rightsquigarrow \Coh^{G^\vee \times \C^\times}(St^\vee) \text{ ``coherent side''} \\
      \mc{H}_{ext} &\rightsquigarrow D^b_{Iw}(G((t))/Iw) \text{ ``constructible side''}
  \end{align*}
 \begin{theorem}
 (Bezrukavnikov's equivalence, most basic version) Let $Iw \subset G((t))$ be an Iwahori subgroup, and $Iw_0 \subset Iw$ the pro-unipotent radical. There is an equivalence of monoidal categories 
 \[
 \left( \stackon[-8pt]{$D^b_{Iw_0}(G((t))/Iw)$}{\vstretch{1.5}{\hstretch{3.4}{\widehat{\phantom{\;\;\;\;\;\;\;\;}}}}}, * \right) \simeq \left( D^b\Coh^{G^\vee \times \C^\times}(\widetilde{St}^\vee), * \right).
 \]
 where $\widetilde{St}^\vee = \widetilde{\mf{g}}^\vee \times_{\mf{g}^\vee} \widetilde{\mf{g}}^\vee$ and the hat indicates the pro-unipotent completion.
 
 \begin{remark}
 There are several aspects of the above that need explanation, hopefully this will occur over the coming weeks and months!
 \end{remark}
 \end{theorem}
\end{itemize}

%% file: lecture-16.tex
\section{Lecture 16: The constructible side of Bezrukavnikov's equivalence}
\label{lecture 16}

Last lecture we ended by stating (modulo several undefined pieces) {\bf Bezrukavnikov's equivalence}:
 \begin{align*}
  \text{``constructible side'' } \hspace{2mm} &\hspace{12mm} \text{``coherent side''} \\
 \left( \stackon[-8pt]{$D^b_{Iw_0}(G((t))/Iw)$}{\vstretch{1.5}{\hstretch{3.4}{\widehat{\phantom{\;\;\;\;\;\;\;\;}}}}}, * \right) &\simeq \left( D^b\Coh^{G^\vee \times \C^\times}(\widetilde{St}^\vee), * \right). 
 \end{align*}

Our goal for today is to motivate the constructible side of this equivalence. To do so, we will explain Grothendieck's function-sheaf correspondence in slightly more detail than is usually done. The starting place is the {\bf Weil conjectures}.

\subsection{Weil conjectures}
The Weil conjectures were a collection of statements (made by Weil while in jail during the second world war) about counting the number of points on an algebraic variety over a finite field. Let $X$ be a smooth projective algebraic variety over $\mathbb{F}_q$. Then we can ask about the number of points $\#X(\mathbb{F}_{q^n})$ for all $n$. Weil conjectured that a generating series built out of $\# X(\mathbb{F}_{q^n})$ has remarkable properties. He also pointed out that his conjecture would follow from an interpretation of $\#X(\mathbb{F}_{q^n})$ as the trace of an operator on a vector space, by taking traces of its powers. The Grothendieck--Lefschetz trace formula, which we briefly visited in the context of elliptic curves in Lecture \ref{sato-tate}, provides such an interpretation.

\vspace{5mm}
\noindent
{\bf Grothendieck--Lefschetz trace formula}: Let $\ell$ be a prime not dividing $q$, and $X$ as above (but not necessarily projective). Then 
\begin{equation}
    \label{Grothendieck-Lefschetz}
\#X(\mathbb{F}_{q^n}) =\sTr(\text{Frob}_{q}^n \circlearrowright H_c^i(X_{\overline{\mathbb{F}}_q}, \overline{\mathbb{Q}}_\ell) ):= \sum_i (-1)^i \Tr\left(\text{Frob}_q^n \circlearrowright H_c^i(X_{\overline{\mathbb{F}}_q}, \overline{\mathbb{Q}}_\ell)\right). 
\end{equation}
Here $H^i_c(X_{\overline{\mathbb{F}}_q}, \overline{\mathbb{Q}}_\ell)$ denotes the compactly supported \'etale cohomology of $X_{\overline{\mathbb{F}}_q}$, the base change of $X$ to $\Spec \overline{\mathbb{F}}_q$. If $X/k$, then $H_{c}^*(X_{\overline{k}}, \overline{\mathbb{Q}}_\ell)$ has a continuous action of $\Gal(\overline{k}/k)$. The symbol $\sTr$ in (\ref{Grothendieck-Lefschetz}) stands for ``supertrace''. Let's see what this formula means in examples. 

\begin{example}
Let $X=\mathbb{P}_{\mathbb{F}_q}^1$, then the cohomology and action of Frobenius are given by the following table:
\begin{center}
 \begin{tabular}{|c|c|c|} 
 \hline
 $i$ & $H^i_c(X_{\overline{\mathbb{F}}_q},\overline{\mathbb{Q}}_\ell)$ & {Frob action}\\ [0.5ex]
 \hline\hline
 2 & $\overline{\mathbb{Q}}_\ell$ & $\circlearrowleft q$\\ 
 \hline
 1 & 0 & \\
 \hline
 0 & $\overline{\mathbb{Q}}_\ell$ & $\circlearrowleft 1$\\
 \hline
\end{tabular}
\end{center}
Hence, 
\[
\sTr\left(\text{Frob}_q^n\circlearrowright H^* \right) = 1+q^n = \#X(\mathbb{F}_{q^n}).
\]
\end{example}
\begin{example}
Let $X=\mathbb{G}_m$. Then the action of Frobenius on cohomology is:
\begin{center}
 \begin{tabular}{|c|c|c|} 
 \hline
 $i$ & $H^i_c(X_{\overline{\mathbb{F}}_q},\overline{\mathbb{Q}}_\ell)$ & {Frob action}\\ [0.5ex]
 \hline\hline
 2 & $\overline{\mathbb{Q}}_\ell$ & $\circlearrowleft q$\\ 
 \hline
  1 & $\overline{\mathbb{Q}}_\ell$ & $\circlearrowleft 1$\\
 \hline
 0 & 0 & \\
 \hline
\end{tabular}
\end{center}
Hence,
\[
\sTr\left(\mathrm{Frob}_q^n\circlearrowright H^* \right) = q^n-1 = \#\mathbb{G}_m(\mathbb{F}_{q^n})=\#\mathbb{F}_{q^n}^\times.
\]
\end{example}
\begin{example}
Let $X=\Spec{\mathbb{F}_{q^2}}/\Spec{\mathbb{F}_q}$. Then 
\begin{align*}
X(\mathbb{F}_{q^n}) &= \Hom(\Spec{\mathbb{F}_{q^n}}, \Spec{\mathbb{F}_{q^2}}) \\
&=\Hom_{\mathbb{F}_q\text{-alg}}(\mathbb{F}_{q^2}, \mathbb{F}_{q^n}) \\
&=\begin{cases}
0 & \text{$n$ odd}, \\ 
2 & \text{$n$ even}.
\end{cases}
\end{align*}
On the geometric side, 
\begin{align*}
    X_{\overline{\mathbb{F}}_q} &= \Spec{\overline{\mathbb{F}}_q} \times_{\Spec{\mathbb{F}_q}} \Spec{\mathbb{F}_{q^2}} \\
    &=\Spec(\overline{\mathbb{F}}_q \otimes_{\mathbb{F}_q} \mathbb{F}_{q^2}) \\
    &= \Spec{\overline{\mathbb{F}}_q}\sqcup \Spec{\overline{\mathbb{F}}_q} \\
    &\simeq \overline{\mathbb{F}}_q \times \overline{\mathbb{F}}_q.
\end{align*}
\begin{remark}
For analogy, an easier version of the computation above is the following. We can realize 
\[
\C \simeq \R[x] / (x^2 +1).
\]
Then 
\begin{align*}
    \C \otimes_{\R}\C &= \C[x]/(x^2+1) \\
    &= \C[x]/(x-i)(x+i) \\
    &\simeq \C \times \C. 
\end{align*}
\end{remark}

\noindent
{\bf The Upshot}: The variety $X_{\overline{\mathbb{F}}_q}$ consists of $2$ points which are interchanged by $\Gal(\overline{\mathbb{F}}_q / \mathbb{F}_q)$. So our table of cohomology is 
\begin{center}
 \begin{tabular}{|c|c|c|} 
 \hline
 $i$ & $H^i_c(X_{\overline{\mathbb{F}}_q},\overline{\mathbb{Q}}_\ell)$ & {Frob action}\\ [0.5ex]
 \hline\hline
 $>0$ & 0 & \\ 
 \hline
  0 & $\overline{\mathbb{Q}}_\ell\oplus \overline{\mathbb{Q}}_\ell$ & $\circlearrowleft \bp 0 & 1 \\ 1 & 0 \ep$\\
 \hline
\end{tabular}
\end{center}
and 
\[
\sTr\left(\text{Frob}_q^n\circlearrowright H^* \right) = \begin{cases}
0 & \text{$n$ is odd}, \\
2 & \text{$n$ is even}. 
\end{cases}
\]
\end{example}
The Grothendieck--Lefschetz formula and the examples above can be realized as shadows of something happening on the level of sheaves. Let $k$ be a field. Then, roughly, 
\[
\Spec{k} = \pt/\Gal(\overline{k}/k). 
\]
In other words, one can think of $\Spec k$ as something like the classifying space of its absolute Galois group. There is a bijection 
\begin{align*}
    \begin{array}{c}
    \text{\'etale cohomology} \\
    \text{sheaves on $\Spec{k}$} \\
    \text{with $\Z/\ell^m \Z$-coefficients} \end{array} 
    &\xleftrightarrow{\sim} \begin{array}{c} \text{finitely generated} \\
    \text{$\Z/\ell^m \Z$-modules with} \\
    \text{an action of $\Gal(\overline{k}/k)$} \end{array}\\
    \text{$\overline{\mathbb{Q}}_\ell$-coefficients } &\hookrightarrow \begin{array}{c} \text{finite-dimensional continuous}\\
    \text{representations of $\Gal(\overline{k}/k)$} \\
    \text{on $\overline{\mathbb{Q}}_\ell$-vector spaces} \end{array}
\end{align*}
Let $\mc{F} \in D_c^b(X, \overline{\mathbb{Q}}_\ell)$ be a object in the bounded derived category of constructible $\overline{\mathbb{Q}}_\ell$-sheaves on $X$. Then for any $x \in X(\mathbb{F}_{q^n})$, we get an inclusion 
\[
\Spec{\mathbb{F}_{q^n}} \xhookrightarrow{i} X, 
\]
and hence an object $i^*\mc{F} \in D_c^b(\Spec{\mathbb{F}_q}, \overline{\mathbb{Q}}_\ell)$. Applying the supertrace of Frobenius to this object results in an element of $\overline{\mathbb{Q}}_\ell$. In this way, we get a map 
\[
D_c^b(X, \overline{\mathbb{Q}}_\ell) \xrightarrow{f} \prod_{n \geq 1} \Fun(X(\mathbb{F}_{q^n}) \rightarrow \overline{\mathbb{Q}}_\ell). 
\]
 If $\mc{F}\rightarrow \mc{G} \rightarrow \mc{F}' \xrightarrow{+1}$ is a distinguished triangle in $D^b_c(X, \overline{\mathbb{Q}}_\ell)$, then $f(\mc{F}) + f(\mc{F}') = f(\mc{G})$. In other words, $f$ factors through the Grothendieck group:
\[
\left[ D_c^b(X, \overline{\mathbb{Q}}_\ell) \right] \xrightarrow{f} \prod_{n \geq 0} \Fun(X(\mathbb{F}_{q^n}), \overline{\mathbb{Q}}_\ell). 
\]
A consequence of the Chebotarev density theorem (Theorem \ref{Chebotarev density}) is that this map is {\em injective}; i.e. for a given $q^n$, the collection of all of the functions $f(\mc{F})$ completely determine the class of $\mc{F}$ in the Grothendieck group. Now, Grothendieck tells us how to view this relationship. 

\vspace{5mm}
\noindent
{\bf Grothendieck's philosophy}: Interesting functions $X(\mathbb{F}_{q^n})\rightarrow \C, \overline{\mathbb{Q}}_\ell$, etc. should be shadows of interesting sheaves. 

\vspace{5mm}
Lusztig provided us with an extraordinary example of a case where Grothendieck's philosophy is exactly true. 

\begin{example}
(Lusztig's theory of character sheaves) There exists a set 
\[
\left\{\mc{F}_\chi \in D_c^b(\GL_n, \overline{\mathbb{Q}}_\ell) \right\} 
\]
such that $\mc{F}_\chi$ yield all irreducible characters $\chi:\GL_n(\mathbb{F}_{q^n}) \rightarrow \C$ of all $\GL_n(\mathbb{F}_{q^n})$ ``at once''. (More precisely, for a given $\mathbb{F}_{q^n}$, one should only consider those character sheaves which are ``defined over $\mathbb{F}_q$'', however we will not go into the details.)
\end{example}

\subsection{The Hecke algebra, revisited}
In the second half of this lecture, we will describe another example of Grothendieck's philosophy: the Hecke category. To start, we will recall the origin of the Hecke algebra. 

Let $G$ be a split reductive group over a finite field $\mathbb{F}_q$ and $B \subset G$ a Borel subgroup. We can define a convolution algebra 
\[
H_{\mathbb{F}_q}=\left(\Fun_{B(\mathbb{F}_q)\times B(\mathbb{F}_q)}(G(\mathbb{F}_q), \C), *\right)
\]
of complex-valued $B(\mathbb{F}_q)$-biinvariant functions on $G(\mathbb{F}_q)$. A priori, the structure of this algebra seems to depend on $q$, but Iwahori found a presentation of $H_{\mathbb{F}_q}$ which depends almost only on the Weyl group. Let $(W, S)$ be the Coxeter system associated to $B \subset G$. Iwahori's presentation has generators $\{T_s \mid s \in S\}$, and relations 
\begin{align*}
    T_s^2 =(q-1)T_s + q \\
    T_sT_t \cdots = T_s T_t \cdots,
\end{align*}
where the products on the second line are of $m_{st}=\text{order}(st)$ generators. In Iwahori's presentation, the element $T_s$ corresponds to the indicator function $\ind_{BsB} \in \Fun_{B(\mathbb{F}_q) \times B(\mathbb{F}_q)}(G(\mathbb{F}_q), \C)$. Iwahori's presentation demonstrated that the Hecke algebra $H_{\mathbb{F}_q}$ is ``defined over $\Z[q]$''. 

The quadratic relation may look a little mysterious at first, but it has a natural geometric origin.

\vspace{5mm}
\noindent
{\bf Origin of the quadratic relation}: Take $G=\SL_2(\mathbb{F}_q)$ and $B = \left\{ \bp * & * \\ 0 & * \ep \right\}$. Then on one hand,
\[
\ind_G * \ind_G = \begin{array}{c} \text{pushforward of the constant function on $G \times_B G$} \\ \text{under the multiplication map $G \times_B G\xrightarrow{m} G$} \end{array}
\]
But on the other hand, there is a natural isomorphism 
\begin{align*}
    G \times_B G &\xrightarrow{\sim} G/B \times G \\
    (g,h) &\mapsto (gB, gh) 
\end{align*}
and the multiplication map factors through this isomorphism:
\[
\begin{tikzcd}
G \times_B G \arrow[dr, "m"] \arrow[r, "\sim"] & G/B \times G \arrow[d, "\text{proj to }G"]\\
& G 
\end{tikzcd}
\]
Hence, 
\begin{align*}
    \ind_G * \ind_G &= \begin{array}{c} \text{pushforward of constant function on} \\ \text{$G/B \times G$ under the projection to $G$} \end{array} \\
    &=|G/B| \cdot \ind_G \\
    &=(1+q) \ind_G.
\end{align*}
From this, we deduce the quadratic relation: Since $\ind_G = \ind_{BsB} + \ind_B=T_s+1$, we have 
\begin{align*}
    T_s^2 + 2T_s + 1 &= (T_s + 1)^2 \\
    &= (q+1)(T_s+1)\\
    &=(q+1)T_s + q+1 \\
    &=(q-1)T_s + q + 2T_s + 1.
\end{align*}
Hence $T_s^2 = (q-1)T_s + q$. 

The geometric origin of the Hecke algebra suggests a categorification via Grothendieck's philosophy. Let $G$ be a split reductive group over $\mathbb{F}_q$, and $B \subset G$ a Borel subgroup. Define the (first incarnation of) the {\bf Hecke category} to be 
\[
\mc{H}:= D_{B \times B}^b(G, \overline{\mathbb{Q}}_\ell), 
\]
the $B \times B$-equivariant derived category of \'{e}tale $\overline{\mathbb{Q}}_\ell$-sheaves, in the sense of Bernstein-Lunts. We won't describe the precise construction of this category, but in its first approximation\footnote{Be careful! This approximation is just an approximation and can get you in trouble if you take it too literally (see Lecture \ref{lecture 24}).}, we can consider objects in $D_{B \times B}^b(G, \overline{\mathbb{Q}}_\ell)$ to be \'{e}tale sheaves which are constructible for $B\times B$-orbits. 

The convolution of functions in $H_{\mathbb{F}_q}$ can be upgraded to a {\bf convolution product} on sheaves. Let $\mc{F}, \mc{G} \in D_{B \times B}^b(G)$. We have maps
\[
\begin{tikzcd}
& G \times G \arrow[dl, "p_1"'] \arrow[dr, "p_2"] \arrow[rr] & & G \times_B G \arrow[r, "\mathrm{mult}"] & G \\
G & & G & & 
\end{tikzcd}
\]
Then if $\mc{F}, \mc{G} \in D_{B \times B}^b(G, \overline{\mathbb{Q}}_\ell)$, we define 
\[
\mc{F} * \mc{G} := \mathrm{mult}_*(\widetilde{\mc{F}\mc{G}}) \in D_{B \times B}^b(G, \overline{\mathbb{Q}}_\ell),
\]
where $\widetilde{\mc{F}\mc{G}}\in D^b_{B \times B}(G \times_B G, \overline{\mathbb{Q}}_\ell)$ corresponds to $\res_{B \times B \times B \times B}^{B \times B \times B}(p_1^*\mc{F} \otimes p_{2}^*\mc{G})$ under the equivalence
\[
D_{B \times B \times B\times B}^b(G \times G, \overline{\mathbb{Q}}_\ell) \xrightarrow{\sim} D^b_{B \times B \times B}(G \times_B G, \overline{\mathbb{Q}}_\ell).
\]
\begin{example}
Let $k=\overline{\mathbb{Q}}_\ell$, and $G=\SL_2$. Then 
\[
\underline{k}_{\SL_2} * \underline{k}_{\SL_2} = \mathrm{mult}_*(\underline{k}_{\SL_2 \times_B  \SL_2}) = H^*(\mathbb{P}^1) \otimes \underline{k}_{\SL_2}. 
\]
This is the categorified version of the Hecke algebra equality $\ind_G * \ind_G = (1+q)\ind_G$ that we discussed earlier! 
\end{example}

\subsection{Perverse sheaves on $\mathbb{P}^1\C$} 
The time has come for an interlude. For the rest of this lecture we will shift gears and consider the variety $\mathbb{P}^1\C$ with the stratification $\Lambda$ given by 
\[
\mathbb{P}^1\C = \{0\} \sqcup \C_\infty.  
\]
Denote by $\{0\} \xhookrightarrow{i} \mathbb{P}^1\C \xhookleftarrow{j} \C_\infty$ the natural inclusions. It is very important in what is coming to understand perverse sheaves on $\mathbb{P}^1\C$ with the stratification $\Lambda$, so we will discuss this (and generalizations) in the upcoming lectures. So what are the perverse sheaves on $\mathbb{P}^1\C$? 

Here's an algebraic answer (c.f. Emily's talk on nearby cycles in the Informal Friday Seminar\footnote{All IFS talk notes can be found on the IFS website \url{https://sites.google.com/view/ifssydney/home}}): 
\[
\Perv_\Lambda(\mathbb{P}^1\C, k) \simeq \left\{ \begin{tikzcd} V_1 \arrow[r, "c", shift left] & V_0 \arrow[l, "v", shift left] \end{tikzcd} \mid v \circ c = 0 \right\}, 
\]
where $V_i$ are finite-dimensional $k$ vector spaces representing the nearby cycles at $0$ (the vector space $V_1$) and the vanishing cycles at $0$ (the vector space $V_0$). From this perspective, we can see that there are five indecomposible objects:
\begin{enumerate}
    \item $\begin{tikzcd} 0 \arrow[r, shift left] & k \arrow[l, shift left] \end{tikzcd} \rightsquigarrow \hspace{2mm} i_*\underline{k}_0$, \text{ skyscraper at $0$} $\vcenter{\hbox{\includegraphics[scale=0.5]{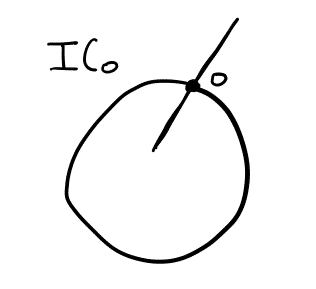}}}$ 
    \item $\begin{tikzcd} k \arrow[r, shift left] & 0 \arrow[l, shift left] \end{tikzcd} \rightsquigarrow \hspace{2mm} \text{constant sheaf}$ $\vcenter{\hbox{\includegraphics[scale=0.4]{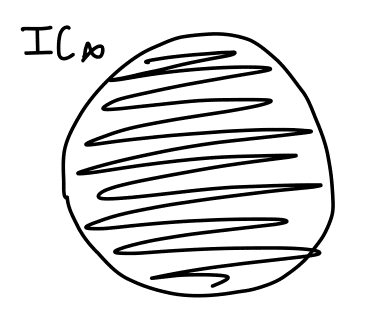}}}$
    \item $\begin{tikzcd} k \arrow[r, "\sim", shift left] & k \arrow[l, "0", shift left] \end{tikzcd} \rightsquigarrow \hspace{2mm} \vcenter{\hbox{\includegraphics[scale=0.5]{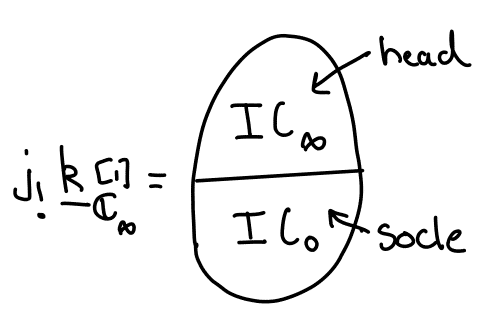}}}$
    \item $\begin{tikzcd} k \arrow[r, "0", shift left] & k \arrow[l, "\sim", shift left] \end{tikzcd} \rightsquigarrow \hspace{2mm} \vcenter{\hbox{\includegraphics[scale=0.5]{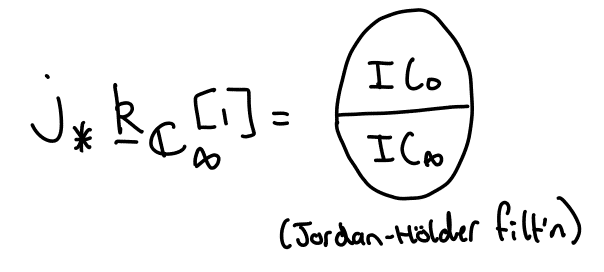}}}$
    \item $\begin{tikzcd} k \arrow[r, shift left] & k\oplus k \arrow[l,  shift left] \end{tikzcd} \rightsquigarrow \hspace{2mm} \text{``big projective/tilting sheaf''}$ $\vcenter{\hbox{\includegraphics[scale=0.5]{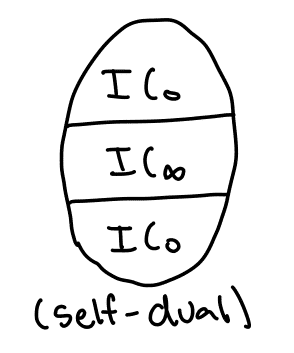}}}$
\end{enumerate}

In the next lecture we will go into more detail and describe perverse sheaves on general curves. 

%% file: lecture-17.tex
\section{Lecture 17: Constructible and perverse sheaves on curves}
\label{lecture 17}

Today we continue our interlude of the previous lecture and discuss perverse sheaves on curves. We are working toward understanding Beilinson glueing on curves, which will be the main topic of next week's lecture. 

\subsection{Constructible sheaves}

Let $k$ be a field, fixed throughout, and let $X/\C$ be a variety, viewed with the classical topology. Let $V$ be a $k$-vector space. We have the notion of a {\bf constant sheaf} $V_X$ with values in $V$:
\[
V_X(U):= \{ f:U \rightarrow V \text{ continuous}\},
\]
where $V$ is viewed with the discrete topology. This leads to the notion of a {\bf local system}, which is a locally constant sheaf with finite-dimensional stalks.
\begin{theorem}
If $X$ is connected, there is a bijection 
\[
\left\{\begin{array}{c} \text{local systems} \\
\text{on $X$} \end{array} \right\} \xleftrightarrow{\sim} \Rep(\pi_1(X,x)). 
\]
\end{theorem}
\begin{remark}
The difference between local systems and vector bundles is captured with the following picture. 
\[
\includegraphics[scale=0.6]{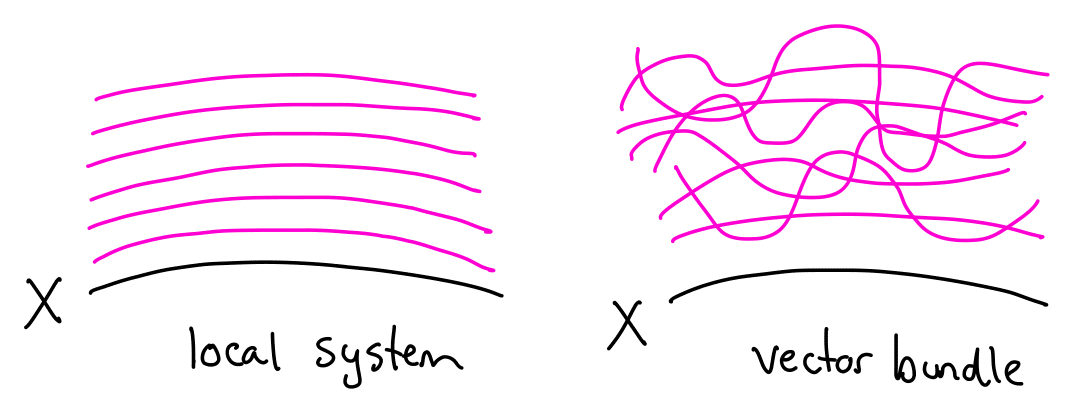}
\]
Here the pink lines are intended to denote the sections of our local systems/vector bundles. Vector bundles with flat connections are equivalent to local systems. However, local systems form an abelian category, whereas vector bundles do not. 
\end{remark}
The notion of a local system leads us to the notion to a {\bf constructible sheaf}: A sheaf $\mc{F}$ on $X$ is constructible if there exists a stratification $X = \bigsqcup_{\lambda \in \Lambda} X_\lambda$ of $X$ by a finite number of subvarieties $X_\lambda$ such that $\mc{F}|_{X_\lambda}$ is a local system for all $\lambda \in \Lambda$. Finally, the notion of a constructible sheaf leads to the notion of the {\bf bounded derived category of constructible sheaves}:
\[
D^b_c(X, k):= \left\{ \mc{F}  \left| \begin{array}{c} \mc{H}^i(\mc{F}) \text{ is constructible for all $i$, and } \\
\mc{H}^i(\mc{F})=0 \text{ for $|i|>>0$} \end{array} \right. \right\} \subset D^b \left( \begin{array}{c} \text{sheaves of $k$-vector} \\ \text{spaces on $X$} \\
\end{array} \right) 
\]

From now on, assume that $X$ is a connected, smooth curve. 
\[
\includegraphics[scale=0.8]{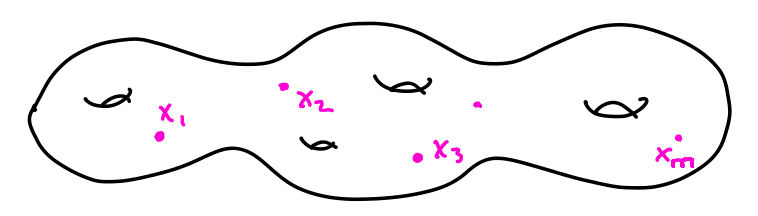}
\]
Choose a set of points $\{x_1, \ldots, x_m\}$, and let $U$ denote their complement in $X$. Fix the stratification $\Lambda$:
\[
X=U \sqcup \{x_1\} \sqcup \cdots \sqcup \{x_m\}.
\]
\begin{theorem}
There is an equivalence (a fun exercise if the reader is tempted!)
\[
\left\{ \begin{array}{c} \text{$\Lambda$-constructible} \\ 
\text{sheaves on $X$} \end{array} \right\} 
\xrightarrow{\sim} \left\{ \begin{array}{c} \text{$\mc{L}$ local system on $U$,}\\
\text{$V_1, \ldots , V_m$ finite-dimensional vector spaces} \\ \text{maps } \phi_i:V_i \rightarrow \left( \mc{L}_{n_i} \right)^{\mu_i} \end{array} \right\}. 
\]
Here $n_i \in X$ is a ``nearby point'' to $x_i$, and $\mc{L}_{n_i}$ denotes the stalk, which carries a monodromy operator $\mu_i$ given by a small loop around $x_i$ (see the diagram below). Note that the stalk depends on the choice of nearby point $n_i$, but the invariants in the stalk do not! 
\[
\includegraphics[scale=0.7]{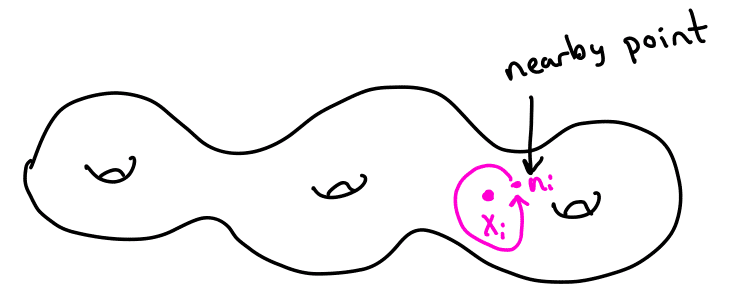}
\]
\end{theorem}

\subsection{Perverse sheaves}
One way to visualize an object $\mc{F} \in D_c^b(X)$ is via its table of stalks:
\begin{center}
 \begin{tabular}{c|c|c|c} 
 & $i-1$ & $i$ & $i+1$ \\ 
 \hline
 $U$ & $\cdots$ & $\mc{H}^i(\mc{F}|_{U})$ & $\cdots$  \\ 
 \hline
 $x_1$ & $\cdots$ & $\mc{H}^i(\mc{F}_{x_i})$ & $\cdots$ \\
 \hline
 $x_2$ & $\cdots$ & $\mc{H}^i(\mc{F}_{x_2})$ & $\cdots$ \\
 \hline 
 $\vdots$ & & $\vdots$ & 
\end{tabular}
\end{center}
We will use these tables frequently in the remainder of this lecture. 

The category of {\bf perverse sheaves}
\[
\Perv_\Lambda(X) \subset D^b_c(X)
\]
is defined as the heart of the $t$-structure
\[
D^{\leq 0}_\Lambda :=\left\{ \mc{F} \left| \begin{array}{c} \mc{H}^i(\mc{F}|_U)=0 \text{ for }i\geq 0 \\ \mc{H}^i(\mc{F}|_{x_j})=0 \text{ for }i>0 \end{array} \right. \right\}, \hspace{3mm} D_{\Lambda}^{\geq 0} = \mathbb{D}(D_\Lambda^{\leq 0}).
\]
Here $\mathbb{D}$ denotes {\bf Verdier duality}. We have two possibilities for what the table of stalks can look like for a perverse sheaf in $\Perv_\Lambda(X)$: 
\begin{center}
$\begin{array}{c} \text{full} \\ \text{support} \end{array}$
\begin{tabular}{c|c|c|c|c} 
 & $-2$ & $-1$ & $0$ & $1$  \\ 
 \hline
 $U$ & $0$ & $\mc{L}$ & $0$ & $0$  \\ 
 \hline
 $x_i$ & $0$ & $V$ & $W$ & $0$ \\
\end{tabular}
\hspace{7mm}
skyscraper
\begin{tabular}{c|c|c|c|c} 
& $-2$ & $-1$ & $0$ & $1$ \\ 
 \hline
 $U$ & $0$ & $0$ & $0$ & $0$  \\ 
 \hline
 $x_i$ & $0$ & $0$ & $W'$ & $0$ \\
\end{tabular}
\end{center}

A general strategy when trying to understand an abelian category is to try to produce and study exact functors from that category to a well-understood category (like $\Vect$ or representations of a group). For perverse sheaves, this strategy proves to be somewhat complicated, and brings nearby and vanishing cycles into our world. 

\begin{theorem}
(to be explained) 
\[
\Perv_{\Lambda}(X) \xleftrightarrow{\sim} \left\{ \begin{array}{c} \mc{L} \text{ local system on $U$,} \\
V_1, \ldots, V_m \text{ finite-dimensional vector spaces (``vanishing cycles''),} \\
\text{maps } \begin{tikzcd} V_i \arrow[r, "v", shift left] & \mc{L}_{n_i} \arrow[l, "u", shift left] \end{tikzcd} \text{ s.t. } v \circ u = id-\mu_i \end{array} \right\} 
\]
\end{theorem}
\begin{remark}
The category $\Perv_\Lambda(X)$ is Verdier self-dual. The category of $\Lambda$-constructible sheaves is not: 
\[
\mathbb{D}(\Lambda\text{-constructible sheaves}) \simeq \left\{ \phi_i: \mc{L}_{n_i}^{\mu_i} \rightarrow V_i \right\}  
\]
\end{remark}

\begin{example}
\label{examples of perverse sheaves}
Here are some examples of perverse sheaves. 
\begin{itemize}
    \item The constant sheaf $k_{x_j}$ on $\{x_j\}$ is perverse. 
    \item For a local system $\mc{L}$ on $X$, $\mc{L}[1]$ is perverse. (This follows because $\mathbb{D} \mc{L}[1] \cong \mc{L}^\vee[1]$, where $\mc{L}^\vee$ denotes the dual local system.)
    \item Let $j:U\hookrightarrow X$ and $\mc{L}$ a local system on $U$. Then we claim that $j_!\mc{L}[1]$ and $j_*\mc{L}[1]$ are both perverse. Because $j_!$ is extension by zero, computing the stalks of $j_!\mc{L}[1]$ is easy:
    \begin{center}
    \begin{tabular}{c|c|c|c|c} 
    & $-2$ & $-1$ & $0$ & $1$ \\ 
    \hline
    $U$ & $0$ & $\mc{L}$ & $0$ & $0$  \\ 
    \hline
    $x_i$ & $0$ & $0$ & $0$ & $0$ \\
   \end{tabular} $\in D_\Lambda^{\leq 0}$
   \end{center}
   Computing the stalks of the direct image $j_*:=Rj_*$ is trickier. Le $x \in X$ be a point. Then 
   \begin{align*}
   H^i((j_*\mc{L})_x) &= \lim_{\longleftarrow \atop B(x, \epsilon)} H^i(B(x, \epsilon) \cap U, \mc{L}) \\
   &= \begin{cases} \mc{L}_x &\text{ if }x \in U, \\ H^i(B(x, \epsilon) \backslash \{x\}, \mc{L}) &\text{ if }x \not \in U. \end{cases}
   \end{align*}
   How can we compute the cohomology $H^i(B(x, \epsilon) \backslash \{x\}, \mc{L})$? The space $B(x, \epsilon) \backslash \{x\}$ is homotopic to $S^1$. We can compute the cohomology of $S^1$ with coefficients in $\mc{L}$ using the two-term complex
   \[
   V \xrightarrow{id - \mu} V. 
   \]
   Hence $H^0(S^1, \mc{L}) = V^\mu$, the invariants of $\mu \circlearrowright V$, and $H^1(S^1, \mc{L}) = V_\mu$, the coinvariants. (Alternatively, one can see that $H^0(S^1, \mc{L})=V^\mu$ directly, and conclude that $H^1(S^1, \mc{L})=V_\mu$ by Poincar\'{e} duality.) Hence, the table of stalks of $j_*\mc{L}[1]$ is
    \begin{center}
    \begin{tabular}{c|c|c|c|c} 
    & $-2$ & $-1$ & $0$ & $1$ \\ 
    \hline
    $U$ & $0$ & $\mc{L}$ & $0$ & $0$  \\ 
    \hline
    $x_i$ & $0$ & $V^\mu$ & $V_\mu$ & $0$ \\
   \end{tabular} $\in D_\Lambda^{\leq 0}$.
   \end{center}
   Because $\mathbb{D}(j_* \mc{L}) = j_!\mathbb{D}(\mc{L})$, we can conclude from these computations that $j_*\mc{L}[1]$ and $j_!\mc{L}[1]$ are both perverse. In other words, $j_*$ and $j_!$ are exact for the perverse $t$-structure. 
\end{itemize}
\end{example}

Next we turn our attention to classifying the simple objects in $\Perv_\Lambda(X)$. Recall that there is a natural map
\[
j_! \rightarrow j_*. 
\]
Given $\mc{L}$ a local system on $\mc{U}$, define
\[
IC(X, \mc{L}):= j_{!*}(\mc{L}):= \im(j_!\mc{L}[1] \rightarrow j_* \mc{L}[1]). 
\]
This is not obvious, but by a construction of Deligne, we have 
\[
IC(X, \mc{L})=\tau_{\leq -1}(j_*\mc{L}[1]),
\]
where $\tau_{\leq -1}$ is the truncation functor in the standard (not perverse) $t$-structure. Hence the table of stalks for $IC(X, \mc{L})$ is 
\begin{center}
    \begin{tabular}{c|c|c|c|c} 
    & $-2$ & $-1$ & $0$ & $1$ \\ 
    \hline
    $U$ & $0$ & $\mc{L}$ & $0$ & $0$  \\ 
    \hline
    $x_i$ & $0$ & $V^\mu$ & $0$ & $0$ \\
   \end{tabular}.
\end{center}
\begin{remark}
We see from the remarks above that when $X$ is a curve, all IC sheaves are shifts of actual constructible sheaves. This is not the case in general! 
\end{remark}

\begin{example}
\label{no roots of unity}
Consider a local system $\mc{L}$ on $U$ such that the monodromy $\mu$ does not have $1$ as an eigenvalue. Then $id - \mu$ is invertible. Hence $V_\mu = V^\mu = 0$. Recall that the general form of the tables of stalks of $j_!\mc{L}[1]$ and $j_*\mc{L}[1]$ are
\begin{center}
    $j_!\mc{L}[1]$: \begin{tabular}{c|c|c|c|c} 
    & $-2$ & $-1$ & $0$ & $1$ \\ 
    \hline
    $U$ & $0$ & $\mc{L}$ & $0$ & $0$  \\ 
    \hline
    $x_i$ & $0$ & $0$ & $0$ & $0$ \\
   \end{tabular} \hspace{5mm} $j_*\mc{L}[1]:$ \begin{tabular}{c|c|c|c|c} 
    & $-2$ & $-1$ & $0$ & $1$ \\ 
    \hline
    $U$ & $0$ & $\mc{L}$ & $0$ & $0$  \\ 
    \hline
    $x_i$ & $0$ & $V^\mu$ & $V_\mu$ & $0$ \\
   \end{tabular}.
   \end{center}
   So the fact that $V^\mu = V_\mu = 0$ implies that the natural map $j_!\mc{L}[1] \rightarrow j_*\mc{L}[1]$ is an isomorphism. We conclude that in this case
   \[
   IC(X, \mc{L}) = j_!\mc{L}[1]=j_*\mc{L}[1]. 
   \]
\end{example}

\begin{theorem}
The category $\Perv_\Lambda(X)$ is a finite-length abelian category with simple objects 
\[
\left\{ IC(X, \mc{L}) \mid \mc{L}|_U \text{ is irreducible} \right\} \cup \left\{ i_*k \mid i:\{x_i\} \hookrightarrow X \right\}.
\]
\end{theorem}

For any irreducible local system $\mc{L}$ on $U$, we have the distinguished triangle
\[
\tau_{\leq -1} (j_*\mc{L}[1]) \rightarrow j_*\mc{L}[1] \rightarrow \tau_{\geq 0}( j_*\mc{L}[1]) \xrightarrow{+1}  
\]
\[
= 
\]
\[
 IC(X, \mc{L}) \rightarrow j_*\mc{L}[1] \rightarrow (V_\mu)_0 \xrightarrow{+1}
\]
in $D_c^b(X, k)$. There is something distinctive about this triangle: all objects are perverse sheaves! Hence, this is an exact sequence and we've found our first composition series. We can draw this composition series with an ``egg diagram:''
\begin{equation}
\label{first egg}
j_*\mc{L}[1]=\vcenter{\hbox{\includegraphics[scale=0.5]{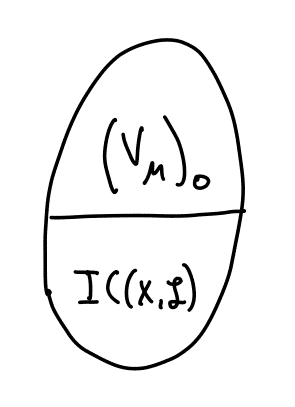}}}
\end{equation}
\begin{example}
Here is a special case of the construction above. Let $D$ be a disc and
\[
D^\times \xhookrightarrow{j} D \xhookleftarrow{i} \{0\}
\]
the natural inclusions. Then the composition series of $j_*k_{D^\times}[1]$ is 
\[
j_*k_{D^\times}[1]=\vcenter{\hbox{\includegraphics[scale=0.5]{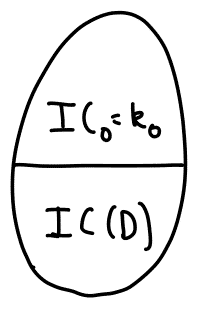}}}.
\]
\end{example}
\begin{example}
Let 
\[
U \xhookrightarrow{j} X \xhookleftarrow{i} Z:=X \backslash U
\]
be the natural inclusions. Then in $D_c^b(X,k)$ we have the distinguished triangle
\begin{equation}
    \label{pink star}
j_!j^!k_X \rightarrow k_X \rightarrow i_*i^*k_X \xrightarrow{+1}.
\end{equation}
By turning triangles, we obtain
\[
i_*k_Z \rightarrow j_!k_U[1] \rightarrow k_X[1]=IC(X) \xrightarrow{+1}. 
\]
All objects in this triangle are perverse sheaves, so this is a composition series! Hence we have the following egg:
\[
j_*k_{D^\times}[1]=\vcenter{\hbox{\includegraphics[scale=0.5]{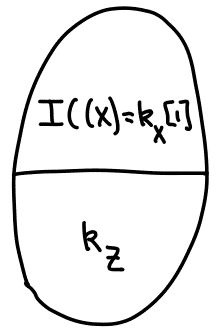}}}.
\]
We could also have obtained this by dualizing (\ref{first egg}).
\end{example}
\begin{remark}
Constructible sheaves on $X$ are not necessarily finite length (see (\ref{pink star})). However, perverse sheaves are. This is one of the reasons why we are so fond of them. 
\end{remark}

Before moving on to glueing, we will do one more fun calculation. We return to the local case:
\[
D^\times \xhookrightarrow{j} D \xhookleftarrow{i} \{0\}
\]
Let $\mc{L}_n$ be a local system on $D^\times$ with monodromy given by a single Jordan block of size $n$:
\[
J_n=\bp 1 & 1 &  &  &  \\  & 1 & 1 & &  \\ & & 1 & \ddots & \\ & & & \ddots & 1\\
& & & & 1 \ep
\]
Our goal is to describe $j_{!*}\mc{L}_n[1]=\im(j_!\mc{L}_n[1] \rightarrow j_*\mc{L}_n[1])$ in this setting. Recall that the functors $j_!$ and $j_*$ are exact because $j$ is affine, and we have the building blocks
\[
j_*k_{D^\times}[1]=\vcenter{\hbox{\includegraphics[scale=0.5]{images/j-star-D.png}}} \hspace{10mm} j_!k_{D^\times}[1]=\vcenter{\hbox{\includegraphics[scale=0.5]{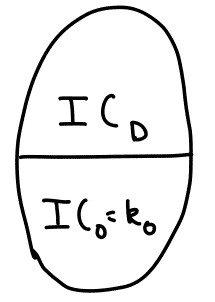}}}.
\]
In local systems on $D^\times$, we have
\[
\includegraphics[scale=0.7]{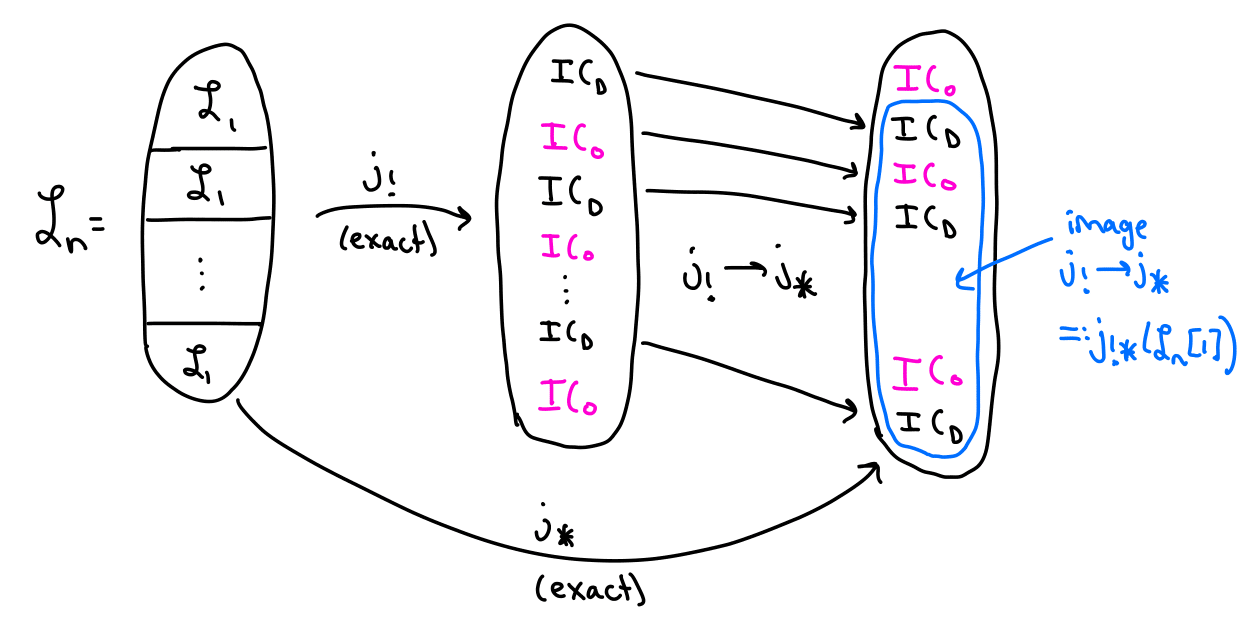}
\]
\begin{remark}
In this example, we see that the functor $j_{!*}$ preserves injections and surjections, but is NOT exact.
\end{remark}

\subsection{Nearby and vanishing cycles}
Let $D$ be the disc, $D^\times$ the punctured disc, and $\widetilde{D^\times}$ its universal cover. Let 
\[
\{0\} \xhookrightarrow{i} D \xhookleftarrow{j} D^\times \xleftarrow{p} \widetilde{D^\times} 
\]
be the natural maps. 

\vspace{3mm}
\noindent
{\bf Motivation}: We want to define the ``stalk'' of a perverse sheaf at singular points. 

\vspace{3mm}
Recall the {\bf nearby cycles} functor $\psi_f$ (c.f. Emily's Oct 18, 2019 IFS talk or Laurentiu's Feb 21/28, 2020 IFS series): For a local system $\mc{F}$ on $D^\times$, 
\begin{align*}
    \psi_f(\mc{F}) &= i^*j_*p_*p^*\mc{F} \\
    &=i^*j_*p_*(p^*\mc{F} \otimes k_{\widetilde{D^\times}}) \\
    &= i^*j_*(\mc{F} \otimes p_*k_{\widetilde{D^\times}}). 
\end{align*}
Here the third equality follows from the projection formula. The sheaf $p_*k_{\widetilde{D^\times}}$ is the ``universal local system on $D^\times$''. It is an infinite-dimensional local system with stalk $k[x^{\pm 1}]$ and monodromy $x$. 

Recall that the monodromy $\mu \circlearrowright \psi_f(\mc{F})$, so we get a decomposition into generalized eigenspaces:
\[
\psi_f(\mc{F}) = \bigoplus \psi_f^\lambda(\mc{F}).
\]
By far the most important of these eigenspaces corresponds to $\lambda = 1$, the ``unipotent nearby cycles''. Indeed, from this, we can recover all others:
\[
\psi_f^\lambda(\mc{F}) = \psi_f^1(\mc{F} \otimes k_{\lambda^{-1}}).
\]
From now on, set 
\[
\psi(\mc{F}):=\psi_f^1(\mc{F}).
\]
Let $\mc{L}_{unip}$ be the {\bf universal unipotent local system}. This is the local system with stalk $k[[u]]$ at $1$ and monodromy given by multiplication by $1+u$. In other words, 
\[
\mc{L}_{unip} = \lim_{\longleftarrow} \mc{L}_n,
\]
where $\mc{L}_n$ is our local system from earlier with stalk $k[u]/(u^n)$ and monodromy multiplication by $1+u$. (Think: ``completion of the augmentation ideal'' should give the same answer.) 

\vspace{3mm}
\noindent
{\bf Idea}: 
\begin{enumerate}
    \item We should think that 
    \[
    \psi(\mc{F}) \text{``=''} \lim_{\longleftarrow}(i^*(j_*(\mc{F} \otimes \mc{L}_n))). 
    \]
    \item $\vcenter{\hbox{\includegraphics[scale=0.7]{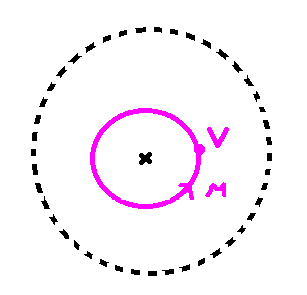}}}$ unipotent $\rightsquigarrow \begin{array}{c} H^0(D^\times, \mc{F}) = V^\mu \\ H^1(D^\times, \mc{F})=V_\mu \end{array}$ \hspace{3mm} How do we recover $V$?
\end{enumerate}

We will see in the next lecture that after tensoring with the local system given by a big Jordan block, we can recover $V$ and its monodromy via taking global sections. This appears to be a basic idea behind Beilinson's construction.

Hence it makes sense to define 
\[
\psi(\mc{F}) = H^1(D_\epsilon^\times, \mc{F} \otimes \mc{L}_n)
\]
for $n$ large. 

%% file: lecture-18.tex
\section{Lecture 18: Beilinson gluing on curves}
\label{lecture 18}

Throughout this lecture, we will work in the following setting. Let $D$ be the disc, $D^\times$ the punctured disc, and 
\[
D^\times \xhookrightarrow{j} D \xhookleftarrow{i} \{0\} 
\]
the natural inclusions. 

Our goal for this lecture is to explain the proof of Beilinson gluing\footnote{In addition to Beilinson's original paper \cite{Beilinson}, the notes \cite{Morel} of Sophie Morel are an excellent reference for this material.}. We'll pick up where we left off last week.

\subsection{A fact about unipotent monodromy}

Let 
\[
e=\bp 0 & 1 \\ 0 & 0 \ep, \hspace{2mm} h = \bp 1 & 0 \\ 0 & {-1} \ep, \hspace{2mm} f = \bp 0 & 0 \\ 1 & 0 \ep 
\]
be the standard basis for $\mf{sl}(2,\C) = \C f \oplus \C h \oplus \C e$. Clebsch--Gordan tells us about how the tensor product of two irreducible finite-dimensional $\mf{sl}(2,
\C)$-modules decomposes:

\begin{theorem}
(Clebsch--Gordan) Let $L_n, L_m$ be irreducible finite-dimensional $\mf{sl}(2,\C)$-modules of highest weight $n,m$, respectively. Then 
\[
L_n \otimes L_m = L_{|n-m|} \oplus \cdots \oplus L_{n+m-2} \oplus L_{n+m}.
\]
\end{theorem}
A consequence of this theorem is that for $n$ large, there are $m+1$ summands, so 
\[
(L_n \otimes L_m)^e = \begin{array}{c} \text{highest weight vectors}\\
\text{in above sum} \end{array} = \C^{m+1} \simeq L_m.
\]
(Here the superscript denotes Lie algebra invariants; i.e. the kernel of the action 
\[
e \cdot v \otimes w = ev \otimes w + v \otimes ew
\]
of $e$ on $L_n \otimes L_m$.) Moreover, this is more than just an isomorphism of vector spaces: the action of $-e \otimes 1$ on $(L_n \otimes L_m)^{e}$ aligns with the action of $e$ on $L_m$. Hence (be exponentiating) we have the following corollary to Clebsch--Gordan's theorem. 
\begin{corollary}
\label{CG corollary}
Let $V$ be a finite-dimensional vector space and $\phi$ a unipotent endomorphism of $V$. Let $J_n$ be a $(n+1)$-dimensional vector space and 
\[
\phi'= \bp 1 & 1 & & \\ & 1 & \ddots & \\ & & \ddots & 1 \\
& &  & 1 \ep 
\]
a Jordan block of size $n+1$. For $n$ large, 
\[
(V \otimes J_n)^{\phi \otimes \phi'} \simeq V,
\]
and this isomorphism is equivariant with respect to the action of $1 \otimes (\phi')^{-1}$ on the LHS and $\phi$ on the RHS.
\end{corollary}
Hence by tensoring with a Jordan block and taking invariants, we can recover the whole vector space and its monodromy action. 

\begin{remark} (Aside for the Lie theorists) Let $\mf{g}/\C$ be a semisimple Lie algebra, and $V$ a finite-dimensional $\mf{g}$-module. For $\lambda$ sufficiently dominant, 
\[
\Delta_\lambda \otimes V \simeq \bigoplus_{\gamma_i \in \Gamma} \Delta_{\lambda + \gamma_i},
\]
where $\Delta_\mu$ is the Verma module of highest weight $\mu$ and $\Gamma$ is the multiset of weights of $V$. Moreover, there is an isomorphism of $\mf{n}_+$-modules:
\[
(\Delta_\lambda \otimes V)^\mf{b} \simeq V.
\]
\end{remark}

\subsection{The unipotent vanishing cycles functor}

Now we return to the world of perverse sheaves. A perverse sheaf $\mc{M} \in \Perv(D^\times)$ is the same thing as a vector space $V$ and monodromy endomorphism $\mu \circlearrowright V$. Let $\mc{L}_n$ be the local system with stalk $k[x]/(x^n)$ and monodromy $\phi=1+x$, as in Lecture \ref{lecture 17}. In light of Corollary \ref{CG corollary} and the third bullet point in Example \ref{examples of perverse sheaves}, for $n$ large enough, we could define a functor $\psi$ by
\[
\psi(\mc{M}) = H^{0}(j_*(\mc{M} \otimes \mc{L}_n)) = (V \otimes J_{n-1})^{\mu \otimes \phi} = V.
\]
This will end up being our definition which allows us to see certain features in the arguments below. The key lemma of today's lecture tells us that this definition stabilizes for large $n$. 
\begin{lemma}
Let $\mc{M} \in \Perv(D^\times)$. For a fixed $m$, the kernel and cokernel of 
\[
j_!(\mc{M} \otimes \mc{L}_n) \xrightarrow{x^m} j_*(\mc{M} \otimes \mc{L}_n)
\]
stabilize for large $n$, and are canonically isomorphic. 
\end{lemma}
\begin{proof}
We will prove the lemma when $\mc{M}$ has rank $1$ (i.e. $\mc{M}$ corresponds to the vector space $k$ with monodromy given by multiplication by $\lambda \in k^\times$.) 

\vspace{2mm}
\noindent
{\bf Case 1}: monodromy $\lambda \neq 1$. Then 
\[
j_!(\mc{M} \otimes \mc{L}_n) = j_* (\mc{M} \otimes \mc{L}_n).
\]
(See Example \ref{no roots of unity}.) The theorem follows in this case because the kernel and cokernel of $\mc{L}_n \xrightarrow{x^m} \mc{L}_n$ stabilize for $m$ fixed, $n$ large and $j_!, j_*$ are exact. 

\vspace{2mm}
\noindent
{\bf Case 2}: monodromy $\lambda =1$; i.e. $\mc{M}=k_{D^\times}[1]$ is the trivial local system. We will illustrate what is happening in this case through an example. Let $m=2$ and $n=3$. We can see what the maps $\mc{L}_3 \xrightarrow{x^2} \mc{L}_3$, $j_!\mc{L}_3 \xrightarrow{x^2} j_! \mc{L}_3$ and $j_!\mc{L}_3 \xrightarrow{x^2} j_* \mc{L}_3$ are doing on egg diagrams: 
\[
 \includegraphics[scale=0.5]{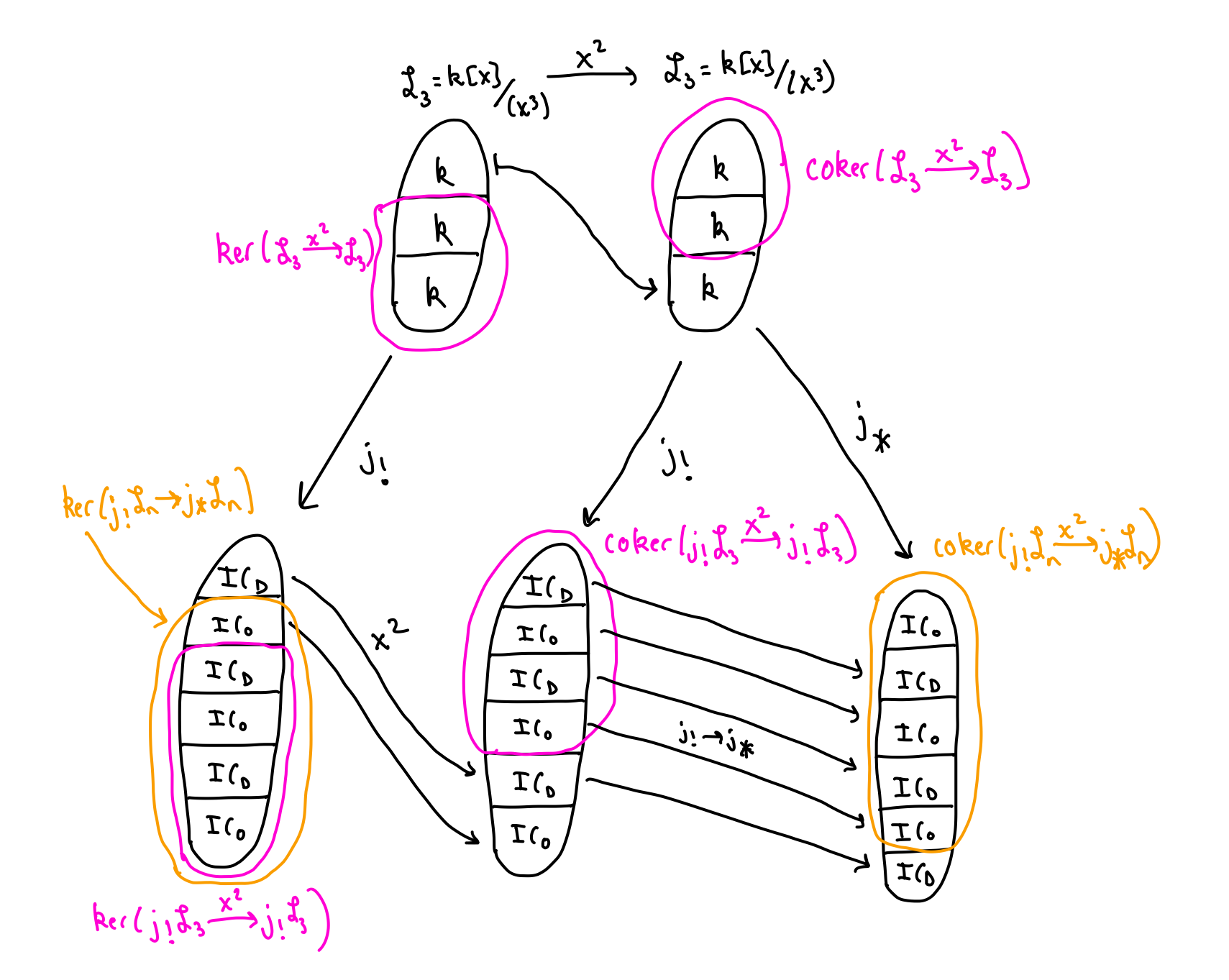}
\]
As $n$ gets larger, the eggs in the illustration above get longer and the image of $j_{!*}$ gets bigger, but the kernel and cokernel stay the same size. 
\end{proof}

\begin{remark}
What Beilinson actually proves is that 
\[
\lim_{\leftrightarrow} j_!(\mc{M} \otimes \mc{L}_n) \simeq \lim_{\leftrightarrow} j_*(\mc{M} \otimes \mc{L}_n),
\]
where `$\lim_{\leftrightarrow}$' is appropriately defined. We'll sum this up with the slogan ``middles agree''. We already saw this happening above, and will see it happening again in our next example.
\end{remark}

\begin{example}
\label{middles agree}
Here is a $2$-dimensional example. Let 
\[
\C^2 \xrightarrow[xy]{f} \C,
\]
and consider $\Perv_\Lambda(\C^2)$, where $\Lambda$ is the stratification via coordinate hyperplanes and their complement $U$.
\[
 \includegraphics[scale=0.5]{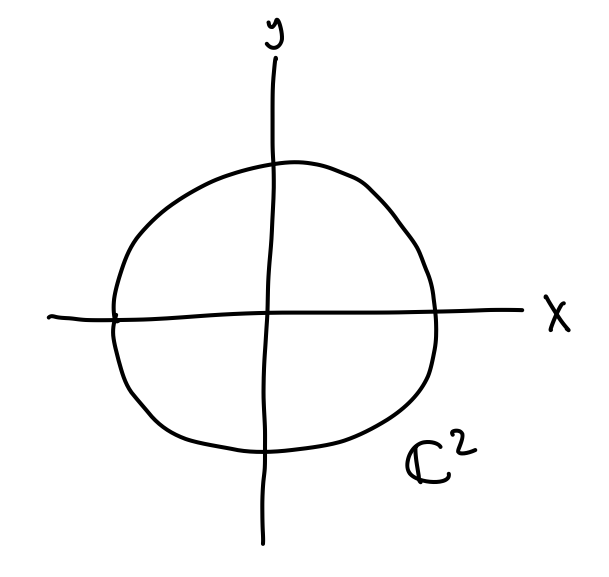}
\]
Let $\mc{M} = k_{U} [2]$, and denote by $IC_1 = IC_{\{y=0\}}$, $IC_2=IC_{\{x=0\}}$, and $IC_0=IC_{\{x=y=0\}}$. The composition series of $j_!k_U[2]$ is given by the following egg: 
\[
j_!k_U[2] = \vcenter{\hbox{
 \includegraphics[scale=0.6]{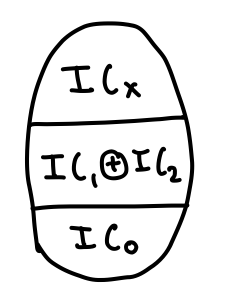}}}
\]
Then the key Lemma can be seen on egg diagrams:
\[
 \includegraphics[scale=0.45]{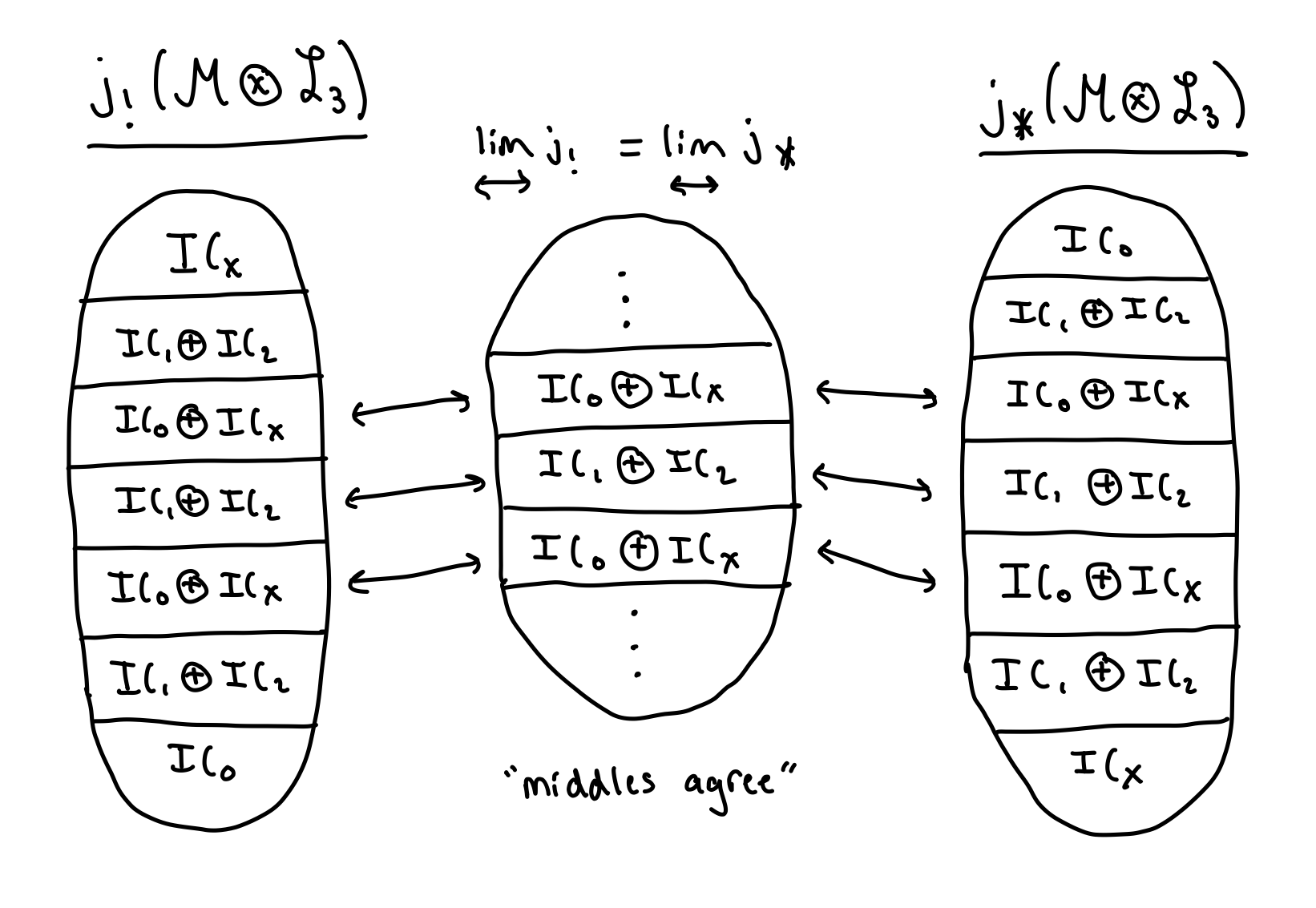}
\]
\end{example}

\begin{definition} (Beilinson) Let $X \xrightarrow{f} D$, $U = f^{-1}(D^\times)$, and $Z=f^{-1}(0)$. For a perverse sheaf $\mc{M} \in \Perv(U)$ and $n$ large, define functors $\psi_f, \Xi_f:\Perv(U) \rightarrow \Perv(X)$ by 
\begin{itemize}
    \item $\psi_f(\mc{M}) = \ker(x^0) \simeq \coker(x^0)$ \hspace{5mm} ``unipotent nearby cycles''
    \item $\Xi_f(\mc{M}) = \ker(x^1) \simeq \coker(x^1)$ \hspace{5mm}``maximal extension''
\end{itemize}
The perverse sheaf $\psi_f(\mc{M})$ is supported on $Z$ and $\Xi_f(\mc{M})$ is supported everywhere. 
\end{definition}
\begin{example} Let $\mc{M}=k_U[1]$, then we can see the images of these functors in eggs:
\[
 \includegraphics[scale=0.5]{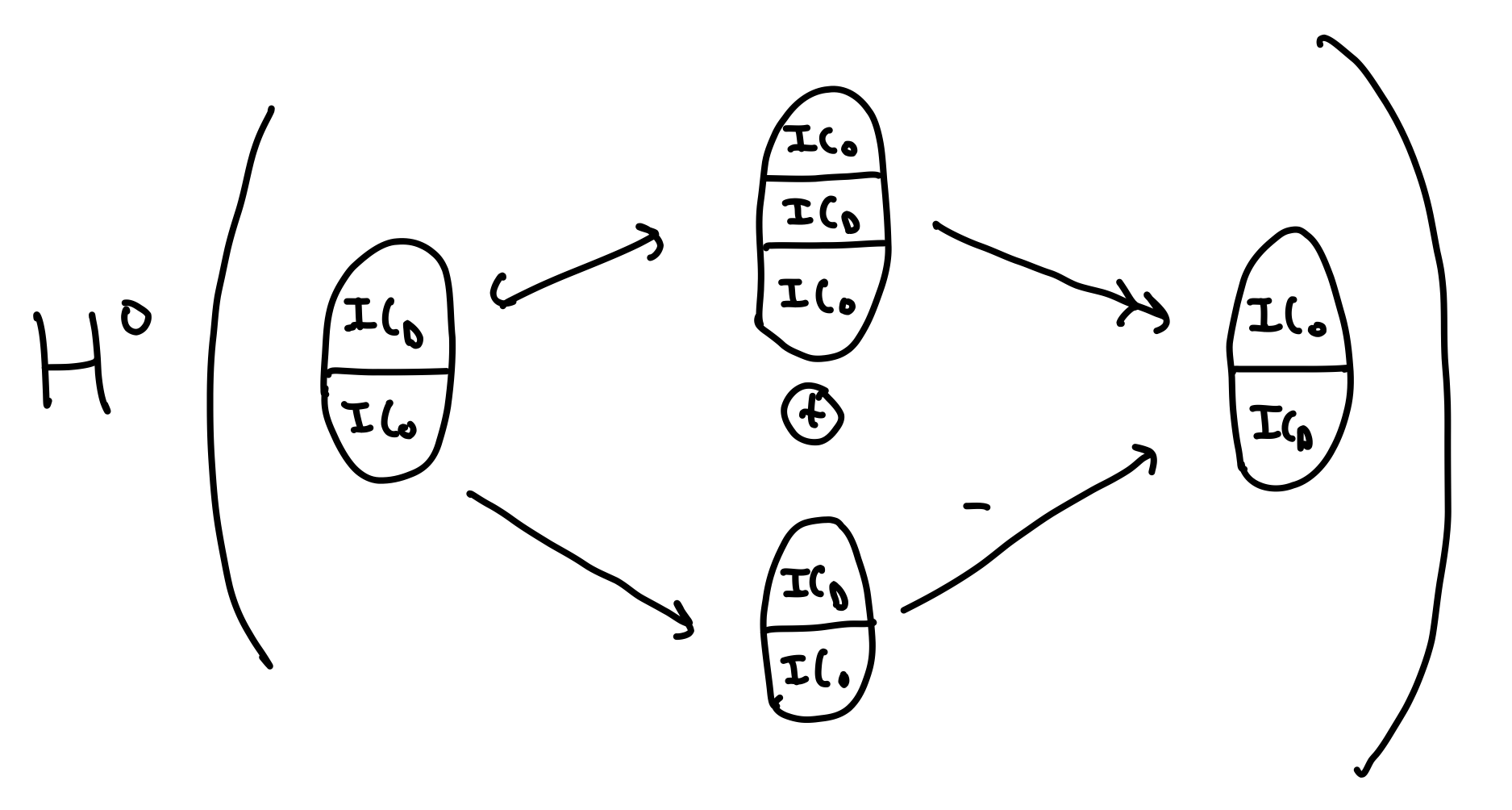}
\]
\end{example}
\begin{example} In our $2$-dimensional example, Example \ref{middles agree} we have:
\[
 \includegraphics[scale=0.6]{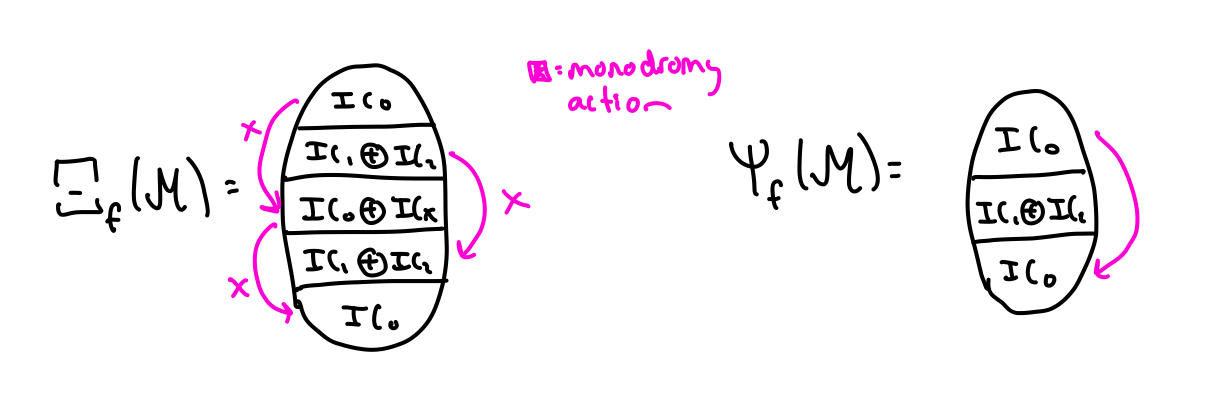}
\]
\end{example}
\begin{lemma}
The functors $\psi_f$ and $\Xi_f$ are exact, and we have functorial short exact sequences
\begin{enumerate}
    \item $j_! \hookrightarrow \Xi_f \twoheadrightarrow \psi_f$ 
    \item $\psi_f \hookrightarrow \Xi_f \twoheadrightarrow j_*$ 
\end{enumerate}
Moreover, the canonical map 
\[
\psi_f \rightarrow \Xi_f \rightarrow \psi_f
\]
agrees with monodromy $-1$.
\end{lemma}

\begin{proof}
See Geordie's handwritten notes on the course website.
\end{proof} 

\begin{definition}
Let $\mc{M} \in \Perv(D)$. The {\bf unipotent vanishing cycles functor} is defined by 
\[
\phi_f(\mc{M}) = H^0 \left( \begin{tikzcd} & \Xi_f(\mc{M}_U) \arrow[dr, twoheadrightarrow] & \\
j_!(\mc{M}_U) \arrow[ur, hookrightarrow] \arrow[dr] & \oplus & j_*(\mc{M}_U) \\
& \mc{M} \arrow[ur, "-1"] & \end{tikzcd} \right) 
\]
All maps in the diagram above are the canonical ones. 
\end{definition}
\begin{remark}
\begin{enumerate}
    \item The map $j_! \rightarrow \Xi_f$ is injective, and the map $\Xi_f \rightarrow j_*$ is surjective, so the cocomplex above only has cohomology in degree $0$, and hence $\phi_f$ is an exact functor on $\Perv(D)$. 
    \item $x$ induces a monodromy endomorphism of $\Xi_f, \psi_f, \phi_f$. 
\end{enumerate}
\end{remark}
\begin{example}
Let $X=D$ and $\mc{M}=j_!k_{D^\times}[1]$. To compute $\phi_f(\mc{M})$, we need to compute the zeroeth cohomology
\[
 \includegraphics[scale=0.3]{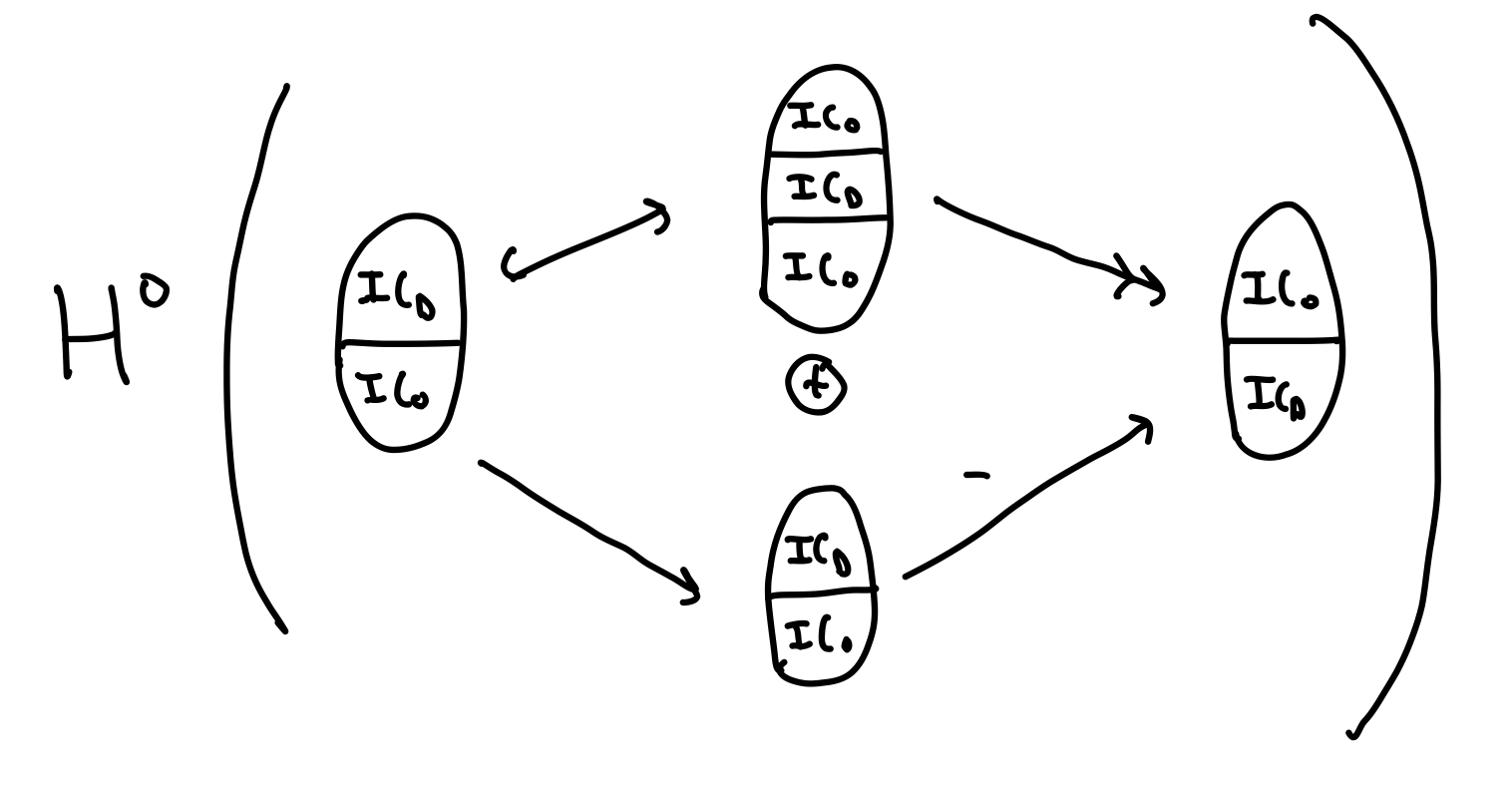}.
\]
The image of the first map is
\[
\includegraphics[scale=0.3]{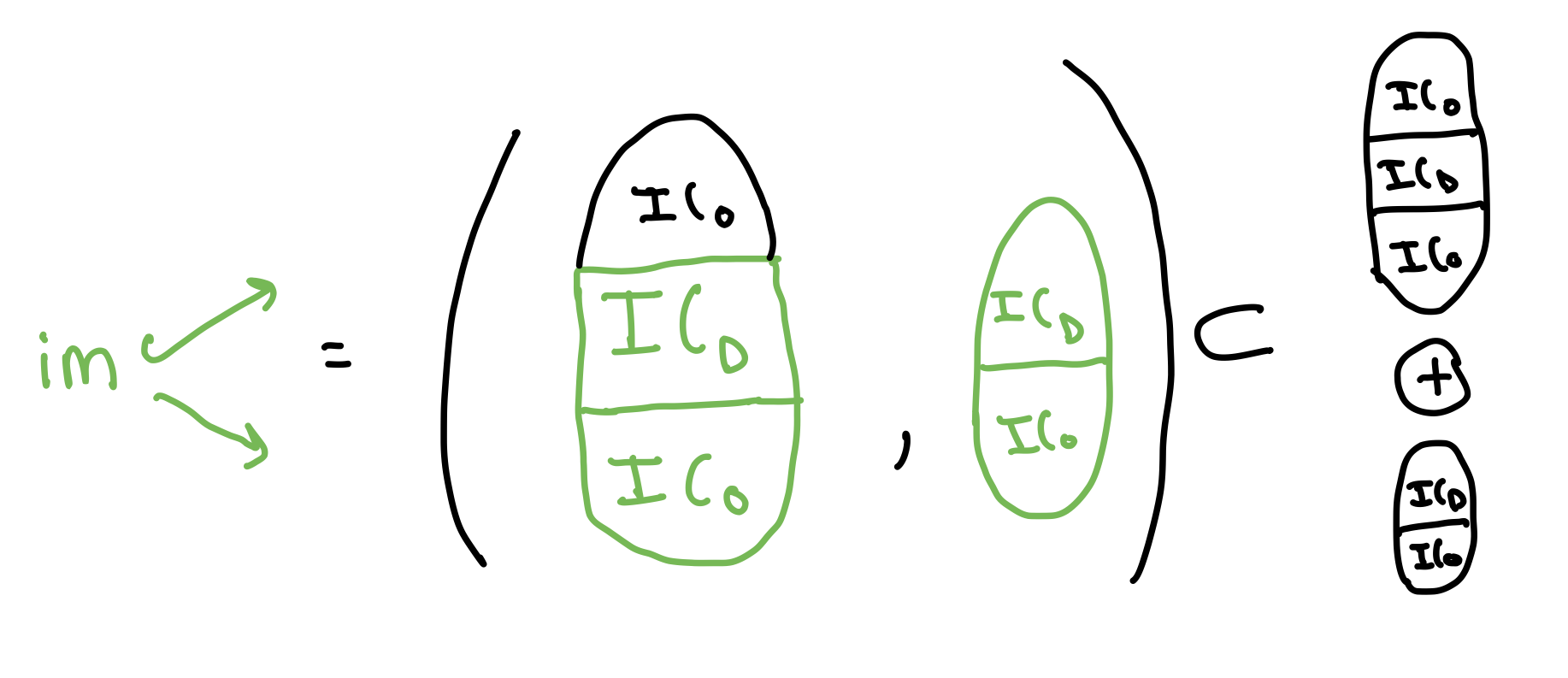}
\]
The kernel of the second map is
\[
\includegraphics[scale=0.4]{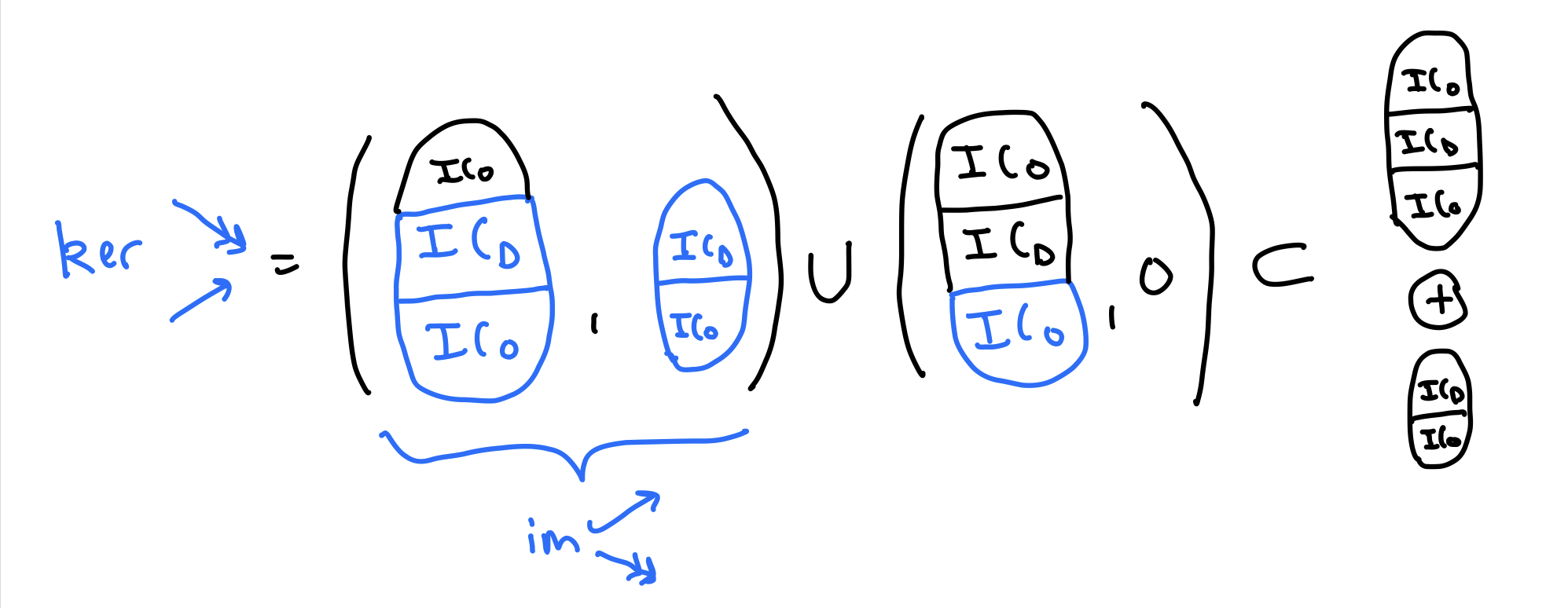}
\]
Notice that because of the sign in the second map, the image will always be contained in the kernel. Hence $\phi_f(\mc{M})$ is equal to 
\[
\includegraphics[scale=0.5]{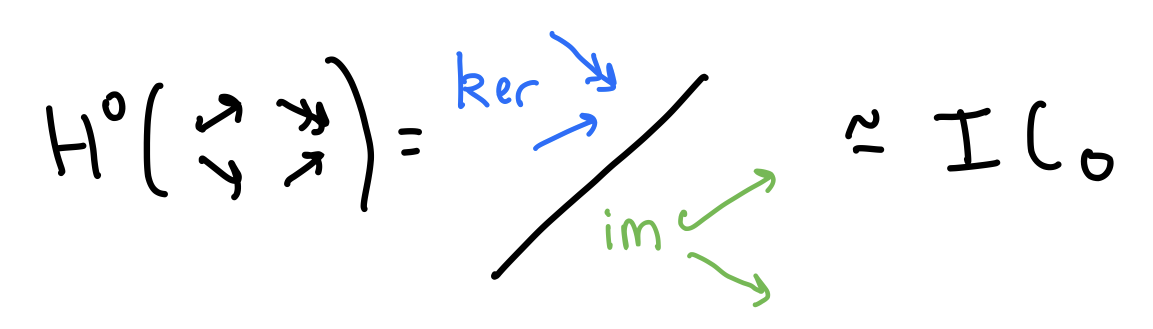}.
\]
\end{example}
\begin{example}
If $\mc{M}=k_D[1]$, then using the same method as above, we compute
\[
\phi_f(\mc{M}) = \vcenter{\hbox{\includegraphics[scale=0.4]{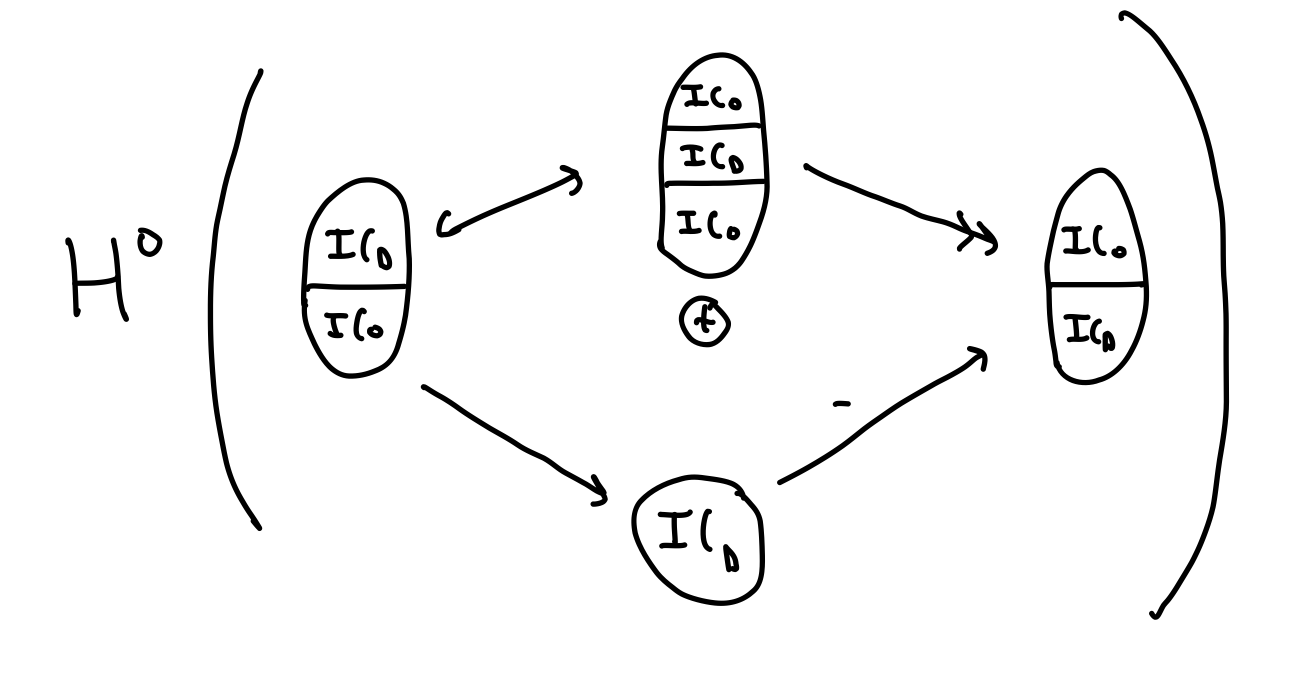}}}.
\]
Here, the kernel and image are equal:
\[
\includegraphics[scale=0.4]{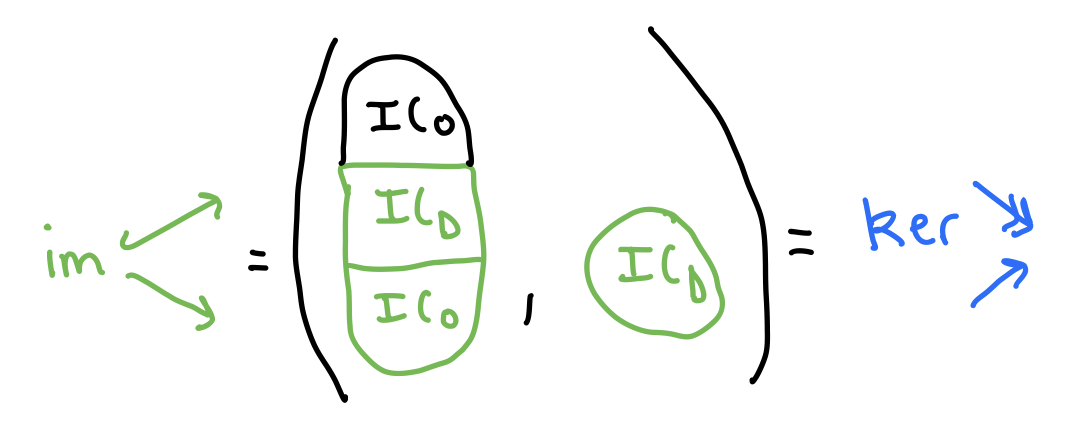}
\]
Notice that the extra piece of the kernel in the previous example is now a subsheaf of the image:
\[
\includegraphics[scale=0.4]{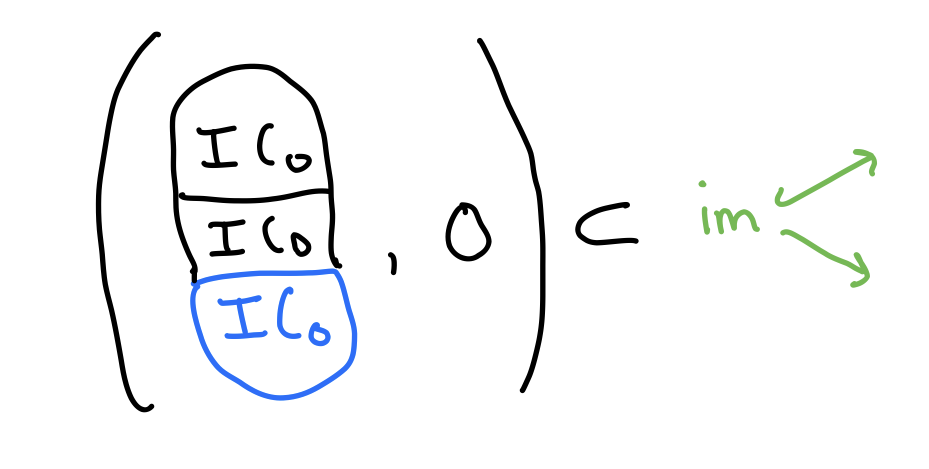}
\]
Hence $\phi_f(\mc{M})=0$.
\end{example}
\begin{exercise}
Show that 
\[
\phi_f(\Xi_f(k_{D^\times}[1]))=IC_0^{\oplus 2},
\]
and show that the monodromy endomorphism is not trivial.
\end{exercise}
\begin{example}
Let $\mc{M}=IC_X$, for $X=\C^2$, and $\C^2 \xrightarrow[xy]{f} \C
$ as earlier. Then 
\[
\phi_f(\mc{M}) = \vcenter{\hbox{\includegraphics[scale=0.4]{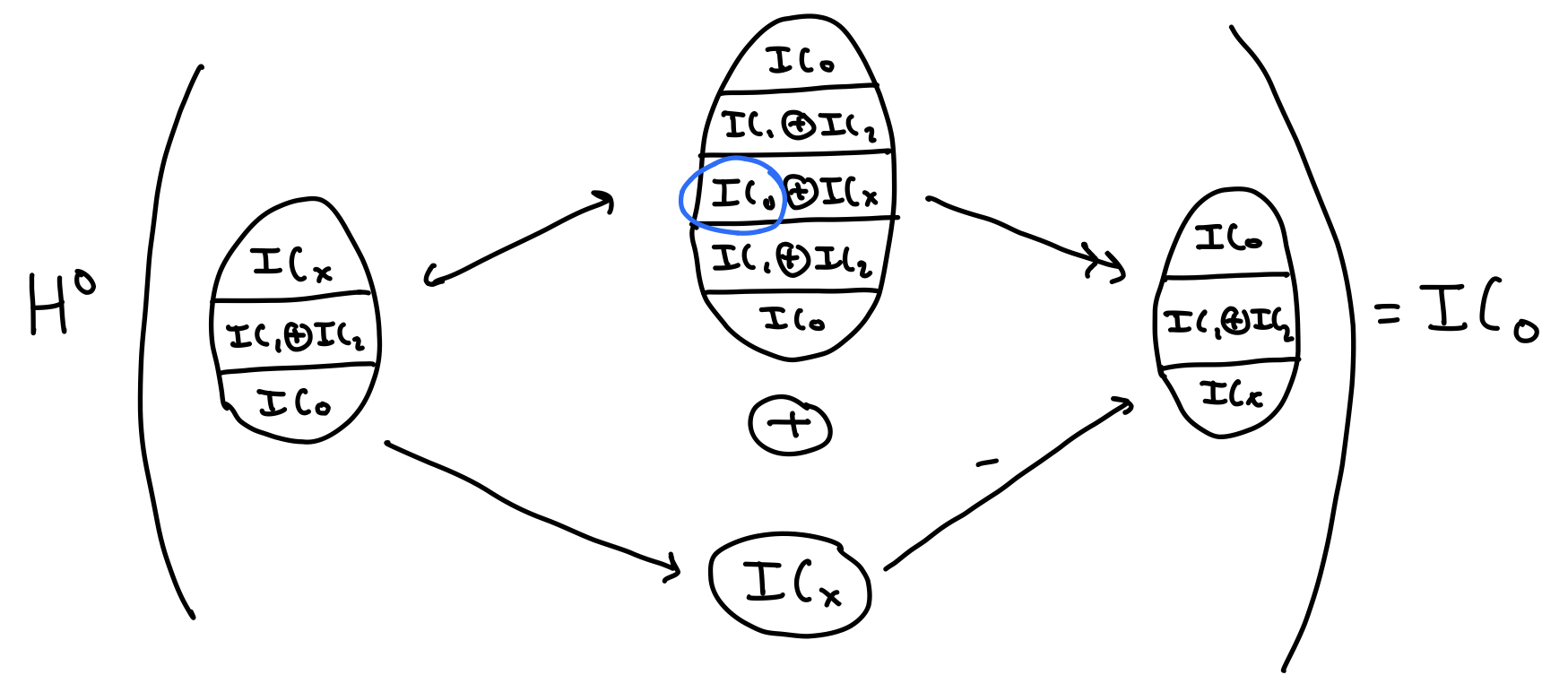}}}
\]
\end{example}
\subsection{Gluing}
With this machinery, Beilinson shows us how we can glue perverse sheaves. To state the theorem we return to a slightly more general setting. 
\[
\begin{tikzcd}
Z \arrow[r, "i", hookrightarrow] \arrow[d] & X  \arrow[d, "f"] & U \arrow[l, "j", hookrightarrow]\arrow[d] \\
\{0\} \arrow[r, hookrightarrow] & D &  D^\times \arrow[l, hookrightarrow]
\end{tikzcd}
\]
\begin{theorem}
(Beilinson) 
\[
\Perv(X) \simeq \left\{ \begin{array}{c} \mc{M}_{U} \in \Perv(U), \\ \mc{M}_{Z} \in \Perv(Z) \end{array} + \begin{tikzcd} \psi_f(\mc{M}_{U}) \arrow[rr, "\text{monodromy -1}"] \arrow[dr] & &\psi_f(\mc{M}_{U}) \\
& \mc{M}_{Z} \arrow[ur] & \end{tikzcd} \right\} 
\]
\[
\hspace{10mm} \text{``pieces''} \hspace{35mm} \text{``glue''}
\]
\end{theorem}
\begin{remark}
Beilinson's theorem glues perverse sheaves on an open set to perverse sheaves on a closed set. Gluing perverse sheaves on two open sets is much easier because perverse sheaves form a stack. 
\end{remark}
\begin{example} Returning to our setting of $X=D=D^\times \sqcup \{0\}$, a perverse sheaf on $D^\times$ (resp. $\{0\}$) is the same thing as a vector space with automorphism $\mu \circlearrowright V$ (resp. a vector space $W$). Hence
\[
\Perv_{\begin{array}{c}\text{constructible} \\ \text{w.r.t. }D^\times, \{0\} \end{array}}(D) = \left\{ \begin{array}{c} V \circlearrowleft \mu \text{ invertible,} \\ W \text{ vector space } \end{array} + \begin{tikzcd} V \arrow[dr, "f"] \arrow[rr, "\mu - id"] & & V \\ & W \arrow[ur, "g"] & \end{tikzcd} \right\}. 
\]
\end{example}

Let us explain why Beilinson gluing holds. We start with a general definition. Let $\mc{A}$ be an abelian category. A {\bf diad} in $\mc{A}$ is 
\[
Q:= \begin{tikzcd}
& A \arrow[rd, "\alpha_+", twoheadrightarrow] & \\
C_- \arrow[dr, "\beta_-"']\arrow[ur, "\alpha_-", hookrightarrow] & & C_+ \\
& B \arrow[ur, "\beta_+"'] & 
\end{tikzcd}
\]
These form a category $\Diads(\mc{A})$ in an obvious way. 

Given a diad, we can associate a complex
\[
Q^\cdot = C_- \xrightarrow{(\alpha_-, \beta_-)} A \oplus B \xrightarrow{(\alpha_+, -\beta_+)} C_+,
\]
and another diad 
\[
r(Q):= \begin{tikzcd}
& A \arrow[rd, twoheadrightarrow] & \\
\ker{\alpha_+} \arrow[dr]\arrow[ur,  hookrightarrow] & & \coker{\alpha_-}\\
& H^0(Q^\cdot) \arrow[ur] & 
\end{tikzcd}
\]
\begin{lemma} $r^2=id$
\end{lemma}
\begin{proof}
We will take this lemma as a black box.
\end{proof}
Given this, 
\begin{align*}
    \Perv(X) &\simeq \text{diads of the form }  \begin{tikzcd}[ampersand replacement=\&]
\& \Xi_f(\mc{M}) \arrow[rd] \& \\
j_!\mc{M} \arrow[dr] \arrow[ur] \& \oplus \& j_*\mc{M}\\
\& \mc{M} \arrow[ur] \& 
\end{tikzcd}\\
&\simeq \text{ diads of the form } \begin{tikzcd}[ampersand replacement=\&]
\& \Xi_f(\mc{M}) \arrow[rd] \& \\
\psi_f(\mc{M}) \arrow[dr] \arrow[ur] \& \oplus \& \psi_f(\mc{M})\\
\& \phi_f(\mc{M}) \arrow[ur] \& 
\end{tikzcd} 
\\
&\simeq \text{ pairs}  \left( \mc{M}_U, \psi_f(\mc{M}) \rightarrow \phi_f(\mc{M}) \rightarrow \psi_f(\mc{M}) \mid \psi_f(\mc{M}) \xrightarrow{\text{monodromy -1}} \psi_f(\mc{M}) \right) \qed
\end{align*}
The second equality is obtained via applying our equivalence $r$ on diads.

%% file: lecture-19.tex
\section{Lecture 19: The derived category of perverse sheaves}
\label{lecture 19}

\subsection{Overview of Beilinson gluing}
Recall the setting of the last lecture: Let $X/\C$ be a variety with the metric topology and $P_X:=\Perv(X,k)$. Beilinson described for us how to glue perverse sheaves on $X$. More specifically, we have the following set-up:
\[
\begin{tikzcd}
Z \arrow[r, "i", hookrightarrow] \arrow[d] & X \arrow[d, "f"] & U \arrow[d] \arrow[l, "j"', hookrightarrow] \\
\{0\} \arrow[r, hookrightarrow] & \mathbb{A}^1 & \mathbb{A}^1 \backslash \{0\} \arrow[l, hookrightarrow]
\end{tikzcd}
\]
We can move between the bounded derived categories of constructible sheaves on $X$, $Z$, and $U$ by pushing and pulling:
\[
\begin{tikzcd}[row sep=large, column sep = large]
D_c^b(Z) \arrow[r, "i_*=i_!" description]  & D_c^b(X) \arrow[l, "i^*"', shift right=2] \arrow[l, "i^!", shift left=2] \arrow[r, "j^*=j^!" description] & D_c^b(U) \arrow[l, "j_*"', shift right=2] \arrow[l, "j_!", shift left=2] \\
P_Z \arrow[r, "i_*=i_!" description] \arrow[u, hookrightarrow] & P_X \arrow[u, hookrightarrow] \arrow[r, "j^*=j^!" description] & P_U \arrow[u, hookrightarrow] \arrow[l, "j_* \text{ ($j$ affine)}"', shift right=2], \arrow[l, "j_! \text{ ($j$ affine)}", shift left=2] 
\end{tikzcd}
\]
The functors $i_*=i_!, j^* =j^!$ preserve perverse sheaves, and if $j$ is affine, then $j_*, j_!$ are exact and also preserve perverse sheaves. However, $i^*, i^!$ do not! 

Last lecture, we constructed two other exact functors
\begin{align*}
    x \circlearrowright \psi_f:P_U &\rightarrow P_Z \text{ ``unipotent nearby cycles''} \\
    x \circlearrowright\Xi_f: P_U &\rightarrow P_X \text{ ``maximal extension''} 
\end{align*}
which each have a ``monodromy'' $x$. These functors fit into short exact sequences 
\begin{align*}
    0 &\rightarrow j_!\mc{M} \rightarrow \Xi_f \mc{M} \rightarrow \psi_f \mc{M} \rightarrow 0 \\ 
    0 &\rightarrow \psi_f \mc{M} \rightarrow \Xi_f\mc{M} \rightarrow j_*\mc{M} \rightarrow 0
\end{align*}
for any $\mc{M} \in P_U$. We then used $\psi_f$ and $\Xi_f$ to construct another exact functor
\[
x \circlearrowright \phi_f:P_X \rightarrow P_Z \text{ ``unipotent vanishing cycles''} 
\]
which we should think about as the ``stalk of a perverse sheaf along $Z$''. It was defined as 
\[
\phi_f\mc{M} = H^0 \left( \begin{tikzcd} & \Xi_f\mc{M}_U \arrow[rd] & \\
j_!\mc{M}_U \arrow[ur] \arrow[dr, "adj"'] & \oplus & j_*\mc{M}_U \\
& \mc{M} \arrow[ur, "-adj"'] & \end{tikzcd} \right)
\]
for $M \in P_X$. 

\vspace{3mm}
\noindent
{\bf Easy (and useful) fact}: For $\mc{M} \in P_Z$, \[
\phi_f(i_*\mc{M})\simeq \mc{M}.
\]
The main result of last lecture was {\bf Beilinson gluing}, which gave an equivalence of categories: 
\begin{align*}
    P_X &\xrightarrow{\sim} \left\{ \begin{array}{c} \mc{M} \in P_U \\ \mc{N} \in P_Z \end{array} \left| \begin{tikzcd}[ampersand replacement=\&] \psi_f \mc{M} \arrow[rr, "1-\mu"] \arrow[rd] \& \& \psi_f \mc{M} \\
    \& \mc{N} \arrow[ru]\& \end{tikzcd} \right. \right\}\\
    \mc{F} &\mapsto \left\{ \begin{array}{c} \mc{F}_U \in P_U \\ \phi_f \mc{F} \in P_Z \end{array} \left| \begin{tikzcd}[ampersand replacement=\&] \psi_f \mc{F}_U \arrow[rr] \arrow[rd] \& \& \psi_f \mc{F}_U \\
    \& \phi_f\mc{F} \arrow[ru]\& \end{tikzcd} \right. \right\} 
\end{align*}

\subsection{A theorem of Beilinson}
The goal of today's lecture is to use the tools above to prove the following theorem. 
\begin{theorem}
$D^b(P_X) \simeq D_c^b(X)$.
\end{theorem}
This result is a beautiful example of a bigger philosophy. Often we have a triangulated category we wish to understand, and within it a collection of well-behaved objects which we understand better than the rest. Then we might hope to reconstruct the entire triangulated category from our special collection of objects. In many situations we cannot, but this is an example where we can. In the poetry of Beilinson, ``the niche $D$ where $P_X$ dwells may be recovered from $P_X$''. 

\vspace{3mm}
\noindent
{\bf Recall}: Let $\Lambda$ be a stratification of $X$. Associated to this stratification is the category $\Perv_\Lambda(X)$ of $\Lambda$-constructible perverse sheaves on $X$. Our category $P_X$ is the limit of such categories:
\[
P_X = \Perv(X) = \lim_{\rightarrow \atop \Lambda} \Perv_\Lambda(X).
\]
\begin{remark}
If $\Lambda$ is a fixed stratification, usually
\[
D^b(\Perv_\Lambda(X)) \not\simeq D_\Lambda^b(X). 
\]
\end{remark}
\begin{example}
\label{not K pi one}
Let $X= \PP^1 \C$ with the trivial stratification $\Lambda$. Then 
\[
\Perv_\Lambda(X) = \text{ local systems on $X$ } \simeq \Vect_k^{f.d.}.
\]
However, 
\[
D^b(\Vect_k^{f.d.}) \not \simeq D^b_\Lambda(X)
\]
because $\Ext^i(k_X, k_X) = H^i(X) \neq 0 $ for $i=2$, but the category $\Vect_k^{f.d.}$ is semisimple, so there are no higher exts.  
\end{example}

\subsection{$K(\pi, 1)$ spaces}
We see in this example that differences in Ext groups prevented us from obtaining our desired derived equivalence. This illustrates a more general phenomenon which we will examine now. 

\begin{definition} 
A nice, path connected space $X$ with base point $x \in X$ is $\bm{K(\pi, 1)}$ if 
\begin{enumerate}
    \item $\pi_i(X)=1 $ for $i>1$, and  
    \item $\pi_1(X) = \pi$.
\end{enumerate}
Equivalently, $X$ is $K(\pi, 1)$ if its universal cover is contractible. 
\end{definition}
\begin{exercise}
Show (easier) that 
\begin{align*}
\Rep^{f.d.}k \pi &\xrightarrow{\sim} \text{ $k$-local systems on $X$} \\
V &\mapsto \mc{L}_V
\end{align*}
and (harder) that
\[
\Ext^i(V, V') \simeq \Ext^i(\mc{L}_V, \mc{L}_{V'}). 
\]
\end{exercise}

\begin{lemma}
\label{exts agree}
Let $F:\mc{D}_1 \rightarrow \mc{D}_2$ be a triangulated functor of triangulated categories. Assume that $\mc{D}_1, \mc{D}_2$ have $t$-structures with hearts $\mc{C}_1$, $\mc{C}_2$, respectively, such that $F:\mc{C}_1 \xrightarrow{\sim} \mc{C}_2$ is an equivalence, and $\mc{D}_i = \langle \mc{C}_i \rangle_\Delta$ for $i=1,2$ (i.e. the $t$-structures giving $\mc{C}_i$ are nondegenerate). Then the following are equivalent: 
\begin{enumerate}[label=(\alph*)]
    \item $F$ is an equivalence. 
    \item For any objects $M,N \in \mc{C}_1$, $F: \Hom_{\mc{D}_1}^i(M,N) \xrightarrow{\sim} \Hom_{\mc{D}_2}^i(F(M), F(N))$. In other words, the ``Exts agree''.
    \item (Won't be used) Assume that $\mc{D}_1 = D^b(\mc{C}_1)$. For any $x \in \Hom^i_{\mc{D}_2}(F(M), F(N))$, there exists an injection $N \hookrightarrow N'$ such that $x$ is zero in $\Hom_{\mc{D}_2}^i(F(M), F(N'))$; i.e. under the natural map 
    \begin{align*}
        \Hom_{\mc{D}_2}^i(F(M), F(N)) &\rightarrow \Hom_{\mc{D}_2}^i(F(M), F(N')) \\
        x &\mapsto 0.
    \end{align*} This condition is called ``effaceability''.
\end{enumerate}
\end{lemma}

\begin{proof}
(a) $\implies$ (b) is immediate. \\
(b) $\implies$ (a) can be shown by induction and the long exact sequence. 
\end{proof}

\begin{example}
Let $X$ be $K(\pi, 1)$. Then 
\[
D^b(\Rep^{f.d.} k\pi) \xrightarrow{\sim} D_\Lambda^b(X).
\]
\end{example}

\noindent 
Example \ref{not K pi one} illustrated that this is not necessarily the case when $X$ is not $K(\pi, 1)$. 

\subsection{Yoneda extensions} 
To understand part (c) of Lemma \ref{exts agree}, we need some facts about Yoneda extensions. Let $\mc{A}$ be an abelian category. For objects $M,N$ in $\mc{A}$, define a category $E^i(M,N)$ with
\begin{itemize}
    \item objects: acyclic complexes 
    \[
    N \rightarrow C^1 \rightarrow C^2 \rightarrow \cdots \rightarrow C^i \rightarrow M
    \]
    \item morphisms: chain complex maps of the form
    \[
    \begin{tikzcd}
    N \arrow[r] \arrow[d, "id"] & C^1 \arrow[d] \arrow[r] & C^2 \arrow[d] \arrow[r] & \cdots \arrow[r] & C^i \arrow[r] \arrow[d] & M \arrow[d, "id"] \\
    N \arrow[r] & \widetilde{C}^1 \arrow[r] & \widetilde{C}^2 \arrow[r] & \cdots \arrow[r] & \widetilde{C}^i \arrow[r] & M 
    \end{tikzcd}
    \]
\end{itemize}
\begin{lemma}
\label{definition lemma}(/Definition) 
\begin{align*}
\Ext^i(M,N) &= \text{ connected components of $E^i(M,N)$} \\
0 &\leftrightarrow \text{ split complexes in $E^i(M,N)$}
\end{align*}
\begin{remark} In Lemma \ref{definition lemma}, two extensions are in the same ``connected component'' if one can pass from one to the other going along arrows in our categories \emph{in either direction}. More formally, we can build a space with 0 (resp. 1) simplices given by objects (resp. arrows) in $E^i(M,N)$; then connected component takes on its topological meaning. One can show two objects  
\[
    N \rightarrow C^1 \rightarrow C^2 \rightarrow \cdots \rightarrow C^i \rightarrow M \text{ and }     N \rightarrow \widetilde{C}^1 \rightarrow \widetilde{C}^2 \rightarrow \cdots \rightarrow \widetilde{C}^i \rightarrow M
\]
in $E^i(M,N)$ are in the same connected component if there exists a commutative  diagram 
    \[
    \begin{tikzcd}
    N \arrow[r] & C^1 \arrow[r] & C^2 \arrow[r] & \cdots \arrow[r] & C^i \arrow[r] & M  \\
    N \arrow[r] \arrow[d, "id"'] \arrow[u, "id"] & D^1 \arrow[r] \arrow[u] \arrow[d] & D^2 \arrow[d] \arrow[u] \arrow[r] & \cdots \arrow[r] & D^i \arrow[r] \arrow[d] \arrow[u] & M \arrow[d, "id"'] \arrow[u, "id"] \\
    N \arrow[r] & \widetilde{C}^1 \arrow[r] & \widetilde{C}^2 \arrow[r] & \cdots \arrow[r] & \widetilde{C}^i \arrow[r] & M 
    \end{tikzcd}
    \]
with 
\[
N \rightarrow D^1 \rightarrow D^2 \rightarrow \cdots \rightarrow D^i \rightarrow M
\]
also in $E^i(M,N)$. 
\end{remark}
\end{lemma}
\begin{exercise}
Show directly that with $\Ext^i$ defined as in Lemma \ref{definition lemma}, 
\[
\Ext^i(M,N) = \Hom_{D(\mc{A})}^i(M,N).
\]
Make no assumptions on $\mc{A}$.
\end{exercise}

\vspace{3mm}
\noindent
What about functoriality? In our usual definition of $\Ext^i$ in the derived category, it is obvious that $\Ext^i$ is a functor. With this new definition, it is not so obvious. In other words, given $N \rightarrow N'$, how to we obtain a morphism $\Ext^i(M,N) \rightarrow \Ext^i(M, N')$? Here is how: 
\begin{itemize}
    \item Given with $N \rightarrow N'$ and an element of $\Ext^i(M,N)$ represented by the complex 
    \[
     N \rightarrow C^1 \rightarrow C^2 \rightarrow \cdots \rightarrow C^i \rightarrow M,
    \]
    we obtain a complex representing an element of $\Ext^i(M, N')$ by forming the push-out
    \[
    \begin{tikzcd}
    N \arrow[r] \arrow[d] & C^1 \arrow[d] \arrow[r] & C^2 \arrow[d, equal] \arrow[r] & \cdots \arrow[r] & C^i \arrow[r] \arrow[d, equal] & M \arrow[d, equal] \\
    N' \arrow[r] & P \arrow[r] & C^2 \arrow[r] & \cdots \arrow[r] & C^i \arrow[r] & M 
    \end{tikzcd}
    \]
    Here $P=(C^1 \oplus N')/N$. This gives a morphism $\Ext^i(M,N) \rightarrow \Ext^i(M, N')$. 
    \item Similarly, for $M \rightarrow M'$, we can use the pull-back to form a morphism $\Ext^i(M,N) \rightarrow \Ext^i(M', N)$. 
\end{itemize}

An important consequence of this is the following. If $x \in \Ext^i(M,N)$ is represented by the complex \[
N \xrightarrow{f} C^1 \rightarrow C^2 \rightarrow \cdots \rightarrow C^i \rightarrow M,
\]
then applying the construction above to the morphism $f:N \rightarrow C^1$, we obtain
    \[
    \begin{tikzcd}
    N \arrow[r] \arrow[d] & C^1 \arrow[d] \arrow[r] & C^2 \arrow[d, equal] \arrow[r] & \cdots \arrow[r] & C^i \arrow[r] \arrow[d, equal] & M \arrow[d, equal] \\
    C^1 \arrow[r] & (C^1 \oplus C^1)/N \arrow[r] & C^2 \arrow[r] & \cdots \arrow[r] & C^i \arrow[r] & M 
    \end{tikzcd}
    \]
    But since $(C^1 \oplus C^1)/N=C^1 \oplus (C^1/N)$, the lower sequence splits, and hence represents zero in $\Ext^i(M,C^1)$. In other words, 
    \begin{align*}
        \Ext^i(M,N) &\xrightarrow{\Ext^i(M, f)} \Ext^i(M, C^1) \\
        x &\longmapsto 0.
    \end{align*}
    This is the origin of effaceability. 
    
    \begin{exercise}
    Prove (ii) $\implies$ (iii) and (iii) $\implies$ (ii) in Lemma \ref{exts agree}. (Hint: See Geordie's hand-written notes if you are stuck.) 
    \end{exercise}
    
    \subsection{A key ingredient}
    A key ingredient in Beilinson's theorem is:
    
    \begin{theorem}
    \label{open k pi one set}
    Let $X/\C$ be a smooth variety. Then there exists a Zariski open set $U \subset X$ which is $K(\pi, 1)$. 
    \end{theorem}
    
    This theorem explains philosophically why Beilinson's theorem is true: for perverse sheaves, we can always refine our stratification to include this $K(\pi, 1)$ open set. We will see this more precisely when we discuss the proof of Beilinson's theorem. 
    
    \begin{example}
    Any curve becomes $K(\pi, 1)$ after deleting a point! 
    \[
    \includegraphics[scale=0.6]{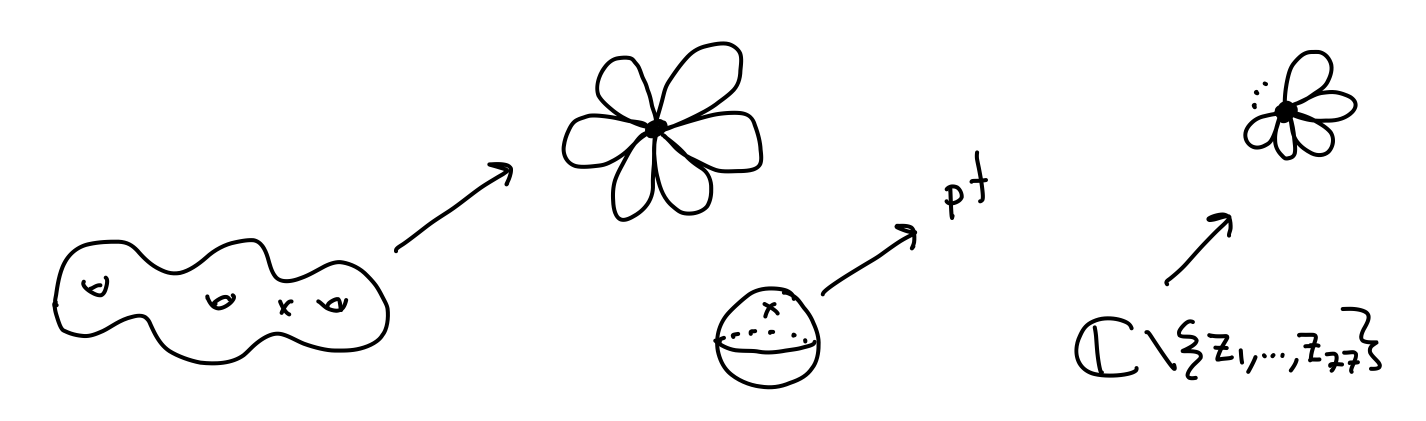}
    \]
    (Here the arrows indicate homotopy equivalences and the $x$'s are missing points.) 
    \end{example}
    We will roughly explain the proof of Theorem \ref{open k pi one set}. There are two main ingredients:
    \begin{enumerate}
        \item ``Noether normalization,'' {\bf (NN)}: If $Z \subset \mathbb{A}^n$ is of dimension $d$, then a generic projection $\mathbb{A}^n \rightarrow \mathbb{A}^d$ is finite when restricted to $Z$. Here's a caricature:
            \[
    \includegraphics[scale=0.5]{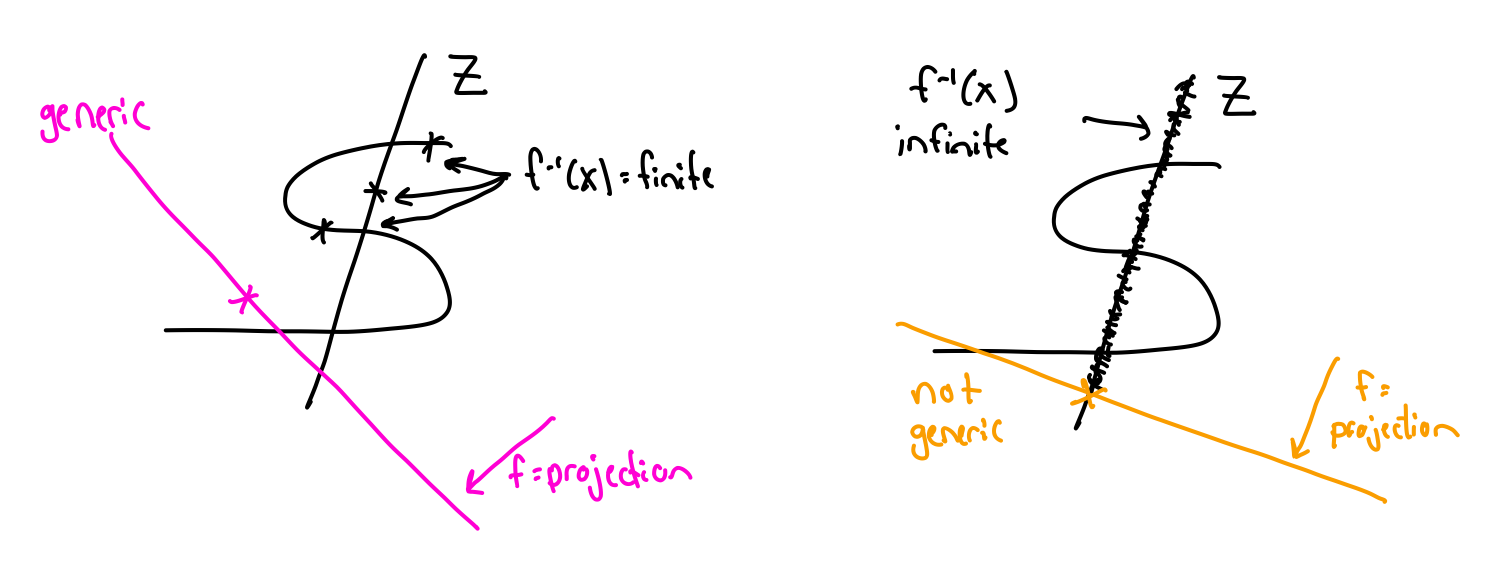}
    \]
        \item ``Extensions of $K(\pi,1)$ are $K(\pi, 1)$,'' ($\bm{EK(\pi, 1)}$): If 
        \[
        \begin{tikzcd} F \arrow[r, hookrightarrow] & E \arrow[d] \\ & B \end{tikzcd}
        \]
        is a fibre bundle, then if $B$ is $K(\pi, 1)$, $E$ is connected, and either $F$ is discrete or $K(\pi, 1)$, then $E$ is $K(\pi, 1)$. (Follows from the long exact sequence of homotopy groups.)
    \end{enumerate}
    
    With these tools in our toolbox, the proof proceeds as follows. 
    
    \vspace{3mm}
    \noindent
    {\bf Step 1:} Reduction to $X \subset \mathbb{A}^1$, Zariski standard open\footnote{``Zariski standard open'' means that $X$ is of the form $D(f)=\{z \in X \mid f(z)\neq 0\}$ for some regular function $f$.}. Without loss of generality, we may assume that $X$ is affine. By {\bf(NN)}, we can find a finite map $f:X \rightarrow \mathbb{A}^d$. Over a standard open $U' \subset \mathbb{A}^d$, this map is \'{e}tale. Hence we have a fibration
    \[
    \begin{tikzcd}
    \text{finite} \arrow[r,hookrightarrow]& U \arrow[d]\\
    & U' \end{tikzcd}
    \]
    so by ($\bm{EK(\pi, 1)}$), if $U'$ is $K(\pi, 1)$, then $U$ is $K(\pi, 1)$.
    
    \vspace{3mm}
    \noindent
    {\bf Step 2:} Induction on dimension. By {\bf(NN)}, there exists a projection $f:\mathbb{A}^n \rightarrow \mathbb{A}^{n-1}$ such that on the hypersurface $Z=X \backslash U$, $f$ is finite with fibres consisting of $m$ points:
    \[
\begin{tikzcd}
U \arrow[r, hookrightarrow] & \mathbb{A}^n \arrow[d, "f"] & Z \arrow[l, hookrightarrow] \arrow[d, "f"] & \{m \text{ points}\} \arrow[l, hookrightarrow]\\
& \mathbb{A}^{n-1} & \mathbb{A}^{n-1} \arrow[l, equal]& 
\end{tikzcd}
\]
Over an open $V \subset \mathbb{A}^{n-1}$, this map will be \'{e}tale, hence 
\[
\begin{tikzcd}
U' \arrow[d] \arrow[r, hookrightarrow] & U \arrow[d] \\ V' \arrow[r, hookrightarrow] & V 
\end{tikzcd}
\]
is fibred in $\C \backslash \{m \text{ points}\}$. By induction, we can shrink $V$ to $V'$ such that $V'$ is $K(\pi, 1)$. Hence $U'$ is $K(\pi, 1)$ too by ($\bm{EK(\pi, 1)}$). \qed 

\subsection{Proof of Beilinson's theorem for curves}

The goal for the rest of the lecture is to prove Beilinson's theorem for curves; that is, for a curve $X$, we wish to show that
\[
D^b(P_X) \simeq D^b_c(X).
\]
To start, we have the following fact \cite{BBD, Beilinson}. There exists a triangulated functor 
\[
\mathrm{real}: D^b(P_X) \rightarrow D^b_c(X) 
\]
called the ``realization functor'' which is the identity on $P_X$. (This holds for general $X$, not just curves.) We will show that $\mathrm{real}$ is an equivalence of categories. 

By Lemma \ref{exts agree}, it is enough to show that for $M,N \in P_X$, 
\[
\Hom^i_{D^b(P_X)}(M, N) \simeq \Hom^i_{D^b_c(X)}(M,N).
\]
Moreover, because $P_X$ is finite length, we can use the long exact sequence in Ext to reduce to the case where $M,N$ are irreducible objects. 

Recall that in lecture \ref{lecture 17} we classified irreducible perverse sheaves on curves. There were two types: (1) skyscrapers, and (2) $IC$ sheaves supported everywhere.

\vspace{3mm}
\noindent
{\bf Step 1:} Let $M,N$ be skyscrapers: $M=i_*k_x$, $N=i'_*k_y$. Then in $D^b_c(X)$, 
\begin{align*}
    \Hom^j(i_*k_x, i'_*k_y) &= \Hom^j(k_x, i^!i'_*k_y) \\
    &= \begin{cases}
    k & \text{ if }x=y, j=0, \\
    0 & \text{ otherwise}. \end{cases} 
\end{align*}
Showing that 
\[
\Hom_{D^b(P_X)}(k_x, k_y)=0
\]
if $x \neq y$ is easier, and left as an exercise. The tricker case is show that $\Ext^i_{P_X}(k_x, k_x)=0$ for $i>0$. To do this, we will use the quiver description of $P_X$ and Yoneda Exts. Recall that
\[
P_X \simeq \left\{ \left. \begin{tikzcd} V_1 \arrow[r, "c", shift left] & V_0 \arrow[l, "v", shift left] \end{tikzcd} \right| c \circ v + id \text{ invertible } \right\}.  
\]
Then an element of $\Ext^1(k_x, k_y)$ has a representative of the form 
\[
k_k \rightarrow F \rightarrow k_x.
\]
In quiver language, such an extension is a complex
\[
k \rightarrow V_0 \rightarrow k.
\]
But any such complex fits into a diagram 
\[
\begin{tikzcd}
0 \arrow[r] \arrow[d, shift left] & V_1 \arrow[r] \arrow[d, shift left, "c"] & 0 \arrow[d, shift left] \\
k \arrow[u, shift left] \arrow[r] & V_0 \arrow[u, shift left, "v"] \arrow[r] & k \arrow[u, shift left]
\end{tikzcd}.
\]
The top sequence is exact, so $V_1=0$. But the bottom sequence is also exact, so $V_1=0$ implies that the bottom sequence must split. Hence our original element is zero in $\Ext^1(k_x, k_y)$, so $\Ext^1(k_x, k_y)=0$. 

Showing that $\Ext^2(k_x, k_y)=0$ is more challenging. An element of $\Ext^2(k_x, k_y)$ has a representative of the form
\[
k \rightarrow V_0^1 \rightarrow V_0^2 \rightarrow k.
\]
This fits into a diagram 
\[
\begin{tikzcd}
0 \arrow[r] \arrow[d, shift left] & V_1^1 \arrow[r] \arrow[d, shift left, "c^1"] & V_1^2 \arrow[r] \arrow[d, shift left, "c^2"] & 0 \arrow[d, shift left]\\
k \arrow[u, shift left] \arrow[r] & V_0^1 \arrow[r] \arrow[u, shift left, "v^1"] & V_0^2 \arrow[r] \arrow[u, shift left, "v^2"] & k \arrow[u, shift left]
\end{tikzcd}
\]
We want to show that we can reduce this to an $\Ext$ supported at $x$. But the basic issue is that there is nothing telling us that the middle terms must be supported on a point. In sheaf language, this is equivalent to the fact that the existence of a sequence 
\[
i_*k_x \rightarrow \mc{F} \rightarrow \mc{G} \rightarrow i_*k_x 
\]
whose first and last terms are supported on a point does not imply that the middle terms $\mc{F}, \mc{G}$ are.

There are two ways we can get around this. First, the ``stupid way''. We can assume that $c \circ v$ is nilpotent (exercise!), then we have a natural map 
\[
\begin{tikzcd}
0 \arrow[r] \arrow[d, shift left] & V_1^1 \arrow[r] \arrow[d, shift left, "c^1"] & V_1^2 \arrow[r] \arrow[d, shift left, "c^2"] & 0 \arrow[d, shift left]\\
k \arrow[u, shift left] \arrow[r] & V_0^1 \arrow[r] \arrow[u, shift left, "v^1"] & V_0^2 \arrow[r] \arrow[u, shift left, "v^2"] & k \arrow[u, shift left]
\end{tikzcd} \leftarrow 
\begin{tikzcd}
0 \arrow[r] \arrow[d, shift left] & \im v^1 \arrow[r] \arrow[d, shift left, "c'^1"] & \im v^2 \arrow[r] \arrow[d, shift left, "c'^2"] & 0 \arrow[d, shift left]\\
k \arrow[u, shift left] \arrow[r] & V_0^1 \arrow[r] \arrow[u, shift left, "v^1"] & V_0^2 \arrow[r] \arrow[u, shift left, "v^2"] & k \arrow[u, shift left]
\end{tikzcd}.
\]
Continuing this process, we have a map 
\[
\begin{tikzcd}
0 \arrow[r] \arrow[d, shift left] & \im v^1 \arrow[r] \arrow[d, shift left, "c'^1"] & \im v^2 \arrow[r] \arrow[d, shift left, "c'^2"] & 0 \arrow[d, shift left]\\
k \arrow[u, shift left] \arrow[r] & V_0^1 \arrow[r] \arrow[u, shift left, "v^1"] & V_0^2 \arrow[r] \arrow[u, shift left, "v^2"] & k \arrow[u, shift left]
\end{tikzcd}
\leftarrow 
\begin{tikzcd}
0 \arrow[r] \arrow[d, shift left] & \im v^1 \arrow[r] \arrow[d, shift left, "c'^1"] & \im v^2 \arrow[r] \arrow[d, shift left, "c'^2"] & 0 \arrow[d, shift left]\\
k \arrow[u, shift left] \arrow[r] & \im c'^1 \arrow[r] \arrow[u, shift left, "v'^1"] & \im c'^2 \arrow[r] \arrow[u, shift left, "v'^2"] & k \arrow[u, shift left]
\end{tikzcd}.
\]
Continuing in this way, we eventually obtain a top sequence of zeros, which lets us conclude that our lower sequence splits. This procedure works for all higher Exts. 

Alternatively, we could use Beilinson's approach using vanishing cycles, which works in greater generality (and in fact provides the skeleton of his argument in the general case). Choose a map $f:X \rightarrow \C$ with zero set $x \in X$. Let $C^\cdot$ be the complex
\[
i_*k_x \rightarrow \mc{F}^1 \rightarrow \cdots \rightarrow \mc{F}^j \rightarrow i_{*}k_x
\]
representing an element in $\Ext^j(i_*k_x, i_*k_x)$. Then Beilinson showed that \begin{equation}
    \label{beilinsons fact}
C^\cdot \simeq \phi_f(C^\cdot),
\end{equation}
where $\phi_f:P_X \rightarrow P_{\{x\}}$ is the vanishing cycles functor defined in Lecture \ref{lecture 18}. The complex $\phi_f(C^\cdot)$ is supported on $x$, so this proves that $C^\cdot=0 \in \Ext^j(i_*k_x, i_*k_x)$. To prove (\ref{beilinsons fact}), we can use the exact sequences from Lecture \ref{lecture 18}:
\[
\begin{tikzcd}
C^\cdot \arrow[r] & C^\cdot \oplus \Xi_f(C^\cdot) \arrow[r] & (C^\cdot \oplus_f(C^\cdot)) / j_!(C^\cdot) \\
& & \phi_f(C^\cdot) \arrow[u] 
\end{tikzcd}
\]
\vspace{3mm}
\noindent
{\bf Step 2:} Let $M$ be irreducible with full support and $N$ a skyscraper supported at $\{x\}$. Let 
\[
j:U=X \backslash \{x\} \hookrightarrow X.
\]
Recall that because the immersion $j$ is affine, $j_!$ and $j_*$ are exact, preserve perverse sheaves, and fit into adjoint pairs $(j^*, j_*), (j_!, j^!)$. The adjunction maps give an exact sequence 
\[
K \hookrightarrow j_!M_U \twoheadrightarrow M,
\]
where $K$ is a perverse sheaf supported on $\{x\}$. The corresponding long exact sequence gives the diagram
\[
\begin{tikzcd}
\cdots \arrow[r] & \Hom^i_{P_X}(M,N) \arrow[d] \arrow[r] & \Hom^i_{P_X}(j_!M_U, N) \arrow[d] \arrow[r] &\Hom^i_{P_X}(K, N) \arrow[d] \arrow[r] &  \cdots \\
\cdots \arrow[r] & \Hom^i_{D^b_c(X)}(M, N) \arrow[r] & \Hom^i_{D^b_c(X)}(j_!M_U, N) \arrow[r] & \Hom^i_{D^b_c(X)}(K, N) \arrow[r] & \cdots 
\end{tikzcd}
\]
Using the adjunctions, we have 
\[
\Hom^i_{P_X}(j_!M_U, N) = \Hom^i_{P_U}(M_U, N_U)=0 
\]
because $N_U=0$. Similarly, $\Hom_{D_c^b(X)}^i(j_!M_U, N)=0$, so the middle vertical arrow is an isomorphism $0 \simeq 0$. The right vertical arrow is an isomorphism by Step 1. Hence we conclude by induction that all other vertical arrows must also be isomorphisms. 

\vspace{3mm}
\noindent
{\bf Step 3:} Let $M$ and $N$ be irreducible objects in $P_X$ with full support. Then by Lemma \ref{open k pi one set}, we can choose $U \subset X$ such that $M_U, N_U$ are (shifts of) local systems and $U$ is $K(\pi, 1)$. Note that the inclusion $j:U \hookrightarrow X$ is still affine. Because $N$ is irreducible, we have a short exact sequence 
\[
N \hookrightarrow j_*j^*N = j_*N_U \twoheadrightarrow k.
\]
Then by the long exact sequence in Hom, we have the diagram
\[
\begin{tikzcd}
\cdots \arrow[r] & \Hom^i_{P_X}(M,N) \arrow[r] \arrow[d] & \Hom^i_{P_X}(M, j_*N_U) \arrow[r] \arrow[d] & \Hom^i_{P_X}(M, k) \arrow[r] \arrow[d] & \cdots \\
\cdots \arrow[r] & \Hom^i_{D^b_c(X)}(M, N) \arrow[r] & \Hom^i_{D^b_c(X)}(M, j_*N_U) \arrow[r] & \Hom_{D^b_c(X)}^i(M, k) \arrow[r] &\cdots
\end{tikzcd}
\]
By adjunction, we have 
\[
\Hom_{P_X}^i(M, j_*N_U)=\Hom^i(M_U, N_U).
\]
Hence the middle vertical arrow is an isomorphism because $U$ is $K(\pi, 1)$. The right vertical arrow is an isomorphism by Step 2, and hence the left vertical arrow must also be an isomorphism.  \qed 

\vspace{5mm}
\noindent
{\bf The Moral:} If we fix a stratification, there can be complicated structure going on in $D^b(P_X)$ which doesn't only come from fundamental groups. However, if we're allowed to refine our stratification as much as we want, then we can choose open strata which are $K(\pi, 1)$, and all information in the category comes from local systems. 

%% file: lecture-20.tex
\section{Lecture 20: Perverse sheaves for affine stratifications}
\label{lecture 20}

Recall the setting of the last few lectures: Fix $k$ to be our field of coefficients, and let $X$ be an algebraic variety over $\C$, with stratification
\[
X = \bigsqcup_{\lambda \in \Lambda} X_\lambda.
\]
The abelian category of $\Lambda$-constructible perverse sheaves, $\Perv_\Lambda(X)$, is a subcategory of the triangulated category $D_\Lambda^b(X)$, the derived category of $\Lambda$-constructible sheaves. Moreover, there is a realization functor 
\[
\mathrm{real}:D^b(\Perv_\Lambda(X))\rightarrow D^b_\Lambda(X).
\]

Last week we saw that for a fixed stratification, $\mathrm{real}$ is rarely an equivalence; however, if we are allowed to refine stratifications, we have Beilinson's theorem:
\begin{theorem}
(Beilinson) $D^b(\Perv(X)) \xrightarrow{\sim} D^b_c(X)$.
\end{theorem}

Today, we will show that if each strata $X_\lambda$ is an affine space, then we have such an equivalence for a fixed stratification.  
\begin{theorem}
(Beilinson--Ginzburg--Soergel) If $\Lambda$ is a stratification of $X$ such that $X_\lambda$ is affine, then $D^b(\Perv_\Lambda(X)) \xrightarrow{\sim} D_\Lambda^b(X).$
\end{theorem}
\begin{remark}
Last week we saw that the failure of the strata to be $K(\pi, 1)$ was the obstruction to $\mathrm{real}$ being an equivalence. Affine spaces are contractible, so this obstruction disappears\footnote{Caution: take this explanation with a grain of salt. The proof of Beilinson's theorem is geometric (via vanishing cycles), and the proof of the BGS theorem is algebraic (via highest weight categories), so a direct comparison doesn't exactly work.}.
\end{remark}

\begin{remark}
Serious homological algebra goes into Beilinson's construction of $\mathrm{real}$. However, if we allow ourselves to use the Riemann--Hilbert correspondence, we have
\[
\begin{tikzcd}
\mc{D}-\mathrm{mod}^{hol, r.s.}_{\Lambda\text{-const}} \arrow[d] \arrow[r, "\sim", leftrightarrow] & \Perv_\Lambda(X)  \\
D^b\left(\mc{D}-\mathrm{mod}^{hol, r.s}_{\Lambda \text{-const}}\right) \arrow[r, "\mathrm{R.H.}"'] & D^b_\Lambda(X) 
\end{tikzcd}.
\]
 Under the equivalence between perverse sheaves and $\mc{D}$-modules given by the top arrow, the bottom arrow in this diagram is $\mathrm{real}$. This gives a high concept (rather than high homological algebra!) construction of the realization functor. 
\end{remark}

The BGS theorem and the techniques involved in proving it will be invaluable as we approach Bezrukavnikov's equivalence. The proof of the BGS theorem is algebraic, using what are today called {\em highest weight categories}. 

\subsection{Highest weight categories}
\begin{definition}
Let $\mc{A}$ be a $k$-linear category. We say that $\mc{A}$ is {\em highest weight} if the following six conditions hold. 
\begin{enumerate}
\item $\mc{A}$ is finite length. 
\item The set $\{ \text{simple objects in }\mc{A} \}/_{\simeq}$ is finite. 
\item For any simple object $L \in \mc{A}$, $\End (L)=k$.

Let $\Lambda$ be an indexing set for isomorphism classes of simple objects and denote by $L_\lambda \in \mc{A}$ the simple object corresponding to $\lambda \in \Lambda$. Assume that $\Lambda$ is a poset (this is part of the data determining a highest weight category). For any closed subset $T \subset \Lambda$ (that is, if $\lambda' \leq \lambda$ and $\lambda \in T$, then $\lambda' \in T$), denote by 
\[
\mc{A}_T = \langle L_\lambda \mid \lambda \in T \rangle_{\mathrm{Serre}}
\]
the Serre subcategory generated by the simple objects $L_\lambda$ with $\lambda \in T$. Assume moreover that for each $\lambda \in \Lambda$, there are objects $\Delta_\lambda$ (``standard'' object), $\nabla_\lambda$ (``costandard'' object) in $\mc{A}$ and maps $\Delta_\lambda \rightarrow L_\lambda$, $L_\lambda \rightarrow \nabla_\lambda$.

\item If $T \subseteq \Lambda$ is closed, then $\Delta_\lambda \rightarrow L_\lambda$ (resp. $L_\lambda \rightarrow \nabla_\lambda$) is a projective cover\footnote{A {\em projective cover} of an object $M$ in a category $\mc{C}$ is a morphism $P \xrightarrow{f} M$ out of a projective object $P \in \mc{C}$ which is a {\em superfluous epimorphism}, meaning that every morphism $N \xrightarrow{g} P$ with the property that $f \circ g$ is an epimorphism is itself an epimorphism.} (resp. injective hull) in $\mc{A}_T$\footnote{Note that we are not requiring $\Delta_\lambda$ (resp. $\nabla_\lambda$) to be projective (resp. injective) objects in $\mc{A}$, just in the Serre subcategory $\mc{A}_T$.}.
\item For $\lambda \in \Lambda$, 
\begin{align*}
    \ker(\Delta_\lambda \rightarrow L_\lambda) &\in \mc{A}_{<\lambda}, \text{ and } \\
    \coker(L_\lambda \rightarrow \nabla_\lambda) &\in \mc{A}_{<\lambda}. 
\end{align*}
This implies that the composition series eggs of $\Delta_\lambda$ and $\nabla_\lambda$ must have the following form 
\[
    \includegraphics[scale=0.3]{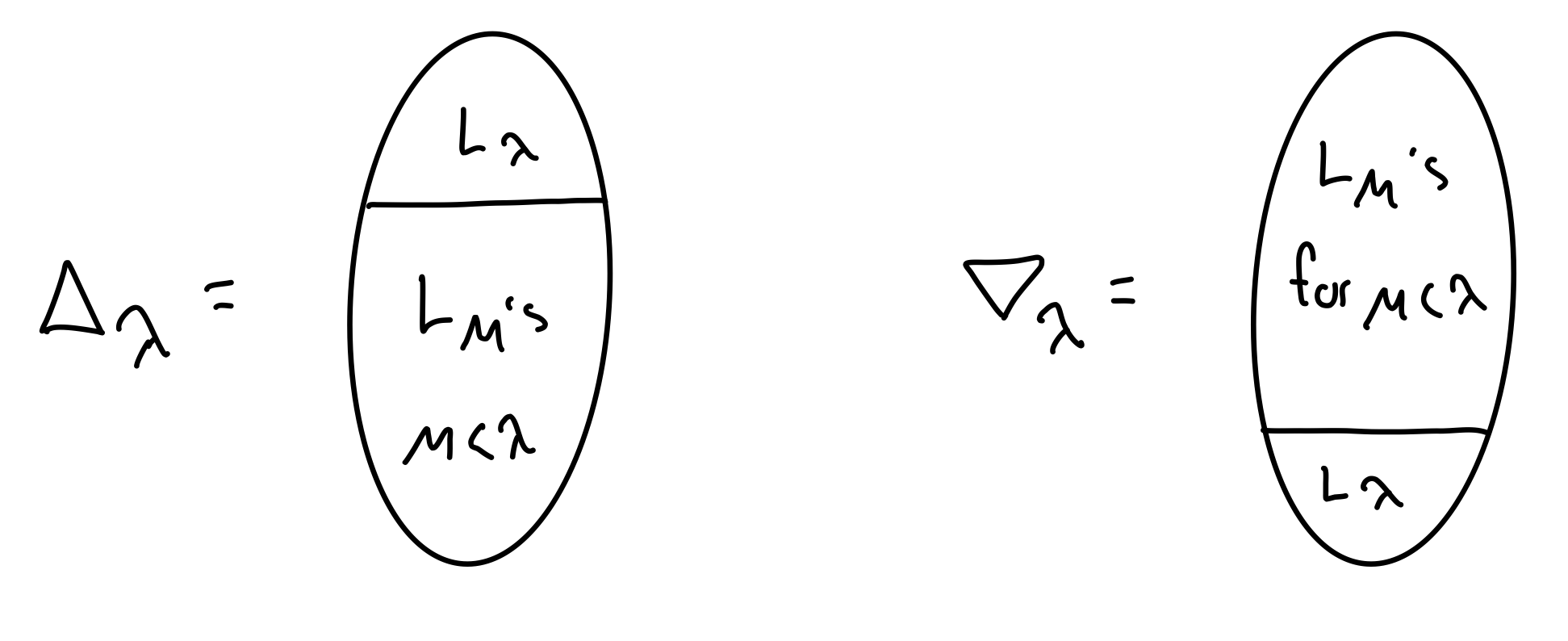}
\]
\item $\Ext^2(\Delta_\lambda, \nabla_\mu)=0$ for all $\lambda, \mu \in \Lambda$.\footnote{This condition is the most mysterious, and often the hardest to show.}
\end{enumerate}
\end{definition}
\begin{theorem}
\label{BGS}
(Beilinson--Ginzburg--Soergel) Let $\mc{A}$ be highest weight. Then $\mc{A}$ has enough projective and injective objects. Moreover, the projective cover $P_\lambda$ of $L_\lambda$ has a standard filtration of the form 
\[
\includegraphics[scale=0.2]{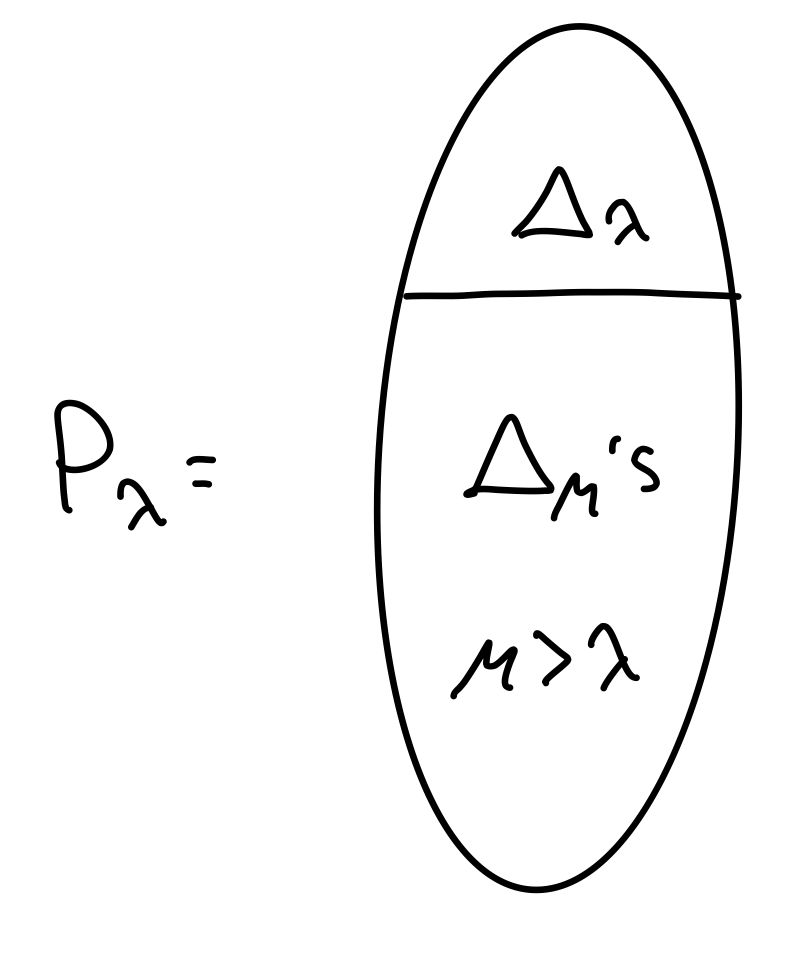}
\]
Similarly, the injective hull $I_\lambda$ has a costandard filtration of the form 
\[
\includegraphics[scale=0.2]{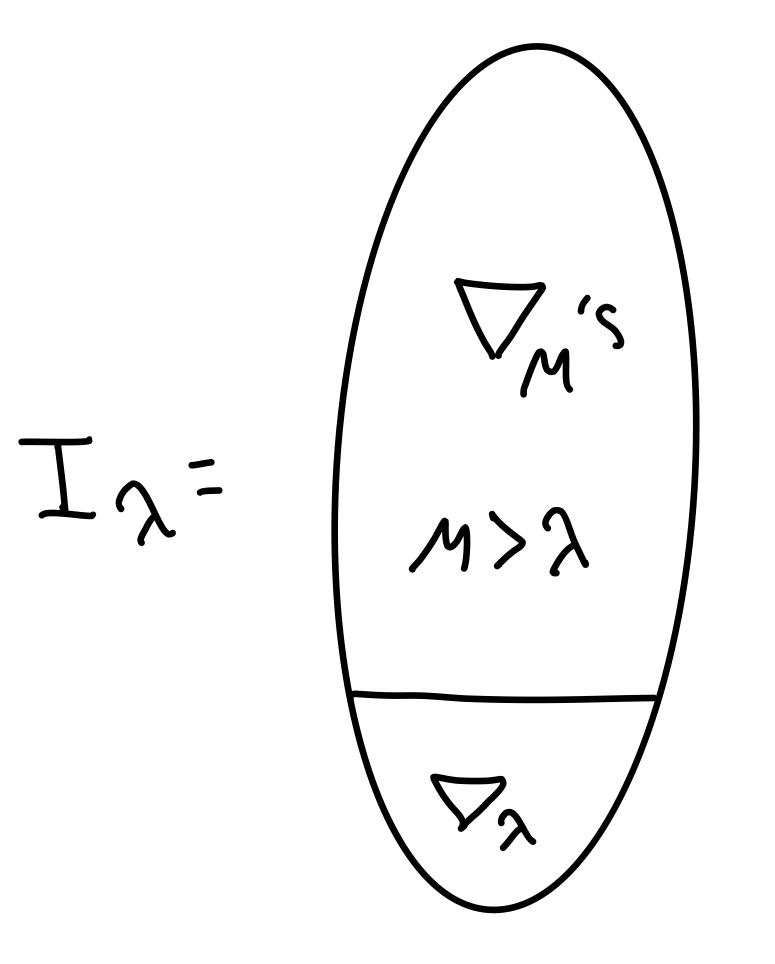}
\]
\end{theorem}
\begin{proof}
For full details, see \cite[Theorem 3.2.1]{BGS}. 

\vspace{2mm}
\noindent
{\bf Sketch:} Instead of proving the theorem as stated, we prove that for any closed subset $T \subseteq \Lambda$, $\mc{A}_T$ has enough projectives (resp. injectives), and they admit standard (resp. costandard) filtrations. We construct the projective cover $P_\lambda^T$ of $L_\lambda$ inductively on $T$:
\begin{itemize}
    \item Let $\mu \in T$ be maximal, and $\lambda \neq \mu \in T$. Assume that the projective cover $P_\lambda^{T \backslash \{\mu\}}$ of $L_\lambda$ in $\mc{A}_{T \backslash \{\mu\}}$ is already constructed.
    \[
    \includegraphics[scale=0.3]{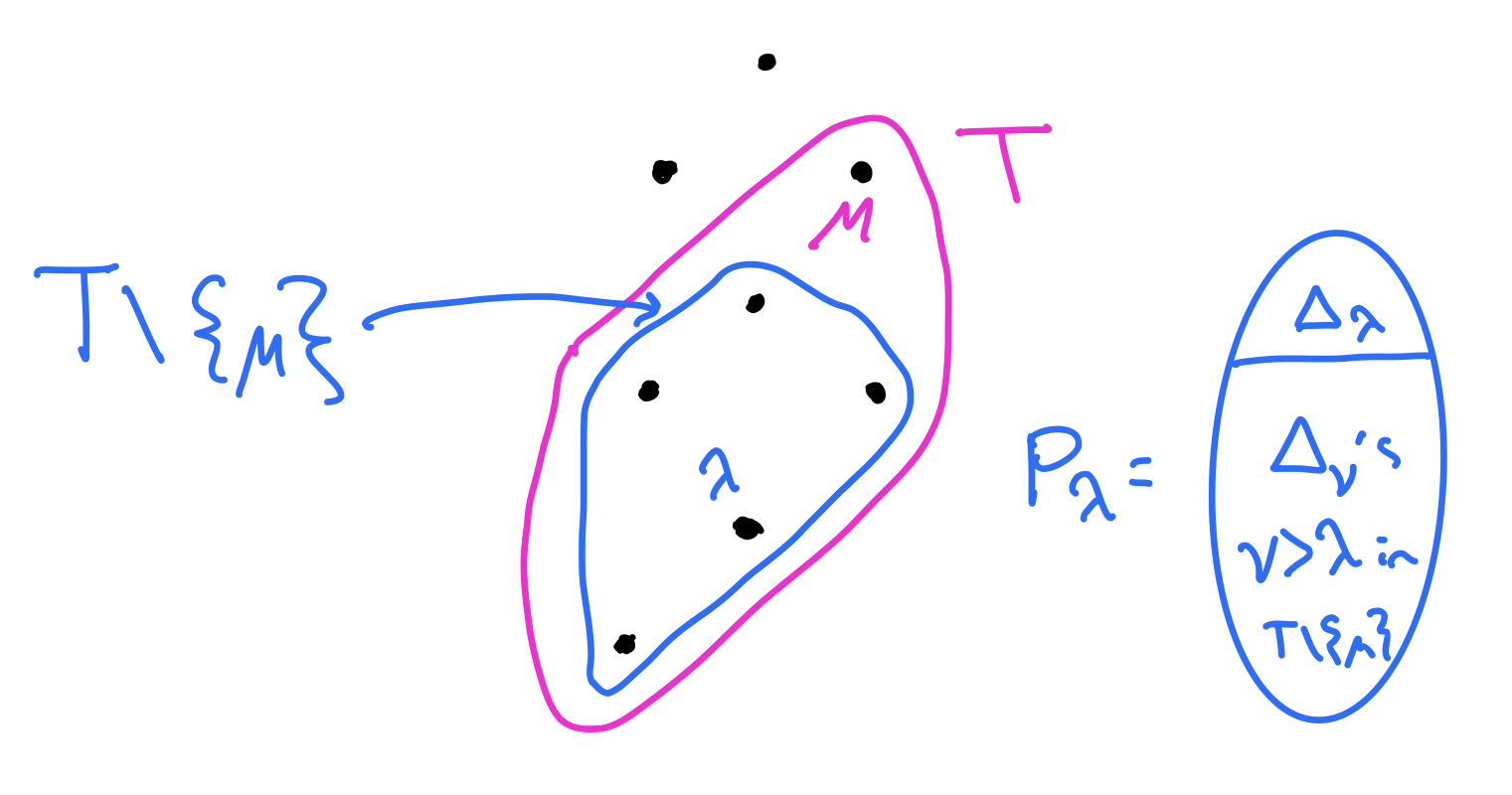}
    \]
    \item Let $E=\Ext^1(P_\lambda^{T \backslash\{\mu\}}, \nabla_\mu)$ be the ext group. 
    \item Every element of $E$ gives rise to an extension
    \[
    \Delta_\mu \hookrightarrow P' \twoheadrightarrow P_\lambda.
    \]
    It turns out that there exists a ``universal'' extension 
    \[
    E^* \otimes \Delta_\mu \hookrightarrow \widetilde{P} \twoheadrightarrow P_\lambda^{T\backslash\{\mu\}} 
    \]
    from which all others are constructed via push-out. (Challenge: construct it!)
    \item {\bf Claim:} $\widetilde{P}$ is the projective cover we are seeking. 
    \begin{proof}
    Use long exact sequence in $\Hom$:
    \[
    \begin{tikzcd}
    \Hom(P_\lambda^{T \backslash \{\mu\}}, \Delta_\mu)  \arrow[r]
& \Hom(\widetilde{P}, \Delta_\mu) \arrow[r]
\arrow[d, phantom, ""{coordinate, name=Z}]
& E \otimes \Hom(\Delta_\mu, \Delta_\lambda) \arrow[dll,
"\sim",
rounded corners,
to path={ -- ([xshift=2ex]\tikztostart.east)
|- (Z) [near end]\tikztonodes
-| ([xshift=-2ex]\tikztotarget.west)
-- (\tikztotarget)}] \\
\Ext^1(P_\lambda^{T \backslash \{\mu\}}, \Delta_\mu) \arrow[r]
& \Ext^1(\widetilde{P}, \Delta_\mu) \arrow[r]
& E \otimes \Ext^1(\Delta_\mu, \Delta_\mu) 
    \end{tikzcd}
    \]
The last term is zero by axoim (4) of a highest weight category. Hence $\Ext^1(\widetilde{P}, \Delta_\mu)=0$. A bit more work using axiom (6) shows that $\widetilde{P} = P_\lambda^T$. 
    \end{proof}
\end{itemize}

\begin{definition}
In a highest weight category $\mc{A}$, an object $T$ is {\em tilting} if it has a standard and costandard filtration. The indecomposable tilting objects in $\mc{A}$ are also indexed by $\Lambda$. 
\end{definition}
Similar arguments give injective and tilting objects. 
\end{proof}

\vspace{5mm}
\noindent
{\bf Important objects in a highest weight category}:
\[
\includegraphics[scale=0.4]{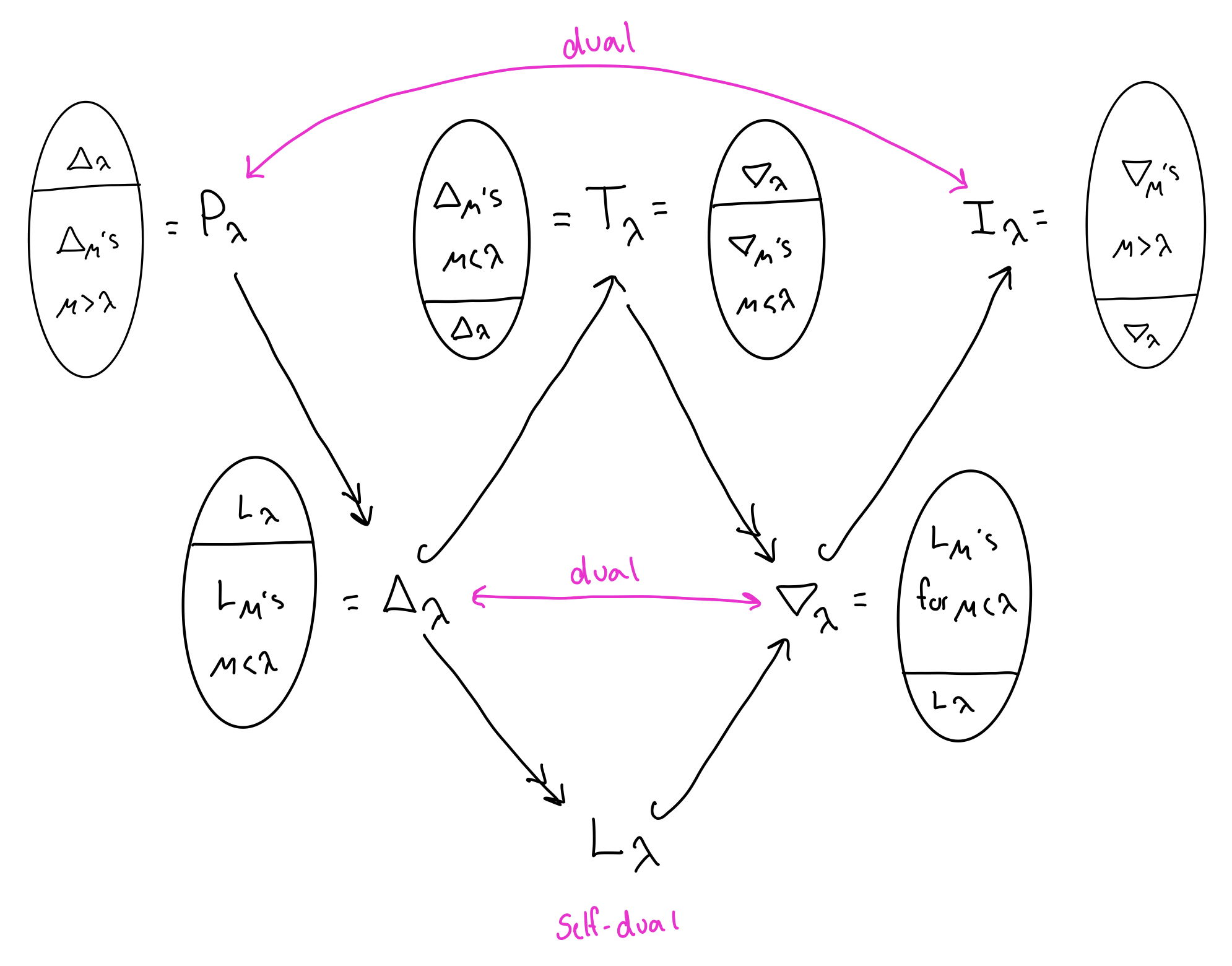}
\]
If our category also has a notion of duality compatible with the highest weight structure, then the costandard/standard and projective/injective objects are dual, as indicated by the pink arrows above. 
\begin{exercise}
\label{ext 2}
Let $\mc{B} \subset \mc{A}$ be a Serre subcategory of an abelian category $\mc{A}$. By definition, for objects $M,N \in \mc{B}$,
\[
\Ext^1_{\mc{B}}(M,N) \xrightarrow{\sim} \Ext^1_\mc{A}(M,N). 
\]
Show that  
\[
\Ext^2_{\mc{B}}(M,N) \hookrightarrow \Ext^2_{\mc{A}}(M,N). 
\]
Hint: Use effaceability and the long exact sequence of Ext. 
\end{exercise}

\subsection{Perverse sheaves for an affine stratification are a highest weight category}

\begin{theorem}
If $X=\bigsqcup_{\lambda \in \Lambda}X_\lambda$ is stratified by affine spaces, then $\Perv_\Lambda(X)$ is a highest weight category. 
\end{theorem}
\begin{proof} 
First, note that $j_\lambda:X_\lambda \hookrightarrow X$ is affine and hence 
\begin{align*}
    \Delta_\lambda &:= j_{\lambda!} k_{X_\lambda}[\dim X_\lambda] \text{ and }\\
    \nabla_\lambda &:= j_{\lambda*} k_{X_\lambda}[\dim X_\lambda] 
\end{align*}
are perverse sheaves. Moreover, we have canonical maps 
\begin{equation}
\label{delta nabla triangle}
\Delta_\lambda \twoheadrightarrow IC_\lambda \hookrightarrow \nabla_\lambda, 
\end{equation}
where $IC_\lambda:=j_{\lambda !*}k_{X_\lambda}$ is the IC sheaf corresponding to the trivial local system on $X_\lambda$. 
\begin{enumerate}
    \item $\Perv_\Lambda(X)$ is finite-length. \checkmark 
    \item Simple objects in $\Perv_\Lambda(X)$ are parameterized by pairs $(\mc{L}, \lambda)$, where $\mc{L}$ is a local system on $X_\lambda$. Because $X_\lambda$ is affine, it is contractible, and thus admits a single local system. Hence there are finitely many simple objects in $\Perv_\Lambda(X)$. \checkmark 
    \item By Schur's Lemma, $\End(IC_\lambda)$ is a division algebra over $k$.
    \begin{claim}
    If $D$ is a division algebra over a field $k$, any non-zero algebra homomorphism $D \rightarrow k$ is an isomorphism.
    \end{claim}
    \begin{proof}
    Let $\varphi:D \rightarrow k$ be a nonzero algebra homomorphism. The kernel of $\varphi$ is an ideal in $D$, but the division algebras have no non-trivial ideals, so $\ker \varphi = \{0\}$. Because $\varphi$ is a $k$-algebra homomorphism, $\varphi(1)=1$. Hence for all $\ell \in k$, $\varphi( \ell \cdot 1) = \ell \varphi(1) = \ell$ and $\mathrm{im} \varphi = k$.  
    \end{proof}
    Assume $X_\lambda \subset X$ is open. Then $IC_\lambda|_{X_\lambda}=k_{X_\lambda}[\dim X_\lambda]$. Hence the restriction map
    \[
    \End(IC_\lambda) \xrightarrow{\text{restriction}} \End(IC_\lambda|_{X_\lambda}) \simeq k
    \]
    is a nonzero morphism from a division algebra to $k$. The claim lets us conclude that $\End(IC_\lambda)\simeq k$. 
    
    For any $\lambda$, the inclusion $X_\lambda \hookrightarrow \overline{X}_\lambda$ is open because $X_\lambda$ is locally closed. There is an equivalence of categories 
    \[
    \begin{array}{c} \text{perverse sheaves} \\
    \text{supported on a} \\
    \text{closed subvariety } \\ Z \subset X \end{array} \xleftrightarrow{\sim} \Perv(Z),
    \]
so    
    \[
    \End_{\Perv_\Lambda(X)}(IC_\lambda) \simeq \End_{\Perv_\Lambda(\overline{X}_\lambda)}(IC_\lambda). 
    \]
    Hence by applying the argument above to the open inclusion $X_\lambda \hookrightarrow \overline{X}_\lambda$, we conclude that $\End(IC_\lambda)=k$ for all $\lambda$. \checkmark 
    
    \item The maps in (\ref{delta nabla triangle}) give a surjection $\Delta_\lambda \twoheadrightarrow IC_\lambda$ and an injection $IC_\lambda \hookrightarrow \nabla_\lambda$. Because the partial order on $\Lambda$ is given by closure of strata, for a closed subset $\{\leq \lambda\} \subset \Lambda$, we have 
    \[
    \Perv_\Lambda(X)_{\{\leq \lambda\}} = \Perv_{\Lambda} (\overline{X}_\lambda).
    \]
    To establish axiom (4), we must show two things: (a) that $\Delta_\lambda$ is a projective object in $\Perv_\Lambda(\overline{X}_\lambda)$, and (b) that $\Delta_\lambda \twoheadrightarrow IC_\lambda$ is a projective cover. 
    
    To show that $\Delta_\lambda$ is projective, we will show that $\Hom_{\Perv_\Lambda(\overline{X}_\lambda)}(\Delta_\lambda, \cdot)$ is exact. Because $X_\lambda \xhookrightarrow{j} \overline{X}_\lambda$ is open, for any $\mc{F} \in \Perv_\Lambda(\overline{X}_\lambda)$,
    \begin{align*}
    \Hom(j_!k_{X_\lambda}[\dim X_\lambda], \mc{F}) &= \Hom(k_{X_\lambda}[\dim X_\lambda], j^!\mc{F}) \\
    &= \Hom(k_{X_\lambda}[\dim X_\lambda], j^*\mc{F})
    \\
    &=\Hom(k_{X_{\lambda}}[\dim X_\lambda], \mc{F}|_{X_\lambda}) \\
    &= \mc{F}|_{X_\lambda}.
    \end{align*}
    Restriction to an open subvariety is an exact functor, so we conclude that $\Hom(\Delta_\lambda, \cdot )$ is exact. 
    
    Because $\Perv_{\Lambda}(\overline{X}_\lambda)$ is a finite-length abelian category (in particular, it is Krull-Schmidt), to show that $\Delta \twoheadrightarrow IC_\lambda$ is a projective cover, it suffices to show that $\Delta_\lambda$ is indecomposable. Now, the endomorphism ring
    \[
    \End(\Delta_\lambda) = \Hom(j_!k_{X_\lambda}, j_! k_{X_\lambda}) = \Hom(k_{X_\lambda}, j^! j_! k_{X_\lambda}) = \End(k_{X_\lambda}) = k
    \]
    is local, so $\Delta_\lambda$ is indecomposable. 
    
    Showing that the injection $IC_\lambda \hookrightarrow \nabla_\lambda$ is an injective hull follows from a similar argument. \checkmark 

    \item The fact that $\ker(\Delta_\lambda \twoheadrightarrow IC_\lambda), \coker(IC_\lambda \hookrightarrow \nabla_\lambda) \in \Perv_{<\lambda}(X)$ is clear from the composition series eggs and the definition of $IC_\lambda = \im(J_{\lambda!} \rightarrow j_{\lambda *})$:
    \[
    \includegraphics[scale=0.3]{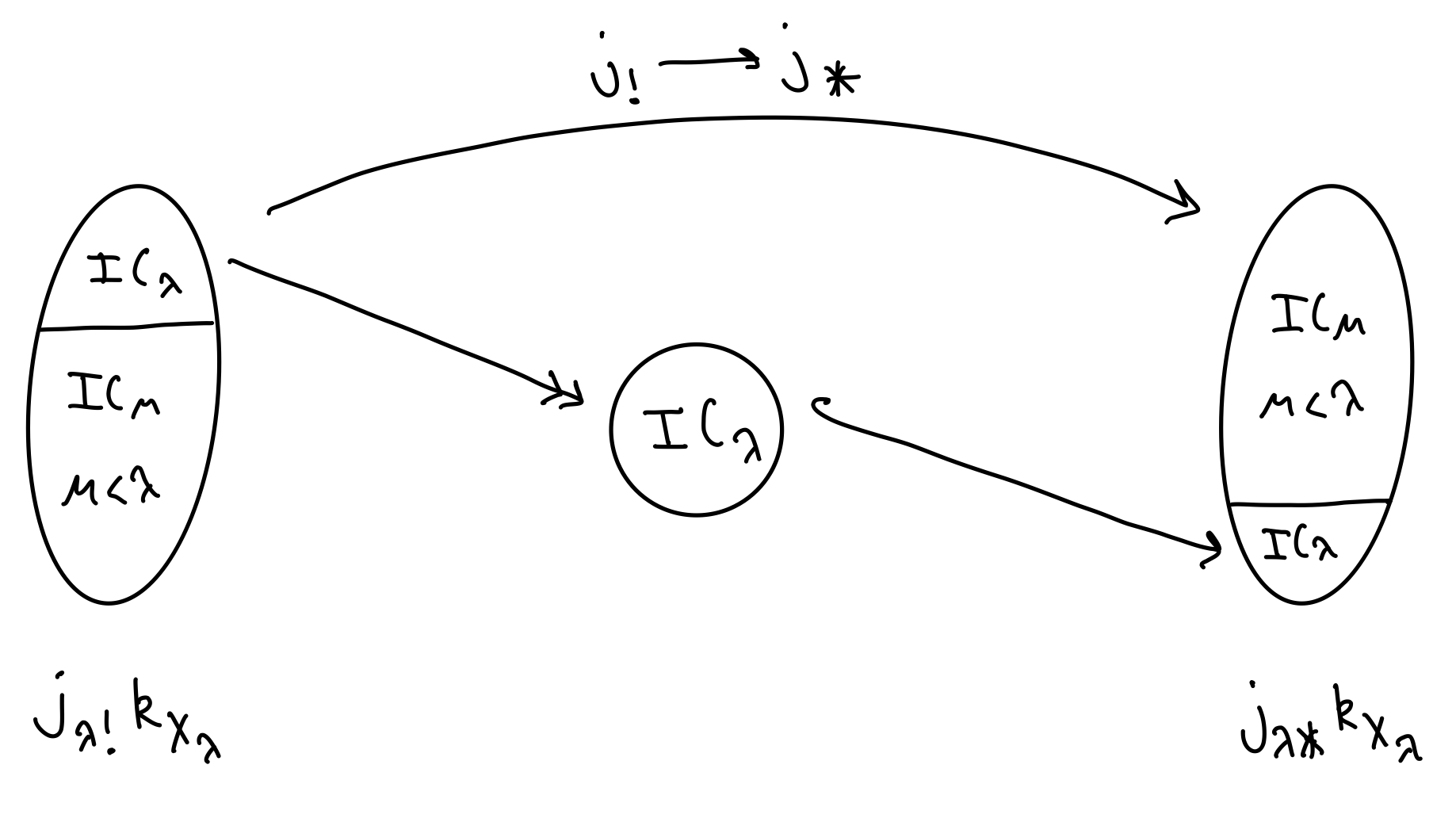} \checkmark 
    \]
    \item To see why condition (6) holds, first observe that by Exercise \ref{ext 2},
    \[
    \Ext^2_{\Perv_\Lambda(X)}(\Delta_\lambda, \nabla_\mu) \hookrightarrow \Ext^2_{\Perv(X)}(\Delta_\lambda, \nabla_\mu).
    \]
    Then, by Beilinson's theorem, 
    \[
    \Ext^2_{\Perv(X)}(\Delta_\lambda, \nabla_\mu)=\Hom^2_{D^b_c(X)}(\Delta_\lambda, \nabla_\mu).
    \]
    Once we are in the derived category, we have full access\footnote{What we mean by this is that the functor $j_\lambda^*$ is not exact if $X_\lambda$ is not open, so it does not preserve perverse sheaves (see beginning of lecture \ref{lecture 19}).} to the adjoint pairs $(j_!, j^!), (j^*, j_{ *})$. Hence we can use adjunctions to compute
    \begin{align*}
    \Hom^2_{D^b_c(X)}(\Delta_\lambda, \nabla_\mu) &= \Hom_{D^b_c(X)}(j_{\lambda!}k_{X_\lambda}[\dim X_\lambda], j_{\mu *} k_{X_\mu}[\dim X_\mu+2]) \\
    &= \Hom_{D^b_c(X)}(j_\mu^*j_{\lambda !}k_{X_\lambda}[\dim X_\lambda], k_{X_\mu}[\dim X_\mu + 2]).
    \end{align*}
    Because the functor $j_{\lambda!}$ is extension by zero, $j_\mu^*j_{\lambda !}k_{X_\lambda}=0$ for $\mu \neq \lambda$. If $\mu=\lambda$, then $j_\mu^*j_{\lambda !}k_{X_\lambda}$ is a $1$-dimensional vector space. There are no higher Exts in the category of vector spaces, so we conclude that 
    \[
    \Ext^2_{\Perv_\lambda(X)}(\Delta_\lambda, \nabla_\mu)=0. \checkmark 
    \]
\end{enumerate}
\end{proof}

\begin{remark}
It's unusual for a category of perverse sheaves to have enough projectives. By proving that $\Perv_\Lambda(X)$ is a highest weight category, we have just algebraically produced a collection of projective perverse sheaves $\{P_\lambda\}_{\lambda \in \Lambda}$. As far as we are aware, it is not known how to construct these perverse sheaves geometrically.  
\end{remark}

\begin{corollary}
\label{real}
If $\Lambda$ is a stratification of $X$ by affine spaces, 
\[
\mathrm{real}:D^b(\Perv_\Lambda(X)) \xrightarrow{\sim} D^b_\Lambda(X)
\]
is an equivalence of categories. 
\end{corollary}
\begin{proof}
Both of the sets $\{\Delta_\lambda\}_{\lambda \in \Lambda}$ and $\{\nabla_\lambda\}_{\lambda \in \Lambda}$ generate $D^b(\Perv_\Lambda(X))$ and $D_\Lambda^b(X)$. By the upper triangularity of $P_\lambda$, the set $\{P_\lambda\}_{\lambda \in \Lambda}$ also generates each category. Hence it is enough to show
\[
\mathrm{real}:\Hom^i_{D^b(\Perv_\Lambda(X))}(P_\lambda, \nabla_\mu) \xrightarrow{\sim} \Hom^i_{D_\lambda^b(X)}(P_\lambda, \nabla_\mu).
\]

In $D^b(\Perv_\Lambda(X))$, 
\[
\Hom^i(P_\lambda, \nabla_\mu) = \begin{cases} 0 & \text{ if }i>0, \\
\Hom_{\Perv_\Lambda(X)}(P_\lambda, \nabla_\mu) & \text{ if }i=0, 
\end{cases}
\]
by the projectivity of $P_\lambda$. 

In $D_\Lambda^b(X)$, 
\[
\Hom^i(P_\lambda, \nabla_\mu) = \begin{cases} 0 &\text{ if }i>0, 
\\
\Hom_{D^b_\Lambda(X)}(P_\lambda, \nabla_\mu) & \text{ if }i=0 \end{cases}
\]
as well. This is because  $\Hom^i_{D^b_\Lambda(X)}(\Delta_\lambda, \nabla_\mu) = 0$ for $i>0$ by adjunction, so the long exact sequence in $\Ext$ and the standard filtration of $P_\lambda$ imply that $\Hom^i(P_\lambda, \nabla_\mu)=0$ for $i>0$. 
\end{proof}
\begin{remark}
In the proof above, the vanishing of $\Hom^i$ for $i>0$ happens for quite different reasons in each of the two categories. In $D^b(\Perv_\Lambda(X)$, the reason is algebraic (projectivity of an object), whereas in $D^b(X)$, the reason is topological. 
\end{remark}

\subsection{Where are we going for the next few weeks?}

For the rest of this lecture, we will reconnect with the big picture of this course and describe our plan for the upcoming weeks. 

Let $(X \supset R, X^\vee \supset R^\vee)$ be a root datum, and let $G, G^\vee$ be the corresponding dual groups over $\C$. Associated to this datum, we have a finite Weyl group $W_f$, and an affine Weyl group $W=W_f \ltimes \Z X^\vee$. Let $H$ be the affine Hecke algebra, and $Z$ its center. In October and November of last year, we constructed the following diagram. 
\[
\begin{tikzcd}
Z \simeq (\Z X^\vee) ^{W_f} \arrow[d, hookrightarrow] \arrow[r, "\sim", dash] &  R_{G^\vee}=[\Coh \mathrm{pt}/_{G^\vee}] \arrow[d, "\text{pull-back}"] \\
H  \arrow[r, "\sim", "KL"', dash] & \left[K^{G^\vee \times \C^\times}(St)\right]
\end{tikzcd}
\]
The isomorphism on the bottom line is the Kazhdan--Lusztig isomorphism, which was merely stated (not yet even sufficiently explained, let alone proved!), and the injection on the left is Bernstein's description of the center of of the affine Hecke algebra. Our goal for the rest of the course is to categorify this picture. This is done via Bezrukavnikov's equivalence, which very roughly is an equivalence of the form
\[
\begin{tikzcd}
\left( \begin{array}{c} \text{constructible} \\ \text{affine Hecke} \\ \text{category} \end{array}, * \right) \arrow[r, "\sim", "B"', dash] & \left( \begin{array}{c} \text{coherent} \\ \text{affine Hecke} \\ \text{category} \end{array}, * \right). 
\end{tikzcd}
\]
The constructible affine Hecke category on the LHS should be (a variant of) a category of construcible sheaves on the affine flag variety of $G$, and the coherent affine Hecke category on the RHS should be (a variant of) a category of $G \times \C^\times$-equivariant coherent sheaves on the Steinberg variety. When we take Grothendieck groups, we should recover the first diagram. 

\vspace{3mm}
\noindent
{\bf Philosophy for now}: Before we can understand this story on the level of categories, we need to better understand the Kazhdan--Lusztig isomorphism.

\vspace{3mm}
\noindent
{\bf General strategy for understanding $KL$}:
\begin{enumerate}
    \item On the Bernstein generators $T_i$, $\mc{O}_\lambda$, define
    \begin{align*}
        T_i &\mapsto \mc{Q}_i \text{ (to be explained next week)}\\
        \mc{O}_\lambda &\mapsto \mc{O}_\lambda \text{ (pull-back of $\mc{O}(\lambda)$ on $G^\vee / B^\vee$ to }T^*G^\vee / B^\vee \xhookrightarrow{\text{diagonal}} St) 
    \end{align*}
    \begin{remark}
    It is very challenging to verify the relations directly! (See Leonardo Maltoni's May 24, 2019 talk in the Informal Friday Seminar.) 
    \end{remark}
    \item The affine Hecke algebra $H$ has two important modules, defined as follows. Let 
    \begin{align*}
        \mathrm{triv}: H_f &\rightarrow \Z[q^{\pm 1}], T_i \mapsto q \\
        \mathrm{sgn}: H_f &\rightarrow \Z[q^{\pm 1}], T_i \mapsto -1
    \end{align*}
    be the (quantized versions of) the trivial and sign representation of the finite Hecke algebra $H_f$. Define two $H$-modules 
\[
M:= H \otimes_{H_f} \mathrm{triv}, \hspace{5mm} N:= H \otimes_{H_f} \mathrm{sgn}
\]
by left multiplication on the first tensor factor. These are, respectively, the {\bf spherical module} and {\bf antispherical module} of the Hecke algebra corresponding to the parabolic subgroup $W_f \subset W$. 
\begin{remark}
By the Bernstein presentation, 
\[
H = \Z[q^{\pm 1}][X^\vee] \otimes H_f.
\]
Hence, as $\Z[q^{\pm 1}]$-modules, 
\[
M \cong N \cong \Z[q^{\pm 1}][X^\vee].
\]
Thus, as modules over $\Z[q^{\pm 1}][X^\vee]$, both the spherical and anti-spherical modules are free of rank $1$. 
\end{remark}
\item By convolution formalism, both $K^{G \times \C^\times}(G^\vee / B^\vee)$ and $K^{G \times \C^\times}(T^* G^\vee/ B^\vee)$ are $K^{G^\vee \times \C^\times}(St)$-modules. 
\begin{remark}
Note that 
\[
\Z[q^{\pm 1}][X^\vee]=K^{B^\vee \times \C^\times}(\mathrm{pt}) \cong K^{G^\vee \times \C^\times}(G^\vee / B^\vee) \xrightarrow[\text{pull-back}]{\sim} K^{G^\vee \times \C^\times}(T^* G^\vee / B^\vee). 
\]
\end{remark}
\item Compute actions of generators in $M$ (resp. $N$) and match them with actions in $K^{G^\vee \times \C^\times}(G^\vee / B^\vee)$ (resp. $K^{G^\vee \times \C^\times}(T^* G^\vee / B^\vee)$ under (a choice of) the above isomorphisms. Because all of these modules are faithful, this implies the Kazhdan--Lusztig isomorphism. 

\begin{remark}
We really only need to do this computation for either $M$ or $N$ to use this argument to deduce the Kazhdan--Lusztig isomorphism. 
\end{remark}
\end{enumerate}

\noindent
{\bf Next week}: We implement 1-4. 

\vspace{5mm}
\noindent
{\bf Following week}: We take the first step toward Bezrukavnikov's equivalence by explaining the Arkhipov--Bezrukavnikov's theorem that 
\[
\mc{AS} \xrightarrow{\sim} D^b\left(\Coh^{G^\vee \times \C^\times}(T^*G^\vee / B^\vee)\right),
\]
where $\mc{AS}$ is the categorical antispherical module. 

%% file: lecture-21.tex
\section{Lecture 21: Equivariant $K$-theory of the Steinberg variety}
\label{lecture 21}

Our goal for the next few lectures is to prove the Kazhdan--Lusztig isomorphism:
\[
H_\mathrm{affine} \simeq K^{G^\vee \times \C^\times}(\mathrm{Steinberg}).
\]
This is an isomorphism of the affine Hecke algebra with the $G^\vee \times \C^\times$-equivariant $K$-theory of the Steinberg variety. Recall from the end of last lecture that our general strategy for proving this isomorphism is to find a vector space on which each algebra acts faithfully by the same operators. 

\begin{remark}
The reader might have the (correct) impression that this is a rather indirect way of seeing that two algebras are isomorphic. However, there are several instances of very indirect techniques to obtain isomorphisms, equivalences, or correspondences in the Langlands program. Another example of this is Soergel's functor, which proves an equivalence of two categories (one geometric and one representation theoretic) by matching their images in a third world of algebra (so-called Soergel bimodules).
\end{remark}

\noindent
References for this lecture are:
\begin{itemize}
    \item Kazhdan--Lusztig, Proof of the Deligne--Langlands conjecture for Hecke algebras, \cite{KL},
    \item Chriss-Ginzburg, Representation theory and complex geometry, Chapters 6 \& 7, \cite{chris-ginzburg},
    \item Henderson, Notes on affine Hecke algebras and $K$-theory.
\end{itemize}

\subsection{Equivariant $K$-theory}

Let $X$ be a scheme with an action by a group $G$. Associated to $X$ are the $G$-equivariant $K$-groups for $i \in \Z_{\geq 0}$:
\[
G \circlearrowright X \rightsquigarrow K^G_i(X). 
\]
When $i=0$, this is 
\[
K_0^G(X) = \begin{array}{c} \text{Grothendieck group of the exact category} \\ \text{of $G$-equivariant perfect complexes on $X$} \\
\text{(i.e. bounded complexes of vector bundles)} \end{array}.
\]
If $X$ is smooth, then every $G$-equivariant coherent sheaf has a resolution by $G$-equivariant vector bundles and thus  
\[
K_0^G(X) = \begin{array}{c} \text{Grothendieck group of the category} \\ \text{of $G$-equivariant coherent sheaves on X} \end{array}.
\]
\begin{remark}
At some future point we should expand on higher $K$-groups, but that time is not today. For those who are interested, Quillen's work and the groundbreaking paper \cite{TT} of Thomason-Trobaugh are highly recommended. In \cite{TT}, they show that if $U \subset X$ is open, we have a sequence\footnote{We can think of this like the long exact sequence in cohomology, but unlike cohomology, where the long exact sequence came almost simultaneously with its definition, this took twenty years after the initial definition of $K$-theory to prove.} 
\[
\cdots \rightarrow K_i(X \text{ on }Y) \rightarrow K_i(X) \rightarrow K_i(U) \rightarrow K_{i-1}(X \text{ on }Y) \rightarrow \cdots
\]
which is almost exact, except that $K_0(X) \rightarrow K_0(U)$ might not be onto for singular $X$. (This result was already proved by Quillen for smooth $X$.)
\end{remark}

In the arguments to come, we only care about $K_0^G$, and at certain points we will simply assert that certain boundary maps vanish; e.g., we have a short exact sequence 
\[
K_0^G(X \text{ on }Y) \hookrightarrow K^G_0(X) \twoheadrightarrow K_0^G(U).
\]
We set  $K^G(X):=K^G_0(X)$.

\vspace{3mm}
\noindent
{\bf Examples}
\begin{enumerate}
    \item $K^G(\mathrm{pt})=R_G= \begin{array}{c} \text{Grothendieck group of finite-dimensional} \\ \text{algebraic representations of }G \end{array}$ 
    \item We claim that $K(\mathbb{P}^1)=\Z[x^{\pm 1}]/(x-1)^2$. How can we see that this is true? To start, every vector bundle on $\mathbb{P}^1$ is a sum of line bundles. Hence there is a surjection 
    \begin{align*}
        \Z[x^{\pm 1}] &\twoheadrightarrow K(\mathbb{P}^1) \\
        x^m &\mapsto \mc{O}(m).
    \end{align*}
    The tautological exact sequence of vector bundles 
    \[
    \mc{O}(-1) \hookrightarrow \mc{O}^{\oplus 2}_{\mathbb{P}^1} \twoheadrightarrow \mc{O}(1)
    \]
    shows that under the surjection above, 
    \[
    x^{-1} - 2 + x \mapsto 0.
    \]
    Hence $(x-1)^2 \mapsto 0$. This shows us that we have a map $\phi:\Z[x^{\pm 1}]/(x-1)^2 \rightarrow K(\mathbb{P}^1)$. 
    
    \vspace{3mm}
    \noindent
    {\bf Claim:} $K(\mathbb{P}^1) = \Z[\mc{O}] \oplus \Z[\mc{O}(1)]$. 
    \vspace{3mm}
    
    The arguments above imply that $[\mc{O}]$ and $[\mc{O}(1)]$ span $K(\mathbb{P}^1)$. It remains to show that they are linearly independent. 
    
    \vspace{2mm}
    \noindent 
    {\em Proof 1}: For $X$ proper, $K(X)$ carries an intersection form:
    \[
    \langle \mc{F}, \mc{G} \rangle := \chi(\mc{F} \otimes \mc{G}), 
    \]
    where $\chi$ is the Euler characteristic. Computing the pairings with respect to this form, we get the following table. 
    \[
    \begin{tabular}{c|cc} 
    & $\mc{O}$ & $\mc{O}(1)$  \\ 
    \hline
    $\mc{O}$ & $1$ & $2$   \\ 
    $\mc{O}(1)$ & $2$ & $3$  \\
   \end{tabular}
    \]
    We can see that the determinant is $-1$, so $\mc{O}$ and $\mc{O}(1)$ are linearly independent. \qed
    
    \vspace{2mm}
    \noindent
    {\em Proof 2:} We have seen in Tom's course that 
    \begin{align*}
        D^b(\Coh \mathbb{P}^1) &\simeq D^b\left(\Rep( \begin{tikzcd}[ ampersand replacement=\&] \bullet \arrow[r, shift left] \arrow[r, shift right] \& \bullet \end{tikzcd}) \right) \\
        \mc{O} &\mapsto L_0:=\begin{tikzcd}[ ampersand replacement=\&] k \arrow[r, shift left] \arrow[r, shift right] \& 0 \end{tikzcd}\\
        \mc{O}(1) &\mapsto L_1:=\begin{tikzcd}[ ampersand replacement=\&] 0 \arrow[r, shift left] \arrow[r, shift right] \& k \end{tikzcd}
    \end{align*}
    The $K$-group of any finite length abelian category has a basis given by the classes of simple objects. In the case of $\Rep( \begin{tikzcd}[ ampersand replacement=\&] \bullet \arrow[r, shift left] \arrow[r, shift right] \& \bullet \end{tikzcd})$, the two simple objects are $L_0$ and $L_0$. Hence we have 
    \[
    K(\PP^1) = K(D^b(\Coh(\PP^1))) = K\left(D^b\left(\Rep( \begin{tikzcd}[ ampersand replacement=\&] \bullet \arrow[r, shift left] \arrow[r, shift right] \& \bullet \end{tikzcd}) \right)\right) = \Z[L_0] \oplus \Z[L_1]. \qed
    \]

    \begin{remark}
    The second proof illustrates a common phenomenon: derived equivalences can have interesting consequences for $K$-theory. 
    \end{remark}
\end{enumerate}

\vspace{5mm}
\noindent
{\bf Fact:} Proper maps $f: X \rightarrow Y$ induce maps in $K$-theory:
\[
p_*: K^G(X) \rightarrow K^G(Y)
\]
The map $p_*$ is given by $p_*(F)=\sum(-1)^iR^if_*(F)$. 
\begin{example}
\begin{enumerate}
\item Projection to a point, $p:\PP^1 \rightarrow \mathrm{pt}$, induces the following map on (non-equivariant) $K$-theory:
\begin{align*}
    p_*:K(\PP^1)\cong \Z[x^{\pm 1}]/(x-1)^2 &\rightarrow K(\mathrm{pt})\cong \Z \\
    x^m &\mapsto \chi(\mc{O}(m)) = m+1
\end{align*}
\item Let $B = \left\{ \bp * & * \\ 0 & * \ep \right\} \subset \SL_2$ act on $\PP^1$ in the standard way. Then the map that $p$ induces on $B$-equivariant $K$-theory is given by the Weyl character formula. Indeed, we can identify 
\[
K^G(\PP^1) = K^B(\mathrm{pt}) = \Z[X(B)] = \Z[x^{\pm 1}]. 
\]
Then under the composition
\[
K^G(\PP^1) \rightarrow K^G(\mathrm{pt}) \rightarrow K^B(\mathrm{pt}) = \Z[x^{\pm 1}].
\]
the element $x^m \in K^B(\PP^1)\cong \Z[x^{\pm 1}]$ maps to $\frac{x^m - x^{-m-2}}{1-x^{-2}} \in K^B(\mathrm{pt})\cong \Z[x^{\pm 1}]$. 
\end{enumerate}
\end{example}

\subsection{Equivariant $K$-theory of the Steinberg variety}

Let 
\[
G \supset B \supset T
\]
be a complex reductive group containing a Borel subgroup containing a maximal torus. Let $X=X(T)$ be the character lattice, $W_f$ the finite Weyl group, and $W=W_f \ltimes \Z X$ the affine Weyl group of the dual\footnote{This convention is introduced to avoid including many checks (to indicate Langlands dual groups) in the rest of this lecture and the next.} root system. Let 
\[
\mc{N} \subset \mf{g}
\]
be the nilpotent cone in the Lie algebra of $G$, and 
\[
\widetilde{\mc{N}}=\{(x, \mf{b}) \in \mc{N} \times \mc{B} \mid x \in \mf{b} \} =T^*\mc{B} \rightarrow \mc{N}
\]
be the Springer resolution. Recall that the {\em Steinberg variety} is defined to be 
\[
St:= \widetilde{\mc{N}} \times_\mc{N} \widetilde{\mc{N}} = \{(x, \mf{b}, \mf{b}') \mid x \in \mf{b}, x \in \mf{b}' \}.
\]
We also explained in Lecture \ref{lecture 13} how this could be realized as the conormal space to the space $\mc{B} \times \mc{B}$ with the stratification by $G$-orbits (which are parameterized by $W$). 

\begin{example}
Let $G=\SL_2$. Then $\mc{B} = \PP^1$, and the $G$-orbit stratification on $\mc{B} \times \mc{B}$ is given by 
\[
\mc{B} \times \mc{B} = \Delta \sqcup Y,
\]
where $Y:=(\PP^1 \times \PP^1) \backslash \Delta$. An illustration:
\[
    \includegraphics[scale=0.25]{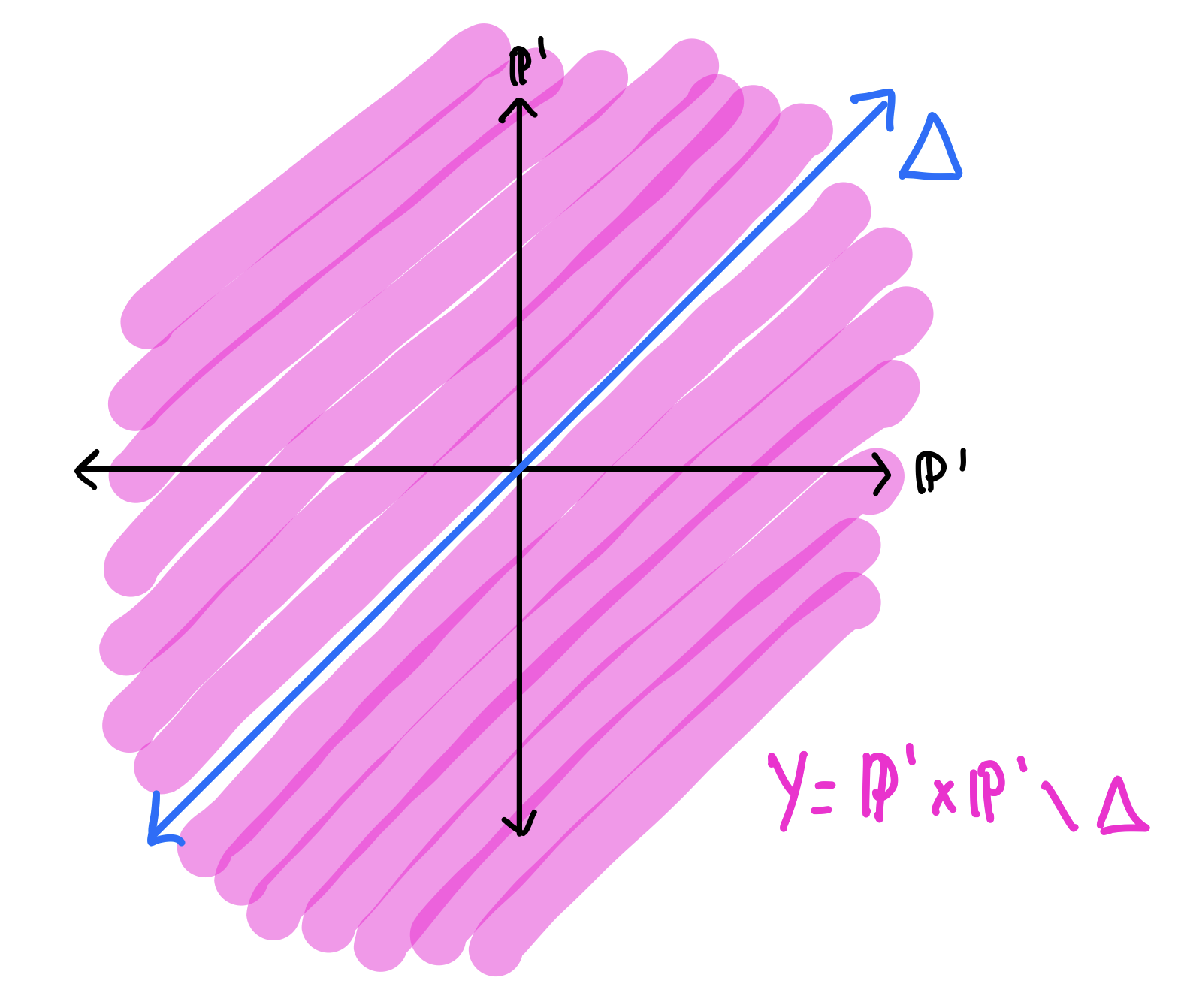}  
\]
Then the Steinberg variety is 
\[
T^*_\Delta(\PP^1 \times \PP^1) \sqcup T_Y^*(\PP^1 \times \PP^1) = T^* \PP^1 \sqcup Y.
\]
An illustration:
    \[
    \includegraphics[scale=0.4]{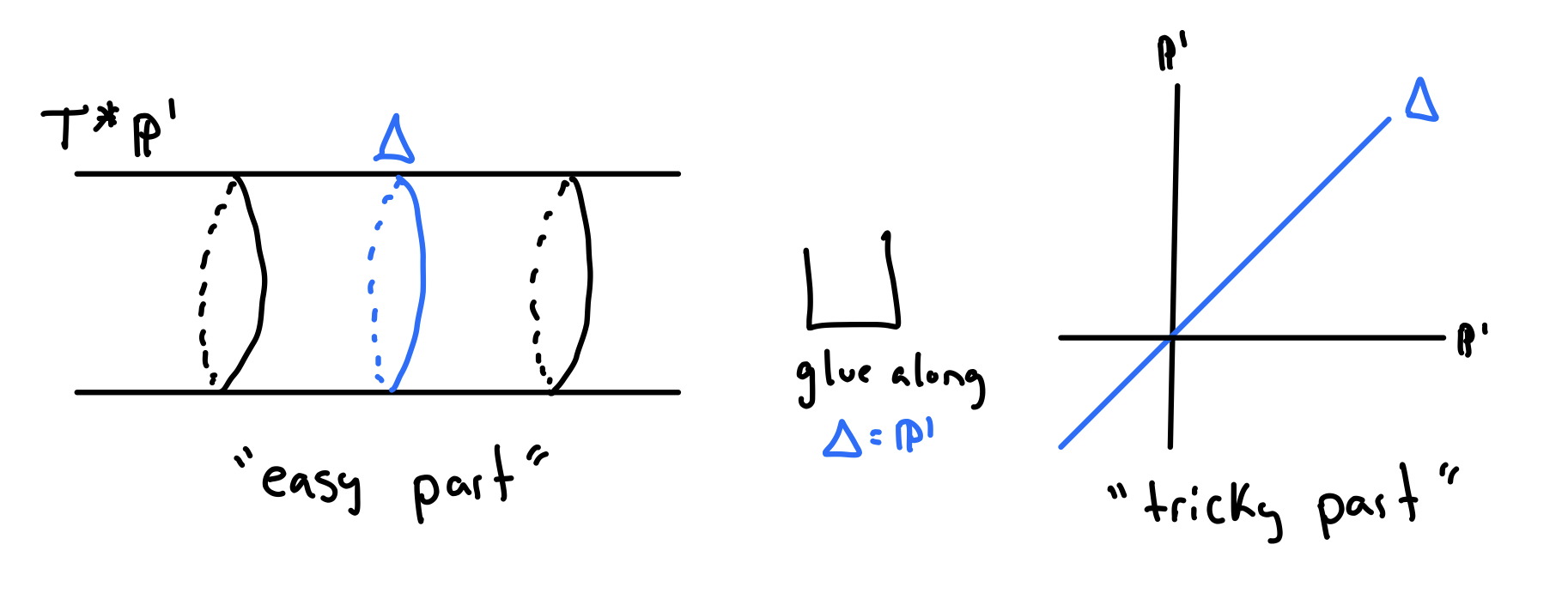} 
    \]
In the ``easy part,'' projection onto each component of $\mc{B} \times \mc{B}$ is an isomorphism. In the ``tricky part,'' projection onto each component hits the zero section in $T^*\PP^1$. 
\end{example}
The group $G \times \C^\times$ acts on $St$ via 
\begin{equation}
    \label{C star action}
(g, z) \cdot (z, \mf{b}, \mf{b}') = (z^2gx, g\mf{b}, g\mf{b}').
\end{equation}

\vspace{3mm}
\noindent
{\bf Algebra structure on $K^{G \times \C^\times}(St)$:}
\vspace{3mm}

Recall that $K^{G \times \C^\times}(St)$ has an algebra structure given by convolution in $D^b(\Coh)$ (see Lectures \ref{lecture 14} and \ref{lecture 15} for a refresher): 
\[
\mc{F} * \mc{G} := p_{13*}(p_{12}^* \mc{F} \otimes p_{23}^* \mc{G}), 
\]
where $p_{ij}$ are the canonical projections \[
\begin{tikzcd}
& \widetilde{\mc{N}} \times_\mc{N} \widetilde{\mc{N}} \times_\mc{N} \widetilde{\mc{N}} \arrow[dl, "p_{12}"'] \arrow[d, "p_{13}"] \arrow[dr, "p_{23}"] & \\
St & St & St 
\end{tikzcd}
\]

\vspace{3mm}
\noindent
{\bf Module structure on $K^{G \times \C^\times}(\widetilde{\mc{N}})$}:
\vspace{3mm}

Similarly, convolution also gives $K^{G \times \C^\times}(\widetilde{\mc{N}})$ the structure of a module for $K^{G \times \C^\times}(St)$:
\[
\mc{F} * \mc{G} := p_{1*}(\mc{F} \times p_{2}^* \mc{G}),
\]
where 
\[
\begin{tikzcd}
& \widetilde{\mc{N}}\times_{\mc{N}} \widetilde{\mc{N}} \arrow[dl, "p_1"'] \arrow[dr, "p_2"] & \\
\widetilde{\mc{N}} & & \widetilde{\mc{N}} 
\end{tikzcd}
\]

Recall that our goal is to prove that 
\[
K^{G \times \C^\times}(St) \cong H.
\]
We will accomplish that through a subgoal, which is to show that 
\[
K^{G \times \C^\times}(\widetilde{\mc{N}}) \cong \text{anti-spherical module}.
\]
We can work toward the subgoal by examining the vector space structure of $K^{G \times \C^\times}(\widetilde{\mc{N}})$ more closely. 

First, note that we have a map
\[
K^{G \times \C^\times}(\widetilde{\mc{N}}) \xleftarrow[back]{pull} K^{G \times \C^\times}(\mc{B}) = K^{B \times \C^\times}(\mathrm{pt}) = \Z[v^{\pm 1}][X]. 
\]
Because $\widetilde{\mc{N}}$ is a vector bundle over $\mc{B}$, pull-back is an isomorphism. Hence 
\[
K^{G \times \C^\times}(\widetilde{\mc{N}}) = \Z[v^{\pm 1}][X].
\]

Now we turn to the vector space structure on $K^{G \times \C^\times}(St)$. The space $\mc{B} \times \mc{B}$ has a filtration via closures of $G$-orbits:
\[
\emptyset =Z_{-1} \subset Z_0 \subset \cdots \subset Z_m = \mc{B} \times \mc{B},
\]
with $Z_i \backslash Z_{i+1} \simeq G \cdot (x_iB, B) \simeq G\times_B X_i$, where $X_i=B x_i B/B$ is a Bruhat cell. This induces a filtration of the Steinberg variety:
\[
\emptyset = \widetilde{Z}_{-1} \subset \widetilde{Z}_0 \subset \cdots \subset \widetilde{Z}_m = St,
\]
with $\widetilde{Z}_i\backslash \widetilde{Z}_{i+1} \simeq T^*_{X_i}(\mc{B} \times \mc{B}) =:T^*_{x_i}.$  We have
\[
K^{G \times \C^\times}(T_{x_i}^*) = K^{G \times \C^\times}(X_i) = K^{B \times \C^\times} (X_i) = K^{B \times \C^\times}(\mathrm{pt}). 
\]
All boundary maps vanish when we apply $K^{G \times \C^\times}$ to the above filtration, and a little work (see \cite{chris-ginzburg}) yields: 
\[
K^{G \times \C^\times}(St) = \bigoplus_{x \in W_f} \Z[v^{\pm 1}][X][\mc{O}_{\overline{T^*_{x_i}}}]. 
\]
\begin{remark}
Compare this to the decomposition 
\[
H = \bigoplus_{x \in W_f} \Z[v^{\pm 1}][X] T_x.
\]
\end{remark}

\subsection{The spherical and anti-spherical modules}

Recall Bernstein's presentation of the affine Hecke algebra (Theorem \ref{Bernstein presentation}): $H$ is an algebra over $\Z[v^{\pm 1}]$ with generators
\[
\begin{array}{c} \{H_s\mid s \in S_f\} \\
\text{(finite part)} \end{array} \hspace{5mm} \text{ and } \hspace{5mm} \begin{array}{c} \{ \theta_\lambda\mid \lambda \in X\} \\ \text{(lattice part)} \end{array}.
\]
The generators $\{H_s\mid s \in S_f\}$ generate a copy of the finite Hecke algebra $H_f$ and the generators $\{\theta_\lambda\mid \lambda \in X \}$ generate $\Z[v^{\pm 1}][X]$. The relations are:
\begin{enumerate}
    \item $H_s^2 = (v^{-1}-v)H_1+1$ for all $s \in S$ + braid relations (finite part),
    \item $\theta_\lambda \theta_\gamma = \theta_{\lambda + \gamma}$ for all $\lambda, \gamma \in X$ (lattice part), and
    \item {\bf the most important relation}: For  a simple reflection $s=s_\alpha \in S$ and $\lambda \in X$,
    \[
    H_s \theta_{s \lambda} - \theta_\lambda H_s = (v-v^{-1}) \left(\frac{\theta_\lambda - \theta_{s \lambda}}{1 - \theta_{-\alpha}}\right). 
    \]
\end{enumerate}
The first relation can be rewritten as 
\[
(H_s +v)(H_s - v^{-1})=0,
\]
which implies that $H_f$ has two natural rank $1$ modules:
\begin{align*}
    \mathrm{triv}:H_s \mapsto v^{-1}, \\
    \mathrm{sgn}: H_s \mapsto -v.
\end{align*}
From these we construct two induced $H$-modules: 
\begin{align*}
    M &= H \otimes_{H_f} \mathrm{triv}, \text{ ``spherical module''} \\
    N&= H \otimes_{H_f} \mathrm{sgn}, \text{ ``anti-spherical module''} 
\end{align*}
As modules over the lattice part, each is isomorphic to $\Z[v^{\pm 1}][X]$. 

\begin{remark} In what follows, we will abuse notation and write
\[
\theta_\lambda:=\theta_\lambda \otimes 1 \in N.
\]
\end{remark}

Using the third relation in the affine Hecke algebra, we can compute the action of the generator $\theta_s \in H$ on $\theta_{s\lambda} \in N$ in the anti-spherical module:
\begin{align*}
    H_s \cdot \theta_{s\lambda} &=(v-v^{-1}) \left(\frac{\theta_\lambda - \theta_{s \lambda}}{1 - \theta_{-\alpha}}\right) + \theta_\lambda H_s \\
    &= v \left(\frac{\theta_\lambda - \theta_{s \lambda} - \theta_{\lambda} + \theta_{\lambda - \alpha}}{1-\theta_{-\alpha}}\right) - v^{-1} \left(\frac{\theta_\lambda - \theta_{s \lambda}}{1-\theta_{-\alpha}}\right) \\
    &=v \left(\frac{\theta_{\lambda - \alpha} - \theta_{s \lambda}}{1 - \theta_{-\alpha}}\right) - v^{-1} \left(\frac{\theta_\lambda - \theta_{s \lambda}}{1- \theta_{-\alpha}}\right). 
\end{align*}
This formula looks nicer if we instead compute the action in terms of the Kazhdan--Lusztig generator $b_s:=H_s + v$:
\begin{align*}
b_s \cdot \theta_{s \lambda} &= v  \left(\frac{\theta_{\lambda - \alpha} - \theta_{s \lambda}}{1 - \theta_{- \alpha}}\right) - v^{-1} \left(\frac{\theta_\lambda - \theta_{s \lambda}}{1 - \theta_{-\alpha}}\right)+v \theta_{s \lambda}\\
&= v \left(\frac{\theta_{\lambda - \alpha} - \theta_{s \lambda-\alpha}}{1 - \theta_{- \alpha}}\right) - v^{-1} \left(\frac{\theta_\lambda - \theta_{s \lambda}}{1 - \theta_{-\alpha}}\right)\\
&= (v\theta_{-\alpha} - v^{-1}) \left( \frac{\theta_\lambda - \theta_{s \lambda}}{1-\theta_{-\alpha}} \right).
\end{align*}
In other words, 
\begin{equation}
    \label{star star}
b_s \cdot \theta_\lambda = (v^{-1} - v\theta_{-\alpha}) \left( \frac{\theta_\lambda - \theta_{s \lambda}}{1 - \theta_{-\alpha}} \right).
\end{equation}

We'll pick up here next week with a very similar looking computation in $K$-theory. 

%% file: lecture-22.tex
\section{Lecture 22: Proof of the Kazhdan--Lusztig isomorphism}
\label{lecture 22}

In today's lecture we will complete the proof of the Kazhdan--Lusztig isomorphism. First, we establish and streamline notation. 

\subsection{Notation and set-up}
Let 
\[
G \supset B \supset T
\]
be a complex reductive group, Borel subgroup, and maximal torus. Let $X$ be the character lattice, and $W_f \subset W$ the finite and affine Weyl groups with simple reflections $S_f \subset W_f$ and $S \subset W$. We denote by $\mc{B}$ the flag variety, which we realize as the variety of Borel subalgebras of $\mf{g} = \Lie{G}$. The group $G$ acts on the product $\mc{B} \times \mc{B} $, and the orbits give a stratification with strata parameterized by $W_f$:
\[
G \circlearrowright \mc{B} \times \mc{B} = \bigsqcup_{x \in W_f} \mathbb{O}_x.
\]
The orbit $\OO_x$ consists of pairs of flags/Borel subalgebras in relative position $x$. For example, $\OO_\mathrm{id}=\Delta$ (the diagonal in $\mc{B} \times \mc{B}$), and $\OO_{w_0}$ (where $w_0 \in W_f$ is the longest element) is open in $\mc{B} \times \mc{B}$.

Let 
\[
\widetilde{\mc{N}}=T^*\mc{B} \rightarrow \mc{N} \subset \mf{g}
\]
be the Springer resolution. The Steinberg variety is 
\[
St = \widetilde{\mc{N}} \times_{\mc{N}} \widetilde{\mc{N}} = \bigsqcup_{x \in W_f} T^*_{\OO_x}(\mc{B} \times \mc{B}) = \bigcup_{w \in W_f} \Lambda_x,
\]
where $T^*_{\OO_x}(\mc{B} \times \mc{B})$ is the conormal bundle of the $G$-orbit $\OO_x$, whose closure $\Lambda_x := \overline{T^*_{\OO_x}(\mc{B} \times \mc{B})}$ is an irreducible component of $St$. 

Last week we discussed the following example. 
\begin{example}
Let $G=\SL_2$, so $\mc{B}=\PP^1$. The $G$-orbit stratification of $\mc{B} \times \mc{B}$ is:
\[
    \includegraphics[scale=0.4]{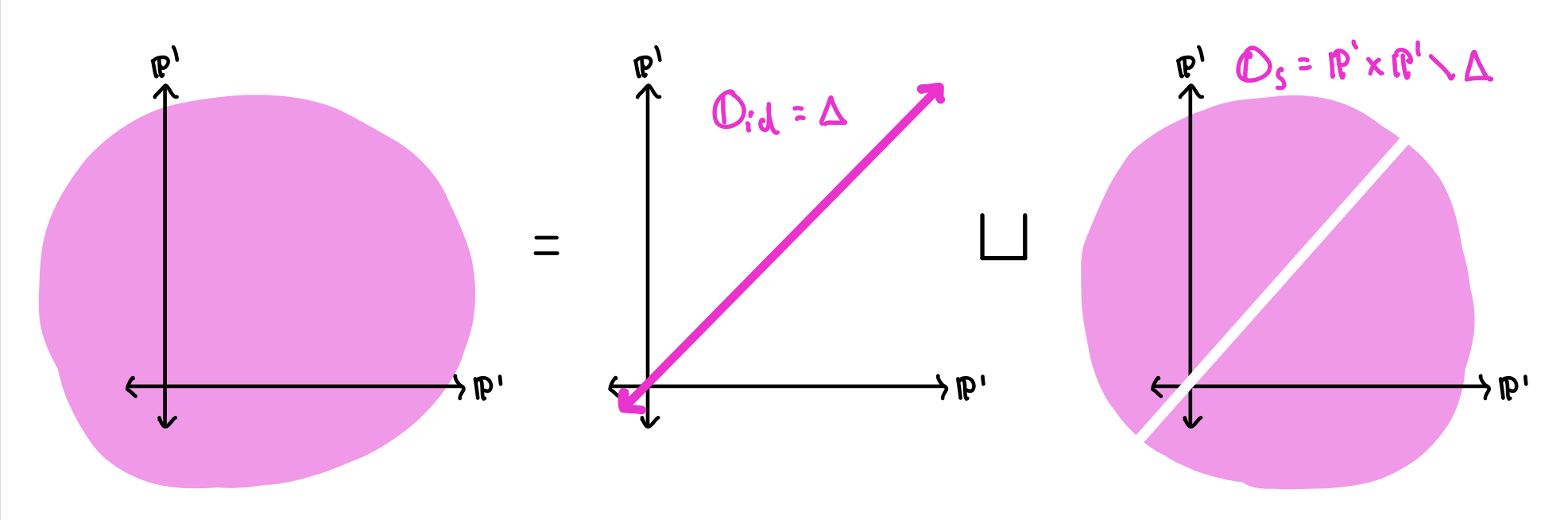} 
\]
The corresponding stratification of the Steinberg variety is 
\[
St = T^*_{\OO_{\mathrm{id}}}(\mc{B} \times \mc{B}) \sqcup T^*_{\OO_s} (\mc{B} \times \mc{B}).
\]
The closures of the strata give the irreducible components $\Lambda_x$, $x \in W_f$ of $St$, which are glued together as follows:
\[
St =\vcenter{\hbox{\includegraphics[scale=0.4]{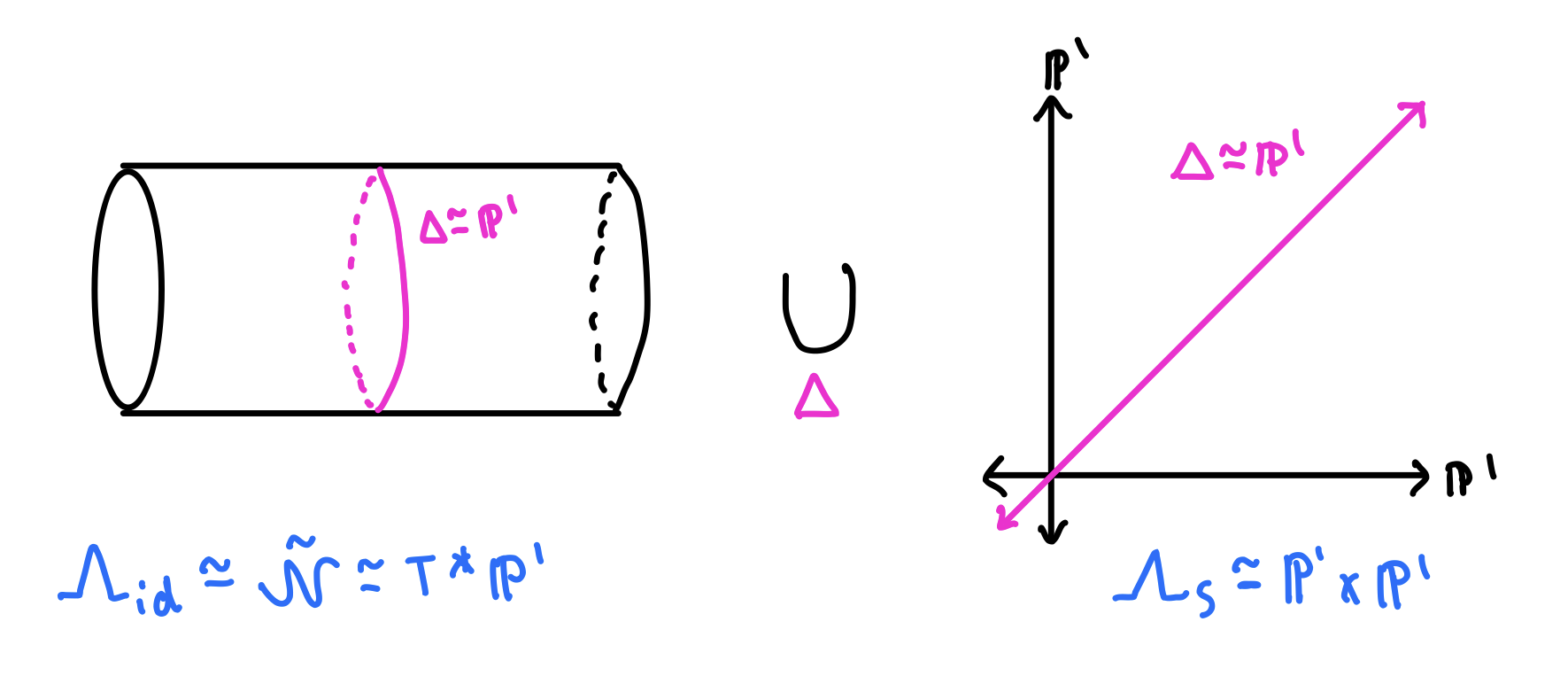} }}
\]
\end{example}
\begin{remark}
In general, the components range from $\Lambda_{\mathrm{id}} \simeq \widetilde{\mc{N}}$ to $\Lambda_{w_0} = \mc{B} \times \mc{B}$. They are not smooth in general, and their intersection pattern is extremely complicated. For example, it is not known in general when two $\Lambda_x$ and $\Lambda_y$ intersect in codimension $1$. 
\end{remark}

Last week we introduced the anti-spherical module 
\begin{align*}
    N=H \otimes_{H_f} \mathrm{sgn} &\simeq \Z[v^{\pm 1}][X] \\
    \theta_\lambda:=\theta_\lambda \otimes 1 &\mapsfrom e^\lambda
\end{align*}
for the affine Hecke algebra $H$ and established the {\bf Demazure--Lusztig formula} 
\begin{equation}
\label{bs action}
    b_s \cdot \theta_\lambda = (v^{-1} - v \theta_{-\alpha}) \left( \frac{\theta_\lambda - \theta_{s\lambda}}{1 - \theta_{-\alpha}} \right).
\end{equation}
Here $H_f =\langle H_s \rangle_{s \in S_f}$ is the finite Hecke algebra and $b_s:=H_s + v$ is the Kazhdan--Lusztig generator corresponding to the simple reflection $s=s_\alpha \in S_f$. 

\begin{remark}
It is not immediately obvious that the formula $\left( \frac{\theta_\lambda - \theta_{s \lambda}}{1-\theta_{-\alpha}} \right)$ gives an element in the lattice part of $H$. However, by rewriting 
\[
\left( \frac{\theta_\lambda - \theta_{s \lambda}}{1 - \theta_{-\alpha}} \right) = (\theta_\lambda + \theta_{\lambda - \alpha} + \cdots + \theta_{s\lambda + \alpha}),
\]
we see that $b_s \cdot \theta_\lambda \in \Z[v^{\pm 1}][X]$. (Here we assume $\langle \lambda, \alpha^\vee \rangle \ge 0$.)
\end{remark}

We want to show that a similar formula holds for the $K^{G \times \C^\times}(St)$-action on $K^{G \times \C^\times}(\widetilde{\mc{N}})$. We will first do so in the special case of $G=\SL_2$, then move on to the general formulation.

\subsection{The case of $\SL_2$}
We start with a baby version of the calculation. We identify $K^G(\PP^1)$ with $\Z[x^{\pm 1}]$ via the chain of isomorphisms 
\[
    K^G(\PP^1) \simeq K^G(G/B) \simeq K^B(\mathrm{pt}) \simeq \Z[x^{\pm 1}]
\]
Under this identification, $\mc{O}(m) \mapsto x^m$. We have a Cartesian square
\[
\begin{tikzcd}
\PP^1 \times \PP^1 \arrow[r, "p_2"] \arrow[d, "p_1"] & \PP^1 \arrow[d, "p_1'"]\\ 
\PP^1 \arrow[r, "p_2'"] & \mathrm{pt}
\end{tikzcd}
\]

\noindent
{\bf Claim 1:} The map $p_{1*}p_2^*:K^G(\PP^1) \rightarrow K^G(\PP^1)$ is given by $x^m \mapsto \frac{x^m - x^{-m-2}}{1-x^{-2}}$. 
\begin{proof}
By smooth base change, $p_{1*}p_2^* = p_2'^*p_{1*}'$. Then the claim follows from Weyl's character formula. 
\end{proof}

\vspace{3mm}
\noindent
{\bf Claim 2:} $p_{1*}(\mc{O}(-2, 0) \otimes p_2^*( - )):K^G(\PP^1) \rightarrow K^G(\PP^1)$ is given by $x^m \mapsto x^{-2} \left(\frac{x^m - x^{-m-2}}{1-x^{-2}}\right)$.
\begin{proof}
The notation $\mc{O}(m,n)$ refers to $p_1^*\mc{O}(m) \otimes p_2^*\mc{O}(n)$. By the projection formula, $p_{1*}(p_1^*(-2) \otimes p_2^*(-)) = \mc{O}(-2) \otimes p_{1*}p_2^*(-)$. The claim then follows from Claim 1. 
\end{proof}

Now we move up to $\widetilde{\mc{N}}$. Recall that the pull-back $q^*$ of the projection $\widetilde{\mc{N}}=T^*\PP^1 \xrightarrow{q} \PP^1$ gives an isomorphism 
\[
K^{G \times \C^\times}(\widetilde{\mc{N}}) \xleftarrow[q^*]{\sim} K^{G \times \C^\times}(\PP^1) = \Z[v^{\pm 1}][X]. 
\]
Recall that $\C^\times$ acts on $\widetilde{\mc{N}}$ via scaling by $z^2$ (equation (\ref{C star action})). For the rest of this section, set $\mc{O}:=\mc{O}_{\widetilde{\mc{N}}}$. 

We have projections
\[
\includegraphics[scale=0.38]{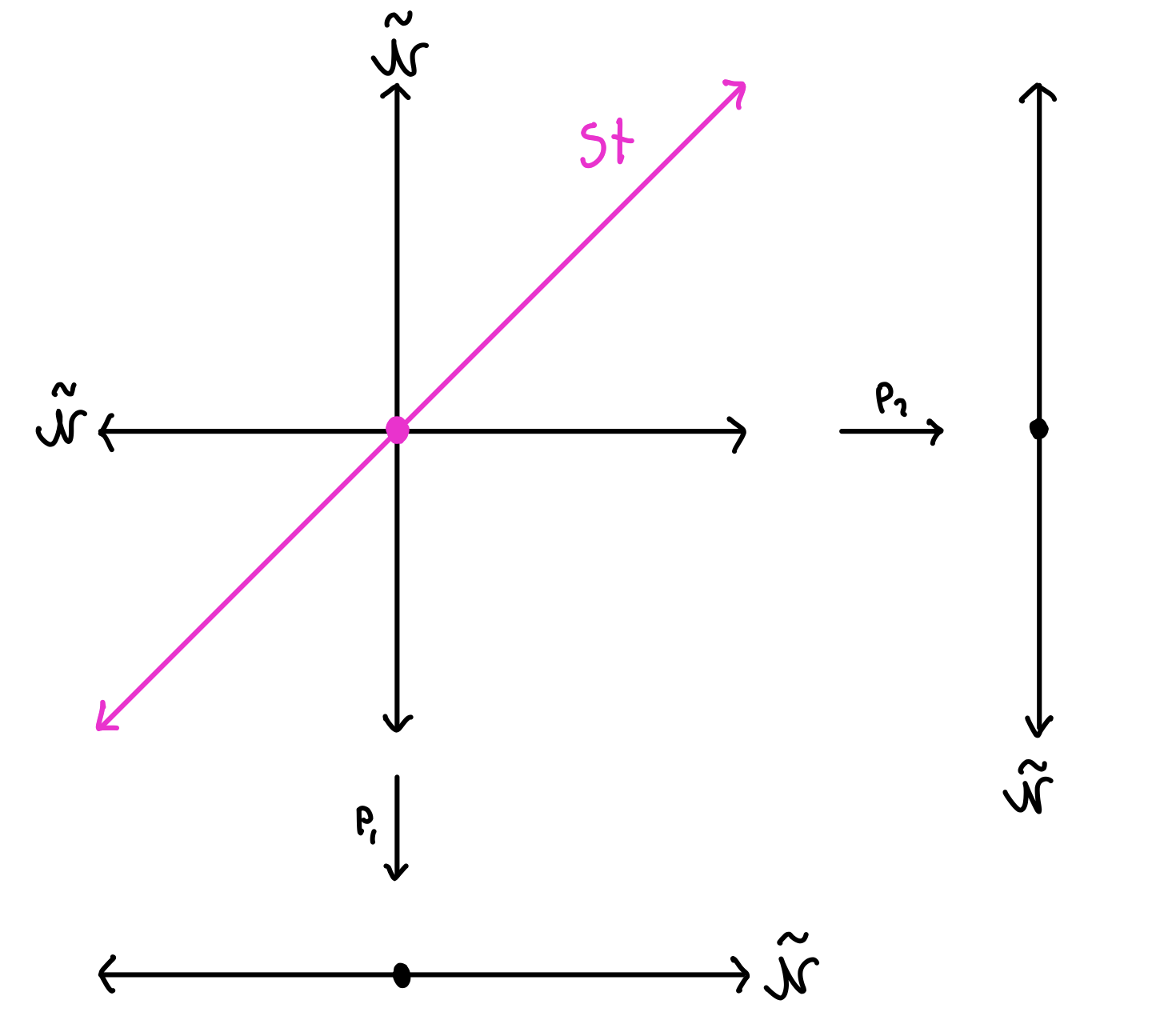} 
\]
Set 
\[
Q_s:= \mc{O}_{\PP^1 \times \PP^1}(-2, 0)
\]
on $\Lambda_s \simeq \PP^1 \times \PP^1 \subset St$. We will see that (up to some simple scalar factors), $Q_s$ acts via convolution on $K^{G \times \C^\times}(St)$ in the same way that $b_s$ acts on $N$. To do so, we want to calculate
\[
p_{1*}(Q_s \otimes p_2^*\mc{O}(m)) \in K^{G \times \C^\times}(\widetilde{\mc{N}}). 
\]

\noindent
{\bf Claim 3:} Let $\PP^1 \xhookrightarrow{i} \widetilde{\mc{N}}$. Then $p_{1*}(Q_s \otimes p_2^*\mc{O}(m))=\left( \frac{x^m - x^{-m-2}}{1-x^{-2}} \right) \left[i_*\mc{O}_{\PP^1}(-2)\right].$

\begin{exercise}
\begin{enumerate}
    \item Suppose $\pi: E \rightarrow Y$ is a vector bundle, $\mc{F}$ a quasi-coherent sheaf on $Y$, $\mc{G}$ a locally free sheaf on $Y$, and $i:Y \hookrightarrow E$ the zero section. Show that 
    \[
    i_*\mc{F} \otimes \pi^* \mc{G} = i_*(\mc{F} \otimes \mc{G}).
    \]
    \item Use 1. applied to $\widetilde{\mc{N}} \times \widetilde{\mc{N}} \rightarrow \PP^1 \times \PP^1$ to deduce Claim 3. 
\end{enumerate}
\end{exercise}

It remains to express $\left[i_*\mc{O}_{\PP^1}(-2)\right]$ in terms of our basis of $K^{G \times \C^\times}(\widetilde{\mc{N}})$; that is, to understand $i_*\mc{O}_{\PP^1}(-2)$ in terms of vector bundles. 

\vspace{5mm}
\noindent
{\bf Useful basic technique:} Given a vector bundle $q:V \rightarrow Y$, there is a bijection 
\[
\left\{ \begin{array}{c} \text{quasi-coherent sheaves} \\ \text{ on the total space of $V$} \end{array} \right\} \simeq \left\{ \begin{array}{c} \text{quasicoherent sheaves of} \\ \text{modules over $q_*\mc{O}_V\simeq \Sym^\bullet(V^*)$} \end{array} \right\}.
\]
(More generally, this holds for any affine morphism, see Hartshorne.) 

\vspace{5mm}
In our case, the projection 
\[
\widetilde{\mc{N}}=T^*\PP^1 \rightarrow \PP^1
\]
is the line bundle associated to $\mc{O}(-2)$, and  
\[
q_*\mc{O} \simeq \Sym^\bullet(\mc{O}_{\PP^1}(2))=\mc{O}_{\PP^1} \oplus \mc{O}_{\PP^1}(2) \oplus \mc{O}_{\PP^1}(4) \oplus \cdots . 
\]
Additionally, $\C^\times$ acts on $T^*\PP^1 \rightarrow \PP^1$ as multiplication by {\color{red} $z^{2}$} along the fibres, which corresponds to multiplication by {\color{red} $z^{-2}$} in the second factor in the direct sum decomposition above. Hence we have an exact sequence 
\[
{\color{red} z^{-2}} \mc{O}(2) \hookrightarrow \mc{O} \twoheadrightarrow i_* \mc{O}_{\PP^1}.
\]

\begin{remark}
Locally, this is the short exact sequence
\[
k[x] \xhookrightarrow{\cdot x} k[x] \twoheadrightarrow k,
\]
but globally, we have some twisting. This is a special example of the Koszul resolution of the zero section of a vector bundle. 
\end{remark}

Hence in $K^{G\times \C^\times}(\widetilde{\mc{N}})$, we have 
\[
\left[i_*\mc{O}_{\PP^1}\right] = [\mc{O}] - v^{-2}[\mc{O}(2)].
\]
By tensoring with $\mc{O}(-2)$ we obtain
\[
\left[ i_* \mc{O}_{\PP^1}(-2) \right] = \left[\mc{O}(-2)\right] - v^{-2}[\mc{O}].
\]
Putting it all together, we see that 
\[
x^m \xmapsto{[Q_{s}]*} (x^{-2}-v^{-2}) \left( \frac{x^m - x^{-m-2}}{1-x^{-2}}\right). 
\]
Hence
\[
x^m \xmapsto{-[vQ_s]*} (v^{-1} - vx^{-2}) \left( \frac{x^m - x^{-m-2}}{1-x^{-2}}\right). 
\]
We only need two last pieces of information to establish the Kazhdan--Lusztig isomorphism for $\SL_2$, which we leave as exercises.
\begin{exercise}
The map $\theta_m \mapsto x^{m-1}$ intertwines the actions of $b_s$ and $-[vQ_s]*$.  
\end{exercise}
\begin{exercise}
Both the representation $N$ of $H$ and the representation $K^{G \times \C^\times}(\widetilde{\mc{N}})$ of $K^{G \times \C^\times}(St)$ are faithful. 
\end{exercise}

\subsection{The general case}
Roughly speaking, in the $\SL_2$ example, we have seen the meaning in $K$-theory of the Demazure--Lusztig formula
\[
b_s \cdot \theta_\lambda = (v^{-1} - v \theta_{-\alpha}) \left(\frac{\theta_\lambda - \theta_{s \lambda}}{1-\theta_{-\alpha}}\right). 
\]
The first factor $(v^{-1} - v \theta_{-\alpha})$ comes from the ``Koszul resolution,'' and the second factor $\left(\frac{\theta_\lambda - \theta_{s \lambda}}{1-\theta_{-\alpha}}\right)$ from the ``push-pull for $\PP^1$''. We will now explain that the same philosophy works in general. 

\vspace{5mm}
\noindent
{\bf Relative cotangent bundle:} 
Let 
\[
Z \xrightarrow{f} Y
\]
be a smooth map. (Think submersion = fibration.) 
\[
\includegraphics[scale=0.25]{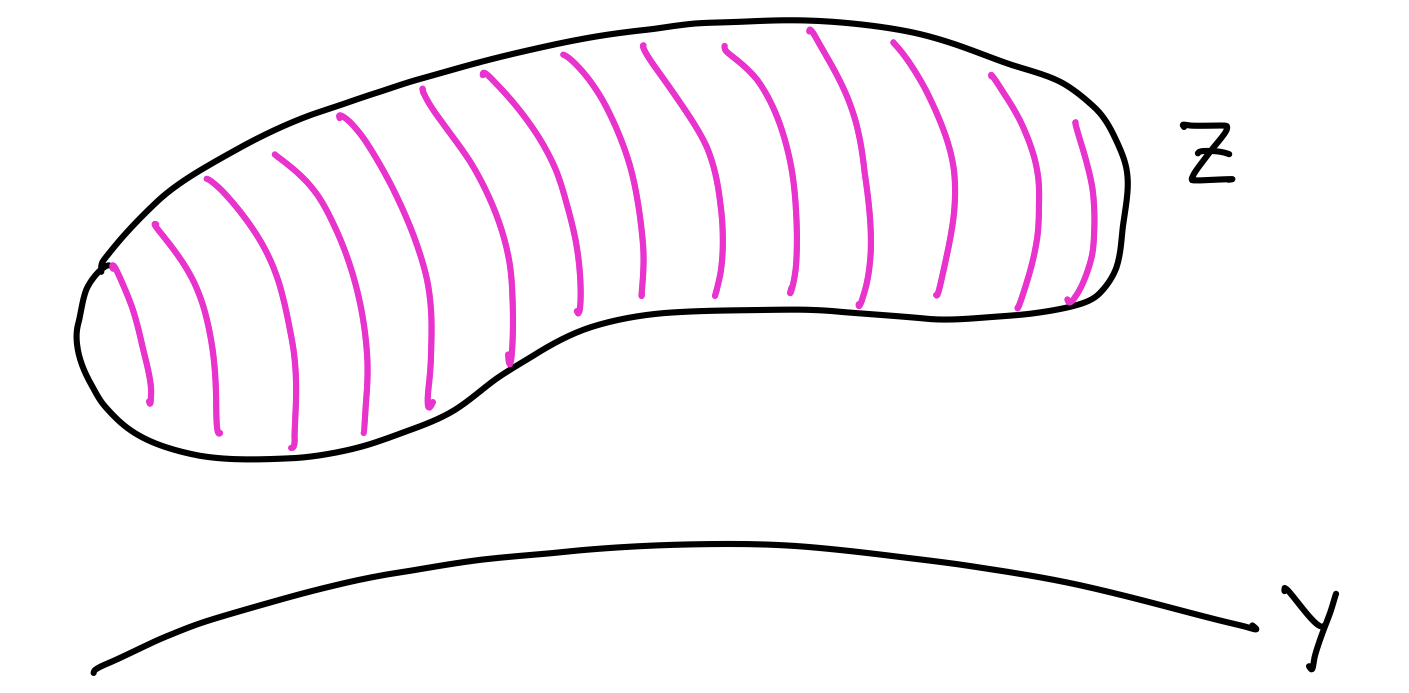} 
\]
The relative tangent bundle is 
\[
T_{Z/Y}:=\ker(df:TZ \rightarrow TY).
\]
We have a short exact sequence 
\[
T_{Z/Y} \hookrightarrow T_Z \twoheadrightarrow f^*T_Y. 
\]
Dually, 
\[
f^*T_Y^* \hookrightarrow T_Z^* \twoheadrightarrow T_{Z/Y}^*.
\]
A caricature: 
\[
\includegraphics[scale=0.4]{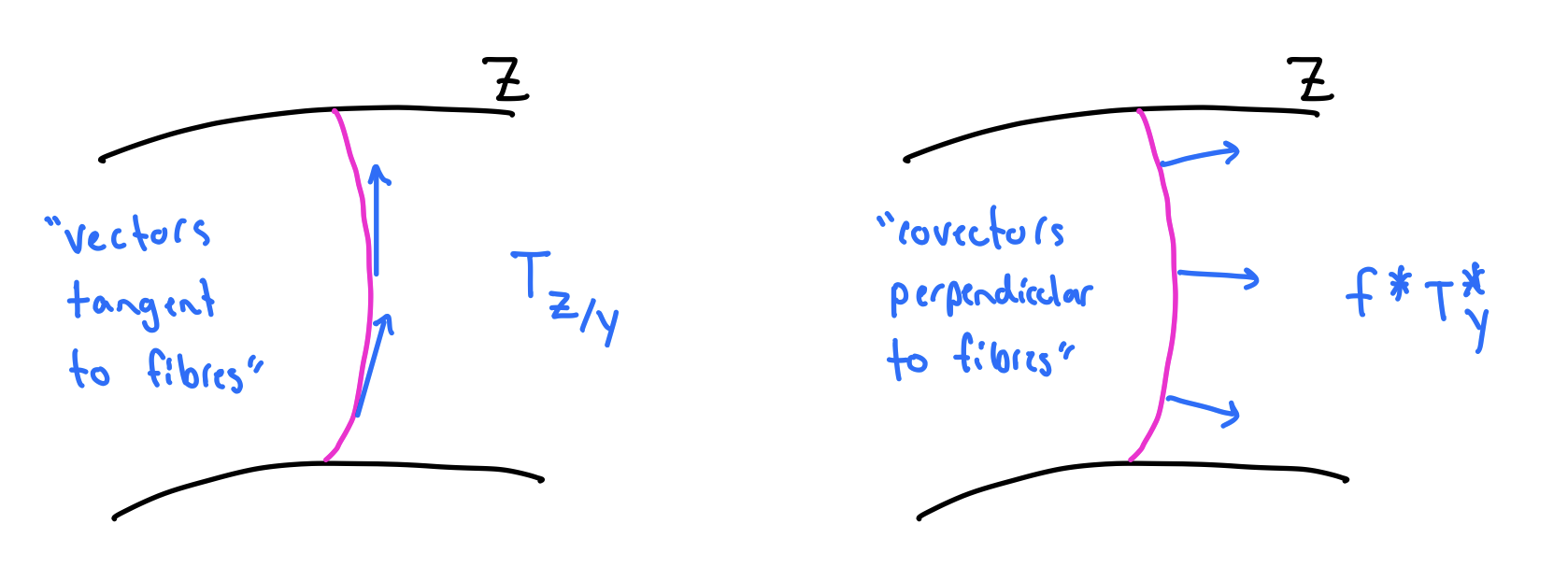} 
\]

In our setting, we have a $\PP^1$ fibration 
\[
\mc{B} \xrightarrow{\pi_s} \mc{B}_s, 
\]
where $\mc{B}_s$ is the variety of parabolic subalgebras of type $s$. 
\[
\includegraphics[scale=0.25]{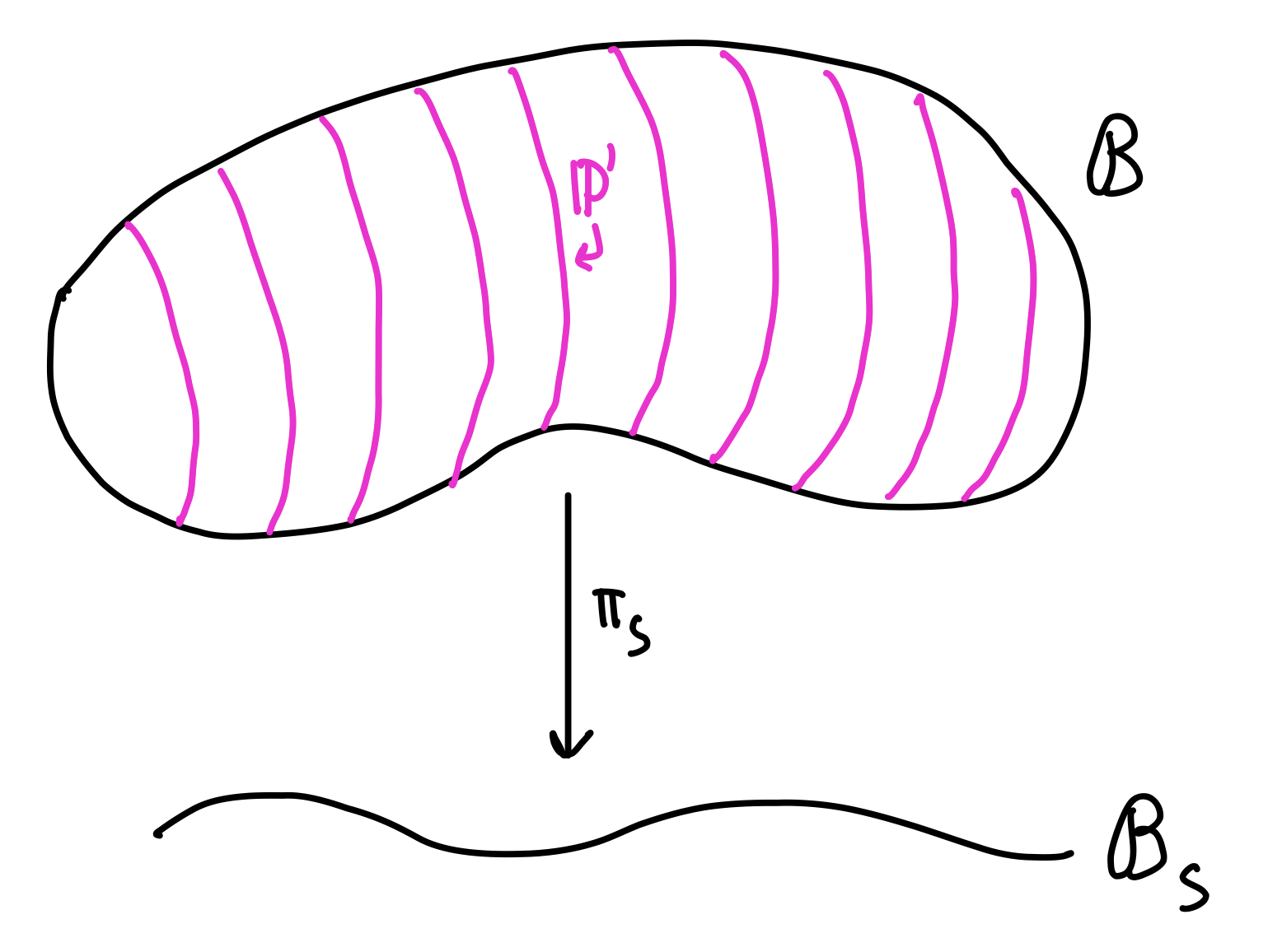} 
\]

\vspace{5mm}
\noindent
{\bf Fundamental Cartesian diagram:} 
\[
\begin{tikzcd}
\overline{\OO}_s \arrow[r, "p_1"] \arrow[d, "p_2"'] & \mc{B} \arrow[d, "\pi_s"] \\ \mc{B} \arrow[r, "\pi_s"] &\mc{B}_s 
\end{tikzcd}
\]
Here $\OO_s \subset \mc{B} \times \mc{B}$ is the $G$-orbit consisting of pairs of flags in relative position $s$. All maps in this diagram are $\PP^1$-fibrations. We identify $K^G(\mc{B}) \simeq \Z[X]$ via $\mc{O}_{\mc{B}}(\lambda) \mapsto e^\lambda$. 

\vspace{3mm}
\noindent
{\bf Claim 4:} $\pi_{s*}\pi_s^*: K^G(\mc{B}) \rightarrow  K^G(\mc{B})$ maps $e^\lambda \mapsto \frac{e^\lambda - e^{s(\lambda) - \alpha}}{1 - e^{-\alpha}}$.
\begin{proof}
Again, this is simply an instance of the Weyl character formula for $\PP^1$. 
\end{proof}
We claim that the corresponding Cartesian diagram ``upstairs'' is
\[
\begin{tikzcd}
\Lambda_s \arrow[r] \arrow[d] & \widetilde{\mc{N}}_s  \arrow[d] \\
\widetilde{\mc{N}}_s \arrow[r] &T^*\mc{B}_s
\end{tikzcd}
\]
To see that this is the correct lift of the first Cartesian square, we start with the lower right corner
\[
T^*\mc{B}_s = \{(\mf{p}, x) \mid x \in \mathrm{nilrad}(\pi_s(\mf{b})\},
\]
which implies that the upper right and lower left corners are
\[
\widetilde{\mc{N}}_s = \pi_s^*T^*_{\mc{B}/\mc{B}_s} = \{(\mf{b}, x) \mid x \in \mathrm{nilrad}(\pi_s(\mf{b})) \} \subset \widetilde{\mc{N}}. 
\]
Hence the fibre product completing the diagram in the upper left corner is given by 
\[
 \left\{(\mf{b}, \mf{b}', x) \mid \begin{array}{c} \mf{b}, \mf{b}' \text{ in relative position $s$, and } \\ x \in \mathrm{nilrad}(\pi_s(\mf{b}))=\mathrm{nilrad}(\pi_s(\mf{b}')) \end{array} \right\} = \Lambda_s. 
\]

Together, these two Cartesian squares form a ``Cartesian cube:''
\[
\begin{tikzcd}[row sep=scriptsize, column sep=scriptsize]
& \Lambda_s \arrow[dl] \arrow[rr] \arrow[dd, red] & & \widetilde{\mc{N}}_s \arrow[dl] \arrow[dd, red] \\
\widetilde{\mc{N}}_s \arrow[rr, crossing over] \arrow[dd, red] & & T^*\mc{B}_s \\
& \overline{\OO}_s \arrow[dl] \arrow[rr] & & \mc{B} \arrow[dl] \\
\mc{B} \arrow[rr] & & \mc{B}_s \arrow[from=uu, crossing over, red]\\
\end{tikzcd}
\]
Define
\[
Q_s:=q^*\Omega_{\overline{\OO}_s/\mc{B}}
\]
to be the relative $1$-forms with respect to the second projection. As earlier, Claim 4 gives 
\[
p_{1*}(Q_s \otimes p_2^* \mc{O}(\lambda)) = \left(\frac{e^\lambda - e^{s(\lambda) - \alpha}}{1-e^{-\alpha}}\right) \left[i_*\mc{O}_{\widetilde{\mc{N}}_s}(-\alpha)\right],
\]
where $\widetilde{\mc{N}}_s \hookrightarrow \widetilde{\mc{N}}$ is the inclusion. 

What remains is to express $\left[i_*\mc{O}_{\widetilde{\mc{N}}_s}(-\alpha)\right]\in K^{G \times \C^\times}(\widetilde{\mc{N}})$ in the basis of line bundles. Again, we can do so using a Koszul-type resolution. There is a short exact sequence of $B$-modules
\[
\mf{p}_s/\mf{b} \hookrightarrow \mf{g}/\mf{b} \twoheadrightarrow \mf{g}/\mf{p}_s. 
\]
The corresponding exact sequence of vector bundles on $\mc{B}=G/B$ is 
\[
L_\alpha \hookrightarrow T_\mc{B} \twoheadrightarrow \pi_s^*T_{\mc{B}_s},
\]
where $L_\alpha$ is the line bundle on $\mc{B}$ associated to $\alpha \in X$. Passing to symmetric algebras, we obtain the Koszul resolution
\[
{\color{red} z^{-2}} \Sym^\bullet(T_\mc{B})(\alpha) \hookrightarrow \Sym^\bullet(T_\mc{B}) \twoheadrightarrow \Sym^\bullet(\pi_s^*T_\mc{B}).
\]
\begin{remark}
This should be thought of as the vector bundle version of the short exact sequence 
\[
k[x_1, \ldots, x_n] \xhookrightarrow{\cdot x_n} k[x_1, \ldots, x_n] \twoheadrightarrow k[x_1, \ldots, x_{n-1}]. 
\]
The ${\color{red} z^{-2}}$ comes from the $\C^\times$-action.
\end{remark}
In other words,
\[
{\color{red} z^{-2}} \mc{O}_{\widetilde{\mc{N}}}(\alpha) \hookrightarrow \mc{O}_{\widetilde{\mc{N}}} \twoheadrightarrow i_*\mc{O}_{{\widetilde{\mc{N}}}_s}.
\]
{\bf The result:}
\[
\left[i_*\mc{O}_{\widetilde{\mc{N}}_s}(-\alpha)\right] = \left[ \mc{O}_{\widetilde{\mc{N}}}(-\alpha)\right] - v^{-2} \left[ \mc{O}_{\widetilde{\mc{N}}}\right].
\]
From here, the proof follows from two exercises analogous to those in the previous section. 
\begin{exercise}
Check that $\theta_\lambda \mapsto e^{\lambda \pm \rho}$ intertwines $b_s \cdot$ and $-[vQ_s]*$. 
\end{exercise}
\begin{exercise}
Complete the proof by showing that $N$ (resp. $K^{G \times \C^\times}(\widetilde{\mc{N}})$) are faithful modules over $H$ (resp. $K^{G \times \C^\times}(\widetilde{\mc{N}})$). 
\end{exercise}

\begin{remark}
Geordie isn't quite sure about whether we should have $\alpha$ or $-\alpha$ above, so the reader should take the final arguments with a grain of salt. 
\end{remark}

%% file: lecture-23.tex
\section{Lecture 23: Gaitsgory's central sheaves}
\label{lecture 23}

Last week we finished the proof of the Kazhdan--Lusztig isomorphism:
\[
K^{G^\vee \times \C^\times}(St) \simeq H.
\]
This is an isomorphism of algebras. We established this isomorphism by proving that two faithful modules for these algebras are isomorphic: 
\begin{equation}
\label{module iso}
K^{G^\vee \times \C^\times}(\widetilde{\mc{N}}) \simeq N = H \otimes_{H_f} \mathrm{sgn}.
\end{equation}
For the next several weeks (months?), we will work toward categorifying the Kazhdan--Lusztig isomorphism via Bezrukavnikov's equivalence. The first step towards a categorification is the Arkhipov-Bezrukavnikov theorem, which categorifies (\ref{module iso}). 
\begin{theorem}
\label{AB}
\cite{AB} 
\[
\begin{tikzcd}
D^b(\Coh^{G^\vee \times \C^\times}(\widetilde{\mc{N}})) \arrow[r, dash, "\sim"] \arrow[d] &D^\mathrm{mix}_{IW} \arrow[d] \\
D^b(\Coh^{G^\vee}(\widetilde{\mc{N}})) \arrow[r, dash, "\sim"] & D_{IW}
\end{tikzcd}
\]
Here $D_{IW}^\mathrm{mix}$ and $D_{IW}$ are the mixed and unmixed version of the antispherical category. 
\end{theorem}

\begin{remark}
If one wants to convince oneself that this is a deep equivalence, one only needs to note that the ``easy'' $\C^\times$-action on the LHS corresponds to the ``hard'' Frobenius action on the RHS! 
\end{remark}

For the remainder of the lecture, we will build the machinery of Gaitsgory's central sheaves, then focus on the example of the natural representation of $\GL_2$. We will return to the Arkhipov-Bezrukavnikov theorem next week. 

\subsection{A change of notation in the Hecke algebra}
We will make a slight change of notation from previous lectures. In the affine Hecke algebra, replace
\[
    H_x \rightsquigarrow \delta_x 
\]
for all $x \in W_f$, so our standard basis of the finite Hecke algebra is now called $\{\delta_x \}_{x \in W_f}$ and the Kazhdan--Lusztig basis is $\{b_x\}_{x \in W_f}$. We make this change so that from now on, small letters indicate objects in the Hecke algebra $H$ and big letters indicate objects in the Hecke category $\mc{H}$.

\subsection{Gaitsgory's central sheaves}
For more details on this construction, see notes from Emily's IFS talks on the Beilinson-Drinfeld Grassmannian\footnote{Talk notes available at \url{https://sites.google.com/view/ifssydney/home}}. 

Recall that within the affine Hecke algebra $H$, we have a large commutative subalgebra 
\[
\mc{L} = \bigoplus_{\lambda \in X^\vee} \Z[v^{\pm 1}]\theta_\lambda \subset H.
\]
In Lecture \ref{lecture 12} we discussed Bernstein's description of the center of $H$ in terms of this subalgebra.
\begin{theorem}
\label{bernstein center}
\[
Z(H) \simeq \mc{L}^{W_f} \simeq \Z[v^{\pm 1}] \otimes_{\Z} [\Rep G^\vee]. 
\]
\end{theorem}
Another perspective from which we can view this description of the center of the affine Hecke algebra was discussed in Emily's IFS talk. Temporarily, let $\mc{K}=\mathbb{F}_q((t)) \supset \mc{O}=\mathbb{F}_q[[t]]$, and build Hecke algebras
\[
\mc{H}^\mathrm{aff}=\Fun^{\mathrm{c.s.}}_I(\mc{F}\ell, \C), \hspace{5mm} \mc{H}^\mathrm{sph} = \Fun^\mathrm{c.s.}_{G(\mc{O})}(\mc{G}r, \C).
\]
Here ``c.s.'' stands for ``compactly supported'', $\mc{F}\ell=G(\mc{K})/I$ is the affine flag variety and $\mc{G}r=G(\mc{K})/G(\mc{O})$ is the affine Grassmannian. These function spaces gain the structure of algebras via convolution. The spherical Hecke algebra $\mc{H}^{\mathrm{sph}}$ is a commutative convolution algebra, $\mc{H}^\mathrm{aff}$ is not. There are natural maps
\[
\begin{tikzcd}
\mc{H}^\mathrm{aff} \arrow[rd, "\int_{G(\mc{O})/I}"] & \\
 & \Fun^\mathrm{c.s.}_{G(\mc{O})}(\mc{F}\ell, \C).\\
\mc{H}^\mathrm{sph} \arrow[ur, "\text{pull-back}"] &  
\end{tikzcd}
\]
Under these maps, the image of the center $Z(\mc{H}^\mathrm{aff})$ and the image of $\mc{H}^\mathrm{sph}$ agree:
\[
\begin{tikzcd}
Z(\mc{H}^\mathrm{aff}) \arrow[rd, "\sim"] & \\ 
& \text{images agree} \\
\mc{H}^\mathrm{sph} \arrow[ru, "\sim"] & 
\end{tikzcd}
\]
This provides an isomorphism
\begin{equation}
\label{isomorphism Emily}
Z(\mc{H}^\mathrm{aff}) \simeq \mc{H}^\mathrm{sph}. 
\end{equation}

The analogous statement for the affine Hecke algebra $H$ is as follows. Let $w_f \in W_f$ be the longest element, and 
\[
b_{w_f}=\sum_{x \in W_f} v^{\ell(w_f) - \ell(x)}\delta_x \in H
\]
the corresponding Kazhdan--Lusztig basis element. Define 
\[
H^\mathrm{sph}:=(b_{w_f}\cdot H) \cap (H \cdot b_{w_f}).  
\]
By the Satake isomorphism, 
\[
H^\mathrm{sph} \simeq \mc{L}^{W_f}.
\]
Then we have maps
\[
\begin{tikzcd}
H \arrow[dr, "h \mapsto h \cdot b_{w_f}"] &  \\
 & H\cdot b_{w_f},\\
 H^\mathrm{sph} \arrow[ur, "\text{identity}"] & 
\end{tikzcd}
\text{ sending }\hspace{2mm}
\begin{tikzcd}
Z(H) \arrow[dr, "\sim"] &  \\
 & \text{images agree}.\\
 H^\mathrm{sph}\simeq \mc{L}^{W_f} \arrow[ur, "\sim"] & 
\end{tikzcd}
\]
This gives us the isomorphism in Theorem \ref{bernstein center}: 
\begin{align*}
    Z &\xrightarrow{\sim} H^\mathrm{sph} \\
    h &\mapsto h \cdot b_{w_f}
\end{align*}

Gaitsgory lifted this statement to the level of categories. For the remainder of the lecture set $\mc{O} = \C[[t]], \mc{K} = \C((t))$, $\mc{G}r=G(\mc{K})/G(\mc{O})$, $\mc{F}\ell=G(\mc{K})/I$. Recall that 
\[
\pi:\mc{F}\ell \rightarrow \mc{G}r
\]
is a $G/B$-fibration. We replace
\begin{align*}
\mc{H}^\mathrm{sph} &\rightsquigarrow \Perv_{G(\mc{O})}(\mc{G}r), \text{ ``Satake category,'' and }\\
\mc{H}^\mathrm{aff} &\rightsquigarrow \Perv_I(\mc{F}\ell), \text{ ``Hecke category ''}. 
\end{align*}
Both of these categories obtain a monoidal structure, see Emily's IFS talk. Gaitsgory upgraded the isomorphism \ref{isomorphism Emily} to a central functor. 
\begin{theorem}
There exists a central functor\footnote{This means that for all $\mc{F} \in \Perv_{G(\mc{O})}(\mc{G}r), \mc{G} \in \Perv_I(\mc{F}\ell)$, $Z(\mc{F}) * \mc{G}$ is perverse, so is $\mc{G} * Z(\mc{F})$, and we have a canonical isomorphism $Z(\mc{F}) * \mc{G} \simeq \mc{G} * Z(\mc{F})$.}
\[
Z: \Perv_{G(\mc{O})}(\mc{G}r) \rightarrow \Perv_I(\mc{F}\ell).
\]
Moreover, the diagram 
\[
\begin{tikzcd}
D^b_I(\mc{F}\ell) \arrow[rd, "\pi_*"] & \\
& D^b_I(\mc{G}r)\\
\Perv_{G(\mc{O})}(\mc{G}r) \arrow[ru, "\text{forget}", "\text{equivariance}"'] & 
\end{tikzcd}
\]
commutes. 
\end{theorem}
\begin{remark}
\begin{enumerate}
    \item It is best to think of a ``central functor to a monoidal category $(M, *)$'' as a ``functor to the Drinfeld centre of $(M, *)$''. Geordie is not 100\% sure, but he thinks that $Z$ actually lands in the {\em symmetric center} (i.e. set of objecs with symmetric braiding) of the Drinfeld center. 
    \item Everything that Bezrukavnikov does is built on $Z$, so we should get used to thinking about it!
    \item If $\mc{G}', \mc{G} \in \Perv_I(\mc{F}\ell)$, then usually $\mc{G}' * \mc{G}$ is not perverse. It's a miracle that $Z(\mc{F})*\mc{G}$ always is! 
\end{enumerate}
\end{remark}

\subsection{Extended example: the natural representation of $\GL_2$} 

For the remainder of this lecture, we will work out the details of this construction for the simplest non-trivial example: the natural representation of $\GL_2$. This example is very beautiful and instructive, so we will discuss it in detail. 

Set $G=\GL_2$, $G^\vee = \GL_2$, and $\mathrm{nat}=$ the natural representation of $\GL_2$. Let $\mc{F}_\mathrm{nat} \in \Perv_{G(\mc{O})}(\mc{G}r)$ be the corresponding perverse sheaf under geometric Satake. Our goal is to describe $Z(\mc{F}_\mathrm{nat})$. Here is an outline of what we'll do:

\vspace{3mm}
\noindent
{\bf Algebra:} (See exercises at end of section.) To start, we have 
\[
H_{\GL_2}=H_{\SL_2}\ltimes \langle \delta_{\varpi} \rangle = \langle \delta_s, \delta_{s_0}, \delta_\varpi \rangle=H_f \otimes \Z[v^{\pm 1}][\theta_1, \theta_2]. 
\]
Here $\delta_\varpi$ is the generator of length zero elements, $s \in S_f$ is the finite simple reflection, and $s_0$ is the affine simple reflection. The elements $\theta_1=\delta_\varpi \delta_s$ and $\theta_2=\delta_\varpi \delta_{s_0}^{-1}$ are the Bernstein generators. Hence 
\[
z_\mathrm{nat} = \theta_1 + \theta_2 = \delta_\varpi(\delta_s + \delta_{s_0}^{-1}) = \delta_\varpi(\delta_s^{-1} + \delta_{s_0}).
\]
A natural question arises: How can we produce $z_\mathrm{nat}$ geometrically? 

\vspace{3mm}
\noindent
{\bf Geometry:} Let $k \in \{\Q, \R, \C\}$. Both $\mc{F}\ell$ and $\mc{G}r$ have components which are indexed by $\Z$ (measuring the valuation of the determinant of the lattice/the polynomial degree of the representation under geometric Satake), and the projection 
\[
\mc{F}\ell_{\GL_2}=\mc{F}\ell_{\SL_2} \times \Z \rightarrow \mc{G}r_{\GL_2} = \mc{G}r_{\SL_2} \times \Z
\]
is a $\PP^1$ fibration. The natural representation corresponds to a sheaf $k_{\PP^1}[1]$ on $\mc{G}r_{\SL_2} \times \{1\}$. Hence all of our arguments will take place on the components $\mc{G}r_{\SL_2} \times \{1\}$ and $\mc{F}\ell_{\SL_2} \times \{1\}$, so we can ignore the rest. 

\vspace{3mm}
\noindent
{\bf Gaitsgory's family:} Gaitsgory constructs a family of varieties over the disc:
\[
\includegraphics[scale=0.4]{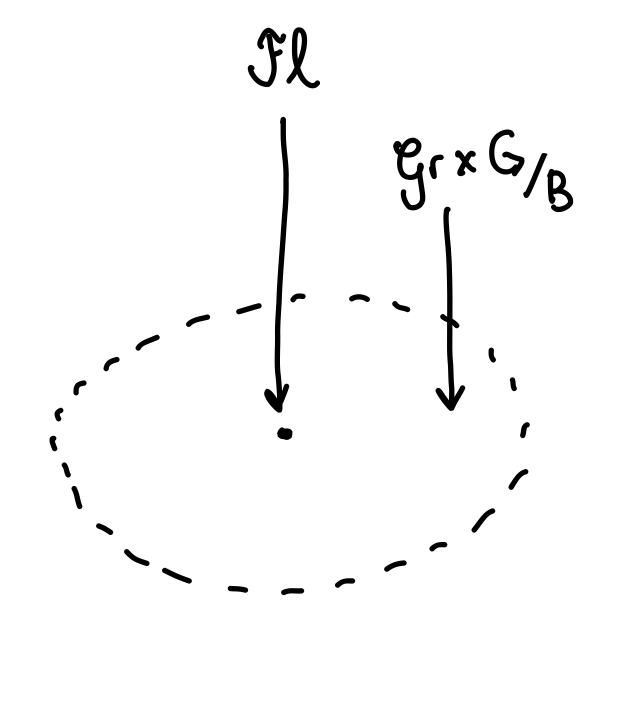}
\]
The $\PP^1 \times \{B/B\}$ on which the sheaf $k_{\PP^1}[1]$ is supported in the fibre $\mc{G}r \times G/B$ degenerates to $\PP^1 \times \PP^1$ (the intersection of two Schubert curves) in $\mc{F}\ell$:
\[
\includegraphics[scale=0.25]{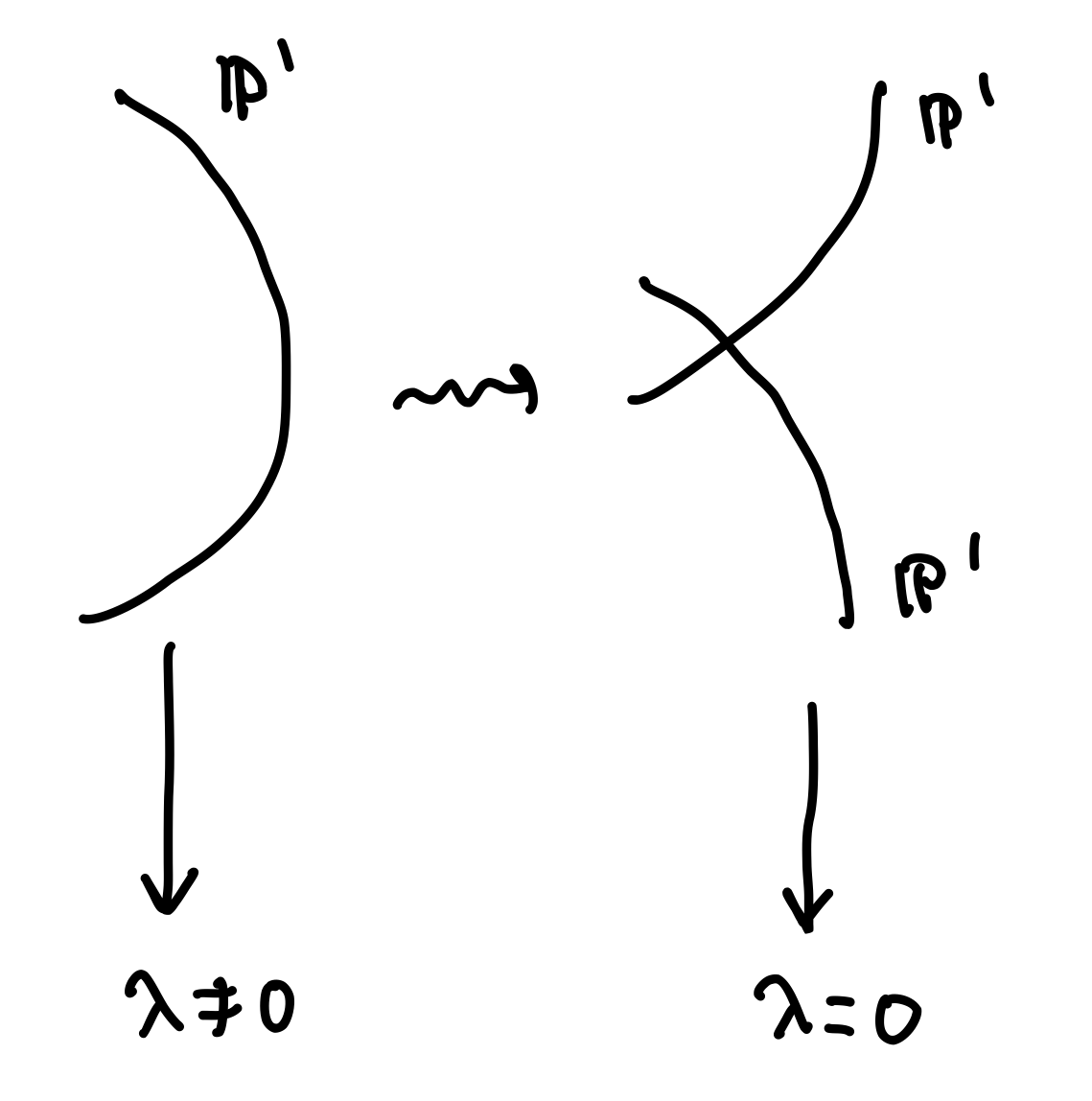}
\]
We will see that $z_\mathrm{nat}$ is categorified by nearby cycles. Moreover, we will see that the ``Wakimoto filtration''
\[
\includegraphics[scale=0.25]{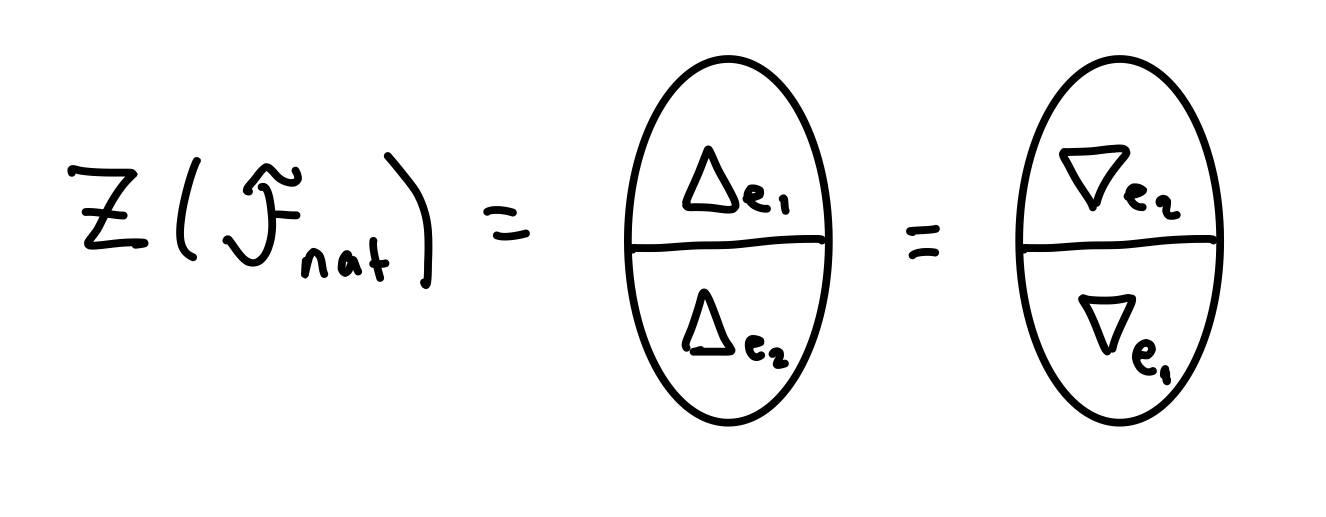}
\]
categorifies the relationships 
\[
z_\mathrm{nat}=\delta_\varpi \delta_s + \delta_\varpi \delta_{s_0}^{-1} = \delta_\varpi \delta_s^{-1} + \delta_\varpi \delta_{s_0}
\]
in $H$. Here it goes! 

\vspace{5mm}
\noindent
{\bf Beilinson gluing:} We want to understand perverse sheaves on $\{xy=0\} \subset \PP^2$; i.e. two $\PP^1$'s meeting transversally at a point. To do so, we can start by examining perverse sheaves on $\{xy=0\} \subset \C^2$, with stratification 
\[
\Lambda = \{0\} \sqcup \C^\times_{x\text{-axis}} \sqcup \C^\times_{y\text{-axis}}. 
\]
Beilinson gluing (Lectures \ref{lecture 17} and \ref{lecture 18}) tells us how! Take $f=x+y$. The zero locus of $f$ restricted to $\{xy=0\}$ is $\{0\}$. By Beilinson gluing,
\begin{align*}
\Perv_\Lambda(\{xy=0\})&=\left\{ \begin{array}{c} \mc{F} \text{ perverse on $\C^\times_x \sqcup \C^\times_y$}, \\ V_0 \text{ perverse on }\{0\} \end{array} + \begin{tikzcd}[ ampersand replacement=\&] \psi_f(\mc{F})\arrow[dr] \arrow[rr, "\mu - 1"] \& \&\psi_f(V) \\ \& V_0 \arrow[ur] \& \end{tikzcd} \right\}\\
&=\left\{ \left. \begin{tikzcd}
[ampersand replacement=\&] V_y \arrow[loop left, "\mu_y"] \arrow[dr, "c_y", shift right] \& \& V_x \arrow[loop right, "\mu_x"] \arrow[dl, "c_x", shift right] \\ 
\& V_0 \arrow[ul, "v_y", shift right] \arrow[ru, "v_x", shift right]\& 
\end{tikzcd} \right| \begin{array}{c} \text{ for all }z \in \{x, y\}, \\ v_z \circ c_z = \mu_z-1, \text{ and } \\
c_y \circ v_x = 0 = c_x \circ v_y \end{array} \right\}. 
\end{align*}
Hence if $\overline{X}=\PP^1_x \cup \PP^1_y=\{xy=0\} \subset \PP^2$ (these will be our intersection of Schubert varieties later), and 
\[
\Lambda = \{0\} \sqcup \C_x \sqcup \C_y,
\]
then we obtain the same description as above, except that the point at infinity forces $\mu_x=\mu_y=1$. In other words, 
\[
\Perv_\Lambda(\overline{X}) = \left\{ \left. \begin{tikzcd}
[ampersand replacement=\&] V_y  \arrow[dr, shift right] \& \& V_x  \arrow[dl, shift right] \\ 
\& V_0 \arrow[ul,shift right] \arrow[ru, shift right]\& 
\end{tikzcd} \right| \begin{array}{c} v_z \circ c_z = 0 \\  \text{ for } z \in \{x, y\}, \text{ and } \\ c_y \circ v_x = 0 = c_x \circ v_y \end{array} \right\}. 
\]

\begin{exercise}
Let 
\[
\mc{Y}=\{ ((x:y:z), \lambda) \in \PP^2 \times \mathbb{A}^1 \mid xy=\lambda z^2 \} \subset \PP^2 \times \mathbb{A}^1, 
\]
and $\mc{Y}_0=\mc{Y} \backslash f^{-1}(0)$, where $f:\PP^2 \times \mathbb{A}^1 \rightarrow \mathbb{A}^1$ is projection:
\[
\includegraphics[scale=0.3]{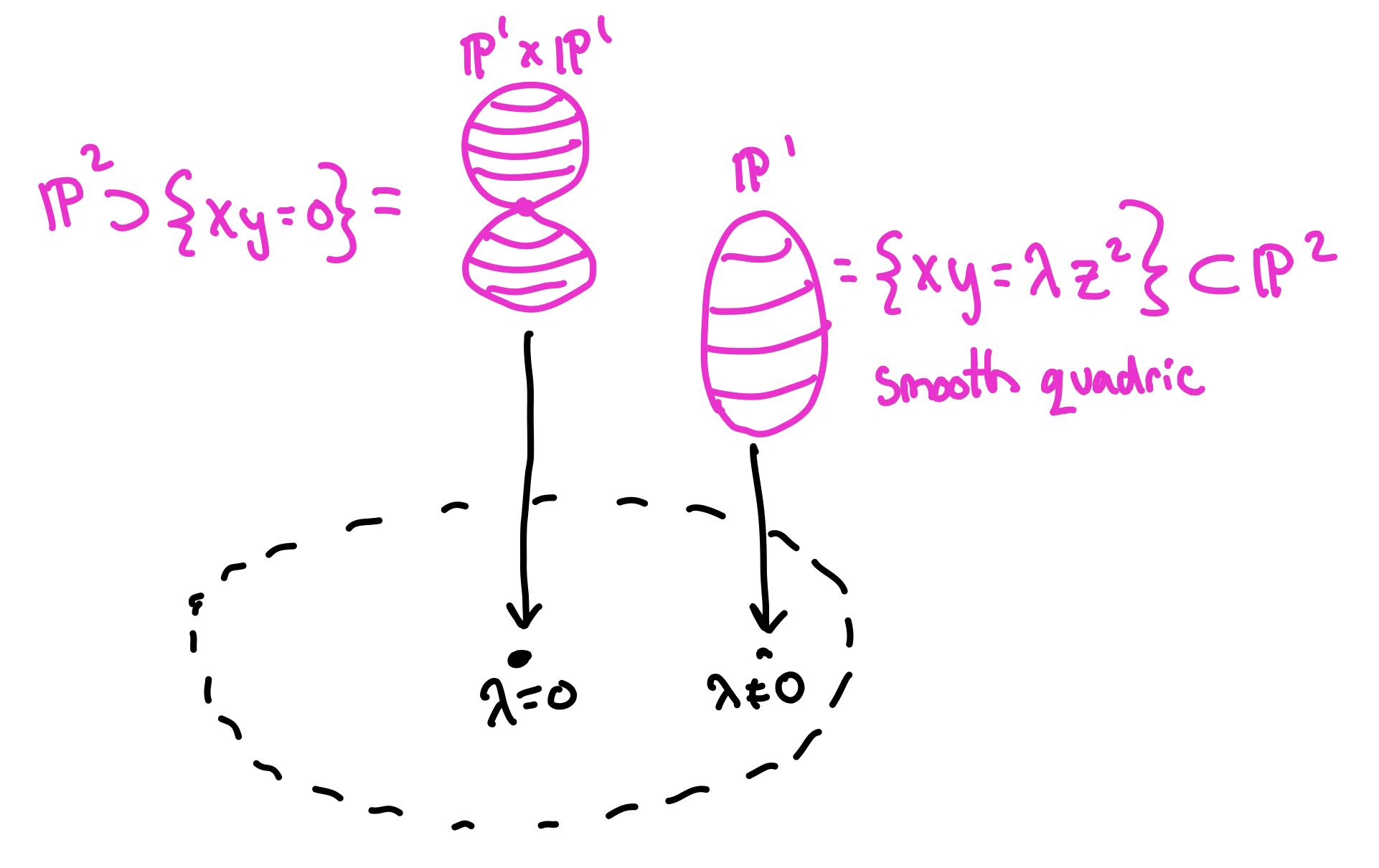}
\]
\end{exercise}
\begin{enumerate}
    \item Show that $\psi_f k_{\mc{Y}_0}[2]$ is described under the equivalence above by the diagram 
    \[
    \includegraphics[scale=0.3]{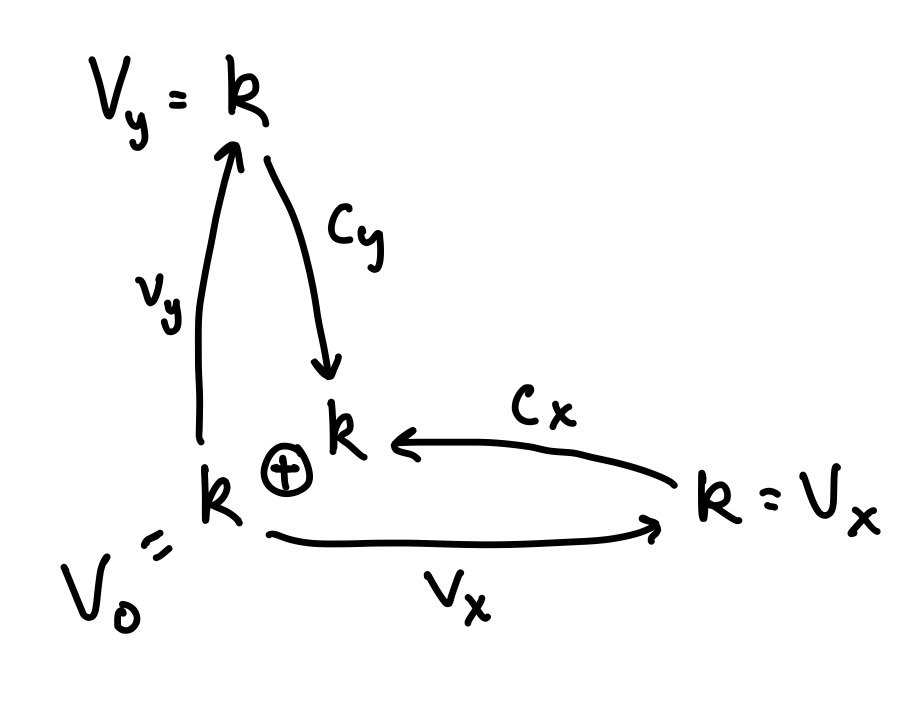}
    \]
    \item Deduce that $\psi_f k_{\mc{Y}_0}[2]$ has composition series 
    \[
    \includegraphics[scale=0.25]{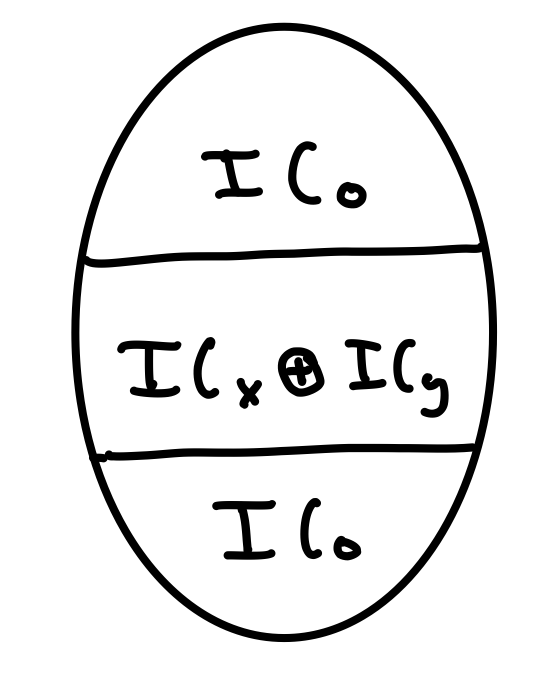}
    \]
    \item Check that the monodromy {\color{red} $\mu$} is given by 
    \[
    \includegraphics[scale=0.25]{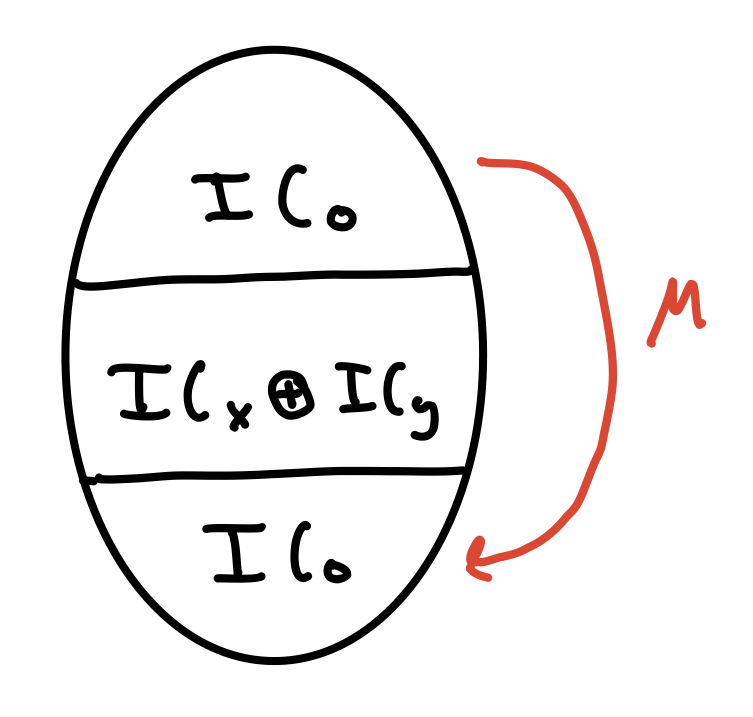}
    \]
\end{enumerate}
\begin{remark}
Let 
\[
\C_x \xhookrightarrow{j} \overline{X} \xhookleftarrow{j'} \C_y
\]
be the natural inclusions. Then 
\[
\includegraphics[scale=0.3]{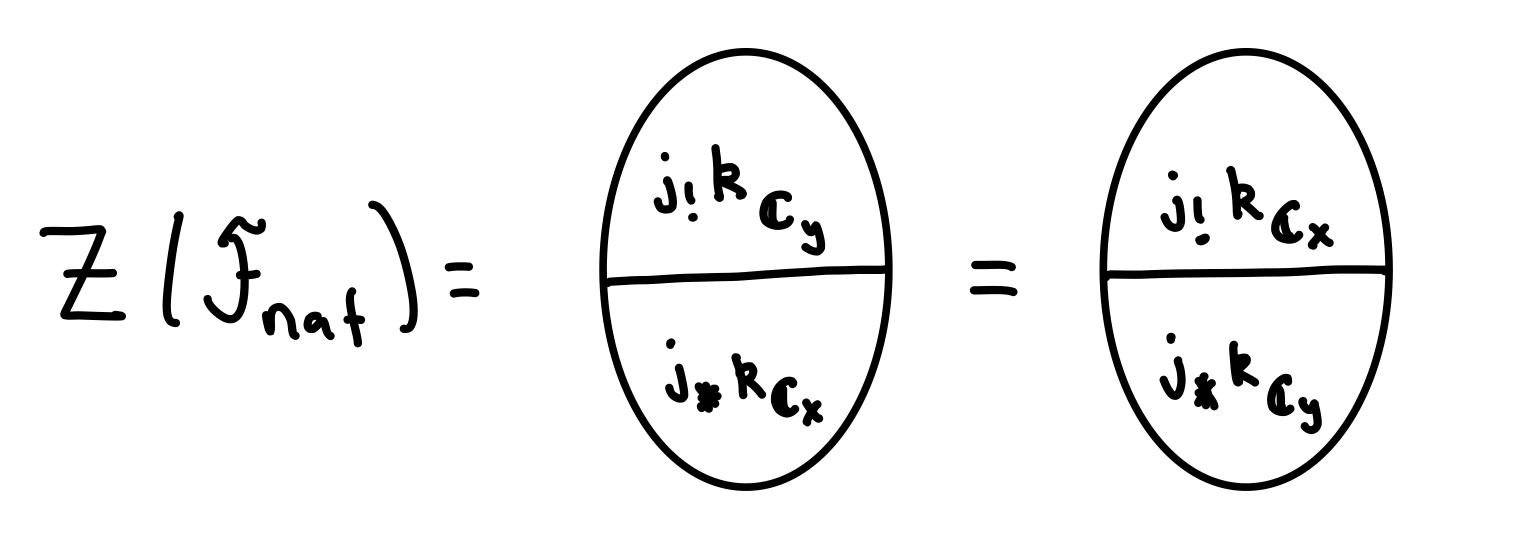}
\]
This is the {\bf Wakimoto filtration}, which categorifies the relation 
\[
z_\mathrm{nat}=\delta_\varpi \delta_s + \delta_\varpi \delta_{s_0}^{-1} = \delta_\varpi \delta_s^{-1} + \delta_\varpi \delta_{s_0}.
\]
in $H$.
\end{remark}

\noindent
{\bf Connect this all to Gaitsgory's picture:} Let $\mathscr{E}_\mathrm{triv}$ be the trivial rank $2$ vector bundle on $\mathbb{A}^1$, and $\mc{L}_\mathrm{triv}=\C[t] e_1 \oplus \C[t] e_2$ the trivial lattice in $\C((t))^2$. In Emily's IFS talks, she introduced the Beilison--Drinfeld Grassmannian, $\mc{BD}$, which provides a $\mc{G}r$-fibration over $\mathbb{A}^1$:
\[
\mc{BD}:=\left\{ (x, \mathscr{E}, \beta) \left| \begin{array}{c} \mathscr{E} \text{ rank 2 vector bundle} \\ \beta:\mathscr{E}|_{X-\{x\}} \xrightarrow{\sim} \mathscr{E}_\mathrm{triv}|_{X - \{x\}}  \end{array} \right. \right\} \xrightarrow{\mc{G}r} \mathbb{A}^1.
\]
She also introduced Gaitsgory's souped-up version of the the Beilinson-Drinfeld Grassmannian, which in the example of $\GL_2$ is a $\mc{F}\ell$-fibration over $\{0\}$ and a $\mc{G}r \times \PP^1$-fibration over $\mathbb{A}^1 \backslash \{0\}$:
\[
\begin{tikzcd}
& \left\{ (x, \mathscr{E}, \beta, \mc{F}) \left| \begin{array}{c} \mathscr{E} \text{ rank 2 vector bundle} \\ \beta:\mathscr{E}|_{X-\{x\}} \xrightarrow{\sim} \mathscr{E}_\mathrm{triv}|_{X - \{x\}} \\ \mc{F} \text{ flag in } \mathscr{E}_0 \end{array} \right. \right\} \arrow[dl, "\mc{F}\ell"'] \arrow[dr, "\mc{G}r \times \mathbb{P}^1"]
 & \\ 
\{0\} & & \mathbb{A}^1 \backslash \{0\} \end{tikzcd}
\]
Note that a flag in $\mathscr{E}_0$ is simply the choice of a line in the two-dimensional vector space $\mathscr{E}_0$.
Within $\mc{BD}$ we have a finite-dimensional closed subvariety $\mc{G}\subset \mc{BD}$. Under the lattice description of $\mc{G}r_{\GL_2}$,
\[
\mc{G}:=\{ \mc{L} \subset \mc{L}_\mathrm{triv} \mid \dim \mc{L}_\mathrm{triv}/\mc{L} = 1 \}. 
\]
We can cover $\mc{G}$ with two charts: if $\lambda$ is the coordinate on  $\mathbb{A}^1$, 
\[
U_0=\left\langle \bp 1 \\ a \ep, \bp 0 \\ t-\lambda \ep  \right\rangle, \hspace{2mm} U_\infty = \left\langle \bp b \\ 1 \ep, \bp t-\lambda \\ 0 \ep \right\rangle. 
\]
Each chart is isomorphic to $\mathbb{A}^2$, and we see that $\mc{G}$ is a trivial $\PP^1$-bundle over $\mathbb{A}^1$. 

The analogous subvariety in Gaitsgory's version is 
\[
\mc{Y} = \left\{ (\mc{L}, \ell) \left| \begin{array}{c} \mc{L} \subset \mc{L}_\mathrm{triv} \text{ a lattice s.t.} \\ \dim \mc{L}_\mathrm{triv}/\mc{L} = 1, \text{ and} \\ \ell \subset \mc{L}/t\mc{L} \end{array} \right. \right\} \xrightarrow{\PP^1\text{-bundle}} \mc{G}. 
\]
Consider the following closed subvariety of $\mc{Y}$:
\[
\mc{Y}^\mathrm{Sp} = \{ (\mc{L}, \ell) \mid \ell \subset \ker(\mc{L} \hookrightarrow \mc{L}_\mathrm{triv} \twoheadrightarrow \mc{L}_\mathrm{triv}/t \mc{L}_\mathrm{triv}=\C \oplus \C \twoheadrightarrow \C \oplus \C / (\C \oplus 0)) \}.
\]

What are the fibres of $\mc{Y}^\mathrm{Sp}$ over $\mathbb{A}^1$? Well, for any $(\mc{L}, \ell) \in \mc{Y}^\mathrm{Sp}$, $\mc{L}$ fits into an exact sequence
\[
\mc{L} \hookrightarrow \mc{L}_\mathrm{triv} \twoheadrightarrow \mc{O}_x,
\]
where $\mc{O}_x$ is a skyscraper. There are two cases:
\begin{enumerate}
    \item If $x \neq 0$, then 
\[
\mc{L}_0 \simeq (\mc{L}_\mathrm{triv})_0.
\]
Hence $\ell$ is uniquely determined, and the fibre is $\PP^1$. 
\item If $x =0$, then 
\[
\mc{L}_0 \xrightarrow{\phi} (\mc{L}_\mathrm{triv})_0
\]
has rank $1$. We have two possibilities: either (1) $\im \phi = \C \oplus 0$, in which case $\ell$ is free, but $\mc{L}$ is fixed, so the fibre is $\PP^1$, or (2) $\ell = \ker \phi$, in which case $\mc{L}$ is free and $\ell$ is fixed, so the fibre is again $\PP^1$. We conclude that the fibre over $0$ is $\PP^1 \bigcup_{\{0\}} \PP^1$, as we had hoped.
\end{enumerate}

In Geordie's hand-written notes, there are charts which show that locally, this degeneration is given by the following picture:
\[
\includegraphics[scale=0.3]{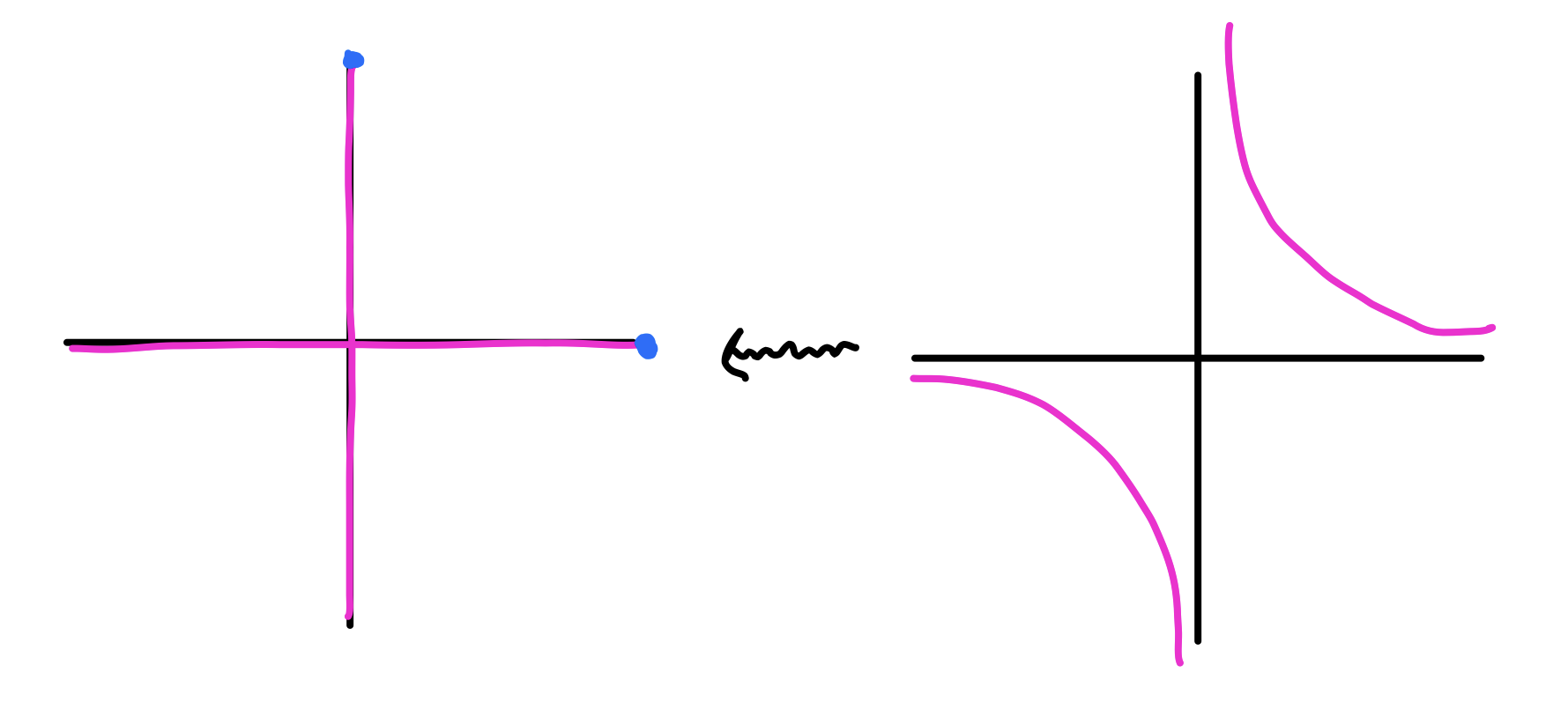}
\]
Furthermore, one can check that this degeneration is isomorphic to the degenerating quadric in $\PP^2$ discussed earlier. 

\subsection{Exercises}
These exercises will examine the structure of the affine Hecke algebra for $\GL_2$. We have
\[
X=\Z e_1 \oplus \Z e_2, \hspace{2mm} X^\vee = \Z e_1^* \oplus \Z e_2^*, \hspace{2mm} \Phi = \{\pm (e_1-e_2)\}, \hspace{2mm}  \Phi^\vee = \{ \pm(e_1^* - e_2^*)\}. \]

\begin{remark}
The group $\GL_2$ is Langlands self-dual, so we don't need to worry too much about which side of Langlands duality we are on.
\end{remark}

The affine Weyl group is 
\[
W=W_f \ltimes \left( \Z e_2 \oplus \Z e_2 \right). 
\]
A picture:
\[
\includegraphics[scale=0.3]{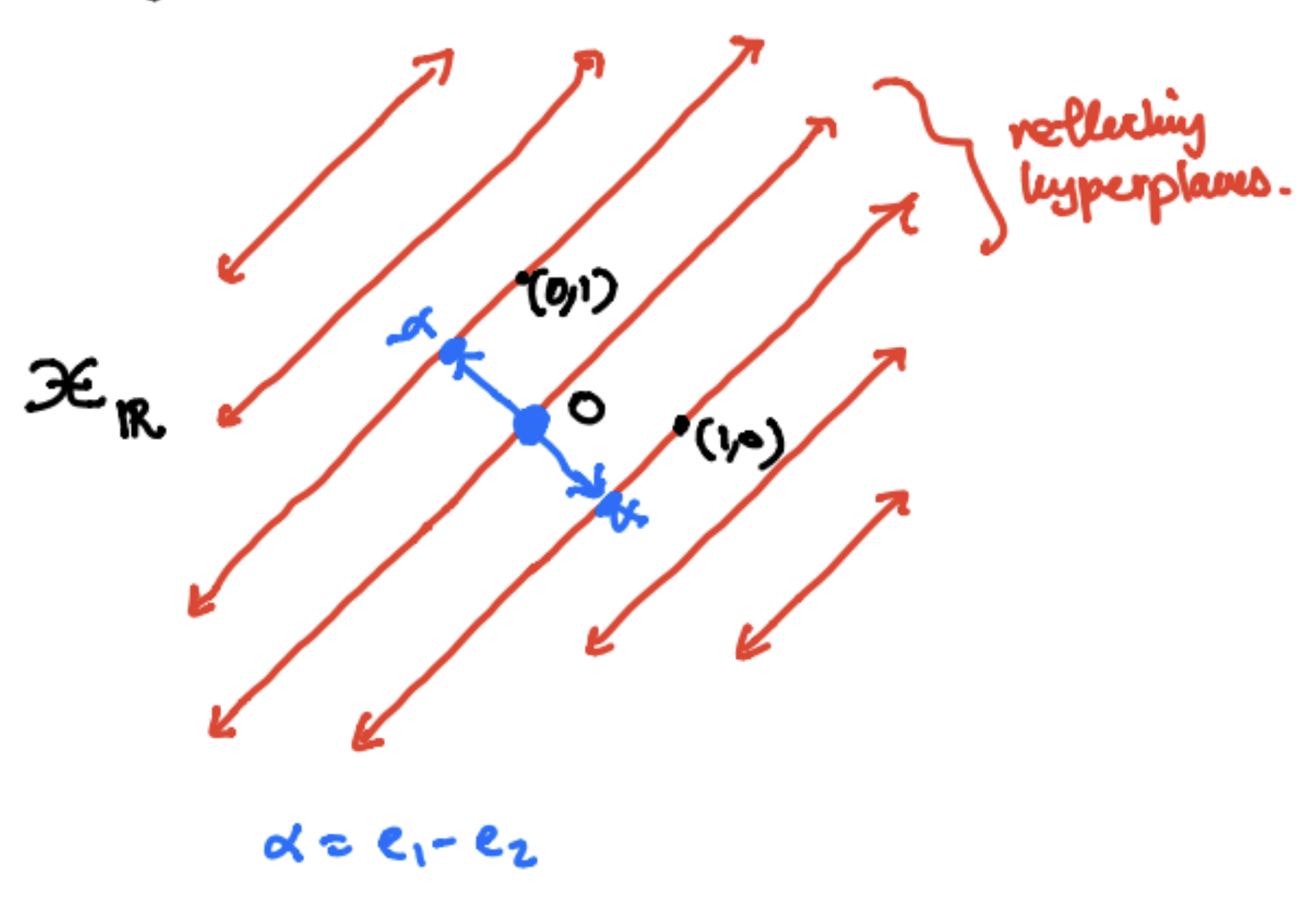}
\]
Given $\gamma \in \Z e_1 \oplus \Z e_2$, write $t_\gamma$ for the corresponding translation. 

\begin{enumerate}
    \item Show that the set of {\em length zero elements} 
    \[
    \Omega = \{ x \in W \mid x \text{ preserves strip } \vcenter{\hbox{\includegraphics[scale=0.2]{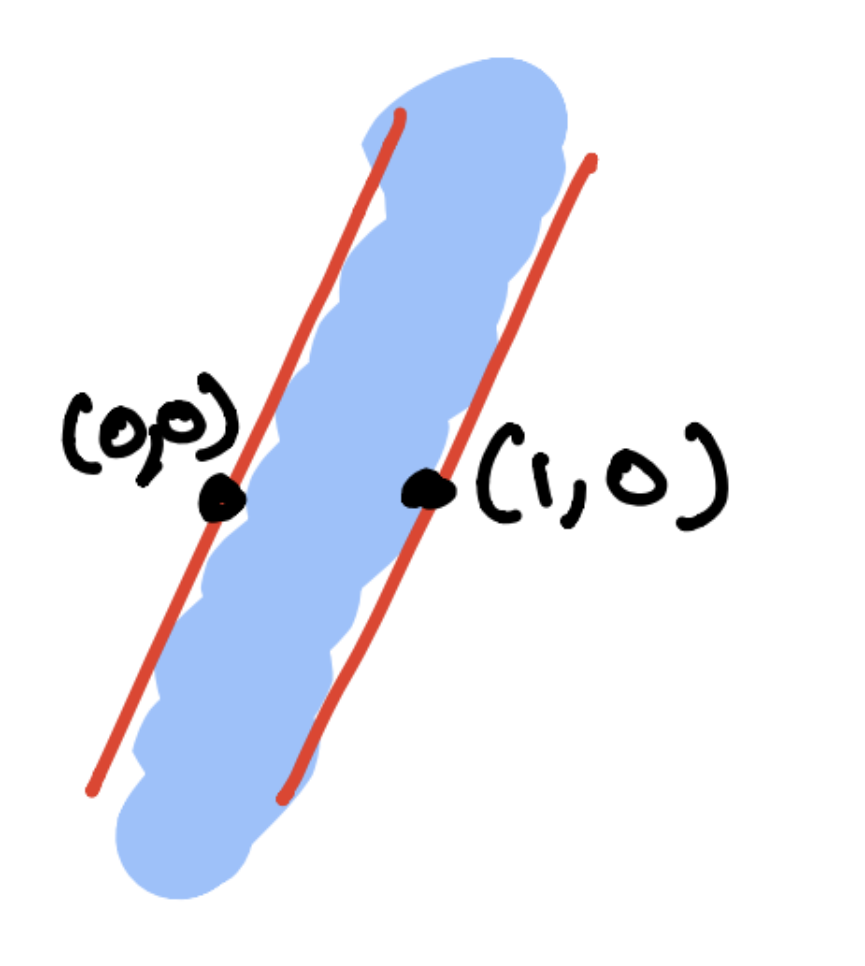}}}  \}
    \]
    is a free abelian group generated by $t_{e_1} s \in W$, where $s=s_\alpha$ is the finite simple reflection.
    \item Set $\varpi = t_{e_1}s, s_0 = t_\alpha s$. Check that $\varpi s = s_0 \varpi$, $\varpi s_0 = a \varpi$, and $\varpi^2$ is central. 
    \item Show that the Bernstein generators are given by 
    \[
    \theta_{e_1} = \delta_\varpi \delta_s, \hspace{2mm} \theta_{e_2}^{-1}=\delta_{\varpi^{-1}}\delta_s \theta_{e_2} = \delta_s^{-1} \delta_\varpi = \delta_\varpi \delta_{s_0}^{-1}. 
    \]
    {\em Hint:} 
    \[
    \includegraphics[scale=0.2]{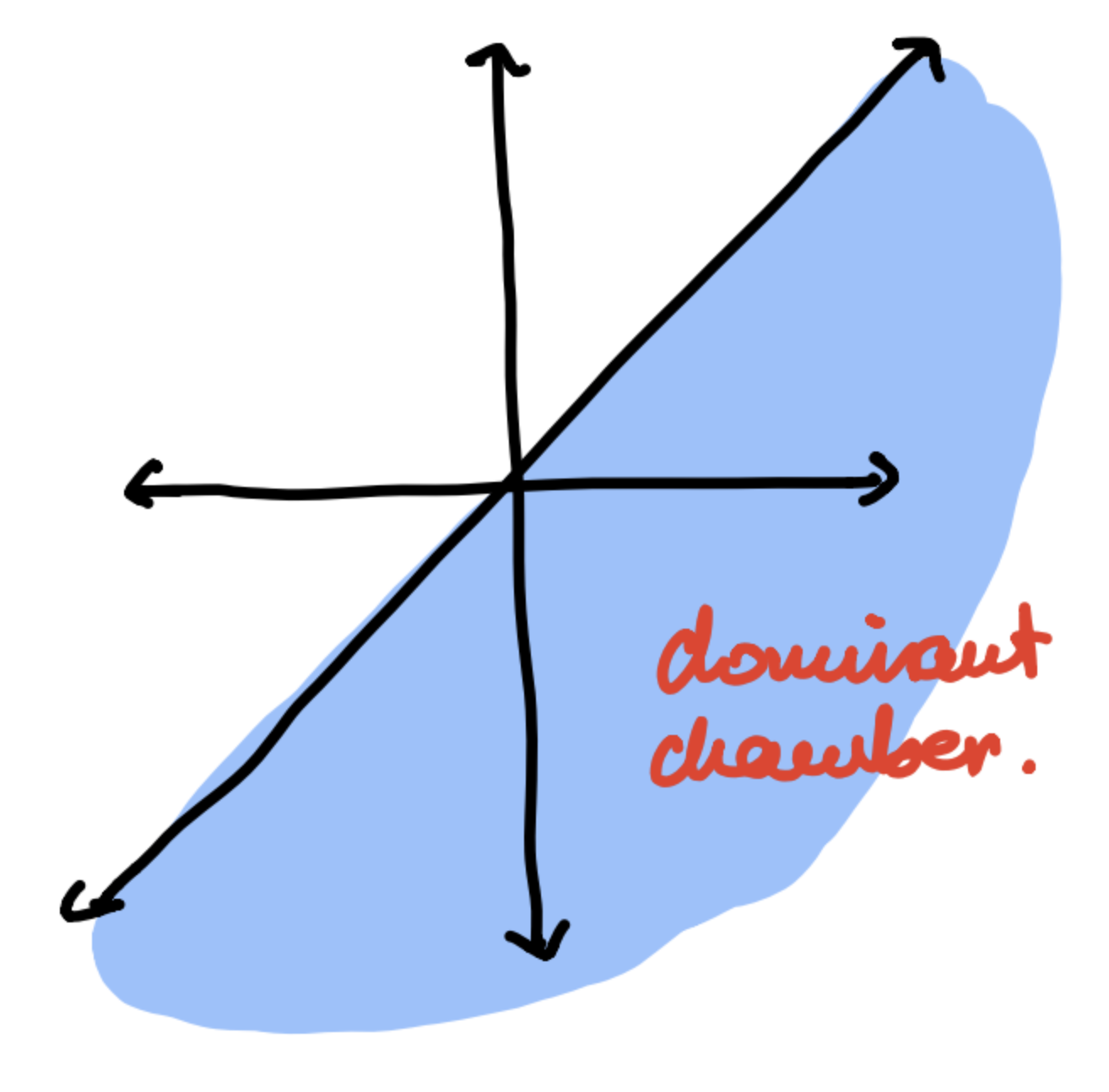}
    \]
    \item Verify that 
    \[
    z_\mathrm{nat}=\theta_{e_1} + \theta_{e_2} = \delta_\varpi (\delta_s + \delta_{s_0}^{-1}) = \delta_\varpi(\delta_s^{-1} + \delta_{s_0})
    \]
    is central. 
    \item Check that $z_\mathrm{nat} \cdot b_s = \delta_\varpi b_{s_0 s}$. 
    \item (Challenge) Prove that $Z(H)$ is generated by $\delta_\varpi^2$ and $z_\mathrm{nat}$ and show that 
    \begin{align*}
        Z(H) &\xrightarrow{\sim} [\Rep \GL_2] \\
        z_\mathrm{nat} &\mapsto \C^2 \\
        \varpi^2 &\mapsto \det
    \end{align*}
    is an isomorphism of rings. 
\end{enumerate}

%% file: lecture-24.tex
\section{Lecture 24: A hitchhiker's guide to the Hecke category}
\label{lecture 24}

In Lecture \ref{lecture 16}, we motivated the constructible side of Bezrukavnikov's equivalence via Grothendieck's function-sheaf dictionary. The material from that lecture is useful background for today's lecture, and the reader may benefit by reviewing Lecture 16 before continuing. 

Let $X$ be a quasi-projective variety over $\mathbb{F}_q$. There is a generating series built out of $\#X(\mathbb{F}_{q^m})$ for $m \geq 1$, the {\bf zeta function}:
\[
\zeta(X, s) = \exp \left( \sum_{m \geq 0} \frac{\# X(\mathbb{F}_{q^m})}{m} q^{-ms} \right). 
\]
Weil noticed that when $X$ is a smooth projective $n$-dimensional variety, its $\zeta$ function has some remarkable properties, which he packaged into the {\bf Weil conjectures}:
\begin{enumerate}
    \item Rationality: $\zeta(X,s)$ can be written as a rational function in $q^{-s}$.
    \item Functional equation: $\zeta(X,s)$ and $\zeta(X, 1-s)$ agree (up to a simple and explicit scalar). 
    \item Riemann hypothesis: roots of $\zeta$ have a specific form. 
\end{enumerate}
\begin{example}
If $X=\PP^1_{\mathbb{F}_q}$, then 
\begin{align*}
\zeta(X, s) &= \exp\left( \sum_{m \geq 0} \frac{1+q^m}{m} q^{-ms} \right) \\
&=\exp \left( \sum_{m \geq 0 } \frac{(q^{-s})^m}{m} + \sum_{m \geq 0} \frac{(q^{1-s})^m}{m} \right) \\
&= \frac{1}{1-q^{-s}} \cdot \frac{1}{1-q^{1-s}}.
\end{align*}
We see from this that: (1) $\zeta(X, s)$ is rational, (2) $\zeta(X,s)=-q^{1-s}\zeta(X, 1-s)$, and (3) we can write 
\[
\zeta(X,s)=\frac{1}{P_0(q^{-1})P_2(q^{-s})},
\]
with $P_0=(1-T)$ and $P_2=1-qT$. Here the Riemann hypothesis is the (elementary) statement that all roots of $P_0$ (resp. $P_2$) have roots of norm $1$ (resp. $q$). 
\end{example}
\begin{example}
If $X$ is a smooth elliptic curve, then 
\[
\zeta(X, s)= \frac{P_1(q^{-s})}{P_0(q^{-s})P_2(q^{-s})} = \frac{P_1(q^{-s})}{(1-q^{-s})(1-q^{1-s})}, 
\]
where $P_1$ has two roots, which are conjugate and have norm $q^{1/2}$ (``Weil numbers of weight $1$''). 
\end{example}

\begin{remark}
The case of an elliptic curve (due to Hasse) predates the Weil conjectures and was an important ingredient in their formulation. Another important ingredient was Artin's computation of the zeta function of a hyperelliptic curve, discussed in Lecture \ref{sato-tate}.
\end{remark}
\begin{remark}
We saw the significance of the roots of $P_1$ (which determine how many points $E$ has) in Lecture \ref{sato-tate} on the Sato-Tate conjecture.
\end{remark}
In the Weil conjectures, {\bf rationality} follows from the Grothendieck-Lefschetz trace formula:
\[
\#X(\mathbb{F}_q) = \mathrm{supertrace}\left(\mathrm{Frob} \circlearrowright H^*_{\text{\'{e}t}}(X, \overline{\mathbb{Q}}_\ell)\right).
\]
The {\bf functional equation} follows from Poincar\'e duality in \'etale cohomology, and the Riemann hypothesis (the most difficult part) follows from purity. This is the statement that the eigenvalues of Frobenius in degree $i$ are all ``Weil numbers of weight $i$''.
\begin{remark} Notice the strange appearance of the auxiliary prime $\ell \neq p$ in the Grothendieck-Lefschetz trace formula. Why do we need to make this choice? Such questions are usually labelled questions of ``independence of $\ell$'' in the literature. We will also see such questions arise in the Hecke category below. 
\end{remark}
The proof of the Weil conjectures needed the full artillery of $D^b_c(X, \overline{\Q}_\ell)$ and its six functor formalism. It used a relative version of the Grothendieck-Lefschetz trace formula, which associated to a sheaf $\mc{F}$ a collection of functions given by ``trace of Frobenius'':
\[
\mc{F} \in D^b_c(X, \overline{\mathbb{Q}}_\ell) \rightsquigarrow \left\{ f^m_{\mc{F}}: X(\mathbb{F}_{q^m}) \rightarrow \overline{\mathbb{Q}}_\ell \right\}. 
\]
Then, as we saw in Lecture \ref{lecture 16}, 
\[
\left[ D^b_c(X, \overline{\mathbb{Q}}_\ell)\right] \hookrightarrow \prod_{m \geq 1} \Fun(X(\mathbb{F}_{q^m}) \rightarrow \overline{\Q}_\ell),
\]
so the collection of functions associated to $\mc{F}$ completely determines its class in the Grothendieck group. This leads to Grothendieck's philosophy of ``dictionaire functions faisceaux'':
\[
\text{functions } \leftrightarrow \text{ sheaves}.
\]
This philosophy can be summed up with the slogan 

\vspace{2mm}
\begin{center}
    {\em ``interesting functions should arise from interesting sheaves''.}
\end{center}

\subsection{The Hecke category: setting the scene}

Let $G$ be split reductive over $\mathbb{F}_q$, and $B \subset G$ a Borel subgroup. Let $\mc{K} = \mathbb{F}_q((t))$, and denote by $G_\mc{K}$ the loop group and $I \subset G_\mc{K}$ the corresponding Iwahori subgroup. Recall that the Hecke algebra (either finite or affine) has its origins as a convolution algebra of bi-invariant functions on a group (see Lecture \ref{lecture 16} for more on this perspective). Using Grothendieck's function-sheaf dictionary, a natural categorification of these functions is the following:  
\[
\begin{tikzcd}
H_f \arrow[rr, bend left, "v \mapsto \frac{1}{\sqrt{q}}"]& ``=" & \left( \Fun_{B(\F_q) \times B(\F_q)} (G(\F_q), \C), * \right) \arrow[r, squiggly, "\mathrm{categorify}"] & \left(D^b_{B \times B}(G, \overline{\Q}_\ell), *\right)  
\end{tikzcd}
\]
\[
\begin{tikzcd}
H \arrow[rr, bend left, "v \mapsto \frac{1}{\sqrt{q}}"]& ``=" & \left( \Fun_{I(\F_q) \times I(\F_q)} (G_\mc{K}(\F_q), \C), * \right) \arrow[r, squiggly, "\mathrm{categorify}"] & \left(D^b_{I \times I}(G_\mc{K}, \overline{\Q}_\ell), *\right)  
\end{tikzcd}
\]
\begin{remark}
The loop group $G_\mc{K}$ is very infinite-dimensional, so we usually work with
\[
\left(D^b_I(\mc{F}\ell, \overline{\Q}_\ell), * \right),
\]
instead. (Here $\mc{F}\ell = G_k/I$ is the affine flag variety.) 
\end{remark}
\begin{remark}
One can see from the above that $H_f$ is ``independent of $q$''. We have one abstract algebra which specialises to all Hecke algebras at once. This is one desirable feature that we are hoping to categorify. 
\end{remark}
Our goal for this lecture is to find a good categorification of the Hecke algebra. It might appear that we've already accomplished this. However, the categories $\left(D^b_{B \times B}(G, \overline{\Q}_\ell), *\right)$ and $\left(D^b_{I \times I}(G_\mc{K}, \overline{\Q}_\ell), *\right)$ are not quite right for several reasons. In the remainder of the lecture, we will explain why these categories are wrong, then slowly fix them. We will concentrate on the finite case (i.e. the Hecke algebra for the finite Weyl group). The affine case is similar. 

\vspace{5mm}
\noindent
{\bf What we want:} A triangulated monoidal category $\mc{H}$ such that: 
\begin{enumerate}
    \item $([\mc{H}], *)\simeq H$ (i.e. $\mc{H}$ categorifies the Hecke algebra), 
    \item $\mc{H}$ is ``independent of $q$ and $\ell,$'' and 
    \item $\mc{H}$ admits a triangulated monoidal functor to $D^b_{B \times B}(G, \overline{\Q}_\ell)$ such that the diagram 
    \[
    \begin{tikzcd}
    \mc{H} \arrow[r] \arrow[d] & D^b_{B \times B}(G, \overline{\Q}_\ell) \arrow[d] \\
    H \arrow[r, "v \mapsto \frac{1}{\sqrt{q}}"] & \left( \Fun_{B(\F_q) \times B(\F_q)}(G(\F_q), \C), * \right) 
    \end{tikzcd}
    \]
    commutes for all $\ell$ and $q$.
\end{enumerate}
\begin{remark}
The existence of such an object is tacitly implied by Bezrukavnikov's equivalence, as $\ell$ and $q$ are nowhere to be seen on the coherent side.
\end{remark}
Now we try and fail and try and fail and try and fail, and then, finally, succeed. 

\subsection{First try}
We start with the most obvious choice:
\[
\mc{H} = D^b_{B \times B}(G, \overline{\Q}_\ell).
\]
\noindent
{\bf Objections:} 
\begin{enumerate}
    \item Depends on $q$ and $\ell$. 
    \item If $G$ is the trivial group, then 
    \[
    \mc{H} = D^b_c(\Spec\F_q, \overline{\Q}_\ell) \hookrightarrow D^b_c\left( \begin{array}{c} \text{continuous representations} \\ \text{of } \pi_1^\text{\'et}(\Spec \F_q)=\widehat{\Z} \text{ on} \\ \overline{\Q}_\ell\text{-vector spaces} \end{array} \right).
    \]
    This is almost an equivalence. In fact, one has
    \begin{equation}
        \label{too big}
    [\mc{H}] = \Z[\overline{\Z}_\ell^\times],
    \end{equation}
    which is way too big. (See \cite[Proposition 5.1.2]{BBD} and the remark following it.) 
    
    Note that the fact that this Grothendieck group is way too big persists for any group: rather than being an algebra over $\Z[v^{\pm 1}]$, our putative definition is an algebra over (\ref{too big}).
\end{enumerate}

\subsection{Second try}
In the first try, we failed to obtain requirement 1 of our desired categorification because our category had a Grothendieck group which was too big. We can attempt to fix this by passing to the algebraic closure. Try 
\begin{align*}
\mc{H} &= D^b_{B \times B}(G, \overline{\Q}_\ell) \text{ for }G/\overline{\F}_q, \text{ or} \\
\mc{H} &= D^b_{B \times B}(G, \Q) \text{ for }G/\C. 
\end{align*}

\noindent
{\bf Objections:}
\begin{enumerate}
    \item If $G$ is the trivial group, then 
    \[
    D^b_{B \times B}(G, \Q) = D^b\left( \begin{array}{c} \text{finite dimensional} \\ \text{vector spaces} \end{array} \right).
    \]
    Hence
    \[
    [\mc{H}] = \Z,
    \]
    which is too small. (We expect $\Z[v^{\pm 1}]$ for the Grothendieck group in this case.) But perhaps this is just a trivial-case phenomenon, we can test with a slightly bigger example. 
    \item Recall (c.f. Lecture \ref{lecture 16}) that in $\SL_2$, the quadratic relation in $\Fun_{B \times B}(\SL_2(\F_q),\C)$ came from the fact that
    \[
    \ind_{\SL_2(\F_q)} * \ind_{\SL_2(\F_q)} = (1+q) \ind_{\SL_2(\F_q)}.
    \]
    Why was this again? Well, the multiplication map $\mathrm{mult}:\SL_2 \times_B \SL_2 \rightarrow \SL_2$ factors through $\SL_2/B \times \SL_2 = \PP^1 \times \SL_2$:
    \[
    \begin{tikzcd}
    \SL_2 \times_B \SL_2 \arrow[rd, "\mathrm{mult}"'] \arrow[rr, "\sim"] & & \PP^1 \times \SL_2 \arrow[ld, "\mathrm{projection}"] \\
    & \SL_2 &
    \end{tikzcd}
    \]
    Here the horizontal arrow is the map $(g,h) \mapsto (gB, gh)$. So we obtain the formula above by pushing forward the constant sheaf on either side of the isomorphism, and the $(1+q)$ comes from $\#\PP^1(\F_q)=1+q$. 
    
    In $D^b_{B \times B}(\SL_2, \Q)$, the same diagram shows 
    \[
    \Q_{\SL_2} * \Q_{\SL_2} = p_*\Q_{\PP^1 \times \SL_2} = H^*(\PP^1) \otimes \Q_{\SL_2} = \Q_{\SL_2} \oplus \Q_{\SL_2}[-2]
    \]
    In the Grothendieck group, this gives 
    \begin{equation}
        \label{in the grothendieck group}
    [\Q_{\SL_2}]^2=2[\Q_{\SL_2}].
    \end{equation}
    \begin{exercise} The map 
    \begin{align*}
        \Z[W] \rightarrow \left[D^b_{B \times B}(\SL_2, \Q) \right] \\
        s \mapsto [\Q_{\SL_2}]-[\Q_{B/B}]
    \end{align*}
    is an isomorphism of algebras. 
    \end{exercise}
\end{enumerate}

We see from the exercise that  the trivial group wasn't just an anomaly - with this definition of $\mc{H}$, the Grothendieck group really is too small. Also notice that we want to replace the $2$ in (\ref{in the grothendieck group}) by $(1+q)$. With the Grothendieck-Lefschetz trace formula in mind, we wish to replace an Euler characteristic $(2)$ by the trace of Frobenius $(1+q)$. Therefore we are led to the following conclusion:  
    
    \vspace{2mm}
    \noindent
{\bf The moral:} Somehow we need to remember weights!

\subsection{Third try}
Again, let $G/\F_q$. This time, we set
\[
\mc{H} = \langle \bm{IC}_x \mid x \in W_f \rangle_{[\Z], (\Z), \Delta} \subset D^b_{B \times B}(G, \overline{\Q}_\ell). 
\]
In this incarnation of $\mc{H}$, we are keeping track of weight by introducing ``Tate twists,''  denoted above by $(\Z)$. 

\vspace{5mm}
\noindent
{\em What is a Tate twist?} Define 
\[
\overline{\Q}_\ell(-1):=H^2_c(\mathbb{A}^1)=\overline{\Q}_\ell \in \mathrm{Sh}_{\text{\'et}}(\Spec \F_q). 
\]
On $\overline{\Q}_\ell(-1)$, $\mathrm{Frob}$ acts via multiplication by $q$. Let $p:X \rightarrow \Spec{\F_q}$. For $\mc{F} \in D^b_c(X, \overline{\Q}_\ell)$, set 
\[
\mc{F}(m):=\mc{F} \otimes \left( p^*(\Q_\ell(-1)^{\otimes (-m)}\right).
\]
(Here we are using that $\Q_\ell(-1)$ is invertible, as an \'etale sheaf on a point. Thus it makes sense to take any integral tensor power.)

\vspace{3mm}
\noindent
{\bf Very nontrivial fact:} This definition of $\mc{H}$ is closed under convolution. (For the experts: This is a consequence of the Decomposition Theorem and the fact that the objects are ``already semi-simple over $\mathbb{F}_q$''.) 

\begin{remark}
({\em Technical point}) In the Hecke algebra, it is very useful to introduce a square root of $q$. To reflect this, we often choose a square root of $q$ in $\overline{\Q}_\ell$ and use it to define $\overline{\Q}_\ell(-1/2)$, the ``square root of Tate twist''. This can also be done purely formally if one prefers. 
\end{remark}

Now we are in good shape: If $G$ is the trivial group, 
    \[
    \mc{H}=\left\langle \left. \overline{\Q}_\ell \left(\frac{m}{2}\right)\right| m \in \Z \right\rangle_{\Delta} \subset D^b(\Spec \F_q).
    \]
    Hence the Grothendieck group is 
    \[
    [\mc{H}] \xrightarrow{\sim} \Z[v^{\pm 1}], v \mapsto \overline{\Q}_\ell(-1/2). 
    \]
    So this $\mc{H}$ passes our first test. Great! 

Using Grothendieck's theory of weights, there is an alternative version of this third try. Recall (see ``Scholze's motivic plane'' in Lecture \ref{lecture 10}): 
\[
\begin{tikzcd}
& & \text{motives} \arrow[dl, squiggly] \arrow[dr, squiggly] & & \\
& \begin{array}{c} \text{\'etale} \\
\overline{\Q}_\ell\text{-sheaves}\end{array} \arrow[d, squiggly] & & \begin{array}{c} \text{Saito's mixed} \\ \text{Hodge modules} \end{array} \arrow[d, squiggly] & \\
& \begin{array}{c} \text{Frobenius}\\ \text{actions on \'etale}\\ \text{cohomology} \arrow[rr, leftrightarrow, red] \arrow[dl, squiggly] \end{array} & & \begin{array}{c} \text{mixed Hodge} \\ \text{structures} \end{array} \arrow[dr, squiggly] \\
\begin{array}{c}\text{Weil} \\ \text{conjectures} \end{array}& X/\F_q& & X/\C & \begin{array}{c}\text{Hodge} \\
\text{theory} \end{array}
\end{tikzcd}
\]
Under the red arrow above, we have a second incarnation of our third try:
\[
\mc{H} = \langle \bm{IC}_x \mid x \in W_f \rangle_{[\Z], (\frac{1}{2} \Z), \Delta} \subset D^b_{B \times B} \left( \begin{array}{c} \text{mixed Hodge} \\ \text{modules} \end{array} \right).
\]
This version of $\mc{H}$ also has the correct Grothendieck group. It is closed under convolution by the Decomposition Theorem. 

\vspace{3mm}
\noindent
{\bf A subtlety:} These two incarnations $\mc{H}_\text{\'etale}$ and $\mc{H}_{mHm}$ of our third attempt are really different. If $G$ is the trivial group, then 
\begin{itemize}
    \item in $\mc{H}_\text{\'etale}$, 
    \[
    \Ext^1\left(\overline{\Q}_\ell, \overline{\Q}_\ell\left(\frac{m}{2}\right)\right)=\begin{cases} \overline{\Q}_\ell & \text{ if } m=0; \\ 0 & \text{otherwise}. \end{cases} 
    \]
    A non-trivial self-extension of $\overline{\Q}_\ell$ is given by $\overline{\Q}_\ell^2$, with Frobenius acting by Jordan block $\bp 1 & 1 \\ 0 & 1 \ep$. Hence, there are extensions between objects of the same weights, but no extensions between objects of different weights. 
    \item However, in $\mc{H}_{mHm}$, there are no extensions between objects of the same weight, but there are extensions between objects of different weights\footnote{Stating exactly what these extension groups are would require us to be more precise about what version of mixed Hodge structures we are using (integral, real, complex, etc.). However, the basic phenomenon that different weights extend persists in all of them.}. 
\end{itemize}

\begin{remark}
Under Bezrukavnikov's equivalence, we expect  
\[
\text{weight} \leftrightarrow \text{weight of $\C^\times$-action}.
\]
So there should be no Exts at all for the trivial group! 
\end{remark}

\noindent
{\bf A historical incarnation of this issue:} Let $\mc{O}_0$ be the principal block of category $\mc{O}$. There is a finite-dimensional $\C$-algebra $A$ such that $\mc{O}_0 \simeq A\text{-mod}_{\mathrm{f.g.}}$. In 1990, Soergel showed that $A$ admits a $\Z$-grading $\widetilde{A}$ defined over $\Q$, and $\widetilde{A} \text{-gmod}_{\mathrm{f.g.}}$ provides a grading on category $\mc{O}$. This grading explains the $q$ in the Kazhdan--Lusztig polynomials. Hence we have the following diagram: 
\[
\begin{tikzcd}
A\text{-mod}_\mathrm{f.g.} \arrow[d, dash, "\sim"] & \widetilde{A}\text{-gmod}_\mathrm{f.g.} \arrow[l] & \\
\mc{O}_0 \arrow[r, dash, "\sim"] & \begin{array}{c}\text{certain} \\ \mc{D}\text{-modules} \end{array} \arrow[r, dash, "\sim"] & \Perv_{(B)}(G/B) 
\end{tikzcd}
\]
What is the geometric meaning of the grading on $\mc{O}_0$ coming from $\widetilde{A}$? This was explained in \cite{BGS}:
\begin{align*}
    \widetilde{A}_{\overline{\Q}_\ell}\text{-gmod} &\simeq \begin{array}{c} \text{some geometric} \\ \text{category} \end{array} \leftrightarrow D^b_B(G/B, \overline{\Q}_\ell), \\ 
    \widetilde{A}_\C \text{-gmod} &\simeq \begin{array}{c} \text{some geometric}\\ \text{category} \end{array} \leftrightarrow D^b_B(MHM_{G/B}).
\end{align*}
In each case, the graded version is explained by some category related, but not equal, to mixed sheaves. In both cases, some cooking is involved to remove the problematic Exts from earlier. The ``cooking'' is different in each case. 
\begin{remark}
Recent motivic versions (Soergel-Wendt \cite{SW} and Soergel-Wendt-Virk, \cite{SVW}) explain how to remove the cooking.
\end{remark}

\subsection{Fourth (and final) try}
We have finally converged on the correct formulation. Let
\[
\mc{H}_\mathrm{s.s}=\langle \bm{IC}_x \mid x \in W_f \rangle_{\oplus, [\Z]} \subset D^b_{B \times B}(G, \Q).
\]
This is an additive category, but it is {\em not} triangulated. Thus it consists of all semi-simple complexes in $D^b_{B \times B}(G, \Q)$. It is closed under convolution by the Decomposition Theorem. We need to make make one necessary change of notation.

\vspace{2mm}
\noindent
{\bf Redefine:} $(1):=[1]$ on $\mc{H}_\mathrm{s.s.}$. 

\vspace{2mm}
\begin{definition}
The {\em geometric Hecke category} is 
\[
\mc{H}:=K^b(\mc{H}_\mathrm{s.s.}).
\]
This category has two shift functors: $(n)$ denotes the shift in $\mc{H}_\mathrm{s.s.}$ and $[n]$ denotes the shift in $K^b$. 
\end{definition}

Before declaring victory, we should verify that this satisfies all of our desired properties for the trivial group. If $G$ is the trivial group, then 
\[
\mc{H}_\mathrm{s.s.} =D^b_c(\mathrm{pt}) = D^b\left( \begin{array}{c} \text{finite-dimensional} \\ \text{vector spaces} \end{array} \right) = \begin{array}{c} \text{finite-dimensional}\\ \text{graded vector spaces.} \end{array}
\]
(Here the final equality is as {\em additive} categories.) Hence
\[
\mc{H} = K^b\left( \begin{array}{c} \text{finite-dimensional}\\ \text{graded vector spaces} \end{array} \right). 
\]
Note that this category has no extensions, so our woes of our third attempt have disappeared. Moreover, we have 
\[
[\mc{H}] = \Z[v^{\pm 1}],
\]
with $v$ corresponding to the shift of grading. Moreover, there are no $q$'s or $\ell$'s in sight. Hurrah! 

\begin{remark}
\begin{enumerate}
    \item We have $[\mc{H}]=[\mc{H}_{s.s.}]=H_f$.
    \item A similar definition over $\overline{\F}_q$ with $\overline{\Q}_\ell$-sheaves leads to an equivalent category (after extensions of scalars to $\overline{\Q}_\ell$). This proves independence of $q$ and $\ell$. The proof is by showing that $\mc{H}_\mathrm{s.s.} \simeq \mathbb{S}\mathrm{Bim}$, the category of Soergel bimodules\footnote{We hope to explain this sometime in the future!}. 
    \item With difficulty, one can construct monoidal realisation functors:
    \[
    \begin{tikzcd}
    & (\mc{H}_\text{\'etale},*) \\
    (\mc{H}, *)\arrow[ur] \arrow[dr] & \\
    & (\mc{H}_\mathrm{mHm}, *) 
    \end{tikzcd}
    \]
\end{enumerate}
\end{remark}
This provides considerable evidence that $\mc{H}$ is the right object. 
\begin{remark}
It is possible that although $\mc{H}$ is the right object, it may not yet have the right definition. Recent motivic versions \cite{SVW} probably provide the ``correct'' definition.
\end{remark}

%% file: lecture-25.tex
\section{Lecture 25: The categorical anti-spherical module and its symmetries}
\label{lecture 25}

Recall that we are working towards proving the following theorem of Arkhipov--Bezrukavnikov:
\begin{theorem} \cite{AB}
There exists an equivalence of triangulated categories
\[
\begin{tikzcd}
D^{G^\vee \times \G_m}(\widetilde{\mc{N}}) \arrow[r, dash, "\sim"] & \mc{M}^\mathrm{asph}.
\end{tikzcd}
\]
This equivalence categorifies the isomorphism
\[
\begin{tikzcd}
K^{G^\vee \times \G_m}( \widetilde{\mc{N}} ) \arrow[r, dash, "\sim"] & H \otimes_{H_f} \mathrm{sgn}.
\end{tikzcd}
\]
\end{theorem}
Today we will define the category $\mc{M}^\mathrm{asph}$ and discuss the philosophy of higher representation theory. 

\subsection{The affine Hecke category}
For the remainder of this course, $k \in \{\Q, \R, \C\}$. Last lecture in our quest to find the correct definition of the geometric Hecke category, we defined
\[
\mc{H}^\mathrm{ss}_f:=\langle \bm{IC}_x \mid x \in W_f \rangle_{\oplus, [\mathbb{Z}]} \subset D_{B \times B}(G, k). 
\]
Here ``ss'' stands for semi-simple. This category is additive, but not triangulated. 
\begin{remark}
(Technical point) It is convenient to normalize $\bm{IC}_x$ so that it corresponds to the usual IC sheaf on $G/B$ under the equivalence
\[
D_{B \times B}(G, k) \simeq D_B(G/B, k);
\]
e.g.
\[
\bm{IC}_{id}=k_B, \bm{IC}_s=k_{P_s}[1], \ldots, \bm{IC}_x|_{BxB}=k_{B x B}[\ell(x)].
\]
\end{remark}
We used $\mc{H}_f^\mathrm{ss}$ to define the geometric Hecke category:
\[
\mc{H}_f:=K^b(\mc{H}^\mathrm{ss}_f).
\]
This is a triangulated category (in contrast to $\mc{H}_f^\mathrm{ss}$). It has two natural shift functors: 
\begin{align*}
    [1] &=\text{ triangulated shift functor, and } \\
    (1) &=\text{ functor induced by the shift functor on } \mc{H}^\mathrm{ss}_f.
\end{align*}
A picture of these shifts on an object $\mc{F} \in \mc{H}_f$ is
\[
\includegraphics[scale=0.2]{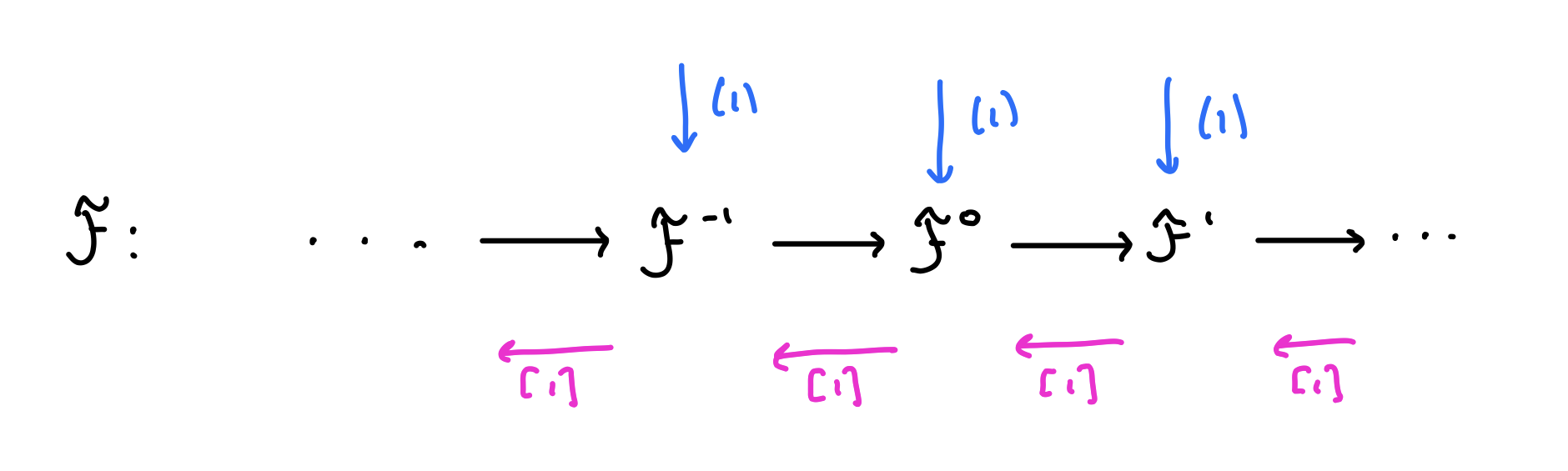}
\]
Here each $\mc{F}^i \in \mc{H}_f^\mathrm{ss}$. 
\begin{theorem}
\label{categorifies}
The map $b_s \mapsto [\bm{IC}_s]$ induces an isomorphism of $\Z[v^{\pm 1}]$-algebras 
\[
\begin{tikzcd}
H_f \arrow[r, dash, "\sim"] & \left[\mc{H}_f^\mathrm{ss}\right]_{\oplus} \simeq [\mc{H}_f]_\Delta.
\end{tikzcd}
\]
Here the subscript $\oplus$ denotes the split Grothendieck group of an additive category, and the subscript $\Delta$ denotes the triangulated Grothendieck group of a triangulated category. Under this isomorphism, the Kazhdan--Lusztig basis element $b_x$ corresponds to the class of an intersection cohomology complex, in formulas
\[
b_x \mapsto [\bm{IC}_x].
\]
\end{theorem}
The same construction works in the affine case. Define an additive category
\[
\mc{H}^\mathrm{ss}:=\langle \bm{IC}_x \mid x \in W \rangle_{\oplus, [\Z]} \subset D^b_I(\mc{F}\ell, k)
\]
for $\mc{F}\ell:=G((t))/I$, and a triangulated category
\[
\mc{H} := K^b(\mc{H}^\mathrm{ss}). 
\]
The analogue of Theorem \ref{categorifies} holds:
\begin{theorem}
\[
H \simeq [\mc{H}^\mathrm{ss}]_\oplus \simeq [\mc{H}]_\Delta, b_x \mapsto [\bm{IC}_x]. 
\]
\end{theorem}

\subsection{The categorical (anti-)spherical module}
In the Hecke algebra, we have the quadratic relation for each $s \in S$:
\[
\delta_s^2=(v^{-1}-v)\delta_s+1, \text{ or reformulated, } (\delta_s + v)(\delta_x - v^{-1})=0. 
\]
This leads to the existence of two $H_f$-modules, $\mathrm{sgn}$ and $\mathrm{triv}$, where $\delta_s$ acts by $-v$ and $v^{-1}$, respectively, and two corresponding induced modules for $H$:
\begin{align*}
M^\mathrm{sph}&:=\mathrm{triv} \otimes_{H_f} H \text{ the spherical module, and }\\
M^{\mathrm{asph}} &:= \mathrm{sgn} \otimes_{H_f} H \text{ the antispherical module}. 
\end{align*}
\begin{remark}
We have made a notational switch from denoting the anti-spherical module by $N$ (c.f. Lecture \ref{lecture 21}) to $M^\mathrm{asph}$. This is because we wish to denote categorifications by the corresponding script letters and the symbol $\mc{N}$ has already been assigned to the nilpotent cone. 
\end{remark}
\begin{remark}
We have also made a switch from left to right modules. From now on, the spherical and anti-spherical module will be considered as right $H$-modules to align with \cite{AB}. 
\end{remark}

In Lecture \ref{lecture 20}, we discussed how as right modules over the lattice part $\mc{L}=\bigoplus_{\lambda \in X^\vee} \Z[v^{\pm 1}]\theta_\lambda$ of $H$, $M^\mathrm{sph}$ is a free module of rank $1$:
\[
M^\mathrm{sph} = \mathrm{triv} \otimes_{H_f}H = \mathrm{triv} \otimes_{H_f} H_f \otimes \mc{L} \simeq \mc{L}.
\]
A similar statement holds for $M^\mathrm{asph}$. This description of the (anti-)spherical module, which follows from the Bernstein presentation of the affine Hecke algebra, is the ``coherent perspective'' of $M^\mathrm{(a)sph}$. On the other hand, using the Coxeter structure of $W$, we can view $M^\mathrm{(a)sph}$ from a ``constructible perspective'' as follows. Any $w \in W$ can be written as $w_\mathrm{fin}{^f w}$ for $w_\mathrm{fin} \in W_f$ and $^f w \in ^fW$, where $^fW$ is the set of minimal coset representatives in $W_f \backslash W$. Hence we have an isomorphism 
\[
H_f \otimes \left( \bigoplus_{x \in {^fW}} \Z[v^{\pm 1}] \delta_x\right)  \xrightarrow{\sim} H,
\]
and $M^\mathrm{(a)sph}$ has a ``standard basis''
\[
\{\delta_x^\mathrm{(a)sph}:= 1 \otimes \delta_x \mid x \in {^fW}\}. 
\]
\begin{exercise} (Fun!) Compute the bijection 
\[
\begin{array}{c} \text{dominant} \\ \text{alcoves} \end{array} \xleftrightarrow{\sim} {^fW} \xleftrightarrow{\sim} W_f \backslash W=W_f \backslash (W_f \ltimes \Z \Phi^\vee) = \Z \Phi^\vee
\]
in a few examples (e.g. for $\SL_2, \SL_3, \ldots$). 
\end{exercise}

Using Kazhdan--Lusztig combinatorics, we can give descriptions of the spherical and anti-spherical module in terms of the Kazhdan--Lusztig basis of $H$. We dedicate the next part of the lecture to doing this. 
\begin{exercise}
\label{KL exercise}
Prove the following:
\begin{enumerate}
    \item If $x \in W$ and $s \in S$ are such that $xs<x$, then $b_x b_s = (v+v^{-1}) b_x$.
    \item  If $t \in S$ and $x \in W$ are such that $tx<x$ and $s \in S$ is arbitrary, then in the decomposition
    \[
    b_x b_s = \sum n_y b_y,
    \]
    $n_y \neq 0$ implies $ty<y$. Conclude from this that $\{b_x \mid tx<x\}$ span a right ideal in $H$.
\end{enumerate}
\end{exercise}

\begin{lemma}
(realisations of $M^\mathrm{sph}$ and $M^\mathrm{asph}$ via Kazhdan--Lusztig theory) Let $w_f$ be the longest element in $W_f$. 
\begin{enumerate}
    \item The map $1 \otimes 1 \mapsto b_{w_f}$ induces an isomorphism of right $H$-modules 
    \[
    M^\mathrm{sph} \xrightarrow{\sim} b_{w_f} H. 
    \]
    \item The $\Z[v^{\pm 1}]$-span of $\{b_s \mid s \not \in {^f W} \}$ is a right ideal. Moreover, 
    \[
    M^\mathrm{asph} \simeq H/\langle b_x \mid x \not \in {^fW} \rangle .
    \]
\end{enumerate}
\end{lemma}
\begin{proof}
By Exercise \ref{KL exercise}.1, $b_{w_f}b_s = (v+ v^{-1}) b_{w_f}$ for all $s \in S_f$, so 
\[
b_{w_f} \delta_s = b_{w_f} (b_s - v \delta_{id}) = v^{-1} b_{w_f}
\]
for all $s \in S_f$. By Frobenius reciprocity, we have a map 
\[
M^\mathrm{sph} \xrightarrow{\phi} b_{w_f}H, \hspace{2mm} \delta_{id}^\mathrm{sph} \mapsto b_{w_f}.
\]
For $x \in {^f W}$
\[
\delta_x^\mathrm{sph} \xmapsto{\phi} b_{w_f} \delta_x = \delta_{w_f x} + \text{ lower terms}. 
\] 
By upper triangularity, we conclude that $\phi$ is an isomorphism. This proves 1. 

By Exercise \ref{KL exercise}.2, $\langle b_x \mid x \not \in {^f W}\rangle$ is a right ideal in $H$. Hence $H/\langle b_x \mid x \not \in {^f W} \rangle$ is a right $H$-module. In $H/\langle b_x \mid x \not \in {^f W} \rangle$, 
\[
0=1 \cdot b_s = \delta_s + v
\]
for $s \in S_f$, so 
\[
1 \cdot \delta_s = -v \cdot 1.
\]
By Frobenius reciprocity, we have a map 
\[
M^\mathrm{asph} \rightarrow H/ \langle B_x \mid x \not \in {^f W} \rangle,
\]
and an analogous argument to the one above shows that this map is an isomorphism. This proves 2. 
\end{proof}

\begin{remark}
One can use the other Kazhdan--Lusztig basis $\{b_x'\}$ to realise $M^\mathrm{sph}$ as a quotient and $M^\mathrm{asph}$ as a submodule. This is categorified by Koszul duality. 
\end{remark}

Now it makes sense to define the categorical versions of $M^\mathrm{sph}$ and $M^\mathrm{asph}$:
\begin{align*}
    \mc{M}^{\mathrm{sph}, \mathrm{ss}} &:= \langle \bm{IC}_{w_f} * \mc{H}^\mathrm{ss} \rangle_{\ominus} \subset \mc{H}^\mathrm{ss}, \\
    \mc{M}^\mathrm{sph} &:= \langle \bm{IC}_{w_f} * \mc{H} \rangle_\Delta \subset \mc{H}, \\
    \mc{M}^\mathrm{asph, ss} &:= \mc{H}^\mathrm{ss} / \langle \bm{IC}_x \mid x \not \in {^fW} \rangle_{\oplus, (\Z)}, \\
    \mc{M}^\mathrm{asph} &:= \mc{H} / \langle \bm{IC}_x \mid x \not \in {^fW} \rangle_{\Delta, (\Z)}.
\end{align*}
Here the quotients are quotients of additive categories, and $\ominus$ denotes the closure under direct summands. 
\begin{lemma}
\[
\mc{M}^\mathrm{asph} =K^b(\mc{M}^\mathrm{asph, ss}). 
\]
\end{lemma}
\begin{theorem}
The category $\mc{M}^\mathrm{(a)sph}$ is a right $\mc{H}$-module. The map $\delta_{id}^\mathrm{asph} \mapsto \bm{IC}_{id}$ induces an isomorphism of right $H$-modules 
\[
M^\mathrm{(a)sph} \simeq [\mc{M}^{(a)sph, ss}]_{\oplus} = [\mc{M}^\mathrm{(a)sph}]_\Delta. 
\]
\end{theorem}
\begin{remark}
In future lectures, we will discuss in more detail what it means for a category to be a module over a monoidal category like $\mathcal{H}$. 
\end{remark}

\subsection{Symmetries of categories}
The remainder of this lecture will be a discussion of higher representation theory. We start at the beginning: {\bf Representation theory} is the study of linear symmetry (of groups, algebras, etc.). This field stems from the underlying observation:

\vspace{3mm}
\begin{center}
{\em Group actions are difficult, but they become easier once we linearize.}
\end{center}
\vspace{2mm}

\begin{example}
\begin{enumerate}
    \item Linearizing the action of $S^1 \circlearrowright S^1$ (resp. $SO(3) \circlearrowright S^2$) leads to the theory of Fourier series (resp. spherical harmonics). 
\item If $X$ is a variety over $\Q$, studying $X(\Q)$ is hard. This can be expressed as a question about the action of the absolute Galois group of $\Q$ on all points of $X(\overline{\Q})$. A central technique in modern number theory is instead to study the linear problem $\Gal(\overline{\Q}/\Q) \circlearrowright H^i(X, \overline{\Q}_\ell).$ 
\end{enumerate}
\end{example}
{\bf 2-representation theory} is the study of symmetries of linear categories (additive, abelian, triangulated, etc.). These symmetries take the form of monoidal categories. 

\vspace{2mm}
\noindent
{\bf Philosophy:} (learned from Manin) In differential geometry, if we have a group acting on a manifold $G \circlearrowright M$, we get a lot of leverage by linearizing and studying $G \circlearrowright L^2(M, \C)$. In algebraic geometry, if we want to do the same thing for a group $G$ acting on an algebraic variety $X$, a useful technique is to study is $G \circlearrowright \Coh(X), \mathrm{QCoh}(X), D^b(\Coh(X))$. The categories $\mathrm{QCoh}(X), \Coh(X), D^b(\Coh(X))$ can be thought of as {\em 2-linearizations} of $X$.
\vspace{2mm}

\begin{remark}
(Side remark, c.f. \cite{manin}) Let $\mc{F}$ be a sheaf on $X$. Early approaches to studying such objects ($\sim$ 1950s) emphasized \v{C}ech coverings. Later, Grothendieck shifted the emphasis to injective resolutions. In sheaf cohomology 
\[
H^i(X, \mc{F})
\]
there is a non-linear variable ($X$), and a linear variable ($\mc{F}$). We can interpret Grothendieck's shift as a movement from the non-linear variable to the linear variable. 
\end{remark}

\noindent
{\bf Long range hope:} From the perspective of $2$-representation theory, our current goal is to show that the natural symmetries of 
\[
D^{G^\vee \times \mathbb{G}_m}(
\widetilde{\mc{N}}) \text{ and } \mc{M}^\mathrm{asph}
\]
agree, and yield the Hecke category. 

\vspace{2mm}

A priori the symmetries on each side look rather different:
\begin{align*}
    \mathrm{RHS} &\circlearrowleft \mc{H} \\
    \mathrm{LHS} &\circlearrowleft \Rep G^\vee, \Rep \G_m, \mathrm{Pic}^{G^\vee \times \G_m}(\widetilde{\mc{N}}), \ldots 
\end{align*}

\vspace{2mm}

\subsection{A notational interlude:} We will use the language of stacks. 
\begin{itemize}
    \item All stacks we study will be of the form $Y/H$, where $Y$ is a quasi-projective variety and $H$ is an affine algebraic group. 
    \item See Emily's IFS talk ``Sheaves on Stacks'' for an excellent introduction to coherent sheaves on stacks. 
    \item The most important fact (which will be used repeatedly) is \begin{align*}
        \mathrm{QCoh}(Y/H)&\simeq \mathrm{QCoh}^H(Y), \\
        \Coh(Y/H) &\simeq \Coh^H(Y).
    \end{align*}
    \item {\bf Basic observation:} If $Y$ is a stack as above, then 
    \begin{enumerate}
        \item $\mathrm{QCoh}(Y)$ and the category of vector bundles on $Y$ are symmetric monoidal categories. 
        \item If $X \xrightarrow{f} Y$ is a map of stacks, then $\mathrm{QCoh}$ is a module over $\mathrm{QCoh}(Y)$ via 
        \[
        \mc{F} * \mc{G} := f^*(\mc{F}) \otimes_{\mc{O}_X} \mc{G}
        \]
        for $\mc{F} \in \mathrm{QCoh}(Y)$, $\mc{G} \in \mathrm{QCoh}(X)$. (Note that here $*$ denotes the action map of $\mathrm{QCoh}(X)$ and not convolution.)
    \end{enumerate}
\end{itemize}

\subsection{The basic approach of Arkhipov-Bezrukavnikov}
The basic approach of \cite{AB} is the following. We wish to show that 
\[
D^{G^\vee \times \G_m} (\widetilde{\mc{N}}) \simeq \mc{M}^\mathrm{asph}.
\]
(For the moment, we ignore $\G_m$.) We have the following maps of stacks:
\[
\begin{tikzcd}
\widetilde{\mc{N}}/G^\vee \arrow[d, "q"] & & \mathrm{QCoh}(\widetilde{\mc{N}}/G^\vee) \\
\mf{g}^\vee/G^\vee \arrow[d, "p"]& \rightsquigarrow & \mathrm{QCoh}(\mf{g}^\vee / G^\vee) \arrow[u, hookrightarrow, "q^*"] \\
\mathrm{pt}/G^\vee & & \mathrm{QCoh}(\mathrm{pt}/G^\vee) \arrow[u, hookrightarrow, "p^*"] 
\end{tikzcd}
\]
The maps of stacks in the column on the left induce the increasing chain of symmetric monoidal categories in the column on the right. Step by step, we will construct an action of the symmetric monoidal categories in this chain on $\mc{M}^\mathrm{asph}$. The action of $\mathrm{QCoh}(\mathrm{pt}/G^\vee)$ will come from Gaitsgory's central functors, the action of $\mathrm{QCoh}(\mf{g}^\vee / G^\vee)$ from the monodromy endomorphism of nearby cycles, and the action of $\mathrm{QCoh}(\widetilde{\mc{N}}/G^\vee)$ from ``Wakimoto arrows'' (which are yet to be defined). This procedure yields a functor 
\[
\mathrm{QCoh}(\widetilde{\mc{N}}/G^\vee) \rightarrow \mc{M}^\mathrm{asph},
\]
which we will argue is an equivalence. 

%% file: lecture-26.tex
\section{Lecture 26: Monoidal categories and their actions}
\label{lecture 26}

In representation theory, we study 
\[
{A \atop \text{algebra}} {\circlearrowright \atop \hspace{2mm}} { M \atop \text{vector space}}. 
\]
This process lets us draw conclusions about $M$ (composition series, semi-simplicity, etc.)

In $2$-representation theory, we study 
\[
{ \mc{A} \atop \text{monoidal category}} {\circlearrowright \atop \hspace{2mm} } {\mc{M} \atop \text{category}}.
\]
What does this process let us conclude about $\mc{M}$? Today we'll examine this question. 

\vspace{3mm}
\noindent
{\bf A theme of this lecture:} There are two approaches to higher algebra:
\begin{enumerate}[label=(\alph*)]
    \item {\em Generators and relations}: carry ``just enough'' coherence data, or
    \item {\em Holistic}: carry ``all'' coherence data. 
\end{enumerate}
Approach (a) is rigid, but (sometimes) computable; whereas approach (b) is flexible but (often) incomputable. Historically, (a) is usually developed first, but (b) proves itself to be more powerful in the long run. The following familiar example from algebraic topology illustrates this. 
\begin{example}
\label{algebraic topology} Given a reasonable topological space $X$, there are two approaches to computing homology $H_*(X)$:
\begin{enumerate}[label=(\alph*)]
\item triangulate your space as efficiently as possible and compute homology using a ``small'' complex (simplicial homology), or 
\item consider all possible singular $n$-simplices and build an enormous chain complex from the spaces $\Hom(\Delta_n, X)$ (singular homology).  
\end{enumerate}
The first approach is by ``generators and relations,'' whereas the second is ``holistic''. 
\end{example}
\begin{remark}
(For those who know about simplicial sets) The two approaches in Example \ref{algebraic topology} underly two approaches to studying the homotopy theory of $X$. 
\end{remark}

\subsection{What is a monoidal category?}
There are two approaches. 

\vspace{3mm}
\noindent
{\bf Generators and relations:}
A monoidal category is a category $\mc{A}$ equipped with a bifunctor 
\[
\otimes: \mc{A} \times \mc{A} \rightarrow \mc{A},
\]
a unit 
\[
\ind \in \mc{A},
\]
and natural transformations
\[
\alpha_{X,Y,Z}:(X \otimes Y) \otimes Z \xrightarrow{\sim} X \otimes (Y \otimes Z),
\]
\[
\lambda_X: \ind \otimes X \xrightarrow{\sim} X, \text{ and }
\]
\[
\rho_X: X \otimes \ind \xrightarrow{\sim} X
\]
such that the diagrams 
\[
\begin{tikzcd}
& \left( (X \otimes Y) \otimes Z\right) \otimes W \arrow[dl, "\alpha_{X, Y, Z} \otimes W"'] \arrow[drr, "\alpha_{X \otimes Y, Z, W}"] & & \\
\left(X \otimes (Y \otimes Z) \right) \otimes W \arrow[ddr, "\alpha_{X, Y \otimes Z, W}"'] & & & (X \otimes Y) \otimes (Z \otimes W) \arrow[d, "\alpha_{X, Y, Z \otimes W}"] \\
& & & X \otimes \left( Y \otimes (Z \otimes W) \right) \\
& X \otimes \left( (Y \otimes Z) \otimes W \right) \arrow[urr, "X \otimes \alpha_{Y, Z, W}"']
\end{tikzcd}
\]
and 
\[
\begin{tikzcd}
(X \otimes \ind) \otimes Y \arrow[dr, "\rho_X \otimes id_Y"']  \arrow[rr, "\rho_X \otimes \lambda_X"]  & &X \otimes (\ind \otimes Y) \arrow[dl, "id_X \otimes \lambda_Y"] \\
& X \otimes Y &
\end{tikzcd}
\]
commute for all objects $X, Y, Z, W$. We refer to the first diagram as the {\em pentagon} and the second as the {\em unit}. With this set-up, we can also formulate the notion of a {\em monoidal functor} between monoidal categories, though we won't state this definition precisely today. 

\vspace{3mm}
\noindent
{\bf Basic claim:} (``MacLane's coherence theorem'') Any two maps consisting of associators and units agree. 

\vspace{3mm}
\noindent
{\bf Holistic approach:} It's possible to give an alternative (and perhaps better?) definition of a monoidal category using the holistic approach, it can be found in \cite[\S 1.1]{Lurie}. This definition is complicated so we won't state it precisely here, but the rough idea is the following. 

Let $\mc{A}^{\otimes n}$ be the category of sequences of $n$-elements in $\mc{A}$. A monoidal category is equivalence to a whole host of functors of the flavor 
\begin{align*}
\mc{A}^{\otimes 5} &\rightarrow \mc{A}^{\otimes 3} \\
(X_1, X_2, X_3, X_4, X_5) &\mapsto (X_1 \otimes X_2, \ind, X_3 \otimes X_4 \otimes X_5),
\end{align*}
together with many many compatibilities.

\subsection{Modules for monoidal categories}
Similarly, we can formulate the notion of a module for a monoidal category either in terms of generators and relations, or holistically.

\vspace{3mm}
\noindent
{\bf Generators and relations:} A left module $M$ for a monoidal category $\mc{A}$ consists of a functor 
\[
\otimes: \mc{A} \times \mc{M} \rightarrow \mc{M},
\]
together with natural transformations
\[
\beta_{X, Y, M}: (X \otimes Y) \otimes M \xrightarrow{\sim} X \otimes (Y \otimes M),
\]
\[
\lambda_M: \ind \otimes M \xrightarrow{\sim} M, \text{ and }
\]
\[
\rho_M: M \otimes \ind \xrightarrow{\sim} M
\]
such that the analogous pentagon and unit diagrams commute for all objects $X, Y \in \mc{A}$ and $M \in \mc{M}$. 

\vspace{3mm}
\noindent
{\bf Holistic:} A similar approach is possible, see \cite[\S2]{Lurie}. 

\vspace{3mm}
In representation theory, we are familiar with the fact that representations of an algebra are equivalent to modules over that algebra. An analogous statement holds in $2$-representation theory. A {\em representation} of $\mc{A}$ is a pair $(\mc{M}, \phi)$ of a category $\mc{M}$ and a $\otimes$-functor $\phi: \mc{A} \rightarrow \End(\mc{M})$. 

\begin{proposition}
\[
\left\{ \mc{A}\text{-modules} \right\} \simeq \left\{ \mc{A}\text{-representations} \right\}.
\]
\end{proposition}
\begin{exercise}
(Which Geordie assumes is true but hasn't done.) Check that this is an equivalence of $2$-categories. 
\end{exercise}

\subsection{Examples of modules over monoidal categories}
\label{examples of modules}
In representation theory, we started by studying groups, then realized that many of our questions about linear group symmetry could be formulated in terms of representations of algebras. In a similar vein, we can start $2$-representation theory by building categories from groups. 

Given a discrete group $G$, construct a category $\mc{A}_G$ as follows: 
\begin{itemize}
    \item {\bf Objects:} $\{r_g \mid g \in G\}$ satisfying $r_g \otimes r_h = r_{gh}$, $\ind = r_{id}$, and $\alpha, \lambda, \rho$ are all the identity. 
    \item {\bf Morphisms:} $\Hom(r_g, r_h) = \emptyset$ for $g \neq h$, and $\End(r_g) = \{id_{r_g}\}$. 
\end{itemize}

\noindent
{\bf What is a module over $\mc{A}_G$?} Given a category $\mc{M}$ and a $\otimes$-functor $F: \mc{A}_G \rightarrow \End(\mc{M})$, define $F_g:=F(r_g)$. Then the $\otimes$-functor $F$ gives us natural isomorphisms 
\[
\mu_{g h}:F_g F_h \xrightarrow{\sim} F_{gh} \text{ and } \epsilon:F_{id} \xrightarrow{\sim} id_{\mc{M}} 
\]
such that 
\[
\begin{tikzcd}
F_g F_h F_k \arrow[r, "F_g \mu_{h, k}"] \arrow[d, "\mu_{g, h}F_k"'] & F_g F_{hk} \arrow[d, "\mu_{g, hk}"] \\
F_{gh} F_k \arrow[r, "\mu_{gh, k}"] & F_{ghk}
\end{tikzcd}, \quad \hspace{2mm} 
\begin{tikzcd}
F_{id} F_g \arrow[r, "\mu_{id, g}"] \arrow[d, "\epsilon F_g"'] & F_g \arrow[dl, "id"] \\
F_g
\end{tikzcd}, \quad \text{ and }
\begin{tikzcd}
F_g F_{id} \arrow[r, "\mu_{g, id}"] \arrow[d, "F_g \epsilon"'] & F_g \arrow[dl, "id"] \\
F_g
\end{tikzcd}
\]
commute for all $g, h, k \in G$. This is called a {\em strict action} of $G$ on $\mc{M}$. 
\begin{exercise}
\label{categorical Z action}
Show that giving an action of $\Z$ on $\mc{M}$ is the same as giving an autoequivalence $E$ of $\mc{M}$. 
\end{exercise}

\begin{example}
\label{not enough coherence}
(An example to urge caution) What does it mean to give an action of $\Z/2\Z$ on $\mc{M}$? 
\begin{itemize}
    \item {\bf First guess:} A category $\mc{M}$ with an autoequivalence $E:\mc{M} \rightarrow \mc{M}$ together with a natural isomorphism $m:E^2 \xrightarrow{\sim} id_\mc{M}$. 

    This is wrong! Here's why: It follows from the definition of a categorical action that the diagram \begin{equation}
    \label{EEE}
        \begin{tikzcd}
        EEE \arrow[d, "mE"'] \arrow[r, "Em"] & E \\ E \arrow[ur, dash, "id"']
        \end{tikzcd}
    \end{equation}
    commutes. (This is a consequence of the previously displayed commutative square.)  But we can cook up an example of data as above where it doesn't. Let $\mc{M}$ be the category with two objects $X$ and $Y$ and morphisms 
    \[
    \Hom(X,Y)=\Hom(Y,X)=0, \hspace{ 2mm} \End(X)=\End(Y)=k,
    \]
    where $k$ is a field. 
    Define $E:\mc{M} \rightarrow \mc{M}$ by $E(X)=Y$, $E(Y)=X$, and $id_X \xmapsto{E} id_Y \xmapsto{E} id_X$. A picture of this category:
    \[
    \includegraphics[scale=0.2]{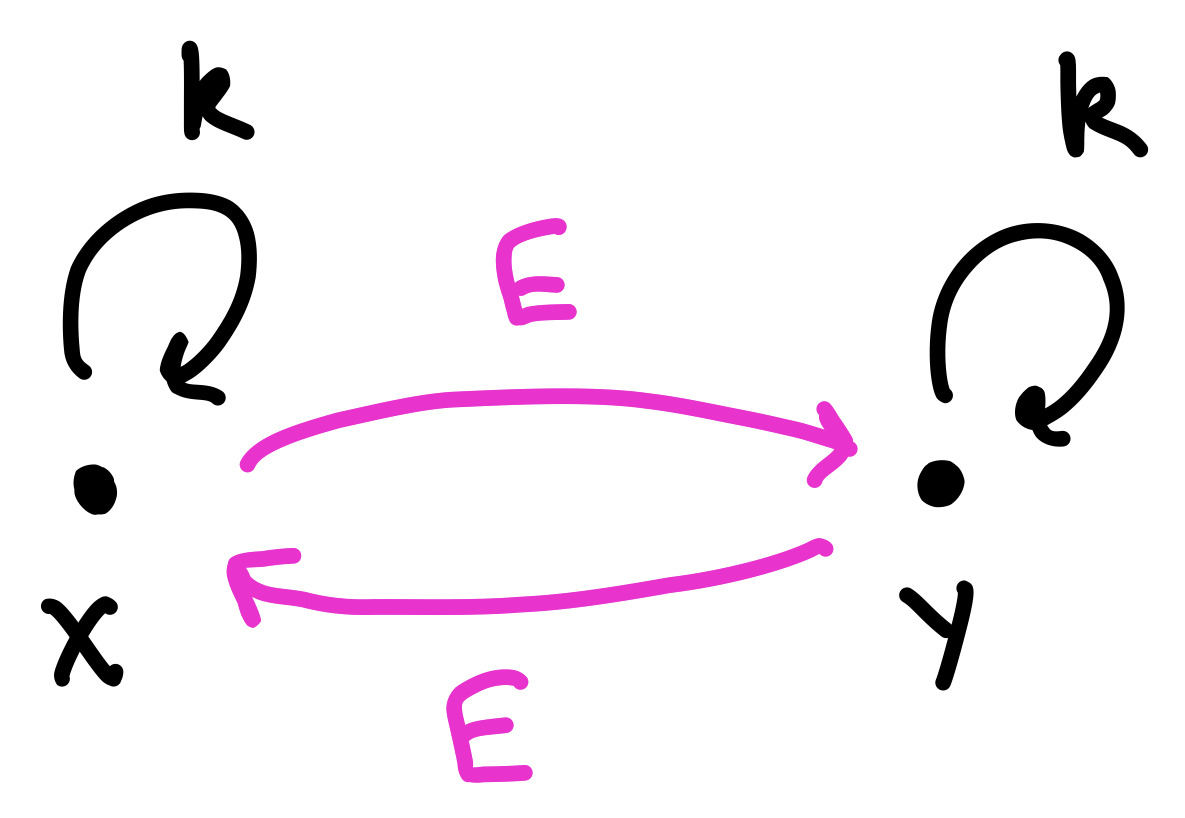}
    \]
    
    Note that $E^2=id_\mc{M}$, so any $m:E^2 \xrightarrow{\sim} id_\mc{M}$ is a natural transformation $m:id_\mc{M} \rightarrow id_\mc{M}$; i.e. $m \in Z(\mc{M})$. Hence 
    \[
    m=\begin{cases} a & \text{ on } \End(X) \\ b & \text{ on } \End(Y) \end{cases}
    \]
    for some $a, b \in k$.  
Now in this example, the diagram (\ref{EEE}) on $X$ becomes 
\[
\begin{tikzcd}
Y \arrow[r, "a"] \arrow[d, "b"'] & Y \\ Y \arrow[ur, "id_Y"'] & 
\end{tikzcd}
\]
This diagram commutes if and only if $a=b$. However, any choice of $a$ and $b$ defines a natural transformation $m:E^2 \rightarrow id_\mc{M}$, so our first guess must be wrong! 
\item {\bf Revised guess (Exercise):}  Show that the data of our first guess along with the requirement that (\ref{EEE}) commutes does determine an action of $\Z/2\Z$ on $\mc{M}$. 
\end{itemize}
\end{example}

\subsection{What is going on here?} 
We will briefly take a small diversion and discuss a beautiful dictionary related to these concepts. As we discussed in the beginning of this lecture, in the generators and relations approach to categorical actions, we want to include ``just enough'' coherence data. Example \ref{not enough coherence} illustrates that sometimes it can be subtle to determine how much coherence data is enough. The reader might be wondering:

\begin{center}
    How do we determine the necessary coherences? 
\end{center}

It turns out that they are determined by the ``cells of $BG$''. We will illustrate what we mean by this through a series of examples.  
\begin{example} Associated to $G$ is its classifying space $BG$\footnote{Recall that $BG$ is only defined up to homotopy. We will ignore this point below.}. 
\begin{enumerate}
    \item For $G=\Z$, $BG=S^1$. This is a CW complex consisting of one $0$-cell and one $1$-cell. Recall that in Exercise \ref{categorical Z action} we showed that an action of $\Z$ on a category $\mc{M}$ is given by an autoequivalence $E: \mc{M} \rightarrow \mc{M}$:
    \[
    \includegraphics[scale=0.2]{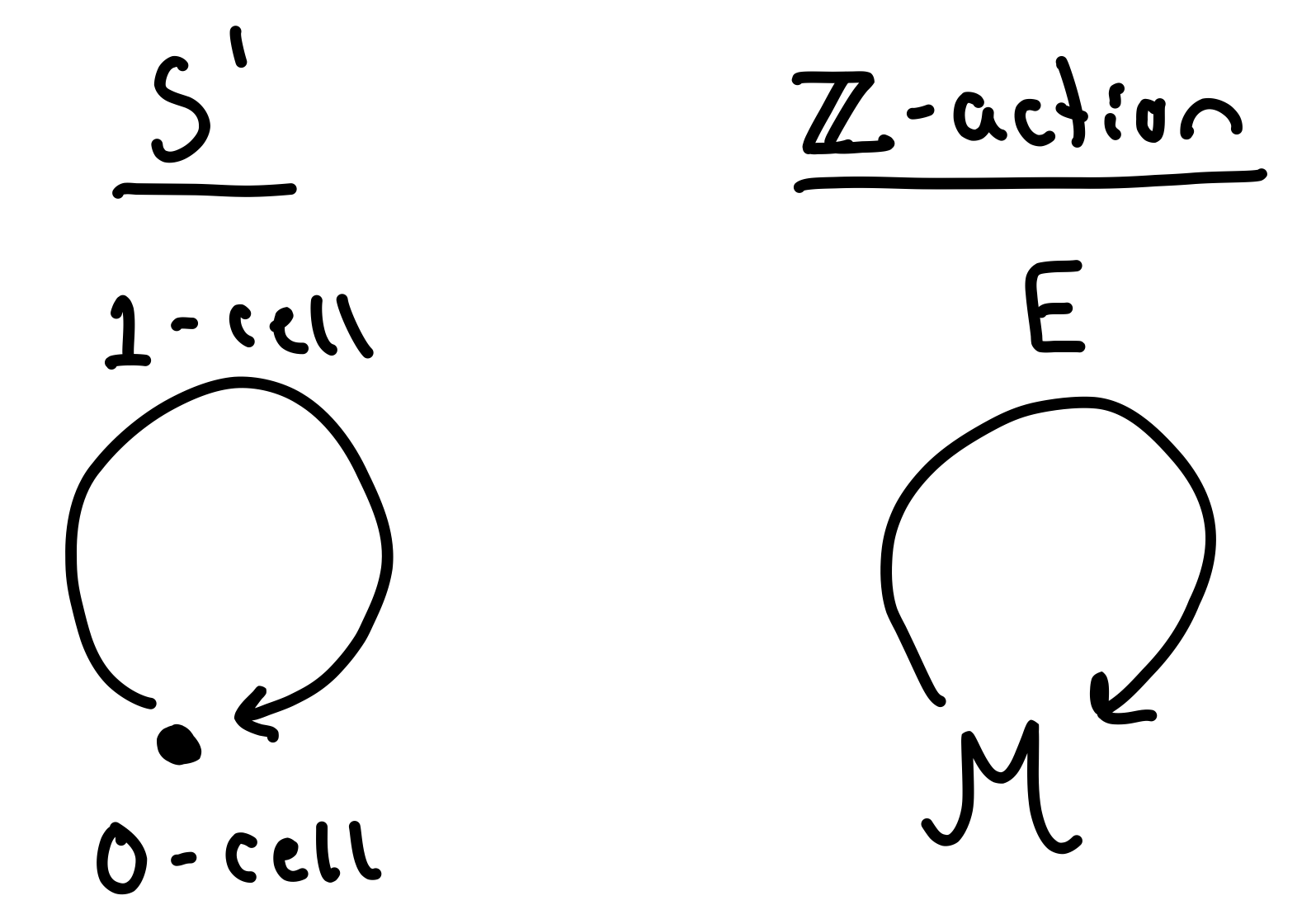}
    \]
    Hence we have one piece of ``coherence data'' for each cell of $BG$.
    \item For $G=\Z/2\Z$, $BG=\R\PP^\infty=S^\infty / \{ \pm 1\}$.  The sphere $S^\infty$ has a CW complex structure with two $n$-cells for each $n \in \Z_{\geq 0}$:
    \[
    \includegraphics[scale=0.2]{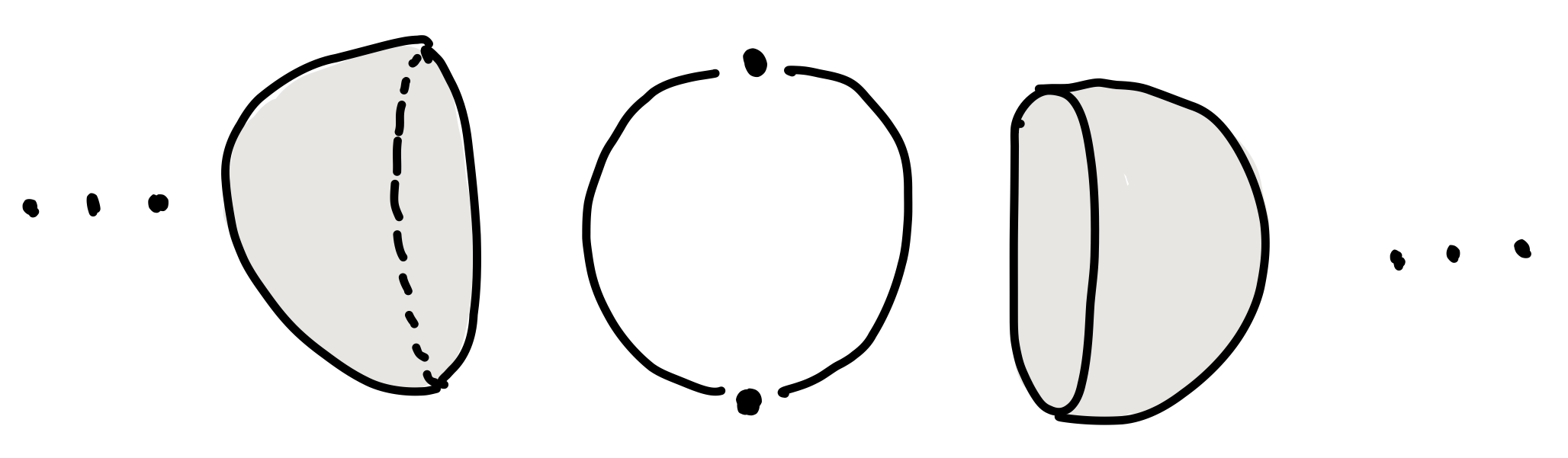}
    \]
    Moreover, the action of $\Z/2\Z$ on this cell complex is cellular. Hence $\R\PP^\infty$ is a CW complex with one $n$-cell for each $n \in \Z_{\geq 0}$. We saw in Example \ref{not enough coherence} that an action of $\Z/2\Z$ on $\mc{M}$ is determined by $(\mc{M}, E, m:E^2 \simeq id_\mc{M}, (\ref{EEE}))$:
    \[
    \includegraphics[scale=0.2]{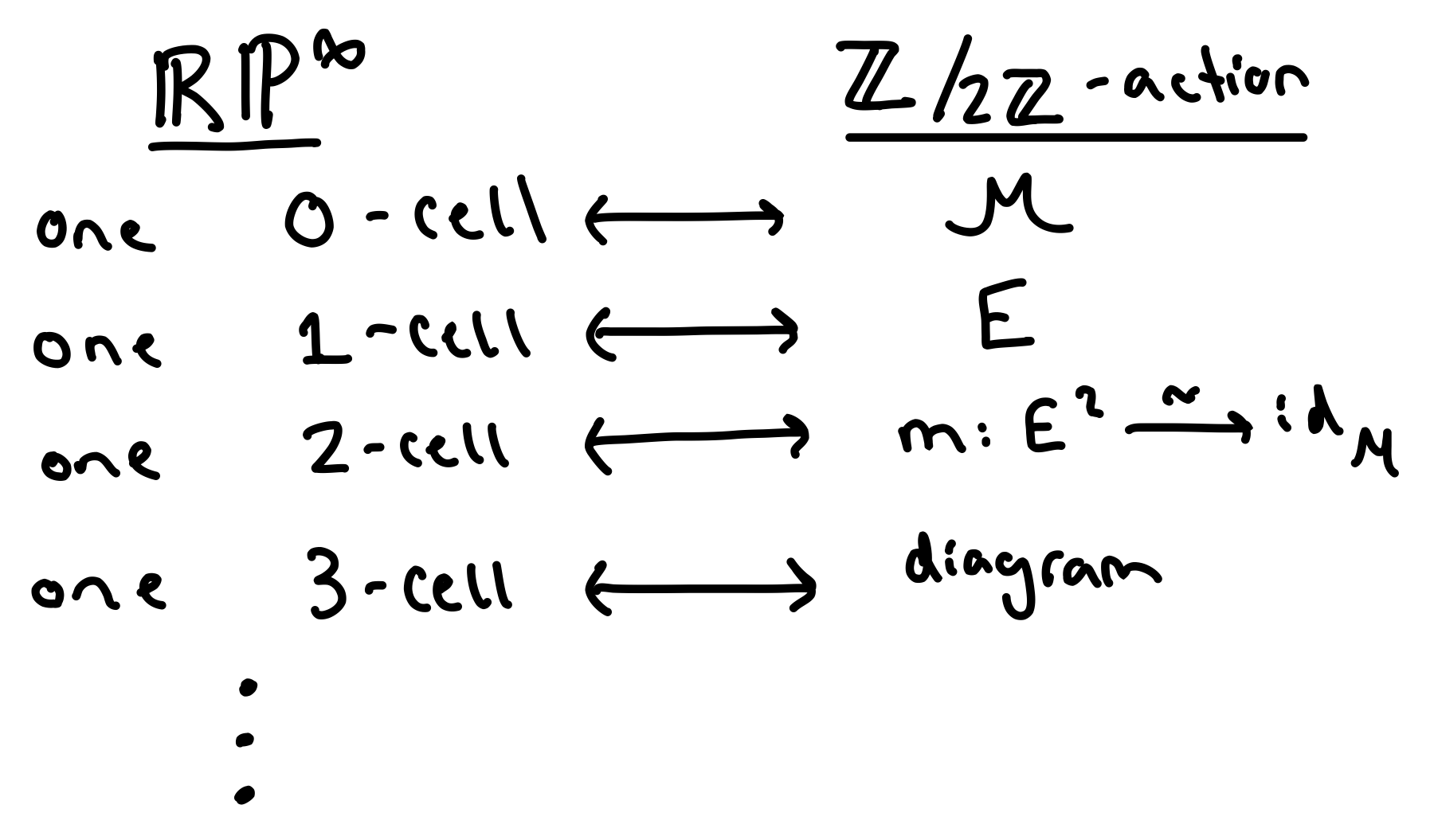}
    \]
    Again, for each $n$-cell with $0 \leq n \leq 3$, we have a corresponding piece of ``coherence data''.
    \item For any group $G$, the Milnor construction of $BG$ yields a CW complex strucure with one $0$-cell, $G$ 1-cells, $G \times G$ 2-cells, $G \times G \times G$ 3-cells, etc. Comparing this to the  definition of a $G$-action on $\mc{M}$ in Section \ref{examples of modules}, we see:
    \[
    \includegraphics[scale=0.2]{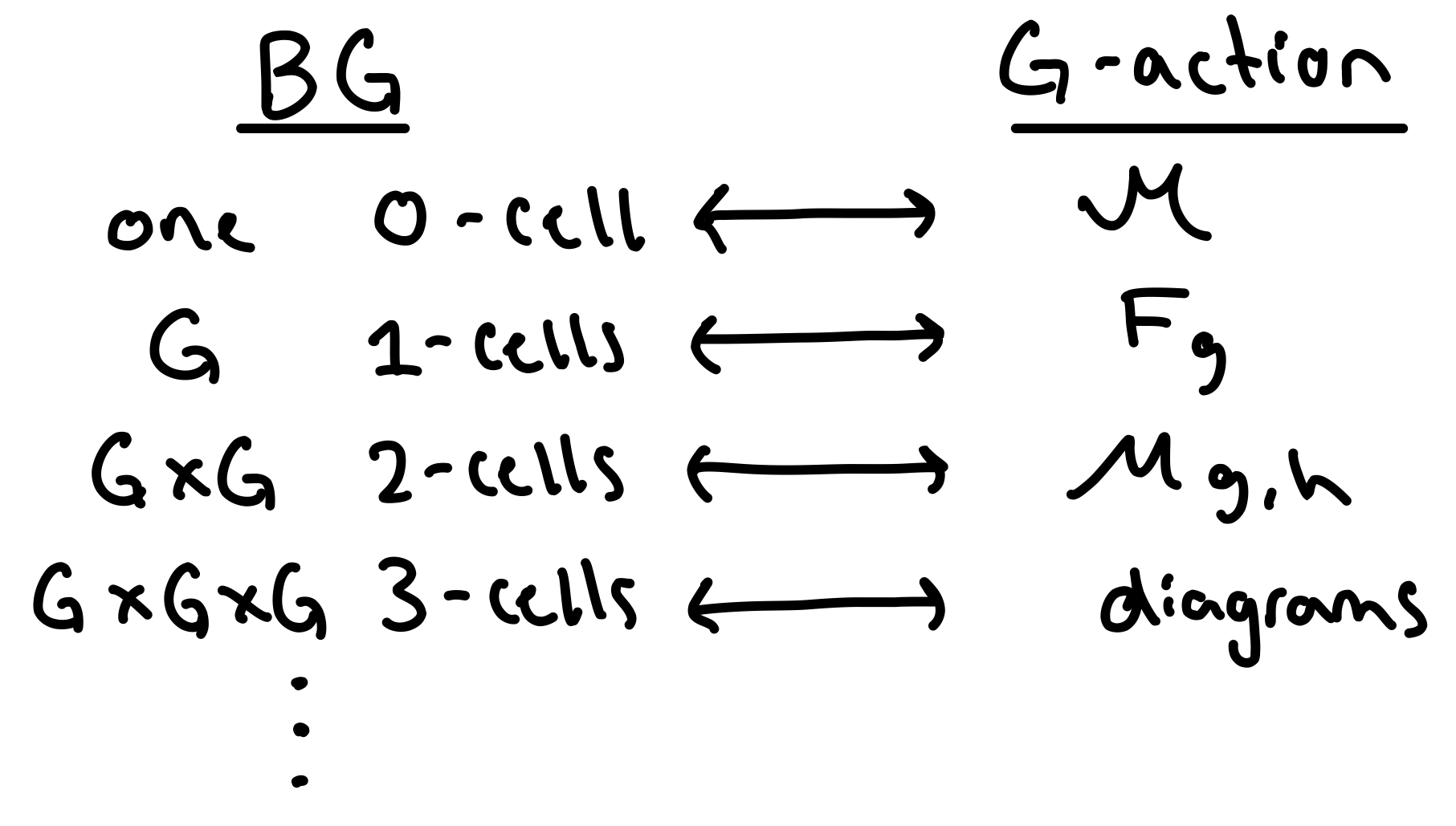}
    \]
    So each of piece of coherence data in the definition of a $G$-action corresponds to an $n$-cell of $BG$.
\end{enumerate}
\end{example}
In each example above, the CW complex structure gives us the necessary coherence data to determine the action of $G$. In parts 1 and 2, we used a CW complex with a small number of cells to obtain ``just enough'' coherence data, illustrating the ``generators and relations'' approach. In part 3, we used a CW complex with many cells to capture ``all'' coherence data of the $G$-action, illustrating the ``holistic'' approach. 

This pattern is very pretty, but {\bf why do we stop at $3$-cells?} This is because we are acting on $1$-categories. In general, 
    \[
    \includegraphics[scale=0.2]{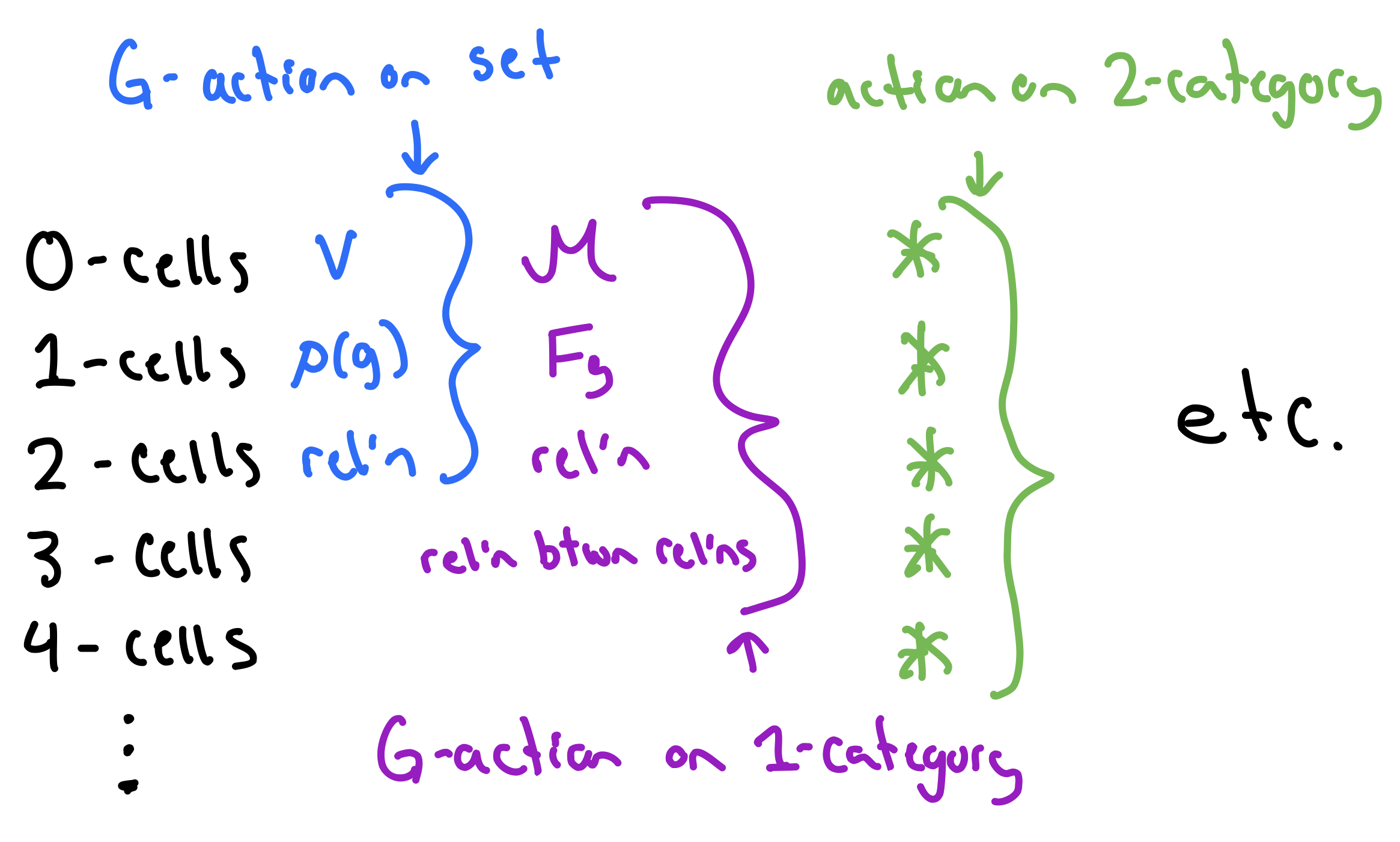}
    \]
    
\subsection{Example: Representations of the Verlinde category} 

For some reason, an action of the category $\mc{A}_G$ is not as powerful as a $G$-action on a vector space. 

\vspace{2mm}
\noindent
{\bf Reasons:} (Speculative)
\begin{enumerate}
    \item We need more structure to categorify linear algebra. For example, what is the eigenvalue of a functor? If $\Z \circlearrowright \mc{M}$, how do we take its logarithm? In examples of categorical actions (e.g. of braid groups in highest weight representation theory), there are often interesting morphisms $F_g \rightarrow F_h$ and these need to be studied. 
    \item It is possible that this isn't the right generalization. In this course, a much more important action will be 
    \[
    \Rep G \circlearrowright \mc{M},
    \]
    where $G$ is a linear algebraic group. In other words, we will study ``representations of representations of $G$''.
\end{enumerate}

Next lecture we will discuss the general case; today we will give a simple and beautiful example. To do so we need to introduce the Verlinde category. 

Recall that for each $m \geq 0$, $\mf{sl}_2(\C)$ has a unique simple module $V_m$ of highest weight $m$ and dimension $m+1$. In the category $\Rep{\mf{sl}_2(\C)}$, the module $V_0$ is the monoidal unit, and other representations multiply according to the rule 
\begin{equation}
\label{multiplication in RepG}
V_1 \otimes V_m \simeq V_m \otimes V_1 \simeq V_{m+1} \oplus V_{m-1}.
\end{equation}

For any positive integer $\ell \geq 1$, there exists a monoidal category $C_\ell$ which can be seen as a ``finite'' version of $\Rep{\mf{sl}_2(\C)}$. It has $\ell+1$ simple objects, $V_0, \ldots, V_\ell$, $V_0$ is the unit, and multiplication\footnote{One can remember (\ref{multiplication in Verlinde}) by noticing that it is the same as (\ref{multiplication in RepG}), except that all classes $V_m$ for $m>\ell$ are zero.} is given by 
\begin{equation}
\label{multiplication in Verlinde}
V_1 \otimes V_m = V_m \otimes V_1 = V_{m+1} \oplus V_{m-1} \text{ for }1 \leq m < \ell, \text{ and } V_\ell \otimes V_1 = V_1 \otimes V_\ell \simeq V_{\ell-1}.
\end{equation}
\begin{remark}
There are two realizations of $C_\ell$ (see \cite{Ostrik}, \cite{Kac}). The first is as the category of level $\ell$ representations of the affine Lie algebra $\widehat{\mf{sl}}_2(\C)$ with its fusion product. The second is as the semi-simplification of representations of the Lusztig form of the quantum group of $\mf{sl}_2$ at a root of unity. In both settings, $C_\ell$ is braided but (in contrast to $\Rep \mf{sl}(2, \C)$) not symmetric. The braiding will not play a role below. 
\end{remark}
\begin{theorem}
Let $\mc{A}=C_\ell$. Then 
\[
\left\{ \begin{array}{c} \text{ semi-simple abelian $\mc{A}$-modules $\mc{M}$ such that} \\ \text{1. $\mc{M}$ has finitely many simple objects, and } \\
\text{2. $\mc{M}$ is generated by a single simple object} \end{array} \right\}_{/{\simeq}} \xleftrightarrow{\sim} \left\{ \begin{array}{c} \text{simply laced} \\ \text{Dynkin diagrams} \\ \text{with Coxeter number $\ell+2$} \end{array} \right\}. 
\]
\end{theorem}
\begin{example}
The category $C_{10}$ admits $3$ module categories on the right hand side, corresponding to the $A_{11}, D_7$, and $E_6$ Dynkin diagrams. 
\end{example}

\begin{proof}
(Sketch) Let us explain how to go from the LHS to RHS. Starting with $\mc{M}$ on the LHS, we can build a graph as follows (cf. McKay correspondence):
\begin{itemize}
      \item {\bf Vertices}: One vertex for every simple module in $\mc{M}$.
      \item {\bf Edges:} An edge between $M$ and $M'$ if $M$ is a summand of $V_1 \otimes M'$. 
   \end{itemize}
Just as in $\Rep{\mf{sl}_2(\C)}$, $V_1$ is self-biadjoint in $\mc{A}$, so 
\[
M \text{ is a summand of } V_1 \otimes M' \iff M' \text{ is a summand of }V_1 \otimes M.
\]
Hence the graph associated to any $\mc{A}$-module $\mc{M}$ in this way is undirected.

From this graph, one can use (\ref{multiplication in Verlinde}) to deduce how $V_m$ acts on the Grothendieck group $[\mc{M}]$ for all $m$. For example, if $D$ denotes the adjacency matrix of the graph associated to $\mc{M}$, then the action of $V_2$ on $\mc{M}$ is described by $D^2-id$ because $[V_2]=[V_1^{\otimes 2}]-[V_0]$ in $C_\ell$. Now some linear algebra allows one to deduce that the graph we construct from $\mc{M}$ must either be a simply laced affine Dynkin diagram or a ``tadpole'' $\vcenter{\hbox{\includegraphics[scale=0.1]{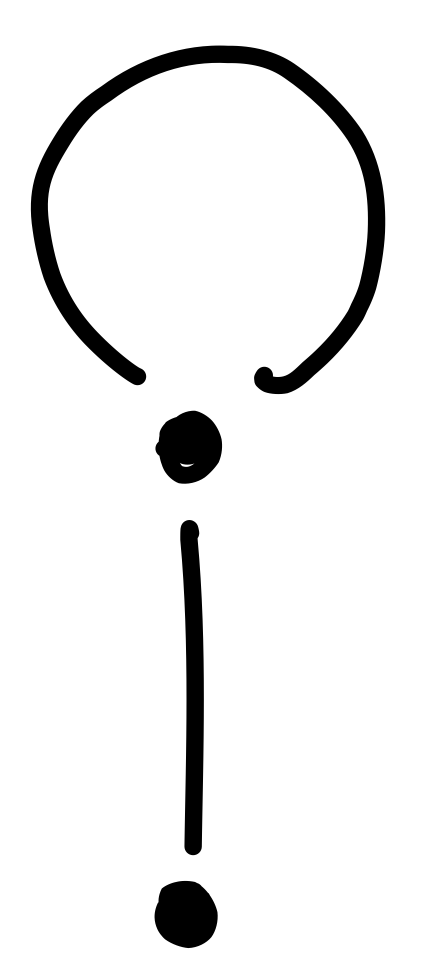}}}$. Then we can work harder to rule out the tadpole and show that the graph determines $\mc{M}$ up to equivalence.  
\end{proof}

\begin{remark}
In lectures, it was incorrectly stated that one can classify module categories over $\Rep{\mf{sl}_2(\C)}$ in a similar fashion. It turns out that the answer here is much more complicated. For example, in this setting, graphs of the form $\vcenter{\hbox{\includegraphics[scale=0.1]{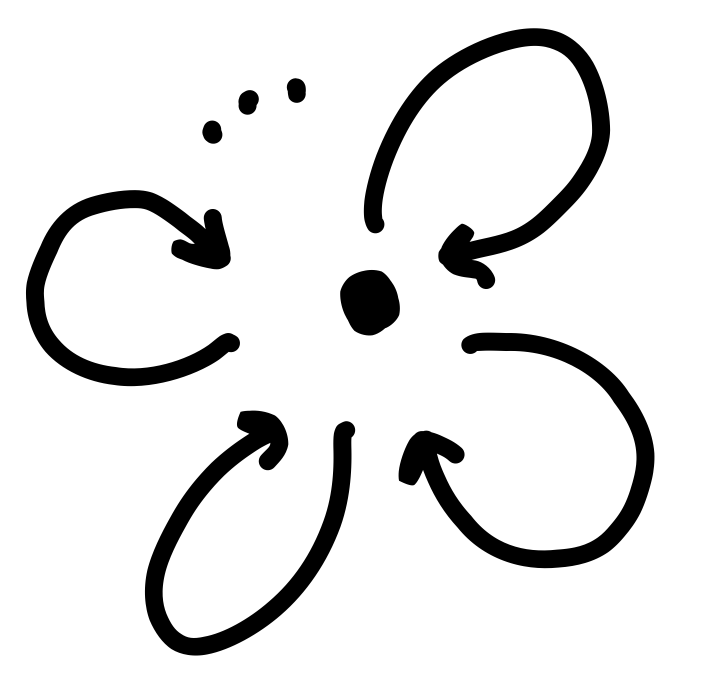}}}$ with any $n\ge 2$ loops appear. For more details\footnote{An explanation of why these flower-shaped graphs appear can also be found in the notes from Anna's talk ``Representations of representations of $\mf{sl}_2(\C)$'' at Macquarie University, available on her website \url{www.maths.usyd.edu.au/u/romanova}.}, the reader is referred to \cite{EO}. 
\end{remark}

%% file: lecture-27.tex
\section{Lecture 27: Abelian categories over stacks}
\label{lecture 27}

\subsection{Modules over algebras in monoidal categories}

{\bf General goal:} Get a feeling for module categories (i.e. categories acted on by monoidal categories). 
\vspace{2mm}

One of the first surprises in representation theory is that what appears to be external search (the search for all representations of a group) turns out to be an internal search (all representations occur in the regular representation). A similar phenomenon happens in the representation theory of monoidal categories.

Given a monoidal category $(\mc{A}, \otimes)$, an {\em algebra} $A$ in $\mc{A}$ is an object $A$ together with morphisms 
$\ind \rightarrow A$ and $m:A \otimes A \rightarrow A$ satisfying the axioms that one might expect. Given an algebra, let 
\[
\Mod_\mc{A}(A) = \text{ category of right $A$-modules in }\mc{A},
\]
where a right $A$-module in $\mc{A}$ is an object $M$ and a morphism $a:M \otimes A \rightarrow M$ satisfying the usual right module axioms. For objects $X \in \mc{A}$ and $M \in \Mod_{\mc{A}}(A)$, $X \otimes M$ is a right $A$-module via $id_X \otimes a$, so the category $\Mod_\mc{A}(A)$ is a left $\mc{A}$-module. 

\vspace{2mm}
\noindent
{\bf Meta Theorem:} ``$\mc{A}$-modules are $A$-modules.'' 
\vspace{2mm}

The Meta Theorem captures the external-becomes-internal phenomenon we see in representation theory: the external search for $\mc{A}$-modules becomes an internal search for algebras in $\mc{A}$. For a precise version of the Meta Theorem in the case of a rigid semisimple category $\mc{A}$ with finitely many simple objects, see \cite{Ostrik}. 

\begin{example} {\bf Module categories for finite groups.}

Let $G$ be a finite group and let $k$ be an algebraically closed field, such that $k[G]$ is semi-simple. Let $H \subset G$ be a subgroup. The restriction functor 
\[
\Rep G \rightarrow \Rep H
\]
is a $\otimes$-functor, and it gives $\Rep H$ the structure of a $\Rep G$-module. The following exercise shows that this $\mc{A}$-module is an $A$-module (for $\mc{A}=\Rep G$, $A=k[G/H]$). 
\begin{exercise}
Show that $k[G/H]$ is an algebra in $\Rep G$, and $\Rep H \simeq \Mod_{\Rep G} -k[G/H]$. 
\end{exercise}

\begin{center}
{\bf Question:} Are these all semisimple indecomposible $\Rep G$-modules?
\end{center}

\begin{theorem}
\label{subgroups with sauce}
\cite{Ostrik} 
\[
\left\{ \begin{array}{c} \text{indecomposible} \\ \text{module categories} \\ \text{over }\Rep G \end{array} \right\}_{/ \simeq} \xleftrightarrow{\sim} \left\{ (H, \omega) \left| \begin{array}{c} H \subset G \text{ subgroup} \\ \omega \in H^2(H, k^\times) \end{array}\right. \right\}_{G\text{-conjugacy}}.\\
\]
Under this correspondence, the pair $(H, \omega)$ corresponds to the category $\Rep^1(\widetilde{H})$ of representations of the central extension 
\[
k^\times \rightarrow \widetilde{H} \rightarrow H
\]
on which $k^\times$ acts via the identity character.
\end{theorem}
We can sum up this example with the following motto:

\begin{center}
 ``Module categories over $\Rep{G}$ are the same thing as subgroups with sauce.''
\end{center}
\end{example}

\begin{remark} 
\begin{enumerate}
    \item The proof of Theorem \ref{subgroups with sauce} is easy once one has the Meta Theorem. 
    \item Objects on the LHS are not always monoidal categories as they are when the central extension is trivial. 
    \item {\bf Question:} In stacks language, $\Rep H$ as a $\Rep G$-module corresponds to $\mathrm{pt}/H \rightarrow \mathrm{pt}/G$. Can one interpret $H^2(H, k^\times)$ as twistings of $\mathrm{QCoh}(\mathrm{pt}/H)$? 
\end{enumerate}
\end{remark}

\subsection{Linear algebra over a stack}
\label{linear algebra over a stack}
For the remainder of this lecture, we will summarize a short note of Gaitsgory on categories with actions of algebraic stacks \cite{Gaitsgory}. Recall our motivation: In lecture \ref{lecture 25}, we introduced the ``constructable'' anti-spherical category $\mc{M}^\mathrm{asp}$. We wish to show that $\mc{M}^\mathrm{asp} \simeq D^{G^\vee \times \mathbb{G}_m}(\widetilde{\mc{N}})$, and our strategy for doing so is to view $\mc{M}^\mathrm{asp}$ progressively as a module over $\mathrm{QCoh}(\mathrm{pt}/G^\vee)$, then $\mathrm{QCoh}(\mf{g}^\vee / G^\vee)$, then $\mathrm{QCoh}(\widetilde{\mc{N}}/G^\vee)$.  The following language will be useful. 

\vspace{2mm}
\noindent
{\bf The notion of a category over an algebraic stack:}

Let $\mc{C}$ be a $k$-linear category. Assume that $\mc{C}$ is closed under inductive limits. (E.g. $\Vect$ and $\mathrm{QCoh}(X)$ are okay, $\Vect_{\mathrm{f.d.}}$ and $\Coh(X)$ are not.) Let $\mc{Y}$ be a stack. Our goal is to address the following question:

\vspace{2mm}
\begin{center}
    What does it mean for $\mc{C}$ to be linear over $\mc{Y}$?
\end{center}
\vspace{2mm}

\noindent
{\bf Step 1:} Start by assuming $\mc{Y}$ is an affine scheme; i.e. $\mc{Y}=\Spec{A}$ for some $k$-algebra $A$. In this case, $\mc{C}$ being linear over $\mc{Y}$ should mean that $\mc{C}$ is $A$-linear; i.e. we have a map
\[
A \rightarrow Z(\mc{C}):=\End(id_\mc{C}).
\]
\begin{example}
If $X \xrightarrow{f} \Spec{A}$, then $\mathrm{QCoh}(X)$ is $A$-linear. 
\end{example}

\noindent
{\bf Claim:} If $\mc{C}$ is $A$-linear, then $\mathrm{QCoh}(\mc{Y})$ acts on $\mc{C}$; i.e. we have 
\[
\mathrm{QCoh}(\mc{Y}) \times \mc{C} \rightarrow \mc{C}.
\]

\noindent
{\bf Example 4.5.}
(continued) $\mathrm{QCoh}(\Spec{A})$ acts on $\mathrm{QCoh}(X)$ via $f^*(-) \otimes (-)$.

\begin{proof}
We want to define $M \otimes X$ for $M \in \mathrm{QCoh}(\mc{Y})$, $X \in \mc{C}$. If $M=A^I$ (recall that $\mc{Y}=\Spec{A}$), define $M \otimes X:=X^I$. In general, choose a presentation 
\[
A^I \rightarrow A^J \twoheadrightarrow M
\]
of $M$, and define $M \otimes X:= \coker(A^I \otimes X \rightarrow A^J \otimes X)$, where the $ij^{\mathrm{th}}$-matrix coefficient is given by the $a_{ij}$ action on $\Hom(X, X)$. 
\begin{exercise}
Show that this is independent of presentation. 
\end{exercise}
We conclude that $\mathrm{QCoh}(\mc{Y}) \circlearrowright \mc{C}$.
\end{proof}

On the other hand, given $\mathrm{QCoh}(\Spec{A}) \circlearrowright \mc{C}$, we can also go back; that is, we can recover the $A$-linear structure on $\mc{C}$. Indeed, because $\ind_{\mathrm{QCoh}(\Spec{A})}=A$, for any $M \in \mc{C}$ we have a map 
\[
\End(A)=A \rightarrow \End(\ind \otimes M)=\End(M). 
\]
Hence we obtain a map $A \rightarrow Z(\mc{C})$; i.e. an $A$-linear structure on $\mc{C}$. 

We would like to have some notion of base change in this setting. But first we need a pull-back. For a map $A' \leftarrow A$ of $k$-algebras, we want functors $f_*, f^*$ fitting into a diagram 
\[
\begin{tikzcd}
\mc{C}' \arrow[r, shift right, "f_*"'] \arrow[d, dashrightarrow] & \mc{C} \arrow[l, shift right, "f^*"'] \arrow[d, dashrightarrow]\\
\Spec(A')=:S' \arrow[r] &\Spec(A)=:S
\end{tikzcd}
\]
We define the category $\mc{C}'=\mc{C} \times_S S'$ as follows: 
\begin{itemize}
    \item {\bf Objects:} objects of $\mc{C}$, equipped with an additional action of $A'$ such that 
    \[
    \begin{tikzcd}
    & A \arrow[dl] \arrow[dr] & \\ 
    A' \arrow[rr] & & \End(M) 
    \end{tikzcd}
    \]
    commutes. 
    \item {\bf Morphisms:} morphisms in $\mc{C}$ compatible with the $A'$-action. 
\end{itemize}
\begin{exercise}
(If you are stuck, ask Emily!) Show that $\mc{C}'=\mc{C} \times_{S} S'$ is abelian.
\end{exercise}

\begin{example}
Let $\mc{C}=\mathrm{QCoh}(S)=A$-Mod. Then $\mc{C} \times_S S'$ is the category of $A$-modules with $A'$-actions; that is, 
\[
\mc{C} \times_S S' = A'\text{-Mod}=\mathrm{QCoh}(S').
\]
\end{example}

Returning to the general setting, we see that our desired functors \[
\begin{tikzcd}
\mc{C} \times_S S' \arrow[r, shift right, "f_*"'] \arrow[d, dashrightarrow] & \mc{C} \arrow[l, shift right, "f^*"'] \arrow[d, dashrightarrow] \\
S' \arrow[r] & S
\end{tikzcd}
\]
are given by 
\begin{align*}
f_*&= \text{ forget $A'$-structure}, \\
f^* &: X \mapsto A' \otimes X.
\end{align*}
It is not difficult to see that they form an adjoint pair $(f^*, f_*)$. 

\subsection{A brief review of descent}

{\bf The idea:} Often in mathematics, we can construct an object by constructing it ``locally''.

\vspace{2mm}
More specifically, the principle of descent is that for some mathematical object $Z$,
\[
\text{giving }Z \text{ on $X$} \longleftrightarrow \begin{array}{c} \text{giving pieces of $Z$ on} \\ \text{a ``cover'' of $X$} + \text{glue}.
\end{array}
\]
\begin{remark}
One of Grothendieck's great insights is that here ``cover'' can be interpreted {\em very} generally.
\end{remark}

\noindent
Different types of objects require different glue:
\begin{itemize}
    \item {\bf Function:} Define a function $f$ on $X$ by giving functions $f_i$ on $U_i$ such that 
    \[
    f_i|_{U_{ij}}=f_j|_{U_{ij}}.
    \]
    In this setting the glue is a ``truth value'' (i.e. the functions agree or don't). If you want to be more fancy, you can interpret a function as a ``sheaf of truth values''.
    \item {\bf Sheaf:} Define a sheaf $\mc{F}$ by giving sheaves $\mc{F}_i$ on $U_i$ and morphisms $\alpha_{ij}:\mc{F}|_{U_{ij}} \rightarrow \mc{F}|_{U_{ij}}$ such that 
    \[
    \begin{tikzcd}
    \mc{F}_i|_{U_{ijk}} \arrow[rr, "\alpha_{ij}"] \arrow[dr, "\alpha_{ik}"'] & & \mc{F}_j|_{U_{ijk}} \arrow[dl, "\alpha_{jk}"] \\
    & \mc{F}_k|_{U_{ijk}}& 
    \end{tikzcd}
    \]
    commutes. In this setting, the morphisms are the glue, which is subject to the restriction imposed by the diagram. 
    \item {\bf Category:} Define a category $\mc{C}$ by giving categories $\mc{C}_i$ on $U_i$, functors $\alpha_{ij}:\mc{C}_i|_{U_{ij}} \rightarrow \mc{C}_j|_{U_{ij}}$, and natural transformations $\beta_{ijk}:\alpha_{jk}\circ \alpha_{ij} \xrightarrow{\sim} \alpha_{ik}$ such that 
    \begin{equation}
    \label{coherence}
    \begin{tikzcd}
    & \alpha_{kl} \circ \alpha_{jk} \circ \alpha_{ij} \arrow[dl] \arrow[dr] & \\ \alpha_{jl} \circ \alpha_{ij} \arrow[dr]  & & \alpha_{kl} \circ \alpha_{ik} \arrow[dl] \\ 
    & \alpha_{il} & 
    \end{tikzcd}
    \end{equation}
    commutes. In this setting, both the functors and natural transformations make up the glue. 
\end{itemize}

\noindent
{\bf Topological approach:} Again coherence conditions come from a topological space. Assume for simplicity $X=\bigcup_
{i \in I} U_i$ is a cover (in the classical sense). Then there is a bijection 
\begin{align*}
    \text{simplicial complexes with $I$ vertices} &\leftrightarrow \text{the indexing set $I$ of the cover} \\
    \text{$n$-simplices} &\leftrightarrow \begin{array}{c}\text{$n$-element subsets $K$ of $I$} \\ \text{such that $\cap_{i \in K} U_i \neq \emptyset$}.\end{array} \end{align*}
\begin{example}
    \[
    \includegraphics[scale=0.2]{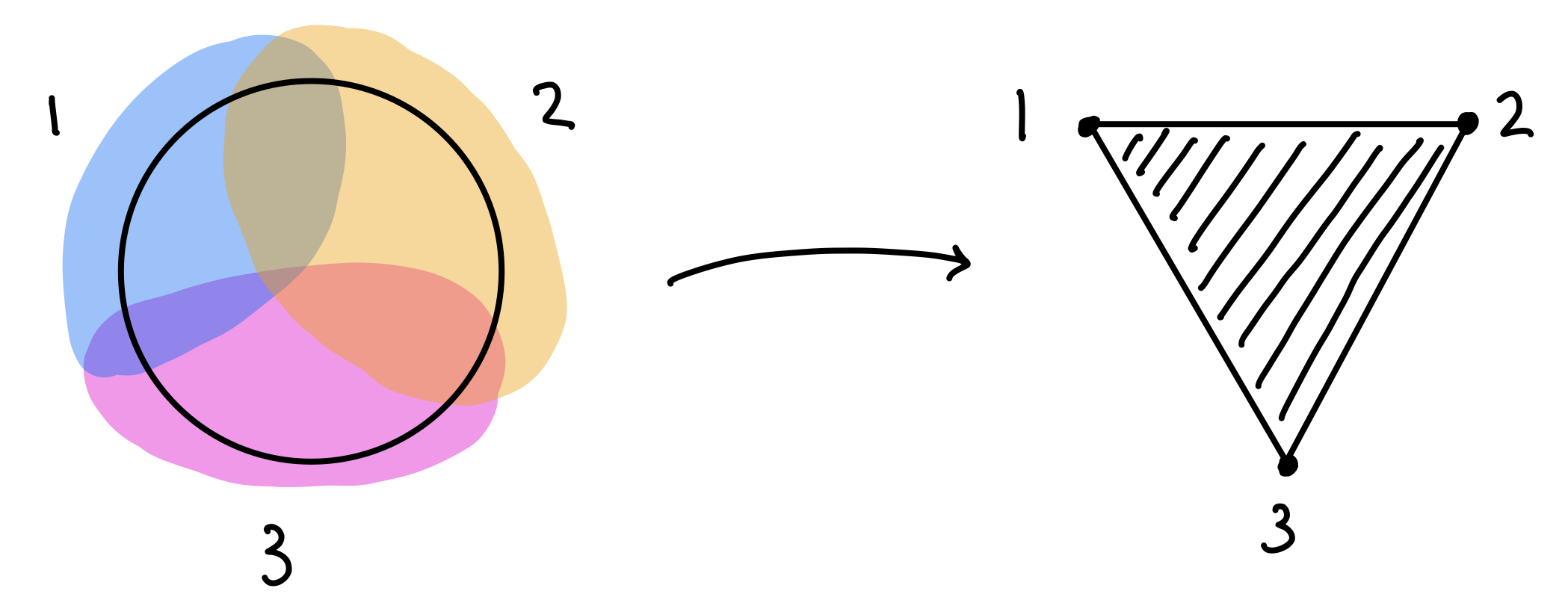}
    \]
\end{example}
Coherence conditions come from cells:
\begin{itemize}
    \item {\bf Function:} 0 and 1 cells. 
    \item {\bf Sheaf:} 0, 1, and 2 cells. 
    \item {\bf Category:} 0, 1, 2, and 3 cells. 
    \item {\bf Etc.} 
\end{itemize}
\begin{example}
    Coherence condition (\ref{coherence}) above corresponds to the $3$-cell in the $3$-simplex:
        \[
    \includegraphics[scale=0.2]{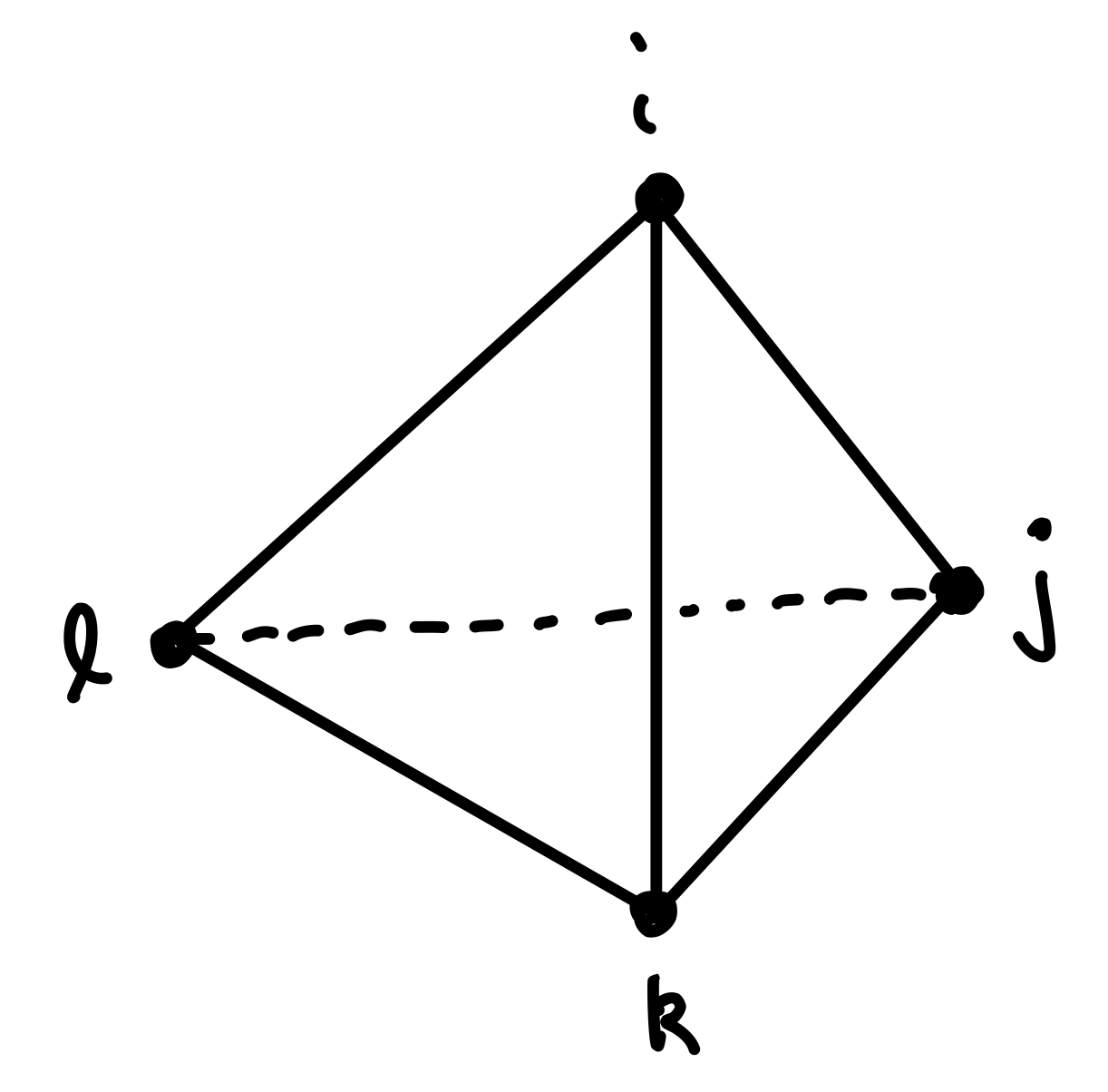}
    \]
\end{example}

\subsection{Back to sheaves over stacks}
Return to the setting of the end of Section \ref{linear algebra over a stack} and the diagram \[
\begin{tikzcd}
\mc{C} \times_S S' \arrow[r, shift right, "f_*"'] \arrow[d, dashrightarrow] & \mc{C} \arrow[l, shift right, "f^*"'] \arrow[d, dashrightarrow] \\
S' \arrow[r] & S
\end{tikzcd}
\]
\begin{lemma} (Key Lemma) Sheaves of categories satisfy descent with respect to flat covers of affine schemes; that is, if $S' \rightarrow S$ is a flat cover of affine schemes, then, 
\[
\begin{array}{c} \text{sheaf of categories}\\ \text{over $S$} \end{array}  \xrightarrow{\sim} \begin{array}{c} \text{sheaf of categories over $\mc{C}'$} \\ + \text{ descent data} \end{array}.
\]
\end{lemma}

Now let $\mc{Y}$ be an algebraic stack in the faithfully-flat sense, with affine diagonal. This means that any morphism 
\[
\Spec{A} \rightarrow \mc{Y}
\]
with source an affine scheme is affine. Recall from Emily's IFS talk {\em Sheaves on stacks}, that we define a quasi-coherent sheaf on $\mc{Y}$ by imagining that there is a quasi-coherent sheaf ``$\mc{F}$'' on $\mc{Y}$ and axiomatising what this would mean for $f^*$``$\mc{F}$'' for all $f:\Spec{A} \rightarrow \mc{Y}$.

Given an algebraic stack $\mc{Y}$ as above, let $\mathrm{Sch}_{\mc{Y}}^\mathrm{aff}$ denote the category of affine schemes over $\mc{Y}$. We define a {\em sheaf of categories over $\mc{Y}$} to be an assignment $\mc{C}^\mathrm{sh}$:
\begin{enumerate}
    \item $S\in \mathrm{Sch}_\mc{Y}^\mathrm{aff} \mapsto \mc{C}_S$, linear over $S$; 
    \item for each $S' \rightarrow S$ in $\mathrm{Sch}_\mc{Y}^\mathrm{aff}$, an $S$-linear functor $f^*:\mc{C}_S \rightarrow \mc{C}_{S'}$ inducing an equivalence $\mc{C}_S \times_S S' \xrightarrow{\sim}\mc{C}_{S'}$;
    \item for $S'' \xrightarrow{g} S' \xrightarrow{f} S$ in $\mathrm{Sch}_\mc{Y}^\mathrm{aff}$, an isomorphism $g^* \circ f^* \simeq (f \circ g)^*$ such that $3$-way compatibility holds.
\end{enumerate}
\begin{remark}
This is holistic definition. We could also give a ``generators and relations'' definition via a fixed flat cover. 
\end{remark}
\begin{example}
    If $\mc{Y}=\mathrm{pt}$, then the assignment $S \mapsto \mathrm{QCoh}(S)$ is a sheaf of categories on $\mc{Y}$. 
\end{example}
\begin{exercise}
Given $\mc{C}^\mathrm{sh}$ on $\mc{Y}$, define $\Gamma(\mc{Y}, \mc{C}^\mathrm{sh})$ and show that it is a $\Vect_\mc{Y}$-module. (Here $\Vect_{\mc{Y}}$ is the category of vector bundles on $\mc{Y}$.)
\end{exercise}
Suppose that $\mc{Y}$ satisfies:
\begin{enumerate}
    \item $\mc{Y}$ is locally Noetherian and every quasicoherent sheaf is a limit of coherent sheaves, and 
    \item every coherent sheaf is a quotient of a vector bundle (e.g. ``enough ample line bundles'').
\end{enumerate}
Then we have the following theorem. (See Gaitsgory's note \cite{Gaitsgory} for a proof.)
\begin{theorem}
\begin{align*}
    \begin{array}{c} \text{sheaves of categories} \\ 
    \text{over }\mc{Y} \end{array} &\leftrightarrow \Vect_{\mc{Y}}\text{-modules} \\
    \mc{C}^\mathrm{sh} &\mapsto \Gamma(\mc{Y}, \mc{C}^\mathrm{sh}).
\end{align*}
\end{theorem}

\begin{example}
\label{key example}
    (Key example) Let $G$ be a linear algebraic group. Then 
    \[
    \begin{array}{c} \text{sheaf of categories} \\ \text{on }\mathrm{pt}/G \end{array} \leftrightarrow \Vect_{\mathrm{pt}/G} \simeq \Rep^{\mathrm{f.d.}}G \text{-module}.
    \]

To see this, recall Emily's approach to understanding sheaves on $\mathrm{pt}/G$. She started with a prestack $(\mathrm{pt}/G)^\mathrm{triv}$ given by \[
(\mathrm{pt}/G)^\mathrm{triv}(S) = \text{trivial $G$-bundle on $S$} \circlearrowleft G(S),
\]
then applied stackification to obtain the stack $\mathrm{pt}/G$. A coherent sheaf on $(\mathrm{pt}/G)^\mathrm{triv}$ is the data 
\[
V + G(A) \circlearrowright V \otimes A
\]
for every affine $k$-scheme $S=\Spec{A}$. In other words, an algebraic representation of $G$. 

Similarly, a category over $(\mathrm{pt}/G)^\mathrm{triv}$ is the data
\[
k\text{-linear category } \mc{C} + G(A) \circlearrowright \mc{C}_S
\]
for all affine $S=\Spec{A}$. 

Consider the case of a $G$-scheme $X$. Then the assignment 
\[
S/k \text{ affine} \mapsto \mathrm{QCoh}(X_S) \circlearrowleft G(S)
\]
is a sheaf of categories over $\mathrm{pt}/G$. On the other hand, 
\[
\mathrm{QCoh}^G(X) = \mathrm{QCoh}(X/G)
\]
is a $\Rep{G}$-module. Under the above equivalence, these correspond to one another. 
\begin{remark}
We will have more to say about this correspondence in a slightly different language in the next lecture.
\end{remark}
\end{example}

%% file: lecture-28.tex
\section{Lecture 28: (De)equivariantisation}
\label{lecture 28}

We explained last week that if $\mc{Y}$ is an algebraic stack satisfying certain assumptions, then \[
\begin{array}{c} \text{sheaf of abelian}\\ \text{categories $\mc{C}^\mathrm{sh}$ on $\mc{Y}$} \end{array} \longleftrightarrow \begin{array}{c} \Vect_{\mc{Y}}\text{-module} \\
\text{(i.e. category with an action}\\
\text{of vector bundles on $\mc{Y}$)} \end{array}
\]
The most basic example of this is the following. Let $G$ be a linear algebraic group, and $\mc{Y} = \mathrm{pt}/G$. Recall that $\mathrm{pt}/G$ is the stackification of the prestack $(\mathrm{pt}/G)^\mathrm{triv}$, which is given by the assignment 
\[
S \mapsto G \times S
\]
of a scheme $S$ to the trivial $G$-bundle on $S$. As a stack in groupoids, this assignment sends
\[
\includegraphics[scale=0.15]{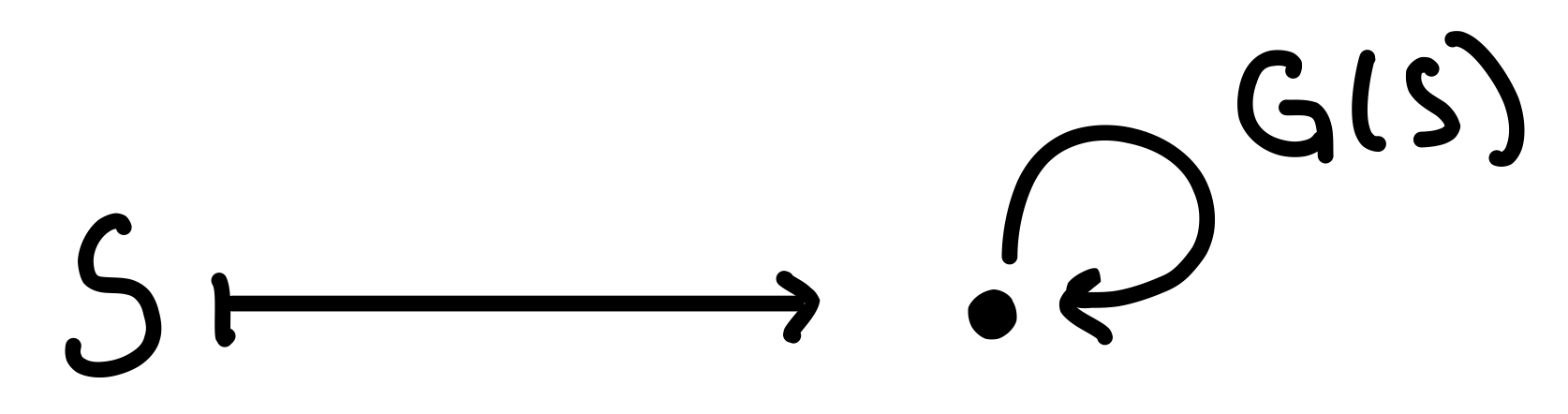}
\]

Recall that a quasi-coherent sheaf on $\mathrm{pt}/G$ is the data of a vector space $V$ (the value of this sheaf on a point), together with an action of the $S$-points $G(S)$ on $V \times S$ for all test schemes $S$. This is known as an algebraic representation. 

Similarly, a sheaf of categories on $\mathrm{pt}/G$ (an object on the left hand side of the correspondence above) is the data of a category $\mc{C}$ (the value of this sheaf on a point), together with an action of the $S$-points $G(S)$ on $\mc{C}_S$ for all test schemes $S$. 

Hence for the stack $\mathrm{pt}/G$, the correspondence above matches \[
\left( S \mapsto \mc{C}_S= \begin{array}{c} \text{ abelian category} \\ \text{with $G(S)$-action} \end{array}\right) \longleftrightarrow \Rep_{f.d.}{G}\text{-module}.
\]

\subsection{A simple example of this phenomenon} 
\begin{remark}
This example should have come earlier, but Geordie didn't discover the very clear reference \cite{DGNO} providing this perspective until last week. 
\end{remark}

Let $G$ be a finite group and $\mc{C}$ an additive monoidal category. 
\begin{theorem}
There are functors ``equivariantisation'' and ``deequivariantisation'' which allow us to move between categories with $G$-actions and categories with $\Rep{G}$-actions:  
\[
\begin{tikzcd}
\left\{ \begin{array}{c} \text{$k$-linear categories} \\ \text{with $G$-action} \end{array} \right\} \arrow[rr, shift right, "\text{equivariantisation}"'] \arrow[dr, leftrightarrow] & & \left\{ \begin{array}{c} \text{$k$-linear categories} \\ \text{with $\Rep{G}$-action} \end{array} \right\} \arrow[ll, shift right, "\text{deequivariantisation}"'] \arrow[dl, leftrightarrow] \\
& \left\{ \begin{array}{c} \text{$\Rep{G}$-enriched} \\ \text{categories} \end{array} \right\} &
\end{tikzcd}
\]
On Karoubian $k$-linear categories, the upper two arrows provide equivalences of categories. The lower arrows can be made into equivalences too, with additional minor assumptions. \end{theorem}

\vspace{1mm}
\begin{center}
{\bf How do these functors work? }
\end{center}

\subsection{Equivariantisation} Suppose $G$ acts on a category $\mc{M}$; i.e. we have morphisms $F_g:\mc{M} \rightarrow \mc{M}$ for each element $g \in G$ and natural isomorphisms $\mu_{gh}:F_g \circ F_h \xrightarrow{\sim} F_{gh}$ for each pair $g, h \in G$. (See lecture \ref{lecture 26} for a review of this notion.) 

An {\em equivariant object} in $\mc{M}$ is a tuple $(X, \{u_g\}_{g \in G})$ where $X \in \mc{M}$ and $u_g:F_g(X) \xrightarrow{\sim} X$ such that the diagram 
\[
\begin{tikzcd}
F_g(F_h(X))\arrow[d, "\mu_{gh}(X)"] \arrow[r, "F_g(u_h)"] & F_g(X) \arrow[d, "u_g"] \\ 
F_{gh}(X) \arrow[r, "u_{gh}"] & X
\end{tikzcd}
\]
commutes. 

An {\em equivariant morphism} in $\mc{M}$ is a morphism $X \xrightarrow{f} Y$ in $\mc{M}$ which commutes with all $u_g$; that is, a morphism such that the diagram 
\[
\begin{tikzcd}
F_g(X) \arrow[d, "F_g(f)"'] \arrow[r, "u_g"] & X \arrow[d, "f"]\\ F_g(Y) \arrow[r, "u_g"] & Y 
\end{tikzcd}
\]
commutes for all $u_g$. 

The {\em equivariantisation} $\mc{M}^G$ of $\mc{M}$ is the category whose objects are equivariant objects in $\mc{M}$ and whose morphisms are equivariant morphisms in $\mc{M}$. 

\begin{remark}
The category $\mc{M}^G$ is a $\Rep{G}$-module. Indeed, given $V \in \Rep{G}$ and $(X, \{u_g\}_{g \in G}) \in \mc{M}^G$, define $V \otimes X:=(V \otimes X, {u_g^\vee}_{g \in G})$, where \[
u_g^\vee:F_g(V \otimes X) \simeq V \otimes F_g(X) \xrightarrow{\rho_g \otimes u_g} V\otimes X. 
\]
Here $\rho:G \rightarrow \Aut(V)$ is the $G$-representation structure on $V$. 
\end{remark}

\begin{exercise} If $X$ is a $G$-scheme, show that 
\begin{enumerate}
    \item $G$ acts on $\Coh{X}$, and 
    \item the equivariantisation of $\Coh(X)$ is equivalent to the category of $G$-equivariant sheaves on $X$:
    \[
    (\Coh(X))^G \simeq \Coh^G(X).
    \]
\end{enumerate}
\end{exercise}

\subsection{Deequivariantisation}
Let $A$ be the algebra of functions from $G$ to $k$. The algebra $A$ carries commuting left and right $G$-actions:
\begin{itemize}
    \item Left action: $(g \cdot f)(x)=f(g^{-1}x)$
    \item Right action: $(f \cdot g)(x) = f(xg)$. 
\end{itemize}
The left action of $G$ gives $A$ the structure of a $G$-representation, and it is an algebra object in the category $\Rep{G}$. The right $G$-action means that $G$ acts on $A$, preserving its $\Rep{G}$-algebra structure.

If $\mc{N}$ is a $\Rep{G}$-module, define a category $\mc{N}_G$, the {\em deequivariantisation} of $\mc{N}$, by 
\begin{itemize}
    \item {\bf Objects:} $A$-modules in $\mc{N}$; i.e. objects $Y$ together with a morphism $a:A \otimes Y \rightarrow Y$ satisfying the appropriate module conditions
    \item {\bf Morphisms:} $A$-linear morphisms in $\mc{N}$.
\end{itemize}

\begin{example}
Let $\mc{N}$ be the category of $G$-equivariant coherent sheaves on $X$, and $A=k[G]$ as above. Then 
\begin{align*}
    \mc{N}_G &= \text{$A$-modules in }\Coh^G(X) \\
    &= \mc{O}_X[G]\text{-modules in }\Coh^G(X) \\
    &= \Coh^G(G \times X) \\
    &= \Coh(X). 
\end{align*}
\end{example}

\subsection{$\Rep{G}$ enrichments}

Both equivariantisation and deequivariantisation can be realised as a two-step process in which the first step passes through a  $\Rep{G}$-enrichment (i.e. a category whose Hom spaces are objects in $\Rep{G}$) of our original category. We describe this perspective now. 

Let $\mc{M}$ be a category with an action of $G$, and let $M, M'$ be equivariant objects in $\mc{M}$. 

\begin{claim}
\label{two steps}
$\Hom_\mc{M}(M, M')$ is a $G$-module, and 
\[
\Hom_{\mc{M}^G}(M, M') = \Hom_\mc{M}(M, M')^G.
\]
\end{claim}
\begin{proof}
The action is given by 
\[
(M \xrightarrow{f} M') \xmapsto{g \cdot} \left( \begin{tikzcd}
F_g(M) \arrow[r, "F_g(f)"] \arrow[d, "\sim"] & F_g(M') \arrow[d, "\sim"] \\ 
M \arrow[r, "g \cdot f"] & M'
\end{tikzcd} \right).
\]
The second statement follows from the definitions. 
\end{proof}
From $\mc{M}$, we can construct a $\Rep{G}$-enriched category by considering just the equivariant objects of $\mc{M}$, but keeping all morphisms between those objects. This is a category where $G$ fixes all objects but moves morphisms. Claim \ref{two steps} shows that from this $\Rep{G}$-enrichment of $\mc{M}$, we can then pass to the equivariantisation of $\mc{M}$ by only keeping morphisms which are fixed by $G$. 

On the other hand, given a $\Rep{G}$-module $\mc{N}$, we can also construct a $\Rep{G}$-enriched category. For $X,Y \in \mc{N}$, consider the functor 
\[
\Rep{G} \rightarrow k: V \mapsto \Hom(V \otimes X, Y).
\]
Because $\Rep{G}$ is semisimple, any $k$-linear functor on $\Rep{G}$ is representable. The functor above is represented by 
\[
\underline{\Hom}(X, Y) \in \Rep{G}; 
\]
i.e. 
\[
\Hom_{\Rep{G}}(V, \underline{\Hom}(X,Y)) \simeq \Hom_{\mc{N}}(V \otimes X, Y).
\]
Hence the category consisting of the objects of $\mc{N}$ with morphisms given by $\underline{\Hom}(X,Y)$ is a $\Rep{G}$-enriched category. Forgetting the $\Rep{G}$-enrichment and taking the Karoubi completion of this category results in the deequivariantisation $\mc{N}_G$ of $\mc{N}$. Note that $\Hom_{\mc{N}}(X, Y)= \underline{\Hom}(X, Y)^G$.

\vspace{2mm}
\noindent
{\bf Attempt at a big picture:}
\[
\includegraphics[scale=0.4]{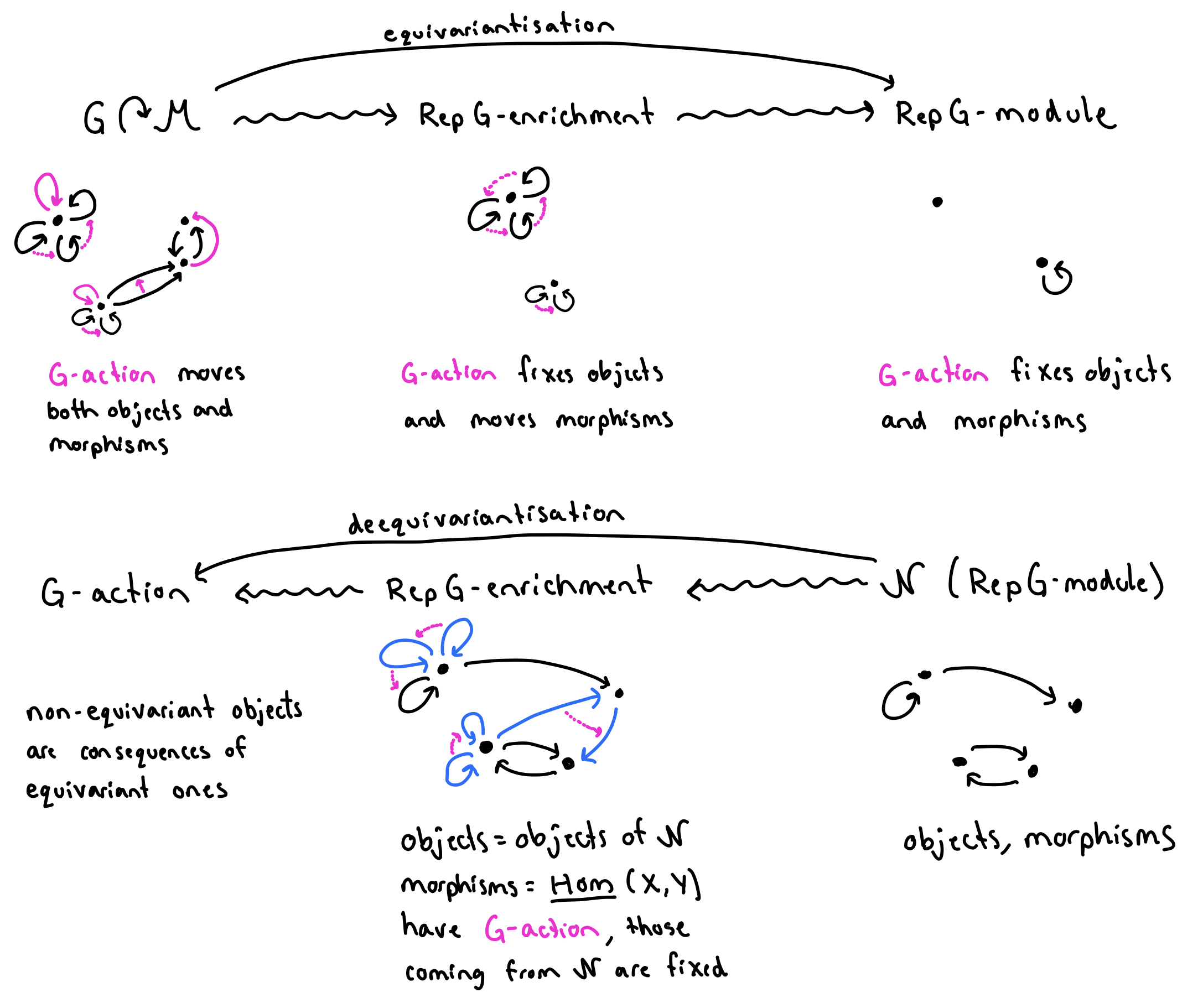}
\]
\begin{remark}
The fact that non-equivariant objects are consequences of equivariant ones is not so surprising in the case of a finite group $G$: Any $X$ is a summand of $\Ind(X) = \bigoplus F_g(X)$, which has a canonical equivariant structure.
\end{remark}

\subsection{Deequivariantisation principle}

Now we return to the setting of the beginning of the lecture. Let $S=\Spec{A}$ be an affine scheme with an action of a linear algebraic group $H$ (e.g. $S=\mf{g}^\vee$, $H=G^\vee$).  A key tool in \cite{AB, Bez} is the {\em deequivariantisation principle}, which answers the question:

\vspace{2mm}
\begin{center}
  {\bf How can we make a category $\mc{C}$ linear over $S/H$?}  
\end{center}
\vspace{2mm}

\noindent
The deequivariantisation principle states that it is enough to give:
\begin{enumerate}
    \item a $\Rep{H}$-module structure on $\mc{C}$; and 
    \item an $H$-equivariant $A$-linear structure on $\mc{C}_\mathrm{deeq}$\footnote{Here $\mc{C}_\mathrm{deeq}$ is the deequivariantisation of $\mc{C}$, the category we denoted in the previous section by $\mc{C}_G$. We are making this notational switch to align with Bezrukavnikov's notation in what is coming.} (which is equivalent to an $H$-equivariant $A$-action on $\Hom_{\mc{C}}(X, \mc{O}_H \otimes Y)$). 
\end{enumerate}

\noindent
{\bf Big diagram:} 
\[
\includegraphics[scale=0.25]{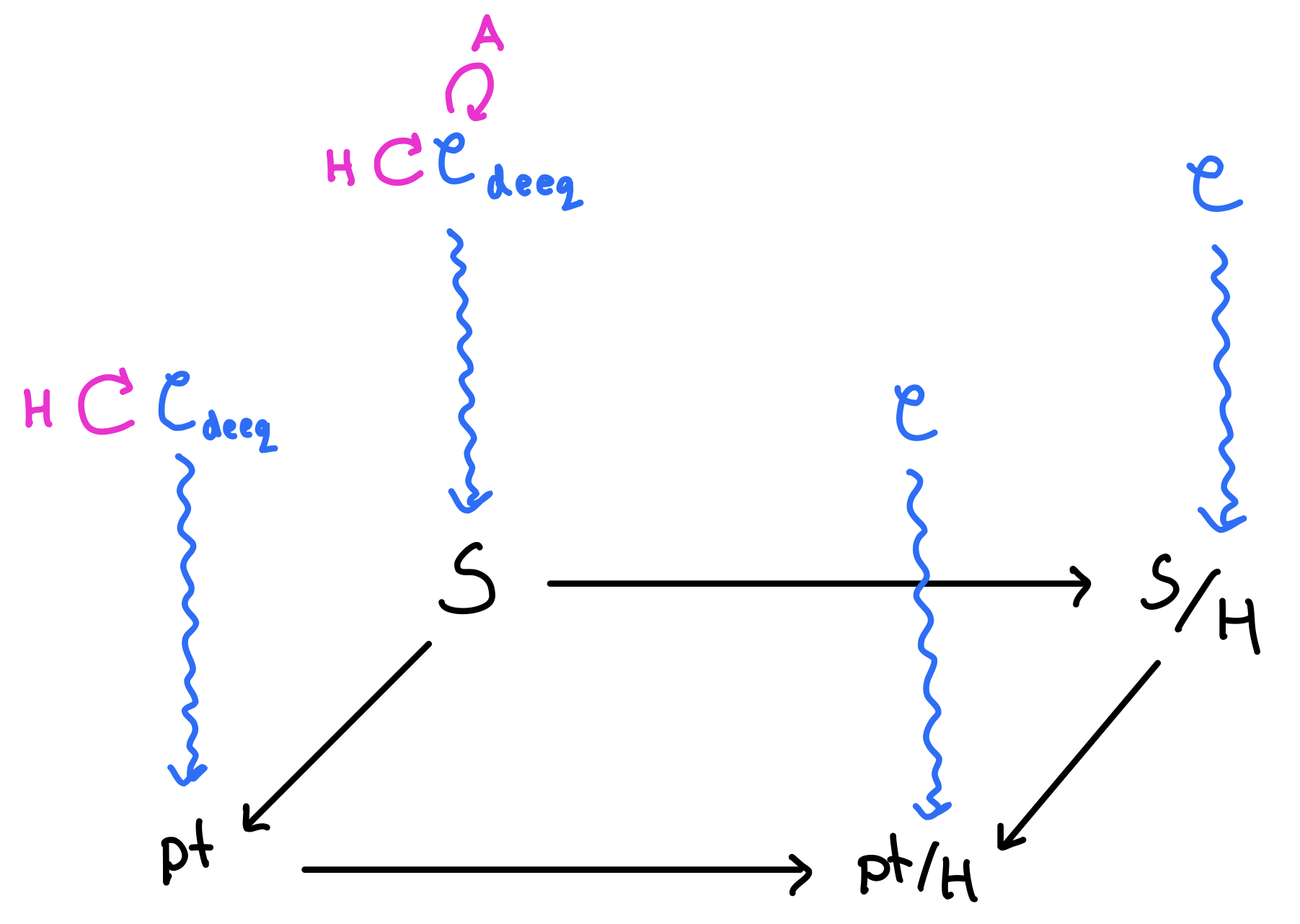}
\]
Recall that since $\mc{C}$ is a $\Rep{H}$-module, $\mc{C}_\mathrm{deeq}$ consists of $\mc{O}_H$-modules in $\mc{C}$, and $\mc{O}_H$ is an ind-object in $\Rep_\mathrm{f.d.}{H}$. Hence the  deequivariantisation functor sends 
\[
\mc{C} \rightarrow \mc{C}_\mathrm{deeq}:X \mapsto \mc{O}_H \otimes X.
\]
\begin{exercise}
The functor $\mc{O}_H \otimes (-)$ is left adjoint to the forgetful functor $\mc{C}_\mathrm{deeq} \rightarrow \mc{C}$. 
\end{exercise}

Moreover, 
\[
\Hom_{\mc{C}_\mathrm{deeq}}(\mc{O}_H \otimes X, \mc{O}_H \otimes Y) = \Hom_{\mc{C}}(X, \mc{O}_H \otimes Y). 
\]

\begin{example}
What is a sheaf of categories on $\mathbb{A}^1/\mathbb{G}_m$? By the deequivariantisation principle, we need to give:
\begin{enumerate}
    \item A $\Rep{\G_m}$-action on $\mc{C}$; i.e. an auto-equivalence $M \mapsto M(1):=\mathrm{nat} \otimes M$ of $\mc{C}$. (Recall that $\Rep{\G_m}$ is freely generated\footnote{as a $k$-linear additive tensor category} by the natural representation $\mathrm{nat}$.) 
    \item A $\G_m$-equivariant  $\mc{O}_{\mathbb{A}^1}=k[X]$-linear structure on $\mc{C}_\mathrm{deeq}$; i.e. a $k[X]$-graded module structure on 
    \begin{align*}
        \Hom_{\mc{C}_\mathrm{deeq}}(\mc{O}_{\G_m} \otimes M, \mc{O}_{\G_m} \otimes N) &= \Hom_{\mc{C}}(M, \bigoplus_{m \in \Z} N(m)) \\
        &= \bigoplus_{m \in \Z} \Hom_\mc{C}(M, N(m)) 
    \end{align*}
    with $M$ in degree $1$. In other words, this is the same data as a morphisms of functors $id \rightarrow (1)$. 
\end{enumerate}
We conclude that a sheaf of categories on $\A^1/\G_m$ consists of: (1) a category $\mc{C}$, an auto-equivalence $M \mapsto M(1)$ of $\mc{C}$, and a natural transformation $X: id \rightarrow (1)$. 

\vspace{2mm}
\noindent
{\bf The point of this example:} 
\[
{ \text{sheaf over $\A^1/\G_m$ } \atop \text{``geometric''} } {\longleftrightarrow \atop \hspace{2mm}} {\text{ auto-equivalence }(1) + \text{ morphism }id \rightarrow(1) \atop \text{``combinatorial''}} 
\]
This pattern is a feature in \cite{AB} and \cite{Bez}. 
\end{example}

\begin{exercise}
(Worthwhile). For each $\mc{Y}$ in the diagram \[
\begin{tikzcd}
& \A^1 \arrow[dr] & \\
0 \arrow[r] & 0/\G_m \arrow[r] & \A^1/\G_m \\
& \G_m/\G_m \arrow[ur] & 
\end{tikzcd}
\]
calculate $\mc{C}_\mc{Y}$ in terms of the above data.
\end{exercise}

\noindent
{\bf A bigger fish:} What does it mean to give a sheaf of categories on $\mf{g}/G$? 

\vspace{2mm}
Based on the deequivariantisation principle, this is
\[
\Rep{G}\text{-module $\mc{C}$ } +  {\text{ $G$-equivariant $\mc{O}_\mf{g}$-linear} \atop \text{structure on }\mc{C}_\mathrm{deeq} }. 
\]
\noindent
{\bf Key observation:} (To be explained below.) Via Tannakian formalism, the $\mc{O}_\mf{g}$-linear structure on $\mc{C}_\mathrm{deeq}$ is simply an ``endomorphism''. 

\subsection{Tannakian formalism}
Let $G/k$ be a group scheme, and $\mathrm{For}:\Rep{G} \rightarrow k$ the forgetful functor. We have a $k$-group functor:
\begin{align*}
    A \mapsto \Aut^\otimes (\mathrm{For} \otimes A) = \text{ $\otimes$-automorphisms of $A \otimes \mathrm{For}$}.
\end{align*}
(A $\otimes$-automorphism of $A \otimes \mathrm{For}$ is a collection of functors $F_v:V \otimes A \rightarrow V \otimes A$ for all $V \in \Rep_\mathrm{f.d.}{G}$ such that 
\[
\begin{tikzcd}
(V \otimes A) \otimes_A (V' \otimes A) \arrow[d, "\sim"] \arrow[r, "F_V \otimes F_{V'}"] & (V \otimes A) \otimes_A (V' \otimes A) \arrow[d, "\sim"] \\
(V \otimes V') \otimes A \arrow[r, "F_{V \otimes V'}"] & (V \otimes V') \otimes A
\end{tikzcd}
\]
commutes.) 

We discover that ``$G$ can be recovered from its category of representations'': 
\begin{theorem}
\label{group}
\[
G(A) \simeq \Aut^\otimes (A \otimes \mathrm{For}).
\]
\end{theorem}
How can we think about this theorem? The $A$-points of $G$ clearly give $\otimes$-automorphisms, so we have a map 
\[
G(A) \rightarrow \Aut^\otimes (\mathrm{For} \otimes A).
\]
The miracle is that this map is an isomorphism. 

\vspace{2mm}
\noindent
{\bf Question:} How can we recover the Lie algebra $\mf{g}=\Lie{G}$ from Tannakian formalism?

\begin{definition} A $\otimes$-derivation of $\mathrm{For} \otimes A$ is an endomorphism 
\[
N_V: V \otimes A \rightarrow V \otimes A
\]
for all $V \in \Rep{G}$ such that $N_{V \otimes V'}=N_V \otimes 1 + 1 \otimes N_{V'}$. \end{definition}

\begin{theorem}
\label{Lie}
\[
\mf{g} \otimes A \simeq \End^\otimes(\mathrm{For} \otimes A).
\]
\end{theorem}
\begin{exercise}
Deduce Theorem \ref{Lie} from Theorem \ref{group} using 
\[
\mf{g}=\ker\left(G\left(k[\epsilon]/(\epsilon^2)\right)\right) \rightarrow G(k).
\]
\end{exercise}
A rather startling consequence of this is the following. If $A$ is a $k$-algebra {\color{blue} with a $G$-action}, then 
\begin{align*}
\begin{array}{c} {\color{blue} \text{$G$-equivariant}} \\ \text{$\otimes$-derivation} \\ V \mapsto V \otimes A \end{array} = \begin{array}{c} \text{element} \\ x \in {\color{blue}(}\mf{g} \otimes A {\color{blue})^G} \end{array} &\iff \Hom^{{\color{blue}G}}_{\color{black}\mathrm{v.s.}}(\mf{g}^*, A)\\
&\iff \begin{array}{c} {\color{blue} \text{$G$-equivariant}} \\ \Hom_\mathrm{alg.}(\mc{O}_\mf{g}, A) \end{array}\\
&\iff {\color{blue} \Spec A/G \rightarrow \mf{g}/G}.
\end{align*}

The moral is that, for a $G$-equivariant algebra, equipping the functor $V \mapsto V\otimes A$ with a $G$-equivariant tensor derivation is the same as equipping its category of equivariant sheaves with a $\mf{g}/G$-linear structure. 

%% file: lecture-29.tex
\section{Lecture 29: Coherent sheaves on base affine space}
\label{lecture 29}

Today we will discuss coherent sheaves on base affine space. But before doing so, we'll start with a brief reminder on where we're going. 

Let $\Coh_\mathrm{free}^{G^\vee}(\widetilde{\mc{N}}) \subset \Coh^{G^\vee}(\widetilde{\mc{N}})$ be the full subcategory consisting of $V \otimes \mc{O}_{\widetilde{\mc{N}}}(\lambda)$ for $V \in \Rep{G}^\vee$ and $\lambda$ a character of $T^\vee$. As in Lectures \ref{lecture 24} and \ref{lecture 25}, let $\mc{H}$ be the affine Hecke category and $\mc{M}^\mathrm{asph}$ the categorical anti-spherical module. Our goal is to give an equivalence of categories $\Coh^{G^\vee}(\widetilde{\mc{N}}) {\color{blue} \xrightarrow{\sim}} \mc{M}^\mathrm{asph}$. To do so, we will construct a monoidal functor $F$ fitting into the following diagram.

\[
\begin{tikzcd}
\Coh^{G^\vee}(\widetilde{\mc{N}}) \arrow[rrr, bend left, blue, "\text{want}"]&  \Coh_\mathrm{free}^{G^\vee}(\widetilde{\mc{N}}) \arrow[l, hookrightarrow] \arrow[r, "\text{construct}", "F"'] & \mc{H} \arrow[r, twoheadrightarrow] & \mc{M}^\mathrm{asph}
\end{tikzcd}
\]
Once we have constructed $F:\Coh_\mathrm{free}^{G^\vee}(\widetilde{\mc{N}}) \rightarrow \mc{M}^\mathrm{asph}$, we can extend $F$ to a functor from $\Coh^{G^\vee}(\widetilde{\mc{N}})$ using the fact that any coherent sheaf admits a resolution via vector bundles. (This is analogous to the extension from an $A$-linear structure to an action of $\mathrm{QCoh}(\Spec{A})$ which occurred at the start of Lecture \ref{lecture 27}.)

How might we go about constructing such a functor $F$? The key philosophical observation is that $F$ can be built ``softly'' once one has a functor $\Rep{G^\vee} \rightarrow \mc{H}$. More precisely, one of the main technical lemmas  in \cite{AB} (in a somewhat diluted form, to aid comprehensibility at this point) is the following. 
\begin{theorem}
\label{technical lemma}
Let $\mc{C}$ be an additive monoidal category, and $F:\Rep{G^\vee} \rightarrow \mc{C}$ a $\otimes$-functor. 
\begin{enumerate}
    \item If $N$ is a $\otimes$-derivation of $F$ (cf. Lecture \ref{lecture 28}), then $F$ extends to a $\otimes$-functor 
    \[
    F:\Coh_\mathrm{free}(\mf{g}^\vee/G^\vee) \rightarrow \mc{C}.
    \]
    \item Suppose that we can upgrade $F$ to a $\otimes$-functor $\Rep{(G^\vee \times T^\vee)} \rightarrow \mc{C}$. Moreover, suppose that we have arrows 
    \[
    b_\lambda: F(V_\lambda) \rightarrow F(k_\lambda)
    \]
    satisfying the ``Pl\"{u}cker relations'' (and an ``acyclicity condition''). Then $F$ extends to a functor 
    \[
    F: \Coh_\mathrm{free}^{G^\vee}(\mc{B}^\vee) \rightarrow \mc{C}.
    \]
    (Here $\mc{B}^\vee$ denotes the flag variety of $G^\vee$.) 
    Combining 1 and 2 we obtain a functor 
    \[
    F:\Coh_\mathrm{free}^{G^\vee}(\mc{B}^\vee \times \mf{g}^\vee)\rightarrow \mc{C}.
    \]
    \item Suppose that for all $\lambda$, the arrow $b_\lambda \circ N_{V_\lambda}:F(V_\lambda) \rightarrow F(V_\lambda) \rightarrow F(k_\lambda)$ is equal to zero. Then $F$ extends to
    \[
    F:\Coh_\mathrm{free}^{G^\vee}(\widetilde{\mc{N}})\rightarrow \mc{C}.
    \]
\end{enumerate}
\end{theorem}

For the remainder of the lecture we will work entirely on the Langland's dual side, so from here on, we drop all checks from our notation.
\subsection{Serre's description of coherent sheaves on $\PP^n$}
Serre's paper \cite{FAC} was very influential, and is recommended reading for any readers who have not done so. It was one of the first papers dealing in detail with coherent sheaves and their cohomology. The last part gives a beautiful description of coherent sheaves on projective space, which we now recall.

The variety $\PP^n$ has a standard affine cover 
\[
\PP^n = \bigcup_{i=0}^n U_i, \hspace{2mm} U_i \simeq \A^n. 
\]
Using this cover, one way to view a coherent sheaf on $\PP^n$ is as 
\[
\begin{array}{c} \text{collection of $n+1$ modules} \\ \text{over $k[x_1, \ldots , x_n]$} \end{array} \hspace{3mm} + \hspace{3mm}  \text{glue}. 
\]
But this is not very practical. Instead, we can make the observation that 
\[
\Coh{\PP^n} = \Coh^{\G_m}(\A^{n+1} \backslash \{0\}) \simeq \Coh^{\G_m}(\A^{n+1})/\Coh^{\G_m}(\{0\}).
\]
If you haven't thought about this before, the second equivalence shouldn't be entirely obvious. Now, \[
\begin{tikzcd}
\Coh^{\G_m}(\A^{n+1}) & = & \begin{array}{c} \text{finitely generated graded} \\ S=k[x_0, \ldots, x_n]\text{-modules} \end{array} \\
\Coh^{\G_m}(\mathrm{pt}) \arrow[u, hookrightarrow] & = & \begin{array}{c} \text{finite-dimensional} \\ \text{modules} \end{array} \arrow[u, hookrightarrow]
\end{tikzcd}
\]
Hence 
\[
\Coh \PP^n = \left( \begin{array}{c} \text{finitely generated} \\ \text{graded $S$-modules} \end{array} \right) \left/{\left\langle \begin{array}{c} \text{finite dimensional} \\ \text{$S$-modules} \end{array} \right\rangle.}\right.
\]
Here $\langle - \rangle$ denotes a ``Serre subcategory'', and the quotient is a ``Serre quotient''. (Perhaps this language shouldn't be a surprise!) If this is new to you, the following exercise is recommended. 
\begin{exercise}
\begin{enumerate}
    \item Convince yourself that it gives the right answer for $\Coh(\PP^0)$. 
    \item The functor $\Coh\PP^n \rightarrow S$-grmod is given by 
    \[
    \Gamma:=\Hom(\bigoplus_{n \in \Z} \mc{O}(n), -).
    \]
    Hence describe $\Gamma(\mc{O}(n))$ and $\Gamma(\text{skyscraper})$. 
\end{enumerate}
\end{exercise}
\begin{remark}
We can view these descriptions through the lens of descent. The ``classical'' cover $\PP^n = \bigcup U_i$ leads to our not-very-useful description. The cover 
\[
\A^{n+1}\backslash\{0\} \rightarrow \PP^n 
\]
leads to Serre's description. Note that $\A^{n+1}\backslash \{0\}$ is ``almost'' affine. 
\end{remark}

\subsection{How can we describe coherent shaves on $G/B$?}
If $G=\SL_2$, $G/B\simeq \PP^1$, and we can use Serre's description of coherent sheaves on $\PP^n$ to describe coherent sheaves on $G/B$. For other groups, a similar philosophy can be employed. At first glance, three options present themselves for describing coherent sheaves on $G/B$:
\begin{enumerate}
    \item via covers (not very useful); 
    \item using $\Coh^G(G/B) \simeq \Coh(\mathrm{pt}/B) \simeq \Rep{B}$; 
    \item via base affine space (analogue of Serre's description). 
\end{enumerate}
What is {\bf base affine\footnote{Note that it is almost never affine, so the nomenclature is a little strange!} space}? Let $U \subset B$ be the unipotent radical. Then $G/U$ is a $G \times T$-variety via the $T$-action \[
gU\cdot t = gtU. 
\]
This is well-defined because $T$ normalizes $U$. Hence we have a quotient map 
\[
G/U \rightarrow G/B \simeq (G/U)/T.
\]
\begin{example}
Let $G=\SL_2$. Then $G \circlearrowright \C^2$ and 
\[
\mathrm{stab}\bp 1 \\ 0 \ep = \bp 1 & * \\ 0 & 1 \ep = U.
\]
Hence $G/U \simeq$ orbit of $\bp 1 \\ 0 \ep \simeq \C^2 \backslash \{0\}$. The $T$-action is given by scaling: $x \cdot \lambda = \lambda x$, and 
\[
\left( \C^2 \backslash \{0\} \right) / \C^\times = \PP^1 \C = G/B.
\]
\end{example}
In general, $G/U$ is {\em quasi-affine}, meaning that it is an open set inside an affine variety. 

\vspace{2mm}
\noindent
{\bf Why?} Denote by $V_\lambda$  a simple highest weight module of $G$ and $v_\lambda \in V_\lambda$ a fixed highest weight vector. We can choose $\{\lambda_1, \ldots  , \lambda_n\}$ such that the stabiliser of 
\[
v:=(v_{\lambda_1}, \ldots, v_{\lambda_n}) \in V:= V_{\lambda_1} \oplus \cdots \oplus V_{\lambda_n}
\]
in $G$ is $U$. Hence 
\begin{align*}
G/U &\hookrightarrow V\\
g &\mapsto gv
\end{align*}
gives a map $G/U \xhookrightarrow{\text{open}}\overline{G/U}\subset V$ and $\overline{G/U}$ is affine. 

\vspace{5mm}
\noindent
{\bf Ring of functions:} Assume $\mathrm{char}k=0$. Then we have the Peter-Weyl theorem:
\[
k[G]=\bigoplus_{\lambda \in X_+} V_\lambda \oplus V_\lambda^* \text{ as $G\times G$-representations}.
\]
Hence,
\[
k[G/U] \simeq k[G]^U\simeq \bigoplus_{\lambda \in X_+} V_\lambda \otimes (V_\lambda^*)^U \simeq \bigoplus_{\lambda \in X_+} V_\lambda.
\]
The first isomorphism follows from the fact that the right action of $U$ on $G$ is free, and the final isomorphism follows from $(V_\lambda^*)^U \simeq \C$. This gives us a description of $k[G/U]$ in terms of representations of $G$, but how do we describe multiplication from this perspective? 

The torus $T$ acts on $(V_\lambda^*)^U$ via the character $\lambda$, so $T$ acts on $V_\lambda\simeq V_\lambda \otimes (V_\lambda^*)^U$ with character $\lambda$. The multiplication map 
\[
m:\left( \bigoplus_{\lambda \in X_+} V_\lambda \right) \otimes \left(\bigoplus_{\lambda \in X_+} V_\lambda \right) \rightarrow \bigoplus_{\lambda \in X_+} V_\lambda 
\]
is $T$-equivariant, so it must map  
\[
V_\lambda \otimes V_\mu \rightarrow V_{\lambda + \mu}.
\]
The space $\Hom_G(V_\lambda \otimes V_\mu, V_{\lambda + \mu})$ is one-dimensional, hence there is a unique map 
\[
m_{\lambda, \mu}:V_\lambda \otimes V_\mu \rightarrow V_{\lambda + \mu}: v_\lambda \otimes v_\mu \mapsto v_{\lambda + \mu}.
\]
(Note that the multiplication map could also be zero on this component, but this can be ruled out.) This gives the multiplication in $k[G/U]$.
\begin{remark}
This is sometimes referred to as {\em Chevalley multiplication} on $k[G/U]$. 
\end{remark}
\begin{example}
If $G=\SL_2$, $k[G/U]=k[x,y]=\bigoplus_{m \in \Z} k[x,y]_m$, and multiplication 
\[
k[x, y]_m \otimes k[x,y]_{m'} \rightarrow k[x,y]_{m + m'} 
\]
is the unique $\SL_2$-equivariant morphism which sends $x^m\otimes x^{m'}$ to $x^{m+m'}$. 
\end{example}

\noindent
{\bf Important Notation:} 
\[
{G/U \atop \text{base affine space}}{\xhookrightarrow{\text{open dense}} \atop \hspace{2mm}} {\hspace{2mm} \overline{G/U} \atop \hspace{2mm}} {= \Spec k[G/U] \atop \text{affine closure} }
\]
We should think that these two spacs are ``almost the same'', like $\A^{n+1}$ and $\A^{n+1} \backslash \{0\}$. 

\begin{exercise}
(Well worth thinking about!) Describe the ideal of the boundary of $G/U$ inside $\overline{G/U}$ in terms of the above.
\end{exercise}

\subsection{Motivational interlude}
The algebra $k[G/U]=\bigoplus_{\lambda \in X_+} V_\lambda$ is a $T$-algebra. The deequivariantisation principle of Lecture \ref{lecture 28} tells us that to make a category $\mc{M}$ linear over $G/B=(G/U)/T$ is the same as giving: 
\begin{enumerate}
    \item a $\Rep{T}$-module $\mc{M}$ $\iff$ an action of the character lattice $X$ of $T$ on a category $\mc{M}$. We will denote this action via $\lambda \cdot M = M(\lambda)$. 
    \item a $T$-equivariant $k[G/U]$-linear structure on $\mc{M}_\mathrm{deeq}$; i.e. for all $M, M' \in \mc{M}$, a $T$-equivariant map 
    \[
    k[G/U] \xrightarrow{\phi_{M, M'}} \Hom(M, \mc{O}_T \otimes M') = \Hom(M, \bigoplus M'(\lambda)). 
    \]
    By the description of $k[G/U]$ above, this is the same data as a collection of maps 
    \[
    V_\lambda \rightarrow  \Hom(M, M'(\lambda)) 
    \]
    for each $\lambda \in X_+$ (the $\lambda$-isotypic components of $\phi_{M, M'}$); or, equivalently (as we will explain momentarily), a collection of ``highest weight arrows''
    \[
    b_\lambda \in \Hom(V_\lambda \otimes M, M'(\lambda))
    \]
    such that the associated map $\phi_{M,M'}$ is $K[G/U]$-linear (``Drinfeld--Pl\"{u}cker relations''). 
    \item with the property that sheaves on the complement 
    \[
    \overline{G/U} \backslash G/U
    \]
    act by zero. 
\end{enumerate}
\begin{remark}
This explains the arrows in Theorem \ref{technical lemma}.2. Note that $\mc{C}$ from Theorem \ref{technical lemma} is $\End{\mc{M}}$ here. 
\end{remark}
\subsection{Arrows between coherent sheaves on $\overline{G/U}$}
Now we will examine the ``highest weight arrows'' in the previous section more carefully. Set 
\[
\mc{O}:=k[G/U] = \bigoplus_{\lambda \in X_+} V_\lambda
\]
with Chevalley multiplication and $X$-grading. Then 
\begin{align*}
    \Hom_{G\times T \text{-eqvt} \atop \text{vect bndls} } (V_\lambda \otimes \mc{O}, \mc{O}(\lambda)) &= \Hom_{G\times T \text{-eqvt}}(V_\lambda, \mc{O}(\lambda)) \\
    &= \Hom_{G\text{-eqvt}} (V_\lambda, V_\lambda)\\
    &=\C 
\end{align*}
is spanned by 
\[
B_\lambda:V_\lambda \otimes \mc{O} \rightarrow \mc{O}(\lambda),
\]
with the property that 
\[
B_\lambda|_{V_\lambda \otimes V_\mu}=m_{\lambda, \mu}: V_\lambda \otimes V_\mu \rightarrow V_{\lambda + \mu}.
\]
\begin{lemma}
These arrows satisfy {\em Pl\"{u}cker relations}:
\[
B_\lambda \otimes B_\mu = B_{\lambda + \mu} \circ (m_{\lambda, \mu} \otimes id_\mc{O}).
\]
That is, the diagram
\[
\begin{tikzcd}
(V_\lambda \otimes \mc{O}) \otimes_\mc{O} (V_\mu \otimes \mc{O}) \arrow[d, equals] \arrow[rr, "B_\lambda \otimes B_\mu"] & & \mc{O}(\lambda) \otimes_\mc{O} \mc{O}(\mu) \arrow[d, "\sim", "\text{mult}"'] \\
(V_\lambda \otimes V_\mu) \otimes \mc{O} \arrow[dr, "m_{\lambda, \mu}"] & & \mc{O}(\lambda + \mu) \\
& V_{\lambda +\mu} \otimes \mc{O} \arrow[ur, "B_{\lambda + \mu}"] & 
\end{tikzcd}
\]
commutes. 
\end{lemma}
\begin{proof}
It is enough to check that the restriction to $V_\lambda \otimes V_\mu$ commutes. This becomes the definition of Chevalley multiplication.
\end{proof}
\begin{remark}
Recall that the classical Pl\"{u}cker relations describe the homogeneous equations defining the Grassmannian of $k$-planes in $\C^n$ in its ``Pl\"{u}cker'' embedding inside $\PP(\Lambda^k\C^n)$. Analogously, the above theory can be used to describe the defining relations of the flag variety in its embedding inside a product of Grassmannians. For more on this (and the connection to moduli of vector bundles and Drinfeld's compactification), see \cite[\S 4]{FGV}. 
\end{remark}

We will end today's lecture by describing the analogue in this setting of the $\mathrm{pt}/G \rightsquigarrow \mf{g}/G$ upgrade that we discussed last lecture. Here, we have an upgrade 
\[
\mathrm{pt}/G \rightsquigarrow G\backslash (\overline{G/U})/T.
\]
\begin{proposition}
Let $A$ be a $G \times T$-algebra. Suppose for all $\lambda \in X_+$, we are given a $G$-equivariant morphism 
\[
b_\lambda:V_\lambda \otimes A \rightarrow A(\lambda)
\]
satisfying the Pl\"{u}cker relations. Then there exists a unique $G \times T$-equivariant homomorphism 
\[
\phi:\mc{O} \rightarrow A
\]
such that $b_\lambda = \phi_*(B_\lambda):=id_A \otimes B_\lambda$.
\end{proposition}
\begin{proof}
What does the last condition mean? It means that
\[
\begin{tikzcd}
V_\lambda \otimes \mc{O} \arrow[d, "B_\lambda"'] \arrow[r] & V_\lambda \otimes A \arrow[d, "b_\lambda"] \\
\mc{O}(\lambda) \arrow[r, "\phi(\lambda)"] &A(\lambda)
\end{tikzcd}
\]
commutes; i.e.
\[
\phi|_{V_\lambda} = b_\lambda|_{V_\lambda \otimes 1}.
\]
This determines $\phi$ uniquely: define $\phi_\lambda := b_\lambda|_{V_\lambda \otimes 1}:V_\lambda \rightarrow A$ and $\phi=\bigoplus \phi_\lambda$. This is clearly $G \times T$-equivariant. One can check that $\phi$ is a homomorphism, which amounts to the Pl\"{u}cker relations holding. 
\end{proof}

%% file: lecture-30.tex
\section{Lecture 30: Constructible sheaves on finite flag varieties and braid groups}
\label{lecture 30}

Today we return to the constructible world. In this lecture and the following one, we will see how the formalism of the last two lectures is used to produce a $\otimes$-functor 
\[
\Coh_\mathrm{free}^{G^\vee}(\widetilde{\mc{N}}) \rightarrow P_I \subset D_I^b(\mc{F}\ell). 
\]
The key tool is provided by BGK central sheaves and certain easier sheaves called {\em Wakimoto sheaves}. Wakimoto sheaves are special to the affine setting; however, to warm up (it's been a few weeks since we've been on this side!), we'll spend today's lecture in the finite case. 

\subsection{A fun calculation}
 Choose points $x \neq y$ on $\PP^1$:
\[
\includegraphics[scale=0.2]{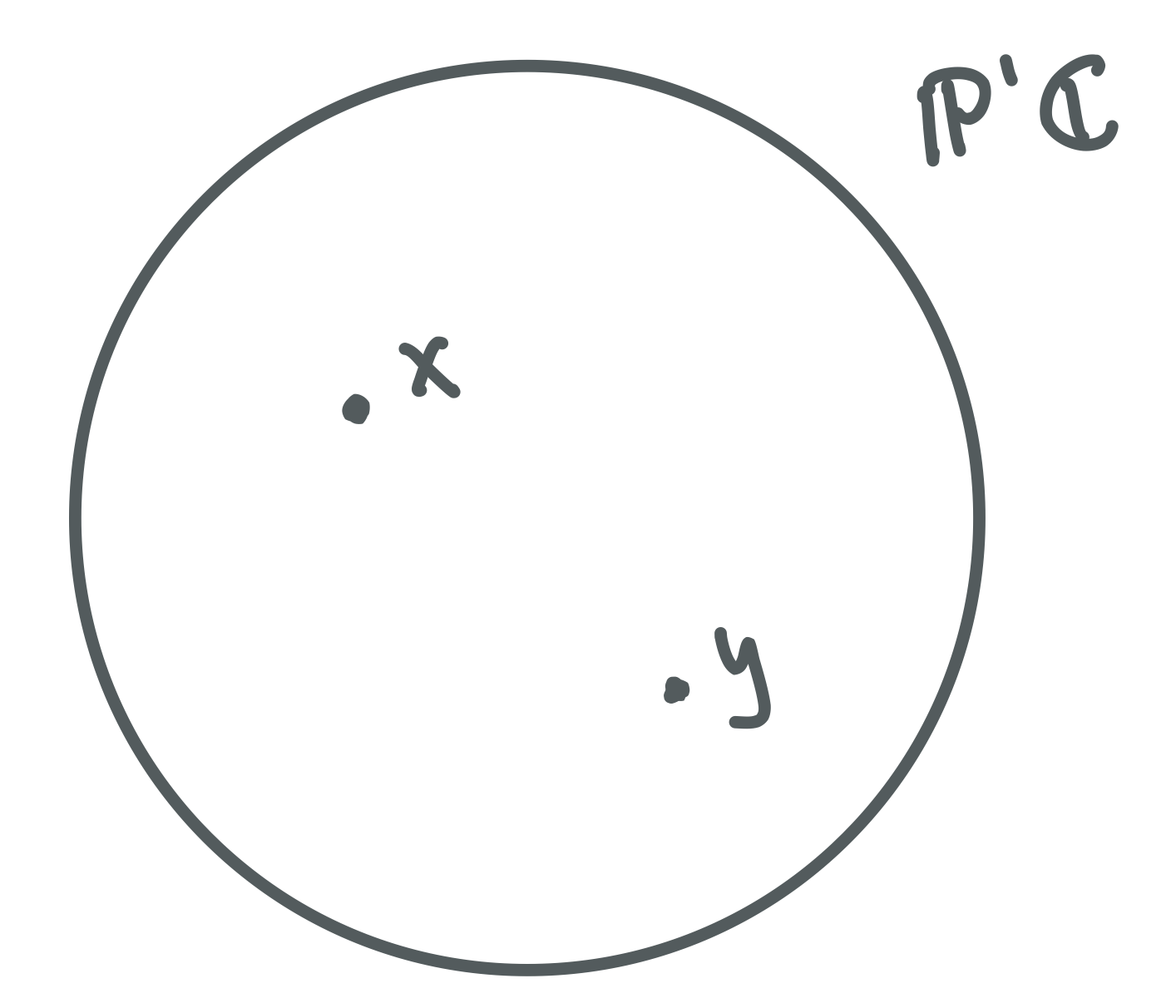}
\]
There are two ways to include these points into $\PP^1$:
\[
\begin{tikzcd}
U:=\PP^1 \backslash \{x, y\} \arrow[r, hookrightarrow, "j_y"] \arrow[d, hookrightarrow, "j_x"] & \PP^1 \backslash \{x \} \arrow[d, hookrightarrow, "j_x"] \\
\PP^1 \backslash \{y\} \arrow[r, hookrightarrow, "j_y"] & \PP^1 
\end{tikzcd}
\]
\begin{exercise}
Show that 
\[
j_{x!}j_{y*}\Q_U \simeq j_{y*} j_{x!} \Q_U =: \Q_{x!, y*}.
\]
({\em Hint}: construct a map and then show that it is an isomorphism.) 
\end{exercise}
A picture: 
\[
\includegraphics[scale=0.2]{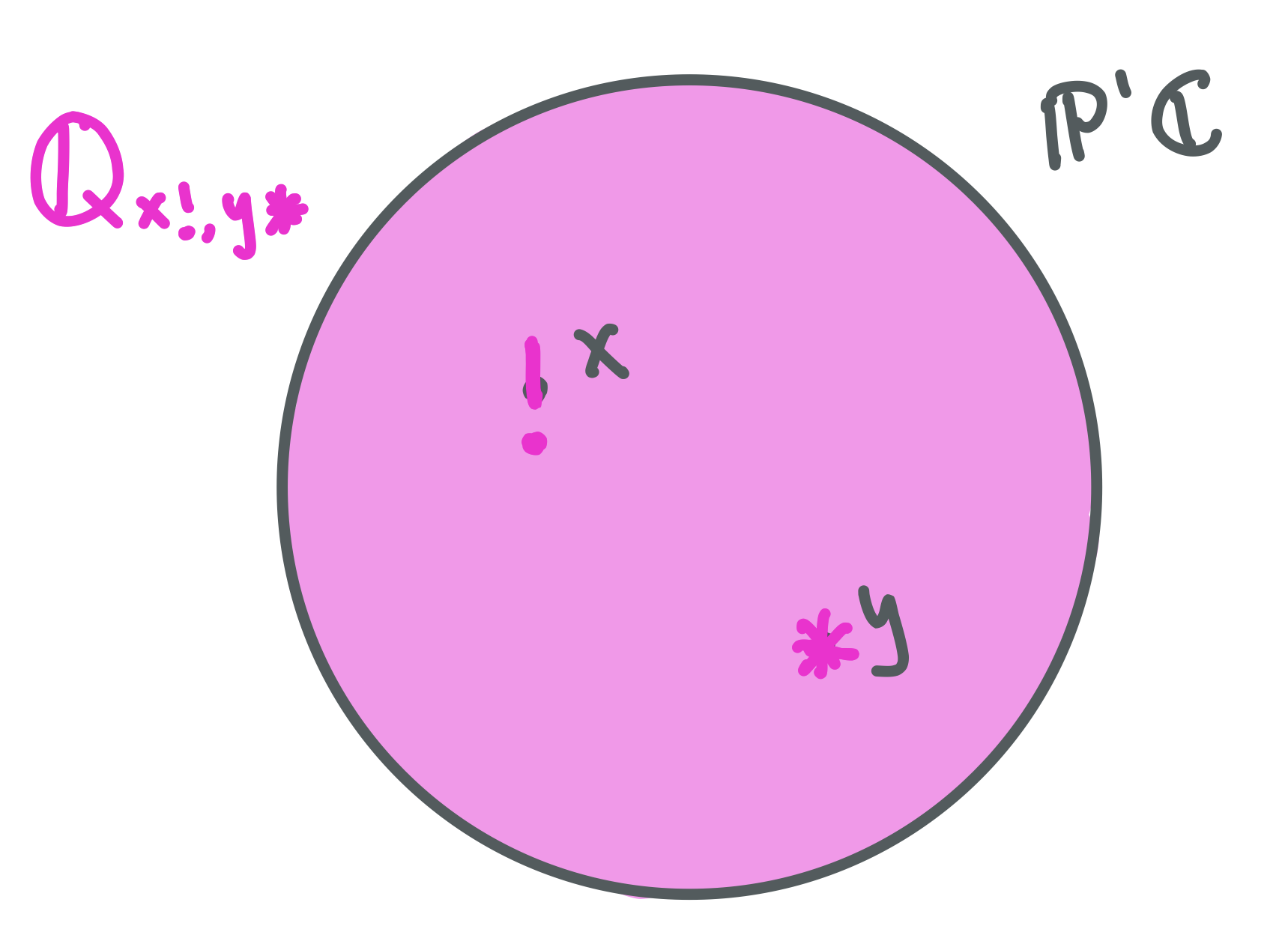}
\]
\begin{claim}
$H^*(\PP^1, \Q_{x!, y*})=0$. 
\end{claim}
\begin{proof}
Let $i_x:\{x\} \hookrightarrow \PP^1$ be the inclusion. By applying the functorial distinguished triangle  
\[
j_{x!}j_x^! \rightarrow id \rightarrow i_{x*}i_x^* \xrightarrow{+1}
\]
to $\Q_{y*}:= j_{y*}\Q_{\PP^1 \backslash \{y\}}$, we obtain a distinguished triangle 
\[
j_{x!}j_x^! \Q_{y*} \rightarrow \Q_{y*} \rightarrow i_{x*} i_x^* \Q_{y*} \xrightarrow{+1}.
\]
By base change, $j_{x!}j_x^! \Q_{y*} \simeq \Q_{x!, y*}$, so the long exact sequence in cohomology
\[
\begin{tikzcd}
 &H^*(\Q_{x!, y*})& H^*(\Q_{y*})=H^*(\C)  & H^*(i_{x*} i_x^* \Q_{y*})  \\
2 & 0 \arrow[r] & 0 \arrow[r] \arrow[d, phantom, ""{coordinate, name=Z}]
& 0  \\
1 & 0 \arrow[r]& 0 \arrow[d, phantom, ""{coordinate, name=Y}] \arrow[r] & 0 \arrow[ull,
rounded corners,
to path={ -- ([xshift=2ex]\tikztostart.east)
|- (Z) [near end]\tikztonodes
-| ([xshift=-2ex]\tikztotarget.west)
-- (\tikztotarget)}] \\
0 & 0 \arrow[r] & \Q \arrow[r, "\sim"] & \Q \arrow[ull,
rounded corners,
to path={ -- ([xshift=2ex]\tikztostart.east)
|- (Y) [near end]\tikztonodes
-| ([xshift=-2ex]\tikztotarget.west)
-- (\tikztotarget)}]
\end{tikzcd}
\]
implies the claim.
\end{proof}
\begin{exercise}
\begin{itemize}
    \item Prove the claim again as follows: use Artin vanishing (which we will state later) and the previous exercise to show that $H^i(\PP^1, \Q_{x!, y*}[1])=0$ for $i \neq 0$, then compute that 
    \[
    \chi(\PP^1, \Q_{x!, y*})=0
    \]
    to deduce the claim.
    \item Find yet another proof of the claim by interpreting $H^*(\PP^1, \Q_{x!, y*})$ as certain locally finite chains. 
\end{itemize}
\end{exercise}

\subsection{The finite Hecke category}

Fix 
\[
G \supset B \supset T
\]
as usual, taken over $\C$. In what follows we will take our coefficients in $\Q$. To simplify the discussion, we will work in 
\[
D^b_{B \times B}(G, \Q) \simeq D^b_{B}(G/B, \Q). 
\]
\begin{remark}
As we learned in Lecture \ref{lecture 24}, this is an approximation to the Hecke category. Everything we do below may be lifted to the Hecke category $\mc{H}$ with a little more effort. However, we choose to focus on $D^b_{B \times B}(G, \Q)$ in today's lecture to illustrate the ideas more clearly. 
\end{remark}

Recall that $D^b_B(G/B)$ comes with a convolution structure
\[
*:D^b_B(G/B, \Q) \times D^b_B(G/B, \Q) \rightarrow D^b_B(G/B, \Q). 
\]
The $B$-orbits on $G/B$ stratify $G/B$ by the Bruhat decomposition
\[
G/B=\bigsqcup_{x \in W_f} BxB/B. 
\]
Hence we have two related (but not equal!) derived categories, each containing a corresponding category of perverse sheaves: the construcible (with respect to the Bruhat stratification) derived category 
\[
D^b_{(B)}(G/B, \Q) \supset P_{(B)},
\]
and the $B$-equivariant derived category 
\[
D^b_B(G/B, \Q) \supset P_B.
\]
Convolution gives a right action 
\[
 D^b_{(B)}(G/B, \Q) \circlearrowleft D_B^b(G/B, \Q). 
\]
\begin{notation}
For $\mc{F} \in D_B(G/B)$, denote by $\dot{\mc{F}}:=\Q_{B/B} * \mc{F} = \mathrm{For}(\mc{F}) \in D_{(B)}^b(G/B)$ the object we get by ``forgetting equivariance''. This gives a map 
\[
P_B \rightarrow P_{(B)}.
\]
\begin{remark}
The map $\mc{F} \mapsto \dot{\mc{F}}$ is fully faithful, but the image is not closed under extensions.
\end{remark}
\end{notation}
Let $j_x:BxB/B \hookrightarrow G/B$. Define  
\[
\Delta_x:=j_{x!}\Q_{BxB/B}[\ell(x)], \hspace{3mm} \nabla_x:=j_{x*}\Q_{BxB/B}[\ell(x)] 
\]
in $P_{B}$. 
\begin{remark}
\label{not highest weight}
Recall that in Lecture \ref{lecture 20}, we showed that $\{\dot{\Delta}_x\}$ and $\{\dot{\nabla}_x\}$ are standard and costandard objects in a highest weight structure on $P_{(B)}$. In general, the category $P_B$ does not admit a highest weight structure. 
\end{remark}

\begin{example}
\label{equivariant category not highest weight}
Let $G=\SL_2$. As we discussed in Lecture \ref{lecture 16}, we can use Beilinson's description 
\[
\Perv_{(B)}(\mathbb{P}^1) \simeq \Rep \left( \left. \begin{tikzcd} \bullet \arrow[r, "e"', shift right] & \bullet \arrow[l, "f"', shift right] \end{tikzcd} \right| fe=0 \right)
\]
of the category $P_{(B)}$ to see that it has five indecomposible objects:
\begin{enumerate}
    \item $\dot{\IC}_{id} \longleftrightarrow \left(\begin{tikzcd} 0 \arrow[r, shift right] & k \arrow[l, shift right] \end{tikzcd}\right)$,
    \item $\dot{\IC}_s\longleftrightarrow \left(\begin{tikzcd} k \arrow[r, shift right] & 0 \arrow[l, shift right] \end{tikzcd}\right)$,
    \item $\dot{\Delta}_s\longleftrightarrow \left(\begin{tikzcd}  k \arrow[r, "\sim"', shift right] & k \arrow[l, "0"', shift right] \end{tikzcd}\right)$,
    \item $\dot{\nabla}_s\longleftrightarrow \left(\begin{tikzcd} k \arrow[r, "0"', shift right] & k \arrow[l, "\sim"', shift right] \end{tikzcd}\right)$,
    \item  $T_s\longleftrightarrow \left(\begin{tikzcd} k \arrow[r, shift right] & k\oplus k \arrow[l, shift right] \end{tikzcd}\right)$, where the arrow $\rightarrow$ embeds $k$ into the first factor of $k \oplus k$ and the arrow $\leftarrow$ projects onto the second factor.
\end{enumerate}    
The object $T_s$ is called the ``big projective'' or ``big tilting sheaf''. 

We can also use a modified version Beilinson's description to describe the indecomposible objects in the category $P_B$. The $B$-equivariance implies that the micolocal monodromy is zero, so we have an extra condition on our quiver representations:
\[
\Perv_B(\mathbb{P}^1) \simeq \Rep \left( \left. \begin{tikzcd} \bullet \arrow[r, "e"', shift right] & \bullet \arrow[l, "f"', shift right] \end{tikzcd} \right| fe=ef=0 \right). 
\]
Hence there are four indecomposible objects in $P_B$: $\IC_{id}$, $\IC_s$, $\Delta_s$, and $\nabla_s$, which correspond to the quiver representations $1-4$ above. The quiver representation $5$ corresponding to the big projective does not satisfy the condition that $ef=0$, so we have no analogue to $T_s$ in $P_B$. With this description, it is easy to see that $P_B$ is not a highest weight category. For instance, minimal projective resolutions of $\IC_{id}$ and $\IC_s$ are
\begin{align*}
    \cdots \rightarrow \Delta_s \rightarrow \nabla_s \rightarrow \Delta_s \rightarrow \nabla_s \rightarrow \Delta_s &\twoheadrightarrow \IC_s, \\
    \cdots \rightarrow \nabla_s \rightarrow \Delta_s \rightarrow \nabla_s \rightarrow \Delta_s \rightarrow \nabla_s &\twoheadrightarrow \IC_{id},
\end{align*}
so $P_B$ does not have finite homological dimension, and hence cannot admit a highest weight structure. 

\end{example}

\subsection{The Braid group} 

Recall that $W_f$ has a presentation \[
W_f = \langle s \in S_f \mid s^2=id, \underbrace{st\ldots}_{m_{st}} = \underbrace{ts \ldots}_{m_{st}} \rangle. 
\]
The {\em braid group} associated to $W_f$ is 
\[
B_{W_f} := \langle \sigma_s, s \in S_f \mid \underbrace{\sigma_s \sigma_t \ldots}_{m_{st}} = \underbrace{\sigma_t \sigma_s \ldots}_{m_{st}} \rangle.
\]
Let $\underline{B}_{W_f}$ be the corresponding monoidal category\footnote{See Lecture \ref{lecture 26}, where this was denoted $\mc{A}_{B_{W_f}}$.}. We have the following important and beautiful theorem. 
\begin{theorem}
\label{beautiful theorem}
The assignment $\sigma_s \mapsto \Delta_s$ extends to a monoidal functor 
\[
\underline{B}_{W_f} \rightarrow D^b_B(G/B, \Q). 
\]
\end{theorem}
\begin{corollary}
\label{braid group action}
There is a (strict) action of $\underline{B}_{W_f}$ on $D^b_{(B)}(G/B)$. 
\end{corollary}
\begin{proof} (of Theorem \ref{beautiful theorem})
We will use the following ``generators and relations'' theorem of Deligne \cite{Deligne}.

\begin{theorem}
\label{Deligne}
Let $\mc{A}$ be a monoidal category. Giving a monoidal functor \[
F:\underline{B}_{W_f} \rightarrow \mc{A}
\]
is the same as giving 
\begin{itemize}
    \item a collection $\{F_x \in \mc{A}\}_{x \in W_f}$ of invertible objects, and
    \item a collection $\{F_x F_y \xrightarrow{\sim} F_{xy}\}_{x, y \in W_f \text{ with }\ell(xy) = \ell(x)+\ell(y)}$ of isomorphisms,
\end{itemize}
such that 
\begin{equation}
    \label{Deligne condition}
\begin{tikzcd}
F_x F_y F_z \arrow[d] \arrow[r] & F_{xy}F_z \arrow[d]\\
F_x F_{yz} \arrow[r] & F_{xyz}
\end{tikzcd}
\end{equation}
commutes whenever $\ell(xyz) = \ell(x) + \ell(y) + \ell(z)$. 
\end{theorem}
We will show that there exists a monoidal functor
\[
F:\underline{B}_{W_f} \rightarrow (D^b_{(B)}(G/B), *)
\]
sending $\sigma_s \mapsto \Delta_s$ by providing the data of Theorem \ref{Deligne}. Most of this data can be obtained from the following lemma:
\begin{lemma}
\label{deltas convolve}
If $\ell(xy) = \ell(x) + \ell(y)$, then $\Delta_x * \Delta_y \simeq \Delta_{xy}$ canonically. 
\end{lemma}

\noindent
{\em Proof.}
Recall that $\Delta_x := j_{x!}(\Q_{BxB/B}[\ell(x)])$. Unpacking definitions, we have 
\begin{align*}
\Delta_x * \Delta_y &= m_*\left(j_!\Q_{BxB\times_B ByB/B}[\ell(x) + \ell(y)]\right). 
\end{align*}
If $\ell(x)+\ell(y) = \ell(xy)$, then we have a commutative diagram:
\[
\begin{tikzcd}
BxB \times_B ByB/B \arrow[r, "m'"] \arrow[d, "j"] & BxyB/B \arrow[d, "j_{xy}"] \\
G \times_B G/B \arrow[r, "m"] & G/B
\end{tikzcd}.
\]
Hence, 
\[
\Delta_x * \Delta_y = j_{xy!}\left(\Q_{BxyB/B}[\ell(xy)]\right) = \Delta_{xy}. \qed
\]
\vspace{2mm}

 Lemma \ref{deltas convolve} provides the canonical isomorphisms $F_x F_y \xrightarrow{\sim} F_{xy}$ of Deligne's theorem, and it is easy to check that (\ref{Deligne condition}) commutes if $\ell(xyz) = \ell(x) + \ell(y) + \ell(z)$. Hence to complete the proof, it remains to show that $\Delta_x$ are invertible. Because  
\[
\Delta_x = \Delta_{s_1} * \cdots * \Delta_{s_m}
\]
for $x = s_1 \cdots s_m$, it is enough to show that $\Delta_s$ is invertible for all $s \in S_f$. 

\begin{lemma} For $s \in S_f$,
\[
\Delta_s * \nabla_s \simeq \nabla_s * \Delta_s \simeq \Delta_{id} = \nabla_{id}.
\]
\end{lemma}
\begin{proof}
(Sketch \#1, a ``hygenic'' proof) First, one can easily reduce to $\SL_2$. Here we have a distinguished triangle of perverse sheaves
\[
\Q_{\PP^1}[1] = \IC_s  \rightarrow \nabla_s \rightarrow \IC_{id}=\Q_{id} \xrightarrow{+1}.
\]
Now, using the fact that convolution with $\Delta_s$ preserves distinguished triangles, we have another distinguished triangle
\[
\Delta_s * IC_s= \Q_{\PP^1} \rightarrow \Delta_s * \nabla_s \rightarrow \Delta_s \xrightarrow{+1}.
\]
From this, we conclude that $\Delta_s * \nabla_s = \ker(\Delta_s \rightarrow \IC_s) = \Q_{id}$, using that the connecting map is actully a map between perverse sheaves to make sense of the kernel.
\end{proof}

\begin{exercise}
Fill in the details of this argument.
\end{exercise}

\begin{proof}
(Sketch \# 2, a more intuitive proof) We have an isomorphism
\begin{align*}
    \SL_2 \times_B \SL_2/B &\xrightarrow{\sim} \PP^1 \times \PP^1 \\
    (g, g'B) &\mapsto (gB, gg'B) 
\end{align*}
Under this isomorphism, $B \times_B \SL_2/B$ corresponds to ${\color{blue}Z}:=x$-axis, and $\SL_2 \times_B B/B$ corresponds to ${\color{magenta} \Delta}:=$ diagonal in the following picture:
\[
\includegraphics[scale=0.3]{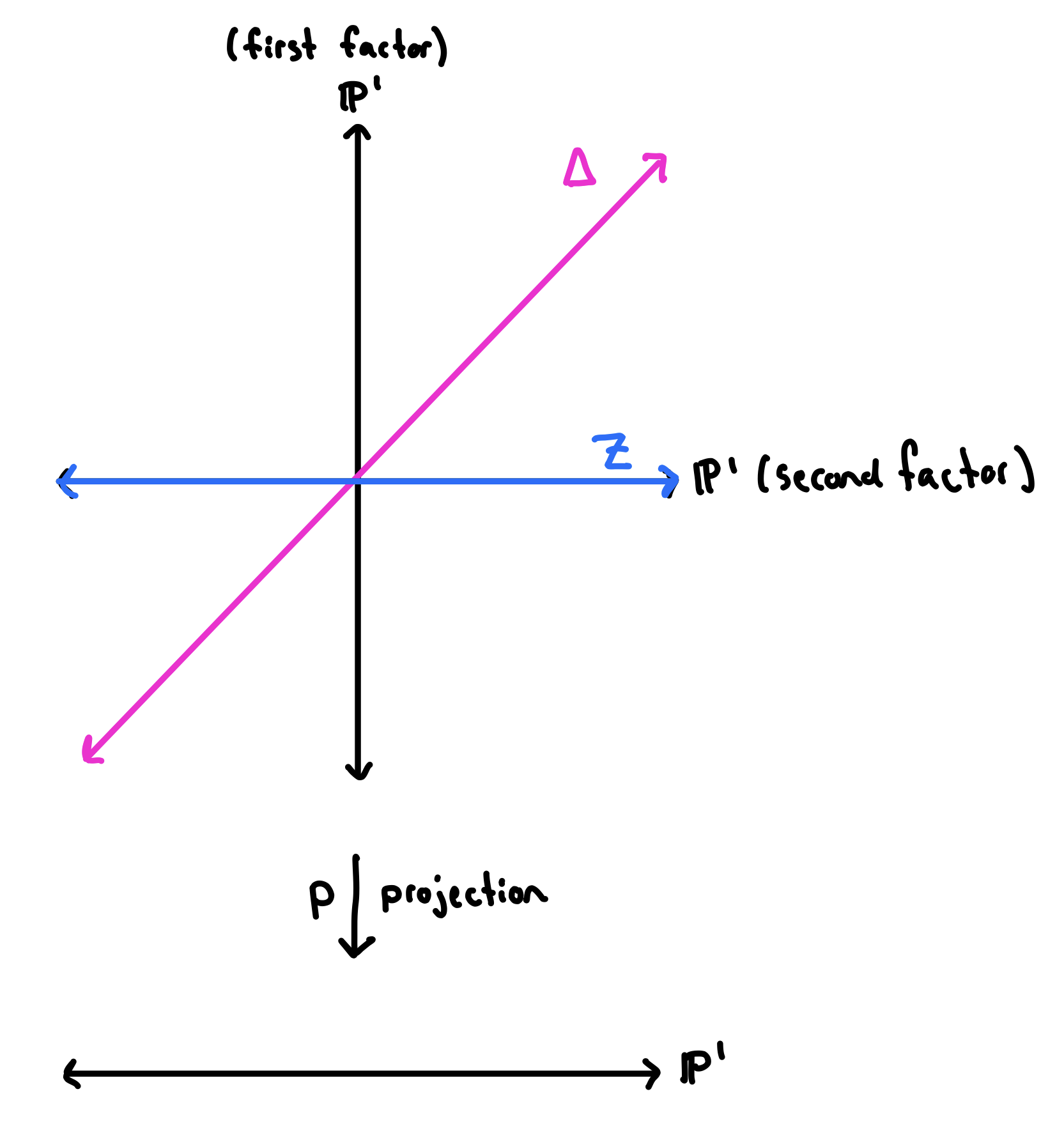}
\]
After forgetting equivariance, 
\begin{equation}
    \label{enough}
\Delta_s * \nabla_s = p_{*}\left((j_{\Delta*}\Q_{\PP^1 \times \PP^1 \backslash \Delta}) \otimes (j_{Z!}\Q_{\PP^1 \otimes \PP^1 \backslash Z} [2])\right), 
\end{equation}
where $j_\Delta: \PP^1 \times \PP^1 \backslash \Delta \hookrightarrow \PP^1 \times \PP^1$ and $j_Z: \PP^1 \times \PP^1 \backslash Z \hookrightarrow \PP^1 \times \PP^1$. By proper base change, we can compute stalks:
\[
(\Delta_s * \nabla_s)_x = \begin{cases} H^*(\PP^1, \Q_{0!, x*}) & \text{if }x \neq 0; \\
H^{*+2}(\PP^1, j_{0!} \Q) & \text{if }x=0. \end{cases}
\]
In the second case, we have $H^{*+2}(\PP^1, j_{0!}\Q)=\Q$, and by the fun computation from earlier, we have $H^*(\PP^1, \Q_{0!, x*})=0$. Hence, 
\[
\Delta_s * \nabla_s \simeq \Q_{B/B}
\]
which completes the proof.
\end{proof}

\begin{exercise}
Fill in the details of this argument. 
\end{exercise}
This completes the proof of Theorem \ref{beautiful theorem}.
\end{proof}

\subsection{Further properties of the braid group action}
We will spend the rest of the lecture exploring the properties of the elements associated to braid group elements. All of this is in preparation for the affine case next week. 

The $\{\Delta_y\}$ and $\{\nabla_y\}$ form ``exceptional collections'':
\begin{align*}
    \Hom^\bullet(\Delta_x, \Delta_y) &= 0 \text{ unless } x \leq y; \\
    \Hom^\bullet(\nabla_x, \nabla_y) &= 0 \text{ unless } x \geq y; \\ 
    \Hom^\bullet(\Delta_x, \nabla_y) &= 0 \text{ if } x \neq y. 
\end{align*}
Indeed, ignoring shifts,
\[
\Hom^\bullet(\nabla_x, \nabla_y) = \Hom^\bullet(j_{x*}\Q_{X_x}, j_{y*} \Q_{X_y} ) = \Hom^\bullet(j_y^*j_{x*} \Q_{X_x}, \Q_{X_y}) = 0 \text{ unless } y \leq x,
\]
and similarly for $\Hom^\bullet(\Delta_x, \Delta_y)$. (Here $X_x = BxB/B$ is the Bruhat cell associated to $x \in W_f$.) 
\begin{remark}
In Lecture \ref{lecture 20}, we showed that 
\[
\Hom^i(\dot{\Delta}_x, \dot{\nabla}_y) = \begin{cases} \Q & \text{if } i=0 \text{ and }x=y; \\
0 & \text{otherwise}. \end{cases}
\]
This was how we established the necessary $\Ext^2$ vanishing property of the highest weight structure on $P_{(B)}$. In contrast, 
\[
\Hom^\bullet(\Delta_x, \nabla_y) = \begin{cases} H^\bullet_B(\mathrm{pt}) & \text{if } x=y; \\ 
0 & \text{if } x \neq y. \end{cases}
\]
Note that in this case, it is not true that the derived category of $P_B$ agrees with the equivariant derived category, so we cannot use this computation to conclude that $\Ext^2$ vanishing fails in $P_B$. However, it is an indication that $P_B$ might not be highest weight, and indeed it is not in general. (See Remark \ref{not highest weight}.) 
\end{remark}

\begin{remark}
In $\underline{B}_{W_f}$, Hom spaces are very easy. By contrast, morphisms between $\Delta_x$'s are complicated and quite mysterious. For example, $\Ext^\bullet(\text{Vermas})$ is an unsolved problem. 
\end{remark}

Recall that convolution gives a right action 
\[
D^b_{(B)}(G/B) \circlearrowleft D^b_B(G/B) \supset P_B.
\]
Within $D^b_{(B)}$ we have the highest weight category $P_{(B)} \subset D^b_{(B)}(G/B)$. The convolution of two perverse sheaves is not generally perverse, so the subcategory $P_{(B)}$ is not preserved under the action of $P_B$.  

\begin{definition}
A perverse sheaf $\mc{F} \in P_{(B)}$ is {\em convolution exact} if $\mc{F} * \mc{G} \in P_{(B)}$ for any $\mc{G} \in P_B$. 
\end{definition}

\noindent
{\bf Mirkovic's observation:} A perverse sheaf $\mc{F} \in P_{(B)}$ is convolution exact if and only if $\mc{F}$ is tilting. 
\begin{proof}
($\implies$) By highest weight formalism, 
\[
\mc{F} \text{ is tilting } \iff \Hom^{>0}(\mc{F}, \dot{\nabla}_y) \overset{(*)}{=} 0 \overset{(!)}{=}\Hom^{>0}(\dot{\Delta}_x, \mc{F}).
\]
We will check that equality $(!)$ holds for $\mc{F}$ convolution exact, $(*)$ is similar. 

Because $\mc{F}$ is convolution exact, $\mc{F} * \Delta_{w_0}$ is perverse. This implies that 
\[
\mc{F} * \Delta_{w_0} \in {^pD}^{\leq 0} = \langle \dot{\Delta}_y[d] \mid y \in W_f, d \in \Z_{\geq 0} \rangle_{ext},
\]
where the subscript ``ext'' denotes the closure under extensions. Hence 
\[
\mc{F} = \mc{F} * \Delta_{w_0} * \nabla_{w_0} \in \langle \dot{\Delta}_y * \nabla_{w_0}[d]=\dot{\nabla}_{yw_0} \mid y \in W_f, d \in \Z_{\geq 0} \rangle_{ext}.
\]
Recall that $w_0=y^{-1}\cdot yw_0$ and $\ell(w_0)=\ell(y^{-1}) + \ell(yw_0)$, so $\nabla_{w_0}=\nabla_{y^{-1}}*\nabla_{yw_0}$. Now any object in $\langle \dot{\Delta}_y * \nabla_{w_0}[d]=\dot{\nabla}_{yw_0} \mid y \in W_f, d \in \Z_{\geq 0} \rangle_{ext}$ satisfies (!), so we are done.

For the ($\impliedby$) direction, we need a very important property of affine morphisms. 

\begin{theorem}
\label{affine morphisms}
Let $f:X \rightarrow Y$ be affine. Then $f_*$ preserves ${^pD}^{\geq 0}$, and $f_!$ preserves ${^pD}^{\geq 0}$. 
\end{theorem}
\begin{remark}
Recall Artin vanishing: if $\mc{F}/X$ is perverse and $X$ is affine,
\[
H^{>0}(f_*\mc{F})=H^{>0}(X, \mc{F}) = 0 = H^{<0}_c(X, \mc{F}) = H^{<0}(f_! \mc{F}).
\]
This is exactly Theorem \ref{affine morphisms} when $Y=\mathrm{pt}$.
\end{remark}
\begin{exercise}
Can you prove the theorem for $\mc{D}$-modules?
\end{exercise}
Now back to our setting. 
\begin{proposition}
The maps 
\[
BxB \times_B G/B \xrightarrow{\lambda_x} G/B
\text{ and }
G \times_B BsB/B \xrightarrow{\rho_x} G/B
\]
are affine. 
\end{proposition}
\begin{proof}
(Idea) We can write $BxB \simeq U_x \cdot \{x\} \cdot B$, hence we have a commutative diagram
\[
\begin{tikzcd}
BxB \times_B G \arrow[r, "\text{mult}"] \arrow[d, dash, "\sim"] & G \arrow[d, dash, "\sim"] \\
U_x \times G \arrow[r, "\text{mult}"] & G
\end{tikzcd}
\]
The space $U_x \times G$ is affine, so $\mathrm{mult}:BxB \times_B G \rightarrow G$ is affine, hence $\lambda_x$ is affine (being affine is local in the flat topology). 
\end{proof}
\begin{exercise} Show that 
\begin{align*}
    \Delta_x * ( - ) &= \lambda_{x!} \lambda_x^*( - ) [\ell(x)], \\ 
    \nabla_x * ( - ) &= \lambda_{x*} \lambda_x^* (-) [\ell(x)], \\ 
    (-) * \Delta_x &= \rho_{x!} \rho_x^* (-) [\ell(x)], \\
    (-) * \nabla_x &= \rho_{x*} \rho_x^* ( - ) [\ell(x)].
\end{align*}
(Note: $\lambda_x^! \simeq \lambda_x^*[2 \ell(x)]$ because $\lambda_x$ is smooth of relative dimension $\ell(x)$.)
\end{exercise}
\begin{corollary}
\label{corollary 1}
$\Delta_x * (-)$, $( - ) * \Delta_x$ preserve ${^p D}^{\geq 0}$, and $\nabla_x * (-)$, $(-) * \nabla_x$ preserve $^p D^{\leq 0}$, whenever this makes sense.
\end{corollary}
\begin{corollary}
\label{corollary 2}
$\Delta_x * \nabla_y$ and $\nabla_y * \Delta_x$ are perverse for all $x, y \in W_f$. 
\end{corollary}
With this we can complete the proof of Mirkovic's observation. Recall that 
\[
{^pD}^{\leq 0}_B = \langle \Delta_x [d] \mid x \in W_f, d \in \Z_{\geq 0} \rangle. 
\]
By Corollary \ref{corollary 2}, 
\[
\langle \dot{\nabla}_y \mid y \in W \rangle * {^p D}^{\leq 0}_B \subset {^p D}^{\leq 0}.
\]
Hence 
\[
\text{tilting}* {^pD}_B^{\leq 0} \subset {^pD}^{\leq 0}. 
\]
Similary, ${^pD}^{\geq 0} = \langle \nabla_x[-d] \mid x \in W_f, d \in \Z_{\geq 0} \rangle$, so 
\[
\text{tilting}*{^pD}^{\geq 0}_B \subset {^pD}^{\geq 0}.
\]
This lets us conclude that 
\[
\text{tilting} * P_B \subset P_{(B)}. \hfill 
\]
This completes the proof of Mirkovic's observation. 
\end{proof}

\begin{exercise}
\begin{enumerate}
    \item Use these ideas to show that if $\mc{F} \in P_B$ is convolution exact, then $\mc{F} = \Delta_{id}$, (i.e. there are no interesting convolution exact sheaves in $P_B$). 
    
    \begin{remark}
    We will see next week that there are many convolution exact $\mc{F} \in P_I \subset D_I(\mc{F}\ell)$ in the affine case.
    \end{remark}
    \item Let $T_x$ be indecomposible tilting. Show that 
    \[
    T_x * \Q_{P_s/B}[1] = 0 
    \]
    unless $x=id$. (This is related to the fact that $H^*(G/B, T_x)=0$ unless $x = id$.) Use this to give another proof of the ($\impliedby$) direction of Mirkovic's observation. 
\end{enumerate}

\end{exercise}

%% file: lecture-31.tex
\section{Lecture 31: Affine flags, affine braids, and Wakimoto sheaves}
\label{lecture 31}

Last week we established the existence of a monoidal functor
\begin{align*}
\underline{B}_{W_f} &\rightarrow (D^b_B(G/B), *) \\ \sigma_x &\mapsto \Delta_x
\end{align*}
and described several of its nice properties. Today we'll describe the situation in the affine case. 

As in previous lectures, let $\mc{F}\ell = G((t))/I$, where $I \subset G((t))$ is an Iwahori subgroup. Denote by $W=X^\vee \rtimes W_f$ the affine Weyl group. For simplicity, assume that $G$ is simply connected, so $W$ is a Coxeter group\footnote{If we drop this assumption, $W$ is a quasi-Coxeter group. This adds a few complications to proofs, but doesn't change the story in any major way.}. 

The same construction as we described last week in the finite case yields a monoidal functor 
\begin{align*}
\underline{B}_W &\rightarrow (D^b_B(G/B, \Q), *) \\ 
\sigma_x &\mapsto \Delta_x.
\end{align*}

\subsection{Wakimoto sheaves}
In Lecture \ref{lecture 12}, we described Bernstein's presentation of the extended affine Hecke algebra. We briefly recall this construction now. In Lecture \ref{lecture 11}, we proved the Iwahori-Matsumoto Lemma (Lemma \ref{IM}): for $w \in W_f, \lambda \in X^\vee$, 
\[
\ell(t_\lambda w_f) = \sum_{\alpha \in R_+ \atop w_f(\alpha)>0} | \langle \lambda, \alpha \rangle | + \sum_{\alpha \in R_+ \atop w_f(\alpha)<0} | \langle \lambda, \alpha \rangle - 1 |.
\]
A consequence of this is that if $\lambda', \lambda'' \in X^\vee$ are dominant, then $\ell(t_{\lambda'} t_{\lambda''}) = \ell(t_{\lambda'}) + \ell(t_{\lambda ''})$. (This is also intuitively clear because $t_{\lambda'}$ and $t_{\lambda''}$ translate in the same direction.) Hence in the Hecke algebra $H$\footnote{Here we are aligning notation with Lecture \ref{lecture 23}, which differs slightly from Lecture \ref{lecture 12}.},
\[
\delta_{t_{\lambda'}} \delta_{t_{\lambda''}} = \delta_{t_{\lambda'+ \lambda''}} = \delta_{t_{\lambda''}} \delta_{t_\lambda'}.
\]
We can write any $\lambda \in X^\vee$ as $\lambda = \lambda' - \lambda''$ with $\lambda', \lambda''\in X^\vee$ dominant. Hence we have have a well-defined map 
\[
\Z[v^{\pm 1}][X^\vee] \hookrightarrow H: \lambda \mapsto \theta_\lambda := \delta_{t_{\lambda'}} \delta_{t_{\lambda''}}^{-1}. 
\]
Bernstein used this map to describe the center of $H$.

Wakimoto sheaves categorify this construction; i.e. given $\lambda \in X^\vee$, write $\lambda = \lambda' - \lambda''$ for $\lambda', \lambda'' \in X^\vee_+$ and define
\[
J_\lambda:= \Delta_{t_{\lambda'}} * \nabla_{t_{-\lambda''}} = \Delta_{t_{\lambda'}}*(\Delta_{t_{\lambda''}})^{-1}.
\]
The $J_\lambda$ are called {\em Wakimoto sheaves}. 

\begin{example}
If $\lambda$ is dominant, then $J_\lambda = \Delta_{t_\lambda}$. If $\lambda$ is anti-dominant, then $J_\lambda = \nabla_{t_{\lambda}}$. 
\end{example}

\begin{lemma}
The map $\lambda \mapsto J_\lambda$ extends to a monoidal functor $\underline{X}^\vee \mapsto (D^b_I(\mc{F}\ell), *),$ and hence to a monoidal functor
\[
(\Rep T^\vee, \otimes) \mapsto (D^b_I(\mc{F}\ell), *).
\]
\end{lemma}

We saw last week that $\Delta_x * \nabla_y$ is perverse for all $x, y \in W_f$. An easy adaptation of this to the affine case implies the following corollary. 
\begin{corollary}
For all $\lambda \in X^\vee$, $J_\lambda$ is perverse; i.e. $J_\lambda \in P_I \subset D^b_I(\mc{F}\ell)$. 
\end{corollary}

\begin{notation}
Let 
\begin{align*}
\mc{A} &= \text{ Wakimoto-filtered objects in $P_I$}, \\
\mathrm{gr}\mc{A} &= \text{ image of }\Rep T^\vee \text{ in } D^b_I(\mc{F}\ell).
\end{align*}
Note that $\mathrm{gr}\mc{A}$ is not a full subcategory in $D^b_I(\mc{F}\ell)$. This notation will make sense in a moment. 
\end{notation}

\begin{remark}
It is immediate from the definitions that $\mc{A}$ and $\mathrm{gr}\mc{A}$ are monoidal subcategories in $P_I$. Note that this is in contrast to the finite case, where $P_B$ has no interesting monoidal subcategories. 
\end{remark}

Recall that in the finite case,
\begin{align*}
    \Hom^\bullet (\Delta_x, \Delta_y) &= 0 \text{ unless } x \leq y, \\
    \Hom^\bullet(\nabla_x, \nabla_y) &=0 \text{ unless } y \leq x. 
\end{align*}
On $X^\vee$, we have the {\em periodic order:}
\[
\lambda \leq \mu \text{ if and only if } \mu - \lambda \in \Z X^\vee_+.
\]
\begin{lemma}
\label{homs} With respect to the periodic order,
\begin{align*}
    \Hom^\bullet(J_\lambda, J_\mu) &=0 \text{ unless } \lambda \leq \mu, \\
    \Hom^\bullet(J_\lambda, J_\lambda) &= \Q.
\end{align*}
\end{lemma}
\begin{proof} For $\gamma$ sufficiently dominant, 
\begin{align*}
    \Hom^\bullet(J_\lambda, J_\mu) &= \Hom^\bullet( J_\lambda * J_\gamma, J_\mu * J_\gamma) \\
    &= \Hom^\bullet (\Delta_{t_{\lambda + \gamma}}, \Delta_{t_{\mu + \gamma}} ) \\
    &= 0
\end{align*}
unless $t_{\lambda + \gamma} \leq t_{\mu + \gamma}$. It is a known fact about the Bruhat order in affine type that for sufficiently dominant weights we have 
\[
t_{\lambda + \gamma} \leq t_{\mu + \gamma}\iff \lambda+ \gamma \leq \mu+ \gamma. 
\]
Hence $\Hom^\bullet(J_\lambda, J_\mu)= 0$ unless $\lambda \leq \mu$.
\end{proof}
\begin{example}
Let $G=\SL_2$, so $X^\vee = \Z \alpha^\vee$. The Homs between Wakimoto sheaves only go in one direction:
\[
\includegraphics[scale=0.25]{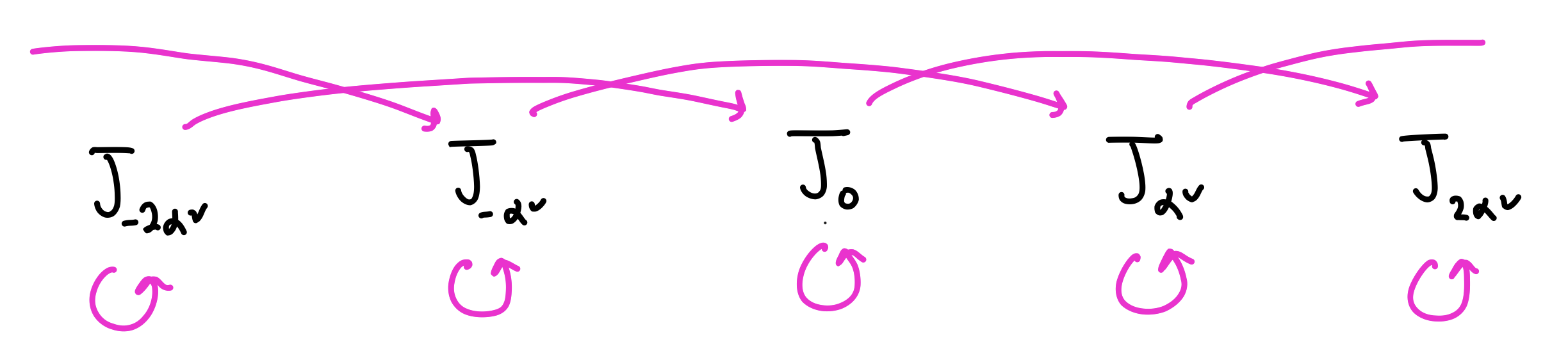}
\]
\end{example}
Lemma \ref{homs} implies:
\begin{exercise}
Every $\mc{F} \in \mc{A}$ admits a functorial\footnote{in particular unique!} filtration (called a ``Wakimoto filtration'') $\mc{F}_I$ indexed by upper closet subsets $I \subset X^\vee$ such that if $\lambda \in I$ is minimal, then 
\[
\mc{F}_I/\mc{F}_{I \backslash \{\lambda\}} \simeq \bigoplus mJ_\lambda^{\oplus m_\lambda}.
\]
\end{exercise}

\begin{remark}
This exercise is an analogue of a standard (but nonetheless beautiful) feature of highest weight categories: (co)standard filtrations are unique when they exist. 
\end{remark}

The exercise gives us a functor
\begin{align*}
    \mathrm{gr}:\mc{A} &\mapsto \mathrm{gr}\mc{A} \\
    \mc{F} &\mapsto \bigoplus \mc{F}_{\{ \geq \lambda\}}/\mc{F}_{\{>\lambda\}}, 
\end{align*}
which sends an object to its ``associated graded under the Wakimoto filtration''. 
\begin{exercise}
The Wakimoto filtration is compatible with $*$, and hence $\mathrm{gr}$ is a $\otimes$-functor. 
\end{exercise}
\subsection{Motivational interlude}
We'll take a brief motivational interlude to explain how Wakimoto sheaves fit into our bigger picture. 

Move back to the coherent side, and let $G^\vee \supset B^\vee$, $R_+^\vee$ be as usual. Any $G^\vee$-module $V$ has a $B^\vee$-filtration indexed by upper closed subsets $I \subset X^\vee$, namely
\[
V_I = \bigoplus_{\lambda \in I} V_\lambda \subset V,
\]
and the associated graded of this filtration is isomorphic to a direct sum of one-dimensional modules $k_\lambda$. 

\begin{example}
If $G=\SL_2$, a picture of this filtration is:
\[
\includegraphics[scale=0.3]{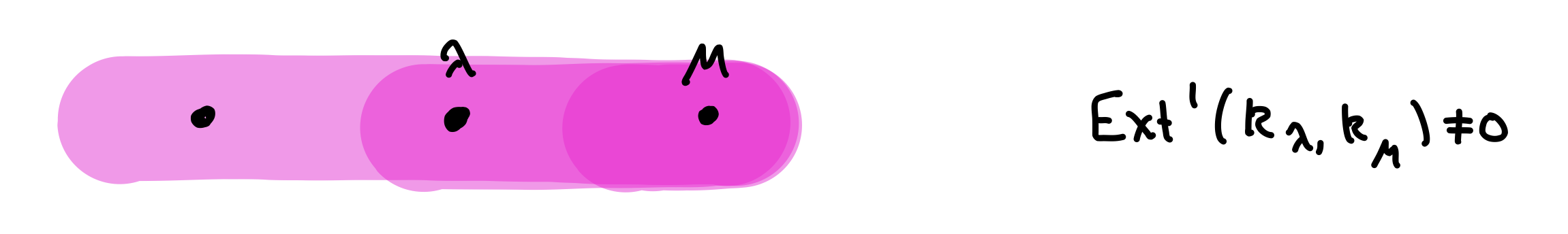}
\]
\end{example}
In fact, $\Hom^\bullet(k_\lambda, k_\mu)=0$ unless $\lambda \leq \mu$. Similarly on $\widetilde{\mc{N}}$, we have 
\[
\Hom_{\Coh^{G^\vee}(\widetilde{\mc{N}})}(\mc{O}(\lambda), \mc{O}(\mu))=0 \text{ unless } \lambda \leq \mu.
\]
Our equivalence will end up mapping 
\[
\mc{O}(\lambda) \mapsto J_\lambda.
\]

\subsection{Beilinson--Gaitsgory--Kottwitz (BGK) central sheaves}
Let $\mc{G}r = G((t))/G[[t]]$ be the affine Grassmannian and $\mc{F}\ell = G((t))/I$ the affine flag variety. 

\begin{theorem}
There exists a central functor 
\[
Z:P_{G(\mc{O})}(\mc{G}r) \rightarrow P_I = P_I(\mc{F}\ell) \subset D_I(\mc{F}\ell).
\]
Moreover, $Z$ is equipped with a $\otimes$-derivation $N$. 
\end{theorem}
\begin{remark}
The necessary compatibilities are checked in the appendix ``braiding compatabilities'' by Gaitsgory to Bezrukavnikov's \cite{Bez-tensor} and in more detail in a book in progress by Achar-Riche. 
\end{remark}

For the purposes of stating the next theorem, we (temporarily) extend the definition of Wakimoto sheaves: write $w \in W$ as $w=t_\lambda w_f$ for $t_\lambda \in X^\vee, w_f \in W_f$, and let
\[
J_w = J_\lambda * \nabla_{w_f}.
\]
The affine analogue of Mirkovic's observation from last lecture is the following. 
\begin{theorem}
\label{filtrations}
\cite{AB} \begin{enumerate}
    \item Any convolution exact $\mc{F} \in P_I$ is $\{J_w \mid w \in W\}$-filtered. 
    \item If in addition $\mc{F}$ is ``weakly central'' (i.e. $\mc{F} * \mc{G} \simeq \mc{G} * \mc{F}$ for all $\mc{G}$), then $\mc{F}$ is Wakimoto-filtered (i.e. only $J_\lambda$ for $\lambda \in X^\vee$ occur in the filtration above). 
\end{enumerate}
\end{theorem}

\begin{corollary}
$Z(\mc{G})$ is Wakimoto-filtered! i.e., 
\[
Z: P_{G(\mc{O})}(\mc{G}r) \rightarrow \mc{A}
\]
\end{corollary}
We omit the proof of Theorem \ref{filtrations}, but note that it is not much more difficult than the proof of Mirkovic's observation that we described last week. 

The following is beautiful and bears a striking resemblance to Mirkivic-Vilonen's theorem on weight functors. (In fact, according to \cite[Remark 6]{AB}, they are ``equivalent''.)

\begin{theorem}
\begin{enumerate}
    \item The following diagram commutes up to canonical isomorphism.
    \[
    \begin{tikzcd}
    \Rep G^\vee \simeq P_{G(\mc{O})}(\mc{G}r) \arrow[r, "Z"] \arrow[d, "\mathrm{res}"] & \mc{A} \arrow[d, "\mathrm{gr}"] \\
    \Rep T^\vee \arrow[r, "\sim"] & \mathrm{gr}\mc{A} 
    \end{tikzcd}
    \]
    \item We have 
    \[
    H^i_c(\iota_w^*(Z(V))) = \begin{cases} V_\nu & \text{ if }w=\nu \in X^\vee, i=\ell(\nu); \\
    0 & \text{ otherwise.} \end{cases}
    \]
\end{enumerate}
\end{theorem}

\begin{remark}
Part 2. is included for experts. The notation involved is the following: The $N_{\C((t))}$-orbits on $\mc{F}\ell$ are indexed by $W$. Given $w \in W$, let $S_w=N_{\C((t))} \cdot wI/I$ and let $\iota_w: S_w \hookrightarrow \mc{F}\ell$ denote its inclusion.
\end{remark}

\noindent
We now have all of the ingredients that we need to define our functor:
\begin{enumerate}
    \item $Z:\Rep{G^\vee} \rightarrow \mc{A}$; 
    \item $F:\Rep T^\vee \xrightarrow{\sim} \mathrm{gr}\mc{A} \subset \mc{A}$ (not full);
    \item $N=$ nilpotent $\otimes$-derivation of $Z$, coming from monodromy of vanishing cycles;
    \item $b_\lambda:J_\lambda \rightarrow Z(V_\lambda)$ highest weight arrow coming from Wakimoto filtration.
\end{enumerate}

Now we need to check:

\begin{enumerate} [label=(\alph*)]
    \item {\bf Pl\"{u}cker relations:} Thus we need to check that the following diagram commutes:
    \[
    \begin{tikzcd}
    J_\lambda \otimes J_\mu \arrow[d] \arrow[r, "b_\lambda \otimes b_\mu"] & Z(V_\lambda) * Z(V_\mu) = V(V_\lambda \otimes V_\mu) \arrow[d, "Z(m_{\lambda, \mu})"] \\
    J_{\lambda + \mu} \arrow[r, "b_{\lambda + \mu}"] & Z(V_{\lambda + \mu})
    \end{tikzcd}
    \]
    This is immediate from the multiplicativity of the Wakimoto functor. 
    \item {\bf Compatibility between $b_\lambda$ and $N$:} (relates $\mf{g}^\vee$ and $G^\vee/N^\vee$)
    
    We want a commutative diagram 
    \[
    \begin{tikzcd}
    J_\lambda \arrow[r, "b_\lambda"] \arrow[dr, "0"] & Z(V_\lambda) \arrow[d, "N_{V_\lambda}"]\\
     & Z(V_\lambda)
    \end{tikzcd}
    \]
    \begin{proof}
    $Z(V_\lambda)$ has a Wakimoto filtration $\vcenter{\hbox{\includegraphics[scale=0.1]{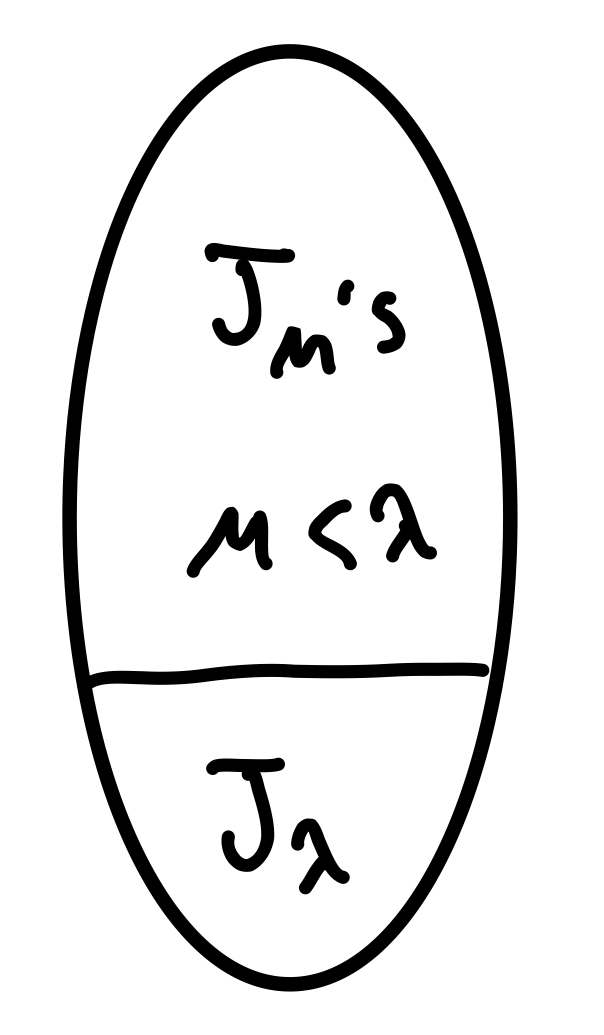}}}$. Because $\Hom(J_\lambda, J_\mu)=0$ for $\lambda>\mu$, $N_{V_\lambda} \circ b_\lambda$ maps $J_\lambda$ into $J_\lambda \subset Z(V_\lambda)$, hence $N_{V_\lambda}$ induces an endomorphism of $J_\lambda$, which has to be zero because $\End(J_\lambda)=\Q$ and $N$ is nilpotent. 
    \end{proof}
\end{enumerate}

Now we have the following set-up:
\[
\begin{tikzcd}
\widehat{\widetilde{\mc{N}}}_\mathrm{aff} \arrow[r, hookrightarrow] & \mf{g}^\vee \times \overline{G^\vee/ U^\vee} \\
\widehat{\widetilde{\mc{N}}} \arrow[d, "T^\vee\text{-bundle}"] \arrow[r, hookrightarrow] \arrow[u, hookrightarrow, "\mathrm{open}"'] & \mf{g}^\vee \times G^\vee/U^\vee \arrow[u, hookrightarrow, "\mathrm{open}"'] \arrow[d, "T^\vee\text{-bundle}"] \\
\widetilde{\mc{N}} \arrow[r, hookrightarrow] & \mf{g}^\vee \times \mc{B} 
\end{tikzcd}
\]
(For the definition of $\widehat{\widetilde{\mc{N}}}$, see \cite{AB}.) By a variant of the Tannakian formalism discussed in previous lectures, the above data gives a functor 
\[
\Coh_{\mathrm{free}}^{G^\vee \times T^\vee} (\widehat{\widetilde{\mc{N}}}) \rightarrow \mc{A}. 
\]
From this we obtain a functor 
\[
K^b\left(\Coh_{\mathrm{free}}^{G^\vee \times T^\vee} (\widehat{\widetilde{\mc{N}}})\right) \rightarrow K^b(\mc{A}).
\]
The functor passes to derived categories, complexes on $\partial \widehat{\widetilde{\mc{N}}}$ go to zero\footnote{This fact must sadly remain a black box here, due to time constraints.}, and we arrive at our desired functor:
\[
\widetilde{F}: D^{G^\vee}(\widetilde{\mc{N}}) \rightarrow D^b(P_I).
\]
All that remains is to see that $\widetilde{F}$ induces an equivalence with the anti-spherical module. 

\begin{example}
We'll end today's lecture with an illustrative example. Recall that we described $Z(\mathrm{nat})$ explicitly for $G=\GL_2$ in Lecture \ref{lecture 23}. We arrived at the following picture. (Here the notation aligns with the exercises of Lecture \ref{lecture 23}.) 
\[
\includegraphics[scale=0.4]{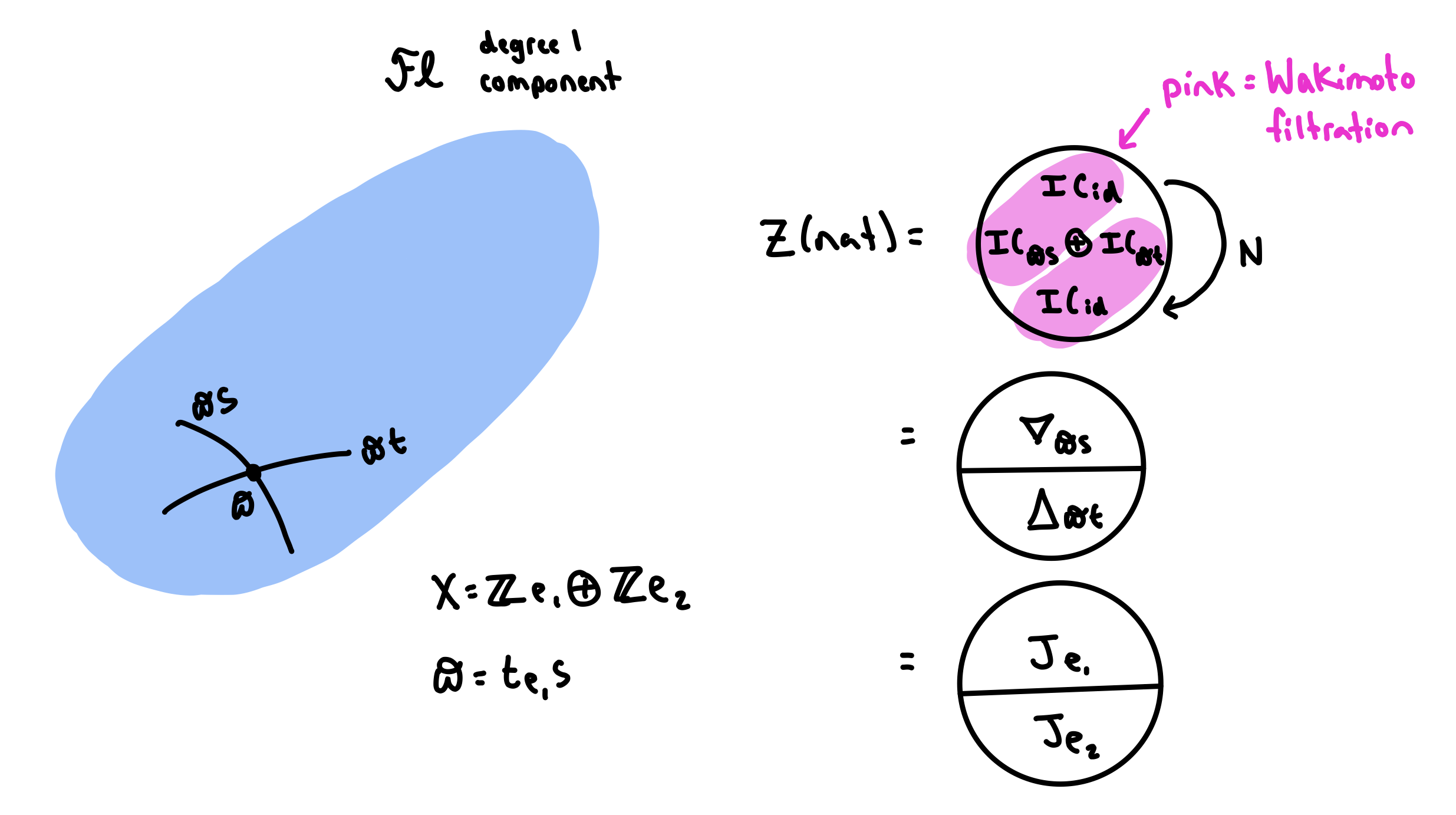}
\]
\end{example}

%% file: lecture-32.tex
\section{Lecture 32: Whittaker sheaves and Arkhipov-Bezrukavnikov theorem} 
\label{lecture 32}

We'll start today with a simple and beautiful example. But first we need to introduce some notation. 

\subsection{Averaging functors}
Let $N \circlearrowright X$ be the action of an algebraic group, and denote by $m$ and $p$ the corresponding action and projection maps:
\[
\begin{tikzcd}
N \times X \arrow[r, "m"] \arrow[d, "p"'] &X \\
X & 
\end{tikzcd}
\]
We define two functors
\[
\begin{tikzcd}
D^b_c(X) \arrow[r, "Av_{N*}"', bend right] \arrow[r, "Av_{N!}", bend left] & D^b_N(X)
\end{tikzcd}
\]
by 
\begin{align*}
    Av_{N*}&:= m_*(\Q_N \boxtimes (-) [\dim N]) \simeq m_*p^*[\dim N] \simeq m_* p^! [-\dim N], \\
    Av_{N!} &:= m_!(\mathbb{Q}_N \boxtimes (-) [\dim N]).
\end{align*}
One can imagine these functors as ``smearing out'' (integrating) our sheaf over the $N$-orbits to make it equivariant.
\begin{remark}
\label{averaging facts}
\begin{enumerate}
    \item We have 
\[
\mathbb{D} Av_{N*} \simeq Av_{N!} \mathbb{D}. 
\]
\item The functors $Av_{N*}$ and $Av_{N!}$ fit into adjoint pairs
\[
(\mathrm{For}[-\dim N], Av_{N*}) \text{ and } (Av_{N!}, \mathrm{For}[\dim{N}]). 
\]
\item The functor $Av_{N*}$ preserves ${^pD}^{\leq 0}$ and $Av_{N!}$ preserves ${^pD}^{\geq 0}$. 
\end{enumerate}
\end{remark}
\begin{remark}
These are a special case of the induction and restriction functors of Bernstein--Lunts. 
\end{remark}

\subsection{Another fun calculation}
Let $X=\PP^1$ and $N = \bp 1 & * \\ 0 & 1 \ep$. Then one can show that $N$-equivariant perverse sheaves are equivalent to perverse sheaves with respect to the $B$-orbit stratification: $P_N = P_{(B)}$. (This follows from the fact that the ``forget equivariance" functor is fully faithful when the group is unipotent.) As we've discussed in previous lectures, there are $5$ indecomposable objects in $P_{(B)}$: $\Delta_s, \nabla_s, \IC_s, \IC_{id}$, and $T_s=P_{id}$. Their relationship is captured by the following diagram:
\[
\begin{tikzcd}
& \IC_{id} \arrow[dl, hookrightarrow] & \\
\Delta_s \arrow[r, hookrightarrow] \arrow[dr, twoheadrightarrow] &P_{id}=T_s \arrow[r, twoheadrightarrow] & \nabla_s \arrow[ul, twoheadrightarrow] \\
& \IC_s \arrow[ur, hookrightarrow] & 
\end{tikzcd}
\]
We have described geometric constructions for $\Delta_s, \nabla_s, \IC_s, \IC_{id}$, but we established the existence of $T_s$ via formal properties of highest weight categories. Is there a geometric construction of $T_s=P_{id}$ as well? 

The answer to this question is yes, and it can be realised using averaging functors. Recall the perverse sheaf $\Q_{x!, y*}[1]$ for $x, y \neq 0$ that we saw in Lecture \ref{lecture 30}:
\[
\includegraphics[scale=0.15]{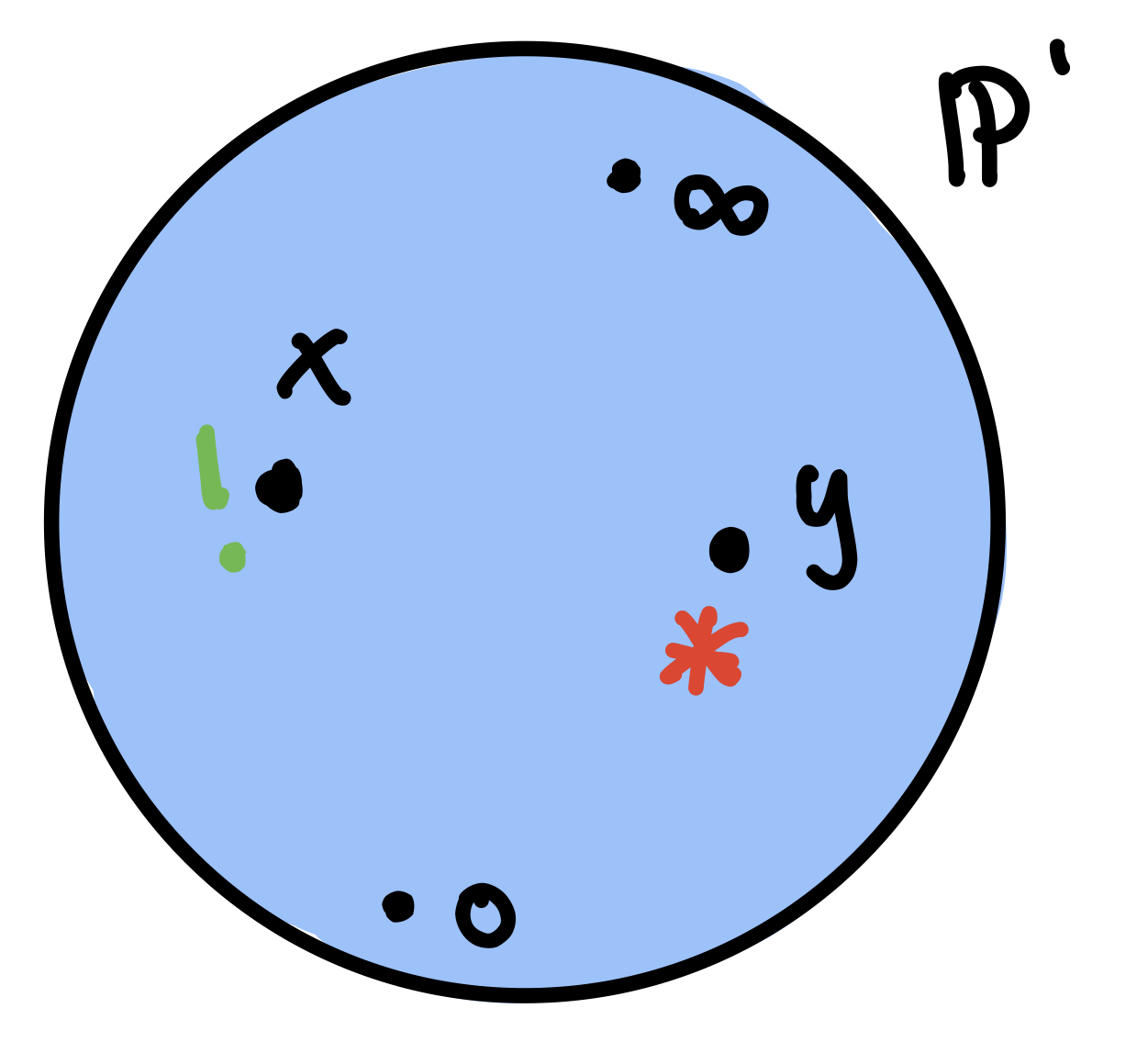}
\]
\begin{lemma}
\[
Av_{N*}(\Q_{x!, y*}[1]) \simeq Av_{N!}(\Q_{x!, y*}[1]) \simeq T_s.
\]
\end{lemma} For $x, y \neq 0$, we have 
\begin{proof}
{\bf Step 1:} Let $j: \C \hookrightarrow \PP^1$ be inclusion of the open Bruhat cell. Averaging functors commute with equivariant inclusions, so we have 
\[
j^* Av_{N*}(\Q_{x!, y*}[1]) \simeq Av_{N*}(j^* \Q_{x!, y*}[1]). 
\]
This is an $N$-equivariant sheaf on $\C$, so it is determined by its global sections. 
\begin{exercise}
\label{fun calculation}
$H^i(\C, \Q_{x!, y*}[1]) = \begin{cases} \Q & \text{ for }i=0, \\ 0 & \text{ otherwise}. \end{cases}$

(Hint: See the ``fun calculation'' from Lecture \ref{lecture 30}.) 
\end{exercise}
The shift by $[1]$ in the definition of $Av$ produces a shift by $[2]$ above, hence by Exercise \ref{fun calculation}, we have:
\[
j^* Av_{N*}(\Q_{x!, y*}[1]) \simeq Av_{N*}(j^* \Q_{x!, y*}[1]) = \Q_{\C}[1]. 
\]
{\bf Step 2:} Let $i:\{0\} \hookrightarrow \PP^1$ be inclusion of the closed Bruhat cell. Then by adjunction,
\begin{align*}
    \Hom^\bullet_{D^b_N}(i_* \Q_0, Av_{N*} \Q_{x!, y*}[1]) &\simeq \Hom^\bullet( \mathrm{For}(i_* \Q_0)[-1], \Q_{x!, y*}[1]) \\
    &\simeq \Hom^\bullet (\Q_0, i^!\Q_{x!, y*}[2]) \\
    &= \begin{cases} \Q &\text{ if } \bullet = 0, \\ 0 &\text{ otherwise}. \end{cases}
\end{align*}
{\bf Step 3:} By Remark \ref{averaging facts}, we know that $Av_{N*} \Q_{x!, y*}[1] \in {^pD}^{\leq 0}$. Moreover, by steps 1 and 2, $Av_{N*} \Q_{x!, y*} \in {^pD}^{\geq 0}$, so we conclude that $Av_{N*} \Q_{x!, y*} \in P_N (\PP^1).$ Hence $Av_{N*} \Q_{x!, y*}$ decomposes into a direct sum of indecomposable perverse sheaves:
\[
Av_{N*} \Q_{x!, y*}[1] \simeq \IC_{id}^{\oplus m_{id}} \oplus \cdots \oplus T_s^{\oplus n}
\]
{\bf Step 4:} Finally, by the fun calculation of last week, we have
\[
H^*(Av_{N*} \Q_{x!, y*}) = H^*(\Q_{x!, y*}) = 0.
\]
This implies that only $T_s$ can occur in the decomposition of $Av_{N*}\Q_{x!, y*}[1]$ into indecomposables, and by steps 1 and 2, there must be only a single copy. We conclude that 
\[
Av_{N*} \Q_{x!, y*}[1]=T_s,
\]
as desired.
\end{proof}

\noindent
{\bf Another construction:} (which Geordie learned from K. Vilonen)

Rough idea: Perhaps, it is easy to calculate the composition series of $\Q_{x!, y*}$. In this case, one could dream of taking a limit as $x, y \rightarrow 0$ to obtain the big tilting sheaf:
\[
\includegraphics[scale=0.3]{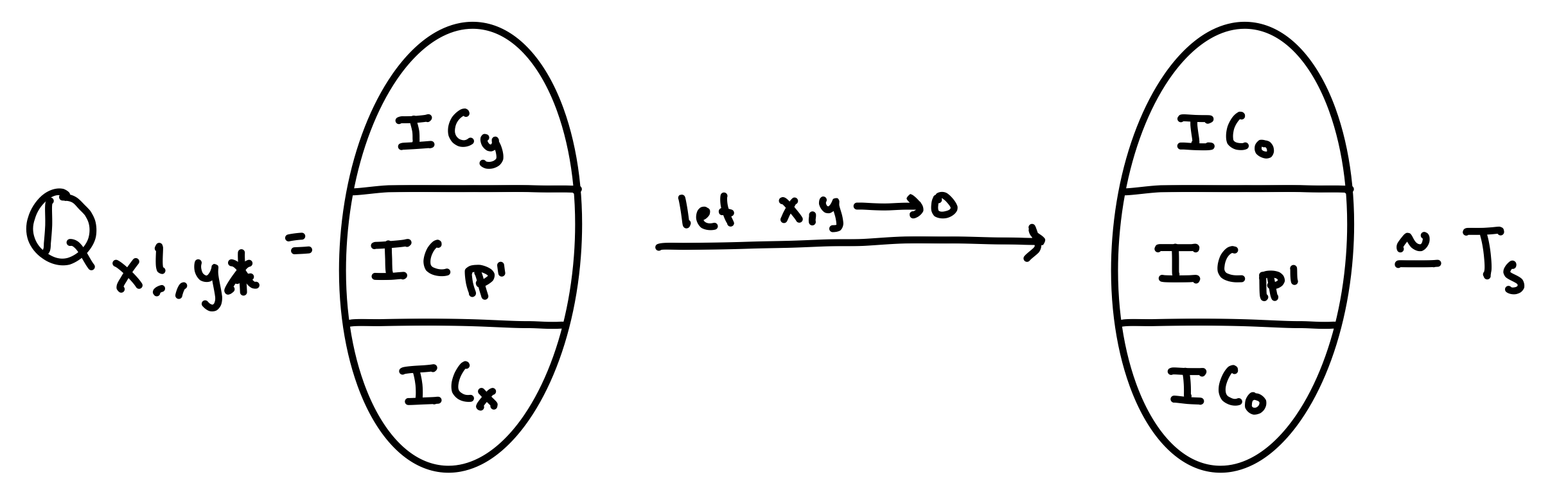}
\]
{\bf Question:} How do we formalize ``taking a limit''? 

\vspace{3mm}
\noindent
{\bf Answer:} Nearby cycles!
\[
\includegraphics[scale=0.3]{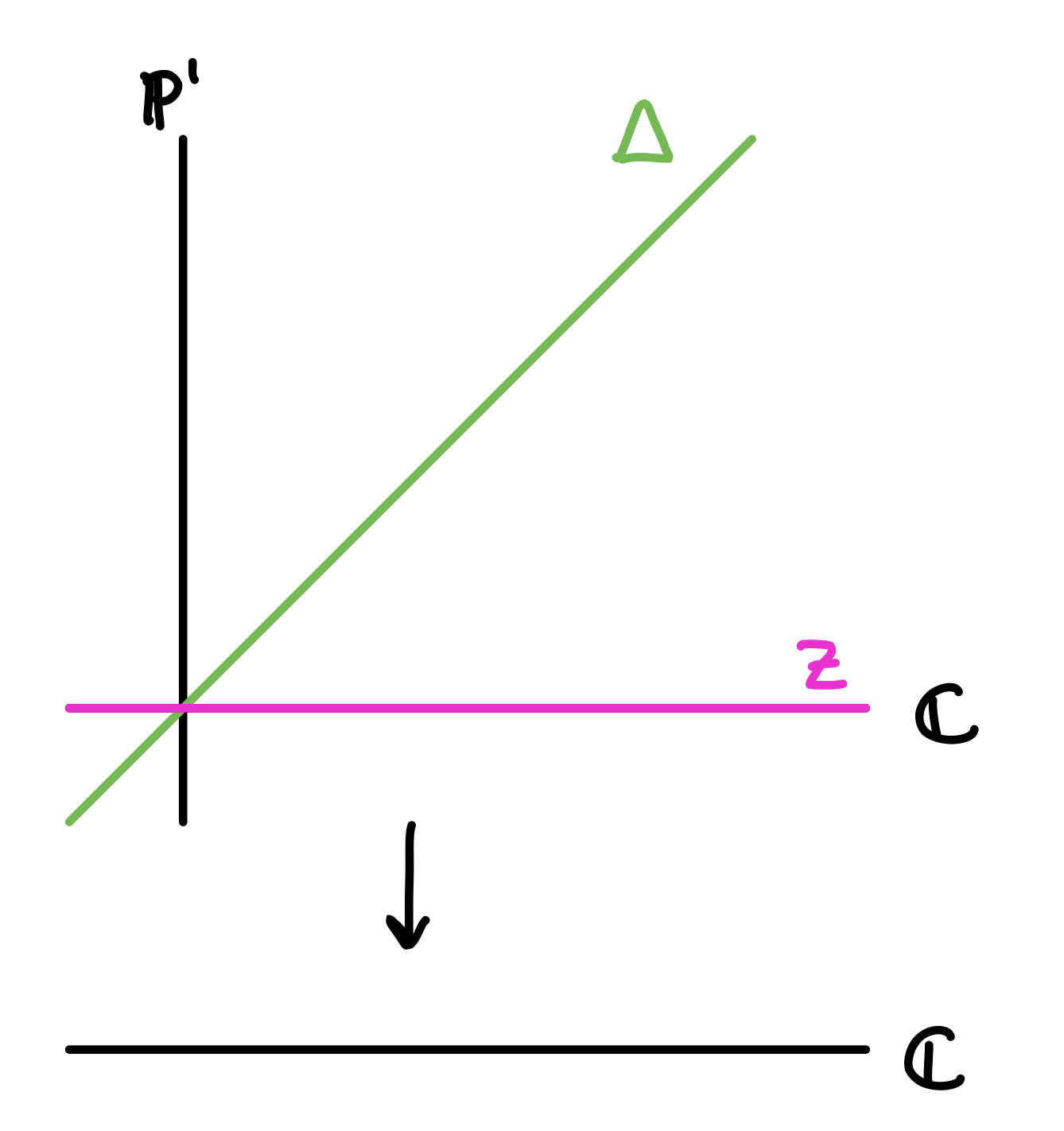}
\]
Let 
\[
\mc{F} = (j_! \Q_{(\PP^1 \times \C) \backslash \Delta}) \otimes (j_* \Q_{(\PP^1 \times \C)\backslash Z})[2]. 
\]
Then 
\[
\mc{F}|_{\PP^1 \times \{x\}} \simeq (\Q_{x!, 0*}[1])[1]
\]
is a ``family of $\Q_{x!, 0*}$ for varying $x$''. 

\begin{exercise}
(Beautiful!) $T_s \simeq \psi_f\mc{F}$. 
\end{exercise}
\begin{remark}
(Spoiler alert): It's perverse and $B$-constructible $\dots$ now what are its global sections?!
\end{remark}

\subsection{Lightning introduction to the Whittaker world}

Let $N \subset B \subset G$ be (for a moment) finite groups of Lie type. Fix $\chi:N \rightarrow \C^\times$ a character. Given a representation $V$ of $G$, consider
\[
V^{(N, \chi)} = \{ v\in V \mid n \cdot v = \chi(n)v \}.
\]
This is the space of {\em Whittaker vectors} in the representation. By Frobenius reciprocity, 
\[
\Hom_N(\C_\chi, V) = \Hom_G(\Ind_N^G \C_\chi, V),
\]
so Whittaker vectors detect whether a representation is ``seen'' by a 1-dimensional character of $N$. Put another way, what irreducible representations can one get by inducing from a character of $N$? Precisely those that admit a non-zero Whittaker vector!

Suppose $V= \Fun(X, \C)$ for some $G$-space $X$. In this setting, what do Whittaker vectors look like? 

Fix an $N$-orbit $U \subset X$. This is a homogeneous space for $N$, so $U \simeq N/K$ for some $K \subset N$. Then 
\[
\Fun(U) = \Ind_K^N \C = \bigoplus_{\lambda \in \widehat{N}} V_\lambda \otimes (V_\lambda^*)^K.
\]
Hence 
\[
\Fun(U)^{(N, \chi)} = \begin{cases} \C & \text{ if $\chi$ is trivial on $K$}, \\
0 & \text{ otherwise}. \end{cases}
\]
So up to scaling, we have ``either 0 or 1 Whittaker vectors per $N$-orbit'' and we can tell exactly which orbits admit non-zero Whittaker vectors.  

\begin{exercise}
(Important) Let $X=G/B$ and $N^-$ be the unipotent radical of the opposite Borel. Then 
\[
N^- / [N^-, N^-] \simeq \bigoplus_{\text{simple} \atop \text{roots}} \F_q^\alpha. 
\]
Hence a character $\chi:N^- \rightarrow \C$ determines a subset of simple reflections:
\[
I = \{s_\alpha \in S \mid \chi\neq 0 \text{ on }\F_q^\alpha \}.
\]
Show that the orbit $N^-x B/B$ supports a Whittaker vector if and only if $x \in {^IW}$ (the set of minimal coset representatives of $W_I\backslash W$). 
\end{exercise}
\begin{example}
If $\chi$ is nondegenerate (i.e. $I=S$), then $\Fun(G/B, \C)^{(N^-, \chi)} = \C$. 
\end{example}

\noindent
{\bf Categorifying Whittaker functions:} Now let $G, B, N, N^-$, etc. be defined over $\overline{\F}_q$. Just as
\[
\Fun(X)^{N^-} \text{ is categorified by }D_{N^-}(X), 
\]
we will see that 
\[
\Fun(X)^{(N^-, \chi)} \text{ is categorified by }D_{(N^-, \mc{L})}(X). 
\]
But what is this category $D_{(N^-, \mc{L})}(X)$? Our first step is to find an appropriate way to categorify the character $\chi$. We do this using the {\bf Artin-Schreier sheaf}.

Before defining this sheaf, we give some motivation. Recall that
\[
\C/\Z \xrightarrow{\sim} \C^\times
\]
via $z \mapsto \exp(2 \pi i z)$. The exponential map is the fundamental additive character of $\C$. Analogously, 
\[
\A^1 / \F_p \xrightarrow{\sim} \A^1 
\]
via the Artin-Schreier map $a: x \mapsto x^p - x$. (The map $a: \A^1 \rightarrow \A^1$ has kernel $\F_p$, so provides the isomorphism above.) Moreover, we have 
\[
a_*(\Q_\ell)_{\A^1} \simeq \bigoplus_{\chi: \F_p \rightarrow \overline{\Q}_\ell^\times} \mc{L}_\chi. 
\]
The sheaves $\mc{L}_\chi$ are examples of {\em character sheaves} on $\A^1$. 

Define a map $p$ by the following diagram: 
\[
\begin{tikzcd}
N^- \arrow[d] \arrow[drr, "p"] & & \\
N^-/[N^-, N^-] \arrow[r, dash, "\sim"] & \prod_{\alpha \text{ simple}} \mathbb{G}_a^\alpha \arrow[r, "\sum"'] & \mathbb{G}_a 
\end{tikzcd}
\]
Fix a non-trivial additive character $\chi$ of $\A^1$, and define 
\[
\mc{L}:=p^* \mc{L}_\chi.
\]
\begin{exercise}
The sheaves $\mathcal{L}_\chi$ and $\mc{L}$ are ``multiplicative''; i.e. we have $m^* \mc{L} \simeq \mc{L} \boxtimes \mc{L}$, where $m:N^- \times N^- \rightarrow N^-$ is multiplication. Multiplicative sheaves are ``one-dimensional character sheaves'' or ``categorified characters''. 
\end{exercise}

Now let $N^- \circlearrowright X$. A {\em $(N^-, \mc{L})$-equivariant complex on $X$} is a pair $(\mc{F}, \beta)$, where $\mc{F} \in D_c^b(X)$ and 
\[
\beta: a^* \mc{F} \xrightarrow{\sim} \mc{L} \boxtimes \mc{F}
\]
is an isomorphism satisfying the usual cocycle condition. Let 
\[
D_{(N^-, \mc{L})}(X) = \text{ category of }(N^-, \mc{L})\text{-equivariant complexes on }X. 
\]
Morphisms in this category are morphisms in $D_c^b(X)$ which commute with $\beta$. We can define averaging functors 
\[
Av_{\mc{L}*}, Av_{\mc{L}!}: D^b_c(V) \rightarrow D^b_{(N^-, \mc{L})}(X)
\]
as we did in the beginning of this lecture, 
\[
Av_{\mc{L}*} = m_*(\mc{L} \boxtimes (-) )[\dim N^-], \text{ etc.} 
\]
{\bf Important fact:} The forgetful functor 
\[
\mathrm{For}: D_{(N^-, \mc{L})}(X) \rightarrow D_c^b(X)
\]
is fully faithful. 

\begin{exercise}
Let 
\[
\begin{tikzcd}
\mathbb{A}^1 \arrow[r, hookrightarrow] \arrow[d, "a"] &\PP^1 \arrow[d]\\
\A^1 \arrow[r, hookrightarrow] & \PP^1
\end{tikzcd}
\]
\begin{itemize}
    \item Show explicitly that the map $a:x \mapsto x^p - x$ extends to $\PP^1$. (Find a formula!)
    \item Show that $a$ breaks all the rules you know about coverings of Riemann surfaces. (The map $a$ at $\infty$ is the simplest example of ``wild ramification''.) 
    \item Show that $j_! \mc{L}_\chi \xrightarrow{\sim} j_* \mc{L}_\chi$. 
    \item Show that $H^*(\PP^1, j_! \mc{L}_\chi)=0$. (We can interpret this as saying that $\mc{L}_\chi$ is a bit like our friend $\Q_{x!, y*}$ from earlier.) 
    \item Show that $j_! \mc{L}_\chi$ is $(N^-, \mc{L}_\chi)$-equivariant. 
    \item Show that $Av_{N*}(j_! \mc{L}_\chi) \simeq Av_{N!}(j_* \mc{L}_\chi) = T_s$. 
\end{itemize}
\end{exercise}

\subsection{Bird's eye view of the rest of the proof}
We return to our usual setting: fix $G$, $\mc{F}\ell$ the corresponding affine flag variety, $I \subset G((t))$ Iwahori subgroup, $I^-$ Iwahori for the opposite Borel,  $I_u \subset I$, $I_u^- \subset I^-$ pro-unipotent radicals. Let 
\[
N_K = \bp 1 & * & * \\ 0 & \ddots & * \\ 0 & 0 & 1 \ep \subset G((t)), \hspace{5mm} N_K^- = \bp 1 & 0 & 0 \\ * & \ddots & 0 \\ * & * & 1 \ep. 
\]
We have seen that the antispherical module can be realized as
\[
M_\mathrm{asph} \simeq H/\langle b_x \mid x \not \in {^fW}\rangle. 
\]
Another realization is important in $p$-adic groups: 
\[
M_\mathrm{asph} = \Fun(\mc{F}\ell)^{(N_K^-, \psi)},
\]
the ``Whittaker vectors in the principal series''. We might hope that on the level of categories, we would also have 
\[
\mc{M}_\mathrm{asph} \simeq D_{(N_K^-, \mc{L})}(\mc{F}\ell). 
\]
{\bf Problem:} $N_K^-$ orbits on $\mc{F}\ell$ are ``$\infty/2$-dimensional'' (i.e. they have neither finite dimension nor finite co-dimension). It is difficult to work with sheaves on an $\infty$-dimensional space. 

\vspace{3mm}
\noindent
{\bf One solution:} Use Drinfeld compactification. This approach is described in \cite{FGV2}. This is not yet understood by Geordie. 

\vspace{3mm}
\noindent
{\bf Another solution:} Use ``Iwahori-Whittaker'' or ``baby Whittaker'' techniques. The idea is to replace $N_K^-$ by $I_u^-$. 

\begin{lemma} For nondegenerate characters $\psi$ of $N_K^-$ and $\chi$ of $I_u^-$, 
\[
\Fun(\mc{F}\ell (\F_q))^{(N_K^-, \psi)} = \Fun(\mc{F}\ell (\F_q))^{(I_u^-, \chi)}. 
\]
\end{lemma}
The $I_u^-$-orbits on $\mc{F}\ell$ are just (opposite) Bruhat cells, so passing to $I_u^-$-orbits resolves our problem of infinte-dimensional orbits. 

Define 
\[
P_{IW} \subset D_{IW} = D_{(I_u^-, \mc{L})}(\mc{F}\ell), \hspace{5mm} \text{``Iwahori-Whittaker sheaves''}. 
\]
The irreducible objects in $P_{IW}$ are $\IC_\chi^{\mc{L}_\chi}$ for $x \in {^fW}$. Last week we constructed a functor 
\[
\widetilde{F}:D^b(\Coh^{G^\vee}(\widetilde{\mc{N}})) \rightarrow D_I 
\]
induced by 
\begin{align*}
    \Coh^{G^\vee}_\mathrm{free}(\widetilde{\mc{N}}) &\rightarrow P_I \\
    V \otimes \mc{O} &\mapsto Z(\mathrm{Sat}(V))  \hspace{5mm} \text{central sheaf}\\
    \mc{O}(\lambda) &\mapsto J_\lambda \hspace{18mm} \text{Wakimoto sheaf} 
\end{align*}
Let 
\[
{^f P}_I = P_I / \langle \IC_x \mid x \not \in {^fW}\rangle. 
\]
The main theorem is the following. 
\begin{theorem}
$\widetilde{F}$ induces an equivalence 
\[
D^b(\Coh^{G^\vee}(\widetilde{N})) \xrightarrow{\sim} D^b({^fP_I}). 
\]
\end{theorem}
The averaging functor 
\[
Av_{\mc{L}}: P_I \rightarrow D_{IW}
\]
factors over ${^fP}_I$. Now 
\[
\begin{tikzcd}
D^b(\Coh^{G^\vee}(\widetilde{\mc{N}})) \arrow[rrr, bend left, "F_{IW}"] \arrow[r, "\widetilde{F}"] & D^b({^fP}_I) \arrow[r, "Av_\chi"] &D^b(P_{IW}) \arrow[r, "\mathrm{real}"] & D_{IW}.
\end{tikzcd}
\]
\begin{theorem}
\begin{enumerate}
    \item ${^fP}_I \simeq P_{IW}$ (reasonably easy consequence of part 3. below)
    \item $D^b(P_{IW}) \simeq D_{IW}$ (same argument as Corollary \ref{real})
    \item $D^b(\Coh^{G^\vee}(\widetilde{\mc{N}}))\xrightarrow[F_{IW}]{\sim} D_{IW}$ (main difficulty)
\end{enumerate}
\end{theorem}
As usual, fully-faithfulness of $F_{IW}$ is the main issue. 

\vspace{3mm}
\noindent
{\bf Step 1: }$F_{IW}$ is faithful

\vspace{2mm}
A game with induction and restriction functors (using the $G/B$ version of the fun calculation at the start of this lecture) implies that $Av_\psi$ is faithful. Hence it is enough to show that $\widetilde{F}$ is faithful. The key idea is the following. We have
\[
\widetilde{\mc{N}} \rightarrow \mc{N} \subset \mf{g}^\vee, 
\]
and $\mc{N}_\mathrm{reg} \subset \mc{N}$ is open and dense. For $V, V' \in \Coh^{G^\vee}_\mathrm{free}(\widetilde{\mc{N}})$, 
\[
\Hom(V, V') \rightarrow \Hom(V|_{\mc{N}_\mathrm{reg}}, V'|_{\mc{N}_\mathrm{reg}}) 
\]
is injective. We have a diagram
\[
\begin{tikzcd}
\Rep Z_{G^\vee}(N_0) \simeq \Coh^{G^\vee}(\mc{N}_\mathrm{reg}) \arrow[rr] & & (*) \\ 
\Coh^{G^\vee}_\mathrm{free}(\widetilde{\mc{N}}) \arrow[u, "\mathrm{res.}"] \arrow[r] & P_I \arrow[r] & \mc{M}_\mathrm{asph} \arrow[u]
\end{tikzcd}
\]
What corresponds to $(*)$? 
\[
D_I \supset D^{\neq id}_I = \langle \IC_x \mid x \neq id \rangle_\Delta. 
\]
Using this, we can define a $\otimes$-category 
\[
D_I^{id} := D_I / D_I^{\neq id}.
\]
Moreover, if $P_I^{id}$ is the image of $P_I$ in $D_I^{id}$, then $P_I^{id}$ is a $\otimes$-category with one simple object. Hence we have 
\[
\Rep G^\vee \xrightarrow{Z} P_I \rightarrow P_I^{id}.
\]
Using the central functor + derivation, Tannakian formalism (+ a bit of sauce) gives us the following diagram 
\[
\begin{tikzcd}
\Rep Z_G(N_0) \arrow[r, "\sim"] &\Coh^{G^\vee}(\mc{N}_\mathrm{reg}) \arrow[r, "\sim"] &P_I^{id} \\
\Rep G^\vee \arrow[u] \arrow[r] & \Coh^{G^\vee}_\mathrm{free} (\widetilde{\mc{N}}) \arrow[u, "\mathrm{res}"] \arrow[r, "\widetilde{F}"] &{^fP}_I \arrow[u]
\end{tikzcd}
\]
Because $\mathrm{res}$ is faithful, $\widetilde{F}$ is faithful. 

\begin{remark}
The above is an instance of an important theme in Bezrukavnikov's work: from a two-sided cell $\underline{c}$, Lusztig constructed a semi-simple abelian tensor category $J_{\underline{c}}$. Bezrukavnikov observed that $Z$ provides a central functor $\Rep{G}^\vee \rightarrow J_{\underline{c}}$. Bezrukavnikov then uses $Z$ to identify $J_{\underline{c}}$. The theorem is the case when $\underline{c}=\{id\}$.
\end{remark}

\noindent
{\bf Step 2:} $F_{IW}$ is full:

\hspace{2mm}
This can be shown using Beilinson's lemma and a generation argument. It reduces to checking 
\[
\Ext^i_{\Coh^{G^\vee}(\widetilde{\mc{N}})}(V \otimes \mc{O}, \mc{O}(\lambda)) \xrightarrow{F_{IW}} \Ext^i(Av_\mc{L}(\mathrm{Sat}(V)), Av_\mc{L}(J_\lambda))
\]
for $V \in \Rep G^\vee$ and $\lambda$ dominant. The left side is
\[
\Hom_{G^\vee} (V, H^i(\widetilde{\mc{N}}, \mc{O}(\lambda))) = \begin{cases} 0 &\text{ if } i \neq 0 \text{ (Frobenius splitting of $T^*\mc{B}$}), \\
V_\lambda &\text{ if } i=0. \end{cases} 
\]
We can chek that the right side has the same dimension and we are done. 

\vspace{5mm}
\noindent
{\bf Potential moral of the proof:} It is ``obvious'' that $P_{IW}$ is highest weight. Hence one can do calculations more more easily here than in ${^fP}_I$. 

%% file: lecture-33.tex
\section{Lecture 33: Soergel bimodules, Soergel calculus, and BGK central sheaves}

We will begin today's lecture by describing a more combinatorial approach to the Hecke category using Soergel bimodules. 

\subsection{The Hecke category and Soergel bimodules}
We start with some motivation. Recall that given $G$ a (split) finite group of Lie type and $B$ a Borel subgroup, we can define a $\C$-algebra 
\[
H_q:=(\Fun_{B \times B}(G, \C), *).
\]
This is the first incarnation of the Hecke algebra. Iwahori showed that $H_q$ admits a presentation which is ``independent'' of $q$, in that it only depends on the Coxeter system $(W,S)$ determined by $B \subset G$. This leads us to define a $\Z[v^{\pm 1}]$-algebra using this presentation, which allows us study ``all $q$ at once''. This construction now makes sense for any Coxeter system $(W,S)$. Can we do a similar thing for the Hecke category?

More specifically, recall from Lecture \ref{lecture 24} that given an algebraic group $G/
\C$, we defined a monoidal\footnote{by the Decomposition Theorem} category 
\[
\mc{H}_{s.s.}:= \left\langle \begin{array}{c} \text{additive category of} \\ \text{semi-simple complexes} \end{array} \right\rangle \subset D_{B \times B}^b(G, \Q).
\]
From this category, we built our final incarnation of the Hecke category:
\[
\mc{H}:= K^b(\mc{H}_{s.s.}). 
\]

\vspace{3mm}
\noindent
{\bf Questions:} Can we present $\mc{H}$ or $\mc{H}_{s.s.}$ by generators and relations? Does $\mc{H}$ makes sense for any Coxeter system? We can also define $\mc{H}$ using $G / \overline{\F}_q$ and \'{e}tale sheaves, do we get equivalent categories after extending scalars?  
\vspace{3mm}

To answer these questions, we need to introduce the notion of {\bf equivariant cohomology}. Let $K$ be a group acting on a space $X$. Define 
\[
H^*_K(X):= H^*(X \times_K EK),
\]
where $EK$ is a classifying space for $K$ (i.e. a path-connected, contractible space with a free $K$-action). This construction is sometimes referred to as the ``Borel construction''. There is a natural map 
\[
X \times_K EK \rightarrow \mathrm{pt} \times_K EK,
\]
which gives $H^*_K(X)$ the structure of a graded $H^*_K(\mathrm{pt})$-module. Similarly, given a complex $\mc{F} \in D^b_K(X)$, $H_K^*(X, \mc{F})$ is a graded module over $H_K^*(\mathrm{pt})$. 

\vspace{3mm}
\noindent
{\bf Equivariant cohomology for tori:} Let $T=\C^\times$. Then a classifying space for $T$ is 
\[
ET = \C^\infty \backslash \{ \infty\} = \lim_{\rightarrow} \C^n \backslash \{0\}.
\]
Hence 
\[
H^*_{\C^\times}(\mathrm{pt}) = H^*(\mathrm{pt} \times_{\C^\times} \C^\infty\backslash \{0\}) = H^*(\PP^\infty) = \C[x], 
\]
where $x$ is in degree $2$. Similarly, if $T \simeq (\C^\times)^n$, then $ET=(\C^\infty \backslash \{0\})^n$, and 
\begin{equation}
    \label{polynomial ring}
H^*_T(\mathrm{pt}) \simeq H^*((\PP^\infty)^n) = \C[x_1, \ldots, x_n]. 
\end{equation}

Note that the isomorphism (\ref{polynomial ring}) depends on the isomorphism $T \simeq (\C^\times)^n$. We can give a more canonical description as follows. Let $\chi:T \rightarrow \C^\times$ be a character. We have an associated $\C^\times$-bundle 
\[
L_\chi:= \C^\times \times_{T, \chi} ET \rightarrow \mathrm{pt} \times_T ET. 
\]
Borel showed that there is an isomorphism 
\begin{align*}
    X &\xrightarrow{\sim} H^2_T(\mathrm{pt}, \Z) \\
    x &\mapsto c_1(L_\chi),
\end{align*}
where $X$ is the character lattice of $T$ and $c_1(\L_\chi)$ is the first Chern class of $L_\chi$. This leads to a canonical isomorphism 
\[
S^\bullet(X) \simeq H^*_T(\mathrm{pt}, \C),
\]
where $S^\bullet(X)$ is the symmetric algebra of the character lattice. This is the {\bf Borel isomorphism}. 

\begin{remark}
Over $\C$, this can be further simplified. Given $\chi:T \rightarrow \C^\times$, we can differentiate to get a linear functional $d\chi: \Lie{T}\rightarrow \C$. Then the Borel isomorphism becomes
\[
\mc{O}(\Lie T) \simeq H^*_T(\mathrm{pt}, \C). 
\]
This is the version which we will use today. 
\end{remark}

\vspace{3mm}
\noindent
{\bf Useful trick:} If $K \subset G$ is a subgroup, then any model of $EG$ is also a model of $EK$ via restriction. Examples:
\begin{enumerate}
    \item $T \subset B$: The map 
    \[
    T\backslash EB \rightarrow B \backslash EB
    \]
    is a $B/T \simeq \C^n$-bundle, so $H_B^*(\mathrm{pt}) \simeq H^*_T(\mathrm{pt})$. 
    \item $T \subset G$: The map 
    \[
    T \backslash EG \xrightarrow{p} G \backslash EG
    \]
    induces a map $H_G^*(\mathrm{pt}) \xrightarrow{p^*} H_T^*(\mathrm{pt})$. 
    \begin{theorem} $p^*$ is injective, and the image is in $(\mc{O}(\Lie T))^W$.
    \end{theorem}
\end{enumerate}
By taking equivariant cohomology, we can give an algebraic description of the Hecke category. 
\begin{align*}
    \mc{H}_{s.s.} \subset D^b_{B \times B}(G) \xrightarrow{H^*_{B \times B}} &H^*_{B \times B}(\mathrm{pt})\text{-graded modules}\\
    &= H^*_{T \times T}(\mathrm{pt})\text{-graded modules} \\
    &= R\text{-gbim},
\end{align*}
where $R = \mc{O}(\Lie T)$. The category $R$-gbim is a monoidal category via $- \otimes_R -$. In most situations, taking cohomology loses a lot of information, so the following theorem is especially remarkable. 
\begin{theorem}
(Soergel) $H^*_{B \times B}$ is fully faithful and monoidal on $\mc{H}_{s.s.}$. 
\end{theorem}
Hence, 
\[
\mc{H}_{s.s.} \hookrightarrow R\text{-gbim}. 
\]
How can we describe the image? To start, observe that $\mc{H}_{s.s.}$ is generated by $IC_s= \C_{\overline{PsP}}[1]$ for $s \in S$ under $*, \oplus, \ominus, [1]$; i.e. 
\[
\mc{H}_{s.s.} = \langle IC_s \mid s \in S \rangle_{*, [\Z], \oplus, \ominus}
\]
This gives us a way of describing the image:
\[
\mathrm{Sbim}:= \langle B_s \mid s \in S \rangle_{*, \oplus, (1), \ominus},
\]
where $B_s = H^*_{B \times B}(IC_s) = R \otimes_{R^s} R(1)$. (Exercise: prove this!) This is the category of {\bf Soergel bimodules}. Unpacking definitions, we obtain equivalences of monoidal categories: 
\[
\begin{tikzcd}
 \mc{H}_{s.s.} \arrow[d] \arrow[r, "\sim"', "H^*_{B \times B}"] &\mathrm{Sbim} \arrow[d] \\
 \mc{H} = K^b(\mc{H}_{s.s.}) \arrow[r, "\sim"', "H_{B \times B}"] &K^b(\mathrm{Sbim})
\end{tikzcd}
\]

\noindent
{\bf Remarkable consequence:} To define $\mc{H}_{s.s.}$, we need $G$, $B$, hundreds of pages of sheaf theory, the decomposition theorem\footnote{can be avoided via the theory of parity sheaves}, and more. To define Sbim, all we need is $W \circlearrowright \Lie T$ and a bit of algebra!

\begin{remark}
\begin{enumerate}
    \item This can be seen as a first step toward freeing $\mc{H}$ from its concrete realization as a category of sheaves. By replacing $\Lie T$ with a reflection representation $\mf{h}$ of $W$, the definition makes sense for {\em any} Coxeter group. This led to the proof by Soergel and Elias-W. that the Kazhdan--Lusztig polynomials have non-negative coefficients. 
    \item One of the main goals of this course is to approach Bezrukavnikov's equivalence:
    \[
    \mc{H}^{\mathrm{affine}} \simeq \text{``coherent sheaves on Steinberg''}. 
    \]
    Soergel's theorem tells us that the Hecke category always has a coherent description in terms of Soergel bimodules. One way of understanding Bezrukavnikov's theorem is that ``Soergel bimodules have another name'' in the affine setting. 
    \item One can see Soergel bimodules as ``half way'' towards a generators and relations description. A generators and relations description was obtained by Elias-W., following Elias-Khovanov, Libedinsky, and Elias. See \cite{book}. An interesting recent take has been given by Abe in \cite{Abe}.
\end{enumerate}
\end{remark}

\subsection{BGK central sheaves for $G=\SL_2$}

{\bf Goal:} Describe $Z: \Rep \SL_2 \rightarrow \mc{H}_{ext}$ by ``generators and relations''. 

\vspace{3mm}
We start with the left hand side, $\Rep \SL_2$, which we can describe via the {\bf Temperly-Lieb category} $TL$: 
\begin{itemize}
    \item Objects: $\Z_{\geq 0}$ 
    \item Morphisms: crossingless matchings, up to isotopy, $\vcenter{\hbox{\includegraphics[scale=0.15]{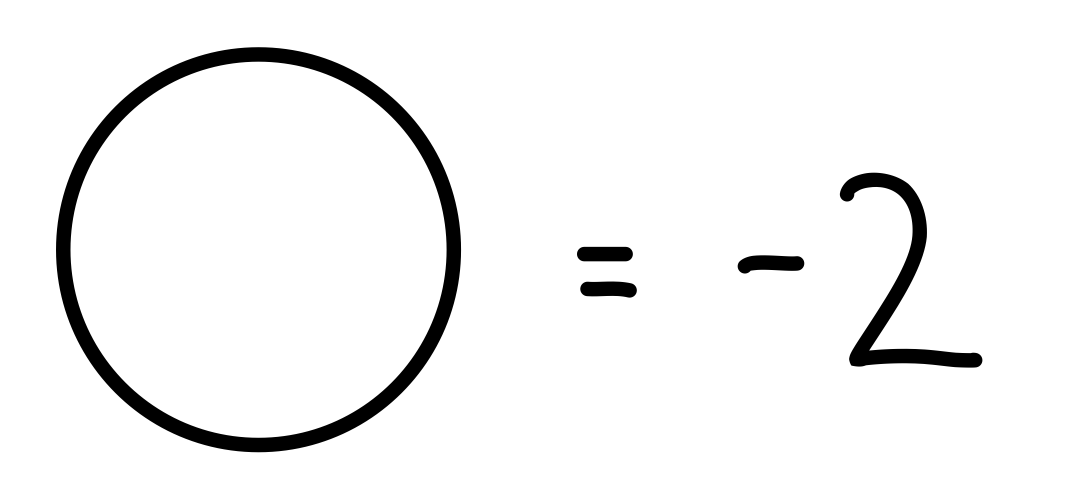}}}$
\end{itemize}
Let $V=\C^2$ be the natural representation of $\SL_2$. 
\begin{theorem}
$\langle V^{\otimes n} \rangle \simeq TL$, hence $(TL)_{\oplus, \ominus} \simeq \Rep \SL_2$. 
\end{theorem}
\begin{remark}
This description of $\Rep \SL_2$ is very useful. For example, an immediate consequence of the theorem is that giving a functor $\Rep \SL_2 \rightarrow (\mc{C}, \otimes)$ with $\mc{C}$ additive Karoubian is the same as giving a self-dual object $X \in \mc{C}$ of dimension $2$. 
\end{remark}

Giving a generators and relations description of the right hand side $\mc{H}_{ext}$ is more complicated. We'll start with the unextended case and describe $\mc{H}$, the Hecke category corresponding to the affine Weyl group $W=\langle s, t \rangle$. (Here $s$ is the finite simple reflection.) Fix a ``realisation'' of $W$; that is, fix the data of
\begin{itemize}
    \item a complex $\C$-vector space $\mf{h}$,
    \item vectors $\alpha_s, \alpha_t\ \in \mf{h}^*$, and 
    \item vectors $\alpha_s^\vee, \alpha_t^\vee \in \mf{h}$, such that the pairing between the $\alpha$ and the $\alpha^\vee$ is given by the Cartan matrix
    \[
    \bp 2 & -2 \\-2 & 2 \ep.
    \]
\end{itemize}
With such a realisation, $W$ acts via automorphisms of $\mf{h}$ according to the usual formulas. 

\vspace{2mm}
\noindent
Most important choices of realisations:
\begin{enumerate}
    \item $\mf{h} = \C$, $\alpha_s = - \alpha_t$, $\alpha_s^\vee = - \alpha_t^\vee$. (Realises canonical quotient $W \twoheadrightarrow W_f$.) 
    \item $\mf{h}_{\mathrm{loop}} = \C \alpha_s \oplus \C \delta$, $\alpha_t = - \alpha_s + \delta$, $\alpha_t^\vee(\delta) = \alpha_s(\delta) = 0$. $\vcenter{\hbox{\includegraphics[scale=0.15]{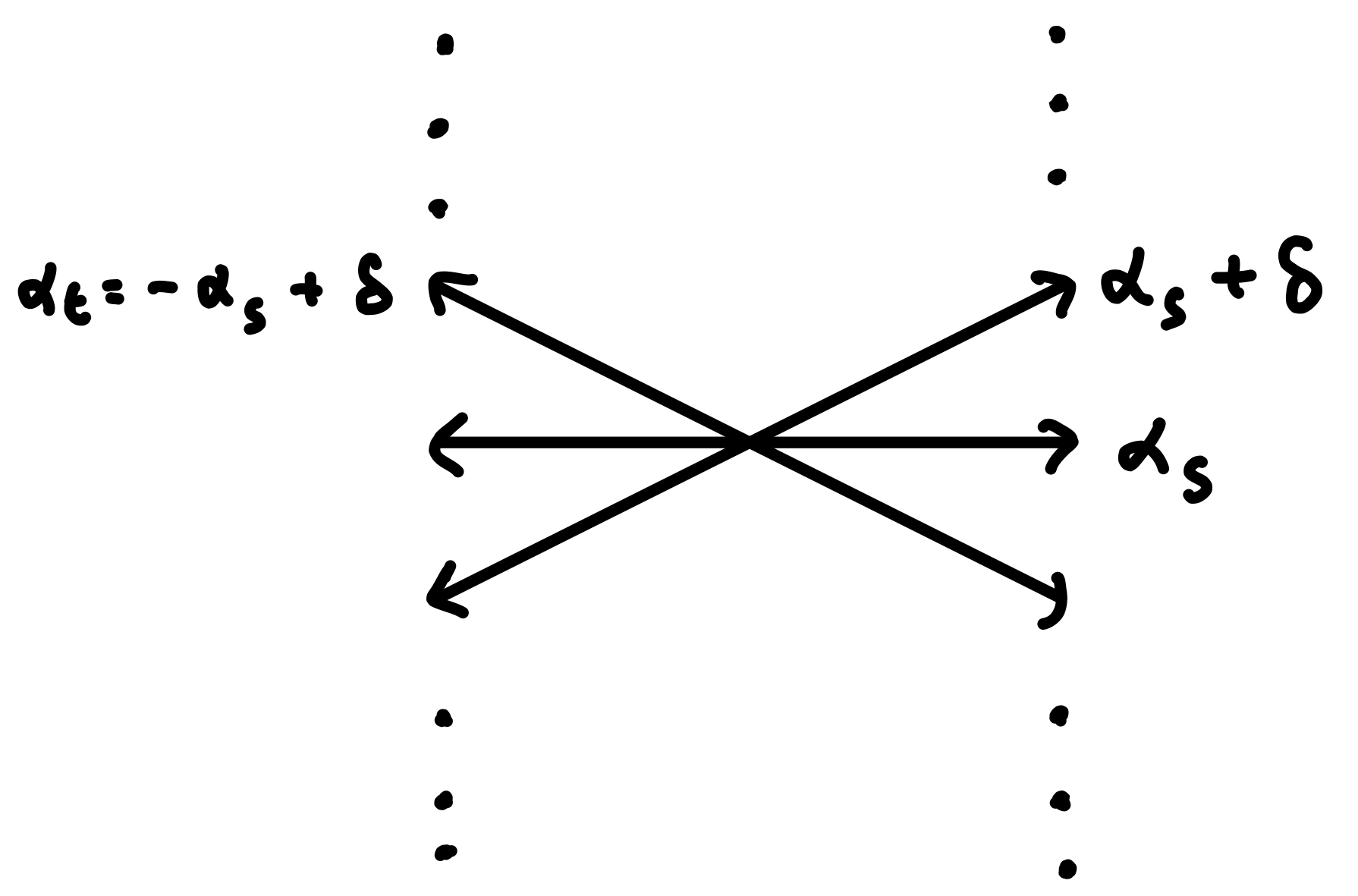}}}$
\end{enumerate}
\begin{remark}
In geometry, $\mf{h}$ arises from $I$, whereas $\mf{h}_\mathrm{loop}$ arises from $I \ltimes \C^\times$ (loop rotation). 
\end{remark}
Now we can describe a diagrammatic version of the Hecke category. Let 
\[
\mc{H}_{BS} = \text{ monoidal category generated by }B_{\color{blue} s} \text{ and } B_{\color{red} t}.
\]
Morphisms in $\mc{H}_{BS}$ are isotopy classes of diagrams generated by 
\[
\includegraphics[scale=0.3]{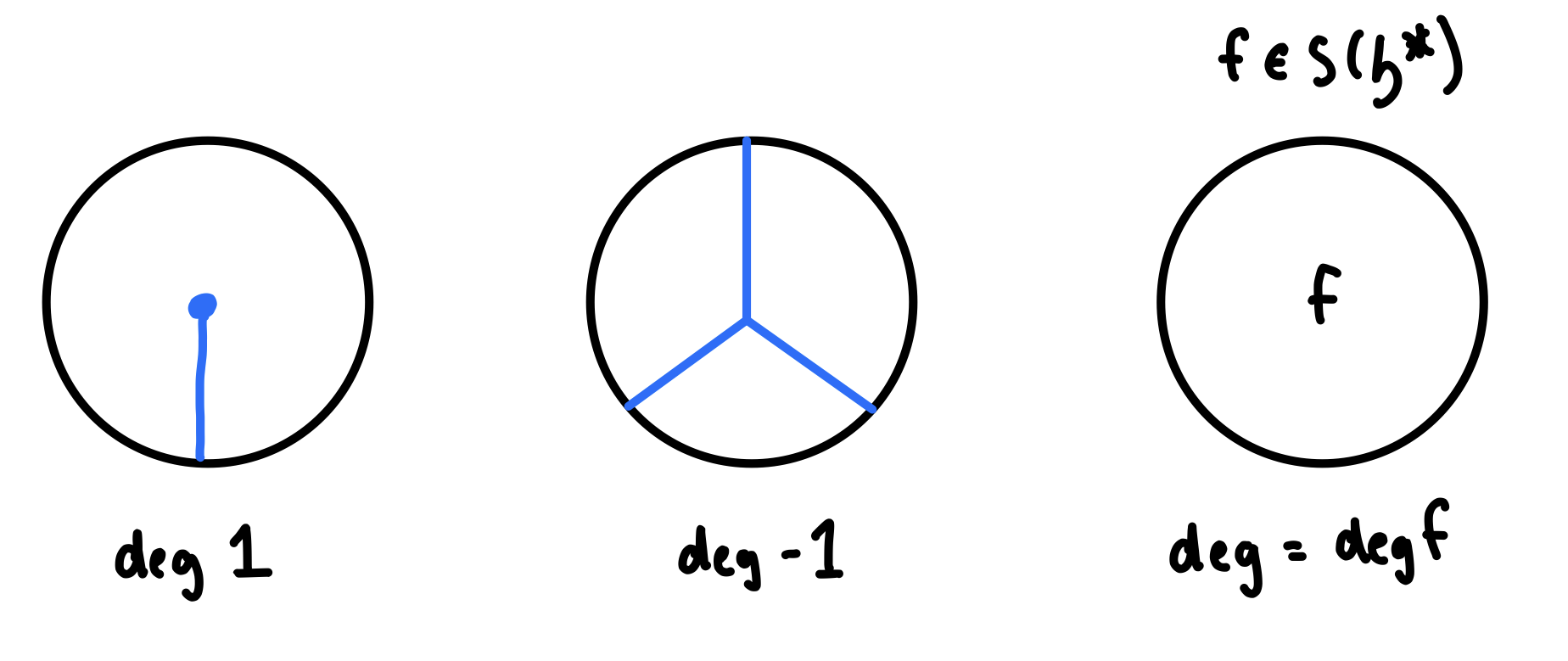},
\]
subject to a collection of relations. The most important of these relations are:
\[
\includegraphics[scale=0.35]{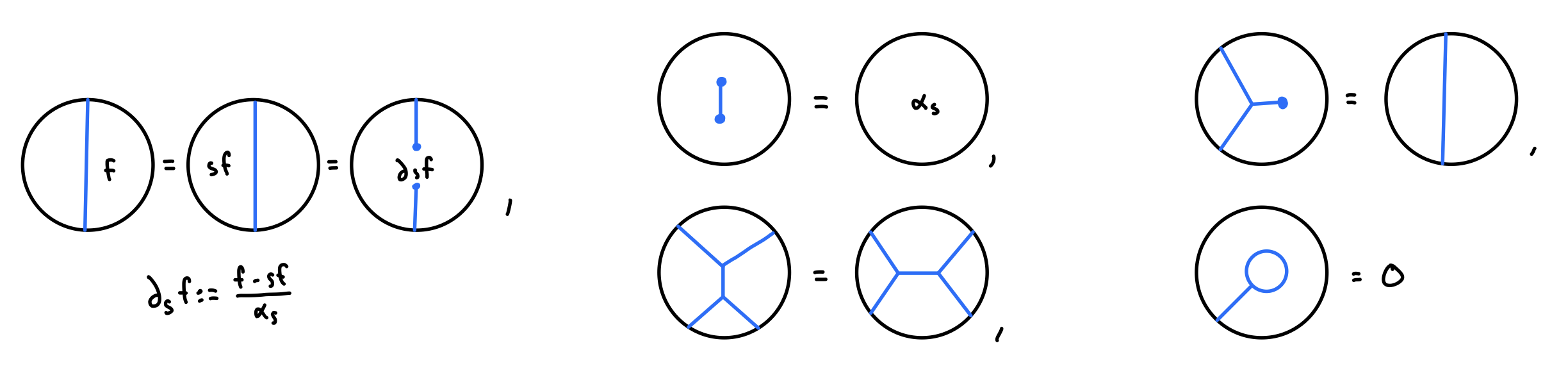}
\]
along with the same relations (and generators) obtained by swapping red $\leftrightarrow$ blue and $s \leftrightarrow t$. 

\begin{example}
The diagram
\[
\includegraphics[scale=0.3]{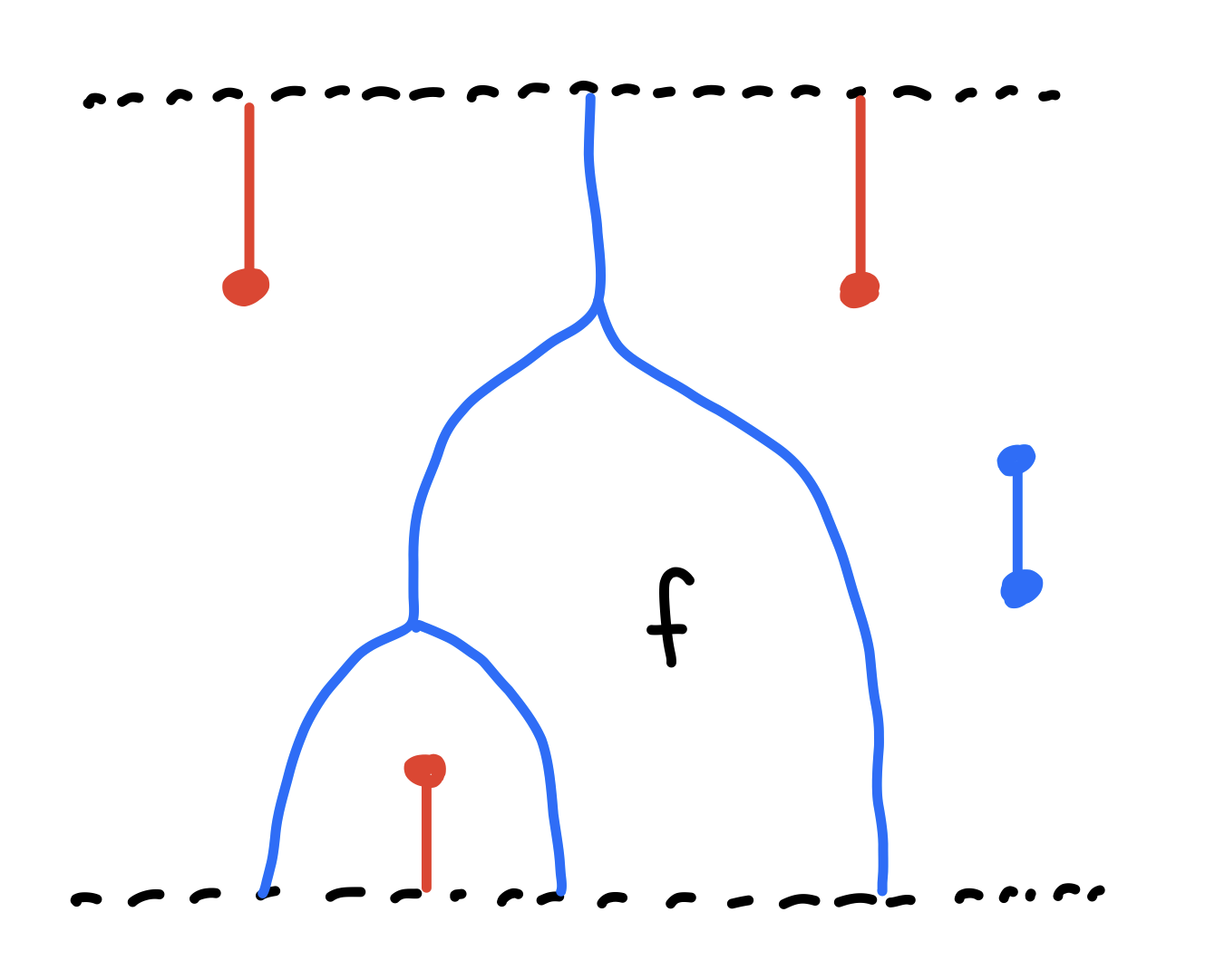}
\]
is a morphism $B_s B_t B_s B_s \rightarrow B_t B_s B_t$. 
\end{example}
From here, we can obtain a version of $\mc{H}_{s.s.}$:
\[
\mc{H}_{BS} \longrightarrow \begin{array}{c} \text{ formally add} \\
\text{shifts, keep} \\
\text{only degree} \\
\text{zero maps} \end{array} \longrightarrow \begin{array}{c} \text{add formal} \\ \text{sums and} \\ 
\text{take Karoubi} \\ 
\text{envelope} \end{array} =: \mc{H}_{s.s.}^\mathrm{diag}
\]
\begin{theorem}
(Elias-W.) Using $\mf{h}_\mathrm{loop}$, 
\[
\mc{H}_{s.s.}^\mathrm{diag} \simeq \mathrm{Sbim} \simeq \mc{H}_{s.s.}.
\]
\end{theorem}
As a corollary to this theorem, we obtain a rather explicit version of the Hecke category $\mc{H}$ where we can do calculations. 
\begin{corollary}
$K^b(\mc{H}_{s.s.}^\mathrm{diag}) \simeq \mc{H}$. 
\end{corollary}

\vspace{3mm}
\noindent
{\bf How do we get $\mc{H}_\mathrm{ext}$?}  We can extend this constuction to $\mc{H}_\mathrm{ext}$ by adding
\begin{itemize}
    \item a generator $\vcenter{\hbox{\includegraphics[scale=0.1]{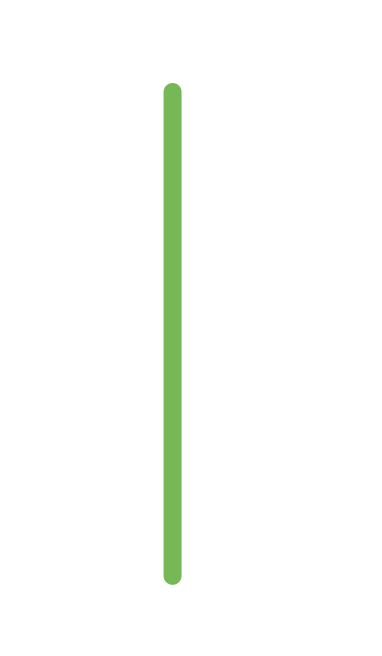}}}$ $\omega$ of $\Omega$, 
    \item morphisms $\vcenter{\hbox{\includegraphics[scale=0.15]{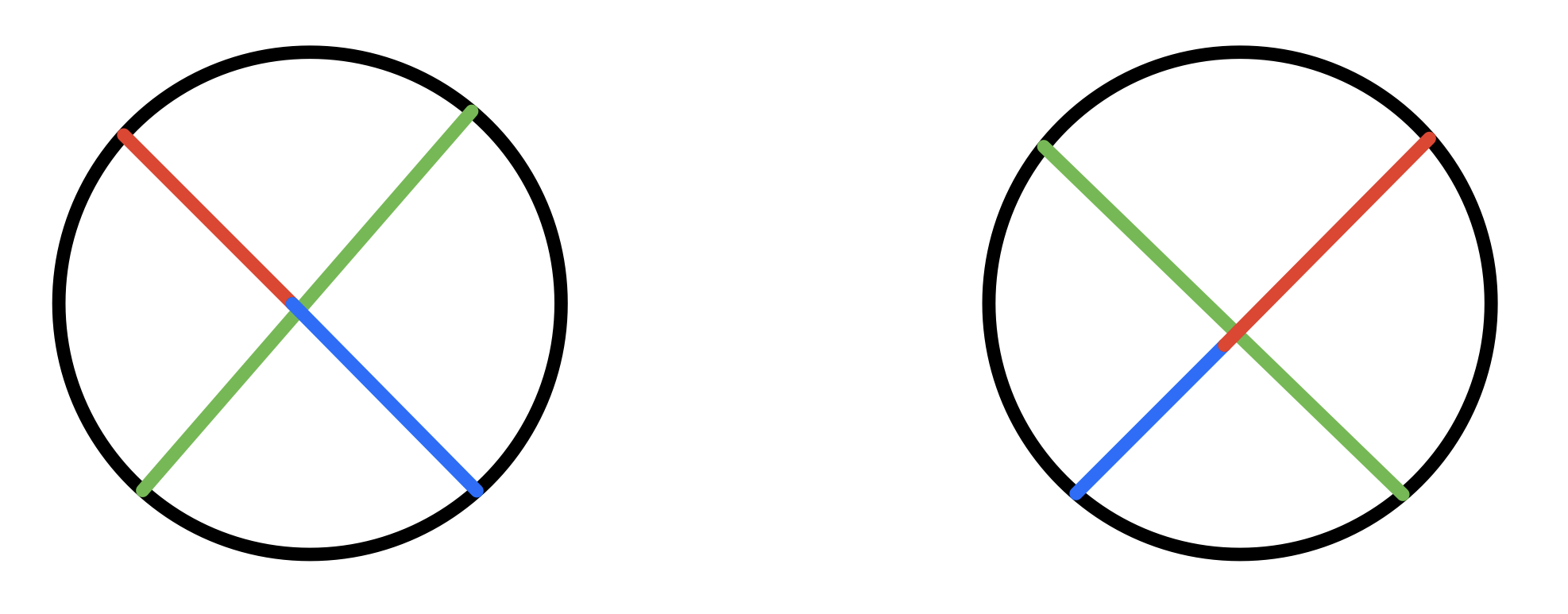}}}$, and 
    \item relations 
    \[
\includegraphics[scale=0.35]{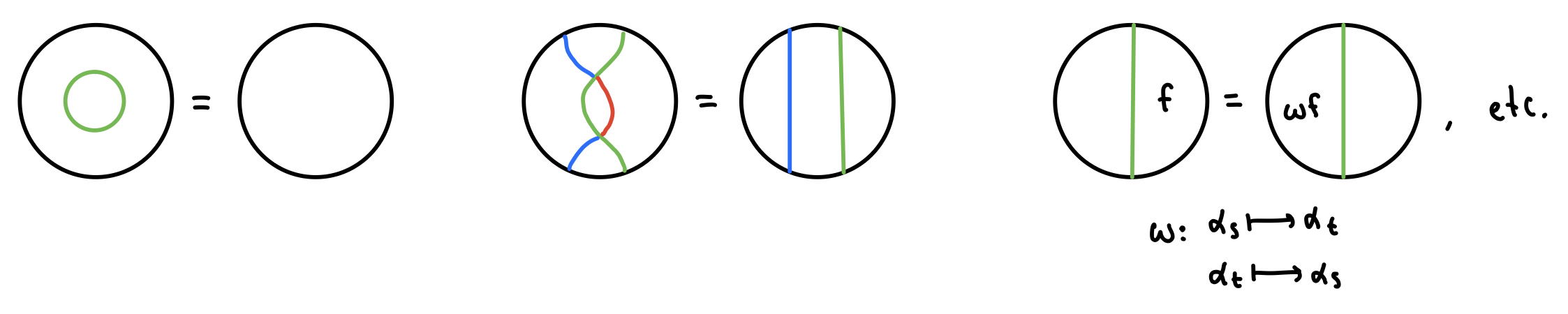}
\]
\end{itemize} 

Now we can translate what we've done in previous lectures to the language of diagrammatic Soergel bimodules. 

\subsection{Braid group categorification}
In Soergel bimodule world, 
\begin{align*}
    \Delta_s &\longmapsto \hspace{5mm} 0 \rightarrow \underline{\underline{B_s}} \xrightarrow{\includegraphics[scale=0.1]{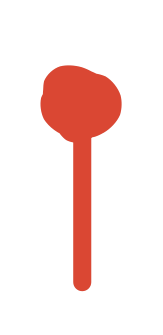}} R(1) \rightarrow 0 \\
    \Delta_w &\longmapsto \hspace{3mm} \text{ ``Rouquier complex'' with many intriguing properties}\\
       \Delta_\omega &\longmapsto \hspace{5mm} 0 \rightarrow \underline{\underline{\omega}} \rightarrow 0  \\
\end{align*}
Here a double underline indicates degree zero. Recall that for the affine simple reflection $t$, we have 
\[
\text{translation by } \varpi = t_\omega.
\]
Hence the Wakimoto sheaves
\[
J_\varpi = \underline{\underline{B_t \omega}} \xrightarrow{\includegraphics[scale=0.1]{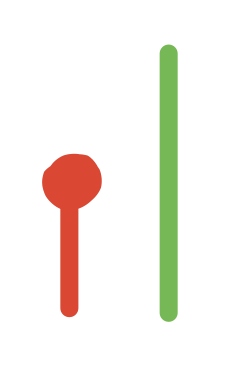}} \omega(1), \hspace{5mm} J_{-\varpi} = \omega(-1) \xrightarrow{\includegraphics[scale=0.1]{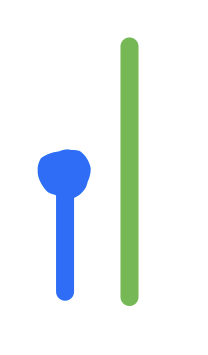}} \underline{\underline{B_s \omega}},
\]
and 
\[
J_{2\varpi} = J_\varpi^2 = \hspace{3mm} \underline{\underline{B_t \omega B_t \omega}} \rightarrow \begin{array}{c} B_t \omega \omega (1) \\ \oplus \\ \omega B_t \omega \end{array} \rightarrow \ind (2) \hspace{3mm} \simeq \hspace{3mm} \underline{\underline{B_t B_s}} \xrightarrow{\includegraphics[scale=0.1]{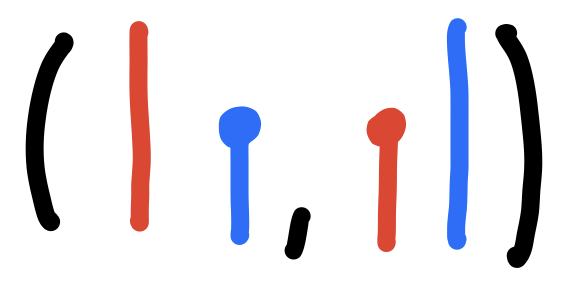}} \begin{array}{c} B_{ts} \omega (1) \\ \oplus \\ B_{st} \omega (1) \end{array} \xrightarrow{\includegraphics[scale=0.1]{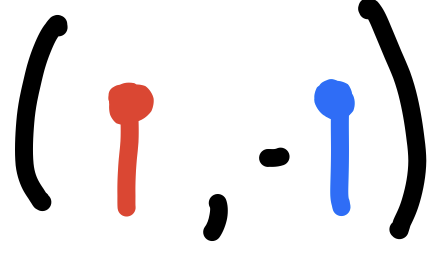}} \ind (2).
\]
\begin{exercise}
Show that 
\[
J_{3 \varpi} = \underline{\underline{B_{tst}}} \rightarrow \begin{array} {c} B_{ts} \omega (1) \\ \oplus \\ B_{st} \omega (1) \end{array} \rightarrow \begin{array}{c} B_s \omega (2) \\ \oplus \\ B_t \omega (2) \end{array} \rightarrow B_{id} \omega
\]
and 
\[
J_{-2 \varpi} = \ind(-2) \rightarrow \begin{array}{c} B_s(-1) \\ \oplus \\ B_t(-1) \end{array} \rightarrow \underline{\underline{B_s B_t}}. 
\]
\end{exercise}
Consider 
\[
\mc{F} = \ind(-1) \xrightarrow{\includegraphics[scale=0.1]{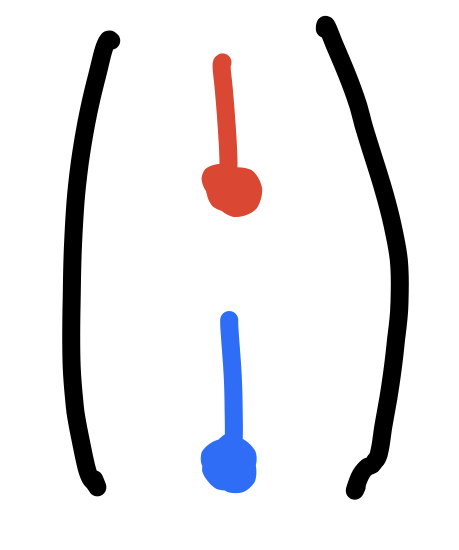}} \begin{array}{c} B_s \\ \oplus \\ B_t \end{array} \xrightarrow{\includegraphics[scale=0.1]{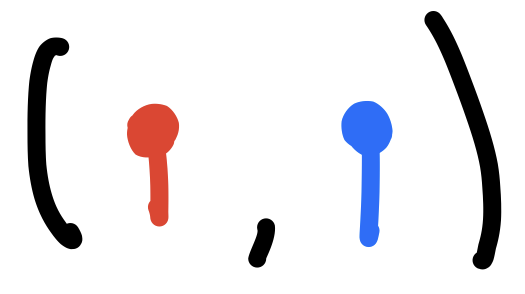}} \ind(1).
\]
Note that $\alpha^2 = \vcenter{\hbox{\includegraphics[scale=0.15]{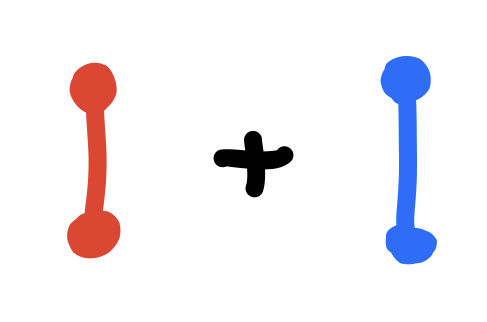}}} = \begin{cases} \delta & \text{ in } \mf{h}_\mathrm{loop}, \\ 
0 &\text{ in } \mf{h}. \end{cases}$ 

We have 
\[
\mc{F} B_s \hspace{3mm} \simeq \hspace{3mm} B_s(-1) \rightarrow \begin{array}{c} B_s B_s \\ \oplus \\ B_t B_s \end{array} \rightarrow B_s(1) \hspace{3mm} \simeq \hspace{3mm} B_s(-1) \rightarrow \begin{array}{c} B_s(-1) \\ \oplus \\ B_s(1) \\ \oplus \\ B_t B_s \end{array} \rightarrow B_s(1) \hspace{3mm} \simeq \hspace{3mm} B_t B_s. 
\]
Similarly, 
\[
\mc{F} B_t \simeq B_s B_t, \hspace{5mm} B_s \mc{F} \simeq B_s B_t, \text{ and } \hspace{5mm} B_t \mc{F} \simeq B_t B_s. 
\]
\begin{exercise}
Checek this. 
\end{exercise}

\begin{theorem}
(Elias) $z_1:= \mc{F} \omega$ has a canonical central structure. Moreover, there exist maps 
\[
\ind \rightarrow z_1^2, \hspace{5mm} z_1^2 \rightarrow \ind
\]
satisfying the Temperly-Lieb relations.
\end{theorem}
Hence we obtain 
\[
Z: \Rep \SL_2 \rightarrow \mc{H}_{ext}. 
\]
\begin{example}
\[
\includegraphics[scale=0.75]{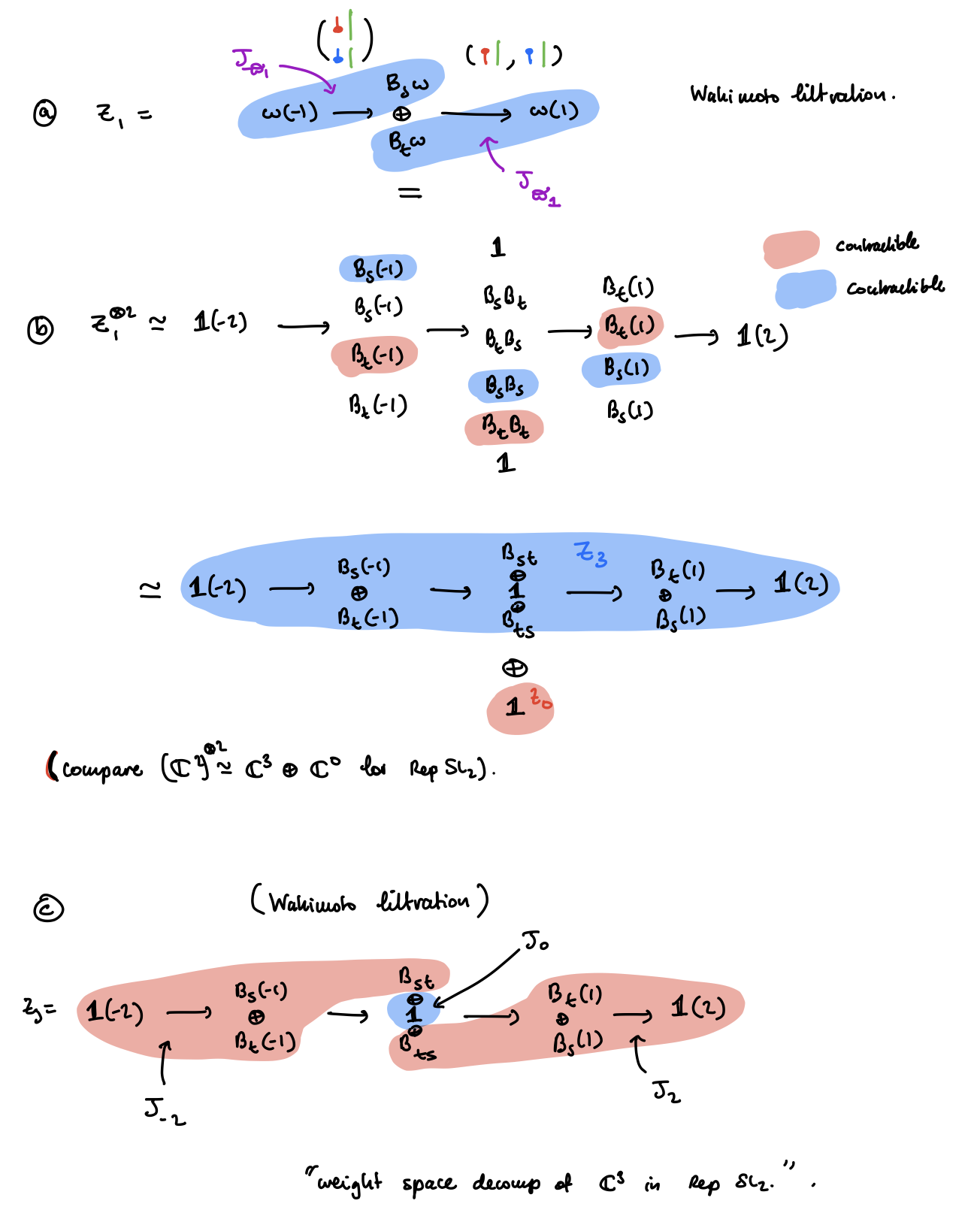}
\]
\end{example}

%% file: lecture-34.tex
\section{Lecture 34: Bezrukavnikov's equivalence}

This is the final lecture of the course! 

\subsection{Potted history of Koszul duality}
A more thorough description of the material in this section can be found in \cite[\S 3]{ICM}. Let $\mc{O}_0$ be the principal block of category $\mc{O}$ for $\mf{g}$, and $\Delta_x, \nabla_x,L_x, T_x, P_x, I_x$ the standard, costandard, simple, tilting, projective, and injective objects in $\mc{O}_0$ corresponding to $x \in W$. Let $\mc{O}_0^\vee$ be the principal block of category $\mc{O}$ for the Langlands dual Lie algebra $\mf{g}^\vee$, and $\Delta_x^\vee, \nabla_x^\vee$, etc. the corresponding objects in $\mc{O}_0^\vee$. 

Beilinson-Ginzburg ('86) conjectured that: 
\begin{enumerate}
    \item $\mc{O}_0$ admits a graded version $\widetilde{\mc{O}}_0$ with a ``shift of grading functor'' $\langle m \rangle$, and all ``canonical'' objects above admit lifts $\widetilde{\Delta}_x, \ldots, \widetilde{I}_x$. 
    \item There exists an equivalence of triangulated categories 
    \[
    \kappa: D^b(\widetilde{\mc{O}}_0) \xrightarrow{\sim} D^b(\widetilde{\mc{O}}_0^\vee) 
    \]
    such that $\kappa \circ \langle -1 \rangle [1] \simeq \langle 1 \rangle \circ \kappa$, and $\kappa$ sends
    \[
    \widetilde{\Delta}_x \mapsto \widetilde{\Delta}_{x^{-1}w_0}^\vee, \hspace{5mm} \widetilde{L}_x \mapsto \widetilde{I}_{x^{-1}w_0}^\vee, \hspace{5mm} \widetilde{P}_x \mapsto \widetilde{L}_{x^{-1}w_0}^\vee
    \]
\end{enumerate}
The conjecture was proved by Beilinson-Ginzburg-Soergel in \cite{BGS}.
\begin{exercise}
Decategorify $\kappa$ and deduce something remarkable about Kazhdan--Lusztig polynomials! (See the first chapter on Koszul duality chapter in \cite{book}.) 
\end{exercise}

\vspace{3mm}
\noindent
{\bf Beilinson-Ginzburg ('99):}  If we compose $\kappa$ with the Radon transform ($* \Delta_{w_0}$) + inversion ($g \mapsto g^{-1}$) and pass to a geometric setting, we obtain a functor 
\[
\widetilde{\kappa}: D^\mathrm{mix}_{(B)}(G/B) \xrightarrow{\sim} D^\mathrm{mix}_{(B^\vee)}(B^\vee \backslash G^\vee)
\]
is an equivalence with $\widetilde{\kappa} \circ \langle -1 \rangle [1] \simeq \langle 1 \rangle \circ \widetilde{\kappa}$, and 
\[
\widetilde{IC}_x \mapsto \widetilde{T}^\vee_x, \hspace{5mm} \widetilde{\Delta}_x \mapsto \widetilde{\Delta}^\vee_x, \hspace{5mm} \widetilde{\nabla}_x \mapsto \widetilde{\nabla}^\vee_x, \hspace{5mm} \widetilde{T}_x^\vee \mapsto \widetilde{IC}_x^\vee. 
\]
After seeing this more symmetric version, one starts to start to dream a little.

\vspace{3mm}
\noindent
{\bf Dream:} \begin{enumerate}
    \item It should work for any Kac-Moody group. (We've removed $w_0$ from the formulas!) 
    \item It should be ``monoidal''; i.e. we should roughly have 
    \[
    (\text{semisimple}, *) \simeq (\text{tiltings}, *^\vee). 
    \]
    \end{enumerate}
Unfortunately, this is very much in dreamland. 

\vspace{3mm}
\noindent
{\bf Basic problem:} $T_x \in D_{(B)}(G/B)$ lifts to $D_B(G/B)$ if and only if $x=id$. 

\begin{remark}
We have seen this a number of times for $\PP^1$ (e.g. Example \ref{equivariant category not highest weight}): $T_s$ does not have an equivariant structure. 
\end{remark}

One way we could try to resolve this problem is to work in $D_U(G/B)$ instead, because $T_x$ does lift to this category. However, in $D_U(G/B)$ we no longer have convolution. To fix this, we can go one step further and work in $D_U(G/U)$, which gives us access to $T_x$ and convolution. But unfortunately, this introduces new problems. 

\vspace{3mm}
\noindent
{\bf Second problem:} $*$ is not exact on $D_U(G/U)$. 

\vspace{3mm}
To illustrate this problem, we can consider a toy example

\subsection{A toy example}
Let $G=S^1$ (if you like Lie groups) or $\G_m$ (if you prefer the algebraic setting). 
\begin{exercise}
\label{cohomology}
Given $\mc{L}_1, \mc{L}_2 \in \Loc(S^1)$, 
\[
(\mc{L}_1 * \mc{L}_2)_1 \simeq H^*(\mc{L}_1 \otimes \mc{L}_2^*).
\]
The subscript $1$ denotes the stalk at the identity. As an additional exercise, describe the monodromy of $\mc{L}_1 * \mc{L}_2$ in terms of the monodromy on $\mc{L}_1$ and $\mc{L}_2$.
\end{exercise}
The basic issue is the following. We have seen (Example \ref{examples of perverse sheaves}) that for a local system $\mc{L}/S^1$ given by monodromy $\mu \circlearrowright V$, 
\[
H^0(S^1, \mc{L}) = V^\mu \hspace{5mm} \text{ and }\hspace{5mm} H^1(S^1, \mc{L}) = V_\mu. 
\]
For $V$ finite dimensional, $V^\mu \neq 0 \iff V_\mu \neq 0$. Hence for any finite-dimensional local system $\mc{L}$ on $S^1$, either $H^*(S^1,\mc{L})=0$ or $H^*(S^1, \mc{L})\neq0$ in both degree $0$ and degree $1$. This implies (by Exercise \ref{cohomology}) that $*$ is either zero or {\em not exact}.

\vspace{3mm}
\noindent
{\bf Solution:} Infinite-dimensional local systems! Take $\mc{L}$ to be the local system corresponding to \[
 \C[[t]] \circlearrowleft \mu=1+t. \]
Note that $\C[[t]]^\mu = 0$, $\C[[t]]_\mu = \C[[t]]/(1-\mu) = \C$. 

\begin{remark}
Roughly one can think of $\mc{L}$ as an infinite-dimensional Jordan block
\[
\bp \ddots & & & & 0 \\ & 1& 1 & & \\ & & 1 & 1 & \\  & & & 1 & 1 \\ 0 & & & & 1 \ep,
\]
which has one-dimensional coinvariants and no invariants. 
\end{remark}

\begin{exercise}
$m_!(\mc{L} \boxtimes \mc{L})[1] \simeq \mc{L}$, hence we have an exact $*$-product! 
\end{exercise}

\subsection{Torus monodromic sheaves \'{a} la Bezrukavnikov-Yun} The ideas we encountered in the toy example above were formalized by Bezrukavnikov-Yun in \cite{BY} (c.f. Beilinson's 1983 ICM talk). Let 
\begin{align*}
    D(B \hspace{1mm} \leftdash G \rightdash \hspace{1mm} B) &:= \langle q^* D(B \backslash G /B) \rangle \\
    &= \langle p^* D(U \backslash G/B) \rangle,
\end{align*}
where 
\[
\begin{tikzcd}
 U \backslash G / U \arrow[rd, "p"] \arrow[dd, "q"']&  \\
& U \backslash G /B \arrow[ld] \\
B \backslash G /B 
\end{tikzcd}
\]

\begin{example}
\begin{enumerate}
\item If $G=\C^\times$, 
\begin{align*}
    D(B \hspace{1mm} \leftdash G \rightdash \hspace{1mm} B) &= \langle p^* D(\mathrm{pt}) \rangle \\
    &= \text{ full subcategory of $D^b_c(\C^\times)$ generated by $\Q_\C$} \\
    &= \{ \mc{F} \mid \mc{H}^i(\mc{F}) \text{ local systems with unipotent monodromy} \}. 
\end{align*}
\item If $G=\SL_2$, $G/U \simeq \C^2 \backslash \{0 \}$, so $D(U \backslash G/B) \simeq \langle \IC_0, \IC_{\PP^1} \rangle$. Hence 
\[
D(B \hspace{1mm} \leftdash G \rightdash \hspace{1mm} B) = \langle \Q_{\C^\times}, \Q_{\C^2 \backslash \{0\}} \rangle \subset D_U(G/U).
\]
\end{enumerate}
\end{example}
By replacing $B$ with $U$, we get a convolution product \[
\overset{U}{*}: D(B \hspace{1mm} \leftdash G \rightdash \hspace{1mm} B) \times D(B \hspace{1mm} \leftdash G \rightdash \hspace{1mm} B) \rightarrow D(B \hspace{1mm} \leftdash G \rightdash \hspace{1mm} B) 
\]
from the diagram 
\[
\begin{tikzcd}
&G \times_U G/U \arrow[dl] \arrow[dr] \arrow[rr] & & G/U \\
G/U & & G/U & &
\end{tikzcd}
\]
using the same definition as previously, except that $m:G \times_U G/U$ is not proper, so we need to make a choice of which pushforward $m_*$ or $m_!$ to use. We choose $m_!$.

\vspace{3mm}
\noindent
{\bf Very technical point:} (cf. appendix by Yun to \cite{BY})
\begin{enumerate}
    \item One can complete $D(B \hspace{1mm} \leftdash G \rightdash \hspace{1mm} B)$ to the ``free monodromic completion'' $\widehat{D}(B \hspace{1mm} \leftdash G \rightdash \hspace{1mm} B)$ to allow pro-local systems (like $\C[[t]] \circlearrowleft$ from earlier) along the fibres of $p$. 
    \item The convolution product $\overset{U}{*}$ extends to $\widehat{D}(B \hspace{1mm} \leftdash G \rightdash \hspace{1mm} B)$. Moreover, we can define a category  
    \[
    \widehat{\mathrm{Tilt}} \subset \widehat{D}(B \hspace{1mm} \leftdash G \rightdash \hspace{1mm} B)
    \]
    of ``free monodromic tilting sheaves,'' and $\overset{U}{*}$ is exact on $\widehat{\mathrm{Tilt}}$. 
\end{enumerate}
 \begin{theorem} \cite{BY} 
    \[
    (\widehat{\mathrm{Tilt}}, \overset{U}{*}) \simeq (\widehat{\mathrm{SBim}}, \underset{\widehat{R}}{\otimes})
    \]
    \end{theorem}
    The category on the right side of this equivalence consists of Soergel bimodules for $R=\mc{O}(\Lie T^\vee) = \mc{O}((Lie T)^*)$ (note the dual!), completed along the grading. 
    
    \vspace{3mm}
    \noindent
    {\bf Idea of proof:} Suppose that $\begin{tikzcd} X \arrow[d, "p"] \\ Y
    \end{tikzcd}$ is a $T$-torsor. Define a category 
    \[
    D^b(X\rightdash \hspace{1mm}T) = \langle p^*D_c^b(Y) \rangle_\Delta \subset D^b(X)
    \]
    of ``unipotently monodromic sheaves''. The fundamental observation of Verdier is that every $\mc{F} \in D^b(X\rightdash \hspace{1mm}T)$ has a canonical  monodromy action of $\pi_1(T)$. The action is unipotent. In other words, $D^b(X\rightdash \hspace{1mm}T)$ is linear over $\Q[\pi_1(T)]$. Taking logs, we obtain that $D^b(X\rightdash \hspace{1mm}T)$ is $\widehat{R}$-linear. 
    
    \begin{theorem}
    \cite{BY} The functor $\mathbb{V}:\widehat{\mathrm{Tilt}} \rightarrow \widehat{\mathrm{SBim}}$ is monoidal and fully-faithful. 
    \end{theorem}
    Adding weights, one gets monoidal equivalences
    \[
    (\widehat{\mathrm{Tilt}}^\mathrm{mixed}, *) \simeq (\mathrm{SBim}, \underset{R}{\otimes}) \simeq (\mc{H}_{s.s.}^{G^\vee}, *),
    \]
    and hence taking homotopy categories, 
    \[
    D^{\mathrm{mixed}}(B \hspace{1mm} \leftdash G\rightdash \hspace{1mm}B) \simeq \mc{H} \simeq D^{\mathrm{mixed}}(B^\vee \backslash G^\vee / B^\vee).
    \]
    This achieves the dream of a ``monoidal Koszul duality''.
    
    \vspace{3mm}
    \noindent
    {\bf A beautiful feature:} Let $\pi:U \backslash G/U \rightarrow U \backslash G/B$ be projection. Then 
    \[
    \pi_*(\widehat{T}_x)[\mathrm{rank}T] = T_x \in D_U(G/B).
    \]
    In other words, taking coinvariants of free monodromic tilting sheaves gives indecomposible tilting sheaves downstairs. 
    
    \begin{remark}
    With enough\footnote{approximately $250$ pages} homological algebra, one can do all of the above algebraically (and also mod $p$), c.f. \cite{AMRW}. Some of the constructions there are very adhoc. Recent foundational work of Hogencamp and Makisumi explains much more satisfactorily what is going on.
    \end{remark}
    
    \subsection{Bezrukavnikov's equivalence}
    {\bf The constructible side:} 
    Associated to our choice $T \subset B \subset G$, we have the familiar zoo of characters: 
    \begin{itemize}
        \item the pro-unipotent radical $I^0\subset I \subset G((t))$,
        \item the extended affine flag variety (aka ``affine base affine space'') $\widetilde{\mc{F}\ell} = G((t))/I^0$,
        \item the affine flag variety $\mc{F}\ell = G((t))/I$, 
        \item the affine Grassmannian $\mc{G}r = G((t))/G[[t]]$,
        \item projections
        \[
        \widetilde{\mc{F}\ell} \xrightarrow{T\text{-torsor}} \mc{F}\ell \rightarrow \mc{G}r, 
        \]
        \item derived categories 
        \begin{align*}
            D_{II} &:= D(I \backslash G((t))/I), \\
            D_{I^0I} &:= D(I^0 \backslash G((t))/I) \rightsquigarrow \widehat{D}_{I^0I}, \hspace{5mm} \text{ free monodromic completion}\\
            D_{I^0I^0} &:= D(I^0 \hspace{1mm} \leftdash G((t))\rightdash \hspace{1mm} I^0) \rightsquigarrow \widehat{D}_{I^0I^0},  \hspace{5mm} \text{ free monodromic completion}
        \end{align*}
        \item and actions
        \begin{align*}
            \widehat{D}_{I^0I^0} &\circlearrowright \widehat{D}_{I^0I} \circlearrowleft \widehat{D}_{II} \\
            D_{I^0I^0} &\circlearrowright D_{I^0I} \circlearrowleft D_{II}. 
        \end{align*}
    \end{itemize}
    
    \noindent
    {\bf The coherent side:} Associated to $G^\vee$, we have 
    \[
    \begin{tikzcd}
    \widetilde{\mc{N}}^\vee \arrow[d] \arrow[r, hookrightarrow] & \widetilde{\mf{g}}^\vee \arrow[d] \\
    \mc{N}^\vee \arrow[r, hookrightarrow] & \mf{g}^\vee, 
    \end{tikzcd}
    \]
    where the left vertical arrow is the Springer resolution and the right vertical arrow is the Grothendieck-Springer resolution sending $(x, \mf{b}) \in \widetilde{\mf{g}}^\vee = \{(x, \mf{b}) \subset \mf{g}^\vee \times \mc{B}^\vee \mid x \in \mf{b} \}$ to $x \in \mf{g}^\vee$.
    
    \vspace{3mm}
    \noindent
    {\bf Recall:} Given $X \xrightarrow{f} Y$ proper, one can construct a convolution structure on $(D^b(X \times_Y X), *)$ (e.g. if $X$ is a finite set and $Y$ is a point, then $D^b(X \times_Y X) \simeq$ $X \times X$ matrices of chain complexes of vector spaces). Earlier we used this construction to give a coherent realization of the affine Hecke algebra via the Kazhdan--Lusztig isomorphism:
    \[ 
    H^\mathrm{aff} \simeq K^{G^\vee \times \C^\times}(\widetilde{\mc{N}}^\vee \times_{\mc{N}} \widetilde{\mc{N}}^\vee).
    \]
    We need a derived version of this (which will mostly disappear in a moment). 
    
    Recall that $X \times_Y X$ is constructed by gluing $\Spec(B \otimes_A B)$ for an affine cover $\Spec B \rightarrow \Spec A$ of $f$. Similarly, we can construct $X \overset{L}\times_Y X$ by gluing the $dg$-schemes $\Spec (B \overset{L}{\otimes}_A B)$ for an affine cover $\Spec B \rightarrow \Spec A$ of $f$.
    
    \vspace{3mm}
    \noindent
    {\bf Key points:}
    \begin{enumerate}
        \item $D^b(X \overset{L}\times_Y X)$ is still triangulated and monoidal. 
        \item If $\mathrm{Tor}_{>0}^{\mc{O}_Y}(\mc{O}_X, \mc{O}_X)=0$, then $X \overset{L} \times_Y X = X \times_Y X$. 
    \end{enumerate}
    From this we define our coherent characters:
    \begin{itemize}
        \item $St_{\mf{g}^\vee} := \widetilde{\mf{g}}^\vee \overset{L} \times_{\mf{g}^\vee} \widetilde{\mf{g}}^\vee = \widetilde{\mf{g}}^\vee \times_{\mf{g}^\vee} \widetilde{\mf{g}}^\vee$, 
        \item $St_{\mc{N}^\vee} := \widetilde{\mf{g}}^\vee \overset{L} \times _{\mf{g}^\vee} \widetilde{\mc{N}}^\vee = \widetilde{\mf{g}}^\vee \times_{\mf{g}^\vee} \widetilde{\mc{N}}^\vee$, and 
        \item $St^L := \widetilde{\mc{N}}^\vee \overset{L}\times _{\mf{g}^\vee} \widetilde{\mc{N}}^\vee$ \hspace{5mm} $(\neq \widetilde{\mc{N}}^\vee \times_{\mf{g}^\vee} \widetilde{\mc{N}}^\vee = \widetilde{\mc{N}} \times_{\mc{N}^\vee} \widetilde{\mc{N}}^\vee$). 
    \end{itemize}
    \begin{theorem}
    (Bezrukavnikov) There exist vertical equivalences 
    \[
    \begin{tikzcd}
    D^b \Coh^{G^\vee}(St_{\mf{g}^\vee}) \arrow[d, "(a)", "\sim"'] & \circlearrowright & D^b \Coh^{G^\vee}(St_{\mc{N}^\vee}) \arrow[d,  "(b)", "\sim"'] & \circlearrowleft & D^b \Coh^{G^\vee} (St^L) \arrow[d, "(c)", "\sim"'] \\ 
    D_{I^0I^0} & \circlearrowright & D_{I^0I} & \circlearrowleft & D_{II}
    \end{tikzcd}
    \]
    making the diagram of module categories commute.
    \end{theorem}
    
    \noindent
    {\bf Remarks on the proof:}
    \begin{enumerate}
        \item Equivalences $(b)$ and $(c)$ and compatibility are reasonably straightforward consequences of $(a)$, which is where Bezrukavnikov spends most of the paper. 
        \item It is technically convenient to instead prove 
        \[
        \widehat{D}_{I^0I^0} \simeq D(\Coh^{G^\vee}(\widehat{St}_{\mf{g}^\vee})),
        \]
        where $\widehat{St}$ denotes the formal completion of $St_{\mf{g}^\vee}$ along the preimage of $\mc{N}^\vee$. 
        \item The proof relies heavily on ideas in \cite{AB}. 
    \end{enumerate}
    
    The theorem has several remarkable consequences, but exploring them will have to wait until the next course! 